\documentclass[12pt,oneside,english,shortlabels]{amsart}

\usepackage[T1]{fontenc}
\usepackage[latin9]{inputenc}
\usepackage{color}
\usepackage{babel}
\usepackage{amstext}
\usepackage{amssymb}
\usepackage{esint}
\usepackage[all]{xy}
\usepackage[unicode=true,pdfusetitle,
 bookmarks=true,bookmarksnumbered=false,bookmarksopen=false,
 breaklinks=false,pdfborder={0 0 0},backref=false,colorlinks=true]
 {hyperref}
\hypersetup{
 linkcolor=blue}
\usepackage{breakurl}

\usepackage{graphicx} 

\makeatletter

\usepackage{amsthm}
\usepackage{enumitem}		

\theoremstyle{definition} 

 \newtheorem{definition}{Definition}[section]
 \newtheorem{remark}[definition]{Remark}
 
 \newtheorem{claim}[definition]{Claim}

\newtheorem*{notation}{Notations}

\theoremstyle{plain}      

 \newtheorem{proposition}[definition]{Proposition}
 \newtheorem{theorem}[definition]{Theorem}
 \newtheorem{corollary}[definition]{Corollary}
 \newtheorem{lemma}[definition]{Lemma}

\newtheorem{conjecture}{Conjecture}

\usepackage{mathrsfs} 
\providecommand{\noopsort}[1]{}

\makeatletter
\newcommand*{\house}[1]{
  \mathord{
    \mathpalette\@house{#1}
  }
}
\newcommand*{\@house}[2]{
  \dimen@=\fontdimen8 %
      \ifx#1\scriptscriptstyle\scriptscriptfont
      \else\ifx#1\scriptstyle\scriptfont
      \else\textfont\fi\fi
      3 %
  \sbox0{%
    $#1%
      \vrule width\dimen@\relax
      \overline{%
        \kern2\dimen@
        \begingroup 
          #2%
        \endgroup
        \kern2\dimen@
      }
      \vrule width\dimen@\relax
      \mathsurround=1.5\dimen@ 
    $
  }
  \ht0=\dimexpr\ht0-\dimen@\relax
  \dp0=\dimexpr\dp0+2\dimen@\relax
  \vbox{
    \kern\dimen@ 
    \copy0 
  }
}

\def\dyg{{\rm dyg}}

\newcommand{\lo}{{\rm Log\,}}

\newcommand{\sbb}{\mathbb{S}}
\newcommand{\rb}{\mathbb{R}}
\newcommand{\pb}{\mathbb{P}}
\newcommand{\ab}{\mathbb{A}}
\newcommand{\cb}{\mathbb{C}}
\newcommand{\db}{\mathbb{D}}
\newcommand{\eb}{\mathbb{E}}
\newcommand{\fb}{\mathbb{F}}
\newcommand{\gb}{\mathbb{G}}
\newcommand{\lb}{\mathbb{L}}
\newcommand{\kb}{\mathbb{K}}
\newcommand{\hb}{\mathbb{H}}
\newcommand{\zb}{\mathbb{Z}}
\newcommand{\tb}{\mathbb{T}}
\newcommand{\qb}{\mathbb{Q}}
\newcommand{\nb}{\mathbb{N}}

\newcommand{\rc}{\mathcal{R}}

\newcommand{\dc}{\mathcal{D}}
\newcommand{\kc}{\mathcal{K}}
\newcommand{\lc}{\mathcal{L}}

\def\11{{\mathbf 1}}

\theoremstyle{remark}

\theoremstyle{remark}
\newtheorem{exampl}[subsubsection]{Example}

\theoremstyle{remark}

\def\bee{\begin{exampl}}
\def\eee{\end{exampl}}
\def\bn{\begin{notation}}
\def\en{\end{notation}}
\def\br{\begin{remark}}
\def\er{\end{remark}}
\def\bp{\begin{prop}}
\def\ep{\end{prop}}
\def\bpr{\begin{proof}}
\def\epr{\end{proof}}
\def\bt{\begin{thm}}
\def\et{\end{thm}}
\def\be{\begin{equation}}
\def\ee{\end{equation}}
\def\bl{\begin{lem}}
\def\el{\end{lem}}
\def\bc{\begin{cor}}
\def\ec{\end{cor}}
\def\bd{\begin{defn}}
\def\ed{\end{defn}}

\numberwithin{equation}{subsection}

\author{Jean-Louis Verger-Gaugry}
\thanks{}
\address{
Laboratoire de Math\'ematiques de l'Universit\'e de Savoie Mont~Blanc, UMR CNRS 5127,
\!B\^atiment Chablais, \!Campus scientifique,
\!73376 Le Bourget-du-Lac cedex, \!France}
\email{Jean-Louis.Verger-Gaugry@univ-smb.fr}

\title[A proof of the Conjecture of Lehmer]
{A proof of the Conjecture of Lehmer and of the Conjecture of Schinzel-Zassenhaus}

\begin{document}

\dedicatory{\today}

\begin{abstract}
The \,conjecture \,of \,Lehmer \,is \,proved \,to \,be \,true. 
The proof mainly relies upon: 
(i) the properties of the Parry Upper functions $f_{\house{\alpha}}(z)$
associated with the
dynamical zeta functions $\zeta_{\house{\alpha}}(z)$
of the R\'enyi--Parry arithmetical dynamical systems,
for $\alpha$ an algebraic integer $\alpha$ of house 
$\house{\alpha}$ greater than 1,
(ii) the discovery of lenticuli of poles of
$\zeta_{\house{\alpha}}(z)$ which uniformly 
equidistribute at the limit
on a limit ``lenticular" arc of the unit circle, 
when $\house{\alpha}$ tends to $1^+$,
giving rise to a continuous lenticular
minorant ${\rm M}_{r}(\house{\alpha})$ of the 
Mahler measure
${\rm M}(\alpha)$,
(iii) the Poincar\'e asymptotic expansions of these poles
and of this minorant ${\rm M}_{r}(\house{\alpha})$
as a function of the dynamical degree. 
With the same arguments the conjecture of
Schinzel-Zassenhaus is proved to be true.
An inequality improving those of Dobrowolski 
and Voutier
ones is obtained. 
The set of Salem numbers is shown to be
bounded from below by the Perron number
$\theta_{31}^{-1} = 1.08545\ldots$, dominant root
of the
trinomial $-1 - z^{30} + z^{31}$.
Whether Lehmer's number is the smallest
Salem number remains open.
A lower bound for the Weil height
of nonzero totally real algebraic numbers,
$\neq \pm 1$, is
obtained (Bogomolov property).
For sequences of algebraic integers 
of Mahler measure smaller than the smallest Pisot number, 
whose houses have
a dynamical degree tending to infinity,
the Galois orbit measures 
of conjugates are proved to
converge 
towards the Haar measure
on $|z|=1$ (limit equidistribution). 
\end{abstract}




\maketitle

\section{Introduction}
\label{S1}

The question asked by Lehmer in
\cite{lehmer} (1933) about the existence of
integer univariate polynomials of Mahler measure
arbitrarily close to one became a conjecture.
Lehmer's Conjecture 
is stated as follows:
there exists an universal constant $c > 0$ such that 
the Mahler measure M$(\alpha)$ satisfies
M$(\alpha) \geq 1 + c$ for all
nonzero algebraic numbers $\alpha$, not being a root of unity (Amoroso \cite{amoroso2}
\cite{amoroso3},
Blansky and Montgomery \cite{blanskymontgomery},
Boyd \cite{boyd4} \cite{boyd5}, 
Cantor and Strauss \cite{cantorstrauss}, 
Dobrowolski \cite{dobrowolski2}, 
Dubickas \cite{dubickas2},  
Hindry and Silverman \cite{hindrysilverman}, 
Langevin \cite{langevin}, 
Laurent \cite{laurent},
Louboutin \cite{louboutin}, 
Masser \cite{masser4},
Mossinghoff, Rhin and Wu \cite{mossinghoffrhinwu},
Schinzel \cite{schinzel2}, 
Silverman \cite{silverman},
Smyth \cite{smyth5} \cite{smyth6},
Stewart \cite{stewart}, 
Waldschmidt \cite{waldschmidt} \cite{waldschmidt2}). 

If $\alpha$ is a nonzero algebraic integer, 
M$(\alpha) = 1$
if and only if $\alpha=1$ or is a root of unity
by Kronecker's Theorem (1857)
\cite{kronecker}.
Lehmer's Conjecture 
asserts a discontinuity
of the value of ${\rm M}(\alpha)$,
$\alpha \in \mathcal{O}_{\overline{\qb}}$, at
1.
In $\S$ \ref{S2} we evoke the meaning 
of this discontinuity  
in different contexts, in particular
in number theory following
Bombieri \cite{bombieri}, 
Dubickas \cite{dubickas9}
and Smyth \cite{smyth6}.

In this note we prove that Lehmer's Conjecture
is true by establishing 
minorations of the Mahler measure M$(\alpha)$
for any nonzero algebraic integer $\alpha$ 
which is not a root of unity.
The proof contains a certain number of new 
notions which call for comments,
and basically 
brings the
dynamical zeta functions 
$\zeta_{\beta}(z)$ of the $\beta$-shift
associated to algebraic
numbers,
their poles
 and their Poincar\'e asymptotic expansions,
into play
\cite{vergergaugry7}. Let us describe
briefly the ingredients which will appear
in the proof.
Due to the invariance of the Mahler measure
M$(\alpha)$ by the transformations
$z \to \pm z^{\pm 1}$ and 
$z \to \pm \overline{z}^{\pm 1}$,
it is sufficient to consider the
two following cases: 
\begin{itemize}
\item[{\bf (i)}] $\alpha$ real algebraic integer $> 1$, in which case $\alpha$ is generically 
named $\beta$,
\item[{\bf (ii)}] $\alpha$ nonreal complex algebraic 
integer, $|\alpha| > 1$,
with $\arg(\alpha) \in (0, \pi/2]$,
\end{itemize}
with $|\alpha| > 1$ sufficiently close to 1 in both cases.
In Section $\S$ \ref{S6} we show how the 
nonreal complex case {\bf (ii)} can be deduced 
from the real case {\bf (i)}
by considering
the R\'enyi-Parry dynamics of the houses
$\house{\alpha}$. Now,
in case {\bf (i)}, by the Northcott property,
the degree $\deg(\beta)$, valued in 
$\nb \setminus\{0,1\}$, is necessarily
not bounded
when $\beta > 1$ tends to $1^+$. To compensate
the absence of an 
integer function  
of $\beta$ which ``measures" the proximity
of $\beta$ with $1$,
we introduce
the natural
integer function of $\beta$, that we call the 
{\em dynamical degree
of $\beta$}, denoted by 
$\dyg(\beta)$, which is defined by
the relation:
for $1 < \beta \leq \frac{1+\sqrt{5}}{2}$ any real
number, $\dyg(\beta)$ is the unique 
integer $n \geq 3$
such that
\begin{equation}
\label{intervv}
\theta_{n}^{-1} \leq \beta < \theta_{n-1}^{-1}
\end{equation}
where $\theta_{n}$ is the unique root
in $(0,1)$ of the trinomial
$G_{n}(z) = -1 + z + z^n$.
The (unique) simple zero $> 1$
of the trinomial
$G_{n}^{*}(z) := 1 + z^{n-1} - z^{n} , 
n \geq 2$, is
the Perron number $\theta_{n}^{-1}$.
The set of dominant roots 
$(\theta_{n}^{-1})_{n \geq 2}$
of the
nonreciprocal 
trinomials $(G_{n}^{*}(z))_{n \geq 2}$ 
constitute
a strictly decreasing sequence
of Perron numbers,
tending to one.
Section $\S$ \ref{S3} summarizes
the properties of these trinomials.
The sequence 
$(\theta_{n}^{-1})_{n \geq 2}$
will be extensively used in the sequel.
It is a fundamental set of
Perron numbers of
the interval $(1, \theta_{2}^{-1})$
simply indexed by the integer $n$,
and this indextion is extended to any real number
$\beta$ lying between two successive Perron numbers
of this family 
by \eqref{intervv}.
Let us note that $\dyg(\beta)$ is well-defined
for algebraic integers $\beta$
and also 
for transcendental numbers $\beta$. 
Let $\kappa:= \kappa(1, a_{\max}) = 0.171573\ldots$ 
(cf below, and $\S$\ref{S3}
$\S$\ref{S5.1} $\S$\ref{S6.5} for the proofs).

\begin{theorem}
\label{dygdeginequality}
\begin{itemize}
\item[(i)]
For $n \geq 2$, 
\begin{equation}
\label{dygdegPERRON}
\dyg(\theta_{n}^{-1}) = n
~=~
\left\{
\begin{array}{ccc}
\deg(\theta_{n}^{-1}) & \quad ~\mbox{if}~ 
& \quad n \not\equiv 5 ~({\rm mod} \,6\,),\\
\deg(\theta_{n}^{-1}) + 2 & \quad ~\mbox{if}~ 
&\quad n \equiv 5 ~({\rm mod} \, 6\, ),
\end{array}
\right.
\end{equation}
\item[(ii)]
if $\beta$ is a real number which satisfies
$\theta_{n}^{-1} \leq \beta < \theta_{n-1}^{-1},
~n \geq 2$, then the 
asymptotic expansion of the
dynamical degree
$\dyg(\beta) = 
 \dyg(\theta_{n}^{-1}) =
 n$ of $\beta$ 
 is:
\begin{equation}
\label{dygExpression_intro}
\dyg(\beta) = 
- \frac{\lo (\beta - 1)}{\beta - 1}
\Bigl[
1 +
O\Bigl(
\Bigl(
\frac{\lo (-\lo (\beta - 1))}
{\lo (\beta - 1)}
\Bigr)^2
\Bigr)
\Bigr],
\end{equation}
with the constant $1$ in the Big O;
moreover,
if $\beta \in (\theta_{n}^{-1}, \theta_{n-1}^{-1}),
~n \geq 260$, 
is an algebraic integer of
degree
$\deg(\beta)$, then
\begin{equation}
\label{dygdeg_intro}
\dyg(\beta) \Bigl(
\frac{2 \arcsin\bigl( \frac{\kappa}{2} \bigr)}{\pi}
\Bigr)
+
\Bigl(
\frac{2 \kappa \, \lo \kappa}
{\pi \,\sqrt{4 - \kappa^2}} 
\Bigr)
 ~\leq~ \deg(\beta).
 \end{equation}
\end{itemize}
\end{theorem}
Moreover,
for
$\alpha$ any nonreal complex algebraic 
integer, $|\alpha| > 1$,
such that
$1 < \house{\alpha} \leq \frac{1+\sqrt{5}}{2}$,
the dynamical degree of $\alpha$
is defined by
$\dyg(\alpha) := \dyg(\house{\alpha})$;
by extension, if $P(X)$ is an irreducible
integer monic polynomial, we
define $\dyg(P)$ as $:= \dyg(\house{\alpha})$
for any root $\alpha$ of $P$. 

In \cite{vergergaugry6}, 
the problem of Lehmer 
for the 
family $(\theta_{n}^{-1})_{n \geq 2}$
was solved
using the Poincar\'e asymptotic expansions
of the roots of $(G_n)$ of 
modulus $< 1$
and of the Mahler measures
$({\rm M}(\theta_{n}^{-1}))_{n \geq 2}$.
The purpose of the present note is to extend
this method to any algebraic integer $\beta$
of dynamical degree 
$\dyg(\beta)$ large enough, to show that 
this method allows to
prove that 
the Conjecture of Lehmer is true in general.

The choice of the sequence of trinomials
$(G_n)$ is fairly natural in the context
of the R\'enyi-Parry dynamical systems
(recalled in $\S$ \ref{S4}) and leads to a
theory of perturbation
of these trinomials compatible with the dynamics
(in $\S$ \ref{S4.5}).
Therefore, in the present dynamical 
approach, taking the integer function
$\dyg(\beta)$ as an integer variable
tending to infinity when
$\beta > 1$ tends to $1^+$ 
is natural. 
All the asymptotic expansions, 
for the roots
of modulus $< 1$ of the minimal polynomials
$P_{\beta}(z)$, for the lower bounds of the
lenticular Mahler measures
${\rm M}_{r}(\beta)$,
will be obtained as a function of
the integer
$\dyg(\beta)$, when $\beta$ tends to $1^+$.

To the $\beta$-shift, to
the R\'enyi-Parry dynamical system
associated with an algebraic integer
$\beta > 1$ are 
attached several analytic functions:
(i) first,
the minimal polynomial function
$P_{\beta}(z)$
which is (monic) reciprocal 
by a Theorem of C. Smyth \cite{smyth}
as soon as
${\rm M}(\beta) < \Theta = 1.3247\ldots$;
(ii) then the (Artin-Mazur)
dynamical zeta function of the $\beta$-shift
\cite{artinmazur},
the generalized Fredholm determinant of
the transfer operator associated with the $\beta$-transformation
$T_{\beta}$ \cite{baladikeller}, 
the Perron-Frobenius operator associated to $T_{\beta}$
\cite{itotakahashi} \cite{mori}
\cite{mori2}
\cite{takahashi},
the kneading determinants, and its variants, 
of the kneading theory of 
Milnor and Thurston \cite{milnorthurston}.
Their are closely related and recalled in $\S$ \ref{S4}.
Finally the main theorems below will be
obtained using the Parry Upper function
$f_{\beta}(z)$, 
constructed from 
the inverse of the dynamical zeta function
$\zeta_{\beta}(z)$.
The Parry Upper function is a generalization of the
Fredholm determinant associated
with the transfer operator
of the $\beta$-transformation. 

Using ergodic theory 
Takahaski \cite{takahashi}
\cite{takahashi2},
Ito and Takahashi \cite{itotakahashi},
Flatto, Lagarias and Poonen \cite{flattolagariaspoonen}
gave an explicit expression of the
Parry Upper function
$f_{\beta}(z)$ as a function of
the dynamical zeta function
$\zeta_{\beta}(z)$ of the $\beta$-shift. 
These expressions ($\S$\ref{S4.2})
will be extensively used
in the sequel.

The Parry Upper upper function at $\beta$
takes the general form, with a 
lacunarity controlled by the dynamical degree
(Theorem \ref{zeronzeron}):
$$f_{\beta}(z) = -1 + z + z^{\dyg(\beta)} + z^{m_1} +
z^{m_2} + z^{m_3} + \ldots \hspace{2cm}\mbox{}$$
\begin{equation}
\label{dygplusdeformation}
= G_{\dyg(\beta)} + z^{m_1} + z^{m_2} + z^{m_3}  
+ \ldots\end{equation}
with $m_1 \geq 2 \, \dyg(\beta) -1,
m_{q+1}-m_q \geq \dyg(\beta) -1, q \geq 1$. For
$\theta_{\dyg(\beta)}^{-1}
\leq \beta < \theta_{\dyg(\beta)-1}^{-1}$,
the
lenticulus $\lc_{\beta}$
of zeroes of $f_{\beta}(z)$
relevant for the Mahler measure is obtained 
by a deformation of the lenticulus of zeroes
$\lc_{\theta_{{\dyg(\beta)}}^{-1}}$ 
of $G_{\dyg(\beta)}$
due to the tail
$z^{m_1} + z^{m_2}  
+ \ldots$ itself. For instance, 
for $\beta = 1.17628\ldots$ 
Lehmer's number (Table 1), $\dyg(\beta) = 12$,
$$f_{\beta}(z) = -1 + z + z^{12}
+ z^{31} + z^{44} + z^{63}
+ z^{86} + z^{105} + z^{118}+\ldots$$
is sparse with gaps of length
$\geq 10 = \dyg(\beta) -2$ and
$\lc_{\beta}$ is close to 
$\lc_{\theta_{12}^{-1}}$ (Fig. 1).

The passage from the Parry Upper function
\!\!$f_{\beta}(z)$ to the Mahler measure
\!\!${\rm M}(\beta)$ (when $\beta > 1$ is an 
algebraic integer) is crucial, constitutes
the main discoveries of the author,
and relies upon two facts:
(i) 
the \underline{discovery of 
lenticular distributions of zeroes}
of $f_{\beta}(z)$ in the annular
region 
$e^{-\lo \beta} = \frac{1}{\beta}  \leq |z| < 1$
which are very close to the
lenticular sets of zeroes of the trinomials
$G_{\dyg(\beta)}(z)$ of modulus $< 1$;
(ii) \underline{the identification} 
\underline{of these zeroes
as conjugates of $\beta$}. 
The quantity
$\lo \beta$ is the 
topological entropy of the $\beta$-shift.
These lenticular distributions of zeroes lie
in the cusp of the fractal of Solomyak
of the $\beta$-shift \cite{solomyak}
(recalled in $\S$ \ref{S4.2.2}).
The key ingredient 
for obtaining the Dobrowolski
type minoration of the Mahler measure
${\rm M}(\beta)$ in 
Theorem \ref{mainDOBROWOLSLItypetheorem}
relies upon the best possible
evaluation of the 
deformation of these lenticuli of zeroes
by the method of Rouch\'e
(in $\S$ \ref{S5}) and the coupling 
between the Rouch\'e conditions and
the asymptotic expansions of the 
lenticular zeroes.

The identification of the complete set of
conjugates of 
$\beta$ 
($\beta > 1$ being an algebraic integer)
seems to be 
unreachable by this method. Only
lenticular
conjugates of modulus $< 1$
can be identified
in an angular subsector
of $\arg(z) \in 
[-\frac{\pi}{3}, \frac{\pi}{3}]$
(Theorem \ref{splitBETAdivisibility+++} and 
Theorem \ref{divisibilityALPHA}). 
Consequently
the present method only gives access to
a ``{\em part}" of the Mahler measure itself.
We denote by
$\lc_{\beta}$, $\lc_{\theta_{\dyg(\beta)}^{-1}}$
the lenticular sets of zeroes of
$f_{\beta}(z)$, resp. of
$G_{\dyg(\beta)}(z)$. We call
$${\rm M}_{r}(\beta)
=
\prod_{\omega \in \lc_{\beta}} \, |\omega|^{-1}$$ 
the lenticular Mahler measure of $\beta$.
It satisfies ${\rm M}_{r}(\beta) \leq
{\rm M}(\beta)$.

We show that $\beta \to 
{\rm M}_{r}(\beta)$ is continuous
on each open interval
$(\theta_{n}^{-1}, \theta_{n-1}^{-1})$
for the usual topology,
and that it admits 
a lower bound which can be expanded
as an asymptotic expansion
of $\dyg(\beta)$
(Theorem \ref{Lrasymptotictheoremcomplexe}).
The general minorant of ${\rm M}(\beta)$
we are looking for to solve
the problem of Lehmer comes from
the asymptotic expansion of
the lower bound of ${\rm M}_{r}(\beta)$,
as in \eqref{lenticularminorationMr_intro}.

Given a nonzero algebraic integer
$\alpha$ which is not a root of unity,
with $|\alpha|$ close enough to $1^+$,
two new general 
notions appear in the present study:
\begin{enumerate}
\item[(i)] the {\em continuity}
properties of the lenticular Mahler measure
with the house of $\alpha$
(Theorem \ref{Qautocontinus}
and Remark \ref{measuremahlercontinuityAUXPerrons}) 
and the importance of 
the {\em dynamics} of
$\house{\alpha}$ in the identification 
of the conjugates,
\item[(ii)] the canonical {\it fracturability}
of the minimal polynomial 
$P_{\alpha}(z)$
as a product of two integer (arithmetic) 
series whose one 
is the Parry Upper function $f_{\beta}(z)$
at $\beta$ 
(Theorem \ref{splitBETAdivisibility+++} and 
Theorem \ref{divisibilityALPHA}).
\end{enumerate}
This {\em cleavability}
of the irreducible elements 
$P_{\alpha}(z)$
which are
the minimal polynomials
of the algebraic integers 
$\alpha$ close to
$1^+$ in modulus, obeying the 
{\em Carlson-Polya dichotomy} 
\cite{carlson}
\cite{polya}
in their canonical decomposition
(cf Theorem \ref{splitBETAdivisibility+++},
Theorem \ref{divisibilityALPHA} for the
definitions), as
\begin{equation}
\label{decompozzz_main}
P_{\beta}(z) = U_{\beta}(z) \times
f_{\beta}(z)
\quad
\left\{
\begin{array}{cc}
\mbox{on}~ \cb & 
\mbox{if $\beta$ is a Parry number, with}\\
& \mbox{$U_{\beta}$ and
$f_{\beta}$ as meromorphic functions},\\
&\\
\mbox{on}~ |z| < 1 & 
\mbox{if~} \beta \mbox{~is a
nonParry number, with $|z|=1$}\\
& \mbox{as natural boundary for both $U_{\beta}$ and $f_{\beta}$,}
\end{array}
\right. 
\end{equation}
seems to be new. It will probably call for
a refoundation of the theory of divisibility in
Commutative Algebra 
based on the theory of the
arithmetic power series coming from the
dynamics, from the positive, negative
\cite{nguemandong}
or generalized $\beta$-shift \cite{gora} \cite{thompson},
which are the Parry Upper functions. 
In the present note
we only use the (positive) $\beta$-shift.

Our main theorems are the following.

\begin{theorem}[ex-Lehmer conjecture]
\label{mainLEHMERtheorem}
For any nonzero algebraic integer
$\alpha$ which is not a root of unity,
$${\rm M}(\alpha) \geq \theta_{259}^{-1} = 1.016126\ldots$$
\end{theorem}
In terms of the Weil height $h$, 
this minoration 
is restated as:
\begin{equation}
\label{mainLEHMERTHEOREM_Weil}
h(\alpha) \geq \frac{\lo (\theta_{259}^{-1})}{\deg(\alpha)}.
\end{equation}

\begin{theorem}[ex-Schinzel-Zassenhaus conjecture]
\label{mainSCHINZELZASSENHAUStheorem}
Let $\alpha$ be a nonzero 
algebraic integer
which is not a root of unity.
Then
\begin{equation}
\label{kappoMain}
\house{\alpha} \geq 1 + \frac{c}{\deg(\alpha)}
\end{equation}
with 
$c = \theta_{259}^{-1} - 1 = 0.016126\ldots$.
\end{theorem}

The following definitions are given in $\S$ \ref{S5}.
We just report them here for stating 
Theorem \ref{mainDOBROWOLSLItypetheorem}.
Denote by $a_{\max} = 5.87433\ldots$ the abscissa
of the maximum of the function
$a \to \kappa(1,a) :=
\frac{1 - \exp(\frac{-\pi}{a})}{2 \exp(\frac{\pi}{a}) -1}$
on $(0, \infty)$ (Figure \ref{h1a}).
Let $\kappa:= \kappa(1, a_{\max}) = 0.171573\ldots$ 
be the value of 
the maximum. Let
$S:= 2 \arcsin(\kappa/2) = 0.171784\ldots$.
Denote
$$
\Lambda_{r}\mu_{r}
:=
\exp\Bigl(
\frac{-1}{\pi}
\int_{0}^{S}
\! \!\lo\Bigl[\frac{1 + 2 \sin(\frac{x}{2})
-
\sqrt{1 - 12 \sin(\frac{x}{2}) 
+ 4 (\sin(\frac{x}{2}))^2}}{4}
\Bigr] dx
\Bigr)$$
\begin{equation}
\label{lehmernumberdefinition}
 = 1.15411\ldots ,
\qquad \mbox{a value slightly below Lehmer's number
} 1.17628\ldots
\end{equation}
Recall that Lehmer's number is
the smallest Mahler measure ($> 1$)
of algebraic integers
known and the
smallest Salem number known
\cite{mossinghofflist}
\cite{mossinghoffrhinwu}, 
dominant root
of the degree 10 Lehmer's polynomial
\begin{equation}
\label{lehmerpolynomialdefinition}
X^{10}+X^9 -X^7 -X^6 -X^5 -X^4 -X^3 + X + 1.
\end{equation}
\begin{theorem}[Dobrowolski type minoration] 
\label{mainDOBROWOLSLItypetheorem}
Let $\alpha$ be a nonzero 
algebraic integer which is not a root of unity
such that
$\dyg(\alpha) 
\geq 260$. Then
\begin{equation}
\label{dodobrobro}
{\rm M}(\alpha) \geq 
\Lambda_r \mu_r \, 
-
\Lambda_r \mu_r \,\frac{S}{2 \pi}\, 
\Bigl(
\frac{1}{\lo (\dyg(\alpha))}
\Bigr)
\end{equation}
\end{theorem}
In terms of the Weil height $h$, 
using Theorem \ref{nfonctionBETA}, 
the asymptotics of the
minoration \eqref{dodobrobro}
takes the following form:
\begin{equation}
\label{dodobrobroWEIL}
\deg(\alpha) h(\alpha)  
~\geq~  
\lo (\Lambda_r \mu_r)
+
\frac{S}{2 \pi}\, 
\frac{1}{\lo (\house{\alpha}-1)} .
\end{equation}
The minoration
\eqref{dodobrobro} can also be 
restated
in terms of the usual degree.
Let $B > 0$.
Let us consider the subset 
$\mathcal{F}_B$ of 
all nonzero algebraic integers
$\alpha$ not being a root of unity
such that $\house{\alpha} < \theta_{259}^{-1}$
satisfying
$n := \deg(\alpha) \leq (\dyg(\alpha))^B$.  
Then
\begin{equation}
\label{dodobrobrousualdegree}
{\rm M}(\alpha) \geq 
\Lambda_r \mu_r \, 
-
\Lambda_r \mu_r \,\frac{S B}{2 \pi}\, 
\Bigl(
\frac{1}{\lo n}
\Bigr) , \qquad 
\alpha \in \mathcal{F}_B .
\end{equation}
Comparatively, in 1979, 
Dobrowolski \cite{dobrowolski2}, using an auxiliary function,
obtained the well-known asymptotic minoration
\begin{equation}
\label{dobrowolski79inequality}
{\rm M}(\alpha) > 1 + (1-\epsilon)
\left(
\frac{\lo \lo n}{\lo n}
\right)^3 , \qquad n > n_0,
\end{equation}
with 
$1-\epsilon$
replaced by $1/1200$ for 
$n \geq 2$, for an effective version of the minoration.
Here, in \eqref{dodobrobro} or
\eqref{dodobrobrousualdegree}, 
the constant in the minorant 
is not any more $1$
but $1.15411\ldots$
and the sign 
of the $n$-dependent term
becomes negative, with an appreciable
gain of
$(\lo n)^2$ in the denominator. 

The minoration 
\eqref{dodobrobro}
is general and admits a much better 
lower bound, in a similar
formulation, 
when $\alpha$ only runs over the set
of Perron numbers
$(\theta_{n}^{-1})_{n \geq 2}$. 
Indeed, in
\cite{vergergaugry6} ,
it is shown that
\begin{equation}
\label{minoVG2015__}
{\rm M}(\theta_{n}^{-1})
~>
~\Lambda - \frac{\Lambda}{6} \, 
\Bigl(\frac{1}{\lo n}\Bigr)  ~
, \qquad ~n \geq 2,
\end{equation}
holds with the following constant of the minorant 
\begin{equation}
\label{limitMahlGn}
\Lambda:=
{\rm exp}\bigl(
\frac{3 \, \sqrt{3}}{4 \, \pi} {\rm L}(2, \chi_3)\bigr) ~=~ 
{\rm exp}\Bigl(
\frac{-1}{\pi} \,
\int_{0}^{\pi/3} \lo \bigl(2 \, \sin\bigl(\frac{x}{2}\bigr) \bigr) dx 
\Bigr)
~=~ 1.38135\ldots,
\end{equation}
much higher than $1$ and $1.1541\ldots$,
and $L(s,\chi_3):= 
\sum_{m \geq 1} \frac{\chi_{3}(m)}{m^s}$
the Dirichlet L-series for the 
character $\chi_{3}$,
with $\chi_3$ the uniquely specified odd
character of conductor $3$
($\chi_{3}(m) = 0, 1$ or $-1$ according to 
whether $m \equiv 0, \,1$ or
$2 ~({\rm mod}~ 3)$, equivalently
$\chi_{3}(m)$
$ = \left(\frac{m}{3}\right)$ 
the Jacobi symbol).
Numerically the constant
$\Lambda_r \mu_r S/(2 \pi) =
0.0315536\ldots$ in \eqref{dodobrobro}   
is much smaller than
$\Lambda/6 = 0.230225\ldots$ in
\eqref{minoVG2015__}.
After Smyth's Conjecture 1.1.,
quoted in Flammang in \cite{flammang},
on height one integer
trinomials, 
the minoration
\eqref{minoVG2015__} 
admits the following conjectural
generalization, 
making sense, with a constant term in the minorant also equal to
$\Lambda$.

\begin{conjecture}
\label{minoTrinomeshauteurUN_CJ}
There exists a constant $\nu > 0$
such that, for $1 \leq k < n/2$, $\gcd(n, k)=1$,
\begin{equation}
\label{minoTrinomesHauteurUN}
{\rm M}(X^n \pm X^k \pm 1)
~>
~\Lambda - \nu \, 
\Bigl(\frac{1}{\lo n}\Bigr)  ~
, \qquad ~n \geq 2.
\end{equation}
\end{conjecture}
In the sense
of
Boyd-Lawton's limit 
Theorem \cite{boyd7} \cite{lawton}
${\rm M}(X^n \pm X^k \pm 1)$
is close to the bivariate Mahler measure
${\rm M}(y \pm x \pm 1)$.
Links between values of Dirichlet L-series 
$L'(\chi,s)$ ($\chi$ 
character) at algebraic numbers
and 
Mahler 
measures of multivariate integer polynomials
were discovered by
Smyth \cite{smyth5} in 1981, together with some
asymptotic formulas, followed by
Ray \cite{ray}. 
The minoration \eqref{minoTrinomesHauteurUN}
would strongly improve the
one deduced from
\cite{smyth5} or
from Boyd-Lawton's 
Theorem completed by the 
analytic asymptotics of
Condon \cite{condon} \cite{condon2} in terms of 
polylogarithms. 

Given 
an integer monic irreducible
univariate polynomial
$P(X)$ for which there exists
an integer $d$-variate polynomial
$Q(X_1, X_2, \ldots, X_d)$ such
that
${\rm M}(P)$ is close to
${\rm M}(Q)$ in the sense
of
Boyd-Lawton's  
Theorem \cite{condon} \cite{condon2} \cite{everest}, 
then, from Theorem 
\ref{mainDOBROWOLSLItypetheorem}
and Conjecture 
\ref{minoTrinomeshauteurUN_CJ}, the following
minoration is conjecturally expected  
\begin{equation}
\label{mainDOBROconjecture_BoydLawton}
{\rm M}(P) > \nu_0 - \nu
\left(
\frac{1}{\lo (\dyg(P))}
\right)
\end{equation}
where $\nu_0 = 
{\rm M}(Q)$ and, in which, 
if $\dyg(P)$ is replaced by
$\deg(P)$,
$\nu$ depends upon 
the ``limit polynomial" $Q$. 
What would be the set
of possible constants 
$\{\nu_0\}$?
This set belongs to the (mostly conjectural)
world of special values of
$L$-functions.
After Deninger \cite{deninger} showed in 1997, 
when $Q$ does not vanish on
$\tb^d$,
how to interpret
the logarithmic Mahler measures
$\lo M(Q)$
as Deligne periods of mixed motives,
dozans of formulas were guessed by
Boyd \cite{boyd16}, many remaining
unproved yet
\cite{boydrodriguezvillegas}
\cite{lalin}
\cite{lalin2}
\cite{rogers}
\cite{shindervlasenko}
\cite{zudilin}; 
Boyd's formulas link
these 
logarithmic Mahler measures to sums of
special values of different $L$-functions
and their derivatives,
i.e. to Dirichlet $L$-series
$L(\chi,s)$ ($\chi$ 
character) at algebraic numbers,
values of Hasse-Weil $L$-funtions 
$L(E,s)$ of 
elliptic curves $E/\qb$, etc.
Rodriguez-Villegas
in \cite{rodriguezvillegas}
showed how to correlate
Boyd's formulas to the
Bloch-Beilinson's conjectures,
investigating the domain of applicability
of these conjectures. 
Bornhorn \cite{bornhorn},
Standfest \cite{standfest}
continued the motivic approach
of Deninger to interpret 
Boyd's formulas in the light of 
Beilinson's conjectures.
In the same direction
Lal\'in \cite{lalin} \cite{lalin2}
continued developping
the ideas of Deninger
towards polylogarithmic motivic complexes
using Bloch groups, a
motivic cohomology of algebraic varieties and
Beilinson's regulator, in the multivariable case.

The Mahler measure
${\rm M}(G_n)$ of 
the trinomial
$G_n$ is equal to the lenticular Mahler measure
${\rm M}_{r}(G_n)$ itself, with limit
$\lim_{n \to +\infty} {\rm M}(G_n) 
= \lim_{n \to +\infty} {\rm M}_{r}(G_n)  
= \Lambda$, having
asymptotic expansion
\begin{equation}
\label{mahlermeasureGn_asympto}
{\rm M}(G_n) ~=~ \Lambda 
\Bigl( 1 + r(n) 
\, \frac{1}{\lo n}
+ O\left(\frac{\lo \lo n}{\lo n}\right)^2 \Bigr)
\end{equation}
with $r(n)$ real,
$|r(n)| \leq 1/6$. In the case of the trinomials
$G_n$ the characterization of the 
roots of modulus $< 1$ is direct 
($\S$\ref{S3}) and does not require
the detection method of Rouch\'e.
In the general case,
with $\beta \in 
(\theta_{n}^{-1}, \theta_{n-1}^{-1})$,
$n= \dyg(\beta)$ large enough,
applying the method of Rouch\'e
only leads to
the following asymptotic lower bound of 
the lenticular minorant,
similarly
as in \eqref{mahlermeasureGn_asympto},
as ($\S$\ref{S6.2}):
\begin{equation}
\label{lenticularminorationMr_intro}
{\rm M}_{r}(\beta) \geq\Lambda_r \mu_r 
(1 + 
\frac{\rc}{\lo n}
+
O\bigl(
\bigl(
\frac{\lo \lo n}{\lo n}
\bigr)^2
\bigr),
\quad
\mbox{with}~~ |\rc| < \frac{\arcsin(\kappa/2)}{\pi}.
\end{equation}

Denote by 
$M_{\inf} := \liminf_{\house{\alpha} \to 1^{+}} 
{\rm M}(\alpha)$ 
the limit infimum
of the Mahler measures ${\rm M}(\alpha)$,
$\alpha \in \mathcal{O}_{\overline{\qb}}$,
when $\house{\alpha} > 1$ tends to $1^+$.
Then
\begin{equation}
\label{limitinfimumMahlermeasure}
\Lambda_r \mu_r \leq M_{\inf} \leq \Lambda.
\end{equation}
Whatever the expression of the constant terms
$\nu_0$ as sums of special values of different 
$L$-functions,
because the $\beta$-shift used in the present
approach for the dynamization of the 
algebraic equations 
($\S$ \ref{S4}) is compact,
we formulate a continuum for the
set of the constant terms of the minorants
in \eqref{mainDOBROconjecture_BoydLawton},
as follows.

\begin{conjecture}
For any $\nu_0 \in [M_{\inf},  \Lambda)$ there exists
a sequence of integer monic irreducible 
polynomials $(H_{m}(z))_{m}$ such that
$\lim_{m \to +\infty} {\rm M}(H_{m}) =\nu_0$.
\end{conjecture}

Lenticuli of conjugates lie in the cusp 
of Solomyak's
fractal ($\S$ \ref{S4.2.2}). The number
of elements of a lenticulus
$\lc_{\alpha}$ is an increasing
function of the dynamical degree
$\dyg(\alpha)$ as soon as 
$\dyg(\alpha)$ is large enough.
The existence of lenticuli composed
of three elements only
(one real, a pair of nonreal complex-conjugated
conjugates) 
is studied carefully
in $\S$ \ref{S7.1}. Such lenticuli appear
at small dynamical degrees. Since Salem
numbers have no nonreal complex conjugate
of modulus $< 1$, they should not possess
3-elements lenticuli of conjugates, therefore
they should possess
a small dynamical degree bounded from above. 
We obtain  31 as an upper bound as follows. 

\begin{theorem}[ex-Lehmer conjecture for Salem numbers]
\label{mainpetitSALEM}
Let $T$ denote the set 
of Salem numbers. Then $T$ is bounded from below:
$$\beta \in T
\qquad
\Longrightarrow
\qquad
\beta > 
\theta_{31}^{-1} = 1.08544\ldots$$
\end{theorem}
Lehmer's number $1.17628\ldots$ 
belongs to the interval 
$(\theta_{12}^{-1}, \theta_{11}^{-1})$
(Table 1).
This interval does not contain any other 
known Salem number. If there is another one,
its degree should be greater than 44 
\cite{mossinghoff}
\cite{mossinghoffrhinwu}. 
\begin{conjecture}
There is no Salem number in the interval
$$(\theta_{31}^{-1}, \theta_{12}^{-1})
= (1.08544\ldots, 1.17295\ldots).$$
\end{conjecture}
Parry Upper functions $f_{\beta}(z)$, 
with $\beta$ being an algebraic integer
of dynamical degree
$\dyg(\beta)$
$= 12$ to $16$,
do possess
3-elements lenticuli of zeroes
in the open unit disc (as in Fig. 1).
Conjecture 1 means that there should exist 
a method
better than Rouch\'e's method to identify
these 3-elements lenticuli of zeroes
as 3-elements lenticuli of conjugates
of $\beta$.

Our results give another proof
that the field of totally real algebraic numbers
$\qb^{tr}$
has the Bogomolov property
relative to the Weil height $h$ ($\S $ \ref{S7.3}),
as follows.

\begin{theorem}
\label{main_hminoration_totallyreal}
Let $\qb^{tr}$ denote the field of all
totally real algebraic numbers.
If $h$ denotes the absolute logarithmic Weil height,
\begin{equation}
\label{hALPHATotallyRealmini}
\alpha \in \qb^{tr}, \alpha \neq 0, \neq \pm 1 
~~\Rightarrow~~
h(\alpha) > \frac{1}{4} \,\lo \theta_{31}^{-1}
= 0.020498\ldots
\end{equation}
\end{theorem}
Since Fili and Miner 
\cite{filiminer3}
\cite{filiminer4}
proved that
$$\liminf_{\alpha \in 
\qb^{tr},\alpha \neq 0,\pm 1} h(\alpha) > 
0.120786\ldots,$$ 
the interval
$(0.020498, 0.120786)$ only 
contains a finite number
of values of $h(\alpha)$
(isolated values). Eventually this number is zero.
What are they?
What are the corresponding
totally real algebraic numbers
$\alpha$?

Though Lehmer's problem arises from
lenticuli of conjugates in
angular sectors containing 1,
the complete set of conjugates
equidistribute on the unit circle at the limit,
once
the Mahler measures are small enough, as follows.

\begin{theorem}
\label{main_EquidistributionLimitethm}
Let $(\alpha_{q})_{q \geq 1}$ be
a sequence of nonzero 
algebraic integers which are not roots of unity
such that
${\rm M}(\alpha_q) < \Theta$,
$\dyg(\alpha_q) \geq 260$
for $q \geq 1$,
with
$\lim_{q \to +\infty} 
\house{\alpha_q} = 1$.
Then the sequence 
$(\alpha_{q})_{q \geq 1}$
is strict and
\begin{equation}
\label{haarmeasurelimite}
\mu_{\alpha_q} 
~\to~ 
\mu_{\tb},
\qquad 
\dyg(\alpha_q) \to + \infty,
\qquad
\mbox{weakly},
\end{equation}
i.e. for all bounded, continuous
functions
$f: \cb^{\times} \to \cb$,
\begin{equation}
\label{haarlimitefunvtions}
\int f d \mu_{\alpha_q}
\to
\int f d \mu_{\tb},
\qquad
\dyg(\alpha_q) \to + \infty.
\end{equation}
\end{theorem}
Theorem 
\ref{main_EquidistributionLimitethm}
can also be considered as an important theorem
for the understanding of the limit
distribution
of roots of families of random polynomials,
of small Mahler measure,  \cite{hughesnikeghbali}
\cite{sinclairyattselev}.

Following work of 
Langevin \cite{langevin} \cite{langevin2}
\cite{langevin3}
\cite{mignotte3}
the importance of the angular
sectors containing
the point 1 was already 
guessed by Dubickas and Smyth \cite{dubickassmyth},
Rhin, Smyth and Wu 
\cite{rhinsmyth} \cite{rhinwu}.
The lenticular probability distribution
of zeroes of the Parry Upper functions 
$f_{\beta}(z)$ in 
this sector, which are
identified as Galois conjugates
of $\beta$, 
admits the uniform limit
$\mu_{\tb, arc}$
which denotes the restriction of
the (normalized) Haar measure
$\mu_{\tb}$ 
to the limit
arc (for the weak topology)
$$\mbox{$\{z \mid |z| = 1, \arg(z)
\in 
[-2 \arcsin 
\bigl( \frac{\kappa}{2} \bigr),
+2 \arcsin 
\bigl( \frac{\kappa}{2} \bigr)]\}$
on the unit circle}$$ 
(Proposition \ref{argumentlastrootJn},
Remark \ref{openingangle_sin_quadratic_alginteger} (i),
Theorem \ref{omegajnexistence},
\cite{vergergaugry6} Theorem 6.2).

The very nature of the set of 
Parry numbers 
(Definition 
\ref{parrynumberdefinition})
is a deep question addressed to
the dynamics of Perron numbers
(Adler and Marcus \cite{adlermarcus},
Bertrand-Mathis \cite{bertrandmathis},
Boyd \cite{boyd7},
Boyle and Handelman \cite{boyle2} 
\cite{boylehandelman}
\cite{boylehandelman2},
Brunotte \cite{brunotte},
Calegari and Huang \cite{calegarihuang},
Dubickas and Sha \cite{dubickassha},
Lind \cite{lind} \cite{lind2},
Lind and Marcus \cite{lindmarcus},
Thurston \cite{thurston2},
Verger-Gaugry \cite{vergergaugry2}
\cite{vergergaugry3}), 
associated with the rationality
of the dynamical zeta function 
of the $\beta$-shift 
\begin{equation}
\label{dynamicalfunction}
\zeta_{\beta}(z) := \exp\left(
\sum_{n=1}^{\infty} \, \frac{\mathcal{P}_n}{n} \, z^n
\right),\quad
\mathcal{P}_n :=\#\{x \in [0,1] \mid 
T_{\beta}^{n}(x) = x\}\end{equation}
counting the number of periodic 
points of period dividing $n$ (isolated points).
For $\alpha$ a nonzero algebraic integer
which is not a root of unity, with
$\beta = \house{\alpha}$,
by Theorem \ref{parryupperdynamicalzeta},
$$\beta \quad 
\mbox{is a Parry number}
\quad \mbox{iff}\quad 
\zeta_{\beta}(z) \quad \mbox{is a rational function};$$
and
$|z|=1$ is the natural boundary of
the domain of fracturability of 
the minimal polynomial 
$P_{\alpha}$ if and only if 
$\beta$ is not a Parry number
(Theorem \ref{divisibilityALPHA}),
as soon as $|\alpha|$ is close enough to 1, 
in the Carlson-Polya dichotomy.
Comparatively,
for complete nonsingular projective 
algebraic varieties
$X$ over the field of
$q$ elements, $q$ a prime power, 
the zeta function $\zeta_{X}(t)$
introduced by 
Weil \cite{weil} is only
a rational function
(Dwork \cite{dwork},
Tao \cite{tao}):
the first Weil's conjecture,
for which there exists a set of 
characteristic values was proved by Dwork
using $p$-adic methods
(Dwork \cite{dwork}),
and ``Weil II", the Riemann hypothesis, 
proved by Deligne using 
$l$-adic \'etale cohomology in 
characteristic $p \neq l$ 
(Deligne \cite{deligne}). 
It is defined as a dynamical zeta 
function with the action of the Frobenius.
The purely $p$-adic methods of Dwork
(Dwork \cite{dwork}), continued by
Kedlaya \cite{kedlaya} for ``Weil II"
in the need of numerically computing zeta functions
by explicit equations, 
allow an intrinsic computability,
as in Lauder and Wan \cite{lauderwan}, towards
a $p$-adic cohomology theory,
are linked to ``extrinsic geometry", to
the defining equations of the variety
itself.
They are 
in contrast with the relative version of crystalline 
cohomology developped by
Faltings 
\cite{faltings}, or
the Monsky-Washnitzer
constructions used by Lubkin \cite{lubkin}.
We refer the reader to Robba \cite{robba},
Kedlaya \cite{kedlaya}, Tao \cite{tao},
for a short survey on other developments.

After Weil \cite{weil}, and introduced
in general terms by 
Artin and Mazur in \cite{artinmazur},
the theory of dynamical zeta functions 
$\zeta(z)$ associated with dynamical systems,
based on an analogy with the number theory zeta functions,
developped under the impulsion
of Ruelle \cite{ruelle4}
in the direction of the thermodynamic formalism
and with Pollicott, Baladi and Keller \cite{baladikeller}
towards transfer operators and counting orbits
\cite{parrypollicott}.
The determination and the
existence of meromorphic extensions
or/and natural boundaries of 
dynamical zeta functions
is a deep problem.

In the present
proof of the conjecture of Lehmer,
the analytic extension of the dynamical 
zeta function of the $\beta$-shift
behaves as an analogue
of Weil's zeta function
(in the sense that both
are dynamical zeta functions).
But it generates questions 
beyond the analogues of
Weil's conjectures
since not only the rational case of
$\zeta_{\beta}$
contribute to the minoration
of the Mahler measure, but also
the nonrational case
with the unit circle
as natural boundary 
and lenticular poles arbitrarily close to it.
The equivalent of ``Weil II" (Riemann Hypothesis)
would be
the determination of the
geometry of the beta-conjugates
in the rationality case. 
Beta-conjugates
are zeroes of Parry polynomials, 
whose factorization was studied in the context of 
the theory
of Pinner and Vaaler \cite{pinnervaaler}
in \cite{vergergaugry3}.

An apparent difficulty for the 
computation of the minorant of
${\rm M}(\alpha)$
comes from the absence of 
complete characterization 
of the set of Parry numbers 
$\pb_P$, 
when $\beta = \house{\alpha}$ 
is close to one, since we never know
whether $\beta$ is a Parry number or not.
But the Mahler measure M$(\alpha)$
is independent of 
the Carlson-Polya dichotomy.
Indeed,
the two domains of definitions
of $\zeta_{\beta}$, 
``$\cb$" and ``$|z| < 1$", 
together with the corresponding splitting
 \eqref{decompozzz_main}, may occur
fairly ``randomly" 
when $\beta$ tends to one.
But $\{|z| < 1\}$ 
is a domain included in both,
M$(\alpha)$ ``reading" only the roots
in it and not taking care of the ``status" of 
the unit circle.
Whether $f_{\beta}(z)$ can be 
continued analytically or not
beyond
the unit circle has no 
effect on the value of the Mahler measure
${\rm M}(\alpha)$.

Many consequences of the above Theorems
can be readily
deduced. Let us mention
a few of them below and in 
$\S$ \ref{S9}. 
A first consequence concerns
the theory of heights on
$\overline{\qb}^{\times}$ written multiplicatively
\cite{allcockvaaler}:
the metric, $t$-metric and ultrametric Mahler measures
 \cite{dubickassmyth} 
\cite{samuels}
\cite{filisamuels}
induce the {\em discrete topology
on $\overline{\qb}^{\times}
/
{\rm Tor}(\overline{\qb}^{\times})$}, with
its consequences.

A second consequence concerns
the difference between two
successive Salem numbers.
In the context of
root separation theorems 
\cite{beresnevichbernikgotze}
\cite{bugeaudmignotte}
\cite{evertse}
\cite{guting}
\cite{mahler2}
and the representability of real
algebraic numbers 
as a difference of two Mahler measures
\cite{drungilasdubickas}
the difference between two
successive Salem numbers
of the same degree
(in particular when the degree 
is very large)
admits the following universal lower bound, 
readily
deduced from Lemma 4 in \cite{smyth6}
and 
Theorem \ref{mainpetitSALEM}.

\begin{theorem}
\label{main_coro_salemnumbersSAMEdegree}
Let $d \geq 4$ be an integer.
Denote by $T_{(d)}$ the subset of $T$ of the Salem numbers
of degree $d$. Then,  
$$\tau, \tau' \in T_{(d)}, \quad \tau' > \tau~~
\Longrightarrow ~~\tau' - \tau \geq 
\theta_{31}^{-1} (\theta_{31}^{-1} - 1)
= 0.0927512\ldots$$
\end{theorem}
Distributions of algebraic numbers in
very small neighbourhoods of Salem numbers
can be studied by the algebraic coding with
Stieltjes continued fractions 
\cite{guichardvergergaugry}.
In higher dimension (cf $\S$\ref{S2.2}), 
amongst other results, 
let us mention 
the following consequence which improves
Laurent's Theorem \ref{laurentthm} \cite{laurent}.
\begin{theorem}[ex-Elliptic Lehmer Conjecture]
Let $E/K$ be an elliptic curve over a 
number field $K$ and $\widehat{h}$ the 
N\'eron-Tate height
on $E(\overline{K})$. There is a positive constant
$c(E/K)$ such that
\begin{equation}
\widehat{h}(P) \geq 
\frac{c(E/K)}{[K(P) : K]}
\qquad \mbox{for all}~ 
P \in E(\overline{K}) \setminus E_{{\rm tors}}.
\end{equation}
\end{theorem}

The paper is self-contained and
organized as follows:
in $\S$\ref{S2}
we adopt an interdisciplinary viewpoint to
evoke the meaning of the
Conjecture of Lehmer in different contexts
together with the analogues of the Mahler measure
and their geometrical counterparts.
In $\S$\ref{S4} we give
a semi-expository presentation 
of the analytic functions involved
in the dynamics of the $\beta$-shift, 
$\beta > 1$, in particular
of the Parry Upper function, 
useful for the sections
$\S$\ref{S5} to $\S$\ref{S7}.
The proofs of Lehmer's Conjecture 
and Schinzel-Zassenhauss's Conjecture
are obtained in $\S$\ref{S5} 
when $\beta$ is a real algebraic integer
$> 1$, and in $\S$\ref{S6}.
They are formulated in $\S$\ref{S6} 
in a general setting for 
any nonzero algebraic integer $\alpha$ 
which is not a root of unity, 
(i) which does not belong to $(1, \infty)$ or
(ii) for which, if $\alpha > 1$, 
 $\alpha \neq \house{\alpha}
=: \beta $.
In $\S$\ref{S7} a Theorem of fracturability
of the minimal integer polynomials 
of small Mahler measure is obtained
which readily 
implies the Conjecture of Lehmer 
for Salem numbers, 
the Bogomolov property for totally 
real algebraic numbers.
In $\S$\ref{S3} are presented
the ``\`a la Poincar\'e" 
technics of asymptotic expansions
transposed (and adapted) to Number Theory
from Celestial Mechanics; they were 
introduced for the first time
in \cite{vergergaugry6}
on the height one 
trinomials $G_n$ and $G_{n}^{*}$, 
their roots, and their Mahler measures.
These asymptotic expansions
are fundamental for dealing with 
the general cases of algebraic integers in
$\S$\ref{S5}, $\S$\ref{S6}
and $\S$\ref{S7}.
Associated limit equidistribution results are obtained
in $\S$\ref{S8}. Theorems proving the equivalence
of the Conjecture
of Lehmer with some other 
statements in Geometry are gathered 
in $\S$\ref{S9}.

\begin{notation}
Let $P(X) \in \zb[X]$,
$m=\deg(P) \geq 1$.
The {\it reciprocal polynomial} of $P(X)$ 
is $P^*(X)=X^m P(\frac{1}{X})$.
The polynomial $P$ is reciprocal if $P^*(X)=P(X)$.
When it is monic,
the polynomial $P$ is said
{\it unramified} if
$|P(1)|=|P(-1)|=1$.
If $P(X) = a_0 \prod_{j=1}^{m} (X- \alpha_j) 
= a_0 X^m + a_1 X^{m-1}+ \ldots +a_m$, with
$a_i \in \cb$, $a_0 a_m \neq 0$,
and roots $\alpha_j$, the {\it Mahler measure} of $P$ is
\begin{equation}
\label{mahlermeasuredefinition}
{\rm M}(P) := |a_0| \prod_{j=1}^{m} \max\{1, |\alpha_j|\}. 
\end{equation}
The absolute Mahler measure
of $P$ is 
${\rm M}(P)^{1/\deg(P)}$,
denoted by
$\mathcal{M}(P)$.
The Mahler measure of an algebraic number $\alpha$ 
is the Mahler of its minimal polynomial
$P_{\alpha}$: ${\rm M}(\alpha) := {\rm M}(P_{\alpha})$.
For any algebraic number
$\alpha$ the house $\house{\alpha}$ of
$\alpha$ is the maximum modulus of its conjugates, 
including $\alpha$ itself; 
by Jensen's formula the Weil height $h(\alpha)$ of $\alpha$
is $\lo {\rm M}(\alpha)/\deg (\alpha)$.
By its very definition, ${\rm M}(P Q) =
{\rm M}(P) {\rm M}(Q)$ (multiplicativity).

A {\it Perron number} is either $1$ or 
a real algebraic integer $\theta > 1$
such that the Galois conjugates
$\theta^{(i)}, i \neq 0$, 
of $\theta^{(0)} := \theta$ satisfy: 
$|\theta^{(i)}| < \theta$. 
Denote by $\mathbb{P}$ the set of Perron numbers. 
A {\it Pisot number} is a Perron number $> 1$ for which 
$|\theta^{(i)}| < 1$ for all $i \neq 0$.
The smallest Pisot number
is denoted by
$\Theta = 1.3247\ldots$, dominant root of
$X^3 -X-1$.
A Salem number is
an algebraic integer $\beta > 1$ such that its Galois conjugates
$\beta^{(i)}$ satisfy:
$|\beta^{(i)}| \leq 1$ for all $i=1, 2, \ldots, m-1$,
with $m = {\rm deg}(\beta) \geq 1$,
$\beta^{(0)} = \beta$ and at least one conjugate 
$\beta^{(i)}, i \neq 0$,
on the unit circle.  
All the 
Galois conjugates 
of a Salem number $\beta$ lie on the unit circle, 
by pairs of complex conjugates, except
$1/\beta$ which lies in the open interval $(0,1)$.
Salem numbers are of even degree $m \geq 4$.
The set of Pisot numbers, resp. Salem numbers, 
is denoted by S, resp. by T. 
If $\tau \in$ S or T, then
${\rm M}(\tau) = \tau$.
A $j$-Salem number \cite{kerada} \cite{samet},
$j \geq 1$,
is an algebraic integer  $\alpha$
such that $|\alpha| > 1$ and $\alpha$
has $j-1$ conjugate roots $\alpha^{(q)}$ 
different from $\alpha$,
satisfying $|\alpha^{(q)}| > 1$, 
while the other conjugate
roots $\omega$ satisfy $|\omega| \leq 1$ 
and at least one of them is on the unit circle. 
We call the minimal
polynomial of  a $j$-Salem number a $j$-Salem
polynomial. Salem numbers are 
$1$-Salem numbers.
A Salem number is said {\it unramified} if
its minimal polynomial is unramified.
We say that two Salem numbers
$\lambda$ and $\mu$ are {\em commensurable}
if there exists positive integers
$k$ and $l$ such that
$\lambda^k = \mu^l$. Commensurability is an
equivalence relation on $T$.
Let $\lambda \in T$,
$K$ a subfield of
$\qb(\lambda + \lambda^{-1})$, and
$P_{\lambda, K}$ the minimal polynomial 
of $\lambda$ over $K$; 
we say that
$\lambda$ is {\em square-rootable} 
over $K$ if
there exists a totally positive element
$\alpha \in K$ and a monic reciprocal polynomial
$q(x)$, whose even degree coefficients
are in $K$ and odd degree coefficients
are in $\sqrt{\alpha}K$ such that
$q(x) q(-x) = P_{\lambda, K}(x^2)$.

The set of algebraic numbers, resp. 
algebraic integers, in $\cb$, 
is denoted by
$\overline{\qb}$, resp.
$\mathcal{O}_{\overline{\qb}}$.
The $n$th cyclotomic polynomial is denoted by
$\Phi_{n}(z)$.
For any postive integer $n$, let $[n] :=
1 + x + x^2 + \ldots + x^{n-1}$.
The (na\"ive) height
of a polynomial $P$
is the maximum of the absolute value
of the coefficients of $P$.
Let $A$ be a countable subset of the line;
the {\it first derived set} $A^{(1)}$
of $A$
is 
the set of the limit points of
nonstationary infinite sequences of elements of
$A$; the {\it $k$-th derived set} $A^{(k)}$
of $A$ is the first derived set 
of $A^{(k-1)}$, $k \geq 2$. 

For $x > 0$, $\lfloor x\rfloor$, $\{x\}$ and
$\lceil x \rceil$ denotes respectively the integer
part, resp. the fractional part,
resp. the smallest integer greater than or equal to $x$.
For $\beta > 1$ any real number, 
the map
$T_{\beta}: [0,1] \to [0,1],
x \to \{\beta x\}$ denotes the $\beta$-transformation.
With
$T_{\beta}^{0} := T_{\beta}$, its
iterates are denoted by 
$T_{\beta}^{(j)} := 
T_{\beta}(T_{\beta}^{j-1})$
for $j \geq 1$. 
A real number $\beta > 1$ 
is a Parry number if 
the sequence
$(T_{\beta}^{(j)}(1))_{j \geq 1}$ 
is eventually periodic;
a Parry number is called simple if in particular 
$T_{\beta}^{(q)}(1) = 0$ for some integer
$q \geq 1$. 
The set of Parry numbers is denoted by
$\pb_{P}$.
The terminology chosen by Parry in \cite{parry} 
has changed: $\beta$-numbers are now 
called Parry numbers, in honor of W. Parry.

The Mahler measure of a nonzero polynomial
$P(x_1, \ldots, x_n) \in \cb[x_1, \ldots, x_n]$
is defined by
$${\rm M}(P)
:=
\exp\left(
\frac{1}{(2 i \pi)^n}
\int_{\tb^n}
\lo |P(x_1, \ldots, x_n)|
\frac{d x_1}{x_1}\ldots \frac{d x_n}{x_n} 
\right),
$$
where $\tb^n = \{(z_1,\ldots, z_n)
\in \cb^n \mid |z_1| = \ldots = |z_n|=1\}$ is the unit torus 
in dimension $n$. If $n=1$, 
by Jensen's formula, it is given by
\eqref{mahlermeasuredefinition}.
A function $f: \rb \to \rb$ is said 
quasiperiodic if it is the sum of finitely many
periodic continuous functions.
The function, defined for $k \geq 2$,
${\rm Li}_{k}(z) = 
\sum_{n=1}^{\infty} \frac{z^n}{n^k}, ~|z| \leq 1,$
is the $k$th-polylogarithm function 
\cite{ganglzagier}
\cite{lewin} \cite{zagier3}.
For $x > 0$,
$\lo^{+} x$ denotes $\max\{0, \lo x\}$.
Let $\mathcal{F}$ be an infinite subset
of the set of nonzero
algebraic numbers which are not a root of unity;
we say that the {\em Conjecture of Lehmer is true
for $\mathcal{F}$} 
if there exists
a constant $c_{\mathcal{F}} > 0$ such that
${\rm M}(\alpha) \geq 1 + c_{\mathcal{F}}$
for all $\alpha \in \mathcal{F}$.
\end{notation}


\section{Lehmer's conjecture and Schinzel-Zassenhaus's conjecture: in different contexts}
\label{S2}
 
The problem of Lehmer raised up
many works
in number theory, and also
in
other domains where questions of minimality
of analogues of
the Mahler measure
intervene. 
The questions related to totally real 
algebraic numbers, 
with minimality of the
Weil height, are
reported in $\S$ \ref{S7.3}.
The dynamics of algebraic numbers 
and its ergodic properties,
related 
to the conjecture of Lehmer,
are reported in $\S$ \ref{S4.3},
after recalling
the $\beta$-shift in $\S$ \ref{S4.1} and presenting
the properties of Parry Upper functions in
$\S$ \ref{S4.2}.
 
\subsection{Number theory: prime numbers, asymptotic expansions, minorations, limit points}
\label{S2.1}

The search for very large prime numbers 
has a long history. The
method of linear recurrence sequences of numbers
$(\Delta_m)$,
typically satisfying
\begin{equation}
\label{recurrencelinear}
\Delta_{m+n+1} = 
A_1 \Delta_{m+1}
+
A_2 \Delta_{m+2} 
+
\ldots
+
A_n \Delta_{m+n},
\end{equation}
in which
prime numbers can be found, 
has been  
investigated from several viewpoints, 
by many authors 
\cite{biluhanrotvoutier}
\cite{evereststevenstamsettward}
\cite{everestvanderpoortenshparlinskiward}:
in 1933 Lehmer \cite{lehmer} developped an exhaustive approach from the Pierce numbers
\cite{pierce} 
$\Delta_n = \Delta_{n}(P) =
\prod_{i=1}^{d} (\alpha_{i}^{n} - 1)$
of a monic integer polynomial $P$
where $\alpha_i$  are the roots of $P$.
The sequence $(A_i)$ in \eqref{recurrencelinear} 
is then
the coefficient vector of the integer
monic polynomial which is the least common multiple
of the $d+1$ polynomials: 
$P_{(0)}(x) = x-1,  P_{(1)}(x) = 
\prod_{i=1}^{d} (x-\alpha_i),
P_{(2)}(x) = 
\prod_{i>j=1}^{d-1} (x-\alpha_i \alpha_j),
\ldots,
P_{(d)}(x) = 
x-\alpha_1 \alpha_2 \ldots \alpha_d$
(Theorem 13 in \cite{lehmer}).
Large prime numbers, possibly at a certain power,
can be found in the factorizations of
$|\Delta_{n}|$ that have large absolute values
(Dubickas \cite{dubickas11}, Ji and Qin
\cite{jiqin} in connection with Iwasawa theory).
This can be done fairly quickly if the
absolute values $|\Delta_n|$ do not 
increase too rapidly (slow growth rate).
If $P$ has no root on the unit circle, Lehmer 
proves
\begin{equation}
\label{DeltanMahlermasure}
\lim_{n \to \infty} \frac{\Delta_{n+1}}{\Delta_{n}}
= {\rm M}(P).
\end{equation} 
Einsiedler, Everest and Ward 
\cite{einsiedlereverestward}
revisited and extended the results of
Lehmer in terms of the dynamics of toral automorphisms
(\cite{everestward},
Lind \cite{lind}).
They considered expansive (no root on $|z|=1$), 
ergodic (no $\alpha_i$ is a root of unity)
and quasihyperbolic (if $P$ is ergodic but
not expansive) polynomials $P$
and number theoretic
heuristic arguments for estimating 
densities of primes in $(\Delta_n)$. 
In the quasihyperbolic case (for instance 
for irreducible 
Salem polynomials $P$), more general than
the expansive case considered by Lehmer,
\eqref{DeltanMahlermasure} does not 
extend but the following more robust
convergence law holds \cite{lind}:
\begin{equation}
\label{DeltanMahlermasure1}
\lim_{n \to \infty} \Delta_{n}^{1/n}
= {\rm M}(P).
\end{equation}
If $P$ has a small Mahler measure,
$< \Theta$, it is reciprocal
by \cite{smyth} and the quotients
$\Delta_{n}/\Delta_{1}$ are
perfect squares for all $n \geq 1$ odd.
With $\Gamma_{n}(P) := \sqrt{\Delta_{n}/\Delta_{1}}$
in such cases, they obtain the existence of the limit
$$\lim_{j \to \infty}
\frac{j}{\lo \lo \Gamma_{n_j}},$$
$(n_j)$ being a sequence of integers for which
$\Gamma_{n_{j}}$ is prime, 
as a consequence of Merten's Theorem.
This limit, say $E_P$, is likely to
satisfy the inequality:
$E_P \geq 2 e^{\gamma}/\lo {\rm M}(P)$, where  
$\gamma = 0.577\ldots$ is the
Euler constant. Moreover, 
by number-fields analogues
of the heuritics for Mersenne numbers 
(Wagstaff, Caldwell), they suggest that
the number of prime
values of $\Gamma_{n_{j}}(P)$
with $n_j \leq x$ is
approximately
$$\frac{2 e^{\gamma}}{\lo {\rm M}(P)} \lo x.$$
This result shows the interest of having a
polynomial $P$ of small Mahler measure 
to obtain a
sequence $(\Delta_n)$ associated
with $P$ very rich in primes.
These authors consider many examples
which fit coherently the heuristics.
However, the discrepancy function is still obscure
and reflects the deep arithmetics of
the factorization of
the integers 
$|\Delta_{n}|$ and of the
quantities $\Gamma_{n}$.
Let us recall the {\em problem of Lehmer}, as
mentioned by
Lehmer \cite{lehmer} in 1933: 
{\it if $\epsilon$ is a positive quantity, to find a polynomial of the form
$$f(x) = x^r + a_1 x^{r-1} + \ldots + a_r$$
where the $a_i$s are integers, such that the absolute value of the product of those roots of $f$ which lie
outside the unit circle, lies between $1$ and $1+\epsilon$... Whether 
or not the problem has a solution for $\epsilon < 0.176$ we do not know}.

In general ($\S$ 7.6
in \cite{everestvanderpoortenshparlinskiward},
\cite{everestward2}), 
the Conjecture of Lehmer means that 
the growth rate 
of an integer linear recurrence sequence
is uniformly bounded from below.

In view of understanding the size
of the primes $p \geq 3$
found in $(\Delta_n)$ 
generated by the exhaustive 
method of the
Pierce's numbers,
Lehmer, in \cite{lehmer2} (1977), 
established correlations between the Pierce's numbers $|\Delta_n|$
and the prime factors of the 
first factor of the class number 
of the cyclotomic fields
$\qb(\xi_p)$ ($\xi_p$ is a primitive 
$p$th root of unity), using 
Kummer's formula, 
the prime factors being sorted out into 
arithmetic progressions: 
let $h(p)$ the class number of
$\qb(\xi_p)$ and let $h^{+}(p)$ be the 
class number of of the real subfield
$\qb(\xi_p + \xi_{p}^{-1})$.
Kummer (1851) established that the ratio
$h^{-}(p) = h(p)/h^{+}(p)$ 
is an integer, called
{\em relative class number} 
or {\em first factor of the class number},
and that
$p$ divides $h(p)$ if and only if
$p$ divides $h^{-}(p)$.
The factorization and the arithmetics
of the large values of $h^{-}(p)$
is a deep problem 
\cite{agoh}
\cite{funggranvillewilliams}
\cite{granville}
\cite{lehmermasley}, related to class field theory
in \cite{masley},
where
the validity of Kummer's conjectured 
asymptotic formula for 
$h^{-}(p)$ was reconsidered by
Granville \cite{debaenne}
\cite{granville}.

The smallest Mahler measure ${\rm M}(\alpha)$
known, where 
$\alpha$ is a nonzero algebraic number 
which is not a root of unity,
is Lehmer's number
$= 1.17628\ldots$
\eqref{lehmernumberdefinition},
the smallest Salem number discovered by Lehmer 
\cite{lehmer} in 1933 as dominant root of
Lehmer's polynomial 
\eqref{lehmerpolynomialdefinition}.
Lehmer also discovered other small 
Salem numbers.
Small Salem numbers were reinvestigated by
Boyd in \cite{boyd2}
\cite{boyd4}
\cite{boyd6}, then by
Flammang, Grandcolas and Rhin 
\cite{flammanggrandcolasrhin}.
The search of small Mahler measures
was reconsidered by 
Mossinghoff \cite{mossinghoff}
\cite{mossinghoff2}
then, using auxiliary functions, by
Mossinghoff, Rhin 
and Wu \cite{mossinghoffrhinwu}.
For degrees up to 180,
the list of Mossinghoff
\cite{mossinghofflist} (2001),
with
contributions of Boyd, Flammang, 
Grandcolas, Lisonek,
Poulet, Rhin and Sac-Ep\'ee
\cite{rhinsacepee}, Smyth,
gives primitive, irreducible, noncyclotomic
integer polynomials of degree at most 180 and of Mahler measure less than 1.3; this list is complete for degrees less than 40 
\cite{mossinghoffrhinwu}, and, for Salem numbers, 
contains the list of the 
47 known smallest Salem numbers, all of degree
$\leq 44$
 \cite{flammanggrandcolasrhin}.

{\em Lehmer's conjecture is true (solved) 
in the following particular cases}:
(i) for the closed set S of Pisot numbers (Salem \cite{salem}, Siegel \cite{siegel}, Bertin et al \cite{bertinetal}), 

(ii) for the set of algebraic numbers $\alpha$ for which the minimal polynomial $P_{\alpha}$ is nonreciprocal by Smyth's Theorem 
\cite{smyth} \cite{smyth2} (1971) 
which asserts:
\begin{equation}
\label{smythMINI}
{\rm M}(\alpha) = {\rm M}(P_{\alpha}) \geq \Theta,
\end{equation}
proved to be an isolated infimum by 
Smyth \cite{smyth2},

(iii) for every nonzero algebraic integer 
$\alpha \in \mathbb{L}$, of degree $d$, 
assuming that $\mathbb{L}$ is a totally real algebraic number
field, or a CM field (a totally complex quadratic extension of a totally real number field); then  Schinzel \cite{schinzel2} obtained the minoration 
\begin{equation}
\label{szzz}
{\rm M}(\alpha) \geq \bigl( \frac{1+\sqrt{5}}{2}\bigr)^{d/2},
\mbox{from which}
:
{\rm M}(\alpha) \geq ((1+\sqrt{5})/2)^{1/2} = 1.2720\ldots
\end{equation}
Improvments of this lower bound,
by Bertin, Rhin, Zaimi, 
are reported in $\S$ \ref{S7.3},

(iv) for $\alpha$ an algebraic number of degree 
$d$ such that there exists a prime number 
$p \leq d \, \lo d$ that is not ramified in the field  
$\qb(\alpha)$; then Mignotte
\cite{mignotte} \cite{mignotte2} showed: 
M$(\alpha) \geq 1.2$; by extension,
Silverman \cite{silverman2} proved that the
Conjecture of Lehmer is true if there exist
primes
${\mathfrak p}_1 , {\mathfrak p}_2 , \ldots , 
{\mathfrak p}_d$ in $\qb(\alpha)$
satisfying $N {\mathfrak p}_i \leq \sqrt{d \,\lo d}$,

(v) for any noncyclotomic irreducible
polynomial $P$ with all odd coefficients;
Borwein, Dobrowolski and Mossinghoff \cite{borweindobrowolskimossinghoff}
\cite{garzaishakmossinghoffpinnerwiles}
proved
(cf Theorem \ref{borweindobrowolskimossingohhthm}
and Silverman's Theorem
\ref{silverman5thm} for details)
\begin{equation}
\label{bodomo}
{\rm M}(P) \geq 5^{1/4} = 1.4953\ldots,
\end{equation}

(vi) in terms of the Weil height, 
Amoroso and David \cite{amorosodavid2} 
proved that there exists
a constant $c> 0$ such that, 
for all nonzero algebraic number 
$\alpha$, of degree $d$, not being a 
root of unity, under the assumption that the extension
$\qb(\alpha)/\qb$ is Galois, then
\begin{equation}
\label{amda}
h(\alpha) \geq \frac{c}{d}.
\end{equation}

Some minorations are known
for some classes of polynomials
(Panaitopol \cite{panaitopol}).
Bazylewicz \cite{bazylewicz}
extended Smyth's Theorem (i.e. 
\eqref{smythMINI}) to polynomials
over Kroneckerian fields $K$ 
(i.e. for which $K/\qb$ is
totally real or is a totally complex nonreal
quadratic extension of such fields).
Notari \cite{notari}
and Lloyd-Smith \cite{lloydsmith}
extended such results to algebraic numbers.
Lehmer's problem is related to
the minoration problem of the
discriminant
(Bertrand \cite{bertrand}).
Mahler measures
$\{{\rm M}(\alpha) \mid \alpha \in \overline{\qb}\}$
are Perron numbers, 
by the work of Adler and Marcus \cite{adlermarcus}
in topological dynamics, 
as a consequence of the Perron-Frobenius theory.
The set $\pb$ of Perron numbers is 
everywhere dense in $[1, +\infty)$ and 
is important since it
contains
subcollections which have particular topological
properties for which
conjectures exist. 
The set $\pb$ admits a nonfactorial multiplicative 
arithmetics \cite{brunotte}
\cite{lind2}
\cite{vergergaugry3} 
for which 
the restriction of the usual addition $+$
to a given subcollection 
is not necessarily internal
\cite{dubickas10}. 
Salem \cite{salem2} proved that
$S \subset \pb$ is closed, and that 
$S \subset \overline{T}$. 
Boyd \cite{boyd}
conjectured that
$T \cup S$ is closed and that
$S = (S \cup T)^{(1)}$
(\cite{boyd2}, p. 237).
This second Conjecture would imply that
all Salem numbers $< \Theta$ would also be
isolated Salem numbers, not only Lehmer's number. 
The set of Mahler measures
$\{{\rm M}(\alpha) \mid \alpha \in \overline{\qb}\}$
and the semi-group
$\{{\rm M}(P) \mid P(X) \in \zb[X]\}$ 
are
strict subsets of $\pb$ and
are distinct 
(Boyd \cite{boyd10}, 
Dubickas \cite{dubickas7}
\cite{dubickas8}).
Values of Mahler measures
were studied by
Boyd \cite{boyd9}
\cite{boyd10} \cite{boyd11},
Chern and Vaaler \cite{chernvaaler},
Dixon and Dubickas \cite{dixondubickas},
Dubickas \cite{dubickas6},
Schinzel \cite{schinzel4},
Sinclair \cite{sinclair}.
Boyd \cite{boyd11} has shown that 
the Perron numbers
$\gamma_n$ which are the dominant roots of
the height one (irreducible)
trinomials $-1 -z +z^n$, 
$n \geq 4$, are not Mahler measures.
The {\em inverse problem} for the
Mahler measure consists in determining 
whether, or not, a Perron 
number $\gamma$ 
is the Mahler measure 
${\rm M}(P)$
of an integer polynomial
$P$, and 
to give formulas for the number 
$\#\{P \mid {\rm M}(P) = \gamma\}$
of such polynomials 
with measure $\gamma$ and given degree 
(Boyd \cite{boyd9}, 
Dixon and Dubickas \cite{dixondubickas}, 
Staines \cite{staines}). 
Drungilas and Dubickas \cite{drungilasdubickas}
and
Dubickas \cite{dubickas6} 
\cite{dubickas8} proved    
that the subset of Mahler measures is 
very rich: namely,
for any Perron number
$\beta$, there exists an integer $n \in \nb$
such that $n \beta$ is a Mahler measure,
and any real algebraic integer is the 
difference of two Mahler measures. 

The set of limit points of 
$\{{\rm M}(P) \mid P(X) \in \zb[X]\}$
is obtained by the following 
useful Theorem of Boyd and Lawton
\cite{boyd7} \cite{boyd8} 
\cite{lawton}, 
which correlates
Mahler measures of univariate polynomials to 
Mahler measures of multivariate polynomials:
\begin{theorem}
\label{boydlawton}
let $P(x_1, x_2, \ldots, x_n)
\in \cb[x_1, \ldots, x_n]$) 
and $\underline{r} = (r_1, r_2, \ldots, r_n)$,
$r_i \in \nb_{> 0}$. Let
$P_{\underline{r}}(x) :=
P(x^{r_1}, x^{r_2}, \ldots, x^{r_n})$.
Let
$$q(\underline{r})~:=\,
\min \{H(\underline{t}) \mid
\underline{t} =(t_1, t_2, \ldots, t_n) \in \zb^{n},
\underline{t}\neq (0, \ldots0),
~\sum_{j=1}^{n} t_j r_j = 0
\},$$
where 
$H(\underline{t}) = \max \{|t_i| \mid 1 \leq j \leq n\}$.
Then
$$\lim_{q(\underline{r}) \to \infty}
{\rm M}(P_{\underline{r}}) =
{\rm M}(P).$$
\end{theorem}
This theorem allows the search of 
small limit points of (univariate)
Mahler's measures, by 
several methods \cite{boydmossinghoff};
another class of methods relies upon the
(EM) Expectation-Maximization algorithm
\cite{maclachlankrishnan}
\cite{elotmanirhinsacepee}.
The set of limit points of the Salem numbers
was investigated either by the 
``Construction of Salem" 
\cite{boyd2} \cite{boyd3}
or sets of 
solutions of some equations \cite{boydparry}.
Everest \cite{everest}, then Condon
\cite{condon} \cite{condon2}
established asymptotic expansions of
the ratio 
${\rm M}(P_{\underline{r}})/
{\rm M}(P)$. 
For bivariate polynomials $P(x,y)
\in \cb[x,y]$ such that $P$ and 
$\partial P/
\partial y$ do not have a common zero
on $\tb \times \cb$, then Condon
(Theorem 1 in \cite{condon2}) establishes the expansion,
for $k$ large enough,
\begin{equation}
\label{condonexpansion}
\lo \Bigl(
\frac{{\rm M}(P_{\underline{r}})}{{\rm M}(P)}
\Bigr)
=
\lo \Bigl(
\frac{{\rm M}(P(x, x^n))}{{\rm M}(P(x,y))}
\Bigr)
=
\sum_{j=2}^{k-1} \, \frac{c_j}{n^j} 
\, + \,
O_{P, k}\left(
\frac{1}{n^k}
\right),
\end{equation}
($n$ is not the degree of the univariate polynomial
$P(x,x^n)$)
where the coefficients $c_j$
are values of a quasiperiodic function of $n$,
as finite sums of real and imaginary parts of
values of ${\rm Li}_a$ polylogarithms, $2 \leq a \leq j$,
weighted by some rational functions
deduced from the derivatives of $P$,
where the sums
are taken over algebraic numbers
deduced from the intersection of $\tb^2$
and the 
hypersurface of $\cb^2$ defined by 
$P$ (affine zero locus). 
In particular, if $P$ is an integer polynomial,
the coefficients $c_j$ are
$\overline{\qb}$-linear combinations 
of polylogarithms
evaluated at algebraic arguments.
For instance,
for $P(x,y) = -1 + x +y$, 
$G_{n}(x) = -1 + x + x^n$,
the coefficient
$c_{2}(n)$ in the expansion of
$\lo({\rm M}(G_n) / {\rm M}(P))$, 
though a  priori quasiperiodic,
is a periodic function of $n$ modulo 6
which can be directly computed
(Theorem 1.3 in \cite{vergergaugry6}), as:
for $n$ odd:
$$
c_{2}(n) = \left\{
\begin{array}{lll}
\sqrt{3} \pi/18 = +0.3023\ldots \qquad
& \quad \mbox{if}~~ n \equiv 1 ~\mbox{or}~ 3 & ~( {\rm mod}~ \, 6),\\
-\sqrt{3} \pi/6 = -0.9069\ldots \qquad 
& \quad \mbox{if}~~ n \equiv 5               & ~( {\rm mod}~ \, 6),
\end{array}
\right.
$$
for $n$ even:
$$
c_{2}(n) = \left\{
\begin{array}{lll}
- \sqrt{3} \pi/36 = -0.1511\ldots \qquad
& \quad \mbox{if}~~ n \equiv 0 ~\mbox{or}~ 4 & ~( {\rm mod}~ \, 6),\\
+\sqrt{3} \pi/12 = +0.4534\ldots \qquad 
& \quad \mbox{if}~~ n \equiv 2               & ~( {\rm mod}~ \, 6).
\end{array}
\right.
$$
For the height one 
trinomial $1 + x + y$, 
Corollary 2
in \cite{condon2} gives
the coefficients $c_j$, $j \geq 2$,
as linear combinations of
polylogaritms evaluated at
third roots of unity, with coefficients
coming from the 
Stirling numbers of the first and second
kind, i.e. in
$\frac{1}{2 \pi} \zb[\sqrt{3}]$. 
The method of Condon also provides
the other coefficients $c_j$, $j \geq 3$,
for the trinomial 
$-1+x+y$ in the same way. 

Doche in
\cite{doche}
obtains
an alternate method to Boyd-Lawton's Theorem, 
in the objective of obtaining 
estimates of the
Mahler measures of bivariate polynomials:
let $P(y,z) \in \cb[y,z]$ be a polynomial such that
$\deg_{z}(P) > 0$,
 let $\xi_n := 
e^{\frac{ 2 i \pi}{n}}$
and assume that $P(\xi_{n}^{k}, z) \not\equiv 0$
for all $n, k$. Then
\begin{equation}
\label{dochelimite}
{\rm M}(P(y,z))^{1/\deg_{z}(P)}
 =
 \lim_{n \to \infty}
 \mathcal{M}
 \Bigl(
 \prod_{k=1}^{n}
P(\xi_{n}^{k}, z)
 \Bigr)
 \end{equation}
($n$ is not the degree of the univariate
polynomial $\prod_{k=1}^{n}
P(\xi_{n}^{k}, x)$).
Doche's and Condon's methods 
cannot be used for the problem of Lehmer
since 1 does not belong to the 
first derived set of the set of
Mahler measures
of univariate integer polynomials
(assuming true Lehmer's Conjecture).

Algebraic numbers close to 1
ask many questions 
\cite{dobrowolski3}
and require
new methods of investigation, 
as reported by Amoroso
\cite{amoroso2}.
For $\alpha$ an algebraic integer of degree $d > 1$,
not a root of unity, Blansky and Montgomery \cite{blanskymontgomery}
showed, with multivariate Fourier series,
$${\rm M}(\alpha) > 
1 + \frac{1}{52} \frac{1}{d \lo (6 d)}.$$
By a different approach, using an auxiliary function and a proof 
of transcendence (Thue's method), 
Stewart \cite{stewart} obtained the  
same minoration but with a constant  
$c \neq 1/52$ instead of $1/52$
\cite{bugeaudmignottenormandin}
\cite{laurent2}
\cite{mignottewaldschmidt}
\cite{waldschmidt}.
In 1979 a remarkable minoration 
was obtained by 
Dobrowolski \cite{dobrowolski2}
who showed 
\begin{equation}
\label{minoDOBRO1979}
\mbox{{\rm M}}(\alpha) ~>~ 
1 + (1-\epsilon) \left(\frac{\lo \lo d}{\lo d}\right)^3 
\, , \quad d > d_{1}(\epsilon).
\end{equation}
for any nonzero algebraic number
$\alpha$ of degree $d$, with $1-\epsilon$ replaced by $1/1200$ for an effective version
(then with $d \geq 2$), in particular for
$|\alpha| > 1$ arbitrarily close to 1.
The minoration \eqref{minoDOBRO1979}
was also obtained by
Mignotte in \cite{mignotte} 
\cite{mignotte2} but with
a constant smaller than
$1-\epsilon$. Cantor and Strauss \cite{cantorstrauss}
\cite{pathiauxdelefosse}, 
then Louboutin \cite{louboutin}, improved the 
constant $1-\epsilon$: they obtained  
$2 (1+o(1))$, resp. $9/4$
(cf also Rausch \cite{rausch}
and Lloyd-Smith \cite{lloydsmith}).
If $\alpha$ is a nonzero algebraic number of 
degree $d \geq 2$, 
Voutier \cite{voutier} obtained 
the better effective minorations:
\begin{equation}
\label{voutierIneq}
{\rm M}(\alpha) > 1 + \frac{1}{4}
\left(
\frac{\lo \lo d}{\lo d}\right)^3 
\qquad
{\rm and} \qquad {\rm M}(\alpha) > 1 + 
\frac{2}{(\lo (3 d))^3}. 
\end{equation}
For sufficiently large degree $d$, Waldschmidt
(\cite{waldschmidt2}, Theorem 3.17) showed that
the constant $1-\epsilon$ could be replaced in
\eqref{minoDOBRO1979} by
$1/250$ with a transcendence proof which uses an interpolation determinant. 
It is remarkable that 
these minorations only depend  
upon the degree of $\alpha$ and not of 
the size of the coefficients, i.e. of
the (na\"ive) height of their 
minimal polynomial.
Dobrowolski's proof is a 
transcendence proof
(using Siegel's lemma, 
extrapolation at finite places) 
which has been extended
to the various Lehmer problems ($\S$ \ref{S2.2}).

In 1965 the following Conjecture was formulated
in \cite{schinzelzassenhaus}.

\begin{conjecture}[Schinzel Zassenhaus]
\label{ConjSZ}
Denote by ${\rm m}_{h}(n)$ the minimum of the 
houses $\house{\alpha}$ 
of the algebraic integers $\alpha$
of degree $n$ which are not a root of unity.
There exists a (universal) constant $C > 0$
such that
\begin{equation}
\label{CJschzass}
{\rm m}_{h}(n) \geq 1 + \frac{C}{n}, \qquad n \geq 2.
\end{equation}
\end{conjecture}

An algebraic integer $\alpha$, of degree $n$, is
said {\it extremal} if $\house{\alpha} = 
{\rm m}_{h}(n)$. An extremal algebraic integer is not necessarily a Perron number
\cite{bertrandmathis2}.

Schinzel and Zassenhaus 
\cite{schinzelzassenhaus}
obtained the first result: 
for $\alpha \neq 0$ being 
an algebraic integer
of degree $n \geq 2$ which is not a root of unity, then
$\house{\alpha} > 1 + 4^{-(s+2)}$,
where $2 s$ is the number of nonreal conjugates of 
$\alpha$. 
For a nonreciprocal algebraic integer
$\alpha$ of degree $n$,
Cassels \cite{cassels} obtained:
\begin{equation}
\label{casselsminorant}
\house{\alpha} > 1 + \frac{c_2}{n},
\qquad \mbox{with}~
c_2 = 0.1;
\end{equation}
Breusch \cite{breusch} independently showed 
that $c_2 = \lo (1.179 \ldots) = 0.165\ldots$
could be taken; Schinzel \cite{schinzel2}
showed that $c_2 = 0.2$ could also be taken.
Finally Smyth \cite{smyth} improved the minoration
\eqref{casselsminorant} with
$c_2 = \lo \Theta = 0.2811\ldots$.
On the other hand,
Boyd \cite{boyd9} 
showed that $c_2$ cannot exceed $\frac{3}{2} \lo \Theta = 0.4217\ldots$. 
In 1997 Dubickas \cite{dubickas3} showed that
$c_2 = \omega -\epsilon$ with $\omega = 0.3096\ldots$ 
the smallest root of an equation in the interval 
$(\lo \Theta , +\infty)$, with 
$\epsilon > 0, \, n_{0}(\epsilon)$ an effective constant, and for all $n > n_{0}(\epsilon)$. These two bounds seem to be the best known extremities delimiting the domain of existence of the constant $c_2$
\cite{dobrowolskilawtonschinzel}.

The expression of the minorant 
in \eqref{CJschzass}, 
``in $1/n$",
as a function of $n$, is not far 
from being optimal, 
being 
``in $1/n^2$" at worse
in \eqref{dobrohouse}. Indeed, 
for nonzero algebraic integers $\alpha$, 
Kronecker's Theorem \cite{kronecker} 
implies that
$\house{\alpha} = 1$ if and only if $\alpha$ is a root of unity. 
The sufficient condition in Kronecker's Theorem 
was weakened by Blansky and Montgomery \cite{blanskymontgomery} who showed that $\alpha$, with $\deg \alpha = n$, is a root of unity provided
$$\house{\alpha} \leq 1 + 
\frac{1}{30 n^2 \lo (6 n)}.$$ 
Dobrowolsky \cite{dobrowolski} sharpened this condition by: if
\begin{equation}
\label{dobrohouse}
 \house{\alpha} < 1 + 
\frac{\lo n}{6 n^2},
\end{equation}
then $\alpha$ is a root of unity.
Matveev \cite{matveev} proved, for 
$\alpha$, with $\deg \alpha = n$, not being 
a root of unity,
$$\house{\alpha}
\geq \exp \frac{\lo (n+\frac{1}{2})}{n^2}.$$
Rhin and Wu \cite{rhinwu} verified
Schinzel Zassenhaus's Conjecture up to $n=28$
and improved Matveev's minoration as:
$$\house{\alpha}
\geq \exp \frac{3 \lo (\frac{n}{3})}{n^2} \qquad 4 \leq n \leq 12,$$
and, for $n \geq 13$,
 $$\house{\alpha}
\geq \exp \frac{3 \lo (\frac{n}{2})}{n^2}.$$
Matveev's minoration is better than
Voutier's lower bound \cite{voutier}
$$m_{h}(n) \geq
\left(
1 + \frac{1}{4}\left(
\frac{\lo \lo n}{\lo n}
\right)^3
\right)^{1/n}
$$
for $n \leq 1434$, and Rhin Wu's minoration
is better than Voutier's bound for
$13 \leq n \leq 6380$. For reciprocal nonzero
algebraic integers
$\alpha$, $\deg(\alpha) = n \geq 2$, not being 
a root of unity, Dobrowolski's lower bound is
$$\house{\alpha}
> 
1 + (2 - \epsilon)\left(
\frac{\lo \lo n}{\lo n}
\right)^3 \frac{1}{n} ,
\qquad n \geq n_{0}(\epsilon),
$$
where the constant $2-\epsilon$
could be replaced by
$\frac{9}{2}-\epsilon$ (Louboutin \cite{louboutin}),
or, better, by
$\frac{64}{\pi^2} - \epsilon$
(Dubickas \cite{dubickas}).
Callahan, Newman and Sheingorn
\cite{callahannewmansheingorn}
introduce a weaker version of 
Schinzel Zassenhaus's Conjecture:
given a number field $K$, they define
the  {\em Kronecker constant} of 
$K$ as the least $\eta_K > 0$ such that
$\house{\alpha} \geq 1 + \eta_K$
for all $\alpha \in K$.
Under certain assumptions on $K$,
they prove that there exists $c > 0$
such that $\eta_K \geq c/[K : \qb]$.

The sets of extremal algebraic integers are still unknown.
In Boyd \cite{boyd7} \cite{wu}
the following conjectures on {\it extremality} are formulated: 
\begin{conjecture}[Lind - Boyd]
The smallest Perron number of degree $n \geq 2$ has minimal polynomial
\begin{center}
$\begin{array}{cl}
X^n -X -1 & \mbox{if ~$n \not\equiv 3, 5 ~\mod~6$} ,\\
(X^{n+2} -X^4 -1)/(X^2 - X + 1)
& \mbox{if ~$n \equiv 3 ~\mod~6$} ,\\
(X^{n+2} -X^2 -1)/(X^2 - X  + 1)
& \mbox{if ~$n \equiv 5 ~\mod~6$}.
\end{array}
$
\end{center}
\end{conjecture}

\begin{conjecture}[Boyd]
\label{CJ6boyd}
(i) If $\alpha$ is extremal, then it is always nonreciprocal,

(ii) if $n = 3k$, then the extremal $\alpha$ has minimal polynomial 
$$X^{3 k} + X^{2 k} -1, \qquad {\rm or}
\quad X^{3 k} - X^{2 k} -1,$$

(iii) the extremal $\alpha$ of degree $n$
has asymptotically a number of conjugates
$\alpha^{(i)}$ outside the closed unit disc equal to
$$\cong \frac{2}{3} \,n, \qquad \quad n \to \infty.$$
\end{conjecture} 
This asymptotic proportion of $\frac{2}{3} \,n$
would correspond to a fairly regular angular 
distribution of the complete set of 
conjugates
in a small annulus containing 
the unit circle, in the sense of the 
Bilu-Erd\H{o}s-Tur\'an -Amoroso-Mignotte 
equidistribution theory 
\cite{amorosomignotte}
\cite{belotserkovski}
\cite{bilu}
\cite{erdosturan}.
As for the trinomials $(G_n)$
Conjecture \ref{CJ6boyd} (iii) is 
compatible with the fact that
the Perron numbers
$\theta_{n}^{-1}, n \neq 2, 3$, 
are not extremal since the prorata of roots
of $G_{n}^{*}$ of modulus $> 1$
is asymptotically $1/3$ and not $2/3$
(Proposition 5.1 in \cite{vergergaugry6}).

The nature of the coefficient vector
of an integer polynomial $P$
is linked to 
the Mahler measure ${\rm M}(P)$
and to extremal properties \cite{mahler}.
If some inequalities between
coefficients occur, then
Brauer \cite{brauer} proved that
$P$ is a Pisot polynomial; in this case Lehmer's 
problem is solved for this class $\{P\}$.
Stankov \cite{stankov} proved that
a real algebraic integer $\tau > 1$
is a Salem number if and only if its minimal
polynomial is reciprocal of even degree $\geq 4$
and if there is an integer $n \geq 2$ such that
$\tau^n$ has minimal polynomial
$P_{n}(x)=a_{0,n}+a_{1,n} x + \ldots +
a_{d,n} x^n$ which is also reciprocal of degree $d$
and satisfies the condition
$$|a_{d-1 , n}|
>
\frac{1}{2} 
\bigl(\frac{d}{d-2}\bigr) 
(2 +\sum_{k=2}^{d-2} |a_{k,n}|).
$$
Related to Kronecker's Theorem \cite{kronecker}
is the problem of finding 
necessary and sufficient
conditions on the coefficient vector of
reciprocal, self-inversive, 
resp. self-reciprocal polynomials to have 
all their roots
on the unit circle (unimodularity): 
Lakatos \cite{lakatos3}
proved that a polynomial 
$P(x) =
\sum_{j=0}^m A_j x^j
\in \rb[x]$ satisfying the conditions
$A_{m-j} = A_j$ for $j \leq m$\,
and
$$|A_m| \geq 
\sum_{j=0}^m |A_j - A_m|$$
has all zeroes on the unit circle.
Schinzel \cite{schinzel5}, Kim and Park
\cite{kimpark}, Kim and Lee \cite{kimlee},
Lalin and Smyth
\cite{lalinsmyth} obtained generalizations
of this result. 
Suzuki \cite{suzuki} established correlations
between this problem and the theory of
canonical systems of ordinary linear
differential equations.
Lakatos and Losonczi \cite{lakatoslosonczi}
\cite{lakatoslosonczi2}
proved that, for a self-inversive polynomial
$P_{m}(z) =
\sum_{j=0}^m A_k z^k \in \cb[z]$,
$m \geq 1$, the roots of $P_m$
are all on the unit circle if 
$|A_m| \geq 
\sum_{k=1}^{m-1} |A_k|$; moreover
if this inequality is strict
then the zeroes
$e^{i \varphi_l}$,
$l=1, \ldots, m$, are simple and can be 
arranged such that, with
$\beta_m = \arg(A_m 
\left(\overline{A_0} / A_m\right)^{1/2})$ ,
$$\frac{2((l-1)\pi - \beta_m)}{m}
< \varphi_l 
<
\frac{2(l\pi - \beta_m)}{m} .
$$
In the direction of Salem polynomials,
$\nu$-Salem polynomials and more
\cite{kerada} \cite{samet},
a generalization was obtained by 
Vieira \cite{vieira}: if a sufficient 
condition is satisfied then 
a self-inversive polynomial
has a fixed number of roots on the unit circle.
Namely, let
$P(z) = a_n z^n + a_{n-1} z^{n-1} +
\ldots
+ a_1 z + a_0 \in \cb[z]$, $a_n \neq 0$, 
be such that
$P(z) = \omega \, z^n \, \overline{P}(1/z)$
with $|\omega| =1$. If the inequality 
$$|a_{n-l}|
>
\frac{1}{2} \left(
\frac{n}{n-2l}\right)
\sum_{k=0, k\neq l, k\neq n-l}^{n}
|a_k| ,\qquad l < n/2$$
is satisfied, then
$P(z)$ has exactly $n-2l$ roots on the unit circle
and these roots are simple; 
moreover, if $n$
is even and $l=n/2$, 
then $P(z)$ has no root
on
$|z|=1$
if the inequality 
$|a_{n/2}| > \sum_{k=0, k\neq n/2}^{n} |a_k|$ is satisfied.

Questions of irreducibility of $P$ 
as a function of the coefficient vector 
were 
studied in \cite{dubickas5}.
Flammang \cite{flammang2} obtained new 
inequalities for the Mahler measure 
${\rm M}(P)$ \cite{waldschmidt3}, 
and
Flammang, Rhin and Sac-Ep\'ee
\cite{flammangrhinsacepee}
proved relations between the 
integer transfinite diameter
and polynomials having a small Mahler measure.
The lacunarity of $P$
and the minoration of 
${\rm M}(P)$ are correlated:
when $P$ is a noncyclotomic (sparse) integer
polynomial, Dobrowolski, 
Lawton and Schinzel
\cite{dobrowolskilawtonschinzel},
then Dobrowolski
\cite{dobrowolski4}
\cite{dobrowolski5},
obtained lower bounds 
of ${\rm M}(P)$
as a function of the number
$k$ of its nonzero coefficients: e.g. in 
\cite{dobrowolski5}, with $a < 0.785$,
$${\rm M}(P) 
\geq
1 +\frac{1}{\exp(a 3^{\lfloor (k-2)/4\rfloor}
k^2 \lo k)},$$
and, if $P$ is irreducible,
$${\rm M}(P) 
\geq
1 +\frac{0.17}{2^{\lfloor k/2 \rfloor} \,  
\lfloor k/2 \rfloor !} .$$
Dobrowolski, then 
McKee and Smyth \cite{mackeesmyth2}
obtained minorants
of ${\rm M}(P)$ for the reciprocal
polynomials
$P(z) = z^{n} D_{A}(z + 1/z)$
where $D_A$ is the characteristic polynomial
of an integer symmetric 
$n \times n$ matrix $A$; 
McKee and Smyth obtained 
${\rm M}(P)=1$ or 
${\rm M}(P) \geq 1.176280\ldots$ (Lehmer's number)
solving the problem of Lehmer
for the family of such polynomials.
Dobrowolski (2008)  proved that
many totally real integer
polynomials $P$ cannot be represented by  
integer symmetric matrices $A$, 
disproving a conjecture of 
Estes and Guralnick.
Dubickas and Konyagin
\cite{dubickas13}
\cite{dubickaskonyagin}
studied the number of integer
polynomials
as a function of their (na\"ive) height and  resp.
their Mahler measure. 
The next two theorems show that Lehmer's
Conjecture is true for the set of
the algebraic integers
which are the roots of 
polynomials
in
particular families of monic integer polynomials.

\begin{theorem}[Borwein, Dobrowolski, Mossinghoff
\cite{borweindobrowolskimossinghoff}]
\label{borweindobrowolskimossingohhthm}
Let $m \geq 2$, and let
$f(X) \in \zb[X]$ be a monic polynomial 
of degree $D$
with no cyclotomic factors that satisfies
$$f(X) \equiv X^D + X^{D-1} +
\ldots + X^2 + X + 1 \quad
\mod m .$$ 
Then
$$\sum_{f(\alpha) = 0} h(\alpha)
\geq \frac{D}{D+1} C_m ,$$
where we may take
$$C_2 = \frac{1}{4} \lo 5
\qquad \mbox{{\it and}} \quad
C_m = \lo \frac{\sqrt{m^2 + 1}}{2}
\quad \mbox{{\it for}}~ m \geq 3.$$
\end{theorem}

\begin{theorem}[Silverman \cite{silverman5}]
\label{silverman5thm}
For all $\epsilon > 0$ there exists a constant
$C_{\epsilon} > 0$ with the following property:
let $f(X) \in \zb[X]$ be a monic polynomial 
of degree $D$ such that
$$f(X) ~\mbox{is divisible by}~ X^{n-1} + X^{n-2} +
\ldots  + X + 1 \quad
\mbox{in}~ (\zb / m \zb) [X] .$$
for some integers $m \geq 2$ and 
$n \geq\max\{\epsilon D, 2\}$.
Suppose further that no root of $f(X)$
is a root of unity. Then
$$\sum_{f(\alpha) = 0} h(\alpha)
\geq ~C_{\epsilon} \, \lo m .$$
\end{theorem}

{\em Limit points of Mahler measures} 
of univariate polynomials are algebraic numbers or transcendental numbers:
by 
\eqref{dochelimite} and 
Theorem \ref{boydlawton},
they are Mahler measures of 
multivariate polynomials. 
The problem of finding a positive lower bound
of the set of such limit points 
of Mahler measures is intimately correlated
to the problem of Lehmer \cite{schinzel}.
Smyth (1971)\cite{smyth4}
found the remarkable identity:
$\lo {\rm M}(1+x+y)= \Lambda$
(given by \eqref{limitMahlGn}).
The values of logarithmic Mahler measures of 
multivariate polynomials
are sums of special values of
different $L$-functions, 
often conjecturally
\cite{boyd16}; the
remarkable conjectural identities
discovered by Boyd in \cite{boyd16} (1998), also
by Smyth \cite{smyth4} and Ray \cite{ray},
serve as starting points
for further studies, some of them being 
now proved, e.g.
\cite{lalin}
\cite{lalin2}
\cite{rogers}
\cite{shindervlasenko}
\cite{zudilin}.

Indeed, after the publication of
\cite{boyd16},
Deninger \cite{deninger}
reinterpreted the logarithmic 
Mahler measures
$\lo {\rm M}(P)$
of Laurent polynomials
$P \in \zb[\zb^n] = 
\zb[x_{1}^{\pm}, \ldots, x_{n}^{\pm}]$
as topological entropies in the 
theory of dynamical systems of algebraic origin,
with 
$\zb^{n}$-actions (Schmidt
\cite{schmidt2}, Chap. V, Theorem 18.1;
Lind, Schmidt and Ward 
\cite{lindschmidtward}). 
This new approach 
makes a link with 
higher $K$-theory, mixed motives
(Deninger \cite{deninger2}),
real Deligne
cohomology, the Bloch-Beilinson conjectures
on special values of $L$- functions,
and Mahler measures.
There are two cases: either
$P$ does not vanish on
$\tb^n$, in which case
$\lo {\rm M}(P)$ is a Deligne period
of the mixed motive over $\qb$
which corresponds
to the nonzero symbol
$\{P, x_1 , \ldots , x_n\}$
(Theorem 2.2 in \cite{deninger}), or,
if $P$ vanishes on
$\tb^n$, under some assumptions,
it is a difference of two
Deligne periods of certain mixed motives, 
equivalently,
the difference of two symbols evaluated against
topological cycles
(``integral $K$-theory cycles")
(Theorem 3.4 in \cite{deninger}, with a 
motivic reinterpretation 
in Theorem 4.1 in \cite{deninger}).

Let $\gb_{m,A}^{n} := {\rm Spec}(A[\zb^n])$
be the split $n$-torus defined over the commutative
ring $A=\zb, \qb$ or $\rb$. The polynomial 
$P \not\equiv 0$
defines the irreducible closed
subscheme
$Z 
:=
{\rm Spec}(\zb[\zb^n]/(P))
\subset \gb_{m,\zb}^{n}, Z \neq \gb_{m,\zb}^{n}$,
For any coherent sheaf
$\mathcal{F}$
on $\gb_{m,A}^{n}$, the group 
$\Gamma(\gb_{m,A}^{n}, \mathcal{F})$
of global sections,
equipped with the discrete topology, 
admits a Pontryagin dual 
$\Gamma(\gb_{m,A}^{n}, \mathcal{F})^{*}$
which is a 
compact group. This compact group endowed
with the canonical $\zb^n$-action 
constitute an arithmetic dynamical system
for which the entropy can be
defined according to
\cite{schmidt2}, and correlated to
the Mahler measure (Theorem 18.1
in \cite{schmidt2});
the application to
$P$,
$A=\zb$ and $\mathcal{F}
=
\mathcal{O}_{Z}$ 
provides the identity with the entropy:
$h(\mathcal{O}_{Z}) = \lo {\rm M}(P)$.
The definition
$$\lo {\rm M}(P)
:=
\frac{1}{(2 i \pi)^n}
\int_{\tb^n}
\lo |P(x_1, \ldots, x_n)|
\frac{d x_1}{x_1}\ldots \frac{d x_n}{x_n} 
$$
corresponds to the integration of a differential form
in connection with
the cup-product
$\lo |P| \cup \lo |x_1| \cup \ldots\cup \lo |x_n|$
in the real Deligne cohomology
of $\gb_{m,\rb}^{n}\setminus Z_{\rb}$.
The link between the
$L$-series $L(M,s)$
of a motive $M$, 
and its derivatives,
Deligne periods, and the Beilinson conjectures, 
comes from
the Conjecture of Deligne-Scholl 
(\cite{deninger}
 Conjecture 2.1). Further,
Rodriguez-Villegas \cite{rodriguezvillegas}
studied the conditions of applicability
of the conjectures
of Bloch-Beilinson for having logarithmic Mahler measures $\lo {\rm M}(P)$ expressed as $L$-series. 

The example of 
$\lo {\rm M}(P(x_1 , x_2)) = 
\lo {\rm M}((x_1 + x_2)^2 + k)$, with
$k \in \nb$, is computed in Proposition 19.10 in 
\cite{schmidt2}. For instance, for $k=3$, we have
$$\lo {\rm M}((x_1 + x_2)^2 + 3) = 
\frac{2}{3} \lo 3 + \frac{\sqrt{3}}{\pi} L(2, \chi_3).$$
Deninger shows
(ex. \cite{deninger}  p. 275) the 
cohomological origin of each term: 
$\frac{\sqrt{3}}{\pi} L(2, \chi_3)$ from the first
$\mathcal{M}$-cohomology group
$H_{\mathcal{M}}^{1}(\partial A, \qb(2))$,
$\frac{2}{3} \lo 3$ from the second 
$\mathcal{M}$-cohomology group
$H_{\mathcal{M}}^{2}(Z^{{\rm reg}}, \qb(2))$.
Bornhorn \cite{bornhorn}, and later 
Standfest \cite{standfest},
reinvestigated further 
the conjectural identities
of Boyd \cite{boyd16} 
in particular the formulas of mixed type, containing
several types of
$L$-series.  The logarithmic Mahler measure
$\lo {\rm M}(P)$ is then written
$= * L'(s_1 , \chi) + * L'(E,s_2)$, where 
$\chi$ is a Dirichlet character, $L(s_1 , \chi)$ 
the corresponding
Dirichlet series,
$L(E,s_2)$ the Hasse-Weil $L$-function
of an elliptic curve
$E/\qb$ deduced from $P$, and $s_1 , s_2$ 
algebraic numbers. Following Deninger
and Rodriguez-Villegas,
Lalin \cite{lalin}
\cite{lalin2} introduces techniques for 
applying Goncharov's constructions
of the regulator on polylogarithmic
motivic complexes in the objective of 
computing Mahler measures of 
multivariate Laurent polynomials.
With some three-variable polynomials, 
whose zero loci define
singular $K3$ surfaces,
Bertin et al \cite{bertinfeaverfuselierlalinmanes}
prove that the logarithmic Mahler measure is of the form
$*L'(g,O)+*L'(\chi, -1)$ where $g$ is the 
weight 4 newform associated with the $K3$ 
surface and $\chi$ is a quadratic character.
Other three-variable Mahler measures are 
associated with
special values of modular and Dirichlet
$L$-series \cite{samart}. Some four-variables
polynomials define a Calabi-Yau threefold
and the logarithmic Mahler measure is of the form
$*L'(f,O)+*\zeta'(-2)$
where $f$ is a normalized newform 
deduced from the Dedekind eta function
\cite{papanikolasrogerssamart}.
Multivariable Mahler measures are also
related to mirror symmetry
and Picard - Fuchs equations 
in Zhou \cite{zhou}.

In comparison, the limit points of the set 
S of Pisot numbers were studied 
by analytical methods
by Amara \cite{amara}.
The set of values
$\{\lo {\rm M}(P) \mid P \in \zb[\zb^n], n \geq 1\}$
is conjecturally (Boyd \cite{boyd8})
a closed subset of $\rb$ 
for the usual topology.

\subsection{Small points and Lehmer problems in higher dimension}
\label{S2.2}

The theory of heights 
\cite{bombierigubler}
\cite{schanuel}
\cite{waldschmidt3}
is a powerful
tool for studying distributions of
algebraic numbers,
algebraic points on algebraic varieties,
and of subvarieties in projective spaces by extension.
Points having a small height, or 
``small points", resp. ``small"
projective varieties, together with
their distribution, have a 
particular interest in the problem of Lehmer 
in higher dimension. 

In the classical Lehmer problem,
the ``height" is the Weil height, and
Lehmer's Conjecture
is expressed by a Lehmer inequality where the
minorant
is ``a function of the degree",
i.e. it
states that there exists a universal
constant $c > 0$ such that
\begin{equation}
\label{weilminorant}
h(\alpha) \geq \frac{c}{\deg(\alpha)}
\end{equation}
unless $\alpha =0$ or is a root of unity.
The generalizations of 
Lehmer's problem are still
formulated by a minoration as in 
\eqref{weilminorant},
but in which
``$\alpha$" is replaced by a 
rational point ``$P$" of a (abelian) 
variety, or replaced by a variety ``$V$", where
``$h$" is replaced by another
height, more suitable, where
the degree ``$\deg(\alpha)$" 
may be replaced by the more convenient 
``{\em obstruction index}"
(``degree of a variety"), where
the minorant function of the ``degree"
may be more sophisticated than
the inverse
``$\deg(\alpha)^{-1}$".
These different minoration forms extend
the classical Lehmer's inequality
into a Lehmer type inequality.
Generalizing Lehmer's problem 
separates into  
three different Lehmer problems: 

(i) the classical
Lehmer problem, 

(ii) the relative Lehmer problem,

(ii) Lehmer's problem for subvarieties.

\noindent
(i) {\em The classical Lehmer problem}:
on $\gb_m$, Dobrowolski's and Voutier's minorations,
given by \eqref{dobrowolski79inequality} 
and \eqref{voutierIneq}, 
with ``$(\lo \deg(\alpha))^3$" at the denominator,
were up till now considered as
the best general lower bounds, 
as functions of the 
degree $\deg(\alpha)$.
Generalizations to higher dimension
(below)
have been largely studied:
e.g.
Amoroso and David 
\cite{amorosodavid2}
\cite{amorosodavid4}
\cite{amorosodavid5},
Pontreau \cite{pontreau}
\cite{pontreau2}, W. Schmidt 
\cite{schmidtw3}
for points on $\gb_{m}^{n}$,
Anderson and Masser
\cite{andersonmasser},
David \cite{david2},
Galateau and Mah\'e \cite{galateaumahe},
Hindry and Silverman \cite{hindrysilverman},
Laurent \cite{laurent},
Silverman
\cite{silverman}
\cite{silverman2}
\cite{silverman4}
for elliptic curves,
David and Hindry
\cite{david}
\cite{davidhindry}, 
Masser \cite{masser4}
for abelian varieties.

\begin{conjecture}(Elliptic Lehmer problem)
Let $E/K$ be an elliptic curve
over a number field $K$.
There is a positive
constant $C(E) > 0$ such that,
if $P \in E(\overline{K})$ has 
infinite order, 
\begin{equation}
\widehat{h}(P) \geq 
\frac{c(E)}{[ K(P) : K ]}.
\end{equation}
\end{conjecture}

\begin{theorem}[Laurent \cite{laurent}]
\label{laurentthm}
Let $E/K$ be an elliptic curve
with complex multiplication over a 
number field $K$. There is a positive constant
$c(E/K)$ such that
\begin{equation}
\widehat{h}(P) \geq 
\frac{c(E/K)}{D}\left(
\frac{\lo \lo 3D}{\lo 2D}
\right)^3
\qquad \mbox{for all}~ 
P \in E(\overline{K}) \setminus E_{{\rm tors}}
\end{equation}
where $D = [K(P) : K]$.
\end{theorem}
Masser \cite{masser} \cite{masser2} 
\cite{masser4}, and
David \cite{david2},
gave estimates of lower bounds of 
$\widehat{h}(P)$ for elliptic curves 
and abelian varieties, on
families of abelian varieties
\cite{masser3},
for $P$ 
of infinite order.
Galateau and Mah\'e \cite{galateaumahe}
solved the elliptic Lehmer problem
in the Galois case, extending
Amoroso David's Theorem 
(\cite{amorosodavid2}, and 
\cite{amorosoviada2} 
for sharper estimates):

\begin{theorem}[Galateau - Mah\'e \cite{galateaumahe}]
\label{galateaumahethm}
Let $E/K$ be an elliptic curve
over a number field $K$.
There is a positive
constant $C(E) > 0$ such that,
if $P \in E(\overline{K})$ has 
infinite order and the field extension
$K(P)/K$ is Galois, 
\begin{equation}
\widehat{h}(P) \geq 
\frac{c(E)}{[ K(P) : K ]}.
\end{equation}
\end{theorem}

Let $\alpha = 
(\alpha_1 , \ldots, \alpha_n)
\in \gb_{m}^{n}(\overline{\qb}) \subset 
\pb^{n}(\overline{\qb})$.
The height of $\alpha$
in $\gb_{m}^{n}(\overline{\qb})$ is
defined by 
$h(\alpha)= h(1 : \alpha)$
the absolute logarithmic height.
Let $F_{0} \in \qb[x_1 , \ldots , x_n]$ 
be a nonzero
polynomial vanishing at $\alpha$.
The {\em obstruction index}
of $\alpha$ is by definition
$\deg(F_0)$, denoted by
$\delta_{\qb}(\alpha)$.

\begin{conjecture}(Multiplicative Lehmer problem)
For any integer $n \geq 1$, there exists a real
number $c(n) > 0$ such that
\begin{equation}
h(\alpha) \geq \frac{c(n)}{\delta_{\qb}(\alpha)}
\end{equation}
for all
$\alpha = 
(\alpha_1 , \ldots, \alpha_n)
\in \gb_{m}^{n}(\overline{\qb})$ such that
$\alpha_1 , \ldots, \alpha_n$ are multiplicatively independent.
\end{conjecture}
Small points of subvarieties of algebraic tori
were studied by Amoroso \cite{amoroso4}.
\begin{theorem}[Amoroso - David]
There exist a positive constant
$c(n) > 0$
such that, 
for all
$\alpha = 
(\alpha_1 , \ldots, \alpha_n)
\in \gb_{m}^{n}(\overline{\qb})$ such that
$\alpha_1 , \ldots, \alpha_n$ are multiplicatively independent,
\begin{equation}
\label{amorosodavidDOBRO_nDimensionMino}
h(\alpha)
\geq \frac{c(n)}{\delta_{\qb}(\alpha)} \bigl(
\lo (3 \delta_{\qb}(\alpha))
\bigr)^{- \,\eta(n)}
\end{equation}
with 
$\eta(n)= (n+1) (n+1)!^n -n$.
\end{theorem}
As a consequence of the main
Theorem in 
\cite{amorosoviada2} Amoroso and Viada improved
the preceding Theorem and proved:
\begin{theorem}[Amoroso - Viada]
Let $\alpha_1 , \ldots, \alpha_n$ be multiplicatively
independent algebraic numbers
in a number field $K$
of degree
$D = [K : \qb]$. Then
\begin{equation}
\label{amorosoviadaDOBRO_nDimensionMino}
h(\alpha_1) \ldots h(\alpha_n)
\geq 
\frac{1}{D} 
\frac{1}{(1050 \, n^5 \lo (3 D))^{n^2 (n+1)^2}}.
\end{equation}
\end{theorem}
The assumption
of being multiplicatively independent was
reconsidered 
in the multiplicative group
$\overline{\qb}^{\times}/
{\rm Tor}(\overline{\qb}^{\times})$
by Vaaler in \cite{vaaler}.

Let $A/K$ be an abelian variety 
over $K$ a number field.
Let $\lc$ be a line bundle over $A$.
Let $V$ be a subvariety
of $A$ defined over $K$. 
The degree $\deg_{\lc}(V)$
of $V$ relatively to
the Cartier divisor $D$
associated with $\lc$ is defined by
the theory of intersection
\cite{ratazzi3}. 
In particular, if $P \in A(\overline{K})$,
and $V = \overline{\{P\}}$, then
$\deg_{\lc}(V) = [K(P) : K]$.

For any $P \in A(\overline{K})$
the {\em obstruction index}
$\delta_{K, \lc}(P)$ of $P$
is now extended as
$$:=\min \{\deg_{\lc}(V)^{\frac{1}{{\rm codim}(V)}}
\mid V_{/K} ~\mbox{
subvariety of } A ,
\mbox{for which}~ P \in V(\overline{K}) \}.$$

\begin{conjecture}(David-Hindry, 2000)
(Abelian Lehmer problem)
Let $A/K$ be an abelian variety 
over a number field $K$ and
$\lc$ an ample symmetric line bundle over $A$.
Then there exists
a real number 
$c(A, K, \lc) > 0$ such that 
the canonical height 
$\widehat{h}_{\lc}(P)$ of $P$ satisfies
\begin{equation}
\widehat{h}_{\lc}(P) 
\geq 
\frac{c(A, K, \lc)}{\delta_{K, \lc}(P)}
\end{equation}
for every point $P \in  A(\overline{K})$
of infinite order modulo every
proper abelian subvariety
$V_{/K}$ of $A$.
Moreover, if 
$D= [K(P):K]$, for any
$P \in A(\overline{K})$ not being in the torsion,
\begin{equation}
\widehat{h}_{\lc}(P) 
\geq 
\frac{c(A, K, \lc)}{D^{1/g_0}}
\end{equation}
where $g_0$ is the 
dimension of the smallest algebraic subgroup
containing $P$.
\end{conjecture}

For any abelian variety $A$ defined over a 
number field $K$ \cite{hindry}, 
let us denote, for any integer $n \geq 1$,
$K_n := K(A[n])$ the extension generated by
the group of the torsion points
$A[n]$, so that $K_{{\rm tors}}
= \cup_{n \geq 1} K(A[n])$.

\begin{theorem}[Ratazzi \cite{ratazzi5}]
Let $A/K$ be a CM abelian variety of
dimension $g$ 
over a number field $K$ and
$\lc$ an symmetric ample line bundle over $A$.
Then there exists
a real number 
$c(A, K, \lc) > 0$ such that,
for every point $P \in  A(\overline{K})$,
the canonical height 
$\widehat{h}_{\lc}(P)$ satisfies
either
\begin{equation}
(i)\qquad \quad \qquad \widehat{h}_{\lc}(P) 
\geq 
\frac{c(A/K, \lc)}{\delta_{K_n , \lc}(P)}
\left(
\frac{\lo \lo 3 \,[K_n : K] \, 
\delta_{K_n , \lc}(P)}{\lo 2 \,[K_n : K] \, 
\delta_{K_n , \lc}(P)}
\right)^{\eta(g)}
\end{equation}
with $\eta(g) = (2g+5) (g+2) (g+1)! (2g. g!)^g$;
or

(ii) the point $P$ belongs to a
proper torsion subvariety, $B \subset A_{K_n}$, 
defined over $K_n$, having a degree bounded by
$$\left(
\deg_{\lc} B
\right)^{1/{\rm codim} B}
\leq 
\frac{1}{c(A/K, \lc)} 
\delta_{K_n , \lc}(P)
\left(\lo 2 [K_n : K] 
\delta_{K_n , \lc}(P)
\right)^{2 g + 2 \eta(g)}
.$$
\end{theorem}

(ii) {\em The relative Lehmer problem}:
the generalization of the classical Lehmer 
problem for subfields 
$K \subset \overline{\qb}$
is decomposed into two steps: 

(ii-i) does there exist a real number 
$c(K) > 0$ such that
$h(\alpha) \geq c(K)$
for all 
$\alpha \in \gb_{m}(K)/\gb_{m}(K)_{{\rm tors}}$?

(ii-ii) if (i) is satisfied, does there exist
a real number $c'(K) > 0$ such that,
for all 
$\alpha \in \gb_{m}(\overline{K})
/\gb_{m}(\overline{K})_{{\rm tors}}$,
$h(\alpha) \geq \frac{c'(K)}{[K(\alpha) : K]}$?

If $K$ is a number field, (ii-i) is satisfied 
by Northcott's Theorem and 
(ii-ii) amounts to the classical
Lehmer problem. If $K$ is an infinite
extension of $\qb$ the problem is more difficult.
In (ii-i), when the field $K$ is
$\qb^{{\rm ab}}$, or the abelian closure
of a number field, it is usual
to speak of the {\em abelian Lehmer problem}. 
The abelian Lehmer problem was solved by
Amoroso and Dvornicich \cite{amorosodvornicich}:
they proved that, 
if $\lb/\qb$ is an abelian extension of number fields, 
$$h(\alpha) \geq \frac{\lo 5}{12}$$
for any nonzero $\alpha \in \lb$ which is not
a root of unity.
As for (ii-ii), it is usual to speak of 
{\em relative Lehmer problem}.
The abelian 
and the relative Lehmer problems
are naturally extended in higher dimension.
If $G$ denotes 
either an abelian variety $A/K$
over a number field $K$ or
the $n$-torus $\gb_{m}^{n}$,
and $K_{{\rm tors}}=
K(G_{{\rm tors}})$, the minorant function 
of the height
is expected to
depend upon the ``nonabelian part of the degree
$D$", where $D = [K(P) : K]$.
This ``nonabelian part : 
$D_{{\rm tors}} = [K_{{\rm tors}}(P) : K_{{\rm tors}}]$
of $D$" is equal to
$[K^{{\rm ab}} (P) : K^{{\rm ab}}]$, where 
$K^{{\rm ab}}$ is the abelian closure of $K$
(if $G=A$, $A$ is assumed CM).

Given an abelian extension 
$\lb/\kb$ of number fields 
and a nonzero algebraic number
$\alpha$ which is not a root of unity, with
$D := [\lb(\alpha)  : \lb]$,
Amoroso and Zannier \cite{amorosozannier}
proved the following result, which makes use of Dobrowolski's minoration and the previous minoration:
$$h(\alpha) \geq
\frac{c(\kb)}{D} 
\left(\frac{\lo \lo 5 D}{\lo 2 D}\right)^{13} ,
$$
where $c(\kb) > 0$, in the direction 
of the relative problem. Amoroso and 
Delsinne \cite{amorosodelsinne} computed
a lower bound, depending upon the degree 
and the discriminant of
the number field $\kb$, for the constant
$c(\kb)$. 
In 2010, given $\kb/\qb$ an extension of algebraic number fields, of degree $d$, Amoroso and Zannier \cite{amorosozannier2} showed
$$h(\alpha) \geq
3^{-d^2 - 2 d -6}$$
for any nonzero algebraic number $\alpha$ 
which is not a root of unity such that 
$\kb(\alpha)/\kb$ is abelian.
As a corollary they obtained
$$h(\alpha) \geq 3^{-14}$$
for any dihedral extension 
$\lb/\qb$ and any nonzero $\alpha \in
\lb$ which is not a root of unity.   
For cyclotomic extensions, they obtained sharper results:
(i) if $\kb$ is a number field of degree $d$,
there exists an absolute constant $c_2 > 0$ such that, with $\lb$ denoting the number field generated by $\kb$ and any given root of unity, then 
$$h(\alpha) \geq
\frac{c_2}{d} 
\frac{(\lo \lo 5 d)^3}{(\lo 2 d)^4},
$$
for any nonzero $\alpha \in \lb$ which is not a root of unity; (ii)
if $\kb$ is a number field of degree $d$,
and $\alpha$ any nonzero algebraic number, not a root of unity, such that $\alpha^n \in \kb$ for some integer $n$ under the assumption
that $\kb(\alpha)/\kb$ is an abelian extension,
then
$$h(\alpha) \geq
\frac{c_3}{d} 
\frac{(\lo \lo 5 d)^2}{(\lo 2 d)^4},
$$
for some constant $c_3 > 0$.

In higher dimension \cite{david} \cite{ratazzi3}, 
with
$G=A$ an abelian variety over 
a number field $K$, 
the {\em torsion obstruction index} 
$\delta_{K, \lc}^{{{\rm tors}}}(P)$
of a point $P$ is now defined by
$$:=\min \{\deg_{\lc^{{\rm tors}}}(V)^{\frac{1}
{{\rm codim}(V)}}
\mid V_{/K_{{\rm tors}}} ~\mbox{
subvariety of } A_{K_{{\rm tors}}} , 
\mbox{for which}
~ P \in V(\overline{K}) \}.$$

\begin{conjecture}(David)
Let $A/K$ be an abelian variety 
over a number field $K$ and
$\lc$ an ample symmetric line bundle over $A$.
Then there exists
a real number 
$c(A, K, \lc) > 0$ such that 
the canonical height 
$\widehat{h}_{\lc}(P)$ satisfies
\begin{equation}
\widehat{h}_{\lc}(P) 
\geq 
\frac{c(A, K, \lc)}{\delta_{K, \lc}^{{\rm tors}}(P)}
\end{equation}
for every point $P \in  A(\overline{K})$
of infinite order modulo every
proper abelian subvariety
$V_{/K}$ of $A$.
\end{conjecture}
The analogue of Amoroso and Dvornicich's theorem
\cite{amorosodvornicich}
(abelian Lehmer problem) was obtained by
Baker and Silverman for abelian varieties 
\cite{baker} \cite{bakersilverman}
and for elliptic curves by Baker
\cite{baker}, 
then by Silverman \cite{silverman4}:

\begin{theorem}[Baker - Silverman]
Let $A/K$ be an abelian variety 
over a number field $K$ and
$\lc$ an symmetric ample line bundle over $A$.
Let $\widehat{h}(P) : A(\overline{K}) \to \rb$
the associated canonical height.
Then there exists
a real number 
$c(A, K, \lc) > 0$ such that 
\begin{equation}
\widehat{h}(P) 
\geq 
c(A, K, \lc)
\qquad
\mbox{for all nontorsion points} 
~P \in  A(K^{{\rm ab}}).
\end{equation}
\end{theorem}
The proof relies upon
Zahrin's theorem on torsion 
points of abelian varieties deduced from 
the proof of Faltings's theorem 
\cite{faltings} 
of the Mordell Conjecture.
\begin{theorem}[Silverman]
\label{silvermanthm}
Let $K/\qb$ be a number field, let $E/K$ be 
an elliptic curve, and 
$\widehat{h}: E(\overline{K}) \to \rb$
be the canonical height on $E$.
There is a constant
$C(E/K) > 0$ such that every nontorsion
point $P \in E(K^{{\rm ab}})$
satisfies
$$\widehat{h}(P) > C(E/K).$$
\end{theorem}
Small points were studied by 
Carrizosa \cite{carrizosa2}.
Ratazzi \cite{ratazzi} obtained 
the relative version of
Amoroso and Zannier's minoration
\cite{amorosozannier}:

\begin{theorem}[Ratazzi]
Let $E/K$ be an elliptic curve
with complex multiplication 
over a number field $K$.
Then there exists
a constant
$c(E, K) > 0$ such that 
\begin{equation}
\widehat{h}(P) 
\geq 
\frac{c(E, K)}{D}
\left(\frac{\lo \lo 5 D}{\lo 2 D}\right)^{13}
\qquad
\mbox{for all nontorsion points} 
~P \in  E(\overline{K}),
\end{equation}
where $D = [K^{{\rm ab}}(P) : K^{{\rm ab}}]$.
\end{theorem}
In the direction of the relative problem, 
a better lower bound of the canonical height 
of a point $P$ in a CM abelian variety 
$A/K$ in terms of the degree
of the field generated by P over 
$K(A_{{\rm tors}})$
was obtained by Carrizosa \cite{carrizosa}.
For tori Delsinne \cite{delsinne}
obtained the following (the obstruction index 
$\omega_{K}(\alpha)$
is defined below):

\begin{theorem}[Delsinne]
\label{delsinnethm}
Let $n \geq 1$ be an integer.
There exist constants $c_{1}(n)$,
$\kappa_{1}(n)$, $\mu(n)$,
$\eta_{1}(n) > 0$
such that,
for any
$\alpha \in \gb_{m}^{n}(\overline{\qb})$
satisfying
$$h(\alpha) \leq
\left(
c_{1}(n) \omega_{\qb^{{\rm ab}}}(\alpha)
\bigl(
\lo (3 \omega_{\qb^{{\rm ab}}}(\alpha))
\bigr)^{\kappa_{1}(n)}
\right)^{-1} ,
$$ 
there exists a torsion subvariety
$B$ containing $\alpha$, the degree of $B$
being bounded by
$$(\deg B)^{1/{\rm codim} B}
\leq
c_{1}(n) \omega_{\qb^{{\rm ab}}}(\alpha)^{\eta_{1}(n)}
\bigl(
\lo (3 \omega_{\qb^{{\rm ab}}}(\alpha))
\bigr)^{\mu(n)};
$$
the constants are effective and one can take
the following values:
$$c_{1}(n) = \exp\Bigl(
64 n n! (2 (n+1)^2 (n+1)!)^{2n}\Bigr),$$
$$\kappa_{1}(n) = 3 (2 (n+1)^2 (n+1)!)^n ,\quad
\mu(n)= 8 n! (2 (n+1)^2 (n+1)!)^n ,$$
$$\eta_{1}(n) = (n-1)! 
\left(
\sum_{i=0}^{n-3} \frac{1}{i!} + 1
\right)$$
\end{theorem}

(iii) {\em Lehmer's problem for subvarieties}:
The extension from points to subvarieties 
has been formulated for nontorsion
subvarieties $V$
of the multiplicative group $\gb_{m}^{n}$
or of an abelian variety $A/K$
over a number field $K$
by David and Philippon 
\cite{davidphilippon} 
\cite{davidphilippon2} 
and 
Ratazzi \cite{ratazzi2} \cite{ratazzi3}.
The natural extension
of the minoration problem
for the height
consists in obtaining
the best minoration of {\em the height}
$\widehat{h}_{\lc}(V)$, resp. 
of the essential minimum, as a function 
of the degree of $V$
or of the 
{\em obstruction index of} $V$.
The 
{\em obstruction index $\delta_{K, \lc}(V)$ of} $V$,
resp.  $\omega_{K}(V)$, extends
the obstruction index 
$\delta_{K, \lc}(P)$
of a point $P$
\cite{davidphilippon2}. 
As for the definition of the
height of $V$ relatively 
to an symmetric ample line bundle
$\lc$,
two approaches were followed
\cite{ratazzi3}:
one by Philippon
\cite{philippon},
another one
by Bost, Gillet and Soul\'e
\cite{bostgilletsoule}, 
using theorems of Soul\'e \cite{soule} 
and Zhang \cite{zhang2}.
In the second construction
Zhang \cite{zhang2} showed how to consider
the canonical height (or N\'eron-Tate height, 
or normalized height) $\widehat{h}_{\lc}(V)$
as a limit of Arakelov
heights.

Define the canonical height  (say)
$\widehat{h}$ on $\gb_{m}^{n}(\overline{\qb})$
by
$\widehat{h}(\alpha_1, \ldots , \alpha_n ) 
= h(\alpha_1) + \ldots + 
h(\alpha_n ).
$
For $\theta > 0$, let
$V$ be a subvariety of 
$\gb_{m}^{n}$ defined over $\overline{\qb}$.
For $\theta > 0$, let: 
$$V_{\theta} := \{ P \in V(\overline{\qb}) \mid  
\leq \theta\},$$
and the {\em essential minimum}
$$\widehat{\mu}^{{\rm ess}}(V)
:= 
\inf\{\theta > 0 \mid V_{\theta} 
\mbox{~is Zariski dense in}~ V \}.$$
The generalized Bogomolov conjecture for subvarieties of tori
asserts that 
$\widehat{\mu}^{{\rm ess}}(V) = 0$ is and only if
$V$ is a torsion subvariety.
In the case where $V$ is a point, 
$V = \{P\}$,
$\widehat{\mu}^{{\rm ess}}(V) = \widehat{h}(P)$.
Zhang \cite{zhang} 
\cite{zhang2} \cite{zhang3}
showed that the minoration problem of
$\widehat{\mu}^{{\rm ess}}(V)$ is 
essentially the same problem as finding 
lower bounds for the
canonical height
$\widehat{h}(V)$ of $V$, 
in the sense of Arakelov theory.
Indeed, from his Theorem of the Successive Minima,
Zhang proved:
$$\widehat{\mu}^{{\rm ess}}(V)
\leq
\frac{\widehat{h}(V)}{\deg(V)}
\leq \, (\dim (V) +1) \,
\widehat{\mu}^{{\rm ess}}(V)$$
for $V$ any subvariety of 
$\gb_{m}^{n}$ over $\overline{\qb}$.
Zhang obtained similar results for
subvarieties of abelian varieties.
The  canonical height
$\widehat{h}(V)$ of $V$ is related to the problem
of minoration of multivariate Mahler 
measures by the following:
for $V$ being 
a hypersurface defined by a polynomial
$F (x_1 , \ldots,  x_n ) 
\in \zb[x_1 , \ldots, x_n ]$ (having 
relatively prime integer
coefficients), then
$$\widehat{h}(V)
=
\int_{0}^{1} \ldots \int_{0}^{1}
\lo |F (e^{2 \pi i t_1} , \ldots, e^{2 \pi i t_n} )|
d t_1 \ldots d t_n
$$
is the logarithmic Mahler measure 
$\lo {\rm M}(F)$ of $F$.
Let K be a field of characteristic zero, 
and let V be a
subvariety of $\gb_{m}^{n}$ 
defined over $\overline{\qb}$.
Define the index of obstruction 
$\omega_{K}(V)$ to be the minimum
degree of a nonzero polynomial 
$F \in K [x_1 , \ldots, x_n ]$ vanishing
identically on $V$.
Equivalently, it 
is the minimum degree of a hypersurface
defined over $K$ and containing $V$.
The
{\em higher-dimensional Lehmer Conjecture}
takes the following form
(i.e. the two following conjectures):
\begin{conjecture} (Amoroso - David, 1999)
\label{amorosodavidCJ99}
Let V be a subvariety of 
$\gb_{m}^{n}$, 
and assume that $V$ is not contained
in any torsion subvariety (i.e., a translate of a proper subgroup by
a torsion point). 
Then there exists a constant 
$C(n) > 0$ such that
$$\widehat{\mu}^{{\rm ess}}(V)
\geq
\frac{C(n)
}{
\omega_{\qb}(V)}.
$$
\end{conjecture}

A $0$-dimensional subvariety 
$V = (\alpha_1 , \ldots, \alpha_n )$ 
of $\gb_{m}^{n}$ is
contained in a 
torsion subvariety if and only if
$\alpha_1 , \ldots, \alpha_n$  are
multiplicatively dependent.

In a similar way, for $\theta > 0$,
$V$ a subvariety
of an abelian variety $A$ defined over a
number field $K$, and $\lc$ a symmetric 
ample line bundle on $A$, we define:
$V(\theta, \lc) := 
\{x \in V(K) \mid 
\widehat{h}_{\lc}(\overline{K}) \leq \theta\}$. 
The {\em essential minimum} of
$V$ is
$$\widehat{\mu}^{{\rm ess}}_{\lc}(V)
:=\{\theta > 0 \mid 
\overline{V(\theta, \lc)} = V\}$$
where $\overline{V(\theta, \lc)}$ is the adherence of
Zariski of
$V(\theta, \lc)$ in $A$.

\begin{conjecture}(David - Philippon, 1996)
\label{davidphilipponCJ1996}
Let $A$ be an abelian variety defined over 
a number field $K$, and $\lc$
a symmetric ample line bundle on $A$.
Let $V/K$ be a proper subvariety
of $A$, $K$-irreducible and such that
$V_{\overline{K}}$ is not the union of 
torsion subvarieties, then
$$\frac{\widehat{h}_{\lc}(V)}
{\deg_{\lc}(V)}
\geq 
~\frac{c(A/K, \lc)}{(\deg_{\lc}(V))^{1/(s - \dim(V))}}$$
for some constant 
$c(A/K, \lc) > 0$ depending on $A/K$ and $\lc$,
where $s$ is the dimension of the smallest
algebraic subgroup containing $V$.
\end{conjecture}

Generalizing \eqref{amorosodavidDOBRO_nDimensionMino}
the {\em higher dimensional Dobrowolski bound} 
takes the following form, proved
in \cite{amorosodavid2} for $\dim (V) = 0$,
in \cite{amorosodavid3} for codim$ (V) = 1$
and
in \cite{amorosodavid4} for varieties of 
arbitrary dimension.

\begin{theorem}[Amoroso - David]
Let $V$ be a subvariety of
$\gb_{m}^{n}$ defined over $\qb$ of codimension $k$.
Let us assume that $V$ is not contained in any
union of proper torsion varieties.
Then,
there exist two constants
$c(n)$ and
$\kappa(n) = (k+1)(k+1)!^k - k$ such that
$$\widehat{\mu}^{{\rm ess}}(V)
\geq
\frac{C(n)
}{
\omega_{\qb}(V)} \frac{1}{(\lo 3 \omega_{\qb}(V))^{\kappa(k)}}.
$$
\end{theorem}
Amoroso and Viada \cite{amorosoviada}
introduced relevant invariants
of a proper
projective
subvariety
$V \subset \pb^n$: e.g.
$\delta(V)$ defined as
the minimal degree  
$\delta$ such that $V$ is, as a set,
the intersection of hypersurfaces of degree
$\leq \delta$.

\begin{theorem}[Amoroso - Viada \cite{amorosoviada2}]
\label{amorosoviadathm1}
Let $V \subset \gb_{m}^{n}$ be a $\qb$-irreducible
variety of dimension $d$. Then, 
for any
$\alpha \in V^{*}(\overline{\qb})$,
$$
h(\alpha) \geq 
\frac{1}{\delta(V)} \frac{1}{
(935 \, n^5 \lo (n^2 \delta(V))^{(d+1)(n+1)^2}}.
$$
\end{theorem}
Following the main Theorem1.3 in 
\cite{amorosoviada2} the essential
minimum admits the following lower bound:
\begin{theorem}[Amoroso - Viada]
\label{amorosoviadathm2}
Let $V \subset \gb_{m}^{n}$ be 
a $\qb$-irreducible
variety of dimension $k$ which is not 
contained in any union of proper torsion
varieties. Then, 
$$
\widehat{\mu}^{{\rm ess}}(V)
\geq 
\frac{1}{\omega_{\qb}(V)} \frac{1}{
(935 \, n^5 \lo (n^2 \omega_{\qb}(V))^{k (k+1)(n+1)}}.
$$
\end{theorem}

\begin{theorem}[Ratazzi \cite{ratazzi3}]
\label{ratazziVARthm}
Let $A$ be a CM abelian variety 
defined over 
a number field $K$, and $\lc$
a symmetric ample line bundle on $A$.
Let $V/K$ be a proper subvariety
of $A$, $K$-irreducible and such that
$V_{\overline{K}}$ is not the union of 
torsion subvarieties. 
Then
$$\frac{\widehat{h}_{\lc}(V)}
{\deg_{\lc}(V)}
\geq
\widehat{\mu}^{{\rm ess}}_{\lc}(V)
\geq  
~
\frac{c(A/K, \lc)}{(\deg_{\lc}(V))^{1/(n - \dim(V)}}
\frac{1}{(\lo (2 \deg_{\lc}(V))^{\kappa(n)}}$$
with $\kappa(n) = (2 n (n+1)!)^{n+2}$,
for some constant 
$c(A/K, \lc) > 0$ depending only on $A/K$ and $\lc$.
\end{theorem}
Ratazzi in \cite{ratazzi3}
obtained more precise minorations of
$\widehat{h}_{\lc}(V)$
in the case where $V$ is an hypersurface.
In \cite{ratazzi2} Ratazzi proves that
the optimal lower bound
given by David and Philippon
\cite{davidphilippon} in Conjecture
\ref{davidphilipponCJ1996} is a consequence
of a Conjecture of David and Hindry 
on the abelian Lehmer problem.

On the way of proving
the relative abelian Lehmer Conjecture,
Carrizosa \cite{carrizosa} 
\cite{carrizosa3}
obtained
a lower bound of the canonical height
of a point $P$ in a
CM abelian variety $A/K$ defined 
 over a number field $K$
in terms of the degree of the field
generated by $P$ over
$K(A_{{\rm tors}})$.
As Corollary of Theorem
\ref{delsinnethm}, with the same constants,
Delsinne obtained the relative result:

\begin{theorem}[Delsinne]
\label{delsinnethm2}
Let $V$ be a subvariety of
$\gb_{m}^{n}$ which is not contained
in any proper algebraic subgroup of $\gb_{m}^{n}$.
Then
$$\widehat{\mu}^{{\rm ess}}(V)
\geq
\left(
c_{3}(n) \omega_{\qb^{{\rm ab}}}(V)
(\lo (3 \omega_{\qb^{{\rm ab}}}(V)))^{\kappa_{1}(n)}
\right)^{-1}
$$
with $c_{3}(n) = c_{1}(n) (\dim(V) +1)$.
\end{theorem}

Concomitantly to the
Lehmer problems,
the geometry of the distribution of the
small points, their Galois orbits,
the limit equidistribution of conjugates
on some subvarieties,
the theorems of finiteness,
were considered e.g. in
Amoroso and David 
\cite{amorosodavid6}
\cite{amorosodavid7}
Bilu \cite{bilu},
Bombieri \cite{bombieri},
Burgos Gil, Philippon, 
Rivera-Letelier and Sombra
\cite{burgosgilphilipponriveraleteliersombra},    
Chambert-Loir \cite{chambertloir},
D'Andrea, Galligo, Narv\'aez-Clauss and Sombra
\cite{dandreagalligosombra}
\cite{dandreanarvezclausssombra},
Favre and Rivera-Letelier
\cite{favreriveraletelier},
Habegger \cite{habegger},
Hughes and Nikeghbali \cite{hughesnikeghbali},
Litcanu \cite{litcanu}, 
Petsche \cite{petsche}
\cite{petsche2},
Pritsker \cite{pritsker3},
Ratazzi and Ullmo \cite{ratazziullmo},
R\'emond \cite{remond},
Rumely \cite{rumely},
Szpiro, Ullmo and Zhang
\cite{szpiroullmozhang},
Zhang \cite{zhang2} \cite{zhang3}.

The type of proof of Dobrowolski
\cite{dobrowolski2}, even revisited  
or generalized (Amoroso and David
\cite{amorosodavid},
Carrizosa \cite{carrizosa3}, 
Laurent \cite{laurent},
Meyer \cite{meyer},
Ratazzi \cite{ratazzi4}),
leads to 
weaker minorations of the height 
than
the better ones obtained by
means of the dynamical zeta 
function of the $\beta$-shift, as in
Theorem \ref{mainDOBROWOLSLItypetheorem},
compared to \eqref{dobrowolski79inequality},
for the classical case. %

\subsection{Analogues of the Mahler measure and Lehmer's problem}
\label{S2.3}

Several generalizations and analogues 
of the Mahler measure
were introduced, for 
which the analogue of the problem of Lehmer 
holds, or not.

The {\em Zhang-Zagier height}
$\mathcal{H}(\alpha)$ of an
algebraic number
$\alpha$ is defined as
$\mathcal{H}(\alpha) = 
\mathcal{M}(\alpha) \mathcal{M}(1-\alpha)$.
After Zhang \cite{zhang}
and Zagier \cite{zagier} \cite{zagier2},
if $\alpha$ is an algebraic number
different from the roots
of $(z^2 - z)(z^2 -z + 1)$,
then
$$\mathcal{H}(\alpha)
\geq \sqrt{\frac{1 + \sqrt{5}}{2}}
=
1.2720196\ldots$$
Doche \cite{doche} \cite{doche2},
using \eqref{dochelimite},
obtains the following 
better minorant:
if $\alpha$ is an algebraic number
different from the roots
of $(z^2 - z)(z^2 -z + 1) 
\Phi_{10}(z) \Phi_{10}(1-z)$,
then  
\begin{equation}
\label{12817770214}
\mathcal{H}(\alpha)
\geq 1.2817770214 =: \eta,
\end{equation}
and the smallest limit point 
of $\{\mathcal{H}(\alpha) \mid
\alpha \in \overline{\qb}\}$
lies in $[1.2817770214, 1.289735]$. 
The lower bound 
$\eta$ is used in Proposition \ref{maincoro7}, 
relative to 
the height one trinomials $G_n$.

Dresden \cite{dresden} introduced 
a {\em generalization of
the Zhang-Zagier height}: given $G$ 
a subgroup of
$PSL(2, \overline{\qb})$, the $G$-orbit height
of $\alpha \in \pb^{1}(\overline{\qb})$ is
$$h_{G}(\alpha)
:=
\sum_{g \in G} h(g \alpha).$$
For $G$
the cyclic group generated by
$$\left(\begin{array}{cc}
1 & 0\\ 0 & 1
\end{array}\right), \left(\begin{array}{cc}
0 & 1\\ -1 & 1
\end{array}\right), 
\left(\begin{array}{cc}
1 & -1\\ 1 & 0
\end{array}\right)$$
Dresden finds,
for $\alpha \neq 0, \neq 1$ not being a primitive
sixth root of unity,
$$h(\alpha) + h(\frac{1}{1-\alpha})+
h(\frac{1}{\alpha}) \geq 0.42179\ldots$$
with equality for $\alpha$ any root of
$(X^2 - X + 1)^3 - X^2 (X -1)^2$; otherwise,
$h_{G}(\alpha) = 0$.

The {\em $G$-invariant Lehmer problem} 
is stated as follows
 in van Ittersum (\cite{ittersum} p. 146):
given $G$  a finite subgroup
of $PSL(2, \qb)$, does there exist a positive constant
$D = D_{G}> 0$ such that
$$h_{G}(\alpha) = 0 \quad \mbox{or}
\quad
h_{G}(\alpha) \geq D, \qquad \mbox{for all}~
\alpha \in \pb^{1}(\overline{\qb})?
$$
If $G$ is trivial this constant $D$ does not exist
\cite{zagier}.
Denote by Orb$_{G}$ the set of all orbits
of the action of $G$ on
$\widehat{\cb} = \cb \cup \{\infty\}$
and Orb$_{G, unit} := 
\{Y \in \mbox{{\rm Orb}}_{G}
\mid \mbox{for all~} \alpha \in Y,
\alpha= 0 \mbox{~or~} 
|\alpha| = 1\}$. 
Dresden's result \cite{dresden} was generalized 
in \cite{ittersum}:
van Ittersum \cite{ittersum} 
proved the $G$-invariant Lehmer problem
under the assumption on $G$
that Orb$_{G, unit}$ is finite.

The {\em (logarithmic) metric Mahler measure}
$\widehat{m}: \mathcal{G} \to [0, \infty)$
was introduced
by Dubickas and Smyth in \cite{dubickassmyth2}
\cite{dubickassmyth3}, where
$$\mathcal{G} := \overline{\qb}^{\times}
/ {\rm Tor}(\overline{\qb}^{\times})
$$
is the $\qb$-vector space
of algebraic numbers modulo torsion, written 
multiplicatively.
For $\underline{\alpha} \in \mathcal{G}$
it is defined by
$$\widehat{m}(\underline{\alpha})
:=
\inf
\Bigl\{
\sum_{n=1}^{N}
\lo {\rm M}(\alpha_n)
\mid
N \in \nb, 
\, \alpha_n \in \overline{\qb}^{\times},
\, \alpha =
\prod_{n=1}^{N}
\alpha_n
\Bigr\}
$$
where the infimum is taken over
all possible ways
of writing any representative 
$\alpha$ of $\underline{\alpha}$ 
as a product of other algebraic numbers.
The construction may be applied to any height 
function \cite{dubickassmyth3} and 
is extremal in the sense that any other
function
$g: \mathcal{G} \to [0, \infty)$
satisfying
(i) $g(\underline{\alpha}) \leq 
\widehat{m}(\underline{\alpha})$ for any 
$\underline{\alpha} \in \mathcal{G}$,
(ii) $g(\underline{\alpha} \,\underline{\beta}^{-1}) 
\leq
g(\underline{\alpha}) 
+
g(\underline{\beta}) $
for any 
$\underline{\alpha}, 
\underline{\beta}
 \in \mathcal{G}$ (triangle inequality),
 is smaller than $\widehat{m}$.

The structure of the completion of 
$\mathcal{G}$, as a Banach 
space over the field
$\rb$ of real numbers,
endowed with the norm deduced from the Weil height
has been studied by Allcock and Vaaler 
\cite{allcockvaaler}.
Indeed, by construction, the Weil height
satisfies: for any $\alpha \in 
\overline{\qb}^{\times}$
and
any root of unity
$\zeta$, 
$h(\alpha)
=
h( \zeta \alpha)$, so that
$h$ extends to 
$h: \mathcal{G} \to \infty$
with the properties:

(i) $h(\underline{\alpha}) = 0$ if and only if
$\underline{\alpha}$ is the 
identity element $\underline{1}$
in $\mathcal{G}$,

(ii) $h(\underline{\alpha}) =
h(\underline{\alpha}^{-1})$ for all 
$\underline{\alpha} \in \mathcal{G}$,

(iii) $h(\underline{\alpha} \, 
\underline{\beta}) \leq
h(\underline{\alpha}) + 
h(\underline{\beta})$
for all
$\underline{\alpha}, \underline{\beta}
\in \mathcal{G}$.

\noindent
These conditions imply
that the map 
$(\underline{\alpha}, \underline{\beta}) \to 
h(\underline{\alpha} \, \underline{\beta}^{-1})$
is a metric on the quotient group
$\mathcal{G}$, on which the
$\qb$-action is defined
by $(r/s , \underline{\alpha})
\to \underline{\alpha}^{r/s}$
by the roots of the polynomials
$z^s - (\zeta \alpha)^r = 0$ for any
$\alpha \in \overline{\qb}^{\times}$
and any root $\zeta$ of unity. 
With the usual absolute value 
$|\cdot|$ on $\qb$,
$h(\alpha^{r/s})
=
|\frac{r}{s}| h(\alpha)$,
and $h$ is
a norm on the $\qb$-vector space
$\mathcal{G}$.

Let $Y$ denote the 
totally disconnected, locally compact, 
Hausdorff space
of all places $y$ of $\overline{\qb}$.
Let
$\mathcal{B}$ be the Borel $\sigma$-algebra
of $Y$. For 
any number field $k \subset \overline{\qb}$
such that $k/\qb$ is Galois and any place $v$ of
$k$, denote
$Y(k,v) := 
\{y \in Y \mid y|v\}$ so that
$$Y = \bigsqcup_{{\rm all places}\, v \,{\rm of}\, k} 
Y(k,v)\qquad
({\rm disjoint~ union}).$$
Let $\lambda$ be the unique regular 
measure on $\mathcal{B}$,
positive on open sets, 
finite on compact sets,
which satisfies: 
$$(i) \quad \lambda (Y(k,v)) = 
\frac{[k_v : \qb_v ]}{[ k : \qb ]}
\quad \mbox{for any Galois $k/\qb$, any place $v$ of
$k$},$$
(ii) $\lambda(\tau E) = \lambda(E)$
for all $\tau \in {\rm Aut}(\overline{\qb}/k)$
and $E \in \mathcal{B}$.
Allcock and Vaaler \cite{allcockvaaler}
proved that 
the (not surjective) map
$$f: \mathcal{G} \to L^{1}(Y, \mathcal{B}, \lambda),
~~\alpha \to f_{\alpha}
~\mbox{given by}~~
f_{\alpha}(y) := \lo \|\alpha\|_{y}$$
is a linear isometry of norm $2 h$, i.e.
$f_{\alpha \beta}(y) =  f_{\alpha}(y)+
f_{\beta}(y)$, 
$f_{\alpha^{r/s}}(y)
=
(r/s) f_{\alpha}(y)$, 
$\int_{Y} |f_{\alpha}(y)| d\lambda(y) =
2 h(\alpha)$, with the property:
$\int_{Y} f_{\alpha}(y) d\lambda(y) =0$.
Denote by
$\mathcal{F} := f(\mathcal{G})$ the image 
of $\mathcal{G}$ in 
$L^{1}(Y, \mathcal{B}, \lambda)$
and
$\chi := \{F \in L^{1}(Y, \mathcal{B}, \lambda) \mid
\int_{Y} F(y) d\lambda(y) = 0\}$ 
the co-dimension one linear subspace of
$L^{1}(Y, \mathcal{B}, \lambda)$.
They proved that
$\mathcal{F}$ is dense
in $\chi$ (\cite{allcockvaaler} Theorem 1),
i.e. that $\chi$
is the completion of $(\mathcal{G},h)$,
up to isometry. They also proved
that, for any real $1 < p < \infty$,
$\mathcal{F}$ is dense
in $L^{p}(Y, \mathcal{B}, \lambda)$
(\cite{allcockvaaler} Theorem 2),
and $\mathcal{F}$ is dense in
the Banach space 
$\mathcal{C}_{0}(Y)$
of continuous
real valued  functions
on $Y$ which vanish at infinity,
equipped with the sup-norm
 (\cite{allcockvaaler} Theorem 3).

Fili and Miner \cite{filiminer2} 
proved that the space $\mathcal{F}$
admits linear operators canonically
associated to the Mahler 
measure and to the $L^p$ norms
on $Y$. They introduced norms, 
called {\em Mahler $p$-norms},
from
orthogonal decompositions
of $\mathcal{F}$,
and, in this context, obtained
extended formulations, called
{\em $L^p$ Lehmer Conjectures},
of the Lehmer Conjecture
and the Conjecture of Schinzel-Zassenhaus.
Namely, let $\mathcal{K}$ be
the set of finite 
extensions of $\qb$
and 
$\mathcal{K}^{G}
:=
\{K \in \mathcal{K} \mid 
\sigma (K) = K ~\mbox{for all}~ \sigma \in 
{\rm Gal}(\overline{\qb}/\qb)
\}$. For each 
$K \in \mathcal{K}$, 
denote by
$V_K := \{f_{\alpha} \mid 
\alpha \in K^{\times}/ {\rm Tor}(K^{\times})\}$ 
the $\qb$-vector 
subspace of $\mathcal{F}$ constituted
by the nonzero elements of $K$
modulo torsion, and, for $n \geq 0$,
$V^{(n)}:= 
\sum_{K \in \mathcal{K}, [K:\qb]\leq n} V_K$.
Denote by
$\langle f,g\rangle
=
\int_{Y} \, f(y) g(y) d\lambda(y)$
the inner product on $\mathcal{F}$.

\begin{theorem}[Fili - Miner]
\label{filiminerthm}
(i) There exist projection operators
$T_{K}: \mathcal{F} \to \mathcal{F}$
for each $K \in \mathcal{K}^{G}$ such that
$T_{K}(\mathcal{F}) \subset V_K$,
$T_{K}(\mathcal{F}) \bot \,T_{L}(\mathcal{F})$
for all $K, L \in \mathcal{K}^{G}$,
$K \neq L$, with respect to the inner product on
$\mathcal{F}$,
and
$$\mathcal{F} = \bigoplus_{K \in \mathcal{K}^{G}}
T_{K}(\mathcal{F}),$$
(ii) for all $n \geq 1$,
there exist projections
$T^{(n)}: \mathcal{F} \to \mathcal{F}$
such that
$T^{(n)}(\mathcal{F})
\subset V^{(n)}$,
$T^{(m)}(\mathcal{F}) \bot\, T^{(n)}(\mathcal{F})$
for all $m \neq n$, and
$$\mathcal{F} = \bigoplus_{K \in \mathcal{K}^{G}}
T^{(n)}(\mathcal{F}),$$
(iii) for every 
$K \in \mathcal{K}^{G}$ and $n \geq 0$, the projections
$T_K$ and
$T^{(n)}$
commute.
\end{theorem}
Now, for any $\alpha \in \overline{\qb}^{\times}$
and any real number $1 \leq p \leq \infty$,
let
$h_{p}(\alpha) := \|f_{\alpha}\|_p$
(recalling that $h_{1}(\alpha) =
2 h(\alpha)$).

\begin{conjecture}(Fili - Miner)($L^p$ Lehmer Conjectures)
\label{filiminerLpLehmerCJ}
For any real number
$1 \leq p \leq \infty$, there exists a
real constant $c_p > 0$
such that
$$(*_p)\qquad m_{p}(\alpha) :=
\deg_{\qb}(\alpha) \, h_{p}(\alpha)
\geq c_p
\qquad \mbox{for all}~
\alpha \in \overline{\qb} \setminus
{\rm Tor}(\overline{\qb}).
$$
\end{conjecture}
For $p=1$ Conjecture \ref{filiminerLpLehmerCJ}
is exactly the classical Lehmer Conjecture.
Moreover, Fili and Miner 
(\cite{filiminer2}, Proposition 4.1) 
proved that, for
$p=\infty$, Conjecture \ref{filiminerLpLehmerCJ}
is exactly the classical Conjecture of Schinzel-Zassenhaus. 

The operator $M: \mathcal{F} \to \mathcal{F},
f \to \sum_{n=1}^{\infty} n T^{(n)} f$ is well-defined, 
unbounded, invertible, and is always a finite sum. 
The norm 
$f \to \|Mf\|_p$
is called the Mahler $p$-norm on $\mathcal{F}$.
For any $f \in \mathcal{F}$,
let
$d(f) := \min\{\deg_{\qb}(\alpha) \mid
\alpha \in \overline{\qb}^{\times},
f_{\alpha} = f \}$ be the smallest degree
possible in the class of $f$.
For any $f \in \mathcal{F}$,
the {\em minimal field}, denoted by
$K_f$, is defined to be the minimal element
of the set 
$\{K \in \mathcal{K} \mid f \in V_K\}$.
Let $\delta(f) = [K_f : \qb]$.
The $P_K$ operators on $\mathcal{F}$
are defined from the
$T_K$ operators as:
$P_K := \sum_{F \in \mathcal{K}^G , F \subset K}
T_F$.
An element $f \in \mathcal{F}$ is said to be
{\em Lehmer irreducible} (or {\em representable})
if 
$\delta(f) = d(f)$. 
The set of 
Lehmer irreducible elements of 
$\mathcal{F}$ is 
denoted by $\mathcal{L}$.
An element $f \in \mathcal{F}$ is said to be
{\em projection irreducible} if
$P_{H}(f)=0$
for all propers subfields
$H$ of $K_f$. The set of 
projection irreducible elements of 
$\mathcal{F}$ is 
denoted by $\mathcal{P}$.
Let $\mathcal{U} =
\{f \in \mathcal{F} \mid {\rm supp}_{Y}(f) \subset
Y(\qb, \infty)\}$ be the subset of algebraic units.

\begin{theorem}[Fili - Miner]
\label{filiminerthm2}
For every real number
$1 \leq p \leq \infty$, the $L^p$ Lehmer Conjecture
$(*_p)$ holds if and only
the following minoration on the Mahler $p$-norms 
holds
$$(**_p) \qquad
\|\sum_{n=1}^{\infty} n T^{(n)} f_{\alpha}\|_p
\geq c_p
\qquad
\mbox{for all}~ f_{\alpha} \in \mathcal{L}
\cap \mathcal{P} \cap \mathcal{U}, \,
f_{\alpha} \neq 0.
$$
Further, for $1 \leq p \leq q \leq \infty$,
if $(**_p)$ holds, then
$(**_q)$ also holds.
\end{theorem}

An element $f_{\beta} \in \mathcal{F}$ 
is said to be a {\em Pisot number}, 
resp. a {\em Salem number},
if it has a representative
$\beta \in \overline{\qb}^{\times}$ which 
is a Pisot number, resp. a Salem number.
Fili and Miner (\cite{filiminer2},
Prop. 4.2 and Prop. 4.3)
proved that every Pisot number
and every Salem number is Lehmer irreducible,
moreover that every Salem number is also
projection irreducible.
A {\em surd} is an element
$f \in \mathcal{F}$ such that
$\delta(f) =1$, i.e. for which
$K_f = \qb$ and $\|Mf\|_p = \|f\|_p = h_{p}(f)$;
a surd is projection irreducible.

The {\em $t$-metric Mahler measure}, 
was introduced
by 
Samuels
\cite{jankauskassamuels} 
\cite{samuels2} \cite{samuels4}.
For $t \geq 1$, the $t$-metric Mahler measure
is defined by
$${\rm M}_{t}(\alpha)
:=
\inf
\Bigl\{
\Bigl(
\sum_{n=1}^{N}
(\lo {\rm M}(\alpha_n))^t 
\Bigr)^{1/t}
\mid
N \in \nb, 
\, \alpha_n \in \overline{\qb}^{\times},
\, \alpha =
\prod_{n=1}^{N}
\alpha_n
\Bigr\}
$$
and, by extension, for $t=\infty$, by
$${\rm M}_{\infty}(\alpha)
:=
\inf
\Bigl\{
\max_{1 \leq n \leq N}
\bigl\{
\lo {\rm M}(\alpha_n)
\big\} 
\mid
N \in \nb, 
\, \alpha_n \in \overline{\qb}^{\times},
\, \alpha =
\prod_{n=1}^{N}
\alpha_n
\Bigr\}.
$$
For $t=1$, ${\rm M}_{1}$ is the 
metric Mahler measure
introduced in \cite{dubickassmyth2}. 
These functions satisfy an analogue of 
the triangle inequality
\cite{jankauskassamuels}, and
the map
$(\alpha, \beta) \to 
{\rm M}_{t}(\alpha \beta^{-1})$
defines a metric on
$\mathcal{G} := \overline{\qb}^{\times}
/ {\rm Tor}(\overline{\qb}^{\times})
$
which induces the discrete topology 
if and only if Lehmer's Conjecture is true.
For $t \in [1, \infty]$ and
$\alpha \in \overline{\qb}$ we say that
the infimum in ${\rm M}_{t}(\alpha)$
is attained by $\alpha_1 , \ldots, \alpha_n$
if the equality case holds: i.e.,
for $1 \leq t < \infty$,
 if
${\rm M}_{t}(\alpha) = \bigl(
\sum_{n=1}^{N}
(\lo {\rm M}(\alpha_n))^t 
\bigr)^{1/t}$
and, for $t=\infty$, 
${\rm M}_{\infty}(\alpha)
=
\max_{1 \leq n \leq N}
\{
\lo {\rm M}(\alpha_n)\} $.
For $\alpha \in 
\overline{\qb}$, denote
by $\kb_{\alpha}$ the Galois
closure of $\qb(\alpha)/\qb$,
and let
${\rm Rad}(\kb_{\alpha}) :=
\{\beta \in \overline{\qb} \mid
\beta^m \in \kb_{\alpha}~
\mbox{for some} ~m \in \nb\}$.
Following a conjecture of Dubickas and Smyth
\cite{dubickassmyth2}, Samuels
\cite{samuels2} \cite{samuels3}
proved that the infimum of
${\rm M}_{t}(\alpha)$
is attained in ${\rm Rad}(\kb_{\alpha})$.
Whether this infimum is attained in proper
subsets of $\overline{\qb}$
leads to many open questions
(\cite{jankauskassamuels}, 
Question 1.5), though
Jankauskas and 
Samuels proved some results for certain cases of 
decompositions of rational numbers
in prime numbers 
(\cite{jankauskassamuels}, 
Theorem 1.3,
Theorem 1.4). In particular
for $\alpha \in \qb$,
they proved that the infimum of
${\rm M}_{t}(\alpha)$
may be attained using only rational points.

The $p$-metric, resp. 
the $t$-metric, constructions of 
Fili and Miner \cite{filiminer} 
and Jan-kauskas and Samuels
\cite{jankauskassamuels}
are of different nature, though they
are esentially the same for $p=1$.
Fili and Miner \cite{filiminer}  
studied the minimality of
the Mahler measure by several 
norms, related to the metric 
Mahler measure 
introduced in \cite{dubickassmyth2},
using results of  
de la Masa and Friedman
\cite{delamasafriedman} 
on heights of algebraic numbers 
modulo multiplicative group actions.
Fili and Miner \cite{filiminer}
introduced an infinite collection $(h_t)_t$
of vector space norms on $\mathcal{G}$,
called {\em $L^t$ Weil heights}, $t \in [1, \infty]$,
which satisfy extremality properties,
and {\em minimal logarithmic $L^t$ Mahler 
measures}
$(m_{t})_t$ from $(h_t)_t$.
By definition, for $\kb$ a number field, 
$\Sigma_{\kb}$ its set of places and 
$||.||_{\nu}$ the absolute value on $\kb$ extending the usual
$p$-adic absolute value on $\qb$ if 
$\nu$ is finite or the 
usual archimedean absolute value
if $\nu$ is infinite, for
$1 \leq t < \infty$ real,
$$h_{t}(\alpha) :=
\Bigl(
\sum_{\nu \in \Sigma_{\kb}}
\frac{[\kb_{\nu} : \qb_{\nu}]}{[\kb : \qb]} 
. \left|\lo ||\alpha||_{\nu}\right|^{t}
\Bigr)^{1/t}, \qquad \quad
\alpha \in \kb^{\times}
$$
for which $2 h = h_1$ \cite{allcockvaaler}, 
and
$$
h_{\infty}(\alpha) :=
\sup_{\nu \in \Sigma_{\kb}}
\left|\lo ||\alpha||_{\nu}\right|,
\qquad \quad
\alpha \in \kb^{\times} .
$$
They reformulated Lehmer's Conjecture in this context
(with $1 \leq t < \infty$).
Because $h_{\infty}$ serves as a generalization 
of the (logarithmic) house of an algebraic integer, 
they also reformulated the Conjecture of
 Schinzel and Zassenhaus.
In \cite{jankauskassamuels}
Jankauskas and Samuels investigate the
$t$-metric Mahler measures of surds 
and rational numbers.

The
{\em ultrametric Mahler measure}
was introduced
by Fili and Samuels \cite{filisamuels}
\cite{samuels3}
to give
a projective height of $\mathcal{G}$, 
which satisfies the strong triangle inequality.
The ultrametric Mahler measure induces 
the discrete topology on $\mathcal{G}$ if
and only if Lehmer's Conjecture is true.

{\em Two $p$-adic Mahler measures}
are introduced by 
Besser and Deninger
in \cite{besserdeninger} in view
of developping natural analogues
of the classical logarithmic 
Mahler measures of Laurent polynomials,
following Deninger \cite{deninger}.
The $p$-adic analogue of Deligne cohomology
is now Besser's modified syntomic cohomology, 
but with
the same symbols in the algebraic $K$-theory groups.  
For one $p$-adic Mahler measure the authors show that
there is {\it no analogue} of Lehmer's problem. 

{\em Generalized Mahler measures, higher Mahler measures and 
multiple $k$-higher Mahler measures}
were introduced by 
Gon and Oyanagi \cite{gonoyanagi}, resp.
Kurokawa, Lal\'in and Ochiai
\cite{kurokawalalinochiai} 
and reveal deep connections between zeta functions, polylogarithms, multiple $L$- functions
(Sasaki \cite{sasaki}) and multiple sine functions. 
For any $n \geq 1$, given 
$P_1 , \ldots, P_s \in \cb[x_1 , \ldots, x_n]$
(not necessarily distinct) nonzero polynomials,
the generalized Mahler measure is defined by
$
{\it m}_{\max}(P_1 , \ldots, P_s) :=$
$$\frac{1}{(2 \pi i)^n}
\int_{\tb^n}
\max\{\lo |P_{1}(x_1 , \ldots , x_n)|,
\ldots, \lo |P_{s}(x_1 , \ldots , x_n)|\}
\frac{d x_1}{x_1} \ldots \frac{d x_n}{x_n},$$
the multiple Mahler measure by
$
{\it m}(P_1 , \ldots, P_s) :=$
$$\frac{1}{(2 \pi i)^n}
\int_{\tb^n}
\lo |P_{1}(x_1 , \ldots , x_n)|
\ldots \lo |P_{s}(x_1 , \ldots , x_n)|
\frac{d x_1}{x_1} \ldots \frac{d x_n}{x_n}
,$$
the $k$-higher Mahler measure of $P$ by
${\it m}_{k}(P):=
{\it m}(P, \ldots, P) =$
$$\frac{1}{(2 \pi i)^n}
\int_{\tb^n}
\lo^{k} |P_{1}(x_1 , \ldots , x_n)|
\frac{d x_1}{x_1} \ldots \frac{d x_n}{x_n}
,$$
The $k$-higher Mahler measures are deeply 
related to
the zeta Mahler measures, and their derivatives, 
introduced by
Akatsuka \cite{akatsuka}.
The {\it problem
of Lehmer} for $k$-higher Mahler measures
is considered by Lal\'in
and Sinha in \cite{lalinsinha}.
Asymptotic formulas of ${\it m}_{k}(P)$,
with $k$,
are given in
\cite{biswas} and \cite{lalinsinha}, 
for some families of polynomials $P$.
Analogues of Boyd-Lawton's Theorem 
are studied in
Issa and Lal\'in \cite{issalalin}.
By analogy with Deninger's approach,
the motivic reinterpretation
of the values of $k$-higher Mahler measures
in terms of Deligne cohomology
is given by Lal\'in in \cite{lalin3}. 

{\em The logarithmic Mahler measure 
${\rm m}_G$ over a compact
abelian group} $G$ is introduced
by Lind \cite{lind4}. The group is equipped
with the normalized Haar measure
$\mu$. By Pontryagin's duality
the dual group 
$\widehat{G}$ (characters) is discrete and
the class of functions $f$ to be considered
is $\zb[\widehat{G}]$. For
$f \in \zb[\widehat{G}]$
$${\rm m}_{G}(f) = 
\int_{G}\lo |f| d\mu
\quad \mbox{generalizes}\quad
{\rm m}(f) = \int_{0}^{1} \lo |f(e^{2 i \pi t})| dt
~\mbox{for}~ f \in \zb[x^{\pm 1}].$$ 
The {\it Lehmer constant of $G$} is then
defined by
$$\lambda(G) := 
\inf\{{\rm m}_{G}(f) \mid f \in \zb[\widehat{G}],
{\rm m}_{G}(f) > 0\}.$$
The author considers several groups $G$ (connected, finite)
and the {\it problem of Lehmer} 
in each case.
The classical Lehmer's problem asks whether
$\lambda(\tb)=0$, where $\tb = \rb/\zb$.
Let $n \geq 2$,
denote by $\rho(n)$ the smallest prime 
number that 
does not divide $n$. Lind proves that
$\lambda(G) = \lambda(\tb)$ for any 
nontrivial 
connected compact abelian group,
and
$\lambda(\zb/n \zb) \leq \frac{\lo \rho(n)}{n}$ 
for $n \geq 2$.
This Lehmer's constant has been named 
Lind-Lehmer's 
constant more recently.
It is known in some cases
\cite{pinnervaaler2}.
Kaiblinger \cite{kaiblinger}
obtained
results on $\lambda(G)$
for finite cyclic groups $G$ of 
cardinality not divisible by 420;
Pigno and Pinner 
\cite{pignopinner} solved the case $|G|=420$.
De Silva and Pinner \cite{desilva}
\cite{desilvapinner} made progress
on noncyclic finite abelian groups
$G = \zb_{p}^{n}$,
then Pigno, Pinner and
Vipismakul \cite{vipismakul} 
\cite{pignopinnervipismakul}
on general $p$-groups
$G_{p} =
\zb_{p^{l_1}} \times \ldots \times \zb_{p^{l_n}}$
and $G = \zb_{m} \times G_p$ for
$m$ not divisible by $p$.

An areal analogue
of Mahler's measure
is reported by Pritsker
\cite{pritsker}, linked to
Hardy and Bergman normed spaces of functions
on the unit disk.

{\em Lehmer's problems in 
positive characteristic and Drinfeld modules}:
let $k = \fb_{q}(T)$  be the fraction field
of the ring $\fb_{q}[T]$ of polynomials
with coefficients in the finite field
$\fb_{q}$ ($p$ is a prime number 
and $q$ a power of $p$). 
Let $k_{\infty} = \fb_{q}((1/T))$
be the completion of $k$
for the $1/T$-adic valuation $v$.
The valuation, still denoted by $v$, 
is extended to
the algebraic closure $\overline{k}$
of $k$,
resp. 
$\overline{k_{\infty}}$ 
of  $k_{\infty}$. 
The degree $\deg(x)$ of $x \in k_{\infty}$
is equal to the integer-valued $-v(x)$, 
with the convention
$\deg(0)= -\infty$.
Let $t$ denote a formal variable.
By definition a $t$-module of 
dimension $N$ and rank $d$
on $\overline{k}$ 
is given 
by the additive group
$(\gb_{a})^N$ and an injective
ring homomorphism
$\Phi: \fb_{q}[T] \to {\rm End}(\gb_{a})^N$
which satisfies:
$$\Phi(t) = a_0 F^0 + \ldots +  a_d F^d ,$$
where $F$ is the Frobenius endomorphism
on $(\gb_{a})^N$ and $a_0, a_1 , \ldots , a_d$ are
$N \times N$ matrices with coefficients in 
$\overline{k}$. In
\cite{denis} Denis constructed a canonical height
$\widehat{h} = \widehat{h}_{\Phi}$
on $t$-modules for which $a_d$ is invertible, 
from the Weil height. 
Denis formulated Lehmer's problem
for $t$-modules as follows, in two steps: 
(i) for
$\alpha \in (\overline{k})^n$, 
defined over 
a field of degree $\leq \delta$, 
not in the torsion
of the $t$-module, does there exist
$c(\delta) =
c_{a_0 , \ldots, a_d, N, d, q, F}(\delta)
> 0$ such that 
$\widehat{h}(\alpha) \geq c(\delta)$?;
(ii) if (i) is satisfied, 
on a Drinfeld module of rank $d$,
does there exist $c > 0$ such that, 
for any $\alpha$ not belonging to the torsion,
$$\widehat{h}(\alpha) \geq \frac{c}{\delta}?$$
The second problem is the
extension of the
classical Lehmer problem
\cite{pacheco}.
Denis partially solved Lehmer's problem
(\cite{denis} Theorem 2) 
in the case of Carlitz modules, i.e. with
$N=1$ and $d=1$ for which
$\Phi(T)(x) = T x + x^q$. He obtained
the following minoration
which is an analogue of Laurent's Theorem 
\ref{laurentthm}
for CM elliptic curves (elliptic Lehmer problem)
and Dobrowolski's inequality \eqref{dobrowolski79inequality}:

\begin{theorem}[Denis]
\label{denisthm}
There exists a real number $\eta > 0$  
such that, for any $\alpha$ belonging to the
regular separable closure of $k$, not to the torsion, 
of degree $\leq \delta$, the minoration holds:
$$\widehat{h}(\alpha) ~\geq~ 
\eta \, \frac{1}{\delta}
\left(
\frac{\lo \lo \delta}{\lo \delta}
\right)^3 $$
(the real number $\eta$ is effective and 
computable from $q$). 
\end{theorem}
Grandet-Hugot in \cite{grandethugot}
studied analogues of Pisot and Salem numbers
in fields of formal series: 
$x \in k_{\infty}$ is a Salem number
if it is algebraic on $k$, $\deg(x) > 0$,
and all its conjugates satisfy: $\deg(x_i) \leq 0$.
In this context Denis (\cite{denis2} Theorem 1)
proved the fact that there is no Salem number
too close to 1, namely:
 
\begin{theorem}[Denis]
\label{denisthm2}
Let $\alpha \in \overline{k}_{\infty}$ 
having at least one conjugate
in $k_{\infty}$. If 
$\alpha$ does not belong to the torsion, 
is of degree $D$ on $k$,
then
$$\widehat{h}(\alpha) ~\geq~ 
\frac{1}{q \, D}
$$ 
\end{theorem}
Extending the previous results,
Denis (\cite{denis2} Theorem 3)
solved Lehmer's problem
for the following infinite family
of $t$-modules:

\begin{theorem}[Denis]
\label{denisthm3}
Let $\Phi(t) = a_0 F^0 + a_1 F + \ldots a_{d-1} F^{d-1} + F^d$ be a $t$-module of dimension 1 such that
$a_i \in k_{\infty} \cap \overline{k}$,
$0 \leq i \leq d-1$,
is integral over $\fb_{q}[T]$.
Then there exists a real number
$c_{\Phi} > 0$ depending only 
upon $\Phi$, such that, if
$\alpha$ is an algebraic 
element of $k_{\infty}$, not in the torsion,
of degree $D$ on $k$, then
$$\widehat{h}_{\Phi}(\alpha) ~\geq~ 
\frac{c_{\Phi}}{ D}
$$ 
\end{theorem}
The abelian Lehmer problem
for Drinfeld modules was solved
by
David and Pacheco \cite{davidpacheco}
using Denis's construction of 
the canonical height (with $A = \fb_{q}[T]$):

\begin{theorem}[David - Pacheco]
\label{davidpachecothm}
Let $K/k$ be a finite extension, $\overline{K}$
an algebraic closure of $K$,
and $K^{{\rm ab}}$ the largest
abelian extension of $K$ in
$\overline{K}$. Let
$\phi: A \to K\{\tau\}$ 
be a Drinfeld module of rank $\geq 1$.
Then there exists
$c = c(\phi, K) > 0$
which depends only upon $\phi$ and $K$
such that,
for any 
$\alpha \in  
K^{{\rm ab}}$, 
not being in the torsion,
$$\widehat{h}_{\Phi}(\alpha) ~\geq~ 
c .$$
\end{theorem}
In \cite{ghioca2}
Ghioca investigates statements, for Drinfeld 
modules of generic characteristic,  
which would imply
that the classical Lehmer problem for 
Drinfeld modules is true. In
\cite{ghioca} Ghioca obtained several 
Lehmer type inequalities
for the height of nontorsion points 
of Drinfeld modules. 
Using them, as consequence of Theorem 
\ref{ghiocathm} below, 
Ghioca proved several Mordell-Weil
type structure theorems for Drinfeld
modules over certain infinitely
generated fields (the definitions of the terms
can be found in \cite{ghioca}):

\begin{theorem}[Ghioca]
\label{ghiocathm}
Let $K/\fb_{q}$ be a field extension, and
$\phi: A \to K\{\tau\}$ be a Drinfeld module.
Let $L/K$ be a finite field extension. Let $t$ be a 
non-constant element of $A$ and assume that 
$\phi_t = \sum_{i=0}^{r} a_i \tau^i$ is monic.
Let $U$ be a good set of valuations on $L$
and let $C(U)$ be the field of constants 
with respect to $U$.   
Let $S$ be the finite set of valuations
$v \in U$ such that $\phi$ has 
bad reduction at $v$. The degree of the valuation $v$
is denoted by $d(v)$. Let $x \in L$.

a) If $S$ is empty, then either
$x \in C(U)$ or there exists
$v \in U$ such that
$\widehat{h}_{U,v}(x) \geq d(v)$,

b) If $S$ is not empty, then either
$x \in \phi_{tors}$, or there exists
$v \in U$ such that
$$\widehat{h}_{U,v}(x) > 
\frac{d(v)}{q^{2 r + r^2 N_{\phi} |S|}}
\geq  \frac{d(v)}{
q^{r \bigl(2 +(r^2 + r) |S|\bigr)}} .
$$
Moreover, if $S$ is not empty and
$x \in \phi_{tors}$, then there exists a polynomial
$b(t) \in \fb_{q}(t)$ of degree at most
$r N_{\phi} |S|$ such that $\phi_{b(t)}(x) = 0$.
\end{theorem}

Let $K$ be a finitely generated 
field extension
of $\fb_q$, and $K^{alg}$ an algebraic closure
of $K$. Ghioca \cite{ghioca2} developped
global heights associated to a Drinfeld
module $\phi: A \to K\{\tau\}$ and, for 
each divisor $v$, 
local heights $\widehat{h}_v : K^{alg} \to \rb^+$ 
associated to $\phi$.
For Drinfeld modules of 
finite characteristic Ghioca \cite{ghioca2} 
obtained Lehmer type inequalities with the local 
heights, extending 
the classical Lehmer problem:

\begin{theorem}[Ghioca]
\label{ghiocathm2}
For $\phi: A \to K\{\tau\}$ a Drinfeld
module of finite characteristic, there 
exist two positive constants 
$C$ and $r$ depending only on $\phi$
such that if $x \in K^{alg}$
and $v$ is a place of $K(x)$
for which $ \widehat{h}_{v}(x) > 0$, then
$$ \widehat{h}_{v}(x) \geq \frac{C}{d^r}$$
where $d = [K(x) : K]$.
\end{theorem}
Bauch\`ere \cite{bauchere} generalized
David Pacheco's Theorem \ref{davidpachecothm}
to Drinfeld modules having complex multiplications, proving the abelian Lehmer problem in this context:

\begin{theorem}[Bauch\`ere]
\label{baucherethm}
Let $\phi$ be a $A$-Drinfeld module
defined over $\overline{k}$
having complex multiplications.
Let $K/k$ be a finite field extension,
$L/K$ a Galois extension (finite or infinite)
with Galois group $G = $Gal$(L/K)$.
Let $H$ be a subgroup of the center of 
$G$ and $E \subset L$
the subfield fixed by $H$.
Let $d_0$ be an integer. 
We assume that there exists a finite 
place $v$ of $K$
such that $[E_{w} : K_{v}] \leq d_0$\,
for every place $w$ of $E$, $v | w$.
Then there exists a constant
$c_0 = c_{0}(\phi) > 0$ such that, for any
$\alpha \in L$, not belonging to the torsion
for $\phi$, 
$$\widehat{h}_{\phi}(\alpha) \geq 
\frac{1}
{q^{c_0 \, d(v)\, d_{0}^{2} \, [K:k]}} .$$
\end{theorem}
Theorem \ref{baucherethm}
is the analogue
of a result obtained by
Amoroso, David and Zannier 
\cite{amorosodavidzannier}
for the multiplicative group.
Theorem \ref{baucherethm} is particularly interesting
when $L/K$ is infinite. Bauch\`ere
\cite{bauchere} 
deduced special minorations of the heights
$\widehat{h}_{\phi}(\alpha)$
in two Corollaries, for
$L= K^{ab}$, and in the case where
the subgroup $H$ is trivial.

\subsection{In other domains}
\label{S2.4}

The conjectural 
discontinuity of the Mahler measure 
${\rm M}(\alpha)$, $\alpha \in \overline{\qb}$,
at 1 
has consequences
in different domains of 
mathematics. It  
is linked to the notions of 
``smallest complexity", 
``smallest growth rate",
``smallest geometrical dilatation", 
``smallest geodesics", ``smallest Salem number" or 
``smallest topological entropy" 
(Hironaka \cite{hironaka4}). 
We will keep an interdisciplinary 
viewpoint as in the recent survey 
\cite{smyth6}
by C. Smyth
and refer the reader to 
\cite{ghatehironaka}
\cite{smyth6}; 
we only
mention below a few more or less new
results. 
The smallest Mahler measures, or smallest
Salem numbers,
correspond to peculiar 
geometrical constructions in their respective
domains.

\subsubsection{Coxeter polynomials, graphs, Salem trees}
\label{S2.4.1}

Let $\Gamma = (\Gamma_0 , \Gamma_1 )$ 
be a simple graph
with set of enumerated vertices 
$\Gamma_0 = \{v_1 , \ldots, v_n\}$, $\Gamma_1$ 
being the set of edges where
$(v_i , v_j) \in \Gamma_1$ if there is an edge 
connecting the vertices $v_i$ and $v_j$.
The adjacency
matrix of $\Gamma$ is Ad$_{\Gamma} :=
[a_{ij}] \in M_{n}(\zb)$ where $a_{i j} = 1$,
if $(v_i , v_j) \in \Gamma_1$ and
$a_{i j} = 0$ otherwise. 
Assume that $\Gamma$ is a tree. 
Denote by
$W_{\Gamma}$ the Weyl group of
$\Gamma$, generated by the reflections
$\sigma_1 , \sigma_2 , \ldots, \sigma_n $
and $\Phi_{\Gamma} :=
\sigma_1 \cdot \sigma_2 \cdot \ldots \cdot \sigma_n
\in W_{\Gamma}$
the Coxeter transformation of
$\Gamma$. The Coxeter polynomial of
$\Gamma$ is the characteristic polynomial
of the Coxeter transformation 
$\Phi_{\Gamma}: \rb^n \to \rb^n$:
$${\rm cox}_{\Gamma}(x) := 
\det(x \cdot {\rm Id}_n - \Phi_{\Gamma}) \in \zb[x].$$
Coxeter (1934) showed remarkable 
properties of
the roots of the Coxeter polynomials.
Coxeter polynomials were extensively
studied for $\Gamma$ any simply laced Dynkin diagram
$\ab_n , \db_n $ and $\eb_n$.
For $\Gamma = \eb_n$, 
Gross, Hironaka and McMullen
\cite{grosshironakamcmullen} have obtained
the factorization of 
Coxeter polynomials ${\rm cox}_{\Gamma}(x)$
as products of cyclotomic polynomials
and irreducible Salem polynomials.
In particular, 
${\rm cox}_{\eb_{10}}(x) :=
x^{10} + x^9
-x^7 - x^6
-x^5 - x^4
-x^3 + x
+1$ is Lehmer's polynomial.
A tree $\mathcal{T}$ is said to be 
{\em cyclotomic}, 
resp. {\em a Salem tree},
if ${\rm cox}_{\mathcal{T}}(x)$ is a 
product of cyclotomic polynomials, 
resp. the product of
cyclotomic polynomials by an  
irreducible Salem polynomial. 
Such objects generalize $\eb_n$ as far
as their Coxeter polynomials 
remains of the same form. 
Evripidou \cite{evripidou},
following Lakatos \cite{lakatos} 
\cite{lakatos2}
\cite{lakatos4}
\cite{lakatos5}
and
\cite{grosshironakamcmullen},
obtained
structure theorems and formulations
for the Coxeter polynomials
of families of Salem trees, for
the spectral radii
of the respective Coxeter transformations.
Lehmer's problem asks whether
there exists a Salem tree of minimal
Salem number; 
what would be its decomposition?

The Mahler measure 
${\rm M}(G)$ of a finite graph
$G$, with
$n$ vertices,
is introduced in
McKee and Smyth 
\cite{mackeesmyth}.
If $D_{G}(z)$ is the characteristic polynomial
of $G$, then the reciprocal
integer polynomial associated with $G$ is
$z^n D_{G}(z + 1/z)$. 
The Mahler measure of this later
polynomial is the Mahler measure
${\rm M}(G)$ of
$G$; explicitely,
$${\rm M}(G) = \prod_{D_{G}(\chi) =0, |\chi|>2}
\frac{1}{2}(|\chi| + \sqrt{\chi^2 - 4}).$$
Cooley, McKee and Smyth
\cite{cooleymackeesmyth}
\cite{mackeesmyth}
\cite{mackeesmyth2}
\cite{mackeesmyth3}
studied Lehmer's problem
from various constructions of finite graphs. 
They prove (\cite{cooleymackeesmyth} Theorem 1
and Figures 1 to 3)
that 
every connected non-bipartite graph
that has Mahler measure smaller than
the golden mean $1.618\ldots$
is one of the following type: 
(i) an odd cycle,
(ii) a kite graph, 
(iii) a balloon graph,
or (iv) one of the eight sporadic examples
$Sp_a , \ldots , Sp_h$.

\subsubsection{Growth series of groups, Coxeter groups, Coxeter systems}
\label{S2.4.2}

Let $G$ be an infinite group. Assume that $G$ admits a finite
generating set $S$. Define the 
length of an element
$g$ in $G = (G, S)$ to be the 
least nonnegative integer $n$
such that $g$ can be expressed as
a product of $n$ elements from
$S \cup S^{-1}$. For every nonnegative
integer $n$ let $N_{S}(n)$
be the number of elements
in $G$ with length $n$.
Following Milnor \cite{milnor}
the growth series of 
the group $(G,S)$ is by definition
$$f(x) = \sum_{n = 1}^{\infty} N_{S}(n) x^n ,
\qquad \mbox{for which}~ N_{S}(n) \leq (2|S|)^n .$$
The asymptotic {\em growth rate}
of $G = (G,S)$, finite and $\geq 1$, 
is by definition
$$\limsup_{n \to \infty} \,(N_{S}(n))^{1/n} ,$$
its inverse, positive, being
the radius of convergence
of $f(x)$.
A Coxeter group $G$, with $S$ being 
a finite generating set for
$G$, is a group where every element
of $S$ has order two and all the other defining 
relators for $G$ are of the form
$(g \,h)^{m(g, h)} = 1_{G}$
where $m(g, h) =m(h, g)$ and
$m(g, h) \geq 2$. 
Steinberg \cite{steinberg}
and Bourbaki \cite{bourbaki}
showed that the growth series
of a Coxeter group
is a rational function.
Salem numbers, Pisot numbers 
and Perron numbers 
occur
as roots of the polynomials at the denominator
(here the definition of a Salem number is often extended to quadratic Pisot numbers,
conveniently and abusively).

Let us consider a hyperbolic 
cocompact Coxeter group G with generating set of reflections S acting in
low dimensions $n \geq 2$.

{\it Case $n=2$}:
for the Coxeter reflection groups
$G_{p_1 , \ldots, p_d}$, 
with $p_i$ any positive integer,
of presentation
$G_{p_1 , \ldots, p_d}
:=
(g_1 , \ldots, g_d \mid
(g_i)^2 = 1 ,
(g_i g_{i+1})^{p_i} = 1)$,
the denominator 
$\Delta_{p_1 , \ldots, p_d}(x)$
of the growth series
$f(x)$ of $G_{p_1 , \ldots, p_d}$
is explicitely given by the 
following theorem
\cite{cannon}.

\begin{theorem}[Cannon-Wagreich \cite{cannonwagreich},
Floyd-Plotnick \cite{floydplotnick}, Parry \cite{parryw}]
\label{cannonwagreichthm}
$$\Delta_{p_1, \ldots, p_d}(x)
=
[p_1] [p_2] \ldots [p_d] (x-d+1)
+ \sum_{i=1}^{d}
[p_1] \ldots \widehat{[p_i]}\ldots [p_d]$$
The polynomial
$\Delta_{p_1, \ldots, p_d}(x)$ is either
a product of cyclotomic polynomials
or a product of cyclotomic polynomials times an
irreducible Salem polynomial. The Salem polynomial 
occurs if and only if 
$G_{p_1 , \ldots, p_d}$ is hyperbolic, that is,
$$\frac{1}{p_1}+ \ldots + \frac{1}{p_d}~<~ d -2.$$
\end{theorem}
Then hyperbolic Coxeter reflection groups have Salem 
numbers as asymptotic growth rates.
Such Salem numbers
form a subclass of the set of Salem numbers.
Lehmer's polynomial is 
$\Delta_{2,3,7}(x)$, denominator of the 
growth series 
of the (2,3,7)-hyperbolic triangle group (Takeuchi \cite{takeuchi}).
The {\em Construction of Salem}
\cite{salem2} \cite{boyd2}, 
for establishing
the existence of
sequences of Salem numbers converging to a given
Pisot number, on the left and on the right, admits
an analogue in terms of geometric 
convergence for the fundamental domains
of cocompact planar hyperbolic Coxeter groups. 
Using the Construction of Salem
Parry \cite{parryw} gives a new proof of
Theorem \ref{cannonwagreichthm}.
Bartholdi and Ceccherini-Silberstein
\cite{bartholdiceccherinisilberstein}
studied the Salem numbers which arise from
some hyperbolic graphs.
Hironaka \cite{ghatehironaka} solves the problem 
of Lehmer for the subclass of Salem numbers 
occuring as such asymptotic growth rates:

\begin{theorem}[Hironaka \cite{hironaka}]
\label{hironakathm1}
Lehmer's number is the smallest Salem number 
occuring as dominant 
roots of $\Delta_{p_1 , \ldots, p_d}$ 
polynomials for any 
$(p_1 , \ldots, p_d)$, 
$p_i$ being positive integers.
\end{theorem}

{\it Case $n=3$}:
Parry \cite{parryw} extends 
its 2-dimensional approach to every hyperbolic cocompact reflection Coxeter
group on $\hb^3$ generated by reflections
whose fundamental domain
is a bounded 
polyhedron (not just tetrahedron).
Parry's approach is based on the properties of 
anti-reciprocal  rational functions 
with Salem numbers.
Kolpakov \cite{kolpakov} provides a 
generalization
to the three-dimensional case, 
by establishing a metric convergence 
of fundamental domains
for cocompact hyperbolic Coxeter groups 
with finite-volume limiting polyhedron;
for instance, the compact polyhedra
$\mathcal{P}(n) \subset \hb^3$
of type $<2,2,n,2,2>$ converging,
as $n \to \infty$, to a polyhedron
$\mathcal{P}_{\infty}$ with 
a single four-valent ideal vertex.
In this context, Kolpakov investigates 
the growth series
of Coxeter groups acting on 
$\hb^n$, $n \geq 3$ and their 
limit properties.
The growth rates of 
ideal Coxeter polyhedra in $\hb^3$ was studied by
Nonaka \cite{nonaka}.

{\it Case $n \geq 4$}:
the growth rates of 
cocompact hyperbolic Coxeter groups are 
not Salem numbers anymore. Kellerhals and Perren
\cite{kellerhalsperren}, $\S$3 Example 2, 
show this fact
with the example of the compact right-angled 120-cell in
$\hb^4$. 

Lehmer's problem asks about the geometry of 
the poles of the growth rates of 
hyperbolic Coxeter groups acting on $\hb^n$, 
and structure theorems about such groups
having denominators of 
growth series
of minimal Mahler measure.

\begin{conjecture}(Kellerhals-Perren)
\label{Kellerhals-PerrenCJ}
Let $G$ be a Coxeter group acting 
cocompactly on $\hb^n$ 
with natural generating set $S$ 
and growth series $f_{S}(x)$. 
Then,

(a) for $n$ even, $f_{S}(x)$ has precisely
$\frac{n}{2}$ poles 
$0 < x_1 < \ldots < x_{\frac{n}{2}} < 1$ 
in the open unit interval $(0, 1)$ ;

(b) for $n$ odd, $f_{S}(x)$ 
has precisely the pole 1 and $\frac{n-1}{2}$
poles $0 < x_1 < \ldots < x_{\frac{n-1}{2}} < 1$ 
in the interval $(0, 1)$.

In both cases, the poles are simple,
and the non-real poles of $f_{S}(x)$
are contained in the annulus of radii $x_m$
and $x_{m}^{-1}$ for some $m \in 
\{1, \ldots, \lfloor\frac{n}{2}\rfloor\}$.
\end{conjecture}

\begin{theorem}[Kellerhals-Perren]
Let $G$ be a Lann\'er group, 
an Esselmann group or a Kaplinskaya group, 
respectively, acting
with natural generating set $S$ on $\hb^4$. 
Then,

(1) the growth series $f_{S}(x)$ 
of $G$ is a quotient of relatively prime, monic and 
reciprocal polynomials of
equal degree over the integers,

(2) the growth series $f_{S}(x)$ 
of $G$ possesses four distinct positive real poles appearing in pairs $(x_1 , x_{1}^{-1})$
and
$(x_2 , x_{2}^{-1})$
with $x_1 < x_2 < 1 < x_{2}^{-1} < x_{1}^{-1}$; 
these poles are simple,

(3) the growth rate $\tau = x_{1}^{-1}$
is a Perron number,

(4) the non-real poles of $f_{S}(x)$  
are contained in an annulus of radii $x_2$ , 
$x_{2}^{-1}$ around the unit circle,

(5) the growth series $f_{S}(x)$ 
of the Kaplinskaya group $G_{66}$ 
with graph $K_{66}$ has four distinct negative
and four distinct positive simple real poles; 
for $G \neq G_{66}$ , $f_{S}(x)$ has no negative pole.
\end{theorem}

Kellerhals and Kolpakov \cite{kellerhalskolpakov}
(2014) prove that the
simplex group $(3, 5, 3)$ has the 
smallest growth rate among all cocompact 
hyperbolic Coxeter groups
on $\hb^3$,
and that it is, as such, unique. 
The growth rate is the 
Salem number 
$\tau' = 1.35098\ldots$ of minimal 
polynomial 
$X^{10} -
X^{9} - X^{6} + X^{5} - X^{4} - 
X + 1$.
Their 
approach provides a different proof for 
the analog situation in
$\hb^2$ where Hironaka \cite{hironaka}
identified Lehmer's number as the 
minimal growth rate among all cocompact
planar hyperbolic Coxeter groups and 
showed that it is (uniquely) achieved by 
the Coxeter triangle group $(3,7)$.

After Meyerhoff who proved 
that among all cusped hyperbolic 3-
orbifolds
the quotient of $\hb^3$ by the 
tetrahedral 
Coxeter group $(3, 3, 6)$ has minimal 
volume,
Kellerhals \cite{kellerhals} (2013)
proves that the group $(3, 3, 6)$ 
has 
smallest growth rate among all non-
cocompact
cofinite hyperbolic Coxeter groups, and 
that it is as such unique. 
This result 
extends to
three dimensions some work of 
Floyd \cite{floyd} 
who showed that the Coxeter triangle 
group
$(3, \infty)$ has minimal 
growth rate among all 
non-cocompact cofinite planar hyperbolic
Coxeter groups. In contrast to Floyd's 
result, the growth rate of the 
tetrahedral group
$(3, 3, 6)$ is not a Pisot number.
The following Theorem completes the picture of 
growth rate minimality for cofinite 
hyperbolic
Coxeter groups in three dimensions.

\begin{theorem}[Kellerhals]
Among all hyperbolic Coxeter groups with 
non-compact fundamental polyhedron of 
finite volume in $\hb^3$, the 
tetrahedral group $(3, 3, 6)$ has 
minimal growth rate,
and as such the group is unique.
\end{theorem}

In \cite{komoriumemoto} Komori and
Umemoto, for
three-dimensional non-compact
hyperbolic Coxeter groups of finite covolume,
show that the growth rate of a three-dimensional generalized simplex reflection group is a
Perron number. In 
\cite{komoriyukita} Komori and Yukita
show that the
growth rates of ideal Coxeter groups in hyperbolic
3-space are also Perron numbers
; a Coxeter polytope P is called ideal if all
vertices of P are located on the ideal boundary of
hyperbolic space. They prove that
the growth rate $\tau$ of an ideal
Coxeter polytope with $n$ facets in
$\hb^n$ satisfies
$n - 3 \leq \tau \leq n - 1$. The smallest growth 
rates occur only if $n \leq 4$. 
They prove that the minimum of the growth
rates is $0.492432^{-1} \thickapprox  2.03074$, 
which is uniquely
realized by the ideal Coxeter 
simplex with p = q = s = 2.
In \cite{umemoto} Umemoto shows that
infinitely many 2-Salem numbers can be realized
as the growth rates of cocompact Coxeter
groups in $\hb^4$; the Coxeter polytopes 
are here constructed by successive gluing of Coxeter
polytopes, which are called Coxeter dominoes
\cite{umemoto2}.
The algebraic integers having a fixed number
of conjugates outside the closed unit disk
were studied by Bertin \cite{bertin},
Kerada \cite{kerada}, Samet \cite{samet},
Zaimi \cite{zaimi}
\cite{zaimi2}, in particular 2-Salem numbers
in \cite{kerada} to which Umemoto refers.
In \cite{zehrtzehrtliebendorfer}
Zehrt and Zehrt-Liebend\"orfer 
construct infinitely many 
growth series of cocompact
hyperbolic Coxeter groups in $\hb^4$, 
whose denominator polynomials have the same
distribution of roots as 2-Salem polynomials;
their Coxeter polytopes
are the Coxeter garlands 
built by the compact truncated 
Coxeter simplex described
by the Coxeter graph on the left side of 
Figure 1 in 
\cite{zehrtzehrtliebendorfer}.
Lehmer's problem asks about the 
minimality of the houses 
of the 2-Salem numbers involved 
in these constructions. 
In \cite{kellerhalsnonaka} 
Kellerhals and Nonaka 
prove that the growth rates
of three-dimensional Coxeter groups 
$(\Gamma,S)$ given by ideal
Coxeter polyhedra of finite volume 
are Perron numbers.

A Coxeter system $(W,S)$
is a Coxeter group $W$ with 
a finite generating set $S$; 
the permuted products
$s_{\sigma(1)} s_{\sigma(2)}\ldots
s_{\sigma(n)}$, $\sigma \in S_n$,
are the {\em Coxeter elements}
of $(W,S)$. The element $w \in W$ 
is said to be {\em essential}
if it is not conjugate into any subgroup
$W_I \subset W$ generated by a proper subset
$I \subset S$. 
The Coxeter group $(W,S)$ acts naturally
by reflections on
$V \equiv \rb^S$. Let $\lambda(w)$
be the spectral radius
of $w|V$. When  $\lambda(w) > 1$ it is also
an eigenvalue of $w$.
MacMullen \cite{macmullen4}
proves the three following results.

\begin{theorem}[MacMullen]
Let $(W,S)$ be a Coxeter system and let $w \in W$
be essential. Then
$$\lambda(w) \geq \inf_{S_n} 
\lambda(s_{\sigma(1)} s_{\sigma(2)}\ldots
s_{\sigma(n)}).$$
\end{theorem}
Let $\alpha(W,S)$ be the dominant eigenvalue 
of the adjacency matrix $(A_{ij})$
of $(W,S)$, defined by
$A_{ij} = 2 \cos(\pi/m_{ij})$
for
$i \neq j$ and $A_{ii} = 0$.
Let $\beta(W,S)$ be the largest root 
of the equation
$\beta + \beta^{-1} + 2 = \alpha(W,S)^2$
provided $\alpha(W,S) \geq 2$.
If $\alpha(W,S) < 2$ then we put:
$\beta(W,S) = 1$. Then $\lambda(w) = \beta(W,S)$
for all bicolored Coxeter element.
\begin{theorem}[McMullen]
For any Coxeter system $(W,S)$, we have 
$$\inf_{S_n} 
\lambda(s_{\sigma(1)} s_{\sigma(2)}\ldots
s_{\sigma(n)}) \geq \beta(W,S).$$
\end{theorem}
\begin{theorem}[McMullen]
\label{mcmullenthm}
There are \!38 minimal hyperbolic Coxeter systems
\!$(W,S)$, and among these the infimum 
$\inf\beta(W,S)$ is Lehmer's number.
\end{theorem}
Lehmer's problem is solved in this context.
The quantity $\beta(W,S)$ can be viewed
as a measure (not in logarithmic terms) of the 
complexity of a Coxeter system.
Denote by $Y_{a,b,c}$ the Coxeter system
whose diagram is a tree with 3 branches
of lenghts $a, b$ and $c$, 
joined by a single node.
MacMullen \cite{macmullen4}
shows that the smallest 
Salem numbers of respective degrees
6, 8 and 10 coincide with
$\lambda(w)$ for the Coxeter elements of
$Y_{3,3,4}$, $Y_{2,4,5}$ and
$Y_{2,3,7}$ respectively; in particular
Lehmer's number is
$\lambda(w)$ for the Coxeter elements of
$Y_{2,3,7}$. MacMullen shows that
the set of irreducible Coxeter systems
with $\beta(W,S) < \Theta$ consists exactly
of $Y_{2,4,5}$ and $Y_{2,3,n}$, $n \geq 7$.
He shows that the infimum of
$\beta(W,S)$ over all high-rank Coxeter
systems coincides with $\Theta$.
There are 6 Salem numbers $< 1.3$ that 
arise as eigenvalues in Coxeter groups, five of them arising from the Coxeter
elements of $Y_{2,3,n}$, $ 7 \leq n \leq 11$.

\subsubsection{Mapping classes: small stretch factors}
\label{S2.4.3}

We refer to Birman \cite{birman},
Farb and Margalit
\cite{farbmargalit} and
Hironaka \cite{hironaka6}.
If S is a surface the {\em mapping class group} of
S, denoted by Mod$(S)$, 
is the group of isotopy classes of 
orientation-preserving 
diffeomorphisms of S (that restrict to the identity
on $\partial S$ if $\partial S \neq \emptyset$).
An irreducible mapping class is an isotopy class of
homeomorphisms $f$ of a compact oriented surface
$S$ to itself so that no power preserves
a nontrivial subsurface. 
The classification of Nielsen-Thuston 
states that a mapping class 
$[f] \in {\rm Mod}(S)$ is either periodic,
reducible or pseudo-Anosov
\cite{farbmargalit} 
\cite{fathilaudenbachpoenaru}.
In the periodic case, the situation is
``analogous to roots of unity"
in Lehmer's problem. The minoration problem
of the Mahler measure 
finds its analogue in the
minoration of the dilatation factors
of the pseudo-Anosovs.
We refer to a mapping class $[f]$ by one of its 
representive $f$.

Let $S_g$ be a closed, orientable surface
of genus $g \geq 2$ and Mod$(S_g)$ its 
mapping class group.
For any Pseudo-Anosov element $f \in $Mod$(S_g)$,
and any integer $0 \leq k \leq 2 g$,
let

(i) $\kappa(f)$ be the dimension of the subspace
of $H_{1}(S_g , \rb)$ fixed by $f$ (for which
$0 \leq \kappa(f) \leq 2 g$),

(ii) $h(f) = \lo (\lambda(f))$ be 
the entropy of $f$, as logarithm of the 
{\em stretch factor}
$\lambda(f) > 1$ (or {\em dilatation}; 
the dilatation measures the dynamical complexity),

(iii) 
$$L(k,g) := \min\{h(f) \mid f : S_g \to S_g
~\mbox{and} ~\kappa(f) \geq k \}.$$
\noindent
Thurston \cite{fathilaudenbachpoenaru}
\cite{thurston}  
noticed that the set of stretch factors
for pseudo-Anosov elements
of Mod$(S_g)$ is closed and discrete
in $\rb$, and
proved that
any dilatation factor 
$\lambda(f) > 1$ is a Perron number,
with $\lambda(f)+ \lambda(f)^{-1}$ an algebraic
integer of degree $\leq 4 g -3$.
The Perron number $\lambda(f)$
is the growth rate of lengths of curves
under iteration (of any representant) of $f$,
in any metric on $S_g$. 
These stretch factors
appear as the length spectrum of the moduli space
of genus $g$ Riemann surface. 

Penner \cite{penner} proved that
the asymptotic behaviour
$L(0, g) \asymp 1/g$ holds. 
With $k=2g$,
Farb, Leininger and Margalit
\cite{farbleiningermargalit} proved
$L(2 g, g) \asymp 1$. For the other values of $k$,
since  
$L(0, g) \leq L(k, g) \leq L(2 g, g)$, 
the following 
inequalities hold, from
\cite{aaberdunfeld}
\cite{kintakasawa}
\cite{hironaka}
\cite{penner},
$$\frac{\lo 2}{6}\left(
\frac{1}{2 g - 2}\right) \leq L(k, g) \leq
\lo (62).$$
For $k=0$ and $g=1$, 
$L(0,1) = \lo \bigl(\frac{3 + \sqrt{5}}{2}\bigr)$
for $\tb^2$.
For $k=0$ and $g=2$, Cho and Ham
\cite{choham} \cite{lanneauthiffeault}
\cite{zhirov}
proved
$L(0,2) \thickapprox 0.5435\ldots$, 
as logarithm of the largest
root of the Salem polynomial
$X^4 - X^3 - X^2 - X +1$; these authors
showed that this minimum dilatation 
is given by Zhirov in \cite{zhirov},
and realized by Pseudo-Anosov 5-braids 
in \cite{hamsong}.
In \cite{agolleiningermargalit}
Agol, Leininger and Margalit 
improved the upper bound to: 
$(2g - 2) L(0, g) < \lo (\theta_{2}^{-4})$
for all $g \geq 2$, where
$\theta_{2}^{-1}$ is the golden mean,
and proved the main theorem: 
$$L(k, g) \asymp \frac{k+1}{g} ,
\qquad g \geq 2, \quad0 \leq k \leq 2 g.$$
Arnoux-Yoccoz's Theorem
\cite{arnouxyoccoz} states that, 
for $g \geq 2$, for any $C \geq 1$, 
there are only finitely
many conjugacy classes in 
Mod$(S_g)$ of pseudo-Anosov mapping classes
with stretch factors at most $C$.

\vspace{0.1cm}
{\em Minimal dilatation problem}: 
what are the values of
$L(k,g)$, except $L(0,g)$ for $g=1,2$ already determined? 
i.e. what are the minima
$\delta_{g} := \exp(L(0,g))$, $g \geq 3$?
\vspace{0.1cm}

Lower bounds of the entropy are difficult 
to establish: e.g. Penner \cite{penner}, 
Tsai \cite{tsai} on punctured surfaces, 
Boissy and Lanneau 
\cite{boissylanneau} on 
translation surfaces
that belong to a hyperelliptic component,
Hironaka and Kin \cite{hironakakin}.
Then
Kin \cite{kin}
\cite{kintakasawa}
\cite{kinkojimatakasawa}
formulated several questions and 
conjectures on the minimal dilatation 
problem and its realizations.
Bauer \cite{bauer} studied upper bounds
of the least dilatations,
and Minakawa \cite{minakawa} 
gave examples of pseudo-Anosovs
with small dilatations.
Farb, Leininger and Margalit 
\cite{farbleiningermargalit2} 
obtained a universal finiteness theorem 
for the set of all small dilatation
pseudo-Anosov homeomorphisms 
$\phi : S \to S$, ranging over all surfaces
$S$.
The following questions 
were posed by
in \cite{macmullen2}
and \cite{farb}. 

\vspace{0.1cm}
{\em Asymptotic behaviour}: 
(i) Does $\lim_{g \to \infty} g ~ L(0,g)$
exist? What is its value?
(ii) Is the sequence $\{\delta_g\}_{g \geq 2}$ 
(strictly)
monotone decreasing?
\vspace{0.1cm}

Kin, Kojima and Takasawa \cite{kinkojimatakasawa},
for monodromies of fibrations on manifolds
obtained from the magic 3-manifold N by Dehn 
filling three cusps with some restriction,
proved  $\lim_{g \to \infty} g ~ L(0,g)
= \lo\bigl(
\frac{3 + \sqrt{5}}{2}\bigr)$; they also
formulated
limit conjectures for the asymptotic behaviour
relative to compact surfaces of 
genus $g$ with $n$ boundary components.

A pseudo-Anosov mapping class 
$[f]$ is said to be
{\em orientable} if the 
invariant (un)-stable foliation
of a pseudo-Anosov homeomorphism
$f \in [f]$ is orientable.
Let $\lambda_{H}(f)$ be
the spectral radius of the action of $f$
on
$H_{1}(S_g , \rb)$. It is the 
{\em homological stretch factor} of $f$.
The inequality $\lambda_{H}(f) \leq \lambda(f)$
holds and equality occurs iff the invariant foliations
for $f$ are orientable. Stretch factors obey  
some constraints \cite{shin}:
(i) $\deg(\lambda(f)) \geq 2$, 
(ii) $\deg(\lambda(f)) \leq 6 g - 6$ , 
(iii) if $\deg(\lambda(f)) > 3 g - 3$, then
$\deg(\lambda(f))$ is even.
Shin \cite{shin} 
formulates the following questions.

\vspace{0.1cm}
{\em Algebraicity of stretch factors}:
(i) Which real numbers can be the stretch factors, 
the homological stretch factors?
(ii) What degrees of stretch factors can occur on $S_g$?
\vspace{0.1cm}

Let us define a mapping class 
$f_{g,k}$ by
$$f_{g,k} = (T_{c_{g}})^k
\bigl(
T_{d_{g}} \ldots T_{c_{2}}  T_{d_{2}}
T_{c_{1}} T_{d_{1}}
\bigr) \in {\rm Mod}(S_g),$$
where $c_i$ and $d_i$ are simple closed curves  
on $S_g$ as in Figure 1 in \cite{shin}, and $T_c$ the
Dehn twist about $c$.

\begin{theorem}[Shin]
For each $g \geq 2, k \geq 3$,
$f_{g,k}$ is a pseudo-Anosov mapping class
which satisfies: (i) 
$\lambda(f_{g,k}) = \lambda_{H}(f_{g,k})$,
(ii) $f_{g,k}$ is a Salem number, (iii)
$\lim_{g \to \infty} f_{g,k} = k-1$, 
where
the minimal polynomial of $\lambda(f_{g,k})$ is the irreducible 
Salem polynomial
$$t^{2 g } -(k-2)\Bigl(
\sum_{j=1}^{2 g -1} t^j
\Bigr) + 1, \qquad \mbox{of degree}~ 2 g.$$
\end{theorem}
Shin \cite{shin}
deduces that, for each
$1 \leq h \leq g/2$,
there exists a pseudo-Anosov mapping
class $\widetilde{f_{h}} \in {\rm Mod}(S_g)$ 
such that
$\deg(\widetilde{f_{h}}) = 2 h$, with
$\lambda(\widetilde{f_{h}})$ a Salem number.
He conjectures that, on $S_g$, 
there exists a pseudo-Anosov
mapping class with a stretch factor of degree
$d$ for any even $1 \leq d \leq 6 g - 6$. He proves that the conjecture is true for $g = 2$ to $5$.
Shin and Strenner \cite{shinstrenner} 
prove that the Perron numbers which 
are the stretch factors of 
pseudo-Anosov mapping classes
arising from Penner's construction \cite{penner}
have conjugates which do not belong
to the unit circle. In $\S$3 in \cite{shinstrenner}
they ask several questions about the geometry
of the Galois conjugates of stretch factors, around
the unit circle,
obtained by several constructions: by Hironaka \cite{hironaka5},
by Dunfeld and Tiozzo,
by Lanneau and Thiffeault
\cite{lanneauthiffeault}
\cite{lanneauthiffeault2}, 
by Shin \cite{shin}.
For $S_{g,n}$ being an orientable surface
with genus $g$ and
$n$ marked points, Tsai \cite{tsai}
proves that the infimum of 
stretch factors is of the 
order of $(\lo n)/n$ for $g \geq 2$ whereas
it is of the order of $1/n$ for 
$g=0$ and $g=1$;
Tsai asks the question about the 
asymptotic behaviour of this infimum
of dilatation factors in the $(g,n)$-plane.
For some subcollections of mapping classes,
by generalizing Penner's construction and
by comparing the smallness of dilatation factors
with trivial homological dilatation, 
Hironaka \cite{hironaka8} 
concludes to the existence of subfamilies
of pseudoanosovs which have asymptotically
small dilatation factors.

In the context of $\zb^n$-actions on 
compact abelian groups 
(Proposition 17.2 and Theorem 18.1 in
Schmidt \cite{schmidt2}) the topological
entropy is equal to the logarithm of
the Mahler measure.
If we assume that the stretch factors
are Mahler measures 
M$(\alpha)$ 
of algebraic numbers $\alpha$
(which are Perron numbers by 
Adler and Marcus \cite{adlermarcus}),
then we arrive at a contradiction
since Penner \cite{penner}
showed that $L(0,g) \asymp \frac{1}{g}$
for surfaces of genus $g$. Indeed, it 
suffices to increase 
the genus $g$ to find pseudo-Anosov elements of
Mod$(S_g)$ with dilatation factors
arbitrarily close to 1, while Theorem 
\ref{mainLEHMERtheorem} states that a discontinuity
should exist. 
As a consequence of \cite{penner},
\cite{tsai}
and of Theorem 
\ref{mainLEHMERtheorem} (ex-Lehmer Conjecture)
we deduce the following claims: 

Assertion 1: {\em The 
stretch factors of the pseudo-Anosov elements of
\,{\rm Mod}$(S_g)$ are Perron numbers which are not Mahler 
measures of algebraic numbers as soon as 
$g$ is large enough.}

Assertion 2: {\em The
stretch factors of the pseudo-Anosov elements of
\,{\rm Mod}$(S_{g,n})$, where 
$S_{g,n}$ is an orientable surface
with fixed genus $g$ and
$n$ marked points, are Perron 
numbers which are not Mahler 
measures of algebraic numbers as soon as 
$n$ is large enough}.

Let $S$ be a connected finite type
oriented surface.
Leininger \cite{leininger} considers subgroups
of Mod$(S)$ generated by two positive
multi-twists; a multi-twist is a product of positive
Dehn twists about disjoint essential simple
closed curves. 
Given $A$ and $B$ two isotopy classes
of essential 1-manifolds on $S$, we denote by
$T_A$, resp. $T_B$, the positive multi-twist
which is the product of positive Dehn twists
about the components of $A$, resp. of $B$.

\begin{theorem}[Leininger]
\label{leiningerthm}
Any pseudo-Anosov element $f \in \langle T_A , T_B \rangle$ 
has a stretch factor which satisfies:
$$\lambda(f) \geq \lambda_{L} 
~\mbox{{\rm (Lehmer's number)}}.$$
\end{theorem}
The realization occurs when $S$ has genus 5
(with at most one marked point),
$A=A_{Lehmer}, B=B_{Lehmer}$ given by Figure 1 in 
\cite{leininger} up to homeomorphism, 
and $f$ conjugate to $( T_A T_B )^{\pm 1}$.
Leininger's Theorem \ref{leiningerthm} is strikingly
reminiscent of McMullen's Theorem
\ref{mcmullenthm}. 
The following questions
are formulated in $\S$9.1
in \cite{leininger}:

Q1: Which Salem numbers occur as dilatation factors of pseudo-Anosov automorphisms?

Q2: Is there some topological condition on a pseudo-Anosov which guarantees that its dilatation factor is a Salem number?

Q3 (Lehmer's problem for dilatation factors): Is there
an $\epsilon > 1$ such that if $f$ is a pseudo-Anosov
automorphism in a finite co-area Teichm\"uller disk stabilizer, then $\lambda(f) \geq \epsilon$?

Since dilatations factors of pseudo-Anosovs are Perron numbers and not necessarily Mahler measures of 
algebraic numbers
(cf Assertions 1 and 2 above), Leininger's Theorem \ref{leiningerthm} and McMullen's Theorem
\ref{mcmullenthm} are addressed to 
the set of Salem numbers and suggest that 
Lehmer's number is actually
the smallest Salem number in this set; 
meaning first that Lehmer's Conjecture 
is true for Salem numbers.

Let 
$$f_{k,l}(t) := t^{2 k} - t^{k+l} -t^k -t^{k-l} +1,$$
$$\mbox{resp.} \qquad
f_{x,y,z}(t) :=
t^{x+y-z} - t^x -t^y -t^{x-z} -t^{y-z}+1,$$
and denote 
$\lambda_{(k,l)} > 1$, resp.
$\lambda_{(x,y,z)} > 1$, its dominant root.

Related to the minimization problem
is the one for orientable pseudo-Anosovs.
The minimal dilatation factor
for orientable pseudo-Anosov elements
of Mod$(S_g)$ is denoted by
$\delta_{g}^{+}$. 
The minimal dilatation problem
for $\delta_{g}^{+}$ is open in general. 
For $g=2$, Zhirov \cite{zhirov}
obtained $\delta_{2}^{+} = \delta_{2}$.
For $g=1$, $\delta_{1}^{+} = \delta_{1}$
holds. From \cite{hironaka5}
\cite{lanneauthiffeault},
$\delta_{g} < \delta_{g}^{+}$ for
$g=4, 6, 8$. Hironaka \cite{hironaka5}
obtained: (i) 
$\delta_{g} \leq \lambda_{(g+1,3)}$ if 
$g \equiv 0, 1, 3, 4 \,({\rm mod}\, 6)$
and $g \geq 3$,
(ii) $\delta_{g} \leq \lambda_{(g+1,1)}$ 
if 
$g \equiv 2, 5 \,({\rm mod}\, 6)$
and $g \geq 5$,
(iii)
$\limsup_{g \to \infty} g \lo \delta_g
\leq \lo\Bigl(
\frac{3 + \sqrt{5}}{2}
\Bigr)$. Kin and Takasawa \cite{kintakasawa3} 
complemented and improved these inequalities.
They showed: (i) 
$\delta_{g} \leq \lambda_{(g+2,1)}$ if 
$g \equiv 0, 1, 5, 6 \,({\rm mod}\, 10)$
and $g \geq 5$,
(ii) $\delta_{g} \leq \lambda_{(g+2,2)}$ 
if 
$g \equiv 7, 9 \,({\rm mod}\, 10)$
and $g \geq 7$;
for $g \equiv 2, 4 \,({\rm mod}\, 10)$, under the assumption
$g+2 \not\equiv 0 \,({\rm mod}\, 4641)$, then
they prove:
(i) $\delta_{g} \leq \lambda_{(g+2,3)}$ if
gcd$(g+2,3)=1$,
(ii) $\delta_{g} \leq \lambda_{(g+2,7)}$ if
$3 |(g+2)$ and gcd$(g+2,7)=1$,
(iii) $\delta_{g} \leq \lambda_{(g+2,13)}$ if
$21 |(g+2)$ and gcd$(g+2,13)=1$,
(iv) $\delta_{g} \leq \lambda_{(g+2,17)}$ if
$273 |(g+2)$ and gcd$(g+2,17)=1$. For $g=8$ and $13$ they obtain the sharper upper bounds:
$\delta_{8} \leq \lambda_{(18,17,7)} < \lambda_{(9,1)}$
and
$\delta_{13} \leq \lambda_{(27,21,8)} < \lambda_{(14,3)}$.
For orientable pseudo-Anosovs,
Lannneau and Thiffeault
\cite{lanneauthiffeault}AIF
obtained $\delta_{3}^{+} = \lambda_{(3,1)}$,
$\delta_{4}^{+} = \lambda_{(4,1)}$,
and
the following lower bounds for
$g=6$ to $8$: (i)
$\delta_{6}^{+} \geq \lambda_{(6,1)}$,
(ii) $\delta_{7}^{+} \geq \lambda_{(9,2)}$
and
$\delta_{8}^{+} \geq \lambda_{(8,1)}$.
Hironaka \cite{hironaka5}
gave the upper bounds:
(i) $\delta_{g}^{+} \leq \lambda_{(g+1,3)}$ if
$g \equiv 1, 3  \,({\rm mod}\, 6)$,
(ii) $\delta_{g}^{+} \leq \lambda_{(g,1)}$ if
$g \equiv 2, 4  \,({\rm mod}\, 6)$,
(iii) $\delta_{g}^{+} \leq \lambda_{(g+1,1)}$ if
$g \equiv 5  \,({\rm mod}\, 6)$.
He obtained
the asymptotics:
$ \displaystyle
\limsup_{g \to \infty, g \not\equiv 0 ({\rm mod} 6)}
g \, \lo \delta_{g}^{+} \leq
\lo\Bigl(
\frac{3 + \sqrt{5}}{2}
\Bigr)$ and, from \cite{lanneauthiffeault},
the equality:
$\delta_{8}^{+} = \lambda_{(8,1)}$. 
Kin and Takasawa \cite{kintakasawa3}
gave the better upper bounds:
(i) $\delta_{g}^{+} \leq \lambda_{(g+2,2)}$ if
$g \equiv 7, 9  \,({\rm mod}\, 10)$
and $g \geq 7$,
(ii) $\delta_{g}^{+} \leq \lambda_{(g+2,4)}$ if
$g \equiv 1, 5  \,({\rm mod}\, 10)$
and $g \geq 5$.
Moreover they proved: $\delta_{7}^{+}
= \lambda_{(9,2)}$ (Aaber and Dunfeld 
\cite{aaberdunfeld} obtained it independently) 
and 
$\delta_{5} < \delta_{5}^{+}$.

The realization of the minimal dilatations is a 
basic question, with the uniqueness problem,
considered by many authors: associated with
least volume
\cite{aaberdunfeld}
\cite{kinkojimatakasawa},  
braids 
\cite{choham}
\cite{hamsong}
\cite{hironakakin}
\cite{kintakasawa}
\cite{kintakasawa2}
\cite{lanneauthiffeault2},
monodromies 
\cite{farbleiningermargalit}
\cite{kinkojimatakasawa}
\cite{kintakasawa3},
Coxeter graphs and Coxeter elements
\cite{leininger}
\cite{hironaka5}
\cite{shin},
quotient families of 
mapping classes 
\cite{hironaka10},
self-intersecting curves
\cite{dowdall},
homology of mapping tori 
\cite{agolleiningermargalit}.
There exists several constructions
of small dilatation families,
e.g. by 
Hironaka \cite{hironaka7}
\cite{hironaka9}, McMullen \cite{macmullen2},
Dehornoy \cite{dehornoy3} with 
Lorenz knots.

\subsubsection{Knots, links, Alexander polynomials, homology growth, Jones polynomials, len-ticularity of zeroes, lacunarity}
\label{S2.4.4}

Constructions of
Alexander polynomials 
of knots and links
are given in
\cite{kawauchi} 
\cite{murasugi}
\cite{rolfsen} \cite{seifert}
\cite{turaev}.
Silver and Williams 
in \cite{silverwilliams2} 
(reported in
\cite{smyth6} $\S$ 4.2 for an overview)
investigate
the Mahler measures 
of various Alexander polynomials of oriented links with 
$d$ components in a homology 3-sphere; they
obtain theorems on
limits of Mahler measures
and Mahler measures of derivatives of
$d$-variate Mahler measures by performing
$1/q$ surgery ($q \in \nb$) on the
$d$th component, allowing $q \to \infty$.
In particular they consider
the topological realizability
of the small Mahler measures and
limit Mahler measures on various examples.

For Pretzel links
Hironaka  (\cite{hironaka}
\cite{hironaka2},
\cite{hironaka4},
\cite{ghatehironaka} p. 308)
solves the minimization problem
for the subclass of Salem numbers
defined in 
Theorem \ref{cannonwagreichthm} by 

\begin{theorem}[Hironaka \cite{hironaka}]
\label{hironakathm2}
Let $p_1 , \ldots , p_d$ positive 
integers. 
Then the Alexander polynomial of the
$(p_1 , \ldots , p_d , -1, \ldots, -1)$-pretzel link
(Coxeter link), where the number of $-1$'s is $d-2$,
with respect to a suitable orientation 
of its components,
is $\Delta_{p_1 , \ldots, p_d}(-x)$.
\end{theorem}
Lehmer's polynomial
of the variable ``$-x$" is the 
Alexander polynomial of
the $(-2, 3, 7)$-pretzel knot 
and the $(-2, 3, 7)$-pretzel knot
is equivalent to
the $(2, 3, 7, -1)$- pretzel knot.
Theorem \ref{hironakathm1}
follows from Theorem
\ref{hironakathm2}.
The Mahler measure of the 
$(2, 3, 7, -1)$- pretzel knot is the
minimum of the set of Mahler measures
of Alexander polynomials of
(suitably oriented)
$(p_1 , \ldots , p_d , -1, \ldots, -1)$-pretzel links, 
over all
$d$ in $2 \nb + 1$.

It is natural to
find counterparts of Lehmer's problem in Topology 
where
several polynomial invariants 
\cite{frankswilliams}
\cite{freydetal}
\cite{jones},
\cite{smyth5} $\S$ 14.6,
were associated
to knots, links and braids, for which 
the notions of convergence and
``limit" can be defined 
(as in \cite{champanerkarkofman}
\cite{dehornoy2} \cite{hironaka3}
\cite{silverwilliams}), in addition to Alexander polynomials.
Indeed a theorem of Seifert \cite{seifert}
asserts 
that (i) any monic reciprocal integer 
polynomial $P(x)$, 
(ii) which satisfies $|P(1)| = 1$,
is the Alexander 
polynomial of (at least) one 
knot, and conversely; Burde
\cite{burde} extended it to fibered knots
(cf Hironaka \cite{hironaka3}).
A Theorem of Kanenobu \cite{kanenobu}
asserts that
any reciprocal monic integer polynomial 
$P(x)$
is, up to multiples of
$x-1$, the 
Alexander polynomial of a fibered link.
Let us recall that infinitely many knots may possess
the same polynomial invariants
(Morton \cite{morton},
Kanenobu \cite{kanenobu2}).

{\it Periodic homology and exponential growth}: 
the $r$-fold cyclic covering $X_{r}(K)$ of a knot
$K \subset \sbb^3$
admits topological invariants, i.e.
homology groups $H_{1}(X_{r}(K),\zb)$, 
which are also invariants of the knot $K$.
To $K$ is associated a sequence of Alexander 
polynomials $(\Delta_{i}), i \geq 1$, in a 
single variable,
such that $\Delta_{i+1} | \Delta_{i}$. 
Likewise, to 
an oriented link with $d$ components is associated
a sequence of Alexander polynomials
in $d$ variables. In both cases, the first Alexander
polynomial of the sequence
is usually called the Alexander polynomial of the knot 
$K$, resp. of the link. For a knot $K$
Gordon \cite{gordon}
proved that the first Alexander invariant
$\lambda_{1}(t)=\Delta_{1}(t)/\Delta_{2}(t)$
satisfies the following equivalence 
:
\begin{equation}
\label{homologyGP}
\lambda_{1}(t) | (t^n -1)
\quad
\Longleftrightarrow
\quad H_{1}(X_{r}(K),\zb) 
~\cong
~ H_{1}(X_{r+n}(K),\zb) \mbox{~~for all }~ r.
\end{equation}
The equivalence 
\eqref{homologyGP}
is an analogue of Kronecker's Theorem.
Gordon used the Pierce numbers of
the Alexander polynomial of $K$, for which 
a linear recurrence is expected as in \cite{lehmer}
\cite{einsiedlereverestward}.
Gordon obtained periods which are not prime powers and
how to find all of them for knots of a given genus.

\begin{theorem}[Gordon]
There exists a knot $K$ of genus $g$ for which
$H_{1}(X_{r}(K),\zb)$ has proper period $n$
if and only if
$n=1$, or $n=\mbox{
{\rm lcm}}\{m_i \mid i= 1, 2, \ldots, r\}$, where the 
$m_i$'s are all distinct,
each has at least two distinct prime factors, and
$\sum_{i=1}^{r} \Phi(m_i)
\leq 2 g$.
\end{theorem}
Departing from ``Kronecker's Theorem"
Gordon conjectured that when some zero of
$\Delta_{1}(t)$ is not a root of unity, then
the order of $H_{1}(X_{r}(K),\zb)$ 
grows exponentially with $r$.
This conjecture was proved by
Riley \cite{riley}, with $p$-adic methods, and 
Gonz\'alez-Acu$\tilde{{\rm n}}$a and Short
\cite{gonzalezacunashort}. Both used the 
Gel'fond-Baker theory of
linear forms in the logarithms of algebraic numbers.
Silver and Williams \cite{silverwilliams}
extended the conjecture of Gordon
and its proof for knots, where
the ``finite order of $H_{1}(X_{r}(K),\zb)$"
is replaced by
the ``order of the torsion subgroup of
$H_{1}(X_{r}(K),\zb)$", and 
for links in $\sbb^3$. They
identified the torsion subgroups with
the connected components of periodic points
in a dynamical system of algebraic origin
\cite{schmidt2}, connected
the limit with the logarithmic 
Mahler measure 
(for any finite-index subgroup 
$\Lambda \subset \zb^d$,
the number of such connected
components is denoted by $b_{\Lambda}$
and $\langle \Lambda \rangle
:=\{|v| \mid v \in \Lambda \setminus \{0\}\}$
is the norm of the 
smallest nonzero vector of $\Lambda$; cf \cite{silverwilliams} 
for the definitions):

\begin{theorem}[Silver-Williams \cite{silverwilliams}]
\label{swtorsionthm}
Let $l = l_1 \cup \ldots \cup l_d$ 
be an oriented link
of $d$ components having nonzero 
Alexander polynomial
$\Delta$, in $d$ variables. 
Then
\begin{equation}
\label{growthSW}
\limsup_{\langle \Lambda \rangle \to \infty}
\frac{1}{| \zb^d / \Lambda |} \lo b_{\Lambda}
=
\lo {\rm M}(\Delta)
\end{equation}
where ``$\limsup$" is replaced by ``$\lim$" if
$d=1$.
\end{theorem} 
Let $M$ be a finitely generated module over
$\zb[\zb^n]$ and $\widehat{M}$ its 
(compact) Pontryagin dual. 
For any subgroup
$\Lambda \subset \zb^n$ of finite index,
let $b_{\Lambda}$ be the number of 
connected components
of the set of elements of 
$\widehat{M}$ fixed by actions 
of the elements of
$\Lambda$.
Le
(\cite{le}, Theorem 1) 
proved a conjecture of
K. Schmidt \cite{schmidt} 
on the growth of the number $b_{\Lambda}$;
as a consequence Le
generalized (\cite{le}, Theorem 2)
Silver Williams's 
Theorem \ref{swtorsionthm}
on the growth of the homology
torsion of finite abelian covering of 
link complements, with the logarithmic
Mahler measure of the
first nonzero Alexander polynomial of the link. 
In each case, since the growth
is expressed by the logarithmic Mahler measure
of the (first nonzero) Alexander polynomial, Lehmer's problem amounts to
establishing a universal minorant $> 0$ of the
exponential base. For non-split links
in $\sbb^3$, that is in the nonabelian covering case,
Le \cite{le} generalized Theorem \ref{swtorsionthm}
using the $L^2$-torsion, i.e. the hyperbolic volume
in the rhs part of \eqref{growthSW}
instead of the logarithmic Mahler measure
of the $0$th Alexander polynomial; in such a case 
the minimality of 
Mahler measures would find its origin
in the minimality of hyperbolic volumes
\cite{le2}.

The growth of the homology torsion depends upon
the (nonzero) logarithmic
Mahler measure of the Alexander polynomial(s) 
of a knot or a link. Hence
the geometry of zeroes of Alexander polynomials
is important for the minoration of the
homology growth \cite{ghatehironaka}. 
At this step, 
let us briefly mention
the importance of other studies on the 
roots of Alexander polynomials:
(i) monodromies and dynamics of surface homeomorphisms
\cite{hironakaliechti}
\cite{rolfsen}, 
(ii) knot groups: factorization and divisibility \cite{murasugi},
(iii) knot groups: orderability (Perron Rolfsen),
(iv) statistical models (Lin Wang).  

Applying solenoidal dynamical systems theory to 
knot theory enabled
Noguchi \cite{noguchi} 
to prove that the dominant coefficient $a_n$
of the Alexander polynomial
$\Delta_{K}(t) = \sum_{i=0}^{n} a_i t^i , 
a_0 a_n \neq 0$,
of a knot $K$, $\alpha_i$ being the zeroes
(counted with multiplicities) of
$\Delta_{K}(t)$, satisfies
($|\cdot|_p$ is the $p$-adic norm normalized
by $|p|_p = 1/p$ on $\qb_p$):
$$\lo |a_n| = \sum_{p < \infty} 
\,\sum_{|\alpha_i|_{p} > 1} \, \lo |\alpha_i|_{p}
$$
He proved that the distribution of zeroes 
measures a ``distance" of the Alexander module from being
finitely generated as a $\zb$-module, and
that the growth of order of the first homology
of the $r$-fold cyclic covering $X_{r}(K)$
branched over $K$
is related to the zeroes by
 $$ \lim_{r \to \infty, |H_{1}(\cdot)| \neq 0}
 \frac{\lo |H_{1}(X_{r}(K); \zb)|}{r}= \sum_{p \leq \infty} 
\,\sum_{|\alpha_i|_{p} > 1} \, \lo |\alpha_i|_{p} .
$$
Therefore the leading coefficient of 
$\Delta_{K}(t)$ is closely related 
to the homology growth
and the $p$-adic norms of the zeroes $\alpha_i$.

A link, or a knot, is said to be alternating
if it admits a diagram where (along every component)
the strands are passed under-over.
In 2002 Hoste 
(Hirasawa and Murasugi \cite{hirasawamurasugi}
)
stated the following conjecture: 
{\it let $K$ be 
an alternating knot
and $\Delta_{K}(t)$ its Alexander polynomial.
If $\alpha$ is a zero of
$\Delta_{K}(t)$, then
$\Re e(\alpha) > -1$.} 

Hoste's Conjecture
is proved in some cases: cf
\cite{hironakaliechti}
\cite{lyubichmurasugi}
\cite{stoimenow} \cite{stoimenow2}. 
The problem of the geometry and the
boundedness of zeroes
of the (knot and link) Alexander polynomials 
is difficult and related
to two other conjectures on the coefficients of these polynomials, namely
the Fox's trapezoidal Conjecture
and the Log-concavity Conjecture 
\cite{koseleffpecker}
\cite{stoimenow} \cite{stoimenow2}.

In his studies of Lorenz knots 
\cite{dehornoy} \cite{dehornoy2} 
\cite{dehornoy3},
Dehornoy obtained the following much more precise
statement on the geometry of the zero locus
($g$ is the smallest genus of a surface spanning 
the knot; the braid index $b$ is the smallest 
number of strands of a braid whose closure is the knot):

\begin{theorem}[Dehornoy \cite{dehornoy2}]
Let $K$ be a Lorenz knot.
Let $g$ denote its genus and $b$ its braid index. Then the zeroes of the Alexander polynomial of $K$ lie
in the annulus
\begin{equation}
\label{annuluszeroesLorenzknot}
\{z \in \cb \mid
(2 g)^{-4/(b-1)}
\leq |z| \leq (2 g)^{4/(b-1)}\}.
\end{equation}
\end{theorem}
The Alexander polynomial of a Lorenz knot
reflects an intermediate
step between signatures and genus 
\cite{dehornoy2}.
A certain proportion of zeroes lie on the unit circle
and are controlled by the $\omega$-signatures
(Gambaudo and Ghys, cited in \cite{dehornoy2}).
The other zeroes lie within a certain distance from the
unit circle and are controlled by the
house of the Alexander polynomial, which is the 
modulus of the largest zero. The problem of the 
minimality of 
the house of this Alexander polynomial 
is reminiscent of the Schinzel-Zassenhaus Conjecture
if it were expressed as a function of its degree. 
For Lorenz knots this house is interpreted as follows:
it is the growth rate of the associated
homological monodromy
(for details, cf \cite{dehornoy2} $\S$ 2).
Figure 3.3 in \cite{dehornoy2} shows two
examples of Lorenz knots, with 
respective braid index and genus
$(b,g) = (40, 100)$ and
$= (100, 625)$;
interestingly, 
the distribution of 
zeroes within the annulus \eqref{annuluszeroesLorenzknot} appears 
angularly fairly regular 
(in the sense of Bilu's Theorem
\cite{bilu}) but exhibit
lenticuli of zeroes
in the angular sector
$\arg(z) \in [\pi - \pi/2,
\pi + \pi/2]$. Such lenticuli
do exist for integer polynomials of small Mahler measure,
of the variable ``$-x$",
and are shown to be at the origin of the minoration
of the Mahler measure in the problem of Lehmer
in the present note. Though Dehornoy
did not publish (yet) further on
the Mahler measures 
of the Alexander polynomials 
of Lorenz knots, in particular in the way 
$(b,g)$ tends to infinity, it is very probable that
such polynomials are good candidates
for giving small Mahler measures together with a 
topological interpretation of the houses.
The above examples suggest that 
the Alexander polynomials of Lorenz knots 
are not Salem polynomials, though no proof 
seems to exist.

Before Le \cite{le} \cite{le2}, 
Boyd and Rodriguez-Villegas \cite{boyd17}
\cite{boyd18}
\cite{boydrodriguezvillegas}
studied the connections between the Mahler 
measure of the {\em A-polynomial of a knot}
and
the hyperbolic volume of its complement.
$A$-polynomials were introduced in hyperbolic geometry
by
Cooper et al 
\cite{coopercullergilletlongshalen}
(are not Alexander polynomials, though
``$A$" is used in homage to Alexander). 
The 
irreducible factors of 
$A$-polynomials
have (logarithmic) Mahler measures which 
are shown to be finite sums of 
Bloch-Wigner dilogarithms \cite{ganglzagier}
\cite{zagier3}
of algebraic numbers.
The values of such dilogarithms are related
to Chinburg's Conjecture.
Several examples are taken by the authors 
to investigate Chinburg's Conjecture and
its generalization refered to as
Boyd's question (cf also Ray \cite{ray}).
Chinburg's Conjecture \cite{boyd18}
is stated as follows:
for each negative
discriminant $-f$ 
there exists a polynomial
$P = P_f \in \zb[x,y]$ and a nonzero rational number
$r_f$ such that
$\lo {\rm M}(P) = r_f \frac{f \sqrt{f}}{4 \pi}
L(2, \chi_{f})$. 
Boyd's question is stated as follows:
for every number field $F$ having a
number of complex embeddings equal to 1
(i.e. $r_2 = 1$), does there exist
a polynomial
$P = P_F \in \zb[x,y]$
and a rational number
$r_F$ such that
$\lo {\rm M}(P) = r_F Z_F$?, where
$\zeta_F$ is the Dedeking zeta function of $F$ and
$$Z_F = \frac{3 \, |{\rm disc}(F)|^{3/2} \,
\zeta_{F}(2) }{2^{2 n -3} \, \pi^{2 n -1}};$$
the starting point being 
(Smyth \cite{smyth4}):  for $f=3$,
$\lo {\rm M}(1+x+y) = \frac{3 \sqrt{3}}{4 \pi}
L(2, \chi_{3})$.

{\it Jones polynomials of knots and links, lacunarity in coefficient vectors}:

Let $L$ be a hyperbolic link and, for 
$m \geq 1$, denote by $L_m$ the link
obtained from $L$ by adding $m$ full twists on
$n$ strands \cite{rolfsen}
\cite{champanerkarkofman}. 
By Thurston's hyperbolic Dehn surgery,
the volume ${\rm Vol}(\sbb^3 \setminus L_m)$
converges to 
${\rm Vol}(\sbb^3 \setminus (L \cup U))$,
as $m$ tends to infinity,
where $U$ is an unknot encircling
$n$ strands of $L$ such that $L_m$ is 
obtained from $L$ by a 
$-1/m$ surgery on $U$.
More generally, let 
$\underline{m} = (1, m_1 , \ldots, m_s)$,
for
$s \geq 1$,
and
$L_{\underline{m}}
:=
L_{m_1 , \ldots , m_s}$ the multi-twisted link
obtained from a
link diagram $L$ by a  $-1/m_i$ surgery on 
an unknot $U_i$,
for $i=1, \ldots, s$.
In the following theorem convergence of Mahler measures
has to be taken in the sense of
the Boyd Lawton's Theorem \ref{boydlawton}.

\begin{theorem}[Champanerkar - Kofman]
\label{champanerkarkofman}
(i) The Mahler measure 
${\rm M}(V_{L_m}(t))$ of the Jones polynomial
of $L_m$ converges to the Mahler measure of a 2-variable polynomial, as $m$ tends to infinity;

(ii) the Mahler measure 
${\rm M}(V_{L_{\underline{m}}}(t))$ of the Jones polynomial
of $L_{\underline{m}}$ 
converges to the Mahler measure of a $(s+1)$-variable polynomial, as $\underline{m}$ 
tends to infinity.
\end{theorem}
In \cite{champanerkarkofman} Theorem 2.4,
Champanerkar and Kofman
\cite{champanerkarkofman} extended Theorem
\ref{champanerkarkofman} to the convergence of
the Mahler measures
of colored Jones polynomials 
$J_{N}(L_m , t)$
and $J_{N}(L_{\underline{m}} , t)$
for fixed $N$, as $m$, resp.
$\underline{m}$, tends to infinity;
here coloring means by the $N$-dimensional irreducible
representation
of $SL_{2}(\cb)$ with the normalization
of $J_{2}(L_m , t)$
as $J_{2}(L_m , t) 
= (t^{1/2}+t^{-1/2}) V_{L_{m}}(t)$,
resp. for $\underline{m}$. 
They proved that the limit 
$\lim_{m \to \infty}
{\rm M}(J_{N}(L_m , t))$, resp. for 
$\underline{m}$,
is the Mahler measure of 
a multivariate polynomial.
What smallness of limit Mahler measures
can be reached by this construction,
and what are the corresponding geometrical
realizations?

In \cite{champanerkarkofman} 
(Theorem 2.5 and Corollary 3.2) 
Champanerkar and Kofman obtain the following  
theorem, reminiscent of
the moderate lacunarity of the Parry Upper function
occuring at small Mahler measure
(Theorem
\ref{lacunarityVG06},
Theorem \ref{zeronzeron} and 
$\S$\ref{S5.3}) and 
the limit equidistribution of 
conjugates on the unit circle
(Theorem \ref{main_EquidistributionLimitethm}), 
which occur concomitantly
in the (classical) Lehmer problem:

\begin{theorem}[Champanerkar - Kofman]
\label{champanerkarkofman2}
Let $N \geq 1$ be a fixed integer. 
With the above notations,

(i) let $\{\gamma_{i, m}\}$ be 
the set of distinct
roots of the Jones polynomial
$J_{N}(L_m , t)$. Then
$\liminf_{\kappa \to \infty} \#\{\gamma_{i,m} \mid
m \leq \kappa\}
= \infty$, and for 
any $\epsilon > 0$,
there exists an integer $q_{\epsilon}$
such that the number of such roots 
satisfies
$$\#\{\gamma_{i,m} \mid
||\gamma_{i,m}|-1|\geq \epsilon\}< 
q_{\epsilon},$$

(ii) for $m$ sufficiently large,
the coefficient vector
of the Jones polynomial
$J_{N}(L_m , t)$ has nonzero 
fixed blocks of integer digits
separated by gaps (blocks of zeroes)
whose length increases as $m$ tends to infinity.
\end{theorem}
In addition to the 
relative 
limitation of the 
multiplicities of the roots,
Theorem \ref{champanerkarkofman2}
means that, in the annulus 
$1-\epsilon < |z|-1 <1+\epsilon$,
the
clustering of the roots occurs, up to
$q_{\epsilon}$ of them (densification), 
and is
associated with a moderate lacunarity
(``gappiness" in the sense of
\cite{vergergaugry}) of the Jones 
polynomials which increases with 
$m$.
This Theorem has been
extended to other Jones polynomials
by these authors \cite{champanerkarkofman}
and
followed previous experimental observations.
From Theorem 
\ref{champanerkarkofman}
and
Theorem 
\ref{champanerkarkofman2}
it is likely that
such Jones polynomials lead to very small
multivariate Mahler measures, at least are 
good candidates.

Other families of Jones polynomials, their zeroes
and their limit distributions, were investigated, 
for which interesting limit Mahler measures 
may be expected:
e.g. 
Chang and Shrock \cite{changshrock},
Wu and Wang \cite{wuwang},
Jin and Zhang \cite{jinzhang}
\cite{jinzhang2} \cite{jinzhang3}, 
related to models in statistical physics.
The moderate lacunarities
occurring
in the coefficient vectors
of
Jones polynomials were studied
by Franks and Williams
\cite{frankswilliams} in 
the context of 
polynomial invariants
associated with 
braids, knot and links which generalize
Alexander polynomials and Jones polynomials
\cite{frankswilliams}
\cite{freydetal} \cite{jones} 
\cite{murasugi}.

\subsubsection{Arithmetic Hyperbolic Geometry}
\label{S2.4.5}

Leininger's constructions in
\cite{leininger} give the dilatation factors of
pseudo-Anosovs as spectral radii of hyperbolic elements 
in some Fuchsian groups. The minimality
of the Salem numbers as dilatation 
factors is defined in a more general 
context (Neuman and Reid \cite{neumannreid},
Maclachlan and Reid \cite{maclachlanreid},
Ghate and Hironaka
\cite{ghatehironaka}
p. 303).

\begin{theorem}[Neuman-Reid]
\label{neumanreidthm}
The Salem numbers are precisely 
the spectral radii
of hyperbolic elements of 
arithmetic Fuchsian groups
derived from quaternion algebras.
\end{theorem}

Arithmetic hyperbolic groups
are arithmetic groups of isometries
of hyperbolic $n$-space $\hb^n$.
Vinberg and Shvartsman 
\cite{vinbergshvartsman} p. 217 
have defined 
the large subclass of the arithmetic
hyperbolic groups of the simplest type,
in terms of an admissible quadratic form
over a totally real number field $K$.
This subclass includes all
arithmetic hyperbolic groups in even dimensions,
infinitely many wide-commensurability
classes of hyperbolic groups in all dimensions
\cite{maclachlan}, and all noncocompact
arithmetic hyperbolic groups in all 
dimensions.
Isometries of $\hb^n$ are either
elliptic, parabolic or hyperbolic.
An isometry $\gamma \in \hb^n$ is hyperbolic
if and only if there is a unique geodesic
$L$ in $\hb^n$, called the {\em axis} 
of $\gamma$,
along which $\gamma$ acts as a translation
by a positive distance
$l(\gamma)$ called the {\em translation 
length} of $\gamma$.

The following theorems
generalize previous results
of Neumann and Reid \cite{neumannreid} 
in dimension 2 and 3 and
show the important role played by 
the smallest Salem numbers:

\begin{theorem}(Emery, Ratcliffe, Tschantz
\cite{emeryratcliffetschantz})
\label{emerythm}
Let $\Gamma$ be an arithmetic group
of isometries of \, $\hb^n$, $n \geq 2$,
of the simplest type defined 
over a totally real algebraic number $K$.
Let $\Gamma^{(2)}$ be the subgroup of
$\Gamma$ of finite index generated
by the squares of the elements of $\Gamma$.
Let $\gamma$ be a hyperbolic element of 
$\Gamma$, and let 
$\lambda = e^{l(\gamma)}$.
If $n$ is even or $\gamma \in \Gamma^{(2)}$,
then $\lambda$ is a Salem number
such that
$K \subset \qb(\lambda + \lambda^{-1})$
and $\deg_{K}(\lambda) \leq n+1$. 

Conversely, if $\lambda \in T$,
$K$ is a subfield of $\qb(\lambda + \lambda^{-1})$
and $n$ such that $\deg_{K}(\lambda) \leq n+1$, 
then there exists an arithmetic
group $\Gamma$ of isometries of $\hb^n$
of the simplest type defined over $K$
and a hyperbolic element $\gamma \in \Gamma$
such that $\lambda = e^{l(\gamma)}$. 
\end{theorem}

\begin{theorem}(Emery, Ratcliffe, Tschantz
\cite{emeryratcliffetschantz})
\label{emerythm2}
Let $\Gamma$ be an arithmetic group
of isometries of \, $\hb^n$, $n \geq 2$ odd,
of the simplest type defined 
over a totally real algebraic number $K$.
Let $\Gamma^{(2)}$ be the subgroup of
$\Gamma$ of finite index generated
by the squares of the elements of $\Gamma$.
Let $\gamma$ be a hyperbolic element of 
$\Gamma$, and let 
$\lambda = e^{l(\gamma)}$.
Then $\lambda$ is a Salem number
which is square-rootable over $K$. 

Conversely, 
if $\lambda \in T$,
$K$ is a subfield of $\qb(\lambda + \lambda^{-1})$
and $n$ an odd  positive integer 
such that $\deg_{K}(\lambda) \leq n+1$, 
and $\lambda$ is square-rootable over $K$,
then there exists an arithmetic
group $\Gamma$ of isometries of $\hb^n$
of the simplest type defined over $K$
and a hyperbolic element $\gamma \in \Gamma$
such that $\sqrt{\lambda} = e^{l(\gamma)}$.
\end{theorem}

\subsubsection{Salem numbers and Dynamics of Automorphisms of Complex Compact Surfaces}
\label{S2.4.6}

Let $X$ be a compact K\"ahler variety and $f$ 
an automorphism of $X$.
The automorphism $f$ induces an invertible
linear map $f^{*}$
on $H^{*}(X,\cb)$, resp. $H^{*}(X,\rb)$,
$H^{*}(X,\zb)$, which preserves the 
Hodge decomposition, the intersection form, 
the K\"ahler cone. 
Iterating $f$ provides a dynamical system
to which real algebraic integers 
$\geq 1$ are associated.
The 
greatest eigenvalue of the
action of $f$ on
$H^{*}(X,\cb)$ is usually 
called the {\em maximal dynamical degree} of $f$. 
This terminology is the same as 
the one used 
for the $\beta$-shift in
the present note, but the notions are different.
The maximal dynamical degree of $f$
is denoted by $\lambda(f)$; it is related to the 
topological entropy $h_{top}(f)$
of $f$ by $\lo \lambda(f) = 
h_{top}(f)$ 
by a Theorem of Gromov and Yomdin
\cite{gromov}\cite{yomdin}.
Saying that an automorphism is of positive entropy 
is equivalent to
saying that its maximal dynamical degree is $> 1$. 
In particular if $X$ is a surface
the characteristic polynomial of
$f^{*}$ on $H^{2}(X,\zb)$ is a 
(not necessarily irreducible) Salem polynomial
(McMullen 
\cite{macmullen3});
the maximal
dynamical degree $\lambda(f)$ of $f$ is
the spectral radius 
of $f^{*}$ on $H^{1,1}(X)$
and is a Salem number.
Salem numbers are deeply linked to the
geometry of the surface.
Among all complex compact surfaces
\cite{barthhulekpetersvandeven}, 
Cantat \cite{cantat} \cite{cantat2} 
showed that,
if $X$ is a complex compact surface for
which there exists an automorphism of
$X$ having a positive entropy, then 
there exists a birational morphism
from $X$ to a torus, a $K3$ surface, a surface of
Enriques, or the projective plane.
Therefore it suffices to consider 
complex tori
(Oguiso and Truong \cite{oguisotruong},
Reschke \cite{reschke} 2017), 
Enriques surfaces
(Oguiso \cite{oguiso} \cite{oguiso3}),
and $K3$ surfaces 
(Gross and McMullen \cite{grossmacmullen},
McMullen \cite{macmullen3},
Oguiso \cite{oguiso2},
Shimada \cite{shimada})
if $X$ is not rational.

The restriction to compact K\"ahler surfaces
is justified by the fact that
the topological entropy of all automorphisms
vanishes on compact complex surfaces which 
are not K\"ahler (Cantat \cite{cantat2}).
The existence of an automorphism
of positive entropy is a deep question
\cite{brandhorst} 
\cite{brandhorst2}
\cite{cantatdupont} 
\cite{esnaultoguisoyu}
\cite{macmullen3}
\cite{oguiso4}
\cite{oguisotruong2}.

On each type of 
surface,
what are the Salem numbers which appear?
In this context the 
problem of Lehmer can be formulated
by asking what are the minimal Salem
numbers which occur, per type of surface, 
and
the corresponding geometrical 
realizations. 

In \cite{macmullen} McMullen 
gives a general construction
of $K3$ surface automorphisms $f$ from unramified
Salem numbers,
such that, for every such automorphism $f$,
the topological entropy 
$\lo \lambda(f)$ is positive,
together with a criterion for the resulting automorphism
to have a Siegel disk
(domains on which $f$ acts by an irrational rotation). 
The Salem polynomials involved, of the respective
dynamical degrees $\lambda(f)$,
have degree 22, trace -1 and are associated to
an even unimodular lattice
of signature $(3, 19)$ on which $f$ acts as an isometry,
by the Theorem of Torelli. 
The surface is non-projective to carry a Siegel disk.

McMullen \cite{macmullen5} (Theorem A.1) proved 
that Lehmer's number 
(denoted by $\lambda_{10}$)
is the smallest Salem number that can appear as
dynamical degree
of an automorphism of a complex compact surface:
$$h(f) \geq \lo \lambda_{10} = 0.162357\ldots.$$
He gave a geometrical realization of
Lehmer's number
in \cite{macmullen5} on a rational surface
(cf also Bedford and Kim \cite{bedfordkim}),
in \cite{macmullen6} on
a nonprojective $K3$ surface, 
in \cite{macmullen7} on a
projective $K3$ surface.
On the contrary Oguiso \cite{oguiso2}
proved that Lehmer's number cannot be realized
on an Enriques surface.
In \cite{macmullen7} McMullen
proved that the value
$\lo \lambda_d$ arises as the entropy of an 
automorphism of a complex projective
$K3$ surface if 
$$d = 2, 4, 6, 8, 10 ~\mbox{or}~ 18, 
\mbox{but not if}
~d = 14, 16 ~\mbox{or}~ d \geq 20.$$ 
Brandhorst and Gonz\'alez-Alonso
\cite{brandhorstgonzalezalonso} completed
the above ``realizability" list 
with the value $d=12$ (Theorem 1.2 in
\cite{macmullen7}).

For projective surfaces, 
the degree of the
Salem number is bounded by the
rank of the N\'eron-Severi group; for
$K3$ surfaces in characteristic zero it 
is at most 20, due to Hodge theory.
In positive characteristic 
the rank 22 is possible 
(case of supersingular $K3$ surfaces)
\cite{brandhorst2} \cite{yu}.
Therefore all such Salem numbers, when
less than 1.3, are listed in Mossinghoff's list
in \cite{mossinghofflist}, the list being complete
up to degree 44.

Reschke \cite{reschke}
\cite{reschke2} gave a necessary and sufficient 
condition for a Salem number 
to be realized as dynamical 
degree of an automorphism of a complex torus, 
with degrees 2, 4 or 6; 
moreover
he investigated the relations between
the values of the Salem numbers
and the corresponding geometry and projectiveness
of the tori.
Zhao \cite{zhao} extended the method of Reschke
for tori endowed with real structures, showing that
it suffices to consider real abelian surfaces.
Zhao classified such real abelian surfaces
into 8 types according to the number of connected 
components and the simplicity of the underlying
complex abelian surface. 
For each type the set of 
Salem numbers which can be realized by real
automorphisms is determined.  
Zhao \cite{zhao} proved that Lehmer's number 
cannot 
be realized by a real $K3$ surface.

Dolgachev  \cite{dolgachev}
investigated automorphisms
on Enriques surfaces of 
dynamical degrees $> 1$ which are
small Salem numbers, of small degree
2 to 10 (Salem numbers of degree 2 are quadratic Pisot numbers). The method does not allow to conclude
on the minimality of the Salem numbers.
The author
uses the lower semi-continuity 
properties of the dynamical degree 
of an automorphism
$g$ of an algebraic surface $S$ when $(S,g)$ 
varies in an algebraic family.

In positive characteristic 
Brandhorst and 
Gonz\'alez-Alonso 
\cite{brandhorstgonzalezalonso}
proved that the values $\lo \lambda_d$
arise as the entropy of an automorphism of a 
supersingular $K3$ surface
over a field of characteristic $p=5$ if and only if
$d \leq 22$ is even and $d \neq 18$, giving in 
their Appendix B the list
of Salem numbers $\lambda_d$ of 
degree $d$ and respective minimal polynomials.
They develop a strategy to characterize 
the minimal Salem polynomials, 
in particular their cyclotomic factors, for various
realizations in supersingular $K3$ surfaces
having Artin invariants $\sigma$ ranging 
from $1$ to $7$, in characteristic 5.
Yu \cite{yu} 
studied the maximal degrees of the 
Salem numbers arising from 
automorphisms of $K3$ surfaces, defined over an 
algebraically closed field of characteristic $p$,
in terms of the
elliptic fibrations having infinite 
automorphism groups, and Artin invariants.

Oguiso and Truong \cite{oguisotruong}
Dinh, Nguyen and Truong 
\cite{dinhnguyen}
\cite{dinhnguyentruong}
investigated the structure of compact
K\"ahler manifolds, in dimension $\geq 3$,
from the point 
of view of establishing relations between
non-trivial invariant meromorphic fibrations, pseudo-automorphisms $f$ and the dynamical degrees
$\lambda_{k}(f)$.  
Lehmer's problem
can be formulated by asking when the first dynamical degree
$\lambda_{1}(f)$
is a Salem number, what minimal value 
for $\lambda_{1}(f)$
can be reached
and what are the possible geometrical
realizations for the minimal ones.

\section{Asymptotic expansions of the Mahler measures \,${\rm M}(-1 + X + X^n)$}
\label{S3}

\subsection{Factorization of the trinomials $-1+X+X^n$, lenticuli of roots}
\label{S3.1}

The notations used throughout this note
come from
the factorization of $G_{n}(X) := -1 +X +X^n$ 
(Selmer \cite{selmer}, Verger-Gaugry 
\cite{vergergaugry6} Section 2). 
Summing in pairs over complex conjugated imaginary roots, the
indexation of the roots and
the factorization of $G_{n}(X)$ 
are taken as follows:
\begin{equation}
\label{factoGG}
G_{n}(X) = (X - \theta_n) \, 
\left(
\prod_{j=1}^{\lfloor \frac{n}{6} \rfloor} (X - z_{j,n}) ( X - \overline{z_{j,n}}) 
\right)
\times q_{n}(X),
\end{equation}
where
$\theta_n$ is the only (real) root 
of $G_{n}(X)$ in the interval $(0, 1)$, where
$$q_{n}(X) ~=~
\left\{
\begin{array}{ll}
\displaystyle
\left(
\prod_{j=1+\lfloor \frac{n}{6} \rfloor}^{\frac{n-2}{2}} (X- z_{j,n}) (X - \overline{z_{j,n}})
\right) \times (X- z_{\frac{n}{2},n})
&
\mbox{if $n$ is even, with}\\
 & \mbox{$z_{\frac{n}{2},n}$ real $< -1$},\\
\displaystyle
\prod_{j=1+\lfloor \frac{n}{6} \rfloor}^{\frac{n-1}{2}} (X- z_{j,n}) (X - \overline{z_{j,n}})
&
\mbox{if $n$ is odd,}
\end{array}
\right.
$$
where the index $j = 1, 2, \ldots$ is such that 
$z_{j,n}$ is a (nonreal) complex zero of $G_{n}(X)$, except if 
$n$ is even and $j=n/2$, such that
the argument $\arg (z_{j,n})$ of $z_{j,n}$
is roughly equal to $2 \pi j/n$
(Proposition \ref{zedeJIargumentsORDRE1})
and that the family of arguments 
$(\arg(z_{j,n}))_{1 \leq j < \lfloor n/2 \rfloor}$ 
forms a strictly increasing sequence with $j$:
$$0 < \arg(z_{1,n}) < \arg(z_{2,n}) < \ldots 
< \arg(z_{\lfloor \frac{n}{2} \rfloor,n}) \leq \pi.$$
For $n \geq 2$ all the roots of $G_{n}(X)$ 
are simple, and
the roots of $G_{n}^{*}(X) = 1 + X^{n-1} - X^n$, 
as inverses of the roots
of $G_{n}(X)$, are classified in the 
reversed order (Figure \ref{perronselmerfig}).

\begin{proposition}
\label{irredGn}
Let $n \geq 2$. 
If $n \not\equiv 5 ~({\rm mod}~ 6)$, 
then $G_{n}(X)$ is irreducible over $\qb$. 
If $n \equiv 5 ~({\rm mod}~ 6)$, then 
the polynomial $G_{n}(X)$ admits 
$X^2 - X +1$ as irreducible factor
in its factorization and $G_{n}(X)/(X^2 - X +1)$ is irreducible.
\end{proposition}

\begin{proof}
Selmer \cite{selmer}.
\end{proof}

\begin{figure}
\begin{center}
\includegraphics[width=8cm]{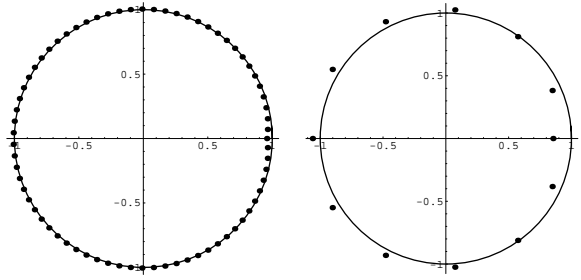}
\end{center}
\caption{
The roots (black bullets) of $G_{n}(z)$
(represented here with $n=71$ and $n=12$)
are uniformly distributed near $|z|=1$
according to the theory of
Erd\H{o}s-Tur\'an-Amoroso-Mignotte.
A slight bump appears in the half-plane
$\Re(z) > 1/2$ in the neighbourhood of $1$, at 
the origin of the different regimes of asymptotic expansions.
The dominant root of
$G_{n}^{*}(z)$ is the Perron number $\theta_{n}^{-1} > 1$, with
$\theta_n$ the unique root of $G_n$
in the interval $(0,1)$.}
\label{perronselmerfig}
\end{figure}

\begin{proposition}
\label{closetoouane}
For all $n \geq 2$, all zeros $z_{j,n}$ and
$\theta_n$
of the polynomials
$G_{n}(X)$ have a modulus in the interval
\begin{equation}
\label{bounboun}
\Bigl[ ~1- \frac{2 \,\lo n}{ n}, ~1 + \frac{2 \,\lo 2}{n} ~\Bigr] ,
\end{equation}

(ii)~ the trinomial $G_{n}(X)$ admits a unique real root $\theta_n$ in
the interval $(0,1)$.
The sequence $(\theta_n)_{n \geq 2}$ is strictly increasing, $\lim_{n \to +\infty} \theta_n = 1$, with
$\theta_2 = \frac{2}{1+\sqrt{5}} = 0.618\ldots$,

(iii)~ 
the root $\theta_n$ 
is the unique root of smallest modulus among all the roots
of $G_{n}(X)$; 
if $n \geq 6$, 
the roots of modulus $< 1$ of $G_{n}(z)$
in the closed upper half-plane have the following properties:

(iii-1)~~ $\theta_n ~<~ |z_{1,n}|$,

(iii-2) ~ for any pair of successive indices
$j, j+1$ in $\{1, 2, \ldots, \lfloor n/6 \rfloor\}$,
$$| z_{j,n} | < |z_{j+1,n} | .$$
\end{proposition}

\begin{proof}
(i)(ii) Selmer \cite{selmer}, pp 291--292; (iii-1) Flatto, Lagarias and Poonen
\cite{flattolagariaspoonen},
(iii-2) Verger-Gaugry \cite{vergergaugry6}.
\end{proof}

The Pisot number (golden mean)
$\theta_{2}^{-1} = \frac{1+\sqrt{5}}{2} = 1.618\ldots$
is the largest
Perron number in the family $(\theta_{n}^{-1})_{n \geq 2}$. 
The interval $( 1, \frac{1+\sqrt{5}}{2}\, ]$ is
partitioned by
the strictly decreasing sequence
of Perron numbers
$(\theta_{n}^{-1})$ as
\begin{equation}
\label{decoupage}
( 1, \frac{1+\sqrt{5}}{2}\, ] ~=~
\left(
\bigcup_{n=2}^{\infty}
\left[ \, \theta_{n+1}^{-1} , \theta_{n}^{-1}
\, \right)
\right)
~ \bigcup ~\left\{
\theta_{2}^{-1}
\right\}.
\end{equation}

By the direct method of asymptotic expansions of the roots, as in \cite{vergergaugry6}, or by Smyth's Theorem \cite{smyth} 
(Dubickas \cite{dubickas}),
since the trinomials $G_{n}(X)$ are not reciprocal, 
the Mahler measure of $G_n$
satisfies 
\begin{equation}
\label{smyththeorem}
{\rm M}(\theta_{n}) ~=~ {\rm M}(G_n) ~\geq~ \Theta = 1.3247\ldots, \qquad n \geq 2,
\end{equation}
where $\Theta = \theta_{5}^{-1}$ is the smallest Pisot number, dominant root of
the Pisot polynomial 
$X^3 - X - 1 = -G_{5}^{*}(X) / (X^2 - X + 1)$.

\begin{proposition}
\label{rootsdistrib}
Let $n \geq 2$. 
Then (i) the number $p_n$ of roots of $G_{n}(X)$
which lie inside the open sector $\mathcal{S} =
\{ z \mid |\arg (z)| < \pi/3 \}$ is equal to
\begin{equation}
\label{pennn}
1 + 2 \lfloor \frac{n}{6} \rfloor,
\end{equation} 

(ii) the correlation between the
geometry of the roots of $G_{n}(X)$ 
which lie inside the unit disc
and the upper half-plane and their indexation is given by:
\begin{equation}
\label{rootsinside}
j \in \{1, 2, \ldots, \lfloor \frac{n}{6} \rfloor \}
~\Longleftrightarrow~ \Re(z_{j,n}) > \frac{1}{2} ~\Longleftrightarrow~ |z_{j,n}| < 1,
\end{equation} 
and
the Mahler measure
M$(G_n)$ of the trinomial $G_{n}(X)$ is
\begin{equation}
\label{mahlerGG}
{\rm M}(G_{n}) ~=~ {\rm M}(G_{n}^{*}) ~=~ \theta_{n}^{-1} \, 
\prod_{j=1}^{\lfloor n/6 \rfloor} |z_{j,n}|^{-2}.
\end{equation}
\end{proposition}
\begin{proof}
Verger-Gaugry \cite{vergergaugry6}, Proposition 3.7.
\end{proof}

\subsection{Asymptotic expansions: roots of $G_n$ and relations}
\label{S3.2}

The (Poincar\'e) asymptotic expansions of the roots of $G_n$
(and $G_{n}^{*}$)
are generically written:
${\rm Re}(z_{j,n}) = 
{\rm D}({\rm Re}(z_{j,n})) + 
{\rm tl}({\rm Re}(z_{j,n}))$,
${\rm Im}(z_{j,n}) = {\rm D}({\rm Im}(z_{j,n})) 
+ {\rm tl}({\rm Im}(z_{j,n}))$,
$\theta_n = {\rm D}(\theta_n) + 
{\rm tl}(\theta_n)$,
where "D" and "tl" stands for
 {\it ``development"}
(or {\it ``limited expansion"}, 
or {\it ``lowest order terms"})
and
"tl" for {\it ``tail"} 
(or ``remainder", or {\it ``terminant"} in \cite{dingle}).
They are given at a sufficiently high order
allowing to deduce the asymptotic expansions 
of the Mahler measures ${\rm M}(G_n)$.
The terminology {\it order} 
comes from the general theory
(Borel \cite{borel}, Copson \cite{copson}, Dingle \cite{dingle}, 
Erd\'elyi\cite{erdelyi}); the 
approximant solutions of a polynomial
equation say $G(z) = 0$ which arise naturally correspond to
{\it order} $1$. The solutions corresponding to {\it order} $2$
are obtained by inserting the {\it order $1$ approximant solutions}
into the equation $G(z) = 0$, 
for getting {\it order $2$ approximant solutions}. 
And so on, as a function of $\deg{G}$. 
The order is
the number of steps in this iterative process. 
Poincar\'e \cite{poincare}
introduced
this method of divergent series
for the $N$ - body problem in celestial mechanics;
this method does not appear in number theory
in the book ``Divergent series" of Hardy.
The equivalent of the variable time $t$ (in celestial 
mechanics) will be the dynamical degree
$\dyg(\house{\alpha})$
of the house of the algebraic integer
$\alpha$ in number theory 
(with $|\alpha| > 1$), 
a new ``variable concept" introduced in the present study; 
for the trinomials
$G_n$ it will be $n$.

The asymptotic expansions of 
$\theta_n$ and those roots $z_{j,n}$
of $G_{n}(z)$ which lie 
in the first quadrant
are (divergent) 
sums of functions of only {\em one} 
variable, which is $n$, while
those of the other roots $z_{j,n}$
are functions of a couple of 
{\em two} variables which is: \begin{itemize}
\item[$\bullet$] $(n, j/n)$
in the angular sector
$\pi / 4 > \arg z  > 2 \pi \, \lo n / n$, 
and 
\item[$\bullet$] $(n, j/\lo n)$
in the angular sector
$2 \pi \, \lo n / n > \arg z > 0$.
\end{itemize}
The first sector is the main angular sector.
The second sector if the bump angular sector.
A unique regime of asymptotic expansion
exists in the main angular sector, whereas
two regimes of asymptotic expansions
do exist in the bump sector
(Appendix, and \cite{vergergaugry6}).
These regimes 
are separated by two sequences
$(u_n)$ and $(v_n)$, to which 
the second variable
$j/n$, resp. $j/\lo n$, is compared.
Details can be found in the Appendix.

\begin{proposition}
\label{thetanExpression}
Let $n \geq 2$.
The root $\theta_n$ can be expressed as: 
$\theta_n = {\rm D}(\theta_n) + {\rm tl}(\theta_n)$ with
${\rm D}(\theta_n) = 1 -$
\begin{equation}
\label{DthetanExpression}
\frac{\lo n}{n}
\left(
1 - \bigl(
\frac{n - \lo n}{n \, \lo n + n - \lo n}
\bigr)
\Bigl(
\lo \lo n - n 
\lo \Bigl(1 - \frac{\lo n}{n}\Bigr)
- {\rm Log} n
\Bigr)
\right)
\end{equation}
and
\begin{equation}
\label{tailthetanExpression}
{\rm tl}(\theta_n) ~=~ \frac{1}{n} \, O \left( \left(\frac{\lo \lo n}{\lo n}\right)^2 \right),
\end{equation}
with the constant $1/2$ involved in $O \left(~\right)$.
\end{proposition}

\begin{proof}
\cite{vergergaugry6} Proposition 3.1.
\end{proof}

\begin{lemma}
\label{remarkthetan}
Given the limited expansion D$(\theta_n)$ of
$\theta_n$ as in \eqref{DthetanExpression}, denote
$$\lambda_n := 1 - (1 - {\rm D}(\theta_{n}))\frac{n}{\lo n}.$$ 
Then $\lambda_n = {\rm D}(\lambda_{n}) + {\rm tl}(\lambda_{n})$, with
\begin{equation}
\label{eqqww}
{\rm D}(\lambda_n) = \frac{\lo \lo n}{\lo n} \left(\frac{1}{1+\frac{1}{\lo n}}\right), \qquad
{\rm tl}(\lambda_n) = O\left( \frac{\lo \lo n}{n}  \right).
\end{equation}
with the constant 1 in the Big O.
\end{lemma}

\begin{proof}
\cite{vergergaugry6} Lemma 3.2.
\end{proof}

In the sequel, for short, we write $\lambda_n$ instead of
${\rm D}(\lambda_n)$.

\begin{proposition}
\label{zjjnnExpression}
Let $n \geq n_0 = 18$ and $1 \leq j \leq \lfloor \frac{n-1}{4} \rfloor$. 
The roots $z_{j,n}$ of $G_{n}(X)$ have the following asymptotic expansions:
$z_{j,n} = {\rm D}(z_{j,n}) + {\rm tl}(z_{j,n})$
in the following angular sectors:

\begin{itemize}
\item[(i)] \underline{
Sector $\frac{\pi}{2} > \arg z > 2 \pi \frac{\lo n}{n}$ (main sector):}
$${\rm D}(\Re(z_{j,n})) = \cos\bigl(2 \pi \frac{j}{n}\bigr) + 
\frac{\lo \bigl(2 \, \sin\bigl(\pi \frac{j}{n}\bigr)
\bigr)}{n},$$
$${\rm D}(\Im(z_{j,n})) ~=~ \sin\bigl(2 \pi \frac{j}{n}\bigr) 
+ 
\tan\bigl(\pi \frac{j}{n}\bigr)
\, 
\frac{\lo \bigl(2 \, \sin\bigl(\pi \frac{j}{n}\bigr)
\bigr)}{n},$$
with
$${\rm tl}(\Re(z_{j,n})) ~=~  {\rm tl}(\Im(z_{j,n})) ~=~
\frac{1}{n} \, O \left( \left(\frac{\lo \lo n}{\lo n}\right)^2 \right)$$
and the constant 1 in the Big $O$,
\vspace{0.1cm}

\item[(ii)] \underline{
``Bump" sector $2 \pi \frac{\lo n}{n} > \arg z > 0$ : }

\subitem \underline{$\bullet$ 
Subsector $2 \pi \frac{\sqrt{(\lo n) (\lo \lo n)}}{n} > \arg z > 0$:}
$${\rm D}(\Re(z_{j,n})) = \theta_n + \frac{2 \pi^2}{n} \left(\frac{j}{\lo n}\right)^2 \bigl( 1+2 \lambda_n \bigr), 
$$
$${\rm D}(\Im(z_{j,n})) = \frac{2 \pi \lo n}{n}  \left(\frac{j}{\lo n}\right)
\left[1 - \frac{1}{\lo n} (1 + \lambda_n)\right],$$

with
$${\rm tl}(\Re(z_{j,n})) = \frac{1}{n \lo n} \left(\frac{j}{\lo n}\right)^2 O\left(
\left(\frac{\lo \lo n}{\lo n}\right)^2\right),
$$
$$
{\rm tl}(\Im(z_{j,n})) = \frac{1}{n \lo n} \left(\frac{j}{\lo n}\right)
O\left(
\left(\frac{\lo \lo n}{\lo n}\right)^2\right),$$

\subitem \underline{$\bullet$
Subsector $2 \pi \frac{\lo n}{n} > \arg z >
2 \pi \frac{\sqrt{(\lo n) (\lo \lo n)}}{n}$:}
$${\rm D}(\Re(z_{j,n})) = \theta_n +
\frac{2 \pi^2}{n} \left(\frac{j}{\lo n}\right)^2
\left(
1 + \frac{2 \pi^2}{3} \left(\frac{j}{\lo n}\right)^2
\left(
1+\lambda_n
\right)
\right)
$$
$$  {\rm D}(\Im(z_{j,n})) ~=~ $$
$$\frac{2 \pi \lo n}{n} \left(\frac{j}{\lo n}\right)
\left[1 \!-\! \frac{1}{\lo n}
\left(
1 - \frac{4 \pi^2}{3} \left(\frac{j}{\lo n}\right)^2
\left( 1 - \frac{1}{\lo n}
(
1 - \lambda_n
) \right)
\right)\right],
$$
with
$${\rm tl}(\Re(z_{j,n})) = \frac{1}{n} O \left(
\left(\frac{j}{\lo n}
\right)^6
\right),
{\rm tl}(\Im(z_{j,n})) = \frac{1}{n} O\left(
\left(\frac{j}{\lo n}\right)^5
\right).$$
\end{itemize} 
\end{proposition}

\begin{proof}
\cite{vergergaugry6} Proposition 3.4.
\end{proof}

Outside the ``bump sector"
the moduli of the roots $z_{j,n}$
are readily obtained as (Proposition 3.5 in
\cite{vergergaugry6}):
\begin{equation}
\label{zjnModulusFirst}
|z_{j,n}| =
 1 + \frac{1}{n} \, \lo \bigl(2 \, \sin\bigl(\frac{\pi j}{n}\bigr)  \bigr) +
\frac{1}{n} O
\left( \frac{(\lo \lo n)^2}{(\lo n)^2}
\right),
\end{equation}
with the constant 1 in the Big $O$ (independent of $j$).
The following expansions of the $|z_{j,n}|$s 
at the order 3 will be needed in the method of Rouch\'e.

\begin{proposition}
\label{zedeJIargumentsORDRE1}
$$\arg(z_{j,n}) = 2 \pi (\frac{j}{n} 
+ {A}_{j,n})
\quad
\mbox{with}\quad 
A_{j,n}
=
-\frac{1}{2 \pi n}
\left[
\frac{1 - \cos(\frac{2 \pi j}{n})}
{\sin(\frac{2 \pi j}{n})}
\lo (2 \sin(\frac{\pi j}{n}))
\right]$$
$$\mbox{and}\qquad\qquad
{\rm tl}(\arg(z_{j,n}))=
\frac{1}{n} O\Bigl(\left(
\frac{\lo \lo n}{\lo n}
\right)^2 \Bigr).
$$
\end{proposition}

\begin{proof}
\S 6 in \cite{vergergaugry6}.
\end{proof}

\begin{proposition}
\label{zedeJImodulesORDRE3}
For all $j$ such that $\pi / 3 \geq \arg z_{j,n} > 
2 \pi \frac{\lceil v_n \rceil}{n}$,
the asymptotic expansions of the
moduli of the roots $z_{j,n}$ are
$$|z_{j,n}| = {\rm D}(|z_{j,n}|) +
{\rm tl}(|z_{j,n}|)$$ 
with
\begin{equation}
\label{absolzjjnn_mieux}
{\rm D}(|z_{j,n}|) = 1 + \frac{1}{n} \, \lo \bigl(2 \, \sin\bigl(\frac{\pi j}{n}\bigr)  \bigr) +
 \frac{1}{2 n}\left( \frac{\lo \lo n}{\lo n} \right)^2
\end{equation}
and
\begin{equation}
\label{absolzjjnntail_mieux}
{\rm tl}(|z_{j,n}|) = \frac{1}{n} O
\left( \frac{(\lo \lo n)^2}{(\lo n)^3}
\right)
\end{equation}
where the constant involved in $O(~)$ is 1
(does not depend upon $j$).
\end{proposition}

\begin{proof}
\cite{vergergaugry6} Section 5.1.
\end{proof}

The following asymptotic expansions 
in Proposition \ref{zedeUNmodule}, 
Proposition \ref{zedeJImoduleMoinsUN} and
Proposition \ref{zedeJIMoinsUNzedeJI}
will be used in the
method of Rouch\'e in \S 5. 

\begin{proposition}
\label{zedeUNmodule}
For $n \geq 18$, the modulus of the first root
$z_{1,n}$ of $G_{n}(z) = -1 +z+z^n$ is
\begin{equation}
\label{zedeUN}
|z_{1,n}| = 1 - \frac{\lo n - \lo \lo n}{n}  
+ \frac{1}{n} O\left(\frac{\lo \lo n}{\lo n}\right)
\end{equation}
and
\begin{equation}
\label{UNmoinszedeUN}
|-1 + z_{1,n}| =
\frac{\lo n - \lo \lo n}{n}
+ \frac{1}{n} O\left(
\frac{\lo \lo n}{\lo n}\right)
\end{equation}
with the constant 1 in the two Big Os.
\end{proposition}

\begin{proof}
The root
$z_{1,n}$ belongs to the
subsector $2 \pi \frac{\sqrt{(\lo n) (\lo \lo n)}}{n} > \arg z > 0$:
first, from Lemma \ref{remarkthetan}, 
the asymptotic expansion of
$\lambda_n$ is
$$\lambda_n = 
\frac{\lo \lo n}{\lo n}
+ O(\frac{\lo \lo n}{(\lo n)^2})
$$
with the constant 1 in the Big $O$.
Since
${\rm D}(|z_{1,n}|) =
{\rm D}(\Re(z_{1,n}))
(1 + \bigl(\frac{{\rm D}(\Im(z_{1,n}))}
{{\rm D}(\Re(z_{1,n}))}\bigr)^2)^{1/2}$,
that
$${\rm D}(\Re(z_{1,n})) = \theta_n + 
\frac{2 \pi^2}{n} 
\bigl(\frac{1}{\lo n}\bigr)^2 
\bigl( 1+2 \lambda_n \bigr), \,
{\rm D}(\Im(z_{1,n})) = \frac{2 \pi}{n}
\bigl[1 - \frac{1}{\lo n} (1 + \lambda_n)\bigr]$$
(Proposition \ref{zjjnnExpression}) and
$$\theta_{n} = 1 - \frac{\lo n}{n}(1 - \lambda_n) 
+ \frac{1}{n} O\left(\left(\frac{\lo \lo n}{\lo n}\right)^2\right)
$$
(Proposition \ref{thetanExpression}) we deduce 
\eqref{zedeUN} and the expansion
\eqref{UNmoinszedeUN} from
the expansion of $\lambda_n$.
\end{proof}

\begin{proposition}
\label{zedeJImoduleMoinsUN}
For $n \geq 18$, the modulus of 
$-1 + z_{j,n}$, where  
$z_{j,n}$ is the $j$-th root
of $G_{n}(z) = -1 +z+z^n$, 
$\lceil v_n \rceil \leq j \leq 
\lfloor n/6 \rfloor$, is
\begin{equation}
\label{UNmoinszedeJI}
|-1 + z_{j,n}| 
~=~
2 \sin(\frac{\pi j}{n})
+ \frac{1}{n} O\Bigl(
\Bigl(
\frac{\lo \lo n}{\lo n}
\Bigr)^2
\Bigr)
\end{equation}
with the constant 1 in the Big O.
\end{proposition}

\begin{proof}
From \eqref{zjnModulusFirst}, 
Proposition
\ref{zjjnnExpression}
and Proposition \ref{zedeJImodulesORDRE3}, 
the identity
$$
|-1 + z_{j,n}|^2
=
(-1 + \Re(z_{j,n}))^2
+ (\Im(z_{j,n}))^2
=
1 + |z_{j,n}|^2
- 2 \Re(z_{j,n})
$$
implies:
$|-1 + z_{j,n}|^2
=$
$$2 
-
2 \cos\bigl(2 \pi \frac{j}{n}\bigr)
+
\frac{1}{n} O\Bigl(\Bigl(
\frac{\lo \lo n}{\lo n}
\Bigr)^2\Bigr)
=
4 \sin^2 \bigl(\frac{\pi j}{n}\bigr)
+
\frac{1}{n} O\Bigl(\Bigl(
\frac{\lo \lo n}{\lo n}
\Bigr)^2\Bigr)
$$
with the constant 4
in the Big $O$.
We deduce \eqref{UNmoinszedeJI}.
\end{proof}

\begin{proposition}
\label{zedeJIMoinsUNzedeJI}
For $n \geq 18$, the modulus of 
$(-1 + z_{j,n})/z_{j,n}$, where  
$z_{j,n}$ is the $j$-th root
of $G_{n}(z) = -1 +z+z^n$, 
$\lceil v_n \rceil \leq j \leq 
\lfloor n/6 \rfloor$, is
\begin{equation}
\label{UNmoinszedeJIzedeJI}
\frac{|-1 + z_{j,n}|}{|z_{j,n}|} 
~=~
2 \sin(\frac{\pi j}{n})
\Bigl(
1 -
\frac{1}{n} \lo (2 \sin(\frac{\pi j}{n}))
\Bigr)
+
\frac{1}{n} O\left(
\left(
\frac{\lo \lo n}{\lo n}
\right)^2
\right)
\end{equation}
with the constant 2
in the Big O.
\end{proposition}

\begin{proof}
The expansion \eqref{UNmoinszedeJIzedeJI}
readily comes 
\eqref{UNmoinszedeJI}
and
$|z_{j,n}|$ given by
Proposition \ref{zedeJImodulesORDRE3}. 
\end{proof}

\subsection{Minoration of the Mahler measure}
\label{S3.3}

In Verger-Gaugry \cite{vergergaugry6} 
the theory of ``\`a la Poincar\'e"
asymptotic expansions
is shown to
give  ``controlled" approximants
of
the set of the values of the Mahler measures
M$(G_n)$ and an exact value of
its limit point. Compared to several methods 
(Amoroso \cite{amoroso2} \cite{amoroso3},
Boyd and Mossinghoff \cite{boydmossinghoff},
Dixon and Dubickas \cite{dixondubickas},
Langevin \cite{langevin}, Smyth \cite{smyth5}),
the present approach is new
in the sense that the use
of auxiliary functions by
Dobrowolski \cite{dobrowolski2}
is replaced by the R\'enyi-Parry
dynamics of the
Perron numbers
$(\theta_{n}^{-1})_{n \geq 2}$.
Let us briefly mention the results.
The product
\begin{equation}
\label{mmapprox}
\Pi_{G_n} ~:=~
D({\rm M}(G_n)) ~=~ D(\theta_n)^{-1} \times 
\prod_{\stackrel{
z_{j,n} ~\mbox{{\tiny in}}~ |z|<1}{{\rm {\tiny outside ~bump}}}}
{\rm D}(|z_{j,n}|)^{-2}
\end{equation}
is considered, instead of 
\begin{equation}
\label{mmapproxM}
{\rm M}(G_n) ~=~
\theta_{n}^{-1} \,
\prod_{j=1}^{\lfloor n/6 \rfloor} |z_{j,n}|^{-2}
 = 
 \prod_{\lc_{\theta_{n}^{-1}}} |z|^{-1} 
\end{equation}
as approximant value of ${\rm M}(G_n)$.
In \eqref{mmapprox}
the zeroes $z_{j,n}$
present in the bump sector are 
discarded since they do not contribute to the
limited asymptotic expansions, as shown in 
\cite{vergergaugry6} Section 4.2.
In \cite{vergergaugry6} Section 4,
the two limits
$\displaystyle \lim_{n \to +\infty} \Pi_{G_n}$ and 
$\displaystyle \lim_{n \to +\infty} {\rm M}(G_n)$ 
are shown to exist, to be equal (and greater 
than $\Theta$). 
Theorem \ref{main1} is obtained either
by Boyd-Smyth's method
of bivariate Mahler measures
(\cite{vergergaugry6} Section 4.1) or by
the method of asymptotic expansions 
(Verger-Gaugry 
\cite{vergergaugry6} Section \ref{S4.2}).
The
{\it first derived set} of $\{{\rm M}(G_n) \mid n \geq 2\}$
is reduced to one element, 
$1.38135\ldots$ as follows.

\begin{theorem}
\label{main1}
Let $\chi_3$ be the uniquely specified odd
character of conductor $3$
($\chi_{3}(m) $
$= 0, 1$ or $-1$ according to whether $m \equiv 0, \,1$ or
$2 ~({\rm mod}~ 3)$, equivalently
$\chi_{3}(m) = \left(\frac{m}{3}\right)$ the Jacobi symbol),
and denote $L(s,\chi_3) = \sum_{m \geq 1} \frac{\chi_{3}(m)}{m^s}$
the Dirichlet L-series for the character $\chi_{3}$. Then,
with $\Lambda$ given by
\eqref{limitMahlGn},
$\lim_{n \to +\infty} {\rm M}(G_n) ~=~ 
{\rm M}(-1+z+y)=\Lambda = 1.38135\ldots$
\end{theorem}

\begin{proof}
\cite{vergergaugry6} Theorem 1.1; Smyth \cite{smyth4}.
\end{proof}

Introduced in the product \eqref{mmapproxM}, the
terminants of the asymptotic expansions 
of the moduli of the roots $z_{j,n}$ and of
$\theta_n$
provide the higher-order terms of the 
asymptotic expansion of  ${\rm M}(G_n)$ as follows.

\begin{theorem}
\label{mahlerGntrinomial}
Let $n_0$ be an integer such that 
$\frac{\pi}{3} > 2 \pi \frac{\lo n_0}{n_0}$,
and let $n \geq n_0$.
Then,
\begin{equation}
\label{mggnfluctu_}
{\rm M}(G_n) ~=~ 
\Lambda ~ \Bigl( 1 + r(n) 
\, \frac{1}{\lo n}
+ O\left(\frac{\lo \lo n}{\lo n}\right)^2 \bigr)
\Bigr)
\end{equation}
with
the constant 
$1/6$
involved in the Big O, and
with $r(n)$ real,
$|r(n)| \leq 1/6$.
\end{theorem}

\begin{proof}
\cite{vergergaugry6} Theorem 1.2.
\end{proof}

In Theorem \ref{mahlerGntrinomial} we take
$n_0 = 18$. For the small values of $n$, we have:
$${\rm M}(G_2) = \theta_{2}^{-1} = 
\frac{1+\sqrt{5}}{2} = 1.618\ldots$$ 
and the following lower bound.

\begin{proposition}
\label{maincoro3}
{\rm M}$(G_n) \geq {\rm M}(G_5) = \theta_{5}^{-1} = 
\Theta = 1.3247\ldots$ for all $n \geq 3$,
with equality if and only if $n=5$.
\end{proposition}

\begin{proof}
\cite{vergergaugry6} Corollary 1.4.
\end{proof}

The minoration of the
residual distance between the two
algebraic integers $1$ and 
$\theta_{n}^{-1}$ is deduced from the
Zhang-Zagier height and Doche's improvement
as follows. 

\begin{proposition}
\label{maincoro7}
Let $u=0$ except if $n \equiv 5$ mod $6$ in which case $u = -2$. Then,
\begin{equation}
\label{maincoro7equation}
{\rm M}(\theta_{n}^{-1} -1) ~\geq~ 
\frac{\eta^{n+u}}
{\Lambda} \bigl(1 - 
\frac{1}{ 6 \, \lo n}
\bigr), \qquad n \geq 2,
\end{equation}
with $\eta = 1.2817770214$.
\end{proposition} 

\begin{proof}
Except for a finite subset of algebraic numbers,
the minoration ${\rm M}(\alpha)
{\rm M}(1-\alpha) \geq 
(\theta_{2}^{-1/2})^{\deg(\alpha)}$ was 
established by Zagier \cite{zagier} and improved by
Doche \cite{doche}, with the lower bound
$\eta > \sqrt{\theta_{2}^{-1}}$.
The minorant
\eqref{maincoro7equation} follows from
\eqref{12817770214}  
and \eqref{mggnfluctu_}.
\end{proof}

The present method of asymptotic expansions 
gives a new insight into 
the problem of the minoration of the 
Mahler measure ${\rm M}(G_n)$.
Comparing with 
Dobrowolski's minoration 
\eqref{minoDOBRO1979}
 \cite{dobrowolski2},
Theorem
\ref{mahlerGntrinomial} implies the following
minoration of ${\rm M}(\theta_{n}^{-1})$ 
which is better than \eqref{minoDOBRO1979}.
\begin{theorem}
\label{maincoro5}
\begin{equation}
\label{minoVG2015}
\mbox{{\rm M}}(\theta_{n}^{-1}) ~>~ \Lambda 
- \frac{\Lambda}{6} 
\left(\frac{1}{\lo n}\right)
\, , \quad n \geq n_{1} = 2.
\end{equation}
\end{theorem}

\begin{proof}
\cite{vergergaugry6} Corollary 1.6.
\end{proof}
The extremality of the Perron numbers
$\theta_{n}^{-1}$ occurs only for
$n=2, 3$.
In general, if extremality holds,
by Conjecture \ref{CJ6boyd} (iii), 
it would be
associated with a lenticular distribution of roots
(of modulus $> 1$)
which admits a proportion asymptotically equal to
$\frac{2}{3} n$. 
For the trinomials $G_{n}^{*}, \,n \geq 4$,
this proportion is only $\frac{1}{3} n$, 
for $n$ large.

\section{The $\beta$-shift, Artin-Mazur dynamical zeta function, generalized Fredholm determinant, Perron-Frobenius operator, Parry Upper function, with $\beta > 1$ a real algebraic number}
\label{S4}

In 1957 R\'enyi \cite{renyi}
introduced new representations of a real number
$x$, using a positive function
$y=f(x)$ and infinite iterations of it,
in the form of an ``$f$-expansion", as
$$x = \epsilon_0 +
f(\epsilon_1
+f(\epsilon_2 + \ldots))$$
with ``digits" $\epsilon_i$
in some alphabet 
and remainders
$f(\epsilon_n
+f(\epsilon_{n+1} + \ldots))$.
This approach considerably enlarged
the usual decimal numeration system, 
and numeration systems
in integer basis, by allowing
arbitrary real bases of 
numeration
(Fraenkel \cite{fraenkel}, Lothaire
\cite{lothaire}, Chap. 7): 
let $\beta > 1$ not an integer
and consider
$f(x) = x/\beta$ if $0 \leq x \leq \beta$,
and $f(x) = 1$ if $\beta < x$. 
Then the
$f$-expansion of $x$ is the
representation of $x$ in base $\beta$
as
$$x = \epsilon_0 + \frac{\epsilon_1}{\beta}
+\frac{\epsilon_2}{\beta^2}
+\ldots
+ \frac{\epsilon_n}{\beta^n} + \ldots.$$
In terms of
dynamical systems, in 1960, 
Parry \cite{parry} 
\cite{parry2}
\cite{parry3}
has reconsidered
and studied the ergodic properties of
such representations of real numbers
in base $\beta$, in particular the conditions of
faithfullness of the map:
$x \longleftrightarrow (\epsilon_i)_i$
and the complete set of 
admissible sequences 
$(\epsilon_0 , \epsilon_1 , \epsilon_2 , \ldots)$
for all real numbers (recalled
in $\S$ \ref{S4.1}).
This complete set is called
the language in base $\beta$.

In the present note, though
the coding of real numbers
in arbitrary basis $\beta > 1$
is important in itself, with many generalizations
\cite{boyle2} 
\cite{baratbertheliardetthuswaldner},
\cite{lindmarcus}
\cite{bertherigo}
\cite{pytheasfogg},
\cite{akiyamapethoe}, 
\cite{rigo},
\cite{schmidt2}
\cite{scheicherthuswaldner}, 
it is not the direction we will follow;
nor the direction of toral automorphisms
and dynamics of Mahler measure
\cite{lind} \cite{einsiedler}.
On the contrary we will
be mostly interested in the analytical functions
associated with the $\beta$-shift, i.e.
with the language in base $\beta$
($\S$ \ref{S4.2}),
$1 < \beta < 2$ first being fixed,
and then vary continuously  
the basis of numeration 
$\beta$ in $\overline{\qb} \cap (1, +\infty)$
taking the limit to $1^+$,
to use the limit properties of these functions
for solving the problem of Lehmer.

The basic analytical function on which relies 
the solving of the 
Lehmer's problem is the Parry Upper function
$f_{\beta}(z)$. We will concentrate 
on its properties, presenting it in a broader context
together with the dynamical zeta function
$\zeta_{\beta}(z)$ and the transfer
operator $\mathcal{L}_{t\beta}$, 
their respective literatures
being complementary.

\subsection{The $\beta$-shift, $\beta$-expansions, lacunarity and symbolic dynamics}
\label{S4.1}

Let $\beta > 1$ be a real number
and let
$\mathcal{A}_{\beta} := \{0, 1, 2, \ldots, 
\lceil \beta - 1 \rceil \}$. If $\beta$ is not
an integer, then 
$\lceil \beta - 1 \rceil
= \lfloor \beta \rfloor$.
Let $x$ be a real number
in the interval $[0,1]$.
A representation in base $\beta$ 
(or a $\beta$-representation; or a $\beta$-ary representation if 
$\beta$ is an integer) of $x$ 
is an infinite word
$(x_i)_{i \geq 1}$ of
$\mathcal{A}_{\beta}^{\nb}$
such that
$$x = \sum_{i \geq 1} \, x_i \beta^{-i} \, .$$
The main difference with the case
where $\beta$ is an integer
is that $x$ may have several representations.
A particular $\beta$-representation,
called the $\beta$-expansion, 
or the greedy $\beta$-expansion,
and denoted
by $d_{\beta}(x)$, of $x$
can be computed either by the 
greedy algorithm, or equivalently by the
$\beta$-transformation
$$T_{\beta} : \,\, x \, \mapsto \,  \beta x \quad(\hspace{-0.3cm}\mod 1) 
= \{\beta x \}.$$
The dynamical system 
$([0,1], T_{\beta})$ is called the
R\'enyi-Parry numeration system in base $\beta$,
the iterates of $T_{\beta}$ providing
the successive digits $x_i$ of $d_{\beta}(x)$
\cite{liyorke}.
Denoting $T_{\beta}^{0} := {\rm Id}, 
T_{\beta}^{1} := T_{\beta}, 
T_{\beta}^{i} := T_{\beta} (T_{\beta}^{i-1})$
for all $i \geq 1$, we have:
$$d_{\beta}(x) = (x_i)_{i \geq 1}
\qquad \mbox{{\rm if and only if}} \qquad
x_ i = \lfloor \beta T_{\beta}^{i-1}(x) \rfloor$$
and we write the $\beta$-expansion of $x$ as
\begin{equation}
\label{xexpansion}
x \, = \, \cdot x_1 x_2 x_3 \ldots 
\qquad \mbox{instead of}
\qquad x = \frac{x_1}{\beta} +\frac{x_2}{\beta^2} +
\frac{x_3}{\beta^3} +
\ldots .
\end{equation}
The digits are
$x_1 = \lfloor \beta x \rfloor$,
$x_2 = \lfloor \beta \{\beta x \} \rfloor$,
$x_3 = \lfloor \beta \{\beta \{\beta x \} \} \rfloor, \, 
\ldots$ \,, depend upon $\beta$.

The R\'enyi-Parry numeration dynamical 
system in base $\beta$
allows the coding, as a (positional) $\beta$-expansion, 
of any real number $x$. 
Indeed, if $x > 0$, there exists $k \in \zb$ such that
$\beta^k \leq x < \beta^{k+1}$. Hence
$1/\beta \leq x/\beta^{k+1} <1$; thus it is enough to 
deal with representations and $\beta$-expansions
of numbers in the interval $[1/\beta,1]$. 
In the case where $k \geq 1$, 
the $\beta$-expansion of $x$ is
$$x \, = \, x_1 x_2 \ldots x_k \cdot x_{k+1} x_{k+2} \ldots ,$$
with $x_1 = \lfloor \beta (x/\beta^{k+1}) \rfloor$,
$x_2 = \lfloor \beta \{\beta  (x/\beta^{k+1})\} \rfloor$,
$x_3 = \lfloor \beta \{\beta \{\beta  (x/\beta^{k+1})\} \} \rfloor$, etc.
If $x < 0$, by definition: $d_{\beta}(x) = - d_{\beta}(-x)$.
The part $ x_1 x_2 \ldots x_k $
is called the $\beta$-integer part of
the $\beta$-expansion of $x$, and
the terminant $ \cdot x_{k+1} x_{k+2} \ldots$ is called the $\beta$-fractional part of $d_{\beta}(x)$.

A $\beta$-integer is a real number $x$ such that
the $\beta$-integer part of $d_{\beta}(x)$ is equal to
$d_{\beta}(x)$ itself (all the digits $x_{k+j}$
being equal to 0 for $j \geq 1$): in this case, 
if $x > 0$ for instance, $x$ is the polynomial
$$x\, = \, \sum_{i=1}^{k} x_i \beta^{k-i} \, ,
\qquad 0 \leq x_i \leq \lceil \beta - 1 \rceil$$
and the set of $\beta$-integers is denoted by
$\zb_{\beta}$. For all $\beta > 1$
$\zb_{\beta} \subset \rb$ is discrete and   
$\zb_{\beta} = \zb$
if $\beta$ is an integer $\neq 0,1$.

The set $\mathcal{A}_{\beta}^{\nb}$ is endowed with
the lexicographical order
(not usual in number theory),
and the product topology. The one-sided
shift 
$\sigma : (x_i)_{i \geq 1} \mapsto (x_{i+1})_{i \geq 1}$
leaves invariant the subset $D_{\beta}$ 
of the $\beta$-expansions of real numbers
in $[0,1)$. The closure of $D_{\beta}$ in
$\mathcal{A}_{\beta}^{\nb}$
is called
the $\beta$-shift, and is denoted by $S_{\beta}$.
The $\beta$-shift is a subshift of
$\mathcal{A}_{\beta}^{\nb}$, for which
$$d_{\beta} \circ T_{\beta} = \sigma \circ d_{\beta}$$
holds on the interval $[0,1]$ (Lothaire \cite{lothaire},
Lemma 7.2.7). In other terms, $S_{\beta}$ is such that
\begin{equation}
\label{betashitcorrespondance}
x \in [0,1] \qquad 
\longleftrightarrow
\qquad
(x_i)_{i \geq 1} \in S_{\beta}
\end{equation}
is bijective. This one-to-one correspondence between
the totally ordered interval
$[0, 1]$ and the totally lexicographically ordered
$\beta$-shift $S_{\beta}$
is fundamental. Parry (\cite{parry} Theorem 3)
has shown 
that only one sequence of digits
entirely controls the $\beta$-shift
$S_{\beta}$, and that the ordering is preserved
when dealing with the greedy $\beta$-expansions. 
Let us precise how the usual 
inequality ``$<$" on the real line is
transformed into the inequality
``$<_{lex}$", meaning ``lexicographically 
smaller with all its shifts". 

\noindent
{\it The greatest element of $S_{\beta}$}:
it comes from $x=1$ and is given either by
the R\'enyi $\beta$-expansion of $1$, or 
by a slight modification of it 
in case of finiteness. Let us precise it.
The greedy $\beta$-expansion of $1$ is by 
definition
denoted by
\begin{equation}
\label{renyidef}
d_{\beta}(1) = 0.t_1 t_2 t_3 \ldots
\qquad
{\rm and ~uniquely ~corresponds~ to}
\qquad 1 = \sum_{i=1}^{+\infty} t_i \beta^{-i}\, ,
\end{equation}
where 
\begin{equation}
\label{digits__ti}
t_1 = \lfloor \beta \rfloor,
t_2 = \lfloor \beta \{\beta\}\rfloor = \lfloor \beta T_{\beta}(1)\rfloor,
t_3 = \lfloor \beta \{\beta \{\beta\}\}\rfloor = \lfloor \beta T_{\beta}^{2}(1)\rfloor,
\ldots
\end{equation} 
The sequence $(t_i)_{i \geq 1}$
is given by the orbit of one 
$(T_{\beta}^{j}(1))_{j \geq 0}$ by
\begin{equation}
\label{polyTbeta}
T_{\beta}^{0}(1)=1, ~T_{\beta}^{j}(1) = \beta^j - t_1 \beta^{j-1} - t_2 \beta^{j-2} - \ldots - t_j
\in \mathbb{Z}[\beta] \cap [0,1]
\end{equation}
for all $j \geq 1$.
The digits
$t_i$ belong to 
$\mathcal{A}_{\beta}$. 
We say that $d_{\beta}(1)$ is finite if it ends in infinitely many zeros.

\begin{definition}
\label{parrynumberdefinition}
If $d_{\beta}(1)$ is finite or ultimately periodic (i.e. eventually
periodic), then the real number $\beta > 1$ 
is said to be a {\it Parry number}.
In particular, a Parry number $\beta$ is 
said to be {\it simple} if $d_{\beta}(1)$ is finite.
\end{definition}

The greedy $\beta$-expansion of\, $1/\beta$ is
\begin{equation}
\label{renyidef_unsurbeta}
d_{\beta}(\frac{1}{\beta}) = 0. 0 \, t_1 t_2 t_3 \ldots
\qquad
{\rm and ~uniquely ~corresponds~ to}
\qquad \frac{1}{\beta} = \sum_{i=1}^{+\infty} t_i \beta^{-i-1}.
\end{equation}
From 
$(t_i)_{i \geq 1} \in \mathcal{A}_{\beta}^{\nb}$
is built  
$(c_i)_{i \geq 1} \in \mathcal{A}_{\beta}^{\nb}$,
defined by
$$
c_1 c_2 c_3 \ldots := \left\{
\begin{array}{ll}
t_1 t_2 t_3 \ldots & \quad \mbox{if ~$d_{\beta}(1) = 0.t_1 t_2 \ldots$~ is infinite,}\\
(t_1 t_2 \ldots t_{q-1} (t_q - 1))^{\omega}
& \quad \mbox{if ~$d_{\beta}(1)$~ is finite,
~$= 0. t_1 t_2 \ldots t_q$,}
\end{array}
\right.
$$
where $( \, )^{\omega}$ means that the word within $(\, )$ is indefinitely repeated.
The sequence $(c_i)_{i \geq 1}$ is the unique
element of $\mathcal{A}_{\beta}^{\nb}$
which allows to obtain all the admissible
$\beta$-expansions of all the elements of
$[0,1)$.

\begin{definition}
A sequence $(y_i)_{i \geq 0}$ of elements of
$\mathcal{A}_{\beta}$ (finite or not) is said
admissible if and only if
\begin{equation}
\label{conditionsParry}
\sigma^{j}(y_0, y_{1}, y_{2}, \ldots)=
(y_j, y_{j+1}, y_{j+2}, \ldots) <_{lex}
~(c_1, \, c_2, \, c_3, \,  \ldots) 
\quad \mbox{for all}~ j \geq 0,
\end{equation}
where $<_{lex}$ means {\it lexicographically smaller}.
\end{definition}

Any admissible representation $(x_i)_{i \geq 1}
\in \mathcal{A}_{\beta}^{\nb}$
corresponds, by \eqref{xexpansion}, 
to a real number $x \in [0,1)$
and conversely the greedy
$\beta$-expansion of $x$
is  $(x_i)_{i \geq 1}$ itself.
For an infinite admissible sequence 
$(y_i)_{i \geq 0}$ of elements of
$\mathcal{A}_{\beta}$ the 
(strict) lexicographical inequalities
\eqref{conditionsParry} constitute an infinite number of inequalities,
unusual in number theory,
which are called  
"Conditions of Parry" \cite{blanchard} 
\cite{frougny} \cite{frougny2}
\cite{lothaire} \cite{parry}.
More recent criteria of ``valid" sequences
are considered in 
Faller and Pfister \cite{fallerpfister}.

In number theory, inequalities are
often associated to collections of half-spaces in
euclidean or adelic Geometry of Numbers
(Minkowski's Theorem, etc).
The conditions of Parry are of totally different nature
since they refer to 
a reasonable control, order-preserving, 
of the gappiness (lacunarity)
of the coefficient vectors of the
power series which are
the generalized Fredholm determinants of the transfer 
operators of the $\beta$-transformations.

In the correspondence
$[0,1] \longleftrightarrow
S_{\beta}$, the element $x=1$ 
admits $d_{\beta}(1)$
as counterpart. The uniqueness 
of the $\beta$-expansion
$d_{\beta}(1)$
and its self-admissibility property
\eqref{self}  
characterize the base of numeration $\beta$ 
as follows.

\begin{proposition}
\label{betacharacterized}
Let $(a_0, a_1, a_2, \ldots)$  be a sequence of non-negative
integers where $a_0 \geq 1$ and $a_n \leq a_0$ for all $n \geq 0$.
The unique solution $\beta > 1$ of 
\begin{equation}
\label{equabase}
1 = \frac{a_0}{x} + \frac{a_1}{x^2} + \frac{a_2}{x^3} + \ldots
\end{equation}
is such that $d_{\beta}(1) = 0. a_0 a_1 a_2 \ldots$ if and only if
\begin{equation}
\label{self}
\sigma^{n}(a_0, a_{1}, a_{2}, \ldots)
=
(a_n, a_{n+1}, a_{n+2}, \ldots) <_{lex} (a_0, a_1, a_2, \ldots) 
\qquad \mbox{for all}~ n \geq 1.
\end{equation} 
\end{proposition}

\begin{proof}
Corollary 1 of Theorem 3 in Parry \cite{parry} (Corollary
7.2.10 in Frougny \cite{frougny2}).
\end{proof}

A sequence 
$(a_i)_{i \geq 0} \in \mathcal{A}_{\beta}^{\nb}$
satisfying  \eqref{self} is said self-admissible.
If $1 < \beta < 2$, then 
the condition
``$a_0 \geq 1$ and $a_n \leq a_0$ for all $n \geq 0$"
amounts to
``$a_0 = 1$"; in this case the
$\beta$-integer part of $\beta$ is equal to $a_0 = 1$
and its $\beta$-fractional part is
$a_1 \beta^{-1} + a_2 \beta^{-2} + a_3 \beta^{-3} + \ldots$.
The base of numeration
$\beta = 1$ would correspond to the sequence
$(1, 0, 0, 0, \ldots)$ in \eqref{equabase} but
this sequence has its first digit $1$
outside the alphabet
$\mathcal{A}_{1} = \{0\}$: it cannot be considered
as a $1$-expansion. 
Fortunately 
numeration in base one is not often used.
The base of numeration
$\beta = 2$ would correspond to
the constant sequence $(1, 1, 1, 1, \ldots)$
in \eqref{equabase}
but this sequence is not self-admissible.
When $\beta = 2$, $2$ being an integer,
$2$-ary representations
differ and $(2, 0, 0, 0, \ldots)$ is taken
instead of $(1, 1, 1, 1, \ldots)$
(Frougny and Sakarovitch \cite{frougnysakarovitch},
Lothaire \cite{lothaire}).

Infinitely many cases of 
lacunarity, between $(1,0, 0, 0, \ldots)$ and
$(1, 1, 1, 1, \ldots)$, may occur in the sequence
$(a_0 , a_1 , a_2 , \ldots)$ in \eqref{equabase}.
If $\beta \in (1, 2)$ is fixed,
with $d_{\beta}(1)=0. t_1 t_2 t_3 \ldots$
then any $x$, $1/\beta < x < 1$, admits a
$\beta$-expansion 
$d_{\beta}(x)$ which lies
lexicographically (Parry \cite{parry}, Lemma 1)
between those of the extremities:
\begin{equation}
\label{dbetaxencadrement}
d_{\beta}(\frac{1}{\beta}) =
0. 0 \, t_1 t_2 t_3 \ldots
~~
<_{lex}
~~
d_{\beta}(x)
~~
<_{lex}
~~ 
d_{\beta}(1) =
0. t_1 t_2 t_3 \ldots.
\end{equation}

Let $1 < \beta < 2$ be a real number, with
$d_{\beta}(1) = 0. t_1 t_2 t_3 \ldots$.
If $\beta$ is a simple Parry number, then
there exists $n \geq 2$, depending upon $\beta$, 
such that
$t_n \neq 0$ and
$t_j = 0 , j \geq n+1$. Parry 
\cite{parry} has shown that
the set of simple Parry numbers is dense in
the half-line
$(1,+\infty)$. If $\beta$ is a Parry number 
which is not simple, the sequence $(t_i)_{i \geq 1}$
is eventually periodic: there exists an integer
$m \geq 1$, the preperiod length, 
and an integer $p \geq 1$, the period length,
such that
$$d_{\beta}(1) = 0. t_1 t_2 \ldots t_m
(t_{m+1} t_{m+2} \ldots t_{m+p})^{\omega},
$$
$m$ and $p$ depending upon $\beta$, with at least
one nonzero 
digit $t_j$, with
$j \in \{m+1, m+2, \ldots, m+p\}$. 
The gaps of successive zeroes in $(t_i)_{i \geq 1}$
are those of the preperiod 
$(t_1 , t_2 , \ldots,  t_m)$ then those
of the period 
$(t_{m+1} , t_{m+2} , \ldots, t_{m+p})$, then
occur periodically up till infinity. The length
of such gaps of zeroes is at most
$\max\{m-2, p-1\}$. The asymptotic lacunarity
is controlled by the periodicity in this case.

If $1 < \beta < 2$ is an algebraic number 
which is not a Parry number, the sequences of 
gaps of zeroes in  $(t_i)_{i \geq 1}$
remain asymptotically moderate 
and controlled by the Mahler measure
${\rm M}(\beta)$ of $\beta$, as follows.

\begin{theorem}[Verger-Gaugry]
\label{lacunarityVG06}
Let $\beta > 1$ be an algebraic number such that 
$d_{\beta}(1)$ is infinite
and gappy in the sense that there exist
two infinite sequences
$\{m_n\}_{n \geq 1}$ 
and $\{s_n\}_{n \geq 0}$
such that
$$1 = s_0 \leq m_1 < s_1 \leq m_2 < s_2 \leq \ldots
\leq m_n < s_n \leq m_{n+1}
< s_{n+1} \leq \ldots$$
with $(s_n - m_n) \geq 2$, $t_{m_n} \neq 0$,
$t_{s_n} \neq 0$
and $t_i = 0$
if
$m_n < i < s_n$ for all $n \geq 1$. 
Then
\begin{equation}
\label{gappiness}
\limsup_{n \to +\infty} \frac{s_n}{m_n}
\leq \frac{\lo ({\rm M}(\beta))}{\lo \beta}
\end{equation}
\end{theorem}

\begin{proof}
\cite{vergergaugry}, Theorem 1.1.
\end{proof}

Theorem \ref{lacunarityVG06} 
also became a consequence of 
Theorem 2 in \cite{adamczewskibugeaud}.
In Theorem \ref{lacunarityVG06} the quotient
$s_n/m_n , n \geq 1$, is called the 
$n$-th Ostrowski quotient of the sequence
$(t_i)_{i \geq 1}$. For a given algebraic number 
$\beta> 1$, whether the upper bound
\eqref{gappiness} is exactly 
the limsup of the sequence of the 
Ostrowski quotient of $(t_i)_{i \geq 1}$
is unknown. For Salem numbers,
this equality always holds since
${\rm M}(\beta) = \beta$, and the upper bound
\eqref{gappiness} is 1.

\noindent
{\it Varying the base of numeration $\beta$
in the interval $(1,2)$:}
for all $\beta \in (1, 2)$, 
being an algebraic number or 
a transcendental number, the alphabet
$\mathcal{A}_{\beta}$
of the $\beta$-shift is
always the same: $\{0,1\}$.
All the digits of all $\beta$-expansions
$d_{\beta}(1)$ are zeroes or ones.
Parry (\cite{parry})
has proved that
the relation of order 
$1 < \alpha < \beta < 2$
is preserved
on the corresponding
greedy $\alpha$- and $\beta$- expansions
$d_{\alpha}(1)$ and
$d_{\beta}(1)$
as follows.

\begin{proposition}
\label{variationbasebeta}
Let $\alpha > 1$ and $\beta > 1$. 
If the R\'enyi $\alpha$-expansion of 1 is
$$d_{\alpha}(1) = 0. t'_1 t'_2 t'_3\ldots, 
\qquad ~i.e.
\quad
1 ~=~ \frac{t'_1}{\alpha} + \frac{t'_2}{\alpha^2} + \frac{t'_3}{\alpha^3} + \ldots$$
and the R\'enyi $\beta$-expansion of 1 is
$$d_{\beta}(1) = 0. t_1 t_2 t_3\ldots, 
\qquad ~i.e. 
\quad
1 ~=~ \frac{t_1}{\beta} + \frac{t_2}{\beta^2} + \frac{t_3}{\beta^3} + \ldots,$$
then $\alpha < \beta$ if and only if $(t'_1, t'_2, t'_3, \ldots) 
<_{lex} (t_1, t_2, t_3, \ldots)$. 
\end{proposition}

\begin{proof}
Lemma 3 in Parry \cite{parry}.
\end{proof}

For any integer $n \geq 1$ the sequence
of digits $1 0^{n-1} 1$, with $n-1$ times ``$0$"
between the two ones, is self-admissible. By
Proposition \ref{betacharacterized} 
it defines an unique
solution $\beta \in (1,2)$ of \eqref{equabase}.
Denote by $\theta_{n+1}^{-1}$ this solution.
From Proposition \ref{variationbasebeta}  
we deduce that the sequence 
$(\theta_{n}^{-1})_{n \geq 2}$ 
is (strictly) decreasing and tends to $1$ when
$n$ tends to infinity.

From \eqref{equabase} 
the real number $\theta_{2}^{-1}$
is the unique root $> 1$ of
the equation $1 = 1/x + 1/x^2$, that is
of $X^2 - X -1$.
Therefore it is the
Pisot number (golden mean)
 $= \frac{1+\sqrt{5}}{2} = 1.618\ldots$. 
Being interested in bases $\beta > 1$ 
close to 1 
tending to $1^+$, we will focus on
the interval $( 1, \frac{1+\sqrt{5}}{2}\, ]$ 
in the sequel.
This interval is
partitioned by
the decreasing sequence
$(\theta_{n}^{-1})_{n \geq 2}$ as
\begin{equation}
\label{jalonnement}
\bigl( 1, \frac{1+\sqrt{5}}{2}\, \bigr] ~=~
\bigcup_{n=2}^{\infty}
\left[ \, \theta_{n+1}^{-1} , \theta_{n}^{-1}
\, 
\right)
~~ \bigcup ~~\left\{
\theta_{2}^{-1}
\right\}.
\end{equation}

Theorem \ref{lacunarityVG06} gives 
an upper bound of
the asymptotic behaviour of the
Ostrowski quotients of the
$\beta$-expansion $(t_i)_{i \geq 1}$ of 1, 
due to the fact that
$\beta > 1$ is an algebraic number. 
The following theorem shows that the gappiness 
of $(t_i)_{i \geq 1}$ also admits some
uniform lower bound, for all
gaps of zeroes.
The condition of minimality on the 
length of the gaps of zeroes 
in $(t_i)_{i \geq 1}$
is only a function of the interval  
$\left[ \, \theta_{n+1}^{-1} , \theta_{n}^{-1}
\, 
\right)$
to which $\beta$ belongs,
when $\beta$ tends to 1.

\begin{theorem}
\label{zeronzeron}
Let $n \geq 2$. A real number 
$\beta \in ( 1, \frac{1+\sqrt{5}}{2}\, ]$ 
belongs to 
$[\theta_{n+1}^{-1} , \theta_{n}^{-1})$ if and only if the 
R\'enyi $\beta$-expansion of unity is of the form
\begin{equation}
\label{dbeta1nnn}
d_{\beta}(1) = 0.1 0^{n-1} 1 0^{n_1} 1 0^{n_2} 1 0^{n_3} \ldots,
\end{equation}
with $n_k \geq n-1$ for all $k \geq 1$.  
\end{theorem}

\begin{proof} 
Since $d_{\theta_{n+1}^{-1}}(1) = 0.1 0^{n-1} 1$ and
$d_{\theta_{n}^{-1}}(1) = 0.1 0^{n-2} 1$, 
Proposition \ref{variationbasebeta} implies that
the condition is sufficient. It is also necessary:
$d_{\beta}(1)$ begins as $0.1 0^{n-1} 1$
for all $\beta$ such that
$\theta_{n+1}^{-1} \leq \beta < \theta_{n}^{-1}$.
For such $\beta$s 
we write $d_{\beta}(1) = 0.1 0^{n-1} 1 u$~  
with digits in the alphabet
$\mathcal{A}_{\beta}
=\{0, 1\}$ common to all $\beta$s, that is
$$u= 1^{h_0} 0^{n_1} 1^{h_1} 0^{n_2} 1^{h_2} \ldots$$
and $h_0, n_1, h_1, n_2, h_2, \ldots$ integers $\geq 0$.
The self-admissibility lexicographic condition
\eqref{self} applied to the sequence
$(1, 0^{n-1}, 1^{1+h_0}, 0^{n_1}, 1^{h_1}, 0^{n_2}, 1^{h_3}, \ldots)$,
which characterizes uniquely the base of numeration $\beta$, 
readily implies
$h_0 = 0$ and 
$~h_k = 1$ and 
$n_k \geq n-1$ for all $k \geq 1$.
\end{proof}

\begin{remark}
The case $n_1 = +\infty$ in \eqref{dbeta1nnn}
corresponds to the simple Parry number
$\beta = \theta_{n+1}^{-1}$. The value
$+\infty$ is not excluded from the set
$(n_k)_{k \geq 1}$ in the following sense:
if there exists $j \geq 2$ such that
$n-1 \leq n_k < +\infty, ~k < j$, with
$n_j = +\infty$, then $\beta$ is a simple Parry number in
$[\theta_{n+1}^{-1} , \theta_{n}^{-1})$
characterized by
$$d_{\beta}(1) = 0 . 1 0^{n-1} 1 0^{n_1} 1 0^{n_2} 1 \ldots
1 0^{n_{j} - 1} 1,$$
that is is a root of an integer polynomial
($\S$ \ref{S4.3}).
All the simple Parry numbers lying
in the interval 
$[\theta_{n+1}^{-1} , \theta_{n}^{-1})$
are obtained in this way. 
On the contrary,
the transcendental numbers $\beta$ in
$[\theta_{n+1}^{-1} , \theta_{n}^{-1})$
have all R\'enyi $\beta$-expansions
$d_{\beta}(1) = 0 . t_1 t_2 t_3 \ldots$ of 1
such that the sequence of exponents 
$(n_k)_{k \geq 1}$ of the successive zeroes, 
corresponding to the sequence
of the lengths of the gaps of zeroes, 
never takes the value $+\infty$.
\end{remark}

\begin{definition}
Let $\beta \in ( 1, \frac{1+\sqrt{5}}{2}\, ]$ be a real number.
The integer $n \geq 3$ such that
$\theta_{n}^{-1} \leq \beta < \theta_{n-1}^{-1}$
is called the dynamical degree of $\beta$, and
is denoted by ${\rm dyg}(\beta)$.
By convention we put:
${\rm dyg}(\frac{1+\sqrt{5}}{2}) = 2$. 
\end{definition}

The function $n={\rm dyg}(\beta)$ is locally
constant on the interval 
$(1, \frac{1+\sqrt{5}}{2}]$, 
is decreasing, 
takes all values in $\mathbb{N} \setminus \{0,1\}$,
and satisfies:
$\lim_{\beta > 1, \beta \to 1} \dyg(\beta) = +\infty$.
The relations between the restriction of the 
dynamical degree ${\rm dyg}(\beta)$ to
$\overline{\qb} \cap (1, \frac{1+\sqrt{5}}{2}]$ and
the (usual) degree $\deg(\beta)$ will be 
investigated later (Theorem
\ref{dygdeginequality}; $\S$ \ref{S5.1}, $\S$ \ref{S6.5}).
Let us observe that
the equality
deg$(\beta) =$ dyg$(\beta) = 2$ 
holds if $\beta = \frac{1+\sqrt{5}}{2}$, but 
the equality case is not the case in general.

\begin{definition}
\label{selfadmissiblepowerseries}
A power series $\sum_{j=0}^{+\infty} a_j z^j$,
with $a_j \in \{0, 1\}$ for all $j \geq 0$,
$z$ the complex variable,
is said to be self-admissible
if its
coefficient vector $(a_i)_{i \geq 0}$ 
is self-admissible.
\end{definition}

\subsection{Parry Upper functions, Dynamical zeta functions and transfer operators}
\label{S4.2}

\begin{definition}
\label{parryupperfunction}
Let $\beta \in (1, (1+\sqrt{5})/2]$ be a real
number, and $d_{\beta}(1) = 0. t_1 t_2 t_3 \ldots$ its
R\'enyi $\beta$-expansion of 1.
The power series $f_{\beta}(z) :=
-1 + \sum_{i \geq 1} t_i z^i$
of the complex variable $z$
is called the Parry Upper function at $\beta$.
\end{definition}

By definition of $d_{\beta}(1) = 0. t_1 t_2 t_3 \ldots$
the Parry Upper function $f_{\beta}(z)$
has a zero at $z=1/\beta$. 
It is such that
$f_{\beta}(z) + 1$ has coefficients in the alphabet
$\mathcal{A}_{\beta} =\{0,1\}$ and is self-admissible
by \eqref{self}; if $\beta$ is an
algebraic number in particular,
the lacunarity of $f_{\beta}(z)$ is moderate by
Theorem \ref{lacunarityVG06}, 
because of the existence of an 
upper bound of the Ostrowski quotients
of $(t_i)_{i \geq 0}$ 
in terms of the Mahler measure
${\rm M}(\beta)$,
and by Theorem \ref{zeronzeron}, 
by the dynamical degree
$\dyg(\beta)$.
Since $d_{\beta}(1) = 0. t_1 t_2 t_3 \ldots$
entirely determines the language 
in base $\beta$, the  
function $f_{\beta}(z)$
is an analytic function whose domain of definition and
zeroes are analytical
characteristics of the language in base $\beta$.
Therefore some continuity properties
of some subcollections of
zeroes are expected, when $\beta$
runs over a small neighbourhood of $1$
($\S$ \ref{S4.4}). 

The definition of 
$f_{\beta}(z)$ seems simple
since
the vector coefficient of
$f_{\beta}(z) + 1$
is only a sequence of integers 
deduced from the
orbit of 1 under the iterates of the
$\beta$-transformation
$T_{\beta}$, by \eqref{digits__ti} 
and \eqref{polyTbeta}
\cite{flattolagarias}, I, 
\cite{flattolagariaspoonen};
nevertheless it is deeply related 
to the Artin-Mazur dynamical zeta function
$\zeta_{\beta}(z)$
(given by \eqref{dynamicalfunction})
of the R\'enyi-Parry dynamical system
$([0, 1], T_{\beta})$, to
the Perron-Frobenius operator $P_{T_{\beta}}$
associated with $T_{\beta}$ and to the
generalized ``Fredholm determinant" \eqref{fredholmdeterminantgeneralized} of this operator,
or more precisely of the transfer operator
of $T_{\beta}$.
In the kneading theory of Milnor and Thurston
\cite{milnorthurston} it is a kneading determinant.
Let us recall these links, knowing
that the theory of Fredholm 
(Grothendieck
\cite{grothendieck} \cite{grothendieck2}, 
Riesz and Nagy \cite{riesznagy} Chap. IV) is done for
compact operators while the Perron-Frobenius
operators associated with the $\beta$-transformations
$T_{\beta}$
are noncompact by nature (Mori \cite{mori}
\cite{mori2}, Takahashi \cite{takahashi}
\cite{takahashi2} \cite{takahashi5}).

Let $(X,\Sigma, \mu)$ be a 
$\sigma$-finite measure space and let 
$T:X \to X$ be a nonsingular transformation, i.e. 
$T$ is measurable and satisfies: for all $A \in \Sigma$,
$\mu(A) = 0 \Longrightarrow 
\mu(T^{-1}(A)) = 0$.
In ergodic theory, by the Radon-Nikodym theorem, 
the operator
$P_T : L^{1}(X,\Sigma, \mu)
\to L^{1}(X,\Sigma, \mu)$ defined by
\begin{equation}
\label{perronfrobeniusdefinition}
\int_A P_{T} f d \mu ~=~
\int_{T^{-1}(A)} f d\mu 
\end{equation}
is called the Perron-Frobenius operator associated with
$T$. 
Let $\beta \in (1, \theta_{2}^{-1})$,
$X = [0, 1]$, $\Sigma$ the Borel 
$\sigma$-algebra and 
$T_{\beta}$ the $\beta$-transformation.
The $T_{\beta}$-invariant probability measure
$\mu = \mu_{\beta}$ of the 
$\beta$-shift, on $\Sigma$, 
is unique (R\'enyi \cite{renyi}), ergodic
(Parry \cite{parry}), 
maximal (Hofbauer \cite{hofbauer})
and absolutely 
continuous with respect to the Lebesgue measure
$dt$, with Radon-Nikodym derivative 
(Lasota and Yorke \cite{lasotayorke},
Parry \cite{parry}, 
Takahashi \cite{takahashi}) :
$$h_{\beta} 
~=~
C \,\sum_{n: x < T_{\beta}^{n}(1)} 
\frac{1}{\beta^{n+1}} ,\qquad 
\mbox{ so that}
\qquad d\mu_{\beta} = h_{\beta} dt,
$$ 
for some constant $C > 0$.
These results were independently discovered by
A.O. Gelfond
\cite{fallerpfister}.
We denote 
by $P_{T_{\beta}}$ the Perron-Frobenius operator
associated with $T_{\beta}$.

The $\beta$-transformation 
$T_{\beta}$ is a 
piecewise monotone map of the interval
$[0, 1]$ with weight function
$g = 1$. Over the last 45 years, the rigourous
theory of dynamical zeta functions, in particular
weighted dynamical zeta functions, 
introduced by Ruelle
\cite{ruelle} \cite{ruelle2} \cite{ruelle3} \cite{ruelle4}
\cite{ruelle5} \cite{ruelle6}
\cite{ruelle7} \cite{ruelle8}
\cite{ruelle9}
by analogy with the
thermodynamic formalism
of statistical mechanism \cite{ruelle4},
settled as a basic analytic theory
for the study of many dynamical systems.
The relations between the
Fredholm determinants, or the 
generalized Fredholm determinants,
and these weighted dynamical zeta functions,
were extensively studied
(Baladi \cite{baladi} \cite{baladi2}
\cite{baladi3},
Baladi and Keller \cite{baladikeller},
Hofbauer \cite{hofbauer},
Hofbauer and Keller \cite{hofbauerkeller},
Milnor and Thurston \cite{milnorthurston},
Parry and Pollicott \cite{parrypollicott},
Pollicott \cite{pollicott} \cite{pollicott2},
Preston \cite{preston},
Takahashi \cite{takahashi3}
\cite{takahashi4}). In the context of
noncompact operators, the objective consists in 
giving a sense to
\begin{equation}
\label{fredholmdeterminantgeneralized}
` {\rm det}` \, (Id - z \, L_{t})
=
\exp\Bigl(
- \sum_{n \geq 1} \frac{`tr` \,L_{t}^{n}}{n}\, z^n
\Bigr),
 \end{equation}
where $L_{t}$ is 
a dynamically defined weighted transfer operator
acting on a suitable Banach space
\cite{baladi2}  \cite{hofbauerkeller2}. 

Let $1 < \beta < (1 + \sqrt{5})/2$ be a real number
and $0 = a_0 < a_1 = \frac{1}{\beta} < a_2 =1$
be the finite partition of $[0, 1]$.
The map $T_{\beta}$ is strictly monotone and continuous
on $[a_0 , a_1)$ and $[a_1 , 1]$.
For each function
$f: [0, 1] \to \cb$, let
$${\rm var}(f) :=
\sup\Bigl\{ \,\sum_{i=1}^{n}\, |f(e_i) - f(e_{i-1})|
\mid
n \geq 1 , 0 \leq e_1 \leq e_2 \leq \ldots
\leq e_n \leq 1 \Bigr\},$$
$$\|f\|_{BV} :=
{\rm var}(f) + \sup(|f|),$$
and denote by
BV the Banach space of functions 
with bounded variation
\cite{keller}
\cite{keller2}:
$$BV := \{f: [0, 1] \to \cb \mid
\|f\|_{BV} < \infty \}.$$
For $g \in BV$, one 
can define the following transfer operator
$$L_{t \beta, g}: BV \to BV,
\quad 
L_{t \beta, g} f(x) := 
\sum_{y , T_{\beta}(y) = x} g(y) f(y) .$$
We will only consider the case 
$g \equiv 1$ in the sequel and
put $L_{t \beta} :=
L_{t \beta,1}$.

\begin{theorem}
\label{dynamicalzetatransferoperator}
Let $\beta \in (1,\theta_{2}^{-1})$. Then,
\begin{itemize}
\item[(i)] the Artin-Mazur
dynamical zeta function
$\zeta_{\beta}(z)$ defined 
by \eqref{dynamicalfunction}
is nonzero and meromorphic
in $\{|z| < 1\}$, and such that
$1/\zeta_{\beta}(z)$ is holomorphic
in $\{|z| < 1\}$,
\item[(ii)] suppose 
$|z| < 1$. Then $z$ is a pole of
$\zeta_{\beta}(z)$ of multiplicity $k$
if and only if
$z^{-1}$ is an eigenvalue of $L_{t \beta}$ of multiplicity 
$k$.
\end{itemize}
\end{theorem}

\begin{proof}
Theorem 2 in Baladi and Keller \cite{baladikeller}, 
assuming that
the set of intervals 
$([0, a_1),[a_1 , 1])$
forming the partition of
$[0,1]$ is generating;
In \cite{ruelle5} \cite{ruelle8} Ruelle shows that this assumption is not necessary, showing how to remove this obstruction.
\end{proof}

The $\beta$-transformation $T_{\beta}$
is one of the simplest transformations among
piecewise monotone intervals maps
(Baladi and Ruelle
\cite{baladiruelle},
Milnor and Thurston \cite{milnorthurston},
Pollicott \cite{pollicott}). 
Theorem \ref{dynamicalzetatransferoperator} 
was conjectured by Hofbauer and Keller
\cite{hofbauerkeller} 
for piecewise monotone maps, for the case
where the function $g$ is piecewise constant.
(cf also Mori \cite{mori} \cite{mori2}).
The case $g=1$ in the transfer operators
was studied by Milnor and Thurston \cite{milnorthurston},
Hofbauer \cite{hofbauer}, 
Preston \cite{preston}. The structure of the
sets of periodic points
of $T_{\beta}$ are studied as Markov 
diagrams by Hofbauer
\cite{hofbauer2}. 

The eigenvalues of the
transfer operators, not only of the transfer operators
$L_{t \beta}$ with $1 < \beta < \theta_{2}^{-1}$, 
are important quantities, 
for instance for resonances 
in dynamical systems, decays of correlations
\cite{baladi4} \cite{eckmann}.
Theorem \ref{dynamicalzetatransferoperator}
was proved and stated in Baladi and Keller
\cite{baladikeller} under more general
assumptions.

The fact 
(Theorem \ref{dynamicalzetatransferoperator} (ii)) 
that the poles of $\zeta_{\beta}(z)$,
lying in the open unit disc,
are of the same multiplicity of the inverses of the 
eigenvalues 
of the transfer
operator $L_{t \beta}$ is a 
extension of Theorem 2, Theorem 3 and 
Theorem 4 in
Grothendieck \cite{grothendieck2} in the 
context of the Fredholm theory 
with compact operators.
When $\beta > 1$ is an algebraic integer and
tends to $1^+$,
we will prove ($\S$ \ref{S5})
that the multiplicity $k$ is equal to
$1$ for the first pole $1/\beta$ 
of $\zeta_{\beta}(z)$ and for
a subcollection of Galois conjugates
of $1/\beta$ in an angular sector.

The relations between 
the poles of the dynamical
zeta function $\zeta_{\beta}(z)$,
the zeroes of the Parry Upper function
$f_{\beta}(z)$ and the eigenvalues of the
transfer operator $L_{t \beta}$
come from Theorem \ref{dynamicalzetatransferoperator}
and from the following theorem.

\begin{theorem}
\label{parryupperdynamicalzeta}
Let $\beta > 1$ be a real number.
Then the Parry Upper function
$f_{\beta}(z)$ satisfies
\begin{equation}
\label{parryupperdynamicalzeta_i}
(i)\qquad f_{\beta}(z) = - \frac{1}{\zeta_{\beta}(z)} 
\qquad \mbox{if}~ \beta ~\mbox{is not a simple Parry number},
\end{equation}
and
\begin{equation}
\label{parryupperdynamicalzeta_ii}
(ii)\quad
f_{\beta}(z) = - \frac{1 - z^N}{\zeta_{\beta}(z)}
\qquad
\mbox{if $\beta$ is a simple Parry number}
\end{equation}
where $N$, 
which depends upon $\beta$, is the minimal positive integer such that $T_{\beta}^{N}(1) = 0$.
It is holomorphic in the open unit disk
$\{|z| < 1\}$.
It has no zero in
$|z| \leq 1/\beta$ except $z=1/\beta$
which is a simple zero.
The Taylor series of 
$f_{\beta}(z)$ at $z=1/\beta$
is $f_{\beta}(z)
=
c_{\beta, 1} \bigl(z - \frac{1}{\beta}\bigr)
+ c_{\beta, 2} \bigl(z - \frac{1}{\beta}\bigr)^2+ \ldots$
with
\begin{equation}
\label{coefficientsfbetaatbetainverse}
c_{\beta, m} = 
\sum_{n=m}^{\infty} \frac{n !}{(n-m) ! ~m !}
\lfloor \beta T_{\beta}^{n-1}(1) \rfloor 
\bigl(\frac{1}{\beta}\bigr)^{n-m}
 ~> 0,
\qquad \mbox{for all}~ m \geq 1.
\end{equation}
\end{theorem}

\begin{proof}
Theorem 2.3 and Appendix A
in Flatto, Lagarias and Poonen 
\cite{flattolagariaspoonen}; 
Theorem 1.2 in
Flatto and Lagarias 
\cite{flattolagarias}, I;
Theorem 3.2
in Lagarias \cite{lagarias}.
From Takahashi \cite{takahashi},
Ito and Takahashi
\cite{itotakahashi}, these authors
deduce 
\begin{equation}
\label{zetafunctionfraction}
\zeta_{\beta}(z) = \frac{1 - z^N}{(1 - \beta z)
\Bigl(
\sum_{n=0}^{\infty}T_{\beta}^{n}(1) \, z^n
\Bigr)}
\end{equation}
where ``$z^N$" has to be replaced by ``$0$"
if $\beta$ is not a simple Parry number.
Since $\beta T_{\beta}^{n}(1) =
\lfloor \beta T_{\beta}^{n}(1)
\rfloor +
\{\beta T_{\beta}^{n}(1)\} = 
t_{n+1} + T_{\beta}^{n+1}(1)$
by \eqref{digits__ti}, for $n \geq 1$,
expanding the power series of the denominator
\eqref{zetafunctionfraction} readily gives:
\begin{equation}
\label{fbetazetabeta}
-1 +t_1 z + t_2 z^2 +\ldots
= f_{\beta}(z) 
= 
-(1 - \beta z)
\Bigl(
\sum_{n=0}^{\infty}T_{\beta}^{n}(1) \, z^n
\Bigr).
\end{equation}
The zeroes of smallest modulus are
characterized in Lemma 5.2, Lemma 5.3 and Lemma 5.4
in \cite{flattolagariaspoonen}. The coefficients
$c_{\beta,m}$ readily come from the derivatives of
$f_{\beta}(z)$.
\end{proof}

The set of Parry numbers $\mathbb{P}_P$ 
is not characterized
($\S$ \ref{S4.3}). 
Case (i) 
means that 
$\zeta_{\beta}^{-1}(z)$, holomorphic in $|z| < 1$, 
is a power series
with coefficients in a finite set of complex numbers,
for $\beta$ being a nonParry number.
As an appreciable advantage,
in the simple Parry number case,
case (ii) means that there exists a product of cyclotomic
polynomials, namely $-(1 - z^N) $, 
which is such that
its product with $\zeta_{\beta}^{-1}(z)$, does the same:
$f_{\beta}(z)$ is also
a power series with coefficients in a
finite set.
Then, since the roots of cyclotomic polynomials
are of modulus 1, the zeroes of 
$f_{\beta}(z)$ within $|z|< 1$ are exactly
the poles of $\zeta_{\beta}(z)$ in this domain,
whatever the R\'enyi-Parry dynamics of $\beta > 1$
is.

The Parry Upper function $f_{\beta}(z)$ 
admits an Euler product and a logarithmic derivative
in terms of a power series with coefficients in $\zb$:
from \eqref{dynamicalfunction} 
and \cite{flattolagarias}, I,
$$\zeta_{\beta}(z) = \prod_{x ~\mbox{{\small periodic}}} 
\Bigl(1 - z^{\mbox{\,{\small order}}\, (x)}  \Bigr)^{-1},
\qquad 
\frac{\zeta'_{\beta}(z)}{\zeta_{\beta}(z)}
= \sum_{n=0}^{\infty} \mathcal{P}_{n+1} z^n
$$
and from \eqref{parryupperdynamicalzeta_i}
\eqref{parryupperdynamicalzeta_ii}, with
$N$ defined inthere,
we deduce, with ``$z^N$" equal to $0$ if $\beta$ is 
not a simple Parry number,
\begin{equation}
\label{eulerproduct}
f_{\beta}(z) = 
- (1 - z^N) \times
\prod_{x ~\mbox{{\small periodic}}} 
\Bigl(1 - z^{\mbox{{\small order}}~ (x)}  \Bigr) ,
\end{equation}
\begin{equation}
\label{logarithmicderivative}
\qquad \frac{f'_{\beta}(z)}{f_{\beta}(z)}
= - N z^{N-1} (1 - z ^N)^{-1} -
\sum_{n=0}^{\infty} \mathcal{P}_{n+1} z^n.
\end{equation} 
The growth of the sequence $(\mathcal{P}_n)_{n \geq 0}$
was investigated in Flatto and Lagarias 
\cite{flattolagarias}
in terms of the lap-function of the $\beta$-shift
as a function of the
two zeroes of smallest modulus of $f_{\beta}(z)$.

\begin{definition}
\label{parrypolynomial}
If $\beta$ is a simple Parry number, 
with $d_{\beta}(1) = 0 . t_1 t_2 \ldots t_m$,
$t_m \neq 0$,
the polynomial
\begin{equation}
\label{parrypolynomesimple}
P_{\beta, P}(X) :=
X^m - t_1 X^{m-1} - t_2 X^{m-2} - \ldots t_m
\end{equation}
is called the Parry polynomial of $\beta$.
If $\beta$ is a Parry number which is not simple,
with
$d_{\beta}(1) = 0 . t_1 t_2 \ldots t_m
(t_{m+1} t_{m+2} \ldots t_{m+p+1})^{\omega}$
and not purely periodic 
($m$ is $\neq 0$), then
$$P_{\beta, P}(X) :=
X^{m+p+1} - t_1 X^{m+p} - t_2 X^{m+p-1}
- \ldots - t_{m+p} X - t_{m+p+1}
\qquad \qquad\mbox{ }$$
\begin{equation}
\label{parrypolynomeperiodicity}
\mbox{ } \qquad \qquad \qquad \qquad
-X^m + t_1 X^{m-1} + t_2 X^{m-2}
+ \ldots + t_{m-1} X + t_{m}
\end{equation}
is the Parry polynomial of $\beta$.
If $\beta$ is a nonsimple Parry number
such that
$d_{\beta}(1) = 0 . 
(t_{1} t_{2} \ldots t_{p+1})^{\omega}$
is purely periodic (i.e. $m=0$),
then
\begin{equation}
\label{parrypolynomepureperiodicity}
P_{\beta, P}(X) :=
X^{p+1} - t_1 X^{p} - t_2 X^{p-1}
- \ldots - t_{p} X - (1 + t_{p+1})
\end{equation}
is the Parry polynomial of $\beta$.
By definition the degree $d_P$ of
$P_{\beta,P}(X)$
is 
respectively $m, m + p -1, p+1$
in the three cases.
\end{definition}

If $\beta$ is a Parry number,
the Parry polynomial $P_{\beta, P}(X)$,
belonging to the ideal 
$P_{\beta}(X) \zb[X]$,
admits $\beta$ as simple root
and is often not irreducible. 
Its factorization properties
were studied in the context of 
the theory of Pinner and Vaaler
\cite{pinnervaaler}
in \cite{vergergaugry2}; see also 
\cite{bertrandmathis2}.
The polynomial 
$\frac{P_{\beta, P}(X)}{P_{\beta}(X)}$ 
was called {\em complementary factor} 
by Boyd. 
For the two cases
\eqref{parrypolynomesimple}
and \eqref{parrypolynomepureperiodicity}
the constant term is $\neq 0$; hence
$\deg(P^{*}_{\beta,P}) =
\deg(P_{\beta,P})$.
In the case of \eqref{parrypolynomeperiodicity}
denote by
$q_{\beta} := 0$ if 
$t_{m} \neq t_{m+p+1}$
and, if $t_{m} = t_{m+p+1}$,
$q_{\beta} := 1 +
\max\{r \in \{0, 1, m-1\} 
\mid t_{m-l} = t_{m+p+1-l} \mbox{~for all}~
0 \leq l \leq r \}$.
Then
$p+1 \leq \deg(P_{\beta,P}) - q_{\beta}
= \deg(P^{*}_{\beta,P}) \leq \deg(P_{\beta,P})$.

Applying the Carlson-Polya dichotomy 
(Bell and Chen \cite{bellchen},
Bell, Miles and Ward
\cite{bellmilesward},
Carlson \cite{carlson} \cite{carlson2}, 
Dienes \cite{dienes},
P\'olya \cite{polya}, Robinson \cite{robinson},
Szeg\H{o} \cite{szego}) to
the power series $f_{\beta}(z)$, for which the coefficients
belong to the finite set
$\mathcal {A}_{\beta} \cup \{-1\}$, gives the following
equivalence.

\begin{theorem}
\label{carlsonpolya}
The real number $\beta> 1$
is a Parry number if and only if
the Parry Upper function
$f_{\beta}(z)$ is a rational function, 
equivalently if and only
if $\zeta_{\beta}(z)$ is a rational function.

The set of Parry numbers, resp.
of nonParry numbers,
in $(1, \infty)$, is not empty.
If $\beta$ is not a Parry number, 
then 
$|z| = 1$ is the natural boundary
of  
$f_{\beta}(z)$. 
If $\beta$ is a Parry number,
with R\'enyi $\beta$-expansion of $1$ given by
$$d_{\beta}(1) = 0 . t_1 t_2 \ldots t_m (t_{m+1} t_{m+2} \ldots
t_{m+p+1})^{\omega}, \qquad t_1 = \lfloor \beta \rfloor,
~t_i \in \mathcal{A}_{\beta}, i \geq 2,$$
the preperiod length being
$m \geq 0$ and the period length $p+1 \geq 1$,
$f_{\beta}(z)$ admits an analytic 
meromorphic extension over $\cb$, of
the following form:
$$
f_{\beta}(z) =
- P_{\beta, P}^{*}(z) \qquad
\mbox{if $\beta$ is simple,}
$$
$$
f_{\beta}(z) =
\frac{- P_{\beta, P}^{*}(z)}{1 - z^{p+1}}
\qquad 
\mbox{if $\beta$ is nonsimple,}
$$
where the Parry polynomial is given by
\eqref{parrypolynomesimple},
\eqref{parrypolynomeperiodicity}
or
\eqref{parrypolynomepureperiodicity}.

If $\beta \in (1, 2)$ is a Parry number,
the (na\"ive) height 
${\rm H}(P_{\beta, P})$
of $P_{\beta, P}$
is equal to 1 except when: 
$\beta$ is nonsimple 
and that 
$t_{p+1} = \lfloor \beta T_{\beta}^{p}(1) \rfloor = 1$,
in which case the Parry polynomial of $\beta$
has na\"ive height
${\rm H}(P_{\beta, P}) = 2$.
\end{theorem}

\begin{proof}
Verger-Gaugry \cite{vergergaugry6}.
The set of nonParry numbers $\beta$
in $(1, \infty)$ is not empty
as a consequence of Fekete-Szeg\H{o}'s Theorem 
\cite{feketeszego} 
since the radius of convergence
of $f_{\beta}(z)$ is equal to 1 in any case
and that its domain of definition always contains
the open unit disk which has a transfinite
diameter equal to 1.
The set of Parry numbers $\beta$
in $(1, \infty)$ is also nonempty.
Indeed Pisot numbers, of degree $\geq 2$,
are Parry numbers
(Schmidt \cite{schmidt}, 
Bertrand-Mathis \cite{bertrandmathis}).
Therefore the dichotomy between
Parry and nonParry numbers in
$(1,\infty)$ has a sense.
\end{proof}

\begin{definition}
\label{betaconjugate}
Let $\beta > 1$ be a Parry number.
If the Parry polynomial $P_{\beta,P}(z)$
of $\beta$ is not irreducible,
the roots of 
$P_{\beta,P}(z)$ which are not Galois conjugates
of $\beta$ are called the beta-conjugates of $\beta$. 
\end{definition}
Beta-conjugates were studied in
\cite{vergergaugry4} 
\cite{vergergaugry5}
in terms of Puiseux theory and
in association with germs of curves.

\begin{remark}
\label{hauteurborneeequidistribution}
For any Parry number
$\beta \in (1, 2)$ the 
(na\"ive) height 
${\rm H}(P_{\beta,P})$
of the Parry polynomial $P_{\beta,P}(z)$
is uniformly bounded by 2, by 
\eqref{parrypolynomesimple} 
\eqref{parrypolynomeperiodicity} 
\eqref{parrypolynomepureperiodicity}.
Now let us consider a convergent
sequence of Parry numbers $(\beta_i)$,
$\beta_i < 2$, tending to $1$.
This remarkable property
has for consequence that
the phenomenon of limit equidistribution 
towards the Haar measure on the unit circle occurs.
Here, ``limit equidistribution" means
that the set constituted by the conjugates
and the beta-conjugates of the $\beta_i$s, 
both denoted by $\beta_{i}^{(j)}$,
tend to the unit circle 
for the Hausdorff topology
(Theorem 3.25 in \cite{vergergaugry3}).
The basic argument is 
the tightness property (Billingsley
\cite{billingsley}) of the convergent sequence 
$$\frac{1}{[\kb : \qb]} \sum \delta_{\beta_{i}^{(j)}} \qquad i \to \infty ,$$
for the weak topology, which is then satisfied, 
following Pritsker \cite{pritsker} 
\cite{pritsker3}
(cf also $\S$ \ref{S8}). 
Removing the 
subcollections of beta-conjugates, 
and leaving only the conjugates of the Parry numbers
$\beta_i$, may
suppress the phenomenon of
limit equidistribution.
\end{remark}

\subsubsection{Algebraicity of the base and zeroes of Parry Upper functions}
When $\beta$ is a Parry number,
the Parry Upper function $f_{\beta}(z)$ has a 
finite number of zeroes by 
Theorem \ref{carlsonpolya}. If $\beta$ is not a
Parry number the presence of zeroes of
$f_{\beta}(z)$ in $1/\beta < |z| < 1$ is more difficult
to describe. The most important region
to be investigated for the presence, the number (eventually infinite) and the geometry
of zeroes of $f_{\beta}(z)$
is the annular region
$\{z \mid \exp(-\mathcal{H}) = 1/\beta < |z| < 1\}$, 
where $\mathcal{H} = \lo \beta$ is the 
topological entropy of the $\beta$-shift
(Proposition 5.1 in \cite{parrypollicott}), 
in particular when $\beta$ tends to $1^+$.
This extended research of zeroes 
is general and
concerns 
(i) the meromorphic extension
of the dynamical zeta function of a dynamical system
outside the disk of convergence whose radius is
$\exp(-\mathcal{H})$ with $\mathcal{H}$
the topological entropy, the pressure, etc, 
of the dynamical system,
and their poles in this annular region
(Haydn \cite{haydn},
Hilgert and Rilke \cite{hilgertrilke},
Parry and Pollicott \cite{parrypollicott},
Pollicott \cite{pollicott}, Ruelle \cite{ruelle7}),
(ii) the structure theorems of orthogonal 
decomposition of the
(Perron-Frobenius operators) 
transfer operators, with the geometry of their 
isolated eigenvalues  (e.g.
Theorem 1 in Baladi and Keller \cite{baladikeller}).

How the zeroes of $f_{\beta}(z)$ lying in the 
annular region
$1/\beta < |z| < 1$ are
correlated to the (Galois-) conjugates of 
an algebraic integer 
$1 < \beta \leq \theta_{2}^{-1}$
will be considered in $\S 5$ and $\S 6$.

\begin{theorem}
\label{zeroesinfinitelymanyrealneg}
There exists a dense subset of $\beta$'s
in $(1, \infty)$ for which the Parry Upper function
$f_{\beta}(z)$, defined in $|z| < 1$, 
has infinitely many 
zeroes lying in
the real interval $(-1, -1+\frac{1}{\beta})$, and 
having $-1$ as unique limit point.
\end{theorem}

\begin{proof}
The function $f_{\beta}(z)$ is holomorphic and nonzero
in the open unit disk. As meromorphic function,
its domain of definition
is $\cb$ or $|z| < 1$, depending upon whether
$\beta$ is Parry or not. 
All the limit points of
any subfamilies of zeroes lie necessarily
on the boundary
$|z| = 1$. The result follows from Theorem 7.1 
in \cite{flattolagariaspoonen}
and Theorem \ref{parryupperdynamicalzeta}.
The last claim comes from
the behaviour of the Parry Upper functions
for $\beta$ running over
the neighbourhoods of the Perron numbers
$\theta_{n}^{-1}, n \geq 3$ (Theorem \ref{parryupperfunctionPerronNeighbourhood}).
\end{proof}

The arguments
invoked in the proof
of Theorem 7.1 
in \cite{flattolagariaspoonen}, 
used for proving
Theorem \ref{zeroesinfinitelymanyrealneg},
are not compatible with
the moderate gappinesses due 
algebraic numbers $\beta$:
indeed
these authors construct gaps of lengths 
going extremely fast to infinity, in contradiction with
Theorem \ref{lacunarityVG06} 
if the algebraicity of
$\beta$ is assumed.
Therefore
Theorem \ref{zeroesinfinitelymanyrealneg} 
seems to be essentially 
addressed to transcendental numbers.

Whether Theorem 
\ref{zeroesinfinitelymanyrealneg} 
is also addressed to
an infinite subcollection 
of algebraic numbers is partially kown:
indeed, 
there exists a ``spike" 
formed by the real negative beta-conjugates of 
subcollections of simple Parry numbers, 
on $[-1, 0)$, in
Solomyak's fractal $\mathcal{G}$ 
(Theorem \ref{solomyakfractalGBordG};
Figure 1 in \cite{solomyak}).

Separating
algebraic numbers 
$\beta$ from 
transcendental numbers
using lacunarity 
properties of 
$\beta$-expansions 
has been considered by several authors
(Adamczewski and Bugeaud \cite{adamczewskibugeaud},
Bugeaud \cite{bugeaud},
Dubickas \cite{dubickas12}).

\subsubsection{Zero-free regions and density distribution of zeroes in Solomyak's fractal}
\label{S4.2.2}

Let $\beta> 1$ be a real number (algebraic 
or transcendental).
The Parry Upper function $f_{\beta}(z)$
has its zeroes of modulus $< 1$
in a region of the open unit disk which is given
by
Solomyak's constructions
in \cite{solomyak} \S 3. 
Let us recall them, focusing only on
the interior of the unit disk 
(Theorem \ref{solomyakfractalGBordG}).
From Theorem \ref{parryupperdynamicalzeta}
and \eqref{zetafunctionfraction},
the zeroes of $f_{\beta}(z)$ in $|z| < 1$,
which are $\neq 1/\beta$,
are the zeroes of modulus $< 1$
of the power series 
$1 + \sum_{j=1}^{\infty} T_{\beta}^{j}(1) \, z^j$.
Then let 
$$\mathcal{B} :=
\{h(z) = 1 + \sum_{j=1}^{\infty} a_j z^j
\mid a_j \in [0, 1] \}$$
be the class of power series defined on
$|z| < 1$ equipped with the 
topology of uniform convergence
on compacts sets of $|z| < 1$.  
The subclass $\mathcal{B}_{0,1}$
of $\mathcal{B}$ denotes functions whose 
coefficients are all zeros or ones. 
The space 
$\mathcal{B}$ is compact and convex.
Let
$$\mathcal{G} := \{\lambda \mid |\lambda| < 1,
\exists ~h(z) \in \mathcal{B} ~\mbox{such that}~ 
h(\lambda)=0 \}
\quad
\subset ~D(0,1)$$
be the set of zeroes of the power series 
belonging to $\mathcal{B}$.
The zeroes
gather within the unit circle and
curves in $|z| < 1$
given in polar coordinates, by
\cite{vergergaugry2}.
The complement 
$D(0,1) \setminus (\mathcal{G} \cup \{\frac{1}{\beta}\})$
is a zero-free region for 
$f_{\beta}(z)$; the
domain $D(0,1) \setminus \mathcal{G}$
is 
star-convex due to the fact that:
$h(z) \in \mathcal{B} \Longrightarrow
h(z/r) \in \mathcal{B}$, for
any $r > 1$ (\cite{solomyak}, \S 3),
and that $1/\beta$ is the unique root
of $f_{\beta}(z)$ in $(0,1)$.
If $\beta$ is a Parry number, 
$\mathcal{G}$
contains all the Galois- and beta-conjugates
(if any) of $\beta$ of modulus $< 1$.

For every $\phi \in (0, 2 \pi)$, there exists
$\lambda = r e^{i \phi} \in \mathcal{G}$;
the point of minimal modulus
with argument $\phi$ is denoted 
$\lambda_{\phi} =
\rho_{\phi} e^{i \phi}
\in \mathcal{G}$,
$\rho_{\phi} < 1$. 
A function
$h \in \mathcal{B}$ is called $\phi$-optimal if
$h(\lambda_{\phi})  = 0$.
Denote by $\mathcal{K}$ the subset of
$(0, \pi)$ for which there exists a
$\phi$-optimal function belonging to
$\mathcal{B}_{0,1}$.
Denote by $\partial \mathcal{G}_{S}$ the
``spike": $[-1, \frac{1}{2}(1 - \sqrt{5})]$
on the negative real axis.

\begin{theorem}[Solomyak]
\label{solomyakfractalGBordG}
(i) The union $\mathcal{G} \cup \mathbb{T}
\cup \partial \mathcal{G}_{S}$ is closed,
symmetrical with respect to the real axis, has
a cusp at $z=1$ with logarithmic tangency 
(Figure 1 in {\rm \cite{solomyak}}),
 
(ii) the boundary $\partial \mathcal{G}$ 
is a continuous curve, given by
$\phi \to |\lambda_{\phi}|$ on 
$[0, \pi)$, taking its values in
$[\frac{\sqrt{5} - 1}{2}, 1)$,
with 
$|\lambda_{\phi}| = 1$ if and only if $\phi = 0$.
It admits
a left-limit at $\pi^{-}$,
$1 > \lim_{\phi \to \pi^{-}} |\lambda_{\phi}|
> |\lambda_{\pi}| = \frac{1}{2}(-1 + \sqrt{5})$,
the left-discontinuity at $\pi$ corresponding 
to the extremity of $\partial \mathcal{G}_{S}$.

(iii) at all points $\rho_{\phi} e^{i \phi} \in \mathcal{G}$
such that $\phi/\pi$ is rational in an open dense
subset of $(0,2)$, $\partial \mathcal{G}$
is non-smooth,

(iv) there exists a nonempty subset
of transcendental numbers $L_{tr}$, of 
Hausdorff dimension zero, such that
$\phi \in (0, \pi)$ and
$\phi \not\in \mathcal{K} \cup~ \pi \qb~ \cup~ \pi L_{tr}$ implies that the boundary curve $\partial \mathcal{G}$
has a tangent at $\rho_{\phi} e^{i \phi}$ (smooth point).
\end{theorem}

\begin{proof}
\cite{solomyak}, $\S$ 3 and $\S$ 4.
\end{proof}

Solomyak's (semi-)fractal $\mathcal{G}$
contains the set $\overline{W}$,
where $W$ 
consists of zeroes $\lambda$, 
$|\lambda|< 1$, of polynomials
$1 + \sum_{j=1}^{q} a_j z^j$ having all coefficients
$a_j$ zeroes and ones, studied by
Odlyzko and Poonen \cite{odlyzkopoonen}.
Inside $\mathcal{G}$ the density distribution of the zeroes
of $f_{\beta}(z)$ admits a majorant function,
as follows.

\begin{theorem}
\label{solomyakfractalCOUCHEn}
For all $\beta > 1$,
$f_{\beta}(z)$ 
has no zero
in $|z| < \frac{1}{\beta}$
and all its zeroes $\neq 1/\beta$,
of modulus $< 1$, lie in 
the interior $\stackrel{o}{\mathcal{G}}$.
For all $\beta \in (1, 
\theta_{6}^{-1})$ and
$r$ satisfying 
$\frac{1}{\beta} \leq  r <1$,
$f_{\beta}(z)$ has
\begin{equation}
\label{zeroesdistributionMajorant}
\leq ~c_{(\beta)} 
\times \, \frac{\lo (1-r)}{\lo r} 
\end{equation}
zeroes in the annular region
$\frac{1}{\beta} \leq |z| \leq r$,
with
$$
c_{(\beta)} 
=
\frac{\lo (1-\frac{1}{\beta}) - \lo (2-\frac{1}{\beta})}{(\frac{\beta - 1}{\lo (\beta - 1)} + 1) \, \lo (1-\frac{1}{\beta} )} .
$$ 
\end{theorem}

\begin{proof}
The number of zeroes in
$|z| < 1/\beta$ is equal to 0
\cite{flattolagariaspoonen}.
The coefficients
$T_{\beta}^{j}(1)$, $j \geq 1$, of 
$1 + \sum_{j=1}^{\infty} T_{\beta}^{j}(1) z^j$
are all fractional parts, then
$\neq 1$.
If $f_{\beta}(z)$ admits a zero
on the boundary
$\partial \mathcal{G}$, then 
$1 + \sum_{j=1}^{\infty} T_{\beta}^{j}(1) z^j$
would be $\phi$-optimal for some
$\phi \in (0, \pi]$.
By Lemma 3.4 in \cite{solomyak}, 
at least one coefficient $T_{\beta}^{j}(1)$
should be equal to 1.
Contradiction.

The density distribution of zeroes relies upon
Jensen's formula applied to the circle
$|z| = R$ with $r < R < 1$. 
For all $\beta > 1$, 
the constant term is
$f_{\beta}(0) = -1$.
Let $\beta \in (\theta_{n+1}^{-1}, 
\theta_{n}^{-1})$, $n \geq 1$.
Here $R$ will be taken equal to $r^{1/n}$.
If $z_1 , z_2 , \ldots , z_m$ are the zeroes
of $f_{\beta}(z)$ in $|z| < R$, 
and $Q$ the number
of zeroes in $|z| < r$, then
$$\sum_{q=1}^{m} \lo (\frac{R}{|z_j|})
=
\frac{1}{2 \pi}
\int_{0}^{2 \pi} \lo |f_{\beta}(R e^{i t})| dt .$$
Therefore,
$$Q (\lo R - \lo r)
\leq
\frac{1}{2 \pi}
\int_{0}^{2 \pi} \lo |f_{\beta}(R e^{i t})| dt .
$$
Since $n+1 = \dyg(\beta)$, 
by Theorem \ref{zeronzeron},
$d_{\beta}(1) = 0. 1 0^{n-1} 1 0^{n_1} 1 0^{n_2} 1 \ldots$,
with $n_j \geq n-1$ for $j \geq 1$, we
deduce
$$|f_{\beta}(R e^{i t})|
\leq \sum_{j=0}^{\infty} R^j 
\leq 1 + R + R^{n+1}
\sum_{q=0}^{\infty} R^{n q}
\leq 2 + \frac{r}{1 - r} ,
$$
$$\mbox{and}
\qquad \qquad 
Q ~\leq~ 
\frac{\lo (\frac{2-r}{1-r})}
{(\frac{1}{n} - 1) \, \lo r} .
$$
By Proposition \ref{dygExpression}, for $n \geq 6$,
$n = - \frac{\lo (\beta - 1)}{\beta - 1} + \ldots, $
whose proof is obtained independently.
Therefore
\begin{equation}
\label{zeroesnumbermajorant}
Q ~<~ 
\frac{\lo (1-r) - \lo (2 -r)}
{(\frac{\beta - 1}{\lo (\beta - 1)} + 1) \, \lo r} .
\end{equation}
For any $1 < \beta \leq \theta_{6}^{-1}$,
the function:
$$r \to \frac{-\lo (2-r) + \lo (1-r)}{(\frac{\beta - 1}{\lo(\beta - 1))} + 1) \, \lo (1-r)}\qquad
\mbox{is decreasing on}~~ 
[1/\beta, 1)$$
and takes its maximum at $r=1/\beta$.
We deduce the claim.
\end{proof}

In $\S$ \ref{S5.3} and 
$\S$ \ref{S6.2} we will show much more, 
i.e. the existence of 
lenticuli of roots
of the Parry Upper functions
$f_{\beta}(z)$, for $\beta > 1$ 
small enough,
in $\mathcal{G}$,
spreading inside the cusp region of
$\partial \mathcal{G}$, stemming from $1/\beta$, 
in the neighbourhood 
of $z=1$ towards $e^{\pm i \frac{\pi}{3}}$.
The minoration of the Mahler measure 
${\rm M}(\beta)$,
for $\beta > 1$ being an algebraic integer,
will be deduced from
the explicit asymptotic
expansions
of these lenticular roots.

The identification of zeroes of 
analytical functions,
which lie 
very close to natural boundaries, 
is a difficult problem. For power series
having Hadamard lacunarity Fuchs 
\cite{fuchs}
obtained results. 
Here the power series
have coefficients in a 
finite set with moderate 
lacunarity, making the problem more delicate,
though general theorems exist
(Levinson \cite{levinson} Chap. VI, Robinson
\cite{robinson}).

\subsection{Dichotomy of Perron numbers, Pisot numbers, Salem numbers, Blanchard's classification, Parry numbers, Ito-Sadahiro (Irrap) numbers, negative $\beta$-shift}
\label{S4.3}

The set $\pb$ of Perron numbers 
contains the subset $\pb_P$ of
all (simple and nonsimple)
Parry numbers by a result of Lind 
\cite{lind} (Blanchard \cite{blanchard},
Boyle \cite{boyle}, 
Denker, Grillenberger and Sigmund 
\cite{denkergrillenbergersigmund}, 
Frougny in 
\cite{lothaire} chap.7). 
The set $\pb \setminus \pb_P$
is not empty, at least by
Akiyama's Theorem \ref{akiyamathm};
it would contain 
all Salem numbers of large degrees 
by Thurston \cite{thurston2} p. 11. 
Parry (\cite{parry}, Theorem 5) proved that
the subcollection 
of simple Parry numbers is dense
in $[1, \infty)$.
Simple Parry numbers
$\beta$ often satisfy the minoration
($\S$ \ref{S4.7}):
${\rm M}(\beta) \geq \Theta$.
In the opposite direction 
a Conjecture of K. Schmidt \cite{schmidt}
asserts that Salem numbers are all Parry 
numbers. Boyd 
\cite{boyd15} established
a simple probabilistic model, based on 
the frequencies of digits
occurring in the R\'enyi $\beta$-expansions of unity,
to conjecture that, more realistically,
Salem numbers are dispatched into 
the two sets of Parry numbers and nonParry numbers,
each of them with densities $> 0$. 
This dichotomy of Salem numbers
was reconsidered
by Hichri \cite{hichri} 
\cite{hichri2} \cite{hichri3}.
Few examples
of nonParry algebraic numbers $> 1$ exist;
Solomyak (\cite{solomyak} p. 483)
gives $\frac{1}{2}(1 + \sqrt{13})$.

\begin{theorem}[Akiyama]
\label{akiyamathm}
The dominant root $\gamma_n > 1$
of 
$-1 -z +z^n$, for $n \geq 2$, 
is a Perron number which is a Parry number if and only
if $n=2, 3$. If $n=2, 3$, 
$\gamma_2 = \theta_{2}^{-1}$ and
$\gamma_3 = \theta_{5}^{-1} = \Theta$
are Pisot numbers which are simple Parry numbers. 
\end{theorem}

\begin{proof}
Theorem 1.1 and Lemma 2.2
in \cite{akiyama} 
using Lagrange inversion formula.
The dynamics of the Perron numbers 
$\theta_{n}^{-1}$ for $n \geq 2$ is
reported in $\S$ \ref{S4.5}.
\end{proof}

This dichotomy separates the
set of real algebraic 
integers
$ > 1$
into two disjoint nonempty parts;
in particular, 
in restriction, the 
set of Salem numbers, as:
T $\cap \,\pb_P$ 
vs.
T $\setminus $ T $\cap \, \pb_P$.
The small Salem numbers found 
by Lehmer in \cite{lehmer},
reported below,
either given 
by their minimal polynomial or 
equivalently by their $\beta$-expansion
(``dynamization"),
are Parry numbers: 

\vspace{0.2cm}
\noindent
\begin{tabular}{c|c|c|l}
\label{smallSalemexpansions}
$\deg(\beta)$ & 
$\beta =
{\rm M}(\beta)$
&
minimal pol. of $\beta$
& 
$d_{\beta}(1)$\\
\hline
4 & 
$1.722\ldots$
&
$X^4 - X^3 - X^2 -X +1$
&
$0 . 1 (1 0 0)^{\omega}$\\
6 & 
$1.401\ldots$
&
$X^6 - X^4 - X^3 -X^2 +1$
&
$0 . 1 ( 0^2 1 0^4)^{\omega}$\\
8 & 
$1.2806\ldots$
&
$X^8 - X^5 - X^4 -X^3 +1$
&
$0 . 1 (0^5 1 0^5 1 0^{7})^{\omega}$\\
10 & 
$1.17628\ldots$
&
$X^{10} + X^9 - X^7 - X^6 - X^5 $
&
$0 . 1 (0^{10} 1 0^{18} 1 0^{12} 1 0^{18} 1 0^{12})^{\omega}$\\
 & &
 $- X^4 -X^3 + X + 1$
 & \\

\end{tabular}

\vspace{0.3cm}

The respective dynamical degrees $\dyg(\beta)$
of the last two 
Salem numbers $\beta$ are 7 and 12,
with Parry polynomials of respective 
degrees 20 and 75, given 
by \eqref{parrypolynomeperiodicity}, for
which $f_{\beta}(z) =$
\begin{equation}
\label{salempetitN2}
-\frac{z^{20} -z^{19} -z^{13} -z^{7} -z +1}{1 - z^{19}} = -1 + z +z^{7}+ z^{13} + \ldots = G_{7}(z) +
z^{13}+\ldots ,
\end{equation}
resp. 
$$f_{\beta}(z) =
-\frac{z^{75} -z^{74} -z^{63} -z^{44} -z^{31} -z^{12} -z +1}{1 - z^{74}}$$
\begin{equation}
\label{salempetitN1Lehmernumber}
= -1 + z +z^{12} + z^{31}+ \ldots 
= G_{12}(z) + z^{31}+\ldots .
\end{equation}

\begin{proposition}
\label{neversimpleParrySalem}
Let $\beta > 1$ be an algebraic 
number which admits
$k$ real (Galois) conjugates
distinct of $\beta$ and are $> 0$. If
$k=1$, then $\beta$ is neither 
a nonsimple Parry number
for which $d_{\beta}(1)$ is purely periodic,
nor
a simple 
Parry number.
If $k \geq 2$, then $\beta$ is not a
Parry number.
\end{proposition}

\begin{proof} 
Let $k=1$.
If 
$\beta$ 
is a simple
Parry number, then
$d_{\beta}(1) = 0.t_1 t_2 \ldots t_N$ 
would denote the R\'enyi $\beta$-expansion of 1,
for some integer $N \geq 2$, with $t_N \neq 0$.
Thus $\beta$ would be a root of the polynomial 
$X^N - t_1 X^{N-1} - t_2 X^{N-2} - \ldots - t_N$ .
Since all the coefficients $-t_j$ are 
$\leq 0$, the number of 
changes of sign in the coefficient vector of 
this polynomial is equal to 1. 
By Descartes's rule of signs this polynomial 
would possess only one positive real root.
Contradiction. The proof is the same if we assume
the pure periodicity (\eqref{parrypolynomepureperiodicity}
admits one change of sign).

Let $k \geq 2$. If $\beta$ is 
assumed a Parry number, simple or not,
the number of sign changes
in the Parry polynomial is either 
equal to 1 or 2, according to
the type of Parry polynomial
\eqref{parrypolynomesimple}
\eqref{parrypolynomepureperiodicity} or
\eqref{parrypolynomeperiodicity}.
The number of positive real roots would be $\leq 2$.
Contradiction.
\end{proof}

A real algebraic integer $\beta > 1$ close to one
is commonly given by its
minimal polynomial, not by its
$\beta$-expansion of unity, 
though it is equivalent.
The Parry Upper function 
at $\beta$ is a sparse power series
having a coefficient vector
with gaps of zeroes (missing monomials)
of minimal 
length 
controlled by
the dynamical degree $\dyg(\beta)$; to
examplify the role played by $\dyg(\beta)$, 
we will reverse the representation of 
the smallest Salem numbers $\beta > 1$
in the following, by considering their 
$\beta$-expansions.
 
The smallest 
Salem numbers 
of degree $\leq 44$ are all known from the complete
list of Mahler measures $\leq 1.3$ of 
Mossinghoff \cite{mossinghofflist} of irreducible monic
integer polynomials of degree $\leq 180$. 
Table 1 gives the subcollection of those
Salem numbers $\beta$ which are Parry numbers
within the intervals of extremities
the Perron numbers 
$\theta_{n}^{-1}, \,n = 5, 6, \ldots, 12$.
In each interval $\dyg(\beta)$ is constant
while the increasing order of the $\beta$s
corresponds to a certain disparity
of the degrees $\deg(\beta)$.
The remaining Salem numbers in \cite{mossinghofflist}
are very probably nonParry numbers 
though proofs are not available yet; they are not included
in Table 1. Apart from them,
the other Salem numbers which exist in the intervals
$(\theta_{n}^{-1}, \theta_{n-1}^{-1}), 
\,n \geq 6$,
if any, should be of (usual) degrees
$\deg > 180$. 

\vspace{0.2cm}

\noindent
\begin{small}
\begin{tabular}{c|c|c|c|c}
$\dyg$ & $\deg$ & $\beta$ & $P_{\beta,P}$ & $d_{\beta}(1)$\\
\hline
\\
5 & 3 & $\theta_{5}^{-1}=1.324717$ & 5 & $0.1 0^{3} 1$\\
\hline
\\
6 & 18 & $1.29567$ & 22 & $0.1(0^4 1 0^9 1 0^6)^{\omega}$\\
6 & 10 & $1.293485$ & 12 & $0.1(0^4 1 0^6)^{\omega}$\\
6 & 24 & $1.291741$ & 24 irr. & $0.1(0^4 1 0^{11} 1 0^6)^{\omega}$ \\
6 & 26 & $1.286730$ & 30 & $0.1(0^4 1 0^{17} 1 0^6)^{\omega}$ \\
6 & 34 & $1.285409$ & 38 & $0.1(0^4 1 0^{25} 1 0^6)^{\omega}$ \\
6 & 30 & $1.285235$ & 45 & $0.1(0^4 1 0^{32} 1 0^6)^{\omega}$ \\
6 & 44 & $1.285199$ & 66 & $0.1(0^4 1 0^{54} 1 0^6)^{\omega}$ \\
6 & 6 & $\theta_{6}^{-1}=1.285199$ & 6 irr.& $0.1 0^{4} 1$\\
\hline
\\
7 & 26 & $1.285196$ & 44 & $
0.1(0^5 1 0^5 1 0^5 1 0^5 1 0^5 1 0^{5} 1 0^7)^{\omega}$ \\
7 & 26 & $1.281691$ & .. & $
0.1(0^5 1 0^5 1 0^9 1 0^5 1 0^{17} 1 0^{7} 1 0^6 1 0^6 1 0^7 1 0^{12})^{\omega}$ \\
7 & 8 & $1.280638$ & 20 & $
0.1(0^5 1 0^5 1 0^7)^{\omega}$ \\
7 & 10 & $1.261230$ & 14 & $
0.1(0^5 1 0^7)^{\omega}$ \\
7 & 24 & $1.260103$ & 28 & $
0.1(0^5 1 0^{13} 1 0^7)^{\omega}$ \\
7 & 18 & $1.256221$ & 36 & $
0.1(0^5 1 0^{21} 1 0^7)^{\omega}$ \\
7 & 7 & $\theta_{7}^{-1}=1.255422$ & 7 irr.& $0.1 0^{5} 1$\\
\hline\\
8 & 18 & $1.252775$ & 120 & $
0.1(0^6 1 0^{6} 1 0^{10} 1 0^{16} 1 0^{12} 1 0^7 1 0^{12} 0^{16} 1 0^{10}1 0^6 1 0^8)^{\omega}$ \\
8 & 12 & $1.240726$ & 48 & $
0.1(0^6 1 0^{11} 1 0^7 1 0^{11} 1 0^8)^{\omega}$ \\
8 & 20 & $1.232613$ & 41 & $
0.1(0^6 1 0^{24} 1 0^8)^{\omega}$ \\
8 & 8 & $\theta_{8}^{-1}=1.232054$ & 8 irr.& $0.1 0^{6} 1$\\
\hline\\
9 & 10 & $1.216391$ & 18 & $
0.1(0^7 1 0^9)^{\omega}$ \\
9 & 9 & $\theta_{9}^{-1}=1.213149$ & 9 irr.& $0.1 0^{7} 1$\\
\hline\\
10 & 14 & $1.200026$ & 20 & $
0.1(0^8 1 0^{10})^{\omega}$ \\
10 & 10 & $\theta_{10}^{-1}=1.197491$ & 10 irr.& $0.1 0^{8} 1$\\
\hline\\
11 & 9 & $\theta_{11}^{-1}=1.184276$ & 11 & $0.1 0^{9} 1$\\
\hline\\
12 & 10 & $1.176280$ & 75 & Lehmer's number $:
0.1(0^{10} 1 0^{18} 1 0^{12} 1 0^{18} 1 0^{12})^{\omega}$ \\
12 & 12 & $\theta_{12}^{-1}=1.172950$ & 12 irr.& $0.1 0^{10} 1$\\
\hline
\end{tabular}

\begin{center}
Table 1.~ Smallest Salem numbers 
$\beta < 1.3$ of degree $\leq 44$, which are Parry numbers, from
\cite{mossinghofflist}. In Column 1 is reported 
the dynamical degree of $\beta$ 
(cf Theorem \ref{dygdeginequality} 
for its asymptotic expression). 
Column 4 gives the degree of the Parry polynomial
$P_{\beta,P}$ of $\beta$; $P_{\beta,P}$ is reducible except if ``irr." is mentioned.
\end{center} 
\end{small}

Generalizing \eqref{salempetitN2}
and \eqref{salempetitN1Lehmernumber} 
to the Salem numbers $\beta$ 
and Perron numbers
of Table 1,
the Parry Upper function at $\beta$
takes the general form, after 
Theorem \ref{zeronzeron}:
\begin{equation}
f_{\beta}(z) = G_{\dyg(\beta)} + z^{n_1} + \ldots
\qquad \mbox{with}~~ n_1 \geq 2 \, \dyg(\beta) -1.
\end{equation}

For some families of algebraic integers,
this ``dynamization" of the minimal polynomial 
is known explicitly, the digits being algebraic
functions of the coefficients of the minimal
polynomials: e.g.
for Salem numbers of degree 4 and 6
(Boyd \cite{boyd11} \cite{boyd12}),
for Salem numbers of degree 8 
(Hichri \cite{hichri} \cite{hichri2}
\cite{hichri3}), 
for Pisot numbers
(Boyd \cite{boyd15}, Frougny and Solomyak
\cite{frougnysolomyak},
Bassino \cite{bassino} in the cubic case,
Hare \cite{hare},
Panju \cite{panju} for regular Pisot numbers).
Schmidt \cite{schmidt}, independently
Bertrand-Mathis \cite{bertrandmathis},
proved that Pisot numbers are Parry numbers. 
Many Salem numbers 
are known to be  Parry numbers.
For Salem numbers of degree 4
it is the case \cite{boyd13}.
For Salem numbers
of degree $\geq 6$, Boyd \cite{boyd14}
gave an heuristic argument and a 
probabilistic model, 
for the existence of nonParry Salem numbers 
as a metric approach of the dichotomy 
of Salem numbers.
This approach, coherent with Thurston's one 
(\cite{thurston2}, p. 11),
is in contradiction with the conjecture of Schmidt. 
Hichri \cite{hichri} 
\cite{hichri2}
\cite{hichri3}
further developped the heuristic approach of Boyd
for Salem numbers of degree 8. 
The Salem numbers of degree $\leq 8$ are all greater
than $1.280638\ldots$ from 
\cite{mossinghofflist}.

Using the ``Construction of Salem'',
Hare and Tweedle \cite{haretweedle}
obtain convergent families of Salem numbers, 
all Parry numbers, having as limit points
the limit points of
 the set S of Pisot numbers 
 in the interval $(1, 2)$ (characterized by
 Amara \cite{amara}).
These families of Parry Salem numbers 
do not contain Salem numbers
smaller than 
Lehmer's number.

Parry numbers are studied from the 
negative $\beta$-shift.
The negative $\beta$-shift 
was introduced by Ito and Sadahiro \cite{itosadahiro}
(Liao and Steiner \cite{liaosteiner},
Masakova and Pelantova \cite{masakovapelantova},
Ngu\'ema Ndong \cite{nguemandong}) 
and 
the generalized $\beta$-shift
by Gora \cite{gora} \cite{gora2} and
Thompson \cite{thompson}), in the general context of
iterated interval maps and post-critical finite (PCF)
interval maps
\cite{milnorthurston} \cite{thurston2}.
Indeed, Kalle \cite{kalle}
showed that nonisomorphisms exist
between the $\beta$-shift and the 
negative $\beta$-shift, possibly 
leading to new Parry numbers arising from
``negative" Parry numbers
(called 
Ito-Sadahiro numbers in
\cite{masakovapelantova}, 
Irrap numbers in \cite{liaosteiner}
reading "Parry" from the right to the left). 
More generally 
negative Parry numbers and 
generalized Parry numbers
are defined
as poles of the corresponding
dynamical zeta functions
\cite{nguemandong}
\cite{thompson2}. 
Negative Pisot and Salem numbers appear 
naturally in several domains:
as roots
of Newman polynomials \cite{haremossinghoff},
in association equations with negative 
Salem polynomials \cite{guichardvergergaugry}, in 
topology with the Alexander polynomials of
pretzel links ($\S$ \ref{S2}),
as Coxeter polynomials
for Coxeter elements
(Hironaka \cite{hironaka}; $\S$ \ref{S2}),
in studies of numeration with negative bases
(Frougny and Lai \cite{frougnylai}). 
Generalizing Solomyak's constructions
to the generalized $\beta$-shift, 
Thompson \cite{thompson2} 
investigates the
fractal domains of existence of the conjugates. 

\vspace{0.3cm}

In terms of Parry Upper functions
$f_{\beta}(z)$ the dichotomy
between Parry numbers and nonParry numbers
(Theorem \ref{carlsonpolya})
corresponds to a rationality criterium.
By \eqref{zetafunctionfraction} and
\eqref{fbetazetabeta} 
the structure properties of
the Parry Upper function
$f_{\beta}(z)$ and 
the dynamical zeta function $\zeta_{\beta}(z)$
\eqref{dynamicalfunction}
are based on the
knowledge of the 
number of orbits of 1 under $T_{\beta}$,
of length dividing $n$,
$n \to \infty$, and the topological properties
of the set $\{T_{\beta}^{n}(1)\}$. 
On this basis Blanchard \cite{blanchard} 
proposed a classification
of real numbers $\beta > 1$ into five classes;
Verger-Gaugry in \cite{vergergaugry} refined it 
in terms of asymptotic gappiness in the direction of
more enlighting the algebraicity of $\beta$: 

\noindent
\begin{tabular}{ll}
Class C1: & $d_{\beta}(1)$ is finite,\\

Class C2: & $d_{\beta}(1)$ is ultimately periodic but not finite,\\
Class C3: & $d_{\beta}(1)$ contains bounded strings of zeroes, but is not ultimately\\ 
& periodic (0 is not an accumulation point of $\{T_{\beta}^{n}(1)\})$,\\
Class C4: & $\{T_{\beta}^{n}(1)\}$ is not dense in $[0, 1]$, but
admits 0 as an accumulation point,\\
Class C5: & $\{T_{\beta}^{n}(1)\}$ is dense in $[0, 1]$.

\end{tabular}

Apart from C1, resp. C2, 
which is exactly 
the set of simple, resp. nonsimple, Parry numbers,
how the remaining algebraic numbers $> 1$
are dispatched in 
the classes
C3, C4 and C5 is obscure.
The specification property, meaning
that 0 is not an accumulation point for $\{T_{\beta}^{n}(1)\}$, 
was weakened
by Pfister and Sullivan \cite{pfistersullivan}
and Thompson \cite{thompson}.
Schmeling \cite{schmeling} proved that the class C3 has full Hausdorff dimension
and that the class C5,
probably  mostly occupied 
by transcendental numbers, is of full 
Lebesgue measure 1.
Lacunarity and Diophantine approximation
were investigated by 
Bugeaud and Liao \cite{bugeaudliao},
Hu, Tong and Yu \cite{hutongyu}, 
Li, Persson, Wang and Wu \cite{liperssonwangwu}.
For any $x_0 \in [0,1]$ the asymptotic distance
$\liminf_{n \to \infty} |T_{\beta}^{n}(1) - x_0|$, 
for almost all
$\beta > 1$ (for the Lebesgue measure),
was studied by Persson and Schmeling
\cite{perssonschmeling} \cite{schmeling}, 
Ban and Li \cite{banli},
Cao \cite{cao}, Fang, Wu and Li \cite{fangwuli},
Li and Chen \cite{lichen},
L\"u and Wu \cite{luwu}, Tan and Wang \cite{tanwang}.
Kwon \cite{kwon} studies the 
subset of Parry numbers whose conjugates
lie close to the unit circle, using 
technics of combinatorics of words.
Adamczewski and Bugeaud 
(\cite{adamczewskibugeaud}, Theorem 4)
show that the class C4 contains self-lacunary numbers,
all transcendental, from Schmidt's
Subspace Theorem
and results of Corvaja and Zannier.

\subsection{Cyclotomic jumps in families of Parry Upper functions, and right-continuity}
\label{S4.4}

Allowing the real 
base $\beta$ to vary continuously
in the neighbourhood $[1, \theta_{2}^{-1})$ of 1, 
except 1,
asks the question whether it has a sense to
consider the continuity of the bivariate
Parry Upper function
$(\beta, z) \to f_{\beta}(z)$, and, if it is the case, 
on which subsets, in $z$, of the complex plane.

Theorem \ref{convergencecompactsetsUNITDISK} 
and its Corollary
show that the open unit disk is a
domain where the continuity of the roots of
$f_{\beta}(z)$ in $|z| < 1$
holds though the
functions f$_{\beta}(z)$ are only right continuous
in $\beta$,
with infinitely many cyclotomic jumps,
while, in the complement $|z| \geq 1$,
either the Parry Upper functions are not defined,
or may exhibit drastic changes 
on the unit circle.

\begin{lemma}
\label{continuity}
Let $1 < \beta < \theta_{2}^{-1}$
and
$0 < x < 1$. Then
(i) the bivariate $\beta$-transformation map 
$(\beta,x) \to
T_{\beta}(x) =\{\beta x\} = 
\beta x -\lfloor \beta x \rfloor$ 
is continuous, 
in $\beta$ and $x$, when 
$\beta x$ is not a positive integer.
If $\beta x$ is a positive integer,
$x = 1/\beta$ and
\begin{equation}
\label{continuitegauche}
\lim_{y \to \frac{1}{\beta}^{-}, \,\gamma \to \beta^{-}}
T_{\gamma}(y) = 1,
\qquad
\lim_{y \to \frac{1}{\beta}^{+}, \, \gamma \to \beta^{+}}
T_{\gamma}(y) = 
T_{\beta}(\frac{1}{\beta}) = 0;
\end{equation}
(ii) for any $(\beta, x)$,
there exists $\epsilon = \epsilon_{\beta, x}$
such that $T_{\gamma}(y)$ is increasing
both in $\gamma
\in [\beta, \beta + \epsilon)$ 
and in $y \in [x, x+\epsilon)$.
\end{lemma}

\begin{proof}
Lemma 3.1 in \cite{flattolagariaspoonen}.
(i) If $\beta x$ is an integer, this integer 
is 1 necessarily. The value $x=1/\beta$ is
a negative power of $\beta$. 
The R\'enyi $\beta$-expansion of $1/\beta$
is deduced from $d_{\beta}(1)$ by a shift,
given in \eqref{renyidef} and 
\eqref{renyidef_unsurbeta};
the sequence
$(T_{\beta}^{n}(\frac{1}{\beta}))_{n \geq 1}$
is directly obtained from
$(T_{\beta}^{n}(1))_{n \geq 1}$.
The fractional part
$\gamma \to \{\gamma\}= T_{\gamma}(1)$ 
is right continuous, 
hence the result; (ii) obvious.
\end{proof}

\begin{lemma}
\label{continuitysimpleParrynumber}
Let $\beta \in (1, \theta_{2}^{-1})$.
(i) If $\beta$ 
is a simple Parry number,
then, for all $n \geq 1$, the map
$\gamma \to T_{\gamma}^{n}(1)$
is right continuous
at $\beta$:
\begin{equation}
\label{continuitedroite}
\lim_{\gamma \to \beta^{+}} T_{\gamma}^{n}(1)
= T_{\beta}^{n}(1),
\end{equation}

(ii) if $\beta$ is a simple Parry number,
such that $T_{\beta}^{N}(1) = 0$
with $T_{\beta}^{k}(1) \neq 0$, $1 \leq k < N$, 
then, for all $n \geq 1$,
\begin{equation}
\label{discontinuitegauche}
\lim_{\gamma \to \beta^{-}} T_{\gamma}^{n}(1)
= 
\left\{
\begin{array}{lc}
T_{\beta}^{n}(1),& \qquad n < N \quad(\mbox{left continuity})\\
T_{\beta}^{n_N}(1),& \qquad
n \geq N,
\end{array}
\right.
\end{equation}
where $n_N \in \{0, 1, \ldots N-1\}$
is the residue of $n$ modulo $N$,

(iii) if $\beta$ 
is a nonsimple Parry number,
then
$\gamma \to T_{\gamma}^{n}(1)$
is continuous 
at $\beta$:
\begin{equation}
\label{continuitegauchedroite}
\lim_{\gamma \to \beta^{-}} T_{\gamma}^{n}(1)
= T_{\beta}^{n}(1) = 
\lim_{\gamma \to \beta^{+}} T_{\gamma}^{n}(1),
\qquad \mbox{for all }~ n \geq 1.
\end{equation}
\end{lemma}

\begin{proof} Lemma 3.2 in 
\cite{flattolagariaspoonen}.
\end{proof}

Denote by
$$\mathcal{F}:=
\{{f_{\beta}}_{|_{|z| < 1}}(z) \mid
1 < \beta < \theta_{2}^{-1}
\}$$
the set of the restrictions of the 
Parry Upper functions $f_{\beta}(z)$, 
$1 < \beta < \theta_{2}^{-1}$,
to the open unit disk. The set $\mathcal{F}$ is 
equipped with
the topology of the uniform convergence on 
compact subsets of $|z| < 1$.

\begin{theorem}
\label{convergencecompactsetsUNITDISK}
In $\mathcal{F}$
the following right and left
limits hold:
(i) if  $\beta$ be a nonsimple Parry number,
then continuity occurs as:
\begin{equation}
\label{ParryUpperfunctioncontinuity}
\lim_{\gamma \to \beta^-}
f_{\gamma}(z) = 
f_{\beta}(z)
=
\lim_{\gamma \to \beta^+}
f_{\gamma}(z),
\end{equation}
(ii) if $\beta$ is a simple Parry number,
and $N$ the minimal value for which
$T_{\beta}^{N}(1) = 0$, then
\begin{equation}
\label{ParryUpperfunctionrightcontinuity}
\lim_{\gamma \to \beta^+}
f_{\gamma}(z) = 
f_{\beta}(z) ,
\end{equation}
\begin{equation}
\label{ParryUpperfunctionleftCycloJumpcontinuity}
\lim_{\gamma \to \beta^-}
f_{\gamma}(z) = 
\frac{f_{\beta}(z)}{(1 - z^N)}.
\end{equation}
\end{theorem}

\begin{proof}
Let $\gamma, \beta \in (1, \theta_{2}^{-1})$
with $|\gamma - \beta| \leq \epsilon$,
$\epsilon > 0$,
$d_{\gamma}(1) = 0 . t'_1 t'_2 \ldots$
and $d_{\beta}(1) = 0 . t_1 t_2 \ldots$. 
Any compact subset
of $|z| < 1$ is included
in a closed disk centered at 0 of radius $r$
for some $0 < r < 1$.
Assume $|z| \leq r$. 
(i) Assume $\beta$ nonsimple. 
Since
$|T_{\gamma}^{m}(1) -
T_{\beta}^{m}(1)| \leq 2$ for $m \geq 1$,
then
$$|f_{\gamma}(z) - f_{\beta}(z)|
=\!
\bigl| \sum_{n \geq 1} 
(t'_n -t_n) z^n
\bigr|
\!=\!
\bigl|\sum_{n \geq 1} 
[(\gamma T_{\gamma}^{n-1}(1)
- \beta T_{\beta}^{n-1}(1))
-
(T_{\gamma}^{n}(1) -
T_{\beta}^{n}(1))]
 z^n \Bigr|
$$
\begin{equation}
\label{lebesgueconvergencedominated}
\leq
\sum_{n \geq 1} 
\Bigl|(\gamma T_{\gamma}^{n-1}(1)
- \beta T_{\beta}^{n-1}(1))
-
(T_{\gamma}^{n}(1) -
T_{\beta}^{n}(1))\Bigr|
 r^n
\leq 2 (\epsilon + \beta + 1) \sum_{n \geq 1} r^n ,
\end{equation}
which is convergent.
By \eqref{continuitegauchedroite} and 
the Lebesgue dominated convergence theorem,
taking the limit termwise in the summation,
$$\lim_{\gamma \to \beta}
|f_{\gamma}(z) - f_{\beta}(z)| = 0,
\qquad
\mbox{uniformly for}~ |z| \leq r.$$
(ii) 
By \eqref{continuitedroite}
and
\eqref{discontinuitegauche}, the iterates
of $1$ under the $\gamma$-transformation
$T_{\gamma}^n(1)$ behave differently at $\beta$
if $\gamma < \beta$ or resp. 
$\gamma > \beta$ when $\gamma$ tends to $\beta$:
if $\gamma > \beta$, we apply
the Lebesgue dominated convergence theorem
in \eqref{lebesgueconvergencedominated}
to obtain the
right continuity at $\beta$, i.e.
\eqref{ParryUpperfunctionrightcontinuity}; 
if
$\gamma \to \beta^-$,
\eqref{ParryUpperfunctionleftCycloJumpcontinuity}
comes from the dominated convergence theorem
applied to
$$
f_{\gamma}(z) - \frac{1}{1 - z^N}f_{\beta}(z)
= 
(\beta z - 1)
\left[
\Bigl( \sum_{n =0}^{\infty} 
T_{\gamma}^{n}(1) 
\frac{\gamma z - 1}{\beta z - 1} 
z^n
\Bigr) 
- 
\Bigl( \sum_{q=0}^{\infty}
\sum_{m = 0}^{N-1} T_{\beta}^{m}(1) z^{m + qN}
\Bigr)
\right].
$$
\end{proof}

Theorem B in Mori \cite{mori}, on 
the continuity properties of 
spectra of Fredholm matrices,
admits the following counterpart
in terms of the Parry Upper functions:

\begin{corollary}
\label{zeroesParryUpperfunctionContinuity}
The root functions of $f_{\beta}(z)$ valued in
$|z| < 1$ are all continuous, 
as functions of $\beta \in (1, \theta_{2}^{-1})
\setminus \bigcup_{n \geq 3} \{\theta_{n}^{-1}\}$.
\end{corollary}

\begin{proof}
Let $(\gamma_{i})_{i \geq 1}$ be a sequence
of real numbers tending to $\beta$.
The (restrictions, to the open unit disk, of
the) functions $f_{\gamma_{i}}(z)$ 
constitute a convergent sequence in
$\mathcal{F}$,
tending
either to $f_{\beta}(z)$ or
$f_{\beta}(z)/(1 - z^N)$ for some integer $N \geq 1$.
By Hurwitz's Theorem (\cite{sakszygmund} (11.1))
any disk in $|z| < 1$, whose closure does not intersect
the unit circle,
which contains a zero $w(\beta)$ of 
$f_{\beta}(z)$ also contains a zero of 
$f_{\gamma_{i}}(z)$ for all $i \geq i_0$,
for some $i_0$.
The multiplicity of $w(\beta)$ is equal 
to the number of zeroes $w(\gamma_{i})$, 
counted with multiplicities, in this disk.
\end{proof}

\begin{theorem}
\label{notsimplesmallMahlermeasure}
(i) If $\beta > 1$
~is an algebraic integer
for which the Mahler measure ${\rm M}(\beta) < \Theta =
\theta_{5}^{-1}$,
then
$\beta$ is not a simple Parry number;
(ii) conversely, if 
$1 < \beta < 2$ is a simple Parry number
for which the Mahler measure
of the complementary factor 
${\rm M}(\frac{P_{\beta, P}}{P_{\beta}})=1$, equivalently for which 
the complementary factor 
$\frac{P_{\beta, P}(X)}{P_{\beta}(X)}$
is a product of cyclotomic polynomials, 
then 
${\rm M}(\beta)
= {\rm M}(P_{\beta}) 
= {\rm M}(P_{\beta, P})
\geq \Theta$.
\end{theorem}

\begin{proof}
By the Theorem of C. Smyth \cite{smyth}
the minimal polynomial $P_{\beta}(X)$ of
$\beta$ is reciprocal. 
Then $\beta$ and $1/\beta$ are 
(Galois-) conjugated and are two 
real positive roots of $P_{\beta}$. 
But the Parry polynomial $P_{\beta, P}(X)$
of $\beta$ is a multiple of $P_{\beta}(X)$.
Therefore the Parry polynomial $P_{\beta, P}$
of $\beta$ would also
have at least two real positive roots.
If we assume that 
$\beta$ is a simple Parry number, 
the number of positive real roots of its Parry polynomial
should be equal to $1$ since 
$P_{\beta, P}(X)$ admits only one change of signs, by
Descartes's
rule of signs. 
From 
Proposition \ref{neversimpleParrySalem} 
we deduce the contradiction.
The converse is obvious from (i) and
by Smyth's Theorem \cite{smyth}.
\end{proof}

An example of sequence of
simple Parry numbers which are Perron 
numbers tending to $1^+$ is the
sequence $(\gamma_{n,k})$
of the dominant roots $\gamma_{n,k}$
of the trinomials
$X^n - X^k -1$ for $1 \leq k < n$ and $n \to +\infty$.
Indeed, the result 
of Flammang \cite{flammang3},
who proved the conjecture of C. Smyth
for height one trinomials, 
for $n$ large enough,
gives 
$\lim_{n \to \infty} {\rm M}(\gamma_{n,k})
=
1.38135\ldots > \Theta = 1.3247\ldots$ as a true limit,
the trinomials $X^n - X^k -1$ being
all Parry polynomials by the self-admissibility of
the coefficient vectors
(taking care of the factorization is useless). 
As soon as $n$ is large enough,
the inequality
${\rm M}(\gamma_{n,k}) \geq \Theta$ 
is fulfilled.

Another consequence, in $\mathcal{F}$,
is the disappearance of the cyclotomic jumps
of the left-discontinuities at 
the algebraic integers
$\beta > 1$ close to $1^+$, of small measure.

\begin{corollary}
\label{disappearanceJumps}
If $\beta \in (1, \theta_{2}^{-1})$ is an algebraic integer
for which the Mahler measure 
${\rm M}(\beta) < \Theta = \theta_{5}^{-1}$, 
then the
left and right continuity of the Parry Upper function
occurs in $\mathcal{F}$ as:
\begin{equation}
\label{ParryUpperfunctioncontinuitysmallMahler}
\lim_{\gamma \to \beta^-}
f_{\gamma}(z) = 
f_{\beta}(z)
=
\lim_{\gamma \to \beta^+}
f_{\gamma}(z).
\end{equation}
\end{corollary}

\begin{proof}
From Theorem \ref{convergencecompactsetsUNITDISK}
and Theorem \ref{notsimplesmallMahlermeasure}, 
the case
\eqref{ParryUpperfunctionleftCycloJumpcontinuity}
cannot occur.
\end{proof}

\subsection{The R\'enyi-Parry dynamical systems in Perron number base inherited from the trinomials $-1+X+X^n$, and a perturbation theory of Parry Upper functions respecting the lexicographical ordering in the $\beta$-shift}
\label{S4.5}

\begin{proposition}
\label{perronparryGn}
Let $n \geq 2$. The Perron number $\theta_{n}^{-1}$,
dominant root of the trinomial 
$G_{n}^{*}(X) = -X^n + X^{n-1} + 1$, is a 
simple Parry number for which $T_{\theta_{n}^{-1}}^{j}(1)$ 
is nonzero for $j=2, 3, \ldots, n-1$,
$T_{\theta_{n}^{-1}}^{j}(1) = 0$ for $j \geq n$ with
$d_{\theta_{n}^{-1}}(1) = 0. 1 0^{n-2} 1.$
\end{proposition}

\begin{proof}
The first digit $t_1$ of $d_{\theta_{n}^{-1}}(1)$
is obviously $\lfloor \theta_{n}^{-1} \rfloor = 1$
for all $n \geq 2$.
Since the identity $G_{n}^{*}(\theta_{n}^{-1}) = 0 =
\theta_{n}^{-n} - \theta_{n}^{-n+1} - 1$
holds we deduce $\theta_{n}^{-n+1} (\theta_{n}^{-1}-1) = 1$,
therefore
$$T_{\theta_{n}^{-1}}(1) = \{\theta_{n}^{-1}\} = 
\theta_{n}^{-1} - 1 = \frac{1}{\theta_{n}^{-n+1}} ~\in (0,1).$$
Then
$\displaystyle \qquad \qquad \quad 
T_{\theta_{n}^{-1}}^{j}(1) = 
\frac{1}{\theta_{n}^{-n+j}} > 0 \qquad \mbox{for}~ j = 1, 2, \ldots, n-1$  

\noindent
and $\displaystyle \qquad \qquad \qquad \quad
T_{\theta_{n}^{-1}}^{n}(1) = \{\theta_{n}^{-1} T_{\theta_{n}^{-1}}^{n-1}(1)\}
= \{\frac{\theta_{n}^{-1}}{\theta_{n}^{-1}}\}  = 0.$

Consequently
$T_{\theta_{n}^{-1}}^{j}(1) = 0$ for all $j > n$.
By \eqref{polyTbeta} we deduce recursively 
the values of the integers
$t_j$ for $j \geq 2$: $t_2 = 0 = t_3 = \ldots = t_{n-1}$,
$t_{n} = 1$ and
$t_j = 0$ as soon as $j > n$.
\end{proof}

If $n \geq 2$, as analytic function, 
the Parry Upper function
$f_{\theta_{n}^{-1}}(z)$
is equal to the height one trinomial 
$-1 + z + z^n = G_{n}(z)$, a polynomial. 
On the other hand the 
Parry polynomial of $\theta_{n}^{-1}$
is the reciprocal of $G_n$, and  a multiple of
its minimal polynomial:
$P_{\theta_{n}^{-1},P}(z) = 
- G_{n}^{*}(z)
=
X^n - X^{n-1} - 1$ is
irreducible by 
Proposition
\ref{irredGn} except
if $n \equiv 5 ~({\rm mod}~ 6)$.

It is tempting 
to establish
a {\em perturbation theory} based on polynomials 
to try to answer 
Lehmer's question when $\beta$ 
tends to 1. A priori, as ``reference" polynomials, 
the families
$(G_n)$ and $(G_{n}^{*})$
are not good candidates 
since, by Smyth's theorem,
they are not reciprocal and
therefore have a Mahler measure
$> \Theta$, far from 1. 
In this direction several 
apparently better starting points 
were investigated:
(i) families of cyclotomic polynomials
(Amoroso \cite{amoroso}),
(ii) families given by a parametrization 
of two-variable polynomials
constructed from cyclotomics,
having minimality properties for the Mahler measure
(Ray \cite{ray2}), 
(iii) families of perturbed polynomials with the Zhang-Zagier height \cite{doche2},
(iv) families of polynomials defined by varying coefficients (Sinclair \cite{sinclair}) and/or 
having their roots on the unit circle
(Mossinghoff, Pinner and Vaaler
\cite{mossinghoffpinnervaaler}),
(v) families of polynomials with coefficients
in the ring of integers of a number field and
having their roots on the unit circle
(Toledano \cite{toledano}). 
Starting from polynomials having a 
Mahler measure equal to 1, as in (i) to (v),
seems natural.

The second direction, 
which is natural in the context
of the $\beta$-shift,
consists in starting from
the family of polynomials
$(G_{n}(z))$, but viewed as set of 
values
of the Parry Upper function
$f_{\beta}(z)$ at all
$\beta = \theta_{n}^{-1}$.
The theory of perturbation we are looking for
is now deduced, 
in the case where $\beta$ is a 
Parry number, from 
Proposition \ref{variationbasebeta}
and Theorem \ref{carlsonpolya}:
for $\theta_{n}^{-1} \leq
\beta <
\theta_{n-1}^{-1}$, then
$d_{\theta_{n}^{-1}}(1) =
0.10^{n-2}1
\leq_{lex} d_{\beta}(1)
= 0.t_1 t_2 \ldots
<_{lex}
d_{\theta_{n-1}^{-1}}(1) =
0.10^{n-3}1$, with
a sequence of digits
$(t_i)$ which is either finite
or ultimately periodic.
The lacunarity in $(t_i)$ is controlled 
by the integer $n$.
From $d_{\beta}(1)$ the power series
$f_{\beta}(z)$ is deduced and
the reciprocal polynomial 
$P^{*}_{\beta,P}(z)$
of the Parry polynomial of 
$\beta$ is obtained, together with 
the minimal
polynomial of $\beta$ which is only one of its
irreducible factors, of multiplicity one.
Though feasible, to recover 
the minimal polynomial of $\beta$, the theory of factorization of
Parry polynomials is a deep question
\cite{vergergaugry3} (for instance using 
Bombieri's norms, introduced in
\cite{beauzamybombierienflomontgomery},  
instead of Mahler measures).
We would obtain:
${\rm M}(\beta) \leq {\rm M}(P^{*}_{\beta,P})$.
The case where $\beta$ is an algebraic integer
which is not a Parry number is more complicated
since it relies upon a theory of
divisibility of the integer power series (not
eventually periodic, with lacunarity controlled
by $\dyg(\beta)$)
which are the denominators
of the dynamical zeta functions
$\zeta_{\beta}(z)$ by integer monic polynomials.

This theory of perturbation respects the
lexicographical ordering in the dynamization
of the defining equations. It is adapted
to the coding of algebraic numbers in 
the $\beta$-shift, $\beta$ being the 
important variable.
This theory of perturbation of the
function
$\beta \to f_{\beta}(z)$ relies first upon
the knowledge of the set of Parry numbers,
by
Theorem \ref{parryupperdynamicalzeta}. 
Finding a rationality
criterion for $\zeta_{\beta}(z)$
is as difficult 
as solving the Weil's conjectures
(Deligne \cite{deligne},
Dwork \cite{dwork}, Kedlaya
\cite{kedlaya}, Weil \cite{weil}).
In the neighbourhood of
the sequence 
$(\theta_{n}^{-1})$ 
this theory of perturbation 
takes the following form.

\begin{theorem}
\label{parryupperfunctionPerronNeighbourhood}
There exists a decreasing 
sequence of positive real numbers
$(\epsilon_n)_{n \geq 3}$, 
tending to $0$, such that, for all $n \geq 3$, 
the condition
$\beta \in (\theta_{n}^{-1} -\epsilon_n ,
\theta_{n}^{-1} + \epsilon_n)$ implies
that $f_{\beta}(z)$
has $1 + 2 \lfloor \frac{n}{6}\rfloor$
simple zeroes in $|z| < 1$,
each zero being
obtained from the roots $z_{j,n}$
of modulus $< 1$
of the trinomials
$-1 + z + z^{n}$,
by continuity with
$\beta$.
\end{theorem}

\begin{proof}
The roots of $G_{n}(z)$ are all simple, by Proposition
\ref{rootsdistrib}.
If $\beta$ is close enough to 
$\theta_{n}^{-1}$
the roots of $f_{\beta}(z)$ which lie 
in $|z| < 1$
are all simple
by Hurwitz's Theorem
(Corollary \ref{zeroesParryUpperfunctionContinuity}
and Theorem \ref{convergencecompactsetsUNITDISK}).
These roots are obtained by continuity
from those of $G_{n}(z)$.
\end{proof}

\begin{remark}
In Theorem \ref{parryupperfunctionPerronNeighbourhood}, 
$\beta$ is either a transcendental number or 
an algebraic number.
In both cases, a lenticulus of simple 
zeroes lies in the angular sector
$\arg(z) \in (-\pi/3 , +\pi/3)$, as a deformed lenticulus
of $\mathcal{L}_{\theta_{n}^{-1}}$.
In the sequel, we will reserve the notation
$\mathcal{L}_{\beta}$ for the zeroes 
of $f_{\beta}(z)$ identified as Galois conjugates of $\beta$,
when $\beta > 1$ is an algebraic integer,
zeroes close to
those of the lenticulus
$\mathcal{L}_{\theta_{\dyg(\beta)}^{-1}}$.
The difficulty of the identification of the zeroes
will be considered
in \S \ref{S5}.
\end{remark}

\subsection{The problem of the identification of the zeroes of $f_{\beta}(z)$ as conjugates of $\beta$}
\label{S4.6}

Hypothesis (H): let $\beta > 1$ be an
algebraic integer.
In this paragraph, let us make the assumption that
all the zeroes of $f_{\beta}(z)$ of modulus $< 1$
are conjugates of $\beta$, and that
all the conjugates of $\beta$ of modulus $< 1$
are zeroes of $f_{\beta}(z)$.

This assumption is very probably wrong. The
difficulty of the identification,
in $|z| < 1$, of 
the zero-locus of 
$f_{\beta}(z)$
with the set of zeroes of the minimal polynomial
$P_{\beta}(z)$ will be partially overcome
in $\S$ \ref{S5.4}
and $\S$ \ref{S6.2}; it will lead
to a new notion of continuity with
the ``house", of the minorant ${\rm M}_{r}$
of ${\rm M}$. 
Nevertheless, interestingly,
assumption (H) would
lead to the following claim, where the
continuity of ${\rm M}$ itself
would occur locally.

\begin{claim}
\label{LehmerCJTrue_simpleparrynumber}
Assuming (H) there would exist a sequence 
$(\epsilon_n)_{n \geq 6}$ of positive real numbers
tending to $0$
such that
any two successive intervals
$(\theta_{n}^{-1} -\epsilon_n , 
\theta_{n}^{-1} + \epsilon_n )$
and
$(\theta_{n+1}^{-1} -\epsilon_{n+1} , 
\theta_{n+1}^{-1} + \epsilon_{n+1} )$
are disjoint, $n \geq 6$, 
and the Mahler measure
$${\rm M}:
\bigcup_{n \geq 6} 
(\theta_{n}^{-1} -\epsilon_n , 
\theta_{n}^{-1} + \epsilon_n )
\cap \mathcal{O}_{\overline{\qb}}
~\to~ \pb,\quad
\beta \to {\rm M}(\beta)$$
be continuous and take values
${\rm M}(\beta) \geq \Theta$ for any algebraic integer $\beta$
in this set.
\end{claim}

\begin{proof} 
The existence of the sequence
$(\epsilon_n)_n$ comes from 
Theorem \ref{parryupperfunctionPerronNeighbourhood}. 
By Theorem
\ref{parryupperfunctionPerronNeighbourhood}
the number of zeroes of 
modulus $< 1$ of $f_{\beta}(z)$ is 
the same as that of
$f_{\theta_{n}^{-1}}(z)$ 
as soon as $\beta > 1$ is close enough to
a Perron number 
$\theta_{n}^{-1}$, $n \geq 3$.
By Corollary \ref{zeroesParryUpperfunctionContinuity}
each zero of $f_{\beta}(z)$ in
$|z|<1$ is 
a continuous function of $\beta$.
Assuming (H) these zeroes are conjugates of
$\beta$.
Then the Mahler measure M
is a continuous function on the
set of algebraic integers
\begin{equation}
\label{algebraicintegersPerronepsilon}
\bigcup_{n \geq 3}
(\theta_{n}^{-1} - \epsilon_n , 
\theta_{n}^{-1} + \epsilon_n)
~~\cap~~ \mathcal{O}_{\overline{\qb}} .
\end{equation}
Since $\lim_{n \to \infty} {\rm M}(\theta_{n}^{-1}) = \Lambda = 1.38135\ldots > \Theta = 1.3247\ldots$
by Theorem \ref{main1},
and that ${\rm M}(\theta_{n}^{-1})
> \Theta$ for $n \geq 6$ by
Proposition \ref{maincoro3}, 
it is possible to choose all
$\epsilon_n > 0$, for $n \geq 6$, small
enough to have
${\rm M}(\beta) \geq \Theta$ for any
$\beta$ belonging to the set \eqref{algebraicintegersPerronepsilon}.
\end{proof}

\subsection{The minoration of ${\rm M}(\beta)$ for $\beta$ a simple Parry number}
\label{S4.7} 
 
\begin{theorem}
\label{minorationSUBCOLLECTIONsimpleParrynbs}
If $1 < \beta < 2$ is a 
simple Parry number
for which the complementary factor
$\frac{P_{\beta,P}}{P_{\beta}}$ in the
Parry polynomial of $\beta$
is a product of cyclotomic polynomials, 
then
${\rm M}(\beta) \geq 
\Theta = \theta_{5}^{-1} = 1.3247\ldots$.
\end{theorem}

\begin{proof}
It is a consequence of Proposition
\ref{neversimpleParrySalem} and
Theorem \ref{notsimplesmallMahlermeasure}.
The equality 
${\rm M}(P_{\beta}) 
= {\rm M}(P_{\beta, P})$
means that the complementary factor
$\frac{P_{\beta, P}(X)}{P_{\beta}(X)}$
admits a Mahler measure
equal to 1, equivalently, since it is monic,
that all its roots are roots of unity
by Kronecker's theorem.
\end{proof}

How often is
the complementary factor
a product of cyclotomic polynomials?
A partial answer can be given by
a variant (Conjecture \ref{odlyzkopoonenCJvariant}
below)
of the Conjecture of Odlyzko and Poonen
\cite{odlyzkopoonen}
when 
$1 <\beta < 2$ is simple
and such that
$P_{\beta,P}(X)$
is irreducible, i.e. when the complementary factor is trivial.
In this case,
$$\{\beta \in (1, 2) \mid
\beta \mbox{~simple Parry number},~
P_{\beta,P}(X) = P_{\beta}(X)\} \qquad \mbox{would be dense}.$$

Let us recall the two Conjectures.

\begin{conjecture}[Odlyzko - Poonen]
\label{odlyzkopoonenCJ}
Let $\mathcal{P}_{d,+1}$ denote
the set of all polynomials of degree $d$
with constant term $1$
and with coefficients in $\{0,1\}$.
Denote $$\mathcal{P}_{+} 
= \bigcup_{d \geq 1} \mathcal{P}_{d,+1}.$$
Then, in $\mathcal{P}_{+}$, 
almost all 
polynomials are irreducible;
more precisely, 
if $\mathcal{I}_{d,+}$ denotes the number of 
irreducible polynomials in
$\mathcal{P}_{d,+1}$, then
$$\lim_{d \to \infty} \frac{\mathcal{I}_{d,+}}{2^{d-1}} = ~1.$$
\end{conjecture}

The best account of Conjecture \ref{odlyzkopoonenCJ} is given by
Konyagin (1999): $\mathcal{I}_{d,+1}
\gg \frac{2^{d}}{\lo d}$.
Now,
changing the last coefficient to $-1$ gives
the following variant.

\begin{conjecture}
\label{odlyzkopoonenCJvariant}
Let $\mathcal{P}_{d,-1}$ denote
the set of all polynomials of degree $d$
with constant term $-1$
and with coefficients in $\{0,1\}$.
Denote $\mathcal{P}_{-} 
= \bigcup_{d \geq 1} \mathcal{P}_{d,-1}$.
Then, in $\mathcal{P}_{-}$, 
almost all 
polynomials are irreducible;
more precisely, 
if $\mathcal{I}_{d,-}$ denotes the number of 
irreducible polynomials in
$\mathcal{P}_{d,-1}$, then
$$\lim_{d \to \infty} \frac{\mathcal{I}_{d,-}}{2^{d-1}} = ~1.$$
\end{conjecture}

Indeed, for any  simple Parry number
$\beta < 2$, 
the opposite $-P_{\beta,P}^{*}(X)$
of the reciprocal of the
Parry polynomial of $\beta$
has coefficients
in $\{0, 1\}$ except the constant term
equal to $-1$, as 
(from \eqref{parrypolynomesimple}):
$-P_{\beta, P}^{*}(X) =
-1 + t_1 X + t_2 X^2 + \ldots +t_m X^m$
to which
Conjecture \ref{odlyzkopoonenCJvariant}
applies. The irreducibility
of $P_{\beta, P}^{*}(X)$ is equivalent to that
of $P_{\beta, P}(X)$.

Consequently, assuming Conjecture
\ref{odlyzkopoonenCJvariant}, 
the minoration 
of the Mahler
measure ${\rm M}(\beta)$, 
with $1 < \beta < 2$ any algebraic integer,
by the smallest Pisot number
$\Theta$, appears as a  
common rule occurring almost
everywhere. Further
studies on
the crucial 
problem of the irreducibility of
integer polynomials with coefficients in a finite set 
were carried out by 
Borwein, Erd\'elyi and Littmann 
\cite{borweinerdelyilittmann},
Chern \cite{chern}, 
Dubickas \cite{dubickas4},
Dubickas \cite{dubickas10}.

\section{Asymptotic expansion of the lenticular minorant of the Mahler measure M$(\beta)$ for $\beta > 1$ a real algebraic integer close to one}
\label{S5}

\subsection{Asymptotic expansions of a real number $\beta > 1$ close to one and of the dynamical degree $\dyg(\beta)$}
\label{S5.1}

\begin{lemma}
\label{longueurintervalthetann}
Let $n \geq 6$. 
The difference
$\theta_{n} - \theta_{n-1} > 0$
admits the following
asymptotic expansion, reduced to its terminant:

\begin{equation}
\label{asymthetanthetan}
\theta_{n} - \theta_{n-1}  ~=~ \frac{1}{n}
O\left(
\left(
\frac{\lo \lo n}{\lo n}
\right)^2
\right),
\end{equation}
with the constant $1$ involved in $O\left( ~\right)$.
\end{lemma}

\begin{proof} 
From \eqref{DthetanExpression} and Lemma
\ref{remarkthetan}, we have
$$\theta_n = 1 - \frac{\lo n}{n}(1 - \lambda_n) +
\frac{1}{n} O\left(
\left(
\frac{\lo \lo n}{\lo n}
\right)^2 
\right)
$$
with
the constant $1/2$ involved in $O\left(~ \right)$, and
$$\lambda_n = \frac{\lo \lo n}{\lo n} \left( 
\frac{1}{1+ \frac{1}{\lo n}}\right) +
O\left( \frac{\lo \lo n}{n}  \right)$$
with the constant 1 in the Big O.
Then we deduce
$${\rm D}(\theta_n) - {\rm D}(\theta_{n-1}) = 
\frac{\lo n}{n^2} + O\left(\frac{\lo \lo n}{n^2} \right).$$
The real function $x^{-2} \lo x$ on $(1,+\infty)$
is decreasing for $x \geq \sqrt{e}$.
Hence 
the sequence 
$({\rm D}(\theta_n) - {\rm D}(\theta_{n-1}))$
is decreasing for $n$ large enough.
By Proposition \ref{closetoouane} 
$(\theta_n - \theta_{n-1})_n$
is already known to tend to $0$. 

Since ${\rm tl}(\theta_n) = 
\frac{1}{n} O\left(
\left(
\frac{\lo \lo n}{\lo n}
\right)^2
\right)$, we have
$$\theta_n - \theta_{n-1} = \left(\theta_n - {\rm D}(\theta_n)\right) + 
[{\rm D}(\theta_n) - {\rm D}(\theta_{n-1})]
- \left(\theta_{n-1} - {\rm D}(\theta_{n-1})\right)$$
$$= {\rm tl}(\theta_n) + 
\left(\frac{\lo n}{n^2} + O\left(\frac{\lo \lo n}{n^2} \right)\right) 
- {\rm tl}(\theta_{n-1})$$
\begin{equation}
\label{diffTH}
= \frac{1}{n} O\left(
\left(
\frac{\lo \lo n}{\lo n}
\right)^2
\right)
\end{equation}
where the constant involved in $O\left(~\right)$ is now 
$1 = 1/2 + 1/2$. Hence the claim. 
\end{proof}

\begin{theorem}
\label{betaAsymptoticExpression}
Let $n \geq 6$. Let $\beta > 1$ be a real
number of dynamical degree $\dyg(\beta) = n$.
Then $\beta$ can be expressed as: 
$\beta = {\rm D}(\beta) + {\rm tl}(\beta)$ with
${\rm D}(\beta) = 1 +$
\begin{equation}
\label{DbetaAsymptoticExpression}
\frac{\lo n}{n}
\left(
1 - \bigl(
\frac{n - \lo n}{n \, \lo n + n - \lo n}
\bigr)
\Bigl(
\lo \lo n - n 
\lo \Bigl(1 - \frac{\lo n}{n}\Bigr)
- {\rm Log} n
\Bigr)
\right)
\end{equation}
and
\begin{equation}
\label{tailbetaAsymptoticExpression}
{\rm tl}(\beta) ~=~ \frac{1}{n} \, O \left( \left(\frac{\lo \lo n}{\lo n}\right)^2 \right),
\end{equation}
with the constant $1$ involved in $O \left(~\right)$.
\end{theorem}

\begin{proof}
By definition  $\theta_{n} \leq \beta^{-1} < \theta_{n-1}$.
The
development term of $\beta^{-1}$ is
$D(\beta^{-1}) = D(\beta^{-1} - \theta_{n})+
D(\theta_{n})$,
with 
$|\beta^{-1} - \theta_{n}| < \theta_{n} -\theta_{n-1}$.
By Lemma \ref{longueurintervalthetann},
$D(\theta_{n} -\theta_{n-1}) = 0$.
Therefore
$\beta = D(\beta)
+{\rm tl}(\beta)$
is deduced from
$D(\theta_{n})$ in 
\eqref{DthetanExpression}
and from
$\beta^{-1} = D(\beta^{-1})
+{\rm tl}(\beta^{-1})$
with $-D(\beta^{-1})+1 =
D(\beta)-1$ given by 
\eqref{DbetaAsymptoticExpression} 
and ${\rm tl}(\beta^{-1}) 
= {\rm tl}(\beta)$
given by \eqref{tailbetaAsymptoticExpression}.
\end{proof}

\begin{theorem}
\label{nfonctionBETA} 
Let $\beta \in (1, \theta_{6}^{-1})$ be a real number. 
The asymptotic expansion 
of the locally constant function 
$n = {\rm dyg}(\beta)$, as a function of the variable 
$\beta -1$, 
is
\begin{equation}
\label{dygExpression}
n = - \frac{\lo (\beta - 1)}{\beta - 1}
\Bigl[1+
O\Bigl(
\Bigl(
\frac{\lo (-\lo (\beta - 1))}
{\lo (\beta - 1)}
\Bigr)^2
\Bigr)
\Bigr]
\end{equation}
with the constant 1 in $O(~)$.
\end{theorem}

\begin{proof}
Inverting \eqref{DbetaAsymptoticExpression} 
gives the asymptotic expansion of 
$n$ as a function of $\beta$:
from \eqref{DbetaAsymptoticExpression} 
readily comes
\begin{equation}
\label{dygAsymptoticExpansion}
n = \frac{\beta}{\beta - 1}
\lo \Bigl(\frac{\beta}{\beta - 1}\Bigr)
\Bigl[ 1+
O\Bigl(
\Bigl(
\frac{\lo \lo \Bigl(\frac{\beta}{\beta - 1}\Bigr)}
{\lo \bigl(\frac{\beta}{\beta - 1} \bigr)}
\Bigr)^2
\Bigr)
\Bigr]
\end{equation}
then \eqref{dygExpression} 
as $\beta \to 1$.
\end{proof}

\begin{remark}
\label{dygvalue}
If $\beta$ runs over the 
set of Perron numbers
$\theta_{n}^{-1}, \, n = 5, 6, \ldots, 12$,
and over 
the smallest
Parry - Salem numbers $\beta \leq 1.240726\ldots$
of Table 1, 
the dynamical degree of
$\beta$ (Table 1, Column 1) is the integer part
of $D(n)$ in \eqref{dygAsymptoticExpansion}:
\begin{equation}
\label{dygintegerpart}
\dyg(\beta) = \lfloor \frac{\beta}{\beta - 1}
\lo \Bigl(\frac{\beta}{\beta - 1}\Bigr) 
\rfloor.
\end{equation}
\end{remark}

\subsection{Fracturability of the minimal polynomial by the Parry Upper function}
\label{S5.2}

In this subsection $\beta > 1$ is assumed to be
an algebraic integer.

When $\beta > 1$ is an algebraic integer, it admits
a Parry Upper function $f_{\beta}(z)$ and
a minimal polynomial $P_{\beta}(z)$. 
In the following the 
interplay between both analytical functions are investigated
when $\beta > 1$ is close to one.

\begin{theorem}
\label{splitBETAdivisibility}
Let $\beta > 1$ be an algebraic integer such that
${\rm M}(\beta) < \Theta$.
The following formal decomposition of 
the (monic) minimal polynomial
\begin{equation}
\label{decompo}
P_{\beta}(X) = P_{\beta}^{*}(X)
=
U_{\beta}(X) \times
f_{\beta}(X),
\end{equation}
holds, as a product
of the Parry Upper function 
\begin{equation}
\label{formlacu}
f_{\beta}(X)= G_{\dyg(\beta)}(X)
+ X^{m_1} + X^{m_2} + X^{m_3}+ \ldots.
\end{equation}
with
$m_0:= \dyg(\beta)$,
$m_{q+1} - m_q \geq \dyg(\beta)-1$ for $q \geq 0$,
and the invertible formal series 
$U_{\beta}(X) \in \zb[[X]]$, 
quotient of $P_{\beta}$ by
$f_{\beta}$.
The specialization $X \to z$ of the formal variable
to the complex variable leads to
the identity between analytic functions,
obeying the Carlson-Polya dichotomy:
\begin{equation}
\label{decompozzz}
P_{\beta}(z) = U_{\beta}(z) \times
f_{\beta}(z)
\quad
\left\{
\begin{array}{cc}
\mbox{on}~ \cb & 
\mbox{if $\beta$ is a Parry number, with}\\
& U_{\beta} ~\mbox{and}~ f_{\beta}~ \mbox{both meromorphic},\\
&\\
\mbox{on}~ |z| < 1 & 
\mbox{if~} \beta \mbox{~is a
nonParry number, with $|z|=1$}\\
& \mbox{as natural boundary for both $U_{\beta}$ and $f_{\beta}$.}
\end{array}
\right.
\end{equation}
In both cases, the domain of holomorphy
of the function $U_{\beta}(z)$   
contains 
the open  
disc 
$D(0, \theta_{\dyg(\beta) - 1})$. 
\end{theorem}

\begin{proof}
The algebraic integer $\beta$
lies between two successive Perron
numbers of the family
$(\theta_{n}^{-1})_{n \geq 5}$, as
$
\theta_{n}^{-1} \leq \beta
< \theta_{n-1}^{-1}, \,
\dyg(\beta) = n \geq 6$. 
By Proposition \ref{zeronzeron}, 
the Parry Upper function
$f_{\beta}(z)$ at 
$\beta$
has the form \eqref{formlacu}.
The algebraic integer $\beta$ is a Parry number or 
a nonParry number. 
In both cases, 
$f_{\beta}(\beta^{-1}) = 0$.
If
$f_{\beta}(z) = - 1 + \sum_{j \geq 1} t_j z^j$, 
the zero
$\beta^{-1}$ of $f_{\beta}(z)$
is simple since
the derivative of $f_{\beta}(z)$
satisfies
$f'_{\beta}(\beta^{-1}) 
= \sum_{j \geq 1} j \, t_j \, \beta^{-j+1} > 0$.
The other zeroes of
$f_{\beta}(z)$ of modulus $< 1$
lie in
$1/\beta \leq |z| < 1$. 
Therefore the poles, if any,
of $U_{\beta}(z) = P_{\beta}(z)/f_{\beta}(z)$
of modulus $< 1$
all lie in the annular region
$\theta_{\dyg(\beta)-1} < |z| < 1$.

The formal decomposition
\eqref{decompo}, in
$\zb[[X]]$, is always possible. 
Indeed, if
we put
$U_{\beta}(X) = 
-1 + \sum_{j \geq 1} b_j X^j$, and
$P_{\beta}(X) = 1 + a_1 X + a_2 X^2 + \ldots
a_{d-1} X^{d-1}
+ X^d $, (with $a_j = a_{d-j}$),
the formal identity 
$P_{\beta}(X) = U_{\beta}(X) \times 
f_{\beta}(X)$
leads to the existence of the coefficient 
vector $(b_j)_{j \geq 1}$ of
$U_{\beta}(X)$, as a function
of $(t_j)_{j \geq 1}$ and
$(a_i)_{i = 1,\ldots, d-1}$, as:
$b_1 = -(a_1 + t_1)$,
and, for $r = 2, \ldots, d-1$,
\begin{equation}
\label{bedeRecurrencedebut}
b_r = -(t_r + a_r - 
\sum_{j=1}^{r-1}b_j t_{r-j}) 
\quad
\mbox{with} \quad
b_d = -(t_d + 1 - 
\sum_{j=1}^{d-1} b_j t_{r-j}),
\end{equation}
\begin{equation}
\label{bedeRecurrence}
b_r = -t_r + 
\sum_{j=1}^{r-1} b_j t_{r-j} \quad \mbox{for}~~  r > d.
\end{equation}
Then
$b_j \in \zb, \, j \geq 1$;
the integers
$b_r , r > d$,
are determined recursively
by \eqref{bedeRecurrence} 
by
the sequence $(t_i)$ and
from the finite subset
$\{b_0 = -1, b_1 , b_2 , \ldots , b_d \}$,
itself determined from $P_{\beta}(X)$
using \eqref{bedeRecurrencedebut}.
They inherit the asymptotic
properties
of the asymptotic lacunarity 
of $(t_i)$ when $r$ is very large
\cite{vergergaugry}.
If $R_{\beta}$ denotes the radius 
of convergence of $U_{\beta}(z)$
the inequality 
$R_{\beta}
\geq 
\theta_{\dyg(\beta)-1}$ 
can be directly obtained
using Hadamard's formula
$R_{\beta}^{-1}
= \limsup_{r \to \infty} \, |b_r|^{1/r}$
and the following Lemma \ref{bound_br_exp} 
(in which $n = \dyg(\beta)$) whose proof
is immediate.

\begin{lemma}
\label{bound_br_exp}
Let $\epsilon > 0$ such that
$\theta_{n-1}^{-1} < \exp(\epsilon)$. 
There exists a constant $C = C(\epsilon) \geq 
\max\{1, \exp(\epsilon (n-1)) - 1\}$
such that:
\begin{equation}
\label{br_expgeneral}
|b_r| \leq C \times \exp(\epsilon \, r),
\qquad \quad \mbox{for all}~ r \geq 0.
\end{equation}
\end{lemma}

\end{proof}

As an example, 
let us consider 
$\beta = 1.291741\ldots$ the Salem number
of degree 24 given in Table 1,
with 
$\dyg(\beta) = 6$.
Since its Parry polynomial
is irreducible 
of degree 24 it is equal to 
the minimal polynomial of $\beta$:
$P_{\beta,P} = P_{\beta} = P_{\beta,P}^{*}$. 
Since $\beta$ 
is not a simple Parry number, 
and has preperiod length 1,
Theorem \ref{carlsonpolya} implies that
$U_{\beta}(z) = -(1 - z^{23})$; in this case
the radius of convergence $R_{\beta}$
of $U_{\beta}(z)$ is infinite. 
Any other Salem 
number $\beta$ of Table 1 
has beta-conjugates $\gamma$ ; then,
for these Salem numbers,
by Theorem \ref{carlsonpolya}, 
$U_{\beta}(z)$ is meromorphic in $\cb$ and 
holomorphic 
in the open disk $D(0, R_{\beta})$
with
$R_{\beta} = 
\min \left\{1, \min \bigl\{ |\gamma|^{-1} 
\mid \gamma ~\mbox{beta-conjugate of}~ \beta\bigr\}\right\}$.

\begin{definition}
\label{fracturabilityminimalpolynomialDEF}
Let $\beta > 1$ be an algebraic integer such that
${\rm M}(\beta) < \Theta$. The minimal polynomial
$P_{\beta}(X)$ is said to be fracturable if
the power series $U_{\beta}(z)$ in
\eqref{decompozzz} is not reduced to a constant.
\end{definition}

The fracturability of the minimal polynomial
$P_{\beta}$ of $\beta > 1$ close to one 
is very often the rule,
at small Mahler measure. In the theory of 
divisibility of polynomials over a field, in
Commutative Algebra,
the minimal polynomials are the irreducible elements
on which the (classical) 
theory of ideals is constructed. These 
irreducible elements are now canonically 
fracturable.
This type of fracturability of
the minimal polynomials implies a new
theory of ideals in extensions of
rings of polynomials
when $\beta > 1$ is close to one.
The author will develop it further later.

\subsection{A lenticulus of zeroes of $f_{\beta}(z)$ in the cusp of Solomyak's fractal}
\label{S5.3}

In this subsection $\beta \in (1, \theta_{6}^{-1})$  
is assumed to be a real number 
(algebraic or transcendental) such that
$\beta \not\in
\{\theta_{n}^{-1} \mid n \geq 6\}$. 
In Theorem \ref{omegajnexistence} 
it will be proved that, to such a $\beta$, 
is associated a lenticulus of zeroes 
of $f_{\beta}(z)$ in the cusp
of Solomyak's fractal $\mathcal{G}$,
complementing Theorem 
\ref{solomyakfractalCOUCHEn}.

Lenticuli of zeroes of $f_{\beta}(z)$
were already proved to exist when
$\beta \in
\{\theta_{n}^{-1} \mid n \geq 12\}$. 
They were used to give a direct
proof of the
Conjecture of Lehmer for 
$\{\theta_{n}^{-1} \mid n \geq 2\}$ in 
\cite{vergergaugry6}.
In the sequel we consider the remaining cases
of $\beta$s, including all algebraic
integers $\beta > 1$.

The method which will be used
to
detect the lenticuli of zeroes of
$f_{\beta}(z)$
is the method of  Rouch\'e. 
This method will be shown to be powerful
enough to reach relevant minorants 
of the Mahler measure
M$(\beta)$ for $\beta > 1$ any 
algebraic integer ($\S$ \ref{S5.4}).

Let $n :=\dyg(\beta)$.
The algebraic integers 
$z_{j,n}, 1 \leq j < \lfloor n/6 \rfloor$,
which constitute the lenticulus
$\lc_{\theta_{n}^{-1}}$ in the upper 
Poincar\'e half-plane
satisfy
(\S \ref{S4.4}):
$$f_{\theta_{n}^{-1}}(\theta_{n}) 
= 
f_{\theta_{n}^{-1}}(z_{1,n}) 
=
f_{\theta_{n}^{-1}}(z_{2,n})
=
f_{\theta_{n}^{-1}}(z_{3,n})
=\ldots = 
f_{\theta_{n}^{-1}}(z_{\lfloor n/6 \rfloor, n})
=
0
,$$ 
with $f_{\theta_{n}^{-1}}(z) 
= -1 + z +z^n$.
\,The Parry Upper function at $\beta$ 
is 
characterized by the sequence 
of exponents $(m_q)_{q \geq 0}$:
\begin{equation}
\label{fbet}
f_{\beta}(z) = -1 + z + z^n + z^{m_1} + z^{m_2}
+ z^{m_3} + \ldots = G_{n}(z) +
\sum_{q \geq 1} z^{m_q} ,
\end{equation}
where $m_0 := n$, 
with the fundamental
minimal gappiness condition:
\begin{equation}
\label{minimalgappiness}
m_{q+1} - m_q \geq n-1 \qquad 
\mbox{for all}~ q \geq 0.
\end{equation}

The R\'enyi $\beta$-expansion 
$d_{\beta}(1)$ of 1
is infinite or not, namely 
the sequence of exponents
$(m_{q})_{q \geq 0}$
is either infinite or finite:
if it is infinite
the integers $m_{q}$
never take  
the value
$+\infty$;
if not
the power series
$f_{\beta}(z)$ is a polynomial
of degree $m_{q}$ for some integer
$m_{q} , q \geq 2$. In both cases,
the integer $m_1 \geq 2 n - 1$ is finite.

We will compute
real 
numbers $t_{j,n} \in (0,1)$ 
such that the small circles
$C_{j,n} := \{ z \mid |z - z_{j,n}|= 
\frac{t_{j,n}}{n}\}$
of respective centers $z_{j,n}$,
$|z_{j,n}| < 1$, 
all satisfy the Rouch\'e conditions:
\begin{equation}
\label{rourouche}
\left|f_{\beta}(z) - G_{n}(z)\right| =
\left|\sum_{q \geq 1} z^{m_q} \right|
< 
\left| G_{n}(z) \right|
\qquad 
\mbox{for} ~z \in C_{j,n},
~\mbox{for}~ j = 1, 2, \ldots, J_n ,
\end{equation}
are pairwise disjoint,
are small enough to
avoid to intersect $|z| = 1$, 
with $J_n \leq \lfloor \frac{n}{6}\rfloor$
the largest possible integer 
(in the sense of Definition
\ref{Jndefinition} 
and Proposition \ref{argumentlastrootJn}).
As a consequence, the number of zeroes
of $f_{\beta}(z)$ and
$G_{n}(z)$ in the open disc
$D_{j,n} := \{ z \mid |z - z_{j,n}|< 
\frac{t_{j,n}}{n}\}$
will be equal, implying
the existence of a simple zero
of the Parry Upper function
$f_{\beta}(z)$ in each disc
$D_{j,n}$.
The maximality of $J_n$ means that
the conditions of Rouch\'e cannot be satisfied
as soon as 
$J_n < j \leq \lfloor \frac{n}{6}\rfloor$
for the reason that the circles 
$C_{j,n}$ are too close to 
$|z|=1$.

The values $t_{j,n}$ are
necessarily smaller than $\pi$ 
in order to  avoid any overlap
between two successive circles
$C_{j,n}$ and $C_{j+1,n}$. Indeed,
since the argument 
$\arg z_{j,n}$ of the $j$-th root
$z_{j,n}$ is roughly equal to
$2 \pi j/n$ 
(Proposition \ref{zedeJIargumentsORDRE1}), 
the distance  $|z_{j,n} - 
z_{j+1,n}|$ is approximately
$2 \pi /n$.

The problem of the choice of the radius
$t_{j,n}/n$ is
a true problem. On one hand, a too small
radius would lead to make impossible
the application of the Rouch\'e conditions,
in particular
for those discs $C_{j,n}$ located
very near the unit circle. Indeed, we do 
not know a priori whether the unit circle is 
a natural boundary or not
for $f_{\beta}(z)$;
locating zeroes close to a natural boundary 
is a difficult problem in general.
On the other hand,
taking larger values of $t_{j,n}$ 
readily leads to
a bad localization of the zeroes of
$f_{\beta}(z)$, and hence, for
algebraic integers $\beta > 1$, to a trivial
minoration of the Mahler measure
${\rm M}(\beta)$. The sequel reports
a compromise, after many trials of the 
author, which works ($\S$ \ref{S5.4}).

For any
real number
$\beta \in (1, \theta_{6}^{-1})$  
such that
$\beta \not\in
\{\theta_{n}^{-1} \mid n \geq 6\}$
let us denote by
$\omega_{j,n} \in D_{j,n}$
the simple zero of
$f_{\beta}(z)$; then 
$|\omega_{j,n}|<1$ and
\begin{equation}
\label{localisationZEROomegajn}
|\omega_{j,n}| ~\leq~ 
|z_{j,n}| 
+ \frac{t_{j,n}}{n}
\qquad 
\mbox{with}
\quad
z_{j,n} ~\neq~
\omega_{j,n},\quad j = 1, 2, \ldots, J_n ;
\end{equation}
if, in addition, $\beta > 1$
is an algebraic integer, the strategy
for obtaining a minorant
of ${\rm M}(\beta)$ will be 
the following: to
identify the 
zeroes
$\omega_{j,n}$ as 
roots of the minimal polynomial
$P_{\beta}^{*}(z)= P_{\beta}(z)$, 
then to obtain 
a lower bound of the Mahler measure
${\rm M}(\beta)$
(will be made explicit 
in \S \ref{S5.4}) from these roots by
\begin{equation}
\label{minoMM}
\beta \times \prod
\, |\omega_{j,n}|^{-2}
~~\geq~~
\theta_{n}^{-1} \times
\prod_{j} \, (|z_{j,n}| 
+ \frac{t_{j,n}}{n})^{-2} ,
\end{equation}
where $j$ runs over  
$\{1, 2, \ldots, J_n
\}$.

In general, 
for
any real number
$\beta \in (1, \theta_{6}^{-1})$
such that
$\beta \not\in
\{\theta_{n}^{-1} \mid n \geq 6\}$,
the quantities $t_{j,n}$ will be 
estimated by the following inequalities:
\begin{equation}
\label{rourouchesimple}
\frac{|z|^{2 (n - 1)+1}}{1-|z|^{n-1}} =
\frac{|z|^{2n - 1}}{1-|z|^{n-1}}
~<~ 
\left| G_{n}(z) \right|
\quad 
\mbox{for} ~z \in C_{j,n},
\quad j = 1, 2, \ldots, J_n
\end{equation}
instead of \eqref{rourouche}, 
too complicated to handle.
In \eqref{rourouchesimple} 
the exponent
``$n-1$" comes from the minimal
gappiness condition \eqref{minimalgappiness}, 
that is from the
dynamical degree $n$
of $\beta$ 
itself, as unique variable.
Indeed,
due to the great variety of possible 
infinite admissible 
sequences $(m_q)_{q \geq 1}$ in
the power series $f_{\beta}(z)$ 
in \eqref{fbet},
for which
$m_q \geq q (n-1) + n$ for all $q \geq 1$,
we will proceed by taking 
the upper bound condition
\eqref{rourouchesimple} 
which comes from the general inequality:
$$\left|f_{\beta}(z) - G_{n}(z)\right|
 =
\left|\sum_{q \geq 1} z^{m_q} \right|
\leq \sum_{q \geq 1}   | z^{m_q}|
\leq \frac{|z|^{2n - 1}}{1-|z|^{n-1}} ,
\qquad \quad |z| < 1.
$$

The radius $t_{0,n}/n$ of the first circle 
$C_{0,n} := \{ z \mid |z - \theta_{n}| 
= \frac{t_{0,n}}{n}\}$, 
which contains 
$\beta^{-1}$, is readily obtained without
the method of Rouch\'e.

\begin{lemma}
\label{tzero}
Let $n \geq 7$.
$$t_{0,n} := 
\left(
\frac{\lo \lo n}{\lo n}
\right)^2 .
$$
\end{lemma}
\begin{proof} 
Since $\beta^{-1}$ runs over
the open interval
$(\theta_{n-1}, \theta_{n})$, this
interval
$(\theta_{n-1}, \theta_{n})$
is necessarily completely included in
$D_{0,n}$, and
the radius
of $C_{0,n}$ is 
$\theta_{n} - \theta_{n-1}$.
We deduce the result 
from Lemma \ref{longueurintervalthetann}. 
From Proposition \ref{zjjnnExpression} 
the root
$z_{1,n}$ admits 
$\Im(z_{1,n}) =
\frac{2 \pi}{n} (1 - \frac{1}{\lo n} + \ldots)$
as imaginary part. Then, 
for any 
$t_{1,n} \in (0,1)$, 
the circle $C_{0,n}$, of radius
$t_{0,n}/n$, and 
$C_{1,n}$ are
disjoint and do not intersect $|z|=1$.
\end{proof}

By Proposition \ref{rootsdistrib}
the only angular sector to be considered
for the roots $z_{j,n}$ of $G_n$
and the
Rouch\'e circles
$C_{j,n}$, up to complex-conjugation,
is $0 \leq \arg(z) \leq +\frac{\pi}{3}$.
In this sector the ``bump" angular sector,
$\arg z \in 
(0, 2 \pi (\lo n)/n)$
(cf Appendix; Remark 3.3 in \cite{vergergaugry6}), 
will be shown to contribute 
negligibly.

The existence of the roots $\omega_{j,n}$
in the main subsector  
is proved in Theorem \ref{cercleoptiMM}, 
then in Theorem 
\ref{omegajnexistence} in a refined version.
Proposition \ref{cercleoptiMMBUMP}
completes the proof of their existence
in the bump angular sector.
In the complement
of the family of the
adjustable Rouch\'e discs
Theorem \ref{absencezeroesOutside} 
asserts the existence of
a zerofree region
depending upon 
the dynamical degree of $\beta$.

In the following
we consider the problem
of the parametrization of the radii
$t_{j,n}/n$ by a unique 
real number $a \geq 1$,
allowing to adjust continuously and uniformly
the size of each circle $C_{j,n}$.
We solve it by finding an optimal value.

\begin{theorem}
\label{cercleoptiMM}
Let $n \geq n_1 = 195$,
$a \geq 1$, and
$j \in \{\lceil v_n \rceil, 
\lceil v_n \rceil + 1, \ldots, \lfloor n/6 \rfloor\}$ . 
Denote by $C_{j,n}
:= \{z \mid |z-z_{j,n}| = \frac{t_{j,n}}{n} \}$
the circle centered at the $j$-th root
$z_{j,n}$ of $-1 + X + X^n$, with
$t_{j,n} = \frac{\pi |z_{j,n}|}{a}$. 
Then the 
condition of Rouch\'e 
\begin{equation}
\label{rouchecercle}
\frac{
\left|z\right|^{2 n -1}}{1 - |z|^{n-1}}
~<~
\left|-1 + z + z^n \right| , 
\qquad \mbox{for all}~ z \in C_{j,n},
\end{equation}
holds true on the circle
$C_{j,n}$
for which the center $z_{j,n}$
satisfies
\begin{equation}
\label{petitsecteur18}
\frac{|-1+z_{j,n}|}{|z_{j,n}|} 
~~<~~ \frac{
1 -
\exp\bigl(
\frac{- \pi}{a}
\bigr)
}{2 \exp\bigl(
\frac{\pi}{a}
\bigr) -1} .
\end{equation}
The condition $n \geq 195$
ensures the existence of such roots $z_{j,n}$.
Taking the value $a = a_{\max}
= 5.87433\ldots$
for which the upper bound of
\eqref{petitsecteur18} 
is maximal, equal to 
$0.171573\ldots$, the roots
$z_{j,n}$ which satisfy 
\eqref{petitsecteur18}
all belong to the angular sector,
independent of $n$:
\begin{equation}
\label{angularsectoramax}
\arg(z) ~\in ~\bigl[\, 0, 
+ \frac{\pi}{18.2880}  \,\bigr] .
\end{equation}
For any real number $\beta > 1$ having 
$\dyg(\beta) = n$,
$f_{\beta}(z)$ admits a
simple
zero $\omega_{j,n}$ in 
$D_{j,n}$
for which the center $z_{j,n}$
satisfies 
\eqref{petitsecteur18}
with $a=a_{{\max}}$, and $j$ in the range
$ \{\lceil v_n \rceil, 
\lceil v_n \rceil + 1, \ldots, \lfloor n/6 \rfloor\}$. 
\end{theorem}

\begin{proof}
Denote by $\varphi := \arg (z_{j,n})$
the argument of the $j$-th root
$z_{j,n}$. 
Since $-1 + z_{j,n} + z_{j,n}^n = 0$, 
we have $|z_{j,n}|^n = |-1 + z_{j,n}|$.
Let us write $z= z_{j,n}+ \frac{t_{j,n}}{n} e^{i \psi}
=
z_{j,n}(1 + \frac{\pi}{a \, n} e^{i (\psi - \varphi)})$
the generic element belonging to $C_{j,n}$, with
$\psi \in [0, 2 \pi]$.
Let $X := \cos(\psi - \varphi)$.
Let us show that if the inequality
\eqref{rouchecercle} of Rouch\'e 
holds true for $X =+1$, for a certain circle
$C_{j,n}$,
then it holds true
for all $X \in [-1,+1]$, that is for 
every argument $\psi \in [0, 2 \pi]$,
i.e. for every 
$z \in C_{j,n}$.
Let us show 
$$
\left|1 + \frac{\pi}{a \,n} e^{i (\psi - \varphi)}
\right|^{n}
=
\exp\Bigl(
\frac{\pi \, X}{a}\Bigr)
\times 
\left(
1 - \frac{\pi^2}{2 a^2 \, n} (2 X^2 -1) 
+O(\frac{1}{n^2})
\right)
$$
and
$$
\arg\left(
\Bigl(1 + \frac{\pi}{a \, n} e^{i (\psi - \varphi)}
\Bigr)^{n}\right)
=
sgn(\sin(\psi - \varphi))
\times
\left( ~\frac{\pi \, \sqrt{1-X^2}}{a}
[1 -
\frac{\pi \, X}{a \, n}
]
+O(\frac{1}{n^2})
\right)
.$$
Indeed, since 
$\sin(\psi - \varphi) = \pm \sqrt{1 - X^2}$,
then
$$
\Bigl(1 + \frac{\pi}{a \, n} e^{i (\psi - \varphi)}
\Bigr)^{n}
=
\exp\left(
n \, \lo (1 + \frac{\pi}{a \, n} e^{i (\psi - \varphi))})
\right)
$$
$$=
\exp\left(
\frac{\pi}{a} (
X \pm i \sqrt{1-X^2}
)
+\left[- \frac{n}{2} (\frac{\pi}{a \, n} (
X \pm i \sqrt{1-X^2}
))^2
+ O(\frac{1}{n^2})
\right]
\right)
$$
$$=
\exp\Bigl(
\frac{\pi \, X}{a}
- \frac{\pi^2}{2 a^2 \, n} (2 X^2 -1) 
+O(\frac{1}{n^2})\Bigr)
\times
\exp\left(
\pm \, i ~\frac{\pi \, \sqrt{1-X^2}}{a}
[1 -
\frac{\pi \, X}{a \,n}
]
+O(\frac{1}{n^2})
\right)
. $$
Moreover,
$$
\left|1 + \frac{\pi}{a \, n} 
e^{i (\psi - \varphi)}
\right|
=
\left|1 + \frac{\pi}{a \, n} 
(X \pm i \sqrt{1-X^2})
\right|
=1 + \frac{\pi \, X}{a \, n} + O(\frac{1}{n^2}).
$$
with
$$\arg(1 + \frac{\pi}{a \, n} e^{i (\psi - \varphi)})
= 
sgn(\sin(\psi - \varphi)) \times
\frac{\pi \sqrt{1 - X^2}}{a \, n} 
+ O(\frac{1}{n^2}).
$$
For all $n \geq 18$ (Proposition 3.5
in \cite{vergergaugry6}), 
let us recall that
\begin{equation}
\label{devopomain}
|z_{j,n}|
=
1 + \frac{1}{n} \lo (2 \sin \frac{\pi j}{n})
+ \frac{1}{n} O \left(
\frac{\lo \lo n}{\lo n}\right)^2 .
\end{equation}
Then the left-hand side term of \eqref{rouchecercle}
is
$$\frac{
\left|z\right|^{2 n -1}}{1 - |z|^{n-1}}
=
\frac{|-1 + z_{j,n}|^2 
\left|1 + \frac{\pi}{a \, n} e^{i (\psi - \varphi)}
\right|^{2 n}}
{|z_{j,n}| \, 
\left|1 + \frac{\pi}{a \, n} e^{i (\psi - \varphi)}\right|
-
|-1 + z_{j,n}| \,
\left|1 + \frac{\pi}{a \, n} e^{i (\psi - \varphi)}\right|^{n}}$$

\begin{equation}
\label{rouchegauche}
=
\frac{|-1 + z_{j,n}|^2 
\left(
1 - \frac{\pi^2}{a \, n} (2 X^2 -1) 
\right)
\exp\bigl(
\frac{2 \pi \, X}{a}\bigr)
}
{\bigl(\,
1 + \frac{1}{n} \lo (2 \sin \frac{\pi j}{n})
+ \frac{\pi \, X}{a \, n} 
\bigr)
-
|-1 + z_{j,n}| \,
\left(
1 - \frac{\pi^2}{2 a \, n} (2 X^2 -1) 
\right)
\exp(
\frac{\pi \, X}{a})
}
\end{equation}
up to
$\frac{1}{n} 
O \left(
\frac{\lo \lo n}{\lo n}
\right)^2$
-terms (in the terminant).
The right-hand side term of 
\eqref{rouchecercle} is 
$$\left|-1 + z + z^n \right|
=
\left|
-1 + z_{j,n}\Bigl(1 + \frac{\pi}{n \, a} e^{i (\psi - \varphi)}\Bigr) +
z_{j,n}^{n}
\Bigl(1 + \frac{\pi}{n \, a} e^{i (\psi - \varphi)}
\Bigr)^n
\right|
$$
$$=
\left|
-1 + z_{j,n}
(1 \pm i \frac{\pi \sqrt{1 - X^2}}{a \, n})
(1 + \frac{\pi \, X}{a \, n}) +
\hspace{5cm} \mbox{} \right.
$$
\begin{equation}
\label{rouchedroite}
\left.
\mbox{} \hspace{-0.2cm}
(1 - z_{j,n})
\bigl(
1 - \frac{\pi^2}{2 a^2 \, n} (2 X^2 -1) 
\bigr)
\exp\bigl(
\frac{\pi \, X}{a}
\bigr) \,
\exp\Bigl(
\pm \,
i \,
\Bigl( ~\frac{\pi \, \sqrt{1-X^2}}{a}
[1 - \frac{\pi \, X}{a \, n}] 
\Bigr)
\Bigr)
\!+ \!O(\frac{1}{n^2})
\right|
\end{equation}

Let us consider
\eqref{rouchegauche}
and
\eqref{rouchedroite}
at the first order for the asymptotic expansions, 
i.e. up to $O(1/n)$ - terms instead of
up to 
$O(\frac{1}{n}(\lo \lo n/ \lo n)^2)$ - terms or
$O(1/n^2)$ - terms.
\eqref{rouchegauche} becomes:
$$\frac{|-1+z_{j,n}|^2 \exp(\frac{2 \pi X}{a})}
{|z_{j,n}| - |-1+z_{j,n}| \exp(\frac{\pi X}{a})}$$
and \eqref{rouchedroite} is equal to:
$$|-1 + z_{j,n}|
\left|
1 -
\exp\bigl(
\frac{\pi \, X}{a}
\bigr) \,
\exp\Bigl(
\pm \,
i \,
\frac{\pi \, \sqrt{1-X^2}}{a} 
\Bigr)
\right|
$$
and is independent of the sign of 
$\sin(\psi - \varphi)$.
Then
the inequality \eqref{rouchecercle} is 
equivalent to
\begin{equation}
\label{roucheequiv1}
\frac{|-1+z_{j,n}|^2 \exp(\frac{2 \pi X}{a})}
{|z_{j,n}| - |-1+z_{j,n}| \exp(\frac{\pi X}{a})}
<
|-1+z_{j,n}| \, 
\left|
1 -
\exp\bigl(
\frac{\pi \, X}{a}
\bigr) \,
\exp\Bigl(
\pm \,
i \,
\frac{\pi \, \sqrt{1-X^2}}{a} 
\Bigr)
\right|
,
\end{equation}
and \eqref{roucheequiv1} to
\begin{equation}
\label{amaximalfunctionX}
\frac{|-1 + z_{j,n}|}{|z_{j,n}|}
~  < ~ \, 
\frac{\left|
1 -
\exp\bigl(
\frac{\pi \, X}{a}
\bigr) \,
\exp\Bigl(
 \,
i \,
\frac{\pi \, \sqrt{1-X^2}}{a} 
\Bigr)
\right|
\exp\bigl(
\frac{-\pi \, X}{a}
\bigr)}{\exp\bigl(
\frac{\pi \, X}{a}
\bigr) +\left|
1 -
\exp\bigl(
\frac{\pi \, X}{a}
\bigr) \,
\exp\Bigl(
 \,
i \,
\frac{\pi \, \sqrt{1-X^2}}{a} 
\Bigr)
\right|} =: \kappa(X,a).
\end{equation}

Denote by
$\kappa(X,a)$ the right-hand side term, 
as a function of $(X, a)$, 
on $[-1, +1] \times [1, +\infty)$.
It is routine to show that, for any
fixed $a$,
the partial derivative $\partial \kappa_X$
of $\kappa(X,a)$ with respect to $X$ 
is strictly negative
on the interior of the domain. 
The function $x \to 
\kappa(x,a)$ takes its minimum
at $X=1$, and this minimum is always 
strictly positive. 
Hence the inequality of Rouch\'e
\eqref{rouchecercle} will be satisfied
on $C_{j,n}$ once it is 
satisfied at $X = 1$.

For which range of values of $j/n$? 
Up to
$O(1/n)$-terms in 
\eqref{amaximalfunctionX},
it is given by the set of integers
$j$ for which $z_{j,n}$
satisfies:
\begin{equation}
\label{amaximalfunction}
\frac{|-1 + z_{j,n}|}{|z_{j,n}|} 
< \kappa(1,a) 
=
\frac{\left|
1 -
\exp\bigl(
\frac{\pi}{a}
\bigr)
\right|
\exp\bigl(
\frac{-\pi}{a}
\bigr)}{\exp\bigl(
\frac{\pi}{a}
\bigr) +\left|
1 -
\exp\bigl(
\frac{\pi}{a}
\bigr)
\right|} .
\end{equation}
In order to take into account
a collection of roots of $z_{j,n}$
as large as possible, i.e.
in order to have a minorant of the Mahler measure
${\rm M}(\beta)$ the largest possible,
the value of $a \geq 1$ has to be chosen
such that $a \to \kappa(1,a)$ is maximal
in \eqref{amaximalfunction}.

\begin{figure}
\begin{center}
\includegraphics[width=6cm]{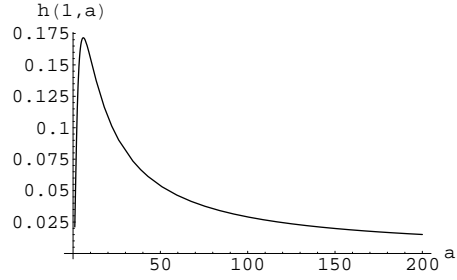}
\end{center}
\caption{
Curve of the Rouch\'e condition $a \to \kappa(1,a)$ 
(upper bound in \eqref{petitsecteur18}),
for the circles
$C_{j,n} = \{z \mid |z-z_{j,n}| =
\pi |z_{j,n}|/(a \, n)\}$
centered at the zeroes $z_{j,n}$
of the trinomial $-1 + X + X^n$,
as a function of
the size of the circles
$C_{j,n}$ parametrized by
the adjustable real number $a \geq 1$.}
\label{h1a}
\end{figure}

The function $a \to \kappa(1,a)$ 
reaches its maximum
$\kappa(1, a_{\max}) := 0.171573\ldots$
at $a_{\max} $
$= 5.8743\ldots$
(Figure \ref{h1a}).
Denote by $J_n$ the maximal integer $j$
for which
\eqref{amaximalfunction} is satisfied
and in which $a$ is taken equal to 
$a_{\max}$ 
(Definition \ref{Jndefinition} 
and Proposition \ref{argumentlastrootJn}).
From Proposition \ref{argumentlastrootJn},
in which are reported the asymptotic expansions
of $J_n$ and $\arg(z_{J_n , n})$,
we deduce 
\begin{equation}
\label{gammeindexj} 
\arg(z_{j,n}) < \frac{\pi}{18.2880\ldots} =
0.171784\ldots
\quad
\mbox{for} ~
j= \lceil v_n \rceil, \lceil v_n \rceil + 1, \ldots,
J_n .
\end{equation}
\begin{remark}
\label{value195}
The minimal value $n_1 = 195$
is calculated by the condition 
$2 \pi \frac{v_{n}}{n}
<
\frac{\pi}{18.2880\ldots} 
$
$= 0.171784\ldots $,
for all
$n \geq n_1$,
for having a strict inclusion,
of the ``bump sector" inside
the angular sector defined by
the maximal opening angle
$0.171784\ldots$
(cf Appendix for the sequence $(v_n)$)
\end{remark}
This finishes the proof.
\end{proof}

Let us calculate the argument of the last root
$z_{j,n}$  for which 
\eqref{amaximalfunctionX} 
is an equality with $X=1$.

\begin{definition}
\label{Jndefinition}
Let $n \geq 195$.
Denote by $J_n$ the largest integer
$j \geq 1$ such that
the root
$z_{j, n}$ of $G_n$ 
satisfies
\begin{equation}
\label{JJJn}
\frac{|-1 + z_{j ,n}|}{|z_{j ,n}|}
~  \leq ~ \, 
\kappa(1,a_{\max}) =
\frac{
1 -
\exp\bigl(
\frac{- \pi}{a_{\max}}
\bigr)
}{2 \exp\bigl(
\frac{\pi}{a_{\max}}
\bigr) 
- 1}
=
0.171573\ldots
\end{equation}
\end{definition}

Let us observe that the upper bound
$\kappa(1, a_{\max})$ is independent of $n$.
From this independence we deduce
the following ``limit" angular sector
in which the Rouch\'e conditions
can be applied.

\begin{proposition}
\label{argumentlastrootJn}
Let $n \geq 195$. Let us
put $\kappa:=\kappa(1, a_{\max})$ for short.
Then
\begin{equation}
\label{argzJJJn}
\arg(z_{J_n , n}) = 
2 \arcsin 
\bigl( \frac{\kappa}{2} \bigr)
+ \frac{\kappa \, \lo \kappa}
{n \sqrt{4 - \kappa^2}}
+
\frac{1}{n} O\bigl(
\bigl(
\frac{\lo \lo n}{\lo n}
\bigr)^2
\bigr),
\end{equation}
\begin{equation}
\label{Jnasymptotic}
J_n = 
\frac{n}{\pi}
\bigl(
\arcsin\bigl( \frac{\kappa}{2} \bigr) 
\bigr)
+
\frac{\kappa \, \lo \kappa}
{\pi \,\sqrt{4 - \kappa^2}}
+ 
O\bigl(
\bigl(
\frac{\lo \lo n}{\lo n}
\bigr)^2
\bigr)
\end{equation}
with, at the limit,
\begin{equation}
\label{philimite}
\lim _{n \to + \infty}
\arg(z_{J_n , n}) = 
\lim _{n \to + \infty}
2 \pi \frac{J_n}{n}
=
2 \arcsin 
\bigl( \frac{\kappa}{2} \bigr) = 0.171784\ldots 
\end{equation}
\end{proposition}

\begin{proof}
Since $\lim_{n \to +\infty} |z_{J_n , n}| = 1$,
we deduce from \eqref{JJJn} that
the limit argument
$\varphi_{lim}$ of $z_{J_n, n}$
satisfies $|-1 + \cos(\varphi_{lim})
+ i \sin(\varphi_{lim})|
= 2 \sin(\varphi_{lim}/2) = \kappa(1,a_{\max})$.
We deduce \eqref{philimite}, and the equality between
the two limits from
\eqref{argumentzjn}.
 
From \eqref{JJJn}, the inequality
$|-1 + z_{j ,n}| \leq |z_{j ,n}| \, \kappa(1,a_{\max})$
already implies that 
$\arg(z_{J_n , n})) < \varphi_{lim}$.
In the sequel, 
we will use the asymptotic expansions
of the roots $z_{J_n , n}$.
From Section 6 in \cite{vergergaugry6}
the argument of $z_{J_n , n}$ takes the following
form
\begin{equation}
\label{argumentzjn}
\arg(z_{J_n , n})) =
2 \pi (\frac{J_n}{n} + \Re)
\quad
\mbox{with}
\quad
\Re =
- \frac{1}{2 \pi n}
\left[
\frac{1 - \cos(\frac{2 \pi J_n}{n})}{\sin (\frac{2 \pi J_n}{n})} 
\lo (2 \sin(\frac{\pi J_n}{n}))
\right]
\end{equation}
with
$${\rm tl}(\arg(z_{J_n , n})))
=
+\frac{1}{n} O\left(
\left(
\frac{\lo \lo n}{\lo n}
\right)^2
\right) .
$$
Its modulus is
\begin{equation}
\label{devopomainCONSTANTEc}
|z_{J_n ,n}|
=
1 + \frac{1}{n} \lo (2 \sin \frac{\pi J_n}{n})
+ \frac{1}{n} O \left(
\frac{\lo \lo n}{\lo n}\right)^2 .
\end{equation}
Denote $\varphi := \arg(z_{J_n ,n})$.
Up to $\frac{1}{n} O\bigl(
\bigl(
\frac{\lo \lo n}{\lo n}
\bigr)^2
\bigr)$-terms, 
we have
$$|-1 + z_{J_n ,n}|^2 =
\Bigl|-1 + 
[1 + \frac{1}{n} \lo (2 \sin \frac{\pi J_n}{n})]
(\cos(\varphi) + i \sin(\varphi))
\Bigr|^2
$$
$$=
[-1 +[1 + \frac{1}{n} \lo (2 \sin \frac{\pi J_n}{n})]
(\cos(\varphi)]^2
+
[1 + \frac{1}{n} \lo (2 \sin \frac{\pi J_n}{n})]^2
(\sin(\varphi)^2
$$
$$=
1 +
[1 + \frac{1}{n} \lo (2 \sin \frac{\pi J_n}{n})]^2
-
2 [1 + \frac{1}{n} \lo (2 \sin \frac{\pi J_n}{n})]
\cos(\varphi)
$$
\begin{equation}
\label{JJJngauche}
=
4 (\sin(\frac{\varphi}{2}))^2
+
\frac{4}{n} (\sin(\frac{\varphi}{2}))^2 \, 
\lo (2 \sin \frac{\pi J_n}{n})
=
4 (\sin(\frac{\varphi}{2}))^2
[1 + \frac{1}{n} \lo (2 \sin \frac{\pi J_n}{n})] .
\end{equation}
Up to $\frac{1}{n} O\bigl(
\bigl(
\frac{\lo \lo n}{\lo n}
\bigr)^2
\bigr)$-terms, due to the definition of $J_n$,
let us consider \eqref{JJJn} as an equality;
hence, from \eqref{JJJngauche} and
\eqref{devopomainCONSTANTEc}, 
the following identity 
should be satisfied
\begin{equation}
\label{argumentINTER}
2 \sin(\frac{\varphi}{2})
=
\kappa \,[1 + \frac{1}{2 n} \lo (2 \sin \frac{\pi J_n}{n})]
\end{equation}
We now use \eqref{argumentINTER} to obtain
an asymptotic expansion of
$\psi_n := 2 \pi \frac{J_n}{n} - \varphi_{lim}$
as a function of $n$ and $\varphi_{lim}$
up to $\frac{1}{n} O\bigl(
\bigl(
\frac{\lo \lo n}{\lo n}
\bigr)^2
\bigr)$-terms.
First, at the first order in $\psi_n$,
$$
\sin(\frac{\pi J_n}{n}) = \frac{\psi_n}{2} \cos(\frac{\varphi_{lim}}{2})
+ \sin(\frac{\varphi_{lim}}{2}),
\quad
\cos(\frac{\pi J_n}{n}) = - \frac{\psi_n}{2} \sin(\frac{\varphi_{lim}}{2})
+ \cos(\frac{\varphi_{lim}}{2}),
$$
$$
\lo \bigl(2 \sin(\frac{\pi J_n}{n})\bigr)
=
\lo \bigl(2 \sin(\frac{\varphi_{lim}}{2})\bigr)
+
\psi_n \frac{\cos(\frac{\varphi_{lim}}{2})}
{2 \sin(\frac{\varphi_{lim}}{2})}
=
\lo \kappa + \psi_n \frac{\cos(\frac{\varphi_{lim}}{2})}{h}.
$$
Moreover,
$$
\Bigl[
\frac{1 - \cos(\frac{2 \pi J_n}{n})}{\sin (\frac{2 \pi J_n}{n})} 
\lo (2 \sin(\frac{\pi J_n}{n}))
\Bigr]
$$
\begin{equation}
\label{RezJJJn}
=
\tan (\frac{\varphi_{lim}}{2})
(\lo \kappa)
\Bigl[1 + \psi_n 
\left(
\frac{1}{\sin (\varphi_{lim})}
+
\frac{\cos (\frac{\varphi_{lim}}{2})}
{\kappa \, \lo \kappa}
\right)
\Bigr].
\end{equation}
Hence,
with $2 \sin(\varphi/2)
=
2 \sin(\pi J_n/n) \cos(\pi \Re)
+
2 \cos(\pi J_n/n) \sin(\pi \Re)
$,
and
from \eqref{argumentzjn},
up to $\frac{1}{n} O\bigl(
\bigl(
\frac{\lo \lo n}{\lo n}
\bigr)^2
\bigr)$-terms, the identity
\eqref{argumentINTER} becomes
$$
[\psi_n \cos(\frac{\varphi_{lim}}{2})
+ 2 \sin(\frac{\varphi_{lim}}{2})]
+(\frac{- 2 \cos(\frac{\varphi_{lim}}{2}) \tan(\frac{\varphi_{lim}}{2})
\, \lo \kappa}{2 n})
=
\kappa \,  \bigl[
1 + \frac{ \lo \kappa}{2 n}
\bigr] .
$$
We deduce 
\begin{equation}
\label{psinasymptotic}
\psi_n 
=
\frac{\kappa \, \lo \kappa}
{n \, \cos(\frac{\varphi_{lim}}{2})}
+
\frac{1}{n} O\bigl(
\bigl(
\frac{\lo \lo n}{\lo n}
\bigr)^2
\bigr)
,
\end{equation} 
then $2 \pi J_n / n = \psi_n + \varphi_{lim}$
and
\eqref{Jnasymptotic}.
With
$2 \pi J_n / n$, and
from \eqref{argumentzjn} and
\eqref{RezJJJn} we deduce
\eqref{argzJJJn}.
This finishes the proof.
\end{proof}

\begin{remark}
\label{openingangle_sin_quadratic_alginteger}
(i) The maximal half-opening angle 
of the sector in which
one can detect zeroes of $f_{\beta}(z)$,
for any $\beta$ such that
$\theta_{n-1} < \beta^{-1} < \theta_n$,
by the method of Rouch\'e, is
$0.17178... =
2 \arcsin(\frac{\kappa(1, a_{\max})}{2})$.
Remarkably
this upper bound $2 \arcsin(\frac{\kappa(1, a_{\max})}{2})$
is independent of $n$. By comparison it is
fairly small with respect to $\pi/3$ for
the Perron numbers $\theta_{n}^{-1}$.

(ii) The curve $a \to \kappa(1,a)$, given by Figure 
\ref{h1a},
is such that any value
in the interval $(0, \kappa(1,a_{max}))$
is reached by the function
$\kappa(1,a)$
from two values say $a_1$
and
$a_2$, of $a$, satisfying $a_1 < a_{max}
< a_2$. On the contrary, the correspondence
$a_{max} \leftrightarrow \kappa(1,a_{max})$ is unique, 
corresponding to a double root.
Denote $D:=\exp(\pi/a_{max})$ and 
$\kappa := \kappa(1,a_{max})$.
It means that the quadratic algebraic equation 
$2 \kappa D^2 - (\kappa+1) D +1 =0$ 
deduced from the upper bound
in \eqref{JJJn} has necessarily
a discriminant equal to zero.
The discriminant is
$\kappa^2 - 6 \kappa +1$. 
Therefore $D = (\kappa+1)/(4 \kappa)$
and the limit value
$x = 2 \arcsin(\kappa/2)$ in
\eqref{philimite} satisfies
the quadratic algebraic equation
$$4 (\sin(x/2))^2   -
12 \sin(x/2) + 1 = 0.$$ 
\end{remark}

\begin{proposition}
\label{cercleoptiMMBUMP}
Let $n \geq n_1 = 195$.
The circles
$C_{j,n}
:= \{z \mid |z-z_{j,n}| = 
\frac{\pi |z_{j,n}|}{n\, a_{\max}} \}$
centered at the roots 
$z_{j,n}$
of the trinomial
$-1 + z + z^n$ which belong to
the ``bump sector", namely 
for
$j \in \{1, 2, \ldots, \lfloor v_n \rfloor \}$,
are such that
the 
conditions of Rouch\'e 
\begin{equation}
\label{rouchecercleBUMP}
\frac{
\left|z\right|^{2 n -1}}{1 - |z|^{n-1}}
~<~
\left|-1 + z + z^n \right| , 
\quad \mbox{for all}~ z \in C_{j,n},
~~\mbox{}~
1 \leq j \leq \lfloor v_n \rfloor ,
\end{equation}
hold true. 
For any real number $\beta > 1$ having 
$\dyg(\beta) = n$,
$f_{\beta}(z)$ admits a
simple
zero $\omega_{j,n}$ in 
$D_{j,n}$
(with $a=a_{{\max}}$), for
$j$ in the range 
$\{1, 2, \ldots,
\lfloor v_n \rfloor\}$. 

\end{proposition}

\begin{proof}
The development terms $``D"$ of the
asymptotic expansions
of $|z_{j,n}|$ change 
from the main angular sector
$\arg z \in (2 \pi (\lo n)/n,
\pi/3)$ to
the first transition region
$\arg z \asymp 2 \pi (\lo n)/n$, the ``bump sector",
further to the second transition region
$\arg z \asymp 2 \pi \sqrt{(\lo n)(\lo \lo n)}/n$,
and to a small neighbourhood of 
$\theta_n$
(Section \ref{S3.2}).

Then the proof of
\eqref{rouchecercleBUMP}
is the same as that
of Theorem \ref{cercleoptiMM} once 
\eqref{devopomain} is 
substituted by the suitable 
asymptotic expansions which correspond to 
the angular sector
of the ``bump". 
The terminants of the 
respective asymptotic expansions
of $|z_{j,n}|$
also change: this change imposes to 
reconsider
\eqref{rouchegauche} 
and
\eqref{rouchedroite}
up to $\lo n/n$ - terms, and not
up to $1/n$ - terms, as in the proof
of Theorem \ref{cercleoptiMM}. 
It is remarkable that 
the  inequality \eqref{amaximalfunctionX} 
remains the same, with the same upper bound function
$\kappa(X,a)$.
Then the equation of the 
curve of the Rouch\'e condition
$a \to \kappa(1,a)$, on $[1, +\infty)$,
is the same as in Theorem 
\ref{cercleoptiMM} for controlling
the conditions of Rouch\'e.
The optimal value 
$a_{\max}$ of $a$ also remains the same, and
\eqref{rouchecercle} also holds true
for those $z_{j,n}$ in the bump sector.
\end{proof}

From the inequalities 
\eqref{petitsecteur18} in
Theorem \ref{cercleoptiMM}, also used
in the proof of
Proposition \ref{cercleoptiMMBUMP},
we now obtain a finer localization
of  a subcollection of the
roots 
$\omega_{j,n}$
of the Parry Upper function
$f_{\beta}(z)$, and
a definition of the lenticulus 
$\lc_{\beta}$ of $\beta$, as follows.

\begin{theorem}
\label{omegajnexistence}
Let $n \geq n_1 = 195$.
Let 
$\beta > 1$ be 
any real number
having $\dyg(\beta) = n$.
The Parry Upper function
$f_{\beta}(z)$ has 
an unique simple zero
$\omega_{j,n}$ in each disc
$D_{j,n} := \{z \mid
|z - z_{j,n}| < 
\frac{\pi \, |z_{j,n}|}{ n \, a_{\max}}\}$,
$j = 1, 2, \ldots, J_n$,
which satisfies the additional inequality:
\begin{equation}
\label{rootslesvoila}
|\omega_{j,n} - z_{j,n}| < \, 
\frac{\pi |z_{j,n}|}{ n \, a_{j,n}} 
\qquad \quad
\mbox{for}~~ j = \lceil v_n \rceil,
\lceil v_n \rceil+1,  \ldots, J_n ,
\end{equation}
where $a_{J_n , n} = a_{\max}$ and, 
for 
$j = \lceil v_n \rceil, \ldots, J_n - 1$,
the value $a_{j,n}$~, $> a_{\max}$, is
defined by
\begin{equation}
\label{petitcercle_jmainDBJN}
D\bigl(\frac{\pi}{a_{j,n}}
\bigr)
~=~
\lo \Bigl[
\frac{1 + B_{j,n}
-
\sqrt{1 - 6 B_{j,n} + B_{j,n}^{2} }}{4 B_{j,n}}
\Bigr]
\end{equation}
$$\mbox{with}\qquad\qquad
B_{j,n}:= 
2 \sin(\frac{\pi j}{n})
\Bigl(
1 - \frac{1}{n} \lo (2 \sin(\frac{\pi j}{n}))
\Bigr) ,
$$
and, putting 
$D:= D\bigl(\frac{\pi}{a_{j,n}}
\bigr)$ for short,
\begin{equation}
\label{petitcercle_jmainTL}
{\rm tl}(\frac{\pi}{a_{j,n}}
\bigr)
~=~
\frac{2}{n} \times
B_{j,n}^{-1} \,
(\frac{-3 + \exp(-D) + 2 \exp(D)}
{4 - \exp(-D) - 2 \exp(D)})
\times
\left(
\frac{\lo \lo n}{\lo n}
\right)^2 .
\end{equation}
An upper bound of the tails, independent of $j$,
is given by
\begin{equation}
\label{petitcercle_jmainTL_majorant}
O \Bigl(
\frac{(\lo \lo n)^2}{(\lo n)^3}
\Bigr)
\end{equation}
with the constant $\frac{1}{7 \pi}$ 
in the Big O.
The lenticulus
$\lc_{\beta}$ associated with $\beta$
is constituted by 
the following subset of roots of
$f_{\beta}(z)$:
\begin{equation}
\label{lenticulusBETA}
\lc_{\beta}
:=
\{1/\beta\}
~~\cup~~
\bigcup_{j=1}^{J_n}
\left(
\{\omega_{j,n} \}
\cup 
\{\overline{\omega_{j,n}} \}\right).
\end{equation}
\end{theorem}

\begin{proof}
The existence 
of the zeroes comes from
Proposition \ref{cercleoptiMMBUMP}
and Theorem \ref{cercleoptiMM}, with 
the maximal value $J_n$
of the
index $j$
given by Proposition \ref{Jndefinition}.
To refine the localization
of $\omega_{j,n}$
in the neighbourhood of
$z_{j,n}$, in the main angular sector, i.e.
for $j \in 
\{\lceil v_n \rceil,
\lceil v_n \rceil+1,  \ldots, J_n\}$,
the conditions of Rouch\'e 
\eqref{rouchecercle}
are now used to define the new radii.

The value $a_{j,n}$ 
is defined by the development term
${\rm D}(\frac{\pi}{a_{j,n}})$, 
itself defined
as follows:
\begin{equation}
\label{petitsecteur_aJI_D}
{\rm D}\left(
\frac{|-1+z_{j,n}|}{|z_{j,n}|}
\right) 
~~=:~~ \frac{
1 -
\exp\bigl(
-{\rm D}(\frac{\pi}{a_{j,n}})
\bigr)}
{2 \exp\bigl(
{\rm D}(\frac{\pi}{a_{j,n}})
\bigr) -1}
\end{equation}
and the tail
${\rm tl}(\frac{\pi}{a_{j,n}})$ 
calculated from 
${\rm tl} \left(
\frac{|-1+z_{j,n}|}{|z_{j,n}|}
\right)$ so that
the Rouch\'e condition
\begin{equation}
\label{petitsecteur_aJItail}
\frac{|-1+z_{j,n}|}{|z_{j,n}|}
=
{\rm D}\left(
\frac{|-1+z_{j,n}|}{|z_{j,n}|}
\right)
+ {\rm tl} \left(
\frac{|-1+z_{j,n}|}{|z_{j,n}|}
\right)
~~<~~ \frac{
1 -
\exp\bigl(
-\frac{\pi}{a_{j,n}}
\bigr)}
{2 \exp\bigl(
\frac{\pi}{a_{j,n}}
\bigr) -1}
\end{equation}
holds true.
From
Proposition \ref{zedeJIMoinsUNzedeJI},
denote
$$B_{j,n}:=
{\rm D}\left(
\frac{|-1+z_{j,n}|}{|z_{j,n}|}
\right)= 
2 \sin(\frac{\pi j}{n})
\Bigl(
1 - \frac{1}{n} \lo (2 \sin(\frac{\pi j}{n}))
\Bigr).$$
Let
$W := \exp({\rm D}(\frac{\pi}{a_{j,n}}))$.
The identity \eqref{petitsecteur_aJI_D}
transforms into the equation of degree 2:
\begin{equation}
\label{bjn_complet}
2 B_{j,n} \, W^2
- 
\Bigl(
B_{j,n}
+ 1
\Bigr) W + 1 ~=~0
\end{equation}
from which \eqref{petitcercle_jmainD} 
is deduced. For the calculation of
${\rm tl}(\frac{\pi}{a_{j,n}})$,
denote $D := {\rm D}(\frac{\pi}{a_{j,n}})$
and $tl_{j,n} := 
{\rm tl}(\frac{\pi}{a_{j,n}})$.
Then, at the first order,
$ \frac{1 -
\exp\bigl(
\frac{-\pi}{a_{j,n}}
\bigr)}
{2 \exp\bigl(
\frac{\pi}{a_{j,n}}
\bigr) -1}$
$$ =
\frac{1 -
\exp\bigl(
-D - tl_{j,n}
\bigr)}
{2 \exp\bigl(
D + tl_{j,n}
\bigr) -1}
= B_{j,n}
[1+
 tl_{j,n} \times
(\frac{4 - \exp(-D) - 2 \exp(D)}
{-3 + \exp(-D) + 2 \exp(D)})].
$$
From \eqref{petitsecteur_aJItail} and
\eqref{UNmoinszedeJIzedeJI}
the following
inequality should be satisfied, 
with the constant 2 in the Big O,
$$
\frac{1}{n} O\Bigl(
\bigl(
\frac{\lo \lo n}{\lo n}
\bigr)^2
\Bigr)
=
{\rm tl} \Bigl(
\frac{|-1+z_{j,n}|}{|z_{j,n}|}
\Bigr)
~<~
tl_{j,n}
\times
B_{j,n} \,
(\frac{4 - \exp(-D) - 2 \exp(D)}
{-3 + \exp(-D) + 2 \exp(D)})] .
$$
The expression of $tl_{j,n}$
in \eqref{petitcercle_jmainTL}
follows, to obtain a strict 
inequality in \eqref{petitsecteur_aJItail}.
By \eqref{petitcercle_jmainDBJN} 
the quantity $\exp(D)$ is a function 
of $B_{j,n}$, which tends to 
$\frac{3}{4}$ when 
$B_{j,n}$ tends to $0$;
hence, at the first order, 
a lower bound of 
the function $$B_{j,n} \to
\Bigl| B_{j,n} \,
(\frac{4 - \exp(-D) - 2 \exp(D)}
{-3 + \exp(-D) + 2 \exp(D)})
\Bigr|$$
is obtained for
$j = \lceil v_n \rceil$, and given by
$2 \pi \frac{\lo n}{n} \times 7$.
Then it suffices to take
$$tl_{j,n} =
cste \Bigl(
\frac{(\lo \lo n)^2}{(\lo n)^3}
\Bigr)
$$
with $cste = 1/(7 \pi)$, to obtain
a tail independent of $j$,
and therefore the
conditions of Rouch\'e 
\eqref{petitsecteur_aJItail}
 satisfied
with these new smaller radii and tails
in the main 
angular sector.
\end{proof}

\begin{remark}
\label{petitcercleapproxinfini}
For $n$ very large, up to second-order terms, \eqref{bjn_complet}
reduces to
$$4 \sin(\frac{\pi j}{n}) \, W^2
- 
\Bigl(
2 \sin(\frac{\pi j}{n})
+ 1
\Bigr) W + 1 ~=~0
$$
and \eqref{petitcercle_jmainDBJN}
to
\begin{equation}
\label{petitcercle_jmainD}
D\bigl(\frac{\pi}{a_{j,n}}
\bigr)
~=~
\lo \Bigl[\frac{1 + 2 \sin(\frac{\pi j}{n})
-
\sqrt{1 - 12 \sin(\frac{\pi j}{n}) 
+ 4 (\sin(\frac{\pi j}{n}))^2}}{8 \sin(\frac{\pi j}{n})}
\Bigr].
\end{equation}
\end{remark}

\begin{lemma}
\label{cnroucheasymptotic}
Let $n \geq 195$
and
$c_n$ defined by
$|z_{J_n , n}| = 1 - \frac{c_n}{n}$.
Let us put
$\kappa:=\kappa(1,a_{\max})$ for short.
Then
\begin{equation}
\label{cncnlimit}
c_n = - (\lo \kappa) \, (1 + \frac{1}{n})
+
\frac{1}{n} 
O\bigl(
\bigl(\frac{\lo \lo n}{\lo n}\bigr)^2
\bigr) ,
\end{equation} 
with $c = \lim_{n \to +\infty} c_n = 
- \lo \kappa = 1.76274\ldots$,
and, up to  $O(\frac{1}{n}\bigl( 
\bigl(
\frac{\lo \lo n}{\lo n}
\bigr)^2\bigr))$-terms,
\begin{equation}
\label{cnroucheminilimit}
\frac{(1 - \frac{c_n}{n})^{2 n}}
{(1 - \frac{c_n}{n}) - (1 - \frac{c_n}{n})^{n}}
=
\frac{e^{-2 c}}{1 - e^{- c}}
\Bigl(
1 +
\frac{c}{2 n (1 - e^{- c})}
\bigl[
2 - c e^{-c} - 2 c
\bigr]
\Bigr) 
\end{equation}
with $e^{-2 c}/(1 - e^{- c}) = 0.0355344\ldots$
\end{lemma}

\begin{proof}
The asymptotic expansion 
\eqref{cncnlimit}
of $c_n$ 
is deduced from the asymptotic expansions
of $\psi_n$ 
and $z_{J_n , n}$ given
by \eqref{psinasymptotic} and
\eqref{devopomainCONSTANTEc}
(Proposition 3.5
in \cite{vergergaugry6}).
We deduce
the limit 
$
c ~:=~ - \lo (\kappa(1,a_{\max})) =
1.76274\ldots
$
and then \eqref{cnroucheminilimit}
follows.
\end{proof}

\begin{definition}
\label{definitionHn}
Let $n \geq n_2 := 260$.
We denote by $H_n$ the 
largest integer $j \geq \lceil v_n \rceil$
such that
\begin{equation}
\label{HnJnmini}
\arg(z_{J_n , n}) - \arg(z_{j, n}) 
\geq 
\frac{(1 - \frac{c_n}{n})^{2 n}}
{(1 - \frac{c_n}{n}) - (1 - \frac{c_n}{n})^{n}}.
\end{equation}
\end{definition}

\begin{proposition}
\label{hnasymptotic}
Let $n \geq 260$. Let denote
$\kappa := \kappa(1,a_{\max})$
for short.
Then
$$
\arg(z_{H_n , n})
~=~
2 \arcsin\bigl(\frac{\kappa}{2}\bigr) ~-~ 
\frac{\kappa^2}{1 - \kappa} \hspace{6.5cm}\mbox{}
$$
\begin{equation}
\label{hhnasymptotic}
+ \frac{\lo \kappa}{n}
\left[
\frac{\kappa}
{\sqrt{4 - \kappa^2}}
+
\frac{2 + \kappa \, \lo(\kappa)+ 2 \, \lo(\kappa)}
{2 ( 1 - \kappa )}
\right]
+ \frac{1}{n} O\bigl(
\Bigl(
\frac{\lo \lo n}{\lo n}
\Bigr)^2
\bigr),
\end{equation}
with, at the limit,
$$\lim_{n \to +\infty}
\arg(z_{H_n , n})
=
2 \arcsin\bigl(\frac{\kappa}{2}\bigr) ~-~ 
\frac{\kappa^2}{1 - \kappa} =
0.13625 .
$$
\end{proposition}

\begin{proof}
The asymptotic expansion 
of the right-hand side
term of \eqref{HnJnmini} is
\begin{equation}
\label{cielrouche}
\frac{(1 - \frac{c_n}{n})^{2 n}}
{(1 - \frac{c_n}{n}) - (1 - \frac{c_n}{n})^{n}}
=
\frac{e^{-2 c}}{1 - e^{- c}} 
\Bigl(
1 +
\frac{c (2 - c e^{-c} - 2 c)}{2 n (1 - e^{- c})}
\Bigr) + \ldots
\end{equation}
Then the
asymptotic expansion
of $\arg(z_{H_n , n})$ comes from
\eqref{HnJnmini}
in which the inequality is 
replaced by an equality,
and from the
asymptotic expansion \eqref{argzJJJn} of
$\arg(z_{J_n , n})$
(Proposition \ref{argumentlastrootJn}).
\end{proof}

For $n$ large enough, $\arg(z_{H_n , n})$
is equal to $2 \pi \frac{H_n}{n}$, up 
to higher order - terms, and a definition of
$H_n$ in terms of asymptotic expansions could be:
\begin{equation}
\label{hhhnasymptotic}
H_n\! =\! \lfloor
\frac{n}{2 \pi}
\bigl(
 2 \arcsin\bigl(\frac{\kappa}{2}\bigr) - 
\frac{\kappa^2}{1 - \kappa}\bigr)
\!-\!
\lo(\kappa)
\Bigl[
\frac{\kappa}
{\sqrt{4 - \kappa^2}}
+
\frac{2 + \kappa \, \lo(\kappa)+ 2 \, \lo(\kappa)}
{2 ( 1 - \kappa )}
\Bigr]
\rfloor,
\end{equation}
For simplicity's sake, we
will take the following 
definition of $H_n$
\begin{equation}
\label{hhhndefinition}
H_n := \lfloor
\frac{n}{2 \pi}
\bigl(
 2 \arcsin\bigl(\frac{\kappa}{2}\bigr) ~-~ 
\frac{\kappa^2}{1 - \kappa}\bigr) - 1
\rfloor .
\end{equation}
\begin{remark}
\label{value260}
The value $n_2 = 260$ is calculated
by the inequality
$ \frac{2 \pi v_n}{n} < \arg(z_{H_n , n})$
which should be valid 
for all $n \geq 260$, 
where $H_n$ is given by
\eqref{hhhndefinition},
$\arg(z_{H_n , n})$ by
\eqref{hhnasymptotic},
where $(v_n)$ is the
delimiting sequence (cf Appendix) of the
transition region of the boundary of the bump sector.
A first minimal value of $n$ 
is first estimated by
$2 \pi \frac{\lo n}{n} < D(\arg(z_{H_n , n}))$
using \eqref{hhnasymptotic}.
Then it is corrected so that
the numerical value
of the tail of the asymptotic expansion
in \eqref{hhnasymptotic} be taken into account
in this inequality.
\end{remark}

\begin{theorem}
\label{absencezeroesOutside}
Let $n \geq n_2 := 260$.
Denote by
$\dc_{n}$ the subdomain of the open 
unit disc, symmetrical with respect to the real axis, 
defined by the conditions:
\begin{equation}
\label{cerclescalottes_0}
 |z| < 1 - \frac{c_n}{n}, 
\qquad \frac{1}{n}
\left(\frac{\lo \lo n}{\lo n}\right)^2 < 
|z - \theta_n|,
\end{equation}
\begin{equation}
\label{cerclescalottes_1}
\frac{\pi |z_{j,n}|}{n \, a_{\max}} < |z - z_{j,n}|,
\quad
\qquad 
\mbox{for}~~ j= 1, 2, \ldots, J_n ,
\end{equation}
and,
for $\displaystyle j = 
J_n + 1,
\ldots, 2 J_n - H_n + 1$, 
\begin{equation}
\label{cerclescalottes}
\displaystyle \frac{\pi |z_{j,n}|}
{n \, s_{j,n}} < |z - z_{j,n}| ,
\qquad \mbox{with}~~
s_{j,n} = a_{\max}
\Bigl[
1 +
\frac{a_{\max}^{2}
(j - J_n)^2}{\pi^2 \,
J_{n}^{2}}
\Bigr]^{-1/2} . 
\end{equation}

Then, for any real number 
$\beta > 1$ 
having $\dyg(\beta) = n$, 
the Parry Upper function
$f_{\beta}(z)$ does not vanish at
any point $z$ in $\dc_{n}$.
\end{theorem}

\begin{proof}
Assume $\beta > 1$ such that
$\theta_{n-1} < \beta^{-1} < \theta_n$.
We will apply the general form of the 
Theorem of Rouch\'e to the compact
$\kc_n$ which is 
the adherence of the domain $\dc_{n}$,
i.e. we will show that
the inequality (and symmetrically with 
respect to the real axis)
\begin{equation}
\label{roucheKn}
|f_{\beta}(z) - G_{n}(z)| < |G_{n}(z)|,
\quad \quad z~ \in~ \partial \kc_{n}^{ext} \,  
\cup \, C_{1,n} \, 
\cup \, C_{2,n} \cup
\ldots \cup \, C_{J_n , n} 
\end{equation}
holds, with $z \in {\rm Im}(z) \geq 0$,
where $\partial \kc_n$
is the union of: (i) the arcs of the
circles defined by the equalities in
\eqref{cerclescalottes_1}
and
\eqref{cerclescalottes},
arcs which lie in $|z| \leq 1 - c_n/n$,  
and circles
for which the intersection
with $|z| = 1 - c_n/n$
is not empty,
(ii) the arcs of 
$C(0,1-c_n/n)$ which have empty intersections
with the interiors
of the discs defined by
the inequalities ``$>$", instead of
``$<$", in
\eqref{cerclescalottes_1}
and
\eqref{cerclescalottes},
which join two successive circles.
The two functions
$f_{\beta}(z)$
and $G_{n}(z)$ are continuous
on the compact
$\kc_n$, holomorphic
in its interior $\dc_n$, and
$G_n$ has no zero 
in $\kc_n$.
As a consequence
the function $f_{\beta}(z)$ will have no zero in
the interior $\dc_n$ of $\kc_n$.

Instead of using $f_{\beta}(z)$ itself
in \eqref{roucheKn},
we will show
that the following inequality
holds true 
\begin{equation}
\label{rouchecercleNONZERO}
\frac{
\left|z\right|^{2 n -1}}{1 - |z|^{n-1}}
~<~
\left|-1 + z + z^n \right| , 
\quad ~\mbox{for all}~ 
z 
\in \partial \kc_{n}^{ext}
\end{equation}
what will imply the claim.

The Rouch\'e inequalities
\eqref{roucheKn} 
\eqref{rouchecercleNONZERO} 
hold true
on the (complete) 
circles \,$C_{j,n}$, $1 \leq j \leq J_n$
by Theorem \ref{cercleoptiMM}
and Proposition \ref{cercleoptiMMBUMP};
these conditions 
become out of reach 
for $j$ taking higher values (i.e.
in $\{J_n + 1, \ldots,
\lfloor n/6\rfloor\}$), but we will show that
they remain true
on the arcs defined by
the equalities
in \eqref{cerclescalottes}.
The domain $\dc_n$ only depends 
upon the dynamical degree $n$ of
$\beta$, not of $\beta$ itself.

Let us prove that 
the external Rouch\'e circle
$|z| = 1 - c_n/n$ intersects
all the circles 
$C_{J_n - k,n}, k = 0 , 1, \ldots, 
k_{\max}$,
with
$k_{\max} := \lfloor 
J_{n} (\frac{\pi}{a_{\max}})\rfloor$.
Indeed, up to 
$\frac{1}{n} O 
\bigl(\bigl(\frac{\lo \lo n}{\lo n}\bigr)^2 \bigr)$-
terms, 
from Proposition
\ref{argumentlastrootJn},
$$\lo (2 \sin(\pi \frac{J_n}{n}))
= \lo (2 \sin(\pi \frac{(J_n - k) + k}{n}))
= \lo \bigl(
2 \pi \frac{J_n - k}{n}(1 + \frac{k}{ J_n - k})
\bigr)$$
\begin{equation}
\label{intero}
=\lo \bigl(2 \sin(\pi \frac{J_n - k}{n})\bigr) 
+ \frac{k}{J_n} . 
\end{equation}
Since $|z_{J_n , n}|=
1 - c_n/n = 
1 +
\frac{1}{n}
\lo (2 \sin(\pi \frac{J_n}{n}))
+
\frac{1}{n} O 
\bigl(\bigl(\frac{\lo \lo n}{\lo n}\bigr)^2 \bigr)$, 
we deduce from
\eqref{intero}, 
with 
$k \leq  k_{\max}$,
that the point 
$z \in C(0,1 - c_n/n )$ 
for which
$\arg(z) =  \arg(z_{J_n - k,n})$
is such that
$$|z_{J_n - k ,n} - z | 
~=~
\frac{k}{n \, J_n}
~\leq~ 
 \, \frac{\lfloor J_n (\frac{\pi}{a_{\max}})
\rfloor}{n \, J_n}
~\leq~ \frac{\pi}{n \, a_{\max}}
$$
up to 
$\frac{1}{n} O 
\bigl(\bigl(\frac{\lo \lo n}{\lo n}\bigr)^2 \bigr)$-
terms. 
As soon as $n$ is large enough, 
we deduce
that $z$ lies in the interior of
$D_{J_n - k,n}$. 
Since the function
$x \to \lo (2 \sin(\pi x))$ is negative and
strictly increasing
on $(0, 1/6)$, the sequence
$(|z_{j,n}|)_{j=H_n , \ldots, J_n}$
is strictly increasing,
by \eqref{devopomainCONSTANTEc}. 
Hence we deduce that the
circle $|z| = 1 - c_n/n$
intersects 
all the circles 
$C_{j,n}$
for $j=J_n - k_{\max}, \ldots, J_n$.

The same arguments show that
the external Rouch\'e circle
$|z|= 1 - c_n/n$ intersects all the circles
$C(z_{j,n}, \frac{\pi |z_{j,n}|}
{n \, s_{j,n}})$
for $j = J_n + 1, J_n + 2, \ldots, 2 J_n - H_n + 1$.

The quantities 
$s_{j,n}$, for 
$j = J_n + 1, \ldots , 2 J_n - H_n +1$,
are easily calculated (left to the reader)
so that
the distance (length of the $j$-th circle segment)
$$\left|
~\frac{z_{j,n}}{|z_{j,n}|} (1 - \frac{c_n}{n}) - y_j
~\right|
=
\left|
~\frac{z_{j,n}}{|z_{j,n}|} (1 - \frac{c_n}{n}) - y'_j 
~\right|$$
for $y_j, y'_j \in 
C(z_{j,n}, \frac{\pi |z_{j,n}|}{n \, s_{j,n}})
\cap C(0, 1 - \frac{c_n}{n}), y_j \neq y'_j$,
be independent of 
$j$ in the interval 
$\{J_n +1 , \ldots, 2 J_n - H_n + 1\}$
and equal to
\begin{equation}
\label{chord}
\frac{\pi |z_{J_n , n}|}{n \, a_{\max}}.
\end{equation}
Then the two sequences
of moduli of centers
$(|z_{j,n}|)_{j=J_n + 1 , \ldots, 2 J_n - H_n +1}$
and of radii
$(\frac{\pi |z_{j,n}|}
{n \, s_{j,n}})_{j=J_n + 1 , \ldots, 2 J_n - H_n +1}$
are both increasing, with the fact that
the corresponding discs
$D(z_{j,n}, \frac{\pi |z_{j,n}|}
{n \, s_{j,n}})$
keep constant the 
intersection chord  
$\arg(y_j) - \arg(y'_j)
= 
\frac{\pi |z_{J_n , n}|}{n \, a_{\max}}$
with the external Rouch\'e 
circle $|z| = 1 - c_n/n$.

\begin{figure}
\begin{center}
\includegraphics[width=8cm]{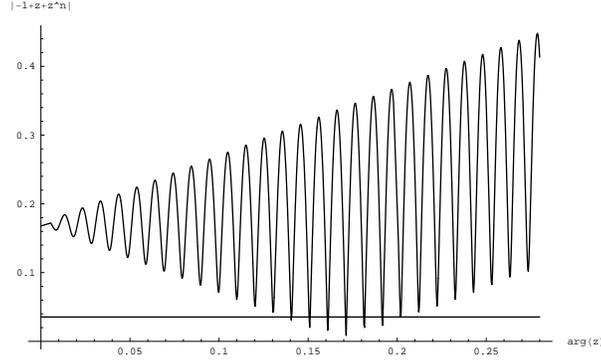}
\end{center}
\caption{
Oscillations of the upper bound
$|-1 + z + z^n|$ of
the Rouch\'e inequality \eqref{rouchecercle},
for $z$ running over the curve
$|z|=1 - c_n / n$ (here represented with $n = 615$)
as a function of $\arg(z)$ in 
$[0, 0.28]$.
The minima correspond to the angular positions
of the zeroes $z_{j,n}$
of the trinomial $-1 + X + X^n$,
for $j = 1, 2, \ldots, H_n,
\ldots, J_n , \ldots, 2 J_n - H_n + 1, \ldots$
($J_{615} = 17 , H_{615} = 12$).
The angular separation between two successive minima
is $\approx 2 \pi/n$. The difference between
two successive minima is
$\approx 2 \pi/n$. For $n=615$,
the arguments
$2 \pi (\lo n)/n$ (limiting the bump sector),
$\arg(z_{H_n , n})$ and
$\arg(z_{J_n , n})$ are respectively equal to
$0.0656\ldots, 0.12189\ldots, 0.17129\ldots$.
The horizontal line at the $y$-coordinate 
$0.0354...$ is the value of the 
left-hand side term of
the Rouch\'e inequality \eqref{rouchecercle}
(Proposition \ref{roucheKn}); 
it is always strictly smaller than
the minimal
value of
the oscillating function $|-1 + z + z^n|$
on the external boundary
$\partial \kc_{n}^{ext}$, 
whose geometry surrounds the roots
$z_{j,n}$ for $j$ 
between $H_n + 1$ and $2 J_n - H_n + 1$. 
}
\label{oscill}
\end{figure}

Let $z \in C(0, 1 - \frac{c_n}{n})$,
$\varphi := \arg(z) \in [0, \pi]$. 
Denote by
$Z(\varphi) := 
| G_{n}((1 - \frac{c_n}{n}) e^{i \varphi})|^2
= \bigl|-1  + (1 - \frac{c_n}{n}) e^{i \varphi}
+ (1 -\frac{c_n}{n})^n e^{i \, n \, \varphi}\bigr|^2$.
The expansion of the function
$Z(\varphi)$ as a function
of $\varphi$, up to
$O(1/n)$- terms, is the following:
$Z(\varphi) =$
$$(-1 + (1 - \frac{c_n}{n}) \cos(\varphi)
\!+\!
(1 - \frac{c_n}{n})^n \cos(n \varphi))^2
\!+\!
((1 - \frac{c_n}{n}) \sin(\varphi)
)\!+\!
(1 - \frac{c_n}{n})^n \sin(n \varphi))^2
$$
$$
=2 + e^{- 2 c} - 2 \cos(\varphi)
- 2 e^{-c} \cos(n \, \varphi)
+ 2 e^{- c} \cos(\varphi) \cos(n \, \varphi)
+ 2 e^{- c} \sin(\varphi) \sin(n \, \varphi)
$$
$$
= 2 + e^{- 2 c} - 2 \cos(\varphi)
- 4 e^{-c} \sin(\frac{\varphi}{2})
\left(\cos(n \, \varphi) \sin(\frac{\varphi}{2})
- \sin(n \, \varphi) \cos(\frac{\varphi}{2})
\right)
$$
\begin{equation}
\label{zedefi}
= 2 + e^{- 2 c} - 2 \cos(\varphi)
+ 4 e^{-c} \sin(\frac{\varphi}{2})
\sin(n \varphi - \frac{\varphi}{2}).
\end{equation}
The function $Z(\varphi)$,
defined on $[0, \pi/3]$,
is almost-periodic (in the sense of Besicovitch
and Bohr), 
takes the value $0$  at
$\varphi = \arg(z_{J_n ,n})$, and
therefore, up to 
$O(1/n)$-terms, has its minima 
at the
successive arguments 
$\arg(z_{J_n ,n}) + \frac{2 k \pi}{n}$
for
$|k|= 0, 1, 2, \ldots, J_n - H_n + 1, \ldots$
(Figure \ref{oscill}).
For such integers $k$,
from \eqref{zedefi}, we deduce
the successive minima
\begin{equation}
\label{successiveminima}
|-1 + z_{J_n ,n} e^{- 2 i k \pi/n}
+ (z_{J_n ,n} e^{- 2 i k \pi/n})^n|
=
|G_{n}(z_{J_n ,n})| 
+ \frac{2 |k| \pi}{n} = \frac{2 |k| \pi}{n}
\end{equation}
up to 
$\frac{1}{n} O \bigl(\bigl(\frac{\lo \lo n}{\lo n}
\bigr)^2\bigr)$- terms, with
$\arg(z_{J_n , n} e^{- 2 i k \pi/n})
=
\arg(z_{J_n - k,n})$
up to
$O(1/n)$-terms.

With the above notations, 
denote by $y_j , y'_j$ the two points
of $C(0, 1 - \frac{c_n}{n})$
which belong to
$C_{j,n}$ for $2 H_n - J_n \leq j \leq J_n$,
to
$C(z_{j,n}, \frac{\pi |z_{j,n}|}{n \, s_{j,n}})$
for $J_n + 1 \leq j \leq 2 J_n - H_n + 1$.
Writing them by increasing argument, we
have:
\begin{equation}
\label{listeyjyprimej}
y_{2 H_n - J_n}, y'_{2 H_n - J_n}, \ldots, 
y_{H_{n}}, y'_{H_{n}}, \ldots,
y_{J_{n}}, y'_{J_{n}},~
y_{J_{n} + 1}, y'_{J_{n} + 1},
\ldots,
y_{2 J_{n} - H_n + 1},~ y'_{2 J_{n} - H_n + 1}.
\end{equation}
The Rouch\'e inequality
\eqref{roucheKn} is obviously satisfied
at each $y_j$ and $y'_j$ for $j= 2 H_n - J_n, \ldots, J_n$.
Let us show that 
this inequality holds
at each point
$y_j$ and $y'_j$ for $j= J_n + 1, \ldots, 
2 J_n - H_n +1$.
Indeed, for such a point, say
$y_j$, 
there exists
$$\xi_j
=
w_j \, z_{J_n , n} e^{2 i (j - J_n) \pi/n} 
+ (1-w_j) \, y_j
, \quad \mbox{for some}~ w_j \in [0,1],$$  
lying in the segment
$\left[z_{J_n , n} e^{2 i (j - J_n) \pi/n}, y_j
\right]$
such that
$$G_{n}(y_j) = G_{n}(z_{J_n , n} e^{2 i (j-J_n) \pi/n})
+ (y_j - z_{J_n , n} e^{2 i (j-J_n) \pi/n}) \, G'_{n}(\xi_j) 
$$
with, using \eqref{chord},
$$|G_{n}(y_j) - G_{n}(z_{J_n , n} e^{2 i (j - J_n) \pi/n})|
=
|y_j - z_{J_n , n} e^{2 i (j - J_n) \pi/n}| |G'_{n}(\xi_j)|
=
\frac{\pi |z_{J_n , n}|}{n \, a_{\max}}
|G'_{n}(\xi_j)|.$$
The derivative of $G_{n}(z)$
is $G'_{n}(z) = 1 + n z^{n-1}$.
Up to $O(1/n)$-terms, 
the line generated by the
segment $\left[z_{J_n , n} e^{2 i (j - J_n) \pi/n}, 
y_j \right]$
is tangent to the circle
$C(0, 1-c_n/n)$, and the
modulus $\frac{1}{n}|G'_{n}(\xi_j)|$ 
satisfies
$$\frac{1}{n}|G'_{n}(\xi_j)| 
\,=\,
\frac{1}{n}|G'_{n}(z_{J_n , n} e^{2 i (j - J_n) \pi/n})|
\,=\,
\frac{1}{n}|G'_{n}(z_{J_n , n})|
\,=\,
\lim_{n \to +\infty}
\frac{1}{n}|G'_{n}(z_{J_n , n})|
\,=\,
e^{-c} .$$ 
From
$|G_{n}(y_j)| \geq
\bigl||G_{n}(y_j) - G_{n}(z_{J_n , n} e^{2 i (j - J_n) \pi/n})|
- |G_{n}(z_{J_n , n} e^{2 i (j - J_n) \pi/n})|\bigr|
$ 
and \eqref{chord}
we deduce
\begin{equation}
\label{gagao_0}
|G_{n}(y_j)| \geq
\,  \frac{\pi |z_{J_n , n}|}{a_{\max}}
 e^{-c} - \frac{2 \pi |j - J_n| }{n} .
\end{equation}
But, by definition of $H_n$, still
up to $O(1/n)$-terms,
for $|j - J_n| \leq J_n - H_n - 1$,
\begin{equation}
\label{gagao_1}
\frac{2 \pi |j - J_n|}{n} \leq
\frac{2 \pi \, (J_n - H_n - 1)}{n} =
\arg(z_{J_n , n}) -\arg(z_{H_n +1, n}) 
\leq \frac{e^{-2 c}}{1 - e^{-c}} .
\end{equation}
This inequality is in particular
satisfied for the last two
values of $| j - J_n |$ which are
$J_n - H_n$ and $J_n - H_n +1$
up to $O(1/n)$-terms.
Since the inequality
\begin{equation}
\label{gagao_2}
0.0710\ldots = 
~~2 \frac{e^{-2 c}}{1 - e^{-c}} 
~<~
\frac{\pi |z_{J_n , n}|}{a_{\max}} e^{-c}
~~=~ 0.0914\ldots
\end{equation}
holds, 
from \eqref{gagao_0}, \eqref{gagao_1}
and \eqref{gagao_2},
as soon as $n$ is large enough,
we deduce the Rouch\'e inequality
$$|G_{n}(y_j)| ~\geq~
\, 
\frac{\pi |z_{J_n , n}|}{a_{\max}} 
- \frac{e^{-2 c}}{1 - e^{-c}} 
~\geq~
\frac{e^{-2 c}}{1 - e^{-c}} .$$
Therefore
the conditions of Rouch\'e 
\eqref{rouchecercleNONZERO} 
hold at all the points $y_j$ and $y'_j$ of
\eqref{listeyjyprimej}.

Let us prove that the conditions of 
Rouch\'e \eqref{rouchecercleNONZERO} 
hold on each arc $y'_j~y_{j+1}$ of
the circle $|z| = 1 - c_n/n$,
for $j = 2 H_n - J_n, 2 H_n - J_n + 1 , 
\ldots, 2 J_n - H_n$.
Indeed, from \eqref{zedefi}, 
the derivative
$Z'(\varphi)$ takes a positive value at
the extremity $y'_j$ 
while it takes a negative value
at the other extremity $y_{j+1}$.
$Z(\varphi)$ is almost-periodic
of almost-period $2 \pi/n$. 
The function $\sqrt{Z(\varphi)}$ is increasing
on
$(\arg(z_{j , n}),
 \arg(z_{j , n})+  \frac{\pi}{n})$
and decreasing on 
$(\arg(z_{j , n}) +  \frac{\pi}{n},
\arg(z_{j , n}) + 2 \frac{\pi}{n})$;
on the arc $y'_j~y_{j+1}$ it takes
the value $|G_{n}(y'_j)| 
\geq \frac{e^{-2 c}}{1 - e^{-c}}$,
admits a maximum, and decreases
to 
$|G_{n}(y_{j+1})|
\geq \frac{e^{-2 c}}{1 - e^{-c}}$. 
Hence, \eqref{rouchecercle} holds true for
all $z \in C(0, 1 - c_n/n)$ with
$\arg(y'_j) \leq \arg(z) \leq \arg(y_{j+1})$.

Let us now prove that
the condition of Rouch\'e 
\eqref{rouchecercle} is
satisfied in the angular sector
$0 \leq \arg(z) \leq \arg(z_{H_n , n})$.
Indeed, in this angular sector,
the successive minima of
$\sqrt{Z(\varphi)}$ are all
above 
$\frac{e^{-2 c}}{1 - e^{-c}}$
by the definition of $H_n$ and
\eqref{successiveminima}. Hence the claim.

Let us prove that
the condition of Rouch\'e 
\eqref{rouchecercle} is
satisfied in the angular sector
$\arg(z_{2 J_n - H_n + 1, n}) \leq \arg(z) 
\leq \frac{\pi}{2}$.
In this angular sector,
the oscillations of $\sqrt{Z(\varphi)}$
still occur by the form
of \eqref{zedefi} and
the successive minima of
$\sqrt{Z(\varphi)}$ are all
above 
$\frac{e^{-2 c}}{1 - e^{-c}}$
for $\frac{2 J_n - H_n + 2}{J_n} \leq \arg(z)
\leq \pi/2$, 
by 
\eqref{successiveminima}
for $k \geq J_n - H_n + 1$. 
We deduce the claim.

The condition of Rouch\'e 
\eqref{rouchecercle} is also
satisfied in the angular sector
$\pi \leq \arg(z) \leq \pi/2$,
since then $\cos(\varphi) \leq 0$
and therefore
$\sqrt{Z(\varphi)} \geq 
\sqrt{2 + e^{-2 c} - 4 e^{- c}} =1.15\ldots$.
Since this lower bound is
greater than the
value
$\frac{e^{-2 c}}{1 - e^{-c}} = 0.0354\ldots$
we deduce the claim.

Let us show that 
the conditions of Rouch\'e 
\eqref{rouchecercle} are also
satisfied on the arcs
$C(z_{j,n}, \frac{\pi |z_{j,n}|}{n \, s_{j,n}})
\cap \overline{D}(0, 1 - \frac{c_n}{n})$
for $j= J_n + 1 , \ldots, 2 J_n - H_n +1$.
For such an integer $j$, 
let us denote such an arc
by $y_j~y'_j$.
The two extremities $y_j$ and $y'_j$
of  the arc $y_j~y'_j$ of
the circle
$C(z_{j,n}, \frac{\pi |z_{j,n}|}{n \, s_{j,n}})$
define the same value of the difference cosine,
say 
$X_j := \cos(\arg(y_j - z_{j,n}) - \arg(z_{j,n}))
=
\cos(\arg(y'_j - z_{j,n}) - \arg(z_{j,n}))$,
by \eqref{chord}.
The conditions of Rouch\'e are 
already satisfied at the points
$y_j$ and $y'_j$ by the above.
Recall that, 
for any fixed 
$a \geq 1$, the function
$\kappa(X,a)$,
defined in \eqref{amaximalfunctionX},
is such that
the partial derivative $\partial \kappa_X$
of $\kappa(X,a)$ 
is strictly negative
on the interior of 
$[-1, +1] \times [1, +\infty)$. 
In particular
the function $\kappa(X,s_{j,n})$
is decreasing.
For any point $\Omega$ of the arc
$y_j~y'_j$, we denote by
$X =\cos(\arg(\Omega - z_{j,n}) - \arg(z_{j,n}))$.
We deduce, up to $O(1/n)$-terms,
$$
\frac{e^{-2 c}}{1 - e^{-c}} 
~\leq~ \kappa(X_j,s_{j,n}) ~\leq~ \kappa(X,s_{j,n}),
\qquad \mbox{for all}~~X \in [-1, X_j],$$
hence the result.
\end{proof}

\begin{remark}
\label{lenticulusforalgebraicintegers}
In the case where
$\beta \in (1, \theta_{6}^{-1})$ is an 
algebraic integer
such that
$\beta \not\in
\{\theta_{n}^{-1} \mid n \geq 6\}$,
the lenticulus $\mathcal{L}_{\beta}$ of 
Galois conjugates of $1/\beta$ 
in the angular sector
$\arg z \in \{-\frac{\pi}{3}, +\frac{\pi}{3}\}$ 
is obtained by truncation and 
a slight deformation
of $\lc_{\theta_{\dyg(\beta)}^{-1}}$.
The asymptotic 
expansion of the minorant of the Mahler measure
${\rm M}(\beta)$ will
be obtained from this lenticulus
as a function of the
dynamical degree $\dyg(\beta)$.
\end{remark}

\subsection{Minoration of the Mahler measure: a continuous lower bound}
\label{S5.4}

The passage from the zeroes $\omega_{j,n}$
of the Parry Upper function
$f_{\beta}(z)$ to the zeroes of the minimal polynomial of $\beta$ is crucial.
It is crucial since the Mahler measure
${\rm M}(\beta)$ is constructed from
the roots of $P_{\beta}(z)$, by its very 
definition, and not
from
the roots of the analytic function
$f_{\beta}(z)$ or the poles of
$\zeta_{\beta}(z)$. 

The key result which makes the link
between  
${\rm M}(\beta)$
and $f_{\beta}(z)$ is 
Theorem \ref{splitBETAdivisibility+++}.
Theorem \ref{splitBETAdivisibility+++}
extends Theorem \ref{splitBETAdivisibility}; 
it gives
the extension of the domain
where the minimal polynomial (function) 
$P_{\beta}(z)$ is fracturable and 
for which the power series
$U_{\beta}(z)$ in its decomposition 
is holomorphic, with
nonvanishing properties on
the lenticulus of zeroes of $f_{\beta}(z)$. 
This extension
allows the identification
of the lenticulus of zeroes $\omega_{j,n}$
of $f_{\beta}(z)$
as lenticulus of conjugates of
$\beta$. The complete 
set of conjugates
of $\beta$ is certainly out of reach
by this method, and the 
domain of holomorphy of
$U_{\beta}(z)$ is very probably larger. 
It only gives
a subproduct of the product defining
${\rm M}(\beta)$, hence a minoration
of ${\rm M}(\beta)$.

\begin{theorem}
\label{splitBETAdivisibility+++}
Let $n \geq 260$ and
$\beta > 1$ any algebraic integer
such that
$\dyg(\beta)=n$.
Let $\mathcal{D}_n$ be the 
subdomain of the open unit disc
defined in Theorem \ref{absencezeroesOutside}.
Denote $D_{j,n} := \{z \mid
|z - z_{j,n}| < 
\frac{\pi \, |z_{j,n}|}{ n \, a_{\max}}\}$,
$j = 1, 2, \ldots, J_n$.
Then the domain of holomorphy
of the analytic function
$U_{\beta}(z) = \frac{P_{\beta}(z)}{f_{\beta}(z)}
\in \zb[[z]]$,
quotient of the minimal polynomial of $\beta$
by the Parry Upper function at $\beta$,
contains
the connected domain
$$
\Omega_{n} :=
\dc_n \, \cup \, \bigcup_{j=1}^{J_n}
\, \left(
D_{j,n} \cup \overline{{D_{j,n}}}
\right)
\cup \, D(\theta_{n}, \frac{t_{0,n}}{n}),
$$
itself containing
the open disc
$D(0, 1 - \frac{c_n}{n} -  
\frac{\pi |z_{J_n ,n}|}{n \, a_{\max}})$.
The fracturability of the minimal
polynomial function
$P_{\beta}(z)$
of $\beta$, in
$\zb[[z]]$, obeys the Carlson-Polya dichotomy 
as in Theorem
\ref{splitBETAdivisibility},
admitting the following factorization
$$P_{\beta}(z) = U_{\beta}(z) \times 
f_{\beta}(z)$$ 
where $U_{\beta}(z)$ 
does not vanish on the 
lenticulus $
\lc_{\beta} = \{
\frac{1}{\beta}\}
\cup 
\bigcup_{j=1}^{J_n}
(\{\omega_{j,n}\} 
\cup 
\{\overline{\omega_{j,n}}\}) 
\subset \Omega_n$.
For any zero
$\omega_{j,n} \in \lc_{\beta}$, 
and symmetrically by complex conjugation
with respect to the real axis, we have
\begin{equation}
\label{nevervanishes}
U_{\beta}(\omega_{j,n}) ~=~ \frac{P'_{\beta}(\omega_{j,n})}{f'_{\beta}(\omega_{j,n})}
\neq 0 \quad \mbox{and}
\quad
U_{\beta}(\frac{1}{\beta}) ~=~ \frac{P'_{\beta}(\frac{1}{\beta})}{f'_{\beta}(\frac{1}{\beta})}
\neq 0.
\end{equation}
All the zeroes of the lenticulus
$\lc_{\beta}$ of 
$f_{\beta}(z)$ 
are zeroes of the minimal polynomial 
$P_{\beta}(z)$ of
$\beta$.
\end{theorem}

\begin{proof}
From Theorem \ref{splitBETAdivisibility}, 
a lower bound of
the radius of convergence
$U_{\beta}(z)$ is 
$1 - \frac{\lo (n-1)}{n-1}$. 
Using Theorem \ref{absencezeroesOutside} 
we will show that
this lower bound 
can be improved to 
$1 - \frac{c_n}{n} -  
\frac{\pi |z_{J_n ,n}|}{n \, a_{\max}}$,
and that the domain of holomorphy of
$U_{\beta}(z)$ even extends to 
a collection of
portions of Rouch\'e discs centered
at roots $z_{j,n}$ of $G_n$, given by
$\Omega_{n}$.

For $\beta$ a Parry number
or not, the analytic function 
$f_{\beta}(z)$ is 
holomorphic on $|z| < 1$
by Theorem \ref{dynamicalzetatransferoperator}
and Theorem \ref{parryupperdynamicalzeta}
and the unit circle is not 
the natural boundary
of $f_{\beta}(z)$ or is, accordingly.
Consequently the 
function
$U_{\beta}(z) = P_{\beta}(z)/f_{\beta}(z)$
is analytic
on the open unit disc
and its poles 
are identified from the 
zeroes
$\omega$ of $f_{\beta}(z)$:
(i) the multiple poles 
(multiplicity $\geq 2$)
come from the
zeroes of
$f_{\beta}$ of higher multiplicity, 
in $|z| < 1$, since the roots
of $P_{\beta}(z)$
are all of multiplicity one,
(ii) the simple poles
come from simple zeroes
of $f_{\beta}(z)$
except if
these zeroes $\omega$
are simultaneously
of multiplicity one
and zeroes of the minimal polynomial
$P_{\beta}(z)$.
We will show that, with all the
zeroes of $\lc_{\beta}$, 
we are in the exception of case (ii). 

The characterization
of the subdomains of $D(0,1)$
on which the function  $f_{\beta}(z)$
has no zero, resp. a simple zero, is
given in
Theorem \ref{absencezeroesOutside}
and Lemma \ref{tzero}:
$\dc_n$, resp. 
$D_{j,n}$ and $\overline{D_{j,n}}$
for each $j=1,2, \ldots, J_n$, symmetrically, 
and $D(\theta_n, \frac{t_{0,n}}{n})$.
Therefore  
$U_{\beta}(z)$ has 
no pole in $\dc_n$.
Let us observe that all the arcs
$y_j~y'_j$,
$j \geq J_n + 1$,
given by
\eqref{cerclescalottes} in
Theorem \ref{absencezeroesOutside},
lie outside the open disc
$D(0, 1 - \frac{c_n}{n} -  
\frac{\pi |z_{J_n ,n}|}{n \, a_{\max}})$.
Now, for 
every $j \in \{1, 2, \ldots, J_n\}$
such that $D_{j,n}$ intersects the circle
$|z| = 1 - \frac{c_n}{n} -  
\frac{\pi |z_{J_n ,n}|}{n \, a_{\max}}$ 
the unique zero
$\omega_{j,n}$ of $f_{\beta}(z)$
which belongs to
$D_{j,n}$
either belongs to
$D(0, 1 - \frac{c_n}{n} -  
\frac{\pi |z_{J_n ,n}|}{n \, a_{\max}})$ 
or not. If 
$\omega_{j,n} \not\in  
D(0, 1 - \frac{c_n}{n} -  
\frac{\pi |z_{J_n ,n}|}{n \, a_{\max}})$,
the function
$U_{\beta}(z)$ has no zero and 
no pole in 
$D(0, 1 - \frac{c_n}{n} -  
\frac{\pi |z_{J_n ,n}|}{n \, a_{\max}}) \cap D_{j,n}$. For those
$j$ for which
$\omega_{j,n}
\in
D(0, 1 - \frac{c_n}{n} -  
\frac{\pi |z_{J_n ,n}|}{n \, a_{\max}})
\cap 
D_{j,n}$, 
the derivation of the
formal identity:
$P_{\beta}(X) = 
U_{\beta}(X) \times f_{\beta}(X)$ 
gives:
\begin{equation}
\label{decompoDERIVALPHA}
P'_{\beta}(X) ~=~
U'_{\beta}(X) \,
f_{\beta}(X)
~+~
U_{\beta}(X) \,
f'_{\beta}(X).
\end{equation}
Since $\omega_{j,n}$ is a simple zero,
$f'_{\beta}(\omega_{j,n}) \neq 0$.
Specializing the formal
variable $X$ to 
$\omega_{j,n}$ 
we obtain:
$$P'_{\beta}(\omega_{j,n}) ~=~
U_{\beta}(\omega_{j,n}) \,
f'_{\beta}(\omega_{j,n}),$$
that is \eqref{nevervanishes}
with
$|U_{\alpha}(\omega)|\neq +\infty$,
meaning that
$\omega_{j,n}$ is not a pole of
$U_{\beta}(z)$.
The zeroes of $\lc_{\beta}$ are not 
singularities of $U_{\beta}(z)$. 
Let us show
that
$U_{\beta}(\omega_{j,n}) \neq 0$.
Let us assume the contrary.
Since 
$P_{\beta}(X)$ has simple roots, the value
$P_{\beta}(\omega_{j,n})$ is either
nonzero
or zero with multiplicity one.
But, from
the relation
$P_{\beta}(z) = U_{\beta}(z) \times
f_{\beta}(z)$,
it would imply that
the multiplicity of
the
zero $z= \omega_{j,n}$ of
$P_{\beta}(z)$ is
$\geq 2$. Contradiction.

With the same arguments, for 
every $j \in \{1, 2, \ldots, H_n\}$ 
such that
$D_{j,n} \subset 
D(0, 1 - \frac{c_n}{n} -  
\frac{\pi |z_{J_n ,n}|}{n \, a_{\max}})
$, 
the (simple) 
zeroes $\omega_{j,n}$
of
the Parry Upper function $f_{\beta}(z)$
are such that
$|U_{\beta}(1/\beta)| \neq 0, \neq +\infty$.
In a similar way, the (simple) zero
$1/\beta$ of $P_{\beta}(z)$ is a (simple)
zero of
the Parry Upper function $f_{\beta}(z)$
at $\beta$ for which 
$|U_{\beta}(1/\beta)| \neq 0, \neq \infty$. 

We deduce that all the lenticular
zeroes of $\lc_{\beta}$ are zeroes
of the minimal polynomial 
$P_{\beta}(X)$ of $\beta$.
\end{proof}

Let $n \geq 260$ and
$\beta > 1$ be an algebraic integer
such that
$\theta_{n-1}^{-1} < \beta <
\theta_{n}^{-1}$. The 
minimal polynomial of $\beta$
is factorized 
under the form:
\begin{equation}
\label{polyminimalLENTICULUS}
P_{\beta}(z)
~=~
\prod_{\omega \in \lc_{\beta}}
(z - \omega) ~\times
~
\prod_{\omega \not\in \lc_{\beta}}
(z - \omega).
\end{equation}
 
Flatto, 
Lagarias and Poonen
\cite{flattolagariaspoonen}
have shown and studied the 
continuity of
the modulus of the second 
smallest root
of $f_{\beta}(z)$
as a function of 
$\beta$.
Theorem \ref{Qautocontinus} 
extends this result.

\begin{theorem}
\label{Qautocontinus}
Let $n \geq 260$. Let
$\beta > 1$ be an algebraic integer
such that
$\dyg(\beta)
=
n$.
The product, called lenticular
Mahler measure of $\beta$, defined by
\begin{equation}
\label{mmrDEF}
{\rm M}_{r}(\beta) ~:=~
\prod_{\omega \in \lc_{\beta}}
|\omega|^{-1}
\end{equation}
is a continuous function
of $\beta$ on the open interval
$(\theta_{n}^{-1}, \theta_{n-1}^{-1})$, 
which admits
the following left and right limits
\begin{equation}
\label{mmmr1}
\lim_{\beta \to \theta_{n-1}^{{-1}^{\mbox{}\, -}}}
{\rm M}_{r}(\beta)
=
\prod_{
{\omega \in \lc_{\theta_{n-1}^{-1}}}}
|\omega|^{-1}
=
\theta_{n-1}^{-1}   \times
\prod_{\stackrel{1 \leq j \leq J_{n}}
{z_{j,n-1} \in \lc_{\theta_{n-1}^{-1}}}}
|z_{j,n-1}|^{-2} ,
\end{equation}
\begin{equation}
\label{mmmr2}
\lim_{\beta \to \theta_{n}^{{-1}^{\mbox{}\, +}}}
{\rm M}_{r}(\beta)
=
\prod_{
{\omega \in \lc_{\theta_{n}^{-1}}}}
|\omega|^{-1}
= 
\theta_{n}^{-1}   \times
\prod_{\stackrel{1 \leq j \leq J_{n}}
{z_{j,n} \in \lc_{\theta_{n}^{-1}}}}
|z_{j,n}|^{-2} .
\end{equation}
The discontinuity (jump) of
${\rm M}_{r}(\beta)$ 
at the Perron number
$\theta_{n-1}^{-1}$, given in the multiplicative form by
\begin{equation}
\label{mmmr3}
\frac{\lim_{\beta \to \theta_{n-1}^{{-1}^{\mbox{}\, -}}}
{\rm M}_{r}(\beta)}{\lim_{\beta \to \theta_{n-1}^{{-1}^{\mbox{}\, +}}}
{\rm M}_{r}(\beta)} ~~=~~ |z_{J_n, n-1}|^{-2},
\end{equation}
tends to 1 (i.e. disappears at infinity)
when
$n = \dyg(\beta)$ tends to infinity.\end{theorem}

\begin{proof}
From Corollary
\ref{zeroesParryUpperfunctionContinuity}
in \S \ref{S4.4} all the maps
$\beta \to \omega(\beta) \in \lc_{\beta}$
are continuous.
Now the identification
of the zeroes of the Parry Upper function
$f_{\beta}(z)$ as conjugates of
$\beta$, from 
Theorem \ref{splitBETAdivisibility+++},
allows to consider this
continuity property
as a continuity property over the
conjugates of $\beta$ which
define the lenticulus
$\lc_{\beta}$.
As a consequence all the maps
$\beta \to |\omega(\beta)| \in \lc_{\beta}$, 
are continuous, as well as
their product \eqref{mmrDEF}.

Let $1 \leq j \leq J_n$.
Let us prove that
$z_{j,n-1} \in D_{j,n}
=
\{z \mid |z-z_{j,n}| < \frac{\pi |z_{j,n}|}{n \, a_{\max}}\}$.
Indeed,
$$|z_{j,n}| = 1 + 
\frac{1}{n} \lo (2 \sin(\frac{\pi \, j}{n}))
 +..., \qquad \arg(z_{j,n}) = ...$$
 and
 $$|z_{j,n-1}| = 1 + 
\frac{1}{n-1} \lo (2 \sin(\frac{\pi \, j}{n-1}))
 +..., \qquad \arg(z_{j,n-1}) =...$$
 so that, easily, 
\begin{equation}
\label{zedeJNmoinsUNdedansDjn}
|z_{j,n} - z_{j,n-1}|
 <~
\frac{\pi |z_{j,n}|}{n \, a_{\max}}.
\end{equation}
The image of the
interval $(\theta_{n}^{-1}, 
\theta_{n-1}^{-1}) \cap 
\mathcal{O}_{\overline{\qb}}$
by a map
$\beta \to \omega_{j,n}(\beta) \in \lc_{\beta}$
is a curve in $D_{j,n}$
over $\mathcal{O}_{\overline{\qb}}$
with extremities 
$z_{j,n}$ and $z_{j,n-1}$, both
in $D_{j,n}$ 
by \eqref{zedeJNmoinsUNdedansDjn}. 
This curve does not
intersect itself. Indeed, 
if it would be 
a self-intersecting curve we would have,
for two distinct algebraic integers
$\beta$ and $\beta'$, the same
conjugate in $D_{j,n}$,
what is impossible since $P_{\beta}$
and $P'_{\beta}$ are both irreducible, and
therefore 
they cannot have a root
in common. This curve does not ramify either
by the uniqueness property imposed locally
by 
the Theorem of Rouch\'e. 
We deduce the left limit \eqref{mmmr1}
and the right limit
\eqref{mmmr2}
by continuity.
\end{proof}

The subscript ``r" added to the ``M"
of the Mahler measure
stands for ``reduced to the lenticulus".

\begin{remark}
\label{measuremahlercontinuityAUXPerrons}
Decomposing the Mahler measure gives
$${\rm M}(\beta) =
\prod_{\omega \in \lc_{\beta}} |\omega|^{-1}
\times
\!\!\! \prod_{\stackrel{\omega \not\in \lc_{\beta}, |\omega|< 1}{P_{\beta}(\omega)=0} } 
|\omega|^{-1} .
$$
Theorem \ref{Qautocontinus}, 
for which the Rouch\'e method
has been applied, shows the continuity
of the partial
product
$\beta \to
\prod_{\omega \in \lc_{\beta}} |\omega|^{-1}$, associated with the
identified lenticulus of 
conjugates of $\beta$, 
with $\beta$ running over each open
interval of extremities 
two successive 
Perron numbers
$\theta_{n}^{-1}$.
It is very 
probable that a method finer than the method of
Rouch\'e would lead to a higher value of
$J_n$, to more zeroes of
$f_{\beta}(z)$
identified as conjugates of $\beta$,
and the 
disappearance of the discontinuities
(jumps) in \eqref{mmmr3}.

\end{remark}

\begin{theorem}
\label{MahlerMINORANTreal}
Let $\beta > 1$ be an algebraic integer
such that 
$\dyg(\beta) \geq 260$.
Denote $\kappa = \kappa(1,a_{\max})$.
The Mahler measure ${\rm M}(\beta)$
is bounded from below by
the lenticular Mahler measure of $\beta$
as
$${\rm M}(\beta) =
~ {\rm M}_{r}(\beta) \times
\!\!\!\prod_{\stackrel{\omega \not\in \lc_{\beta}, |\omega|< 1}{P_{\beta}(\omega)=0}} 
|\omega|^{-1}
\geq 
~ {\rm M}_{r}(\beta).$$
Denoting
$$
\Lambda_r := \exp\Bigl(
\frac{-1}{\pi}
\int_{0}^{2 \arcsin(\frac{\kappa}{2})}
 \, \lo\bigl(
2 \sin\bigl(
\frac{x}{2}
\bigr)
\bigr)
dx
\Bigr)
=
1.16302\ldots,$$
and
$$
\mu_r :=
\exp\Bigl(
\frac{-1}{\pi}
\int_{0}^{2 \arcsin(\frac{\kappa}{2})}
\! \!\lo\Bigl[\frac{1 + 2 \sin(\frac{x}{2})
-
\sqrt{1 - 12 \sin(\frac{x}{2}) 
+ 4 (\sin(\frac{x}{2}))^2}}{8 \sin(\frac{x}{2})}
\Bigr] dx
\Bigr)$$
$$
=
0.992337\ldots,
$$
the lenticular Mahler measure 
${\rm M}_{r}(\beta)$ of $\beta$
admits a liminf and a limsup 
when $\beta$ tends to $1^{+}$, 
equivalently
when $\dyg(\beta)$ tends to infinity,
respectively bounded from below 
and above as
\begin{equation}
\label{inekinf}
\liminf_{\dyg(\beta) \to +\infty} 
\!\!{\rm M}_{r}(\beta)
~\geq~
\, 
\Lambda_r \cdot \mu_r
~=~ 1.15411\ldots ,
\end{equation}
\begin{equation}
\label{ineksup}
\limsup_{\dyg(\beta) \to +\infty} 
\!\!{\rm M}_{r}(\beta)
~\leq~
\, 
\Lambda_r  \cdot \mu_{r}^{-1}
~=~
1.172\ldots
\end{equation}
Then the ``limit
minorant" of the Mahler
measure ${\rm M}(\beta)$ of $\beta$, 
$\beta > 1$ running over
$\mathcal{O}_{\overline{\qb}}$,
when
$\dyg(\beta)$ tends to infinity, is given by
\begin{equation}
\label{inekmahler}
\liminf_{\dyg(\beta) 
\to \infty}{\rm M}(\beta) 
~\geq~
\Lambda_r \cdot \mu_r = 
1.15411\ldots
\end{equation}
\end{theorem}

\begin{proof}
The value
$2 \arcsin(\kappa(1,a_{\max})/2)
= 0.171784\ldots$ is given
by Proposition
\ref{argumentlastrootJn},
and $a_{\max} = 5.8743\ldots$ by 
Theorem \ref{cercleoptiMM}.
The variations of the 
Mahler measure ${\rm M}(\beta)$
of $\beta$ can be fairly large when
$\beta$ approaches $1^{+}$. 
On the contrary 
the lenticular Mahler measure
${\rm M}_{r}(\beta)$ is a continuous
function of $\beta$ on
$(1, \theta_{260}^{-1})$
except at the point discontinuities
which are
the Perron numbers
$\theta_{n}^{-1}$ by 
\eqref{mmmr1}, \eqref{mmmr2}
and \eqref{mmmr3},
writing
$n = \dyg(\beta)$ for short.

First, by Proposition
\ref{argumentlastrootJn}
let us observe that 
the Riemann-Stieltjes sum
$$
S(f,n) := - 2 \sum_{j = 1}^{J_n}
\frac{1}{n} \, \lo \bigl(2 \, \sin\bigl(\frac{\pi j}{n}\bigr)  \bigr)
=
\frac{-1}{\pi} \,
\sum_{j= 1}^{J_n}
(x_{j} - x_{j-1}) f(x_j)$$
with
$x_j = \frac{2 \pi j}{n}$
and
$f(x) := \lo \bigl(2 \, \sin\bigl(\frac{x}{2}\bigr)  \bigr)
$
converges to the limit
\begin{equation}
\label{SfnLAMBDAr}
\lim_{n \to \infty} S(f,n) ~=~ \frac{-1}{\pi} \,
\int_{0}^{0.171784\ldots} f(x) dx 
~=~ 
\lo \, \Lambda_r 
~=~ \lo (1.16302\ldots).
\end{equation}
This limit is a log-sine
integral 
\cite{borweinborweinstraubwan}
\cite{borweinstraub}.
Let us now show how $\Lambda_r$
is related to  
$\liminf_{\dyg(\beta) \to \infty}{\rm M}_{r}(\beta)$  
and 
$\limsup_{\dyg(\beta) \to \infty}{\rm M}_{r}(\beta)$
to deduce
\eqref{inekinf} and \eqref{ineksup}.

Taking only into account the lenticular 
zeroes of $P_{\beta}(z)$, which constitute
the lenticulus $\lc_{\beta}$,
from Theorem \ref{cercleoptiMM}
and
Proposition \ref{cercleoptiMMBUMP}, 
we obtain
$$
\lo {\rm M}_{r}(\beta) =
- \lo (\frac{1}{\beta}) - 
2 \sum_{j=1}^{J_n} \lo |\omega_{j,n}|
= \lo (\frac{1}{\beta}) - 
2 \sum_{j=1}^{J_n} \lo |(\omega_{j,n}-z_{j,n})
+z_{j,n}|
$$
\begin{equation}
\label{logmmmr}
=~  \lo (\beta) - 
2 \sum_{j=1}^{J_n} \lo |z_{j,n}|
- 2 \sum_{j=1}^{J_n} 
\lo \bigl|1 + \frac{\omega_{j,n}-z_{j,n}}
{z_{j,n}}\bigr|.
\end{equation}
Obviously the first term 
of \eqref{logmmmr} tends to 0 when
$\dyg(\beta)$ tends to $+\infty$
since $\displaystyle \lim_{n \to \infty} \theta_n = 1$
(Proposition \ref{thetanExpression}).
Let us turn to the 
third summation in \eqref{logmmmr}.
The $j$-th root 
$\omega_{j,n} \in \lc_{\beta}$ of
$f_{\beta}(z)$ 
is the unique root
of $f_{\beta}(z)$ in the disc
$D_{j,n}
=\{z \mid |\omega_{j,n}-z_{j,n}| < \frac{\pi |z_{j,n}|}{n \, a_{\max}}\}$.
From Theorem \ref{omegajnexistence} 
we have the more precise localization 
in $D_{j,n}$:
$|\omega_{j,n}-z_{j,n}|
< \frac{\pi |z_{j,n}|}{n \, a_{j, n}}$
for $j = \lceil v_n \rceil, 
\ldots, J_n$ (main angular sector),
with
$$D(\frac{\pi}{a_{j,n}}) 
=
\lo\Bigl[\frac{1 + B_{j,n}
-
\sqrt{1 - 6 B_{j,n} 
+ B_{j,n}^2}}{4 B_{j,n}}
\Bigr]
$$ 
and 
$B_{j,n} =
2 \sin(\frac{\pi j}{n})
\Bigl(
1 - \frac{1}{n} \lo (2 \sin(\frac{\pi j}{n}))
\Bigr)
$ (from \eqref{petitcercle_jmainDBJN}).
For $j = \lceil v_n \rceil, 
\ldots, J_n$
the following inequalities hold:
$$1 - \frac{1}{n} D(\frac{\pi}{a_{j,n}})
\leq
|1 + \frac{\omega_{j,n}-z_{j,n}}
{z_{j,n}}|
\leq 
1 + \frac{1}{n} D(\frac{\pi}{a_{j,n}}),$$
up to second order terms.  
Let us apply
the remainder Theorem
of alternating series:
for $x$ real,
$|x| < 1$,
$|\lo(1+x) - x | ~\leq~ \frac{x^2}{2}$. 
Then the third summation 
in \eqref{logmmmr} satisfies 
$$-2 \,\lim_{n \to \infty} 
\sum_{j=1}^{J_n}
\frac{1}{n}
\lo\Bigl[\frac{1 + 2 \sin(\frac{\pi j}{n})
-
\sqrt{1 - 12 \sin(\frac{\pi j}{n}) 
+ 4 (\sin(\frac{\pi j}{n}))^2}}{8 \sin(\frac{\pi j}{n})}
\Bigr]
$$ 
\begin{equation}
\label{inflenticulusMAHLERm}
\leq
\liminf_{n \to \infty}
\left(
- 2 \sum_{j=1}^{J_n} 
\lo \bigl|1 + \frac{\omega_{j,n}-z_{j,n}}
{z_{j,n}}\bigr|\right)
\end{equation}
and
$$
\limsup_{n \to \infty}
\left(- 2 \sum_{j=1}^{J_n} 
\lo \bigl|1 + \frac{\omega_{j,n}-z_{j,n}}
{z_{j,n}}\bigr|
\right) \leq
$$
\begin{equation}
\label{suplenticulusMAHLERm}
+ 2 \,\lim_{n \to \infty} 
\sum_{j=1}^{J_n}
\frac{1}{n}
\lo\Bigl[\frac{1 + 2 \sin(\frac{\pi j}{n})
-
\sqrt{1 - 12 \sin(\frac{\pi j}{n}) 
+ 4 (\sin(\frac{\pi j}{n}))^2}}{8 \sin(\frac{\pi j}{n})}
\Bigr]
\end{equation}
Let us convert the limits to
integrals. 
The Riemann-Stieltjes sum
$$
S(F,n) :=
- 2 \sum_{j = 1}^{J_n}
\frac{1}{n} \, \lo \Bigl[
\frac{1 + 2 \sin(\frac{\pi j}{n})
-
\sqrt{1 - 12 \sin(\frac{\pi j}{n}) 
+ 4 (\sin(\frac{\pi j}{n}))^2}}{8 \sin(\frac{\pi j}{n})}
\Bigr]
$$
$$
=
\frac{-1}{\pi} \,
\sum_{j= 1}^{J_n}
(x_{j} - x_{j-1}) F(x_j)$$
with
$x_j = \frac{2 \pi j}{n}$
and
$F(x) := \lo \Bigl[\frac{1 + 2 \sin(\frac{x}{2})
-
\sqrt{1 - 12 \sin(\frac{x}{2}) 
+ 4 (\sin(\frac{x}{2}))^2}}{8 \sin(\frac{x}{2})}
\Bigr]$
converges to the limit
\begin{equation}
\label{limiteF}
\lim_{n \to \infty} S(F,n) 
= \frac{-1}{\pi} \,
\int_{0}^{0.171784\ldots} F(x) dx 
= 
\lo \, \mu_r 
\quad \mbox{with}
\quad
\mu_r ~=~ 0.992337\ldots.
\end{equation}
From the inequalities 
\eqref{inflenticulusMAHLERm}
and
\eqref{suplenticulusMAHLERm},
with the limit
\eqref{limiteF} as an integral,
and 
by taking the exponential of
\eqref{logmmmr}, we
obtain 
the two multiplicative factors
$\mu_r$
and
$\mu_{r}^{-1}$
of $\Lambda_r$
in
\eqref{inekinf}, resp. in
\eqref{ineksup}.

Let us show that the second 
summation in \eqref{logmmmr}
gives the term
$\Lambda_r$ 
in the 
inequalities \eqref{inekinf} 
and \eqref{ineksup}, 
when $n$ tends to infinity.
From \eqref{SfnLAMBDAr} 
it will suffice to show
that 
\begin{equation}
\label{limmm}
\lim_{n \to \infty} S(f,n) = 
-2 \, \lim_{n \to \infty} \sum_{j=1}^{J_n} \lo |z_{j,n}| 
\end{equation} 
The identity \eqref{limmm}
only concerns the roots
of the trinomials $G_n$. It was already 
proved to be true, but 
with $\lfloor n/6\rfloor$ 
instead of $J_n$ as maximal index $j$, 
in the summation,
in \cite{vergergaugry6} \S 4.2, pp 111--115.
The arguments of the proof are the same,
the domain of integration being now
$(0, \lim_{n \to \infty} 2 \pi \frac{J_n}{n}]$
given by
Proposition \ref{argumentlastrootJn}. 
\end{proof}

\subsection{Poincar\'e asymptotic expansion of the lenticular Mahler measure}
\label{S5.5}

The aim of this subsection is to prove 
Theorem \ref{Lrasymptotictheorem}, 
in the continuation of the last paragraph.

The logarithm of the
lenticular Mahler measure 
${\rm M}_{r}(\beta)$
of
$\beta > 1$, with
$\dyg(\beta) \geq 260$, given 
by \eqref{logmmmr}, 
admits the lower bound 
\begin{equation}
\label{logmmmrMINORANT}
L_{r}(\beta) 
=  \lo (\beta) - 
2 \sum_{j=1}^{J_n} \lo |z_{j,n}|
- 2 \sum_{j=1}^{\lfloor v_n \rfloor} 
\lo (1 + \frac{\pi}{n \, a_{max}})
- 2 \sum_{j=\lceil v_n \rceil}^{J_n} 
\lo (1 + \frac{\pi}{n \, a_{j,n}})
\end{equation}
which 
is only a function
of $n = \dyg(\beta)$, where
$(a_{j,n})$ is given 
by Theorem \ref{omegajnexistence},
the sequence $(v_n)$ by the Appendix,
and $J_n$ by Definition
\ref{Jndefinition}
and Proposition \ref{argumentlastrootJn}. 
From \eqref{inflenticulusMAHLERm},
\eqref{limiteF}
and \eqref{limmm}, the limit is
$\lim_{\dyg(\beta) \to \infty}
L_{r}(\beta) = \lo \Lambda_r + \lo \mu_r$.
In Theorem \ref{Lrasymptotictheorem}, 
we will gather the asymptotic contributions 
of each term and obtain  
the asymptotic expansion 
of $L_{r}(\beta)$
as a function of $n$.

(i) First term in \eqref{logmmmrMINORANT}:
from Lemma \ref{remarkthetan}
and 
Theorem \ref{betaAsymptoticExpression},
\begin{equation}
\label{sz1}
\lo (\beta) = 
\frac{\lo n}{n}(1 - \lambda_n) +
\frac{1}{n} O\left(
\left(
\frac{\lo \lo n}{\lo n}
\right)^2 
\right)
=
O\left(
\frac{\lo n}{n}
\right);
\end{equation}

(ii) second term in \eqref{logmmmrMINORANT}: 
from Proposition
\ref{zedeJImodulesORDRE3},
$\displaystyle
\sum_{j= \lceil v_n \rceil}^{J_n}
\lo |z_{j,n}| =$
$$
\sum_{j= \lceil v_n \rceil}^{J_n} \lo \Bigl(
1 + \frac{1}{n} \, \lo \bigl(2 \, \sin\bigl(\frac{\pi j}{n}\bigr)  \bigr)
+ \frac{1}{2 n}
\left( \frac{\lo \lo n}{\lo n} \right)^2
+
 \frac{1}{n} O\left(
\frac{(\lo \lo n)^2}{(\lo n)^3}
\right) \Bigr)$$
with the constant 1 involved in the Big O. Let us 
apply the remainder Theorem of alternating series: 
for $x$ real, $|x|<1$,
$\left| \lo (1+x)  - x \right| \leq 
\frac{x^2}{2}$.
Then
$$\left|
\sum_{j= \lceil v_n \rceil}^{J_n}
\lo |z_{j,n}|
-
\sum_{j= \lceil v_n \rceil}^{J_n}
\frac{1}{n} \, \lo \bigl(2 \, \sin\bigl(\frac{\pi j}{n}\bigr)  \bigr)
- \sum_{j= \lceil v_n \rceil}^{J_n}
\frac{1}{2 n}
\left( \frac{\lo \lo n}{\lo n} \right)^2
\right| $$
$$\leq
\sum_{j= \lceil v_n \rceil}^{J_n}
\frac{1}{n} \left| O
\left( \frac{(\lo \lo n)^2}{(\lo n)^3} \right) \right|
$$ 
\begin{equation}
\label{mmm03_mieux}
+
\frac{1}{2} 
\sum_{j= \lceil v_n \rceil}^{J_n}
\frac{1}{n^2}
\left[
\lo \bigl(2 \, \sin\bigl(\frac{\pi j}{n}\bigr)  \bigr)
+ 
\frac{1}{2}
\left( \frac{\lo \lo n}{\lo n} \right)^2
+
O\left(
\frac{(\lo \lo n)^2}{(\lo n)^3} \right)
\right]^2 .
\end{equation}
For $1 \leq j \leq J_n$,
the inequalities
$0 < 2 \sin(\pi j/n)\leq 1$
and
$\lo(2 \sin(\pi j/n))<0$ hold.
Then 
$|\lo(2 \sin(\pi j/n)|
\leq 
|\lo(2 \sin(\pi/n))|=O(\lo n)
$.
On the other hand, 
the two $O(~)$s in the rhs of
\eqref{mmm03_mieux}
involve a constant which 
does not depend upon $j$.
Therefore, from Proposition 
\ref{argumentlastrootJn},
the rhs of
\eqref{mmm03_mieux} is
$$=
O\left(
\Bigl(\frac{(\lo \lo n)^2}{(\lo n)^3}
\Bigr)
\right)
+ O\left(
\frac{\lo^2 n}{n}
\right)
=
O\left(
\Bigl(\frac{(\lo \lo n)^2}{(\lo n)^3}
\Bigr)
\right).
$$

On the other hand, 
the two regimes of asymptotic
expansions
in the Bump give (Appendix)
$$  
\sum_{j= \lceil u_n \rceil}^{\lfloor v_n \rfloor}
\lo |z_{j,n}|
= O\left(\frac{(\lo n)^{2+\epsilon}}{n}
\right),\quad 
\sum_{j= 1}^{\lfloor u_n \rfloor}
\lo |z_{j,n}|
=
O\left(\frac{(\lo n)^2}{n}
\right)$$
and
$$\sum_{j=\lceil \lo n \rceil}^{\lceil v_n \rceil}
\frac{2}{n} \lo \Bigl(
2 \sin(\frac{\pi j}{n})
\Bigr)
=
O\left(
\frac{(\lo n)^{2+\epsilon}}{n}
\right).
$$
Therefore
\begin{equation}
\label{sz2}
-2
\sum_{j=1}^{J_n}
\lo |z_{j,n}|
=
-
\sum_{j=\lceil \lo n \rceil}^{J_n}
\frac{2}{n} \, 
\lo \bigl(2 \, \sin\bigl(\frac{\pi j}{n}\bigr)  
\bigr)
+ O\left(
\left( \frac{\lo \lo n}{\lo n} \right)^2
\right)
\end{equation}
with the constant 
$\frac{1}{2 \pi} \arcsin(\frac{\kappa}{2})$ 
(from Proposition \ref{argumentlastrootJn})
involved in the Big O. 

(iii) third term in \eqref{logmmmrMINORANT}: 
with the definition
of $\epsilon$ and $(v_n)$ (Appendix),
\begin{equation}
\label{sz3}
- 2 \sum_{j=1}^{\lfloor v_n \rfloor} 
\lo (1 + \frac{\pi}{n \, a_{max}})
=
O\left(
\frac{(\lo n)^{1+\epsilon}}{n}
\right);
\end{equation}

(iv) fourth term in \eqref{logmmmrMINORANT}: 
from the 
Theorem of alternating series,
\begin{equation}
\label{ineqmu_r}
|\sum_{j=\lceil v_n \rceil}^{J_n} 
\lo (1 + \frac{\pi}{n \, a_{j,n}})
- \sum_{j=\lceil v_n \rceil}^{J_n}
\frac{1}{n}
D(\frac{\pi}{a_{j,n}})
- \sum_{j=\lceil v_n \rceil}^{J_n}
\frac{1}{n}
{\rm tl}(\frac{\pi}{a_{j,n}})|
\leq
\frac{1}{2}
\sum_{j=\lceil v_n \rceil}^{J_n}
\left(
\frac{\pi}{n \, a_{j,n}}
\right)^2 .
\end{equation}
The terminant
${\rm tl}(\frac{\pi}{ a_{j,n}})
=
O \Bigl(
\frac{(\lo \lo n)^2}{(\lo n)^3}
\Bigr)$
is given by 
\eqref{petitcercle_jmainTL_majorant}.
From Theorem \ref{omegajnexistence}, 
with $B_{j,n} =
2 \sin(\frac{\pi j}{n})
\Bigl(
1 - \frac{1}{n} \lo (2 \sin(\frac{\pi j}{n}))
\Bigr)
$,
it is easy to show
$$
D(\frac{\pi}{a_{j,n}}) 
=
\lo\Bigl[\frac{1 + B_{j,n}
-
\sqrt{1 - 6 B_{j,n} 
+ B_{j,n}^2}}{4 B_{j,n}}
\Bigr]$$
$$
= 
\lo\Bigl[\frac{1 + 
2 \sin(\frac{\pi j}{n})
-
\sqrt{1 - 12 \sin(\frac{\pi j}{n})
+ 4 \sin(\frac{\pi j}{n})^2}}
{8 \sin(\frac{\pi j}{n})}
\Bigr]
+O\left(
\frac{\lo n}{n}
\right).
$$ 
The rhs of \eqref{ineqmu_r} is
$= O\left(\frac{1}{n}
\right)$.
Then~
$- 2 \, \sum_{j=\lceil v_n \rceil}^{J_n} 
\lo (1 + \frac{\pi}{n \, a_{j,n}})
=$
\begin{equation}
\label{sz4}
\sum_{j=\lceil v_n \rceil}^{J_n}
\frac{-2}{n}\lo\Bigl[\frac{1 + 
2 \sin(\frac{\pi j}{n})
-
\sqrt{1 - 12 \sin(\frac{\pi j}{n})
+ 4 \sin(\frac{\pi j}{n})^2}}
{8 \sin(\frac{\pi j}{n})}
\Bigr]
+ O\Bigl(\frac{(\lo \lo n)^2}{(\lo n)^3}
\Bigr) .
\end{equation}
The summation 
$\sum_{j=\lceil v_n \rceil}^{J_n}$ can 
be replaced by 
$\sum_{j=\lceil \lo n \rceil}^{J_n}$. 
Indeed, from the definition of the sequence
$(v_n)$
(Appendix),
$$\sum_{j=\lceil \lo n \rceil}^{\lceil v_n \rceil}
\frac{2}{n} \, 
\lo\Bigl[\frac{1 + 
2 \sin(\frac{\pi j}{n})
-
\sqrt{1 - 12 \sin(\frac{\pi j}{n})
+ 4 \sin(\frac{\pi j}{n})^2}}
{8 \sin(\frac{\pi j}{n})}
\Bigr]
= 
O\left(
\frac{(\lo n)^{2+\epsilon}}{n}
\right).$$
Inserting the 
contributions 
\eqref{sz1} 
\eqref{sz2}
\eqref{sz3}
\eqref{sz4} in 
\eqref{logmmmrMINORANT} leads to
$$
L_{r}(\beta) =
 \lo \Lambda_r + \lo \mu_r
+
\Bigl(-
\lo \Lambda_r
-
\sum_{j=\lceil \lo n \rceil}^{J_n}
\frac{2}{n} \, \lo \bigl(2 \, \sin\bigl(\frac{\pi j}{n}\bigr)  \bigr)
\Bigr)$$
$$
+
\Bigl(-
\lo \mu_r
-
\sum_{j=\lceil \lo n \rceil}^{J_n}
\frac{2}{n} \, \lo \bigl(
\frac{1 + 
2 \sin(\frac{\pi j}{n})
-
\sqrt{1 - 12 \sin(\frac{\pi j}{n})
+ 4 \sin(\frac{\pi j}{n})^2}}
{8 \sin(\frac{\pi j}{n})}
\bigr)  \bigr)
\Bigr)
$$
\begin{equation}
\label{esti}
+O\Bigl(
\Bigl( \frac{\lo \lo n}{\lo n} \Bigr)^2
\Bigr)
\end{equation}
with the constant 
$\frac{1}{2 \pi} \arcsin(\frac{\kappa}{2})$ 
involved in the Big O.
Let us denote by
$\Delta_1$ the first
term within brackets,
resp.
$\Delta_2$ the second term
within brackets,  in \eqref{esti}
so that
\begin{equation}
\label{LrDelta1Delta2}
D(L_{r}(\beta)) = \lo ( \Lambda_r \mu_r ) + \Delta_1 + \Delta_2.
\end{equation}

\noindent
{\em Calculation of $|\Delta_1|$:}
let us estimate 
and give an upper bound of $|\Delta_1| =$
\begin{equation}
\label{diffloglambda}
\left|
\frac{-1}{\pi}\int_{0}^{2 \arcsin(\kappa/2)}
\lo \Bigl(
2 \sin (x/2)
\Bigr) dx 
-
\sum_{j=\lceil \lo n \rceil}^{J_n}
\frac{-2}{n} \, \lo \bigl(2 \, \sin\bigl(\frac{\pi j}{n}\bigr)  \bigr)
\right|.
\end{equation}
In \eqref{diffloglambda} the sums are truncated 
Riemann-Stieltjes sums of $\lo \Lambda_r$, the 
integral being 
$\lo \Lambda_r$. Referring to Stoer and Bulirsch 
(\cite{stoerbulirsch}, pp 126--128) we now replace 
$\lo \Lambda_r$ by an approximate value obtained 
by integration of an interpolation polynomial
by the methods of Newton-Cotes; we just need to know this approximate value
up to 
$O\bigl(
\Bigl( \frac{\lo \lo n}{\lo n} \Bigr)^2
\bigr)$. 
Up to 
$O\bigl(
\Bigr( \frac{\lo \lo n}{\lo n} 
\Bigr)^2
\bigr)$, we will show that: 

\noindent
(i--1) {\it an upper bound of 
\eqref{diffloglambda}  is ($\kappa$ stands for
$\kappa(1,a_{max})$ as in} Proposition
\ref{argumentlastrootJn})
$$ \frac{\arcsin(\kappa/2)}{\pi} \, \frac{1}{\lo n},$$ 
(ii--1) {\it the approximate value of $\lo \Lambda_r$ is independent of the integer $m$ (i.e. step length) used in the Newton-Cotes formulas, assuming the weights $(\alpha_q)_{q=0, 1, \ldots, m}$ associated with $m$ all positive. Indeed, if $m$ is arbitrarily large, the estimate of the integral should be very good by these methods, ideally exact at the limit ($m ``=" +\infty$)}. 

\vspace{0.2cm}

\noindent
{\it Proof of (i--1)}: 
we consider the decomposition of the interval of integration as

\noindent
$\bigl(0, 2 \arcsin(\kappa/2)\bigr] =$ 
\begin{equation}
\label{intervaldecomposition}
\bigl(0, \frac{2 \pi \lceil \lo n \rceil}{n} \bigr]
\cup \Bigl(
\bigcup_{j=\lceil \lo n \rceil}^{J_n - 1}
\bigl[
\frac{2 \pi j}{n}, \frac{2 \pi (j+1)}{n}
\bigr]
\Bigr)
\cup 
\Bigl[\frac{2 \pi J_n}{n}, 
2 \arcsin(\kappa/2)\Bigr]
\end{equation}
and proceed by calcutating the estimations of 
\begin{equation}
\label{estijj}
\left|
\frac{-1}{\pi}\int_{\frac{2 \pi j}{n}}^{\frac{2 \pi (j+1)}{n}}
\lo \Bigl(
2 \sin (x/2)
\Bigr) dx
-
\frac{-2}{n} \, \lo \bigl(2 \, \sin\bigl(\frac{\pi j}{n}\bigr)  \bigr)
\right|
\end{equation}
on the intervals $\mathcal{I}_j := \bigl[
\frac{2 \pi j}{n}, \frac{2 \pi (j+1)}{n}
\bigr]$, $j=
\lceil \lo n \rceil, 
\lceil \lo n \rceil + 1, \ldots,
J_n - 1$.
On each such $\mathcal{I}_j$, the function
$f(x)$ is approximated by its interpolation polynomial
$P_{m}(x)$, where $m \geq 1$ is the number of subintervals forming an uniform partition of
$\mathcal{I}_j$ given by
\begin{equation}
\label{yqinterpolation}
y_q = \frac{2 \pi j}{n} + q \frac{2 \pi}{n} \frac{1}{m}, \qquad q = 0, 1, \ldots, m, 
\end{equation}
of step length 
$h_{NC} := \frac{2 \pi}{n \, m}$, and
$P_m$ the interpolating polynomial 
of degree $m$ or less with
$$P_{m}(y_q) = f(y_q), \qquad \mbox{for}~~
q = 0, 1, \ldots, m.$$
The Newton-Cotes formulas
$$\int_{\frac{2 \pi j}{n}}^{\frac{2 \pi (j+1)}{n}} 
P_{m}(x) dx = h_{NC} \, \sum_{q=0}^{m} \alpha_q f(y_q)$$
provide approximate values of 
$\int_{\frac{2 \pi j}{n}}^{\frac{2 \pi (j+1)}{n}} f(x)
dx$, 
where the $\alpha_q$ are the weights obtained by 
integrating 
the Lagrange's interpolation polynomials.
Steffensen \cite{steffensen} (\cite{stoerbulirsch}, p 127) 
showed that the approximation error may be expressed as 
follows:
$$\int_{\frac{2 \pi j}{n}}^{\frac{2 \pi (j+1)}{n}} P_{m}(x) dx - \int_{\frac{2 \pi j}{n}}^{\frac{2 \pi (j+1)}{n}} f(x)dx = h_{NC}^{p+1} \cdot K \cdot f^{(p+1)}(\xi),
\qquad \xi \in \stackrel{o}{\mathcal{I}_j},$$
where $p \geq 2$ is an integer related to $m$, and $K$ a constant.

Using \cite{stoerbulirsch}, p. 128, and $m=1$, the method 
being the ``Trapezoidal rule", we have: 
``$p=2$, $K = 1/12, 
\alpha_0 = \alpha_1 = 1/2$". Then
\eqref{estijj}
is estimated by
$$\left|
\frac{1}{2} \frac{2 \pi}{n}
\left[
\frac{-1}{\pi} \, \lo \bigl(2 \, \sin\bigl(\frac{\pi j}{n}\bigr)  \bigr) +
\frac{-1}{\pi} \, \lo \bigl(2 \, \sin\bigl(\frac{\pi (j+1)}{n}\bigr)  \bigr)
\right]
-
\frac{-2}{n} \, \lo \bigl(2 \, \sin\bigl(\frac{\pi j}{n}\bigr)  \bigr)\right|$$
\begin{equation}
\label{majo}
= \frac{1}{n}
\left|
\, \lo \bigl(2 \, \sin\bigl(\frac{\pi j}{n}\bigr)  \bigr) -
 \, \lo \bigl(2 \, \sin\bigl(\frac{\pi (j+1)}{n}\bigr)  \bigr)
)\right|
=
\frac{2 \pi}{n^2}
\left|
\frac{\cos(\xi / 2)}{2 \sin (\xi / 2)}
\right|
\leq 
\frac{1}{n} \, \frac{1}{\lo n} 
\end{equation}
for some $\xi \in \stackrel{o}{\mathcal{I}_j}$, for large $n$.
The (Steffensen's) approximation error 
``$h_{NC}^3 \cdot (1/12) \cdot f^{(2)}(\xi)$"
for the trapezoidal rule, relative to
\eqref{estijj}, is
\begin{equation}
\label{steff+}
\frac{1}{\pi} \,
\left(\frac{2 \pi}{n}\right)^3 \, \frac{1}{12}
\, \left|
\frac{-1}{4 \sin^2(\xi/2)}
\right| \leq
\frac{1}{6 n} \frac{1}{(\lo n)^2} .
\end{equation}
By Proposition \ref{argumentlastrootJn}  
the integral
$$\left|
\frac{-1}{\pi}\int_{\frac{2 \pi J_n}{n}}^{2 \arcsin{\kappa/2}}
\lo \bigl(
2 \sin (x/2)
\bigr) dx
\right|\qquad \mbox{is a} \qquad 
O\bigl(\frac{1}{n}\bigr).
$$
Then, summing up the contributions of all the intervals 
$\mathcal{I}_j$, we obtain the 
following upper bound of \eqref{diffloglambda} 
\begin{equation}
\label{cekireste}
\left|
\frac{-1}{\pi}\int_{0}^{(2 \pi \lo n)/n}
\lo \bigl(
2 \sin (x/2)
\bigr) dx
\right|
+
\frac{\arcsin(\kappa/2)}{\pi} \, \frac{1}{\lo n}.
\end{equation}
with global (Steffensen's) approximation error, from \eqref{steff+}, 
$$O(\frac{1}{(\lo n)^2})$$
By integrating by parts the integral 
in \eqref{cekireste}, for large $n$, it is easy to 
show that this integral is 
$=O\left(\frac{(\lo n)^2}{n}\right)$. 
We deduce the 
following asymptotic expansion
\begin{equation}
\Delta_1 =
\frac{\rc}{\lo n}
+
O(\frac{1}{(\lo n)^2})
\qquad
\mbox{with}\quad |\rc| < \frac{\arcsin(\kappa/2)}{\pi}.
\end{equation}

\noindent
{\it Proof of (ii--1)}: Let us show that 
the upper bound 
$\frac{\arcsin(\kappa/2)}{\pi} \, \frac{1}{\lo n}$
is independent of the integer $m$ used, 
once assumed the positivity
of the weights $(\alpha_q)_{q=0, 1, \ldots, m}$. 
For $m \geq 1$ fixed, this is merely a consequence
of the relation between the weights in 
the Newton-Cotes formulas. 
Indeed, we have $\sum_{q=0}^{m} \alpha_q = m$, and therefore
$$
\left|\int_{\frac{2 \pi j}{n}}^{\frac{2 \pi (j+1)}{n}} P_{m}(x) dx - h_{NC} m f(y_0)\right|= 
h_{NC} \, 
\left|\sum_{q=0}^{m} \alpha_q (f(y_q) - f(y_0))\right|
$$
$$
\leq 
h_{NC} \, 
\bigl(\sum_{q=0}^{m} | \alpha_q | \bigr) 
\sup_{\xi \in \mathcal{L}_j}
\left| f'(\xi)
\right|.$$
Since $h_{NC} m = \frac{2 \pi}{n}$
and that
the inequality 
$\sup_{\xi \in \mathcal{L}_j}
\left| f'(\xi)
\right| \leq |f'((2 \pi \lo n)/n)|$ holds
uniformly for all $j$,
we deduce the same upper bound
as in \eqref{majo} for the Trapezoidal rule.
Summing up the contributions over all the
intervals $\mathcal{I}_j$, we obtain
the same upper bound
\eqref{cekireste} of \eqref{diffloglambda} as before.

As for the (Steffensen's) approximation errors, they make use of the successive derivatives
of the function $f(x) = \lo(2 \sin (x/2))$. We have: 
$$f'(x) = \frac{\cos(x/2)}{2 \sin(x/2)}, ~~
f''(x) = -\frac{1}{4 \sin^{2}(x/2)},
~~f'''(x) = \frac{\cos(x/2)}{4 \sin^{3}(x/2)}\ldots$$
Recursively, it is easy to show that
the $q$-th derivative of $f(x)$, $q \geq 1$, is a rational function of the two
quantities $\cos(x/2)$ 
and $\sin(x/2)$ with bounded numerator on the interval $(0, \pi/3]$, and
a denominator which is $\sin^{q}(x/2)$.
For the needs of majoration in the Newton-Cotes formulas
over each interval of the collection
$(\mathcal{I}_j)$, 
this denominator takes its smallest value
at $\xi = (2 \pi \lceil \lo n \rceil)/n$.
Therefore, for large $n$, 
the (Steffensen's) approximation error
``$h_{NC}^{p+1} \cdot K \cdot f^{(p)}(\xi)$"
on one interval $\mathcal{I}_j$ is
$$O\left(\Bigl(\frac{2 \pi}{n m}\Bigr)^{p+1}
\cdot K \cdot 
\frac{n^p}{(\pi \, \lo n)^p}\right) = 
O\left(\frac{1}{
n (\lo n)^p}
\right).$$
By summing up over the intervals
$\mathcal{I}_j$, we obtain
the global (Steffensen's) approximation error ($p \geq 2$)
$$O\left(\frac{1}{(\lo n)^p}\right)
\qquad \mbox{which is a}
\qquad 
O\left(
\left( \frac{\lo \lo n}{\lo n} \right)^2
\right).$$

\noindent
{\em Calculation of $|\Delta_2|$}:
we proceed as above for establishing an upper bound of
$$|\Delta_2|
=
\Bigl|
\frac{-1}{\pi}
\int_{0}^{2 \arcsin(\frac{\kappa(1,a_{\max})}{2})}
\! \!\lo\Bigl[\frac{1 + 2 \sin(\frac{x}{2})
-
\sqrt{1 - 12 \sin(\frac{x}{2}) 
+ 4 (\sin(\frac{x}{2}))^2}}{8 \sin(\frac{x}{2})}
\Bigr] dx\bigr.$$
\begin{equation}
\label{diffloglambda_plus}
-
\sum_{j=\lceil \lo n \rceil}^{J_n}
\frac{- 2}{n} \, \lo \bigl(
\frac{1 + 
2 \sin(\frac{\pi j}{n})
-
\sqrt{1 - 12 \sin(\frac{\pi j}{n})
+ 4 \sin(\frac{\pi j}{n})^2}}
{8 \sin(\frac{\pi j}{n})}
\bigr)  
\Bigr|
\end{equation}

In \eqref{diffloglambda_plus} the sums are truncated 
Riemann-Stieltjes sums of $\lo \mu_r$, the 
integral being 
$\lo \mu_r$. 
As above, the methods of Newton-Cotes
(Stoer and Bulirsch 
(\cite{stoerbulirsch}, pp 126--128) 
will be applied
to compute an 
approximate value of the integral up to 
$O\bigl(
\Bigl( \frac{\lo \lo n}{\lo n} \Bigr)^2
\bigr)$. 
Up to 
$O\bigl(
\Bigr( \frac{\lo \lo n}{\lo n} 
\Bigr)^2
\bigr)$, we will show that:

(i--2) {\it an upper bound of 
\eqref{diffloglambda_plus}  is ($\kappa$ stands for
$\kappa(1,a_{max})$ as in} Proposition
\ref{argumentlastrootJn})

\begin{equation}
\label{delta2upperbound}
 \frac{4 \,\arcsin(\kappa/2)}{\kappa \sqrt{2 \kappa (3-\kappa) \lo (1/\kappa)}} \, \frac{1}{\sqrt{n}}
\qquad \mbox{{\it which is a}}
~~
O\bigl(
\Bigr( \frac{\lo \lo n}{\lo n} 
\Bigr)^2
\bigr),
\end{equation}
{\it in other terms that 
\eqref{diffloglambda_plus}
is equal to zero 
up to} 
$O\bigl(
\Bigr( \frac{\lo \lo n}{\lo n} 
\Bigr)^2
\bigr)$,

(ii--2) {\it the approximate value of 
$\lo \mu_r$ is independent of the step length $m$ 
used in the Newton-Cotes formulas, 
assuming the weights $(\alpha_q)_{q=0, 1, \ldots, m}$ 
associated with $m$ all positive}. 

\vspace{0.2cm}

\noindent
{\it Proof of (i--2)}: The decomposition
of the interval 
of integration
$\bigl(0, 2 \arcsin(\kappa/2)\bigr]$
remains the same as above, given by
\eqref{intervaldecomposition}.
Let us treat the
complete interval 
of integration
$\bigl(0, 2 \arcsin(\kappa/2)\bigr]$
by subintervals.
We first
proceed by 
estimating an upper bound of
$$\Bigl|
\frac{-1}{\pi}
\int_{\frac{2 \pi j}{n}}^{\frac{2 \pi (j+1)}{n}}
\! \!\lo\Bigl[\frac{1 + 2 \sin(\frac{x}{2})
-
\sqrt{1 - 12 \sin(\frac{x}{2}) 
+ 4 (\sin(\frac{x}{2}))^2}}{8 \sin(\frac{x}{2})}
\Bigr] dx\bigr.$$
\begin{equation}
\label{estijj_plus}
-
\frac{- 2}{n} \, \lo \bigl(
\frac{1 + 
2 \sin(\frac{\pi j}{n})
-
\sqrt{1 - 12 \sin(\frac{\pi j}{n})
+ 4 \sin(\frac{\pi j}{n})^2}}
{8 \sin(\frac{\pi j}{n})}
\bigr)  
\Bigr|
\end{equation}
on the intervals $\mathcal{I}_j := \bigl[
\frac{2 \pi j}{n}, \frac{2 \pi (j+1)}{n}
\bigr]$, $j=
\lceil \lo n \rceil, 
\lceil \lo n \rceil + 1, \ldots,
J_n - 1$. 
Let
$$
F(x):= \lo\Bigl[\frac{1 + 2 \sin(\frac{x}{2})
-
\sqrt{1 - 12 \sin(\frac{x}{2}) 
+ 4 (\sin(\frac{x}{2}))^2}}{8 \sin(\frac{x}{2})}
\Bigr].
$$
On each interval
$\mathcal{I}_j$
the function
$F(x)$ is approximated by its interpolation polynomial
(say) $P_{F, m}(x)$, where $m \geq 1$ is the number of 
subintervals 
of
$\mathcal{I}_j$ given by their extremities $y_q$ 
by \eqref{yqinterpolation},
of step length 
$h_{NC} := \frac{2 \pi}{n \, m}$, and
$P_{F, m}$ the interpolating polynomial 
of degree $m$ or less with
$$P_{F, m}(y_q) = F(y_q), \qquad \mbox{for}~~
q = 0, 1, \ldots, m.$$
The Newton-Cotes formulas
\begin{equation}
\label{NCformula}
\int_{\frac{2 \pi j}{n}}^{\frac{2 \pi (j+1)}{n}} 
P_{F, m}(x) dx = h_{NC} \, \sum_{q=0}^{m} \alpha_q F(y_q)
\end{equation}
provide the approximate values 
$\int_{\frac{2 \pi j}{n}}^{\frac{2 \pi (j+1)}{n}} F(x)
dx$, 
where the $\alpha_q$s are the weights obtained by 
integrating 
the Lagrange's interpolation polynomials.
Using \cite{stoerbulirsch}, p. 128, and $m=1$, the method 
being the ``Trapezoidal rule", we have: $p=2$, $K = 1/12, 
\alpha_0 = \alpha_1 = 1/2$. Then
\eqref{estijj_plus}
is estimated by
$$\left|
\frac{1}{2} \frac{2 \pi}{n}
\left[
\frac{-1}{\pi} \, F\bigl(\frac{2 \pi j}{n}\bigr)  +
\frac{-1}{\pi} \, F\bigl(\frac{2\pi (j+1)}{n}\bigr) 
\right]
-
\frac{-2}{n} \,  F\bigl(\frac{2 \pi j}{n}\bigr) 
\right|$$
\begin{equation}
\label{majoF}
= \frac{1}{n}
\left|
\, F\bigl(\frac{2 \pi j}{n}\bigr) -
 \, F\bigl(\frac{2 \pi (j+1)}{n}\bigr)
)\right|
=
\frac{2 \pi}{n^2}
\left|F'(\xi)
\right|
\end{equation}
for some $\xi \in \stackrel{o}{\mathcal{I}_j}$, 
for large $n$. 
As in 
Remark 
\ref{openingangle_sin_quadratic_alginteger},
let $x = 2 \arcsin(\kappa/2)$.
The derivative
\begin{equation} 
\label{derivativeFF}
F'(y)= \frac{\cos(y/2)
(-2 \sin(y/2) + 1 -
\sqrt{4 \sin^{2}(y/2) - 12 \sin(y/2)
+1}}{4 \sin(y/2) \sqrt{4 \sin^{2}(y/2) - 12 \sin(y/2)
+1})} \, > 0
\end{equation}
is increasing on the interval 
$(0, x)$.
When $y=\frac{2 \pi J_n}{n} < x$ tends to $x^{-}$,
by Proposition \ref{argumentlastrootJn}
and Remark \ref{openingangle_sin_quadratic_alginteger},
since
$0 < \sqrt{4 \sin^{2}(y/2) - 12 \sin(y/2)
+1} \leq 1$ is close to zero for 
$y=2 \pi J_n /n$, 
the following
inequality holds
\begin{equation}
\label{zouzou1}
F'(\frac{2 \pi J_n}{n})
\leq \frac{2/\kappa}{
\sqrt{4 \sin^{2}(\frac{\pi J_n}{n}) - 
12 \sin(\frac{\pi J_n}{n})+1}} .
\end{equation}
The upper bound is a function 
of $n$ which comes from
the asymptotic expansion of 
$\frac{\pi J_n}{n} - \frac{x}{2}$, as
deduced from \eqref{Jnasymptotic}.
Indeed, from \eqref{Jnasymptotic}
and using Remark 
\ref{openingangle_sin_quadratic_alginteger} (ii), 
$$4 \sin^{2}(\frac{\pi J_n}{n}) 
\!-\! 
12 \sin(\frac{\pi J_n}{n})
\!+\!1 
\!=\! 
(\frac{\pi J_n}{n} \!- \!\frac{x}{2})
[8 \sin(x/2) \cos(x/2) 
\!-\! 12 \cos(x/2)] 
\!+\! O(\frac{1}{n^2})
$$
\begin{equation}
\label{zouzou2}
= \frac{2 \kappa (3-\kappa) \lo (1/\kappa)}{n}
+
\frac{1}{n}
O\bigl(
\bigl(
\frac{\lo \lo n}{\lo n}
\bigr)^2
\bigr)
\end{equation}
From \eqref{zouzou1}
and \eqref{zouzou2} we deduce 
$|F'(\frac{2 \pi J_n}{n})|
<  \frac{(2/\kappa)}{\sqrt{2 \kappa (3-\kappa) \lo (1/\kappa)}} \sqrt{n}$.
From \eqref{majoF}, we deduce
the following upper bound of 
\eqref{estijj_plus} on each $\mathcal{I}_j := \bigl[
\frac{2 \pi j}{n}, \frac{2 \pi (j+1)}{n}
\bigr]$:
\begin{equation}
\label{pitou}
\frac{4 \pi}{\kappa \sqrt{2 \kappa (3-\kappa) \lo (1/\kappa)}} \,\frac{1}{n^{3/2}} .
\end{equation}
By summing up the contributions, 
for
$j=
\lceil \lo n \rceil, \ldots, J_n - 1$,
from \eqref{pitou}
and the asymptotics of $J_n$ given
by \eqref{Jnasymptotic},
we deduce 
the upper bound \eqref{delta2upperbound}
of $|\Delta_2|$.

Let us prove that the method of 
numerical integration we use leads to 
a (Steffensen's) approximation error 
which is a $O\bigl(
\bigl(
\frac{\lo \lo n}{\lo n}
\bigr)^2
\bigr)$.
The (Steffensen's) approximation error 
``$h_{NC}^3 \cdot (1/12) \cdot F^{(2)}(\xi)$"
for the trapezoidal rule
applied to \eqref{estijj_plus} 
(\cite{stoerbulirsch}, p. 127--128) is
\begin{equation}
\label{steff++}
\frac{1}{\pi} \,
\left(\frac{2 \pi}{n}\right)^3 \, \frac{1}{12}
\, \left|
F^{(2)}(\xi)
\right| 
\qquad \mbox{
for some~} 
\xi \in \stackrel{o}{\mathcal{I}_j}.
\end{equation}
The second derivative $F''(y)$
is positive and increasing on $(0,\frac{2 \pi J_n}{n})$.
It is easy to show that there exists a constant
$C > 0$ such that
$$F''(\frac{2 \pi J_n}{n})
\leq \frac{C}{ 
(4 \sin^{2}(\frac{\pi J_n}{n}) - 
12 \sin(\frac{\pi J_n}{n})+1)^{3/2}} .$$
Using the asymptotic expansion of $J_n$ 
(\eqref{Jnasymptotic};
Remark 
\ref{openingangle_sin_quadratic_alginteger} (ii);
\eqref{zouzou2}),
there exist 
$C_1 > 0$ such that
\begin{equation}
\label{steff++majo}
F''(\frac{2 \pi J_n}{n})
\leq  C_1 \, n^{3/2} .
\end{equation}
From \eqref{steff++} and \eqref{steff++majo}, 
summing up 
the contributions for
$j= \lceil \lo n \rceil, 
 \ldots,
J_n - 1$,  
the global (Steffensen's) approximation error
of \eqref{diffloglambda_plus} 
for $|\Delta_2|$
admits the following upper bound, 
for some constants $C'_{2} > 0, C_2 > 0$, 
$$C'_2 \,  \frac{J_n}{n^3} \, n^{3/2} =
C_2 \frac{1}{\sqrt{n}} \qquad \mbox{which is a}~~
O\bigl(
\bigl(
\frac{\lo \lo n}{\lo n}
\bigr)^2
\bigr) .$$ 
Now let us turn to the extremity intervals.
Using the Appendix, and  
\eqref{Jnasymptotic} in
Proposition \ref{argumentlastrootJn}, 
it is easy to show that
the two integrals
$$\frac{-1}{\pi} 
\int_{0}^{\frac{2 \pi \lceil \lo n \rceil}{n}}
\quad
{\rm and}
\quad
\frac{-1}{\pi} 
\int_{\frac{2 \pi J_n }{n}}^{2 \arcsin(\kappa/2)}
\qquad
\mbox{are}
\qquad
O\left(
\left(
\frac{\lo \lo n }{ \lo n} 
\right)^2
\right).$$

\noindent 
{\it Proof of (ii--2)}: 
On each interval
$\mathcal{I}_j := \bigl[
\frac{2 \pi j}{n}, \frac{2 \pi (j+1)}{n}
\bigr]$, $j=
\lceil \lo n \rceil, 
\ldots,
J_n - 1$, let us assume that
the number $m$ of 
subintervals 
of
$\mathcal{I}_j$ given by their extremities $y_q$ 
by \eqref{yqinterpolation}, is $\geq 2$.
The weights $\alpha_q$  
in \eqref{NCformula} are assumed to be positive.

The upper bound 
 $\frac{4 \,\arcsin(\kappa/2)}{\kappa \sqrt{2 \kappa (3-\kappa) \lo (1/\kappa)}} \,
  \frac{1}{\sqrt{n}}$
of 
\eqref{diffloglambda_plus}
is independent of $m \geq 2$, 
once assumed the positivity
of the weights $(\alpha_q)_{q=0, 1, \ldots, m}$,
since, due to 
the relation between the weights in 
the Newton-Cotes formulas 
$\sum_{q=0}^{m} \alpha_q = m$, 
$$
\left|\int_{\frac{2 \pi j}{n}}^{\frac{2 \pi (j+1)}{n}} 
P_{m}(x) dx - h_{NC} m F(y_0)\right|= 
h_{NC} \, 
\left|\sum_{q=0}^{m} \alpha_q (F(y_q) - F(y_0))\right|
$$
$$
\leq 
h_{NC} \, 
\bigl(\sum_{q=0}^{m} | \alpha_q | \bigr) 
\sup_{\xi \in \mathcal{L}_j}
\left| F'(\xi)
\right|.$$
Since $h_{NC} m = \frac{2 \pi}{n}$
and that
$\sup_{\xi \in \mathcal{L}_j}
\left| F'(\xi)
\right| \leq |F'((2 \pi J_n)/n)|$ holds
uniformly for all $j=
\lceil \lo n \rceil, \ldots, J_n - 1$,
we deduce the same upper bound
\eqref{pitou} as for the Trapezoidal rule.
Summing up the contributions over all the
intervals $\mathcal{I}_j$, we obtain
the same upper bound
\eqref{delta2upperbound} 
of \eqref{diffloglambda_plus}, as before.

As for the (Steffensen's) approximation 
errors involved 
in the numerical integration
\eqref{NCformula} there are 
``$h_{NC}^{p+1} \cdot K \cdot F^{(p)}(\xi)$"
on one interval $\mathcal{I}_j$, for some $p \geq 2$.
They make use of the successive derivatives
of the function $F(x)$.
It can be shown that they contribute negligibly,
after summing up over all the intervals
$\mathcal{I}_j$, as
$
O\left(
\left( \frac{\lo \lo n}{\lo n} \right)^2
\right).$

Gathering the different terms from (i--1)(i--2),
the Steffenssen's error terms
and the error terms due to the 
numerical integration by the
Newton-Cotes method (ii--1)(ii--2),
we have proved the following theorem.

\begin{theorem}
\label{Lrasymptotictheorem}
Let $\beta > 1$ be an algebraic integer
such that
$n=\dyg(\beta) \geq 260$.
The asymptotic expansion
of the minorant
$L_{r}(\beta)$ of \,$\lo {\rm M}_{r}(\beta)$
is
\begin{equation}
\label{Lrasymptotics}
L_{r}(\beta) = \lo \Lambda_r \mu_r 
+
\frac{\rc}{\lo n}
+
O\bigl(
\bigl(
\frac{\lo \lo n}{\lo n}
\bigr)^2
\bigr),
\quad
\mbox{with}\quad |\rc| < \frac{\arcsin(\kappa/2)}{\pi}
\end{equation}
and $\rc$ depending upon $\beta$ and $n$.
\end{theorem}

\subsection{A Dobrowolski type minoration}
\label{S5.6}

The first two terms  
of $L_{r}(\beta)$
in \eqref{Lrasymptotics}
provide the following
Dobrowolski type minoration
of the Mahler measure ${\rm M}(\beta)
\geq {\rm M}_{r}(\beta)$.

\begin{theorem}
\label{dobrominorationREEL}
Let $\beta > 1$ be an algebraic integer
such that $\dyg(\beta) \geq 260$. Then
\begin{equation}
\label{dobrominoREEL}
{\rm M}(\beta) \geq 
\Lambda_r \mu_r \, 
-
\frac{\Lambda_r \mu_r \,\arcsin(\kappa/2)}{\pi}\, 
\frac{1}{\lo (\dyg(\beta))}
\end{equation}
\end{theorem}
\begin{proof}
Taking the exponential of
\eqref{Lrasymptotics} gives
$$
{\rm M}_{r}(\beta)
\geq
\exp(L_{r}(\beta))
=
\Lambda_r \mu_r 
\Bigl(
1 
+
\frac{\rc}{\lo n}
+
O\bigl(
\bigl(
\frac{\lo \lo n}{\lo n}
\bigr)^2
\bigr)
\Bigr)
$$
and \eqref{dobrominoREEL} follows from 
the condition 
$|\rc| < \frac{\arcsin(\kappa/2)}{\pi}$.
\end{proof}

Theorem \ref{dobrominorationREEL}
is the first step of the proof
of Theorem \ref{mainDOBROWOLSLItypetheorem}. 
It is will complemented
in $\S$\ref{S6.3}.

\section{Minoration of the Mahler measure M$(\alpha)$ for $\alpha$ a nonreal complex algebraic integer of house $\house{\alpha} > 1$ close to one}
\label{S6}

Let $\alpha$ be a nonreal complex 
algebraic integer, for which
$\house{\alpha} > 1$. Let us assume that the Mahler measure
${\rm M}(\alpha)$ is smaller than the
smallest Pisot number
$\Theta = 1.3247\ldots$.
The minimal polynomial
$P_{\alpha}(X) \in \zb[X]$ 
of $\alpha$ 
 is monic and
reciprocal. 
If $\alpha^{(i)}$  
is a conjugate of maximal
modulus of $\alpha$, $\alpha^{(i)}$ 
is conjugated with
$(\alpha^{(i)})^{-1}$; 
the house
$\house{\alpha}$ of $\alpha$, resp.
its inverse $\house{\alpha}^{-1}$,
is a root of the quadratic equation
$$X^2 - \alpha^{(i)} \overline{\alpha^{(i)}} =0,
\qquad \mbox{resp. of} \quad
X^2 - (\alpha^{(i)})^{-1} (\overline{\alpha^{(i)}})^{-1} =0 .$$
Therefore $\house{\alpha}$ 
and $\house{\alpha}^{-1}$
are real algebraic integers
of degree $\leq \deg(\alpha)+2$ for which 
$1 <\house{\alpha} < \Theta = \theta_{5}^{-1}$.
The mapping $\alpha \to \house{\alpha}, 
\mathcal{O}_{\overline{\qb}} \to \mathcal{O}_{\overline{\qb}}
\cap (1, \infty)$ is 
not continuous.
However, writing $\beta = \house{\alpha}$, 
the preceding analytic functions
of the R\'enyi-Parry dynamical system 
of the $\beta$-shift,
defined in \S \ref{S4},
can be applied once $|\alpha|> 1$ 
tends to $1^{+}$. 
The quantities
$\dyg(\beta)$, $f_{\beta}(z)$,
$\lc_{\beta}$, ${\rm M}(\beta)$,
${\rm M}_{r}(\beta)$ become 
well-defined
so that the minoration
of the Mahler measure
${\rm M}(\house{\alpha})$
is of Dobrowolski type as in 
Theorem \ref{dobrominorationREEL} in
$\S$ \ref{S5.6},
as a function of the dynamical degree
$\dyg(\beta)$.

\subsection{Fracturability of the minimal polynomial of $\alpha$ by the Parry Upper function at $\house{\alpha}$}
\label{S6.1}

In the following Theorem, which is
an extension of 
Theorem \ref{splitBETAdivisibility}
and Theorem \ref{splitBETAdivisibility+++}, 
the canonical
fracturability of the minimal polynomial
$P_{\alpha}(X)$ 
by the Parry Upper function
$f_{\house{\alpha}}(z)$ is proved, in full
generality, as well
as the crucial
fact that $\alpha$ is conjugated to
its house $\house{\alpha}$.
The fundamental consequence
(Theorem \ref{divisibilityALPHA}
and
Theorem \ref{Lrasymptotictheoremcomplexe})
is the passage
from the minoration of
${\rm M}_{r}({\house{\alpha}})$
to that of
${\rm M}(\alpha)$ itself, 
by the identification
of the lenticulus $\lc_{\house{\alpha}}$
as a lenticulus of conjugates 
of $\alpha$ itself.

\begin{theorem}
\label{divisibilityALPHA}
Let $\alpha$ be a nonreal complex 
algebraic integer, which satisfies
$\house{\alpha} > 1$,
${\rm M}(\alpha) < \Theta = 1.3247\ldots,
 ~\deg(\alpha) \geq 6$. Denote
 $\beta = \house{\alpha}$.
The following formal decomposition of 
the (monic) minimal polynomial of $\alpha$
\begin{equation}
\label{decompoAlphaBeta}
P_{\alpha}(X) = P_{\alpha}^{*}(X)
=
U_{\alpha}(X) \times
f_{\beta}(X)
\end{equation}
holds, as a product
of the Parry Upper function 
\begin{equation}
\label{formlacuA}
f_{\beta}(X)= G_{\dyg(\beta)}(X)
+ X^{m_1} + X^{m_2} + X^{m_3}+ \ldots.
\end{equation}
with
$m_0:= \dyg(\beta)$,
$m_{q+1} - m_q \geq \dyg(\beta)-1$ for $q \geq 0$,
and the invertible formal series 
$U_{\alpha}(X) \in \zb[[X]]$, 
quotient of $P_{\alpha}$ by
$f_{\beta}$.
The specialization $X \to z$ of the formal variable
to the complex variable leads to
the identity between analytic functions,
obeying the Carlson-Polya dichotomy with
the house $\house{\alpha}$ of $\alpha$:
\begin{equation}
\label{decompozzzALPHA}
P_{\alpha}(z) = U_{\alpha}(z) \times
f_{\house{\alpha}}(z)
\quad
\left\{
\begin{array}{cc}
\mbox{on}~ \cb & \mbox{if $\house{\alpha}$ is a Parry number},\\
&\\
\mbox{on}~ |z| < 1 & 
\mbox{if~~} \house{\alpha} \mbox{~~is 
nonParry, with $|z|=1$}\\
& \mbox{as natural boundary
for both $U_{\alpha}$ and $f_{\beta}$.}
\end{array}
\right.
\end{equation}
Assume now $\dyg(\beta) \geq 260$. 
Let $\mathcal{D}_n$ be the 
subdomain of the open unit disc
defined in Theorem \ref{absencezeroesOutside}.
Denote $D_{j,n} := \{z \mid
|z - z_{j,n}| < 
\frac{\pi \, |z_{j,n}|}{ n \, a_{\max}}\}$,
$j = 1, 2, \ldots, J_n$.
Then the domain of holomorphy
of  
$U_{\alpha}(z) = \frac{P_{\alpha}(z)}
{f_{\beta}(z)}$
contains
the connected domain
$$
\Omega_{n} :=
\dc_n \, \cup \, \bigcup_{j=1}^{J_n}
\, \left(
D_{j,n} \cup \overline{{D_{j,n}}}
\right)
\cup \, D(\theta_{n}, \frac{t_{0,n}}{n});
$$
$U_{\alpha}(z)$
does not vanish on the 
lenticulus $
\lc_{\beta} = \{
\frac{1}{\beta}\}
\cup 
\bigcup_{j=1}^{J_n}
(\{\omega_{j,n}\} 
\cup 
\{\overline{\omega_{j,n}}\}) 
\subset \Omega_n$.
For any zero
$\omega_{j,n} \in \lc_{\beta}$, 
\begin{equation}
\label{nevervanishesA}
U_{\alpha}(\omega_{j,n}) = 
\frac{P'_{\alpha}(\omega_{j,n})}{f'_{\beta}(\omega_{j,n})}
\neq 0 \, ,\quad
U_{\alpha}(\overline{\omega_{j,n}}) = 
\frac{P'_{\alpha}(\overline{\omega_{j,n}})}{f'_{\beta}(\overline{\omega_{j,n}})}
\neq 0 \,
\quad \mbox{and}
~~
U_{\alpha}(\frac{1}{\beta}) = 
\frac{P'_{\alpha}(\frac{1}{\beta})}{f'_{\beta}(\frac{1}{\beta})}
\neq 0.
\end{equation}
The real algebraic integer $\house{\alpha}$ 
admits $\alpha$ as conjugate, 
its minimal polynomial 
$P_{\alpha}
=
P_{\house{\alpha}}$ 
satisfying
\begin{equation}
\label{touslessixALPHA}
P_{\alpha}(\house{\alpha}) = 
P_{\alpha}(\house{\alpha}^{-1}) = 
P_{\alpha}(\alpha^{-1}) = 
P_{\alpha}(\alpha) = 
P_{\alpha}(\overline{\alpha^{-1}}) = 
P_{\alpha}(\overline{\alpha}) = 0.
\end{equation}

\end{theorem}

\begin{proof}
Since $\theta_{5}^{-1} =
\Theta > {\rm M}(\alpha)
\geq \house{\alpha} > 1$, there exists
an integer $n \geq 6$ such that
$\house{\alpha}$
lies between two successive Perron
numbers of the family
$(\theta_{n}^{-1})_{n \geq 5}$, as
$
\theta_{n}^{-1} \leq \house{\alpha}
< \theta_{n-1}^{-1}$. 
By Proposition \ref{splitBETAdivisibility+++}, 
the Parry Upper function
$f_{\house{\alpha}}(z)$ at 
$\house{\alpha}$
has the form:
\begin{equation}
\label{formlacuALPHA}
f_{\house{\alpha}}(z) = 
-1 + z + z^n + z^{m_1} + z^{m_2}
+ z^{m_{3}}+ \ldots
\end{equation}
with $m_0 = n$ and
$m_{q+1}-m_q \geq n-1$ for $q \geq 0$.
Whether $\house{\alpha}$ is a Parry number or 
a nonParry number is unkown. 
In any case, by construction, it is such that
$f_{\house{\alpha}}(\house{\alpha}^{-1}) = 0$.
Let us write the Parry Upper function as
$f_{\house{\alpha}}(z) = 
- 1 + \sum_{j \geq 1} t_j z^j$.
The zero
$\house{\alpha}^{-1}$ of $f_{\house{\alpha}}(z)$
is simple since
the derivative of $f_{\house{\alpha}}(z)$
satisfies
$f'_{\house{\alpha}}(\frac{1}{\house{\alpha}}) 
= \sum_{j \geq 1} j \, t_j \, \house{\alpha}^{-j+1} > 0$.

Let us show that the formal decomposition
\eqref{decompoAlphaBeta} is always possible. 
We proceed as in the proof of 
Theorem \ref{splitBETAdivisibility}.
Indeed, if
we put
$U_{\alpha}(X) = 
-1 + \sum_{j \geq 1} b_j X^j$, and
$P_{\alpha}(X) = 1 + a_1 X + a_2 X^2 + \ldots
a_{d-1} X^{d-1}
+ X^d $, (with $a_j = a_{d-j}$),
the formal identity 
$P_{\alpha}(X) = U_{\alpha}(X) \times 
f_{\house{\alpha}}(X)$
leads to the existence of the coefficient 
vector $(b_j)_{j \geq 1}$ of
$U_{\alpha}(X)$, as a function
of $(t_j)_{j \geq 1}$ and
$(a_i)_{i = 1,\ldots, d-1}$, as:
$b_1 = -(a_1 + t_1)$,
and, for $r = 2, \ldots, d-1$,
\begin{equation}
\label{bedeRecurrencedebutALPHA}
b_r = -(t_r + a_r - 
\sum_{j=1}^{r-1}b_j t_{r-j}) 
\quad
\mbox{with} \quad
b_d = -(t_d + 1 - 
\sum_{j=1}^{d-1} b_j t_{r-j}),
\end{equation}
\begin{equation}
\label{bedeRecurrenceALPHA}
b_r = -t_r + 
\sum_{j=1}^{r-1} b_j t_{r-j} \quad \mbox{for}~~  r > d.
\end{equation}
For $j \geq 1$,
$b_j \in \zb$,
and the integers
$b_r , r > d$,
are determined recursively
by \eqref{bedeRecurrenceALPHA}
from 
$\{b_0 = -1, b_1 , b_2 , \ldots , b_d \}$.
Every $b_j$ in 
$\{b_1 , b_2 , \ldots , b_d \}$
is computed from the coefficient vector
of $P_{\alpha}(X)$
using \eqref{bedeRecurrencedebutALPHA}, 
starting by $b_1 = -1 - a_1$. The disc 
of convergence of $U_{\alpha}(z)$
has a radius greater than or equal to 
$\theta_{n-1}$ by Theorem 
\ref{splitBETAdivisibility}.
By Theorem \ref{splitBETAdivisibility+++},
assuming $n \geq 260$,
the domain of holomorphy of
$U_{\alpha}(z)$ contains
$\Omega_n$. By analogy with 
\eqref{decompoDERIVALPHA} 
the formal identity holds:
\begin{equation}
\label{decompoDERIVAlphaBeta}
P'_{\alpha}(X) ~=~
U'_{\alpha}(X) \,
f_{\beta}(X)
~+~
U_{\alpha}(X) \,
f'_{\beta}(X).
\end{equation}
For $j = 1, 2, \ldots, J_n$,
since $\omega_{j,n}$ is a simple zero,
$f'_{\beta}(\omega_{j,n}) \neq 0$;
specializing $X$ to 
$\omega_{j,n}$ 
gives
$$P'_{\alpha}(\omega_{j,n}) ~=~
U_{\alpha}(\omega_{j,n}) \,
f'_{\beta}(\omega_{j,n}),
\qquad 
P'_{\alpha}(\overline{\omega_{j,n}}) ~=~
U_{\alpha}(\overline{\omega_{j,n}}) \,
f'_{\beta}(\overline{\omega_{j,n}});$$
as in the proof of
Theorem \ref{splitBETAdivisibility+++},
it is easy to show that
$|U_{\alpha}(\omega_{j,n})|\neq +\infty$,
$U_{\alpha}(\omega_{j,n}) \neq 0$.
Similarly,
$|U_{\alpha}(\house{\alpha}^{-1})| \neq +\infty$
and
$U_{\alpha}(\house{\alpha}^{-1}) \neq 0$.
The zeroes of $\lc_{\house{\alpha}}$ are not 
singularities of $U_{\alpha}(z)$.
We deduce \eqref{nevervanishesA}
and \eqref{touslessixALPHA}.

\begin{definition}
\label{complexeALPHA}
Let $\alpha$ be a nonzero algebraic integer
such that
$1 < \house{\alpha} < \Theta = \theta_{5}^{-1}$.
The dynamical degree of $\alpha$ 
is defined by
$$\dyg(\alpha) := \dyg(\house{\alpha});$$
if $n := \dyg(\alpha)$ is $\geq 260$,
the lenticulus 
$\lc_{\alpha}$ of conjugates of $\alpha$ 
is defined by
$$\lc_{\alpha} := \lc_{\house{\alpha}}
=
\{\overline{\omega_{J_n , n}} ,\ldots,
\overline{\omega_{1,n}},
\house{\alpha}^{-1},
\omega_{1,n}, \ldots,
\omega_{J_n , n}\}$$
and the reduced Mahler measure of $\alpha$ 
by
$${\rm M}_{r}(\alpha) := 
{\rm M}_{r}(\house{\alpha})
= \house{\alpha}
\prod_{j=1}^{J_n} |\omega_{j,n}|^{-2}
.$$
\end{definition}

Definition
\ref{fracturabilityminimalpolynomialDEF}
can be extended to nonreal complex algebraic integers
$\alpha$ as follows.

\begin{definition}
\label{fracturabilityminimalpolynomialDEFcomplexe}
Let $\alpha$ be a nonreal complex 
algebraic integer such that
${\rm M}(\beta) < \Theta$,
$\deg(\alpha) \geq 6$,
$\house{\alpha} > 1$. The minimal polynomial
$P_{\alpha}(X)$ is said to be fracturable if
the power series $U_{\alpha}(z)$ in
\eqref{decompozzzALPHA} 
is not reduced to a constant.
\end{definition}

\begin{remark}
Theorem \ref{divisibilityALPHA} 
extends easily
to the three cases
(keeping the assumptions
$\house{\alpha} > 1$, 
${\rm M}(\alpha) < \Theta$,
and $\dyg(\beta) \geq 260$,
in each case): (i) when
$\alpha$ is a real algebraic integer which is not
$> 1$, for instance when
$\alpha$ is a negative Salem number
in which case $\beta = -\alpha > 1$,
(ii) when $\alpha$ is real $> 1$ and that
$\alpha \neq \house{\alpha} = \beta$, for instance
if $\alpha$ is a totally real algebraic
integer $> 1$, or eventually
a partially real algebraic integer
$> 1$, and distinct of its house,
(iii) when $\alpha$ is a nonreal 
complex algebraic integer
such $\beta = \house{\alpha}$
admits real 
Galois conjugates $\gamma$ satisfying:
$1 < \gamma <  \beta = \house{\alpha}$.
 
In the last two cases (ii) and 
(iii), if we enumerate the distinct
real conjugates of $\beta$ as
$1 < \gamma_1 < \gamma_2 < \ldots < \gamma_s 
<  \beta$, $s \geq 1$,
the $s$ Parry Upper functions $f_{\gamma_{i}}(z)$,
$i=1, \ldots, s$, have to be 
considered together with their respective lenticuli
of zeroes, in addition to
the Parry Upper function $f_{\beta}(z)$. 
It may occur that some zeroes belonging 
to the lenticuli
of zeroes
of the functions $f_{\gamma_{i}}(z)$ lie inside
the Rouch\'e disks centered at the 
points $z_{j,n}$, relative to
$f_{\beta}(z)$. These extra zeroes, distinct of
the zeroes $\omega_{j,n}$,
are conjugates of $\alpha$ as well.
They should be included in the computation of
the minorant of
the reduced Mahler measure
${\rm M}(\alpha)$.
This extra collection of zeroes
is not really controlled
as a continuous function of
$\house{\alpha}$ unfortunately.
In other terms, in both cases
(ii) and (iii),
the minorant of the Mahler measure
${\rm M}(\alpha)$,
obtained below, has certainly
to be multiplied by a factor 
$> 1$ depending upon the number 
of real conjugates of the
house of $\alpha$.

For partially or totally real algebraic integers,
let us recall Garza's lower bound.
Garza \cite{garza} established
the following minoration of the Mahler measure
${\rm M}(\alpha)$ 
for $\alpha$ an algebraic 
number, different from $0$ and $\pm 1$,
having a certain proportion 
of real Galois conjugates:
if $\deg(\alpha) =d \geq 1$
and $1 \leq r \leq d$ be the 
number of real Galois
conjugates of $\alpha$, then
$${\rm M}(\alpha)
\geq \left(
\frac{2^{1-1/R} + \sqrt{4^{1-1/R}+4}}{2}
\right)^\frac{r}{2},$$
where $R := r/d$.
An elementary proof of this minoration
was given by H\"ohn \cite{hoehn}.
If $r=d$, Garza's bound is Schinzel's bound
\eqref{schinzel0001}
for totally real algebraic integers
\cite{hoehnskoruppa}.

Garza's minorant 
satisfies
$\lim_{d \to \infty}
2^{-r/2}
\bigl(
2^{1 - d/r} + \sqrt{4^{1 - d/r} + 4}
\bigr)^{r/2} = 1$, for any $r$ fixed,
where the limit 1
is reached ``without any discontinuity". 
In some sense, a better minorant
is expected,
and Garza's lower bound
does not take into account 
the discontinuity claimed by
the Conjecture of Lehmer. 
\end{remark}

\end{proof}

\subsection{A strange continuity. Asymptotic expansion of the lenticular minorant ${\rm M}_{r}(\house{\alpha})$}
\label{S6.2}

Let $\alpha$ be an algebraic 
integer such that 
$\beta = \house{\alpha}$ has
dynamical degree $\dyg(\beta) \geq 260$.
The continuity
of the first nonreal complex root
$\omega_{1,n}$ of the lenticulus
$\lc_{\beta}$, with $\beta$,
was proved and studied 
by Flatto, Lagarias and Poonen
\cite{flattolagariaspoonen}. 
By Corollary 
\ref{zeroesParryUpperfunctionContinuity} and
Theorem \ref{divisibilityALPHA} the others roots
of modulus $< 1$ of 
the Parry Upper function
$f_{\house{\alpha}}(z)$
which are conjugates of
$\alpha$ are continuous functions of
$\beta = \house{\alpha}$. 
These facts suggest the 
conjecture that the (true) Mahler measure
${\rm M}(\alpha)$ is a continuous function
of the house $\house{\alpha}$ of
$\alpha$.
On the contrary, the nonderivability
of the function $\beta 
= \house{\alpha} \to \omega_{1,\dyg(\beta)}$ 
conjectured in
 \cite{flattolagariaspoonen}
 would suggest that
 the (true) Mahler measure
${\rm M}(\alpha)$ is not 
derivable as a function
of $\beta 
= \house{\alpha}$, in general.
 
 \begin{theorem}
\label{Lrasymptotictheoremcomplexe}
Let $\alpha$ be an algebraic integer
such that
$n=\dyg(\house{\alpha}) \geq 260$.
The asymptotic expansion
of the minorant
$L_{r}(\house{\alpha})$ of 
\,$\lo {\rm M}_{r}(\alpha)$
is
\begin{equation}
\label{Lrasymptoticscomplexe}
L_{r}(\house{\alpha}) = \lo \Lambda_r \mu_r 
+
\frac{\rc}{\lo n}
+
O\bigl(
\bigl(
\frac{\lo \lo n}{\lo n}
\bigr)^2
\bigr),
\quad
\mbox{with}\quad |\rc| < \frac{\arcsin(\kappa/2)}{\pi}
\end{equation}
and $\rc$ depending upon $\house{\alpha}$ and $n$.
\end{theorem}

\begin{proof}
Using
Definition \ref{complexeALPHA},
with $\beta = \house{\alpha}$,
we readily deduce
the result
from
Theorem \ref{Lrasymptotictheorem} and
Theorem \ref{dobrominorationREEL}.
\end{proof}

As in the proof of 
Theorem \ref{divisibilityALPHA}
the minoration
of the lenticular Mahler measure of $\beta$ follows from
\eqref{Lrasymptoticscomplexe}:

\begin{equation}
\label{lenticularminorationMr}
{\rm M}_{r}(\beta) \geq\Lambda_r \mu_r 
(1 + 
\frac{\rc}{\lo n}
+
O\bigl(
\bigl(
\frac{\lo \lo n}{\lo n}
\bigr)^2
\bigr),
\quad
\mbox{with}~~ |\rc| < \frac{\arcsin(\kappa/2)}{\pi} .
\end{equation}

\subsection{A Dobrowolski type minoration with the dynamical degree of $\house{\alpha}$ - Proof of Theorem~\ref{mainDOBROWOLSLItypetheorem}} 
\label{S6.3}

Using Definition \ref{complexeALPHA}, 
with $\beta = \house{\alpha}$,
we deduce \eqref{dodobrobro} from
Theorem \ref{Lrasymptotictheoremcomplexe}
and \eqref{lenticularminorationMr}.
In the case where $\alpha$ is the conjugate of
a Perron number
$\theta_{n}^{-1} , n \geq 260$,
the minorant in \eqref{dodobrobro}
takes much higher values
and is already given in $\S$\ref{S3}
(Theorem \ref{maincoro5}).

\subsection{Proof of the Conjecture of Lehmer (Theorem \ref{mainLEHMERtheorem})}
\label{S6.4}

Let $\alpha \neq 0$ be an algebraic integer
which is not a root of unity.
Since ${\rm M}(\alpha) =
 {\rm M}(\alpha^{-1})$
there are three cases to be considered:
\begin{enumerate}
\item[(i)]
the house of $\alpha$ satisfies $\house{\alpha} \geq \theta_{5}^{-1}$,
\item[(ii)]
the dynamical degree of $\alpha$ satisfies:
$6 \leq \dyg(\alpha) < 260$,
\item[(iii)]
the dynamical degree of $\alpha$ 
satisfies: 
$\dyg(\alpha) \geq 260$.
\end{enumerate}
In the first case, ${\rm M}(\alpha) \geq 
\theta_{5}^{-1} \geq \theta_{259}^{-1}$. 
In the second case,
${\rm M}(\alpha) \geq 
\theta_{259}^{-1}$.
In the third case, 
the Dobrowolski type inequality
\eqref{dodobrobro}
gives the following
lower bound of the Mahler measure
$${\rm M}(\alpha) \geq 
\Lambda_r \mu_r \, 
-
\frac{\Lambda_r \mu_r \,\arcsin(\kappa/2)}{\pi ~\lo (\dyg(\alpha))}\, 
\geq
\Lambda_r \mu_r \, 
-
\frac{\Lambda_r \mu_r \,\arcsin(\kappa/2)}{\pi ~\lo (259)}\, 
, = 1.14843\ldots$$
by Proposition
\ref{argumentlastrootJn}
and Theorem \ref{MahlerMINORANTreal}.
This lower bound is numerically greater than
$\theta_{259}^{-1} = 1.016126\ldots$. 
Therefore, in any case,
the lower bound 
$\theta_{259}^{-1}$
of ${\rm M}(\alpha)$ holds true.
We deduce the claim.

\subsection{Proof of the Conjecture of Schinzel-Zassenhaus (Theorem \ref{mainSCHINZELZASSENHAUStheorem})}
\label{S6.5}

\begin{proposition}
\label{degreeBETApolymini}
Let $\alpha$ be an algebraic
integer such that
$\dyg(\alpha) \geq 260$. 
The degree 
$\deg(\alpha)$ 
of $\alpha$ 
is related to its dynamical degree 
$\dyg(\alpha)$ by
\begin{equation}
\label{dygdeg}
\dyg(\alpha) \Bigl(
\frac{2 \arcsin\bigl( \frac{\kappa}{2} \bigr)}{\pi}
\Bigr)
+
\Bigl(
\frac{2 \kappa \, \lo \kappa}
{\pi \,\sqrt{4 - \kappa^2}} 
\Bigr)
 ~\leq~ \deg(\alpha).
 \end{equation}
\end{proposition}

\begin{proof}
By Theorem \ref{omegajnexistence}
and Theorem \ref{splitBETAdivisibility+++}
the number of zeroes in the lenticulus
$\lc_{\alpha}$ is
$1 + 2 J_n $; these zeroes
are all conjugates of
$\alpha$. The total number of conjugates 
of $\alpha$
is the degree $\deg(\alpha)$ of
the minimal polynomial $P_{\alpha}$.
By Proposition \ref{argumentlastrootJn},
with $n:= \deg(\alpha)$,
$$1 + 2 J_n = 
\frac{2 n}{\pi}
\bigl(
\arcsin\bigl( \frac{\kappa}{2} \bigr) 
\bigr)
+
\Bigl(
\frac{2 \kappa \, \lo \kappa}
{\pi \,\sqrt{4 - \kappa^2}}
\Bigr)
+ \bigl(1 
+ 
\frac{1}{n}
O\bigl(
\bigl(
\frac{\lo \lo n}{\lo n}
\bigr)^2
\bigr)\bigr).
$$
The inequality 
\eqref{dygdeg} follows.

\end{proof}

\begin{theorem}
\label{schinzelZ}
Let $\alpha$ be a nonzero 
algebraic integer
which is not a root of unity.
Then
\begin{equation}
\label{kappo}
\house{\alpha} \geq 1 + \frac{c}{\deg(\alpha)}
\end{equation}
with 
$c = \theta_{259}^{-1} - 1 = 0.016126\ldots$.
\end{theorem}

\begin{proof}
There are two cases: either (i) $\house{\alpha} \geq \theta_{259}^{-1}$, or (ii) $n \geq 260$.

(i) If $\house{\alpha} \geq \theta_{259}^{-1}$, then,
whatever the degree $\deg(\alpha) \geq 1$,
$$\house{\alpha} \geq 1 +\frac{(\theta_{259}^{-1} -1)}
{\deg(\alpha)}.$$

(ii) The minoration of the house
$\beta = \house{\alpha}$ can easily be obtained
as a function of the dynamical degree of $\alpha$.
Let $n = \dyg(\beta)$ and assume  $n \geq 260$.
By definition $\theta_{n}^{-1}
\leq \beta < \theta_{n-1}^{-1}$.
Theorem 1.8 in \cite{vergergaugry6} 
(cf also \cite{vergergaugry6} \S 5.3)
implies 
\begin{equation}
\label{tintin}
\beta = \house{\alpha}
~\geq~ \theta_{n}^{-1}
~\geq~ 1 + 
\frac{(\lo n) (1 - \frac{\lo \lo n}{\lo n})}{n}.
\end{equation}
From Proposition
\ref{degreeBETApolymini},
\begin{equation}
\label{milou}
\frac{1}{n} =
\frac{1}{\dyg(\beta)}
\geq
\frac{2 \arcsin(\kappa/2)}{\pi~\deg(\alpha)}
\left(
1 + \frac{\kappa \lo \kappa}{n \, \arcsin(\kappa/2) \, \sqrt{4 - \kappa^2}}
\right) .
\end{equation}
The function
$\frac{\lo x - \lo \lo x}{\lo x}\left(
1 + \frac{\kappa \lo \kappa}{x \, \arcsin(\kappa/2) \, \sqrt{4-\kappa^2}}
\right)$ 
is increasing for $x \geq 260$.
From \eqref{tintin} and \eqref{milou} 
we deduce
$$\house{\alpha} \geq 1 + \frac{\tilde{c}}{\deg(\alpha)}$$
with
$\tilde{c}= \frac{2}{\pi}
\frac{\lo 260 - \lo \lo 260}{\lo 260}
\Bigl(
\arcsin(\kappa/2) + \frac{ \kappa \lo \kappa}{260 \, \sqrt{4-\kappa^2}}
\Bigr) = 0.0375522\ldots$.

From (i) and (ii), we deduce
that \eqref{kappo} holds with
$c = \min\{\tilde{c}, (\theta_{259}^{-1} -1)\}
= (\theta_{259}^{-1} -1) = 0.016126\ldots$ for every
nonzero algebraic integer
$\alpha$ which is not a root of unity.
\end{proof}

\section{Salem numbers, totally real algebraic numbers, Bogomolov property}
\label{S7}

The set of Pisot numbers admits the minorant
$\Theta$ by a result of Siegel \cite{siegel}.
We extend this result in the sequel as a 
consequence of
Theorem \ref{secondseriesUU}: 
in fact 
Theorem \ref{secondseriesUU} 
implies boundedness 
from below to 
(i) the set of Salem numbers
($\S$ \ref{S7.2}), 
(ii) the set of totally real algebraic numbers
in terms of the Weil height
($\S$ \ref{S7.3}).

\subsection{Existence and localization of the first nonreal root of the Parry Upper function $f_{\beta}(z)$ of modulus $<1$ in the cusp of the fractal of Solomyak}
\label{S7.1}

\begin{theorem}
\label{_cercleoptiSALEM}
Let $n \geq 32$.
Denote by $C_{1,n}
:= \{z \mid |z-z_{1,n}| = 
\frac{\pi |z_{1,n}|}{n \, a_{\max}} \}$
the circle centered at the first root
$z_{1,n}$ of $G_{n}(z) = -1 + z + z^n$.
Then the 
condition of Rouch\'e 
\begin{equation}
\label{rouchecercleSALEM}
\frac{
\left|z\right|^{2 n -1}}{1 - |z|^{n-1}}
~<~
\left|-1 + z + z^n \right| , 
\qquad \mbox{for all}~ z \in C_{1,n},
\end{equation}
holds true. 
\end{theorem}

\begin{proof}
Let $a \geq 1$ and $n \geq 18$.
Denote by $\varphi := \arg (z_{1,n})$
the argument of the first root
$z_{1,n}$ (in ${\rm Im}(z) > 0$). 
Since $-1 + z_{1,n} + z_{1,n}^n = 0$, 
we have $|z_{1,n}|^n = |-1 + z_{1,n}|$.
Let us write $z= z_{1,n}+ 
\frac{\pi |z_{1,n}|}{n \, a} e^{i \psi}
=
z_{1,n}(1 + \frac{\pi}{a \, n} e^{i (\psi - \varphi)})$
the generic element belonging to 
$C_{1,n}$, with
$\psi \in [0, 2 \pi]$. 
Let $X := \cos(\psi - \varphi)$.
Let us show that if the inequality
\eqref{rouchecercleSALEM} of Rouch\'e 
holds true for $X =+1$,
then it holds true
for all $X \in [-1,+1]$, that is for 
every argument $\psi \in [0, 2 \pi]$,
i.e. for every 
$z \in C_{1,n}$.
As in the proof
of Theorem \ref{cercleoptiMM},
$$
\left|1 + \frac{\pi}{a \,n} e^{i (\psi - \varphi)}
\right|^{n}
=
\exp\Bigl(
\frac{\pi \, X}{a}\Bigr)
\times 
\left(
1 - \frac{\pi^2}{2 a^2 \, n} (2 X^2 -1) 
+O(\frac{1}{n^2})
\right)
$$
and
$$
\arg\left(
\Bigl(1 + \frac{\pi}{a \, n} e^{i (\psi - \varphi)}
\Bigr)^{n}\right)
=
sgn(\sin(\psi - \varphi))
\times
\left( ~\frac{\pi \, \sqrt{1-X^2}}{a}
[1 -
\frac{\pi \, X}{a \, n}
]
+O(\frac{1}{n^2})
\right)
.$$
Moreover,
$$
\left|1 + \frac{\pi}{a \, n} 
e^{i (\psi - \varphi)}
\right|
=
\left|1 + \frac{\pi}{a \, n} 
(X \pm i \sqrt{1-X^2})
\right|
=1 + \frac{\pi \, X}{a \, n} + O(\frac{1}{n^2}).
$$
with
$$\arg(1 + \frac{\pi}{a \, n} e^{i (\psi - \varphi)})
= 
sgn(\sin(\psi - \varphi)) \times
\frac{\pi \sqrt{1 - X^2}}{a \, n} 
+ O(\frac{1}{n^2}).
$$
For all $n \geq 18$, from
Proposition \ref{zedeUNmodule}, 
we have
\begin{equation}
\label{devopomainSALEM}
|z_{1,n}|
= 1 - \frac{\lo n - \lo \lo n}{n}
+ \frac{1}{n} O \left(
\frac{\lo \lo n}{\lo n}\right).
\end{equation}
from which we deduce the following equality, 
up to $O(\frac{1}{n})$ - terms,
$$
|z_{1,n}|
\,
\left|1 + \frac{\pi}{a \, n} e^{i (\psi - \varphi)}\right|
=
|z_{1,n}| .
$$
Then the left-hand side term of \eqref{rouchecercleSALEM}
is
$$\frac{
\left|z\right|^{2 n -1}}{1 - |z|^{n-1}}
=
\frac{|-1 + z_{1,n}|^2 
\left|1 + \frac{\pi}{a \, n} e^{i (\psi - \varphi)}
\right|^{2 n}}
{|z_{1,n}| \, 
\left|1 + \frac{\pi}{a \, n} e^{i (\psi - \varphi)}\right|
-
|-1 + z_{1,n}| \,
\left|1 + \frac{\pi}{a \, n} e^{i (\psi - \varphi)}\right|^{n}}$$

\begin{equation}
\label{rouchegaucheSALEM}
=
\frac{|-1 + z_{1,n}|^2 
\left(
1 - \frac{\pi^2}{a \, n} (2 X^2 -1) 
\right)
\exp\bigl(
\frac{2 \pi \, X}{a}\bigr)
}
{|z_{1,n}|
\left|1 + \frac{\pi}{a \, n} e^{i (\psi - \varphi)}\right|
-
|-1 + z_{1,n}| \,
\left(
1 - \frac{\pi^2}{2 a \, n} (2 X^2 -1) 
\right)
\exp(
\frac{\pi \, X}{a})
}
\end{equation}
up to
$\frac{1}{n} 
O \left(
\frac{\lo \lo n}{\lo n}
\right)$
-terms (in the terminant).
The right-hand side term of 
\eqref{rouchecercleSALEM} is 
$$\left|-1 + z + z^n \right|
=
\left|
-1 + z_{1,n}\Bigl(1 + \frac{\pi}{n \, a} e^{i (\psi - \varphi)}\Bigr) +
z_{1,n}^{n}
\Bigl(1 + \frac{\pi}{n \, a} e^{i (\psi - \varphi)}
\Bigr)^n
\right|
$$
$$=
\left|
-1 + z_{1,n}
(1 \pm i \frac{\pi \sqrt{1 - X^2}}{a \, n})
(1 + \frac{\pi \, X}{a \, n})
+
(1 - z_{1,n})
\left(
1 - \frac{\pi^2}{2 a^2 \, n} (2 X^2 -1) 
\right)
\right. \hspace{1cm} \mbox{}
$$
\begin{equation}
\label{rouchedroiteSALEM}
\left.
\times
\exp\bigl(
\frac{\pi \, X}{a}
\bigr) \,
\exp\Bigl(
\pm \,
i \,
\Bigl( ~\frac{\pi \, \sqrt{1-X^2}}{a}
[1 - \frac{\pi \, X}{a \, n}] 
\Bigr)
\Bigr)
+ O(\frac{1}{n^2})
\right|
\end{equation}

Let us consider
\eqref{rouchegaucheSALEM}
and
\eqref{rouchedroiteSALEM}
at the first order for the asymptotic expansions, 
i.e. up to $O(1/n)$ - terms instead of
up to 
$O(\frac{1}{n}(\lo \lo n/ \lo n))$ - terms or
$O(1/n^2)$ - terms.
\eqref{rouchegaucheSALEM} becomes:
$$\frac{|-1+z_{1,n}|^2 \exp(\frac{2 \pi X}{a})}
{|z_{1,n}| - |-1+z_{1,n}| \exp(\frac{\pi X}{a})}$$
and \eqref{rouchedroiteSALEM} is equal to:
$$|-1 + z_{1,n}|
\left|
1 -
\exp\bigl(
\frac{\pi \, X}{a}
\bigr) \,
\exp\Bigl(
\pm \,
i \,
\frac{\pi \, \sqrt{1-X^2}}{a} 
\Bigr)
\right|
$$
and is independent of the sign of 
$\sin(\psi - \varphi)$.
Then
the inequality \eqref{rouchecercleSALEM} is 
equivalent to
\begin{equation}
\label{roucheequiv1SALEM}
\frac{|-1+z_{1,n}|^2 \exp(\frac{2 \pi X}{a})}
{|z_{1,n}| - |-1+z_{1,n}| \exp(\frac{\pi X}{a})}
<
|-1+z_{1,n}| \, 
\left|
1 -
\exp\bigl(
\frac{\pi \, X}{a}
\bigr) \,
\exp\Bigl(
\pm \,
i \,
\frac{\pi \, \sqrt{1-X^2}}{a} 
\Bigr)
\right|
,
\end{equation}
and \eqref{roucheequiv1SALEM} to
\begin{equation}
\label{amaximalfunctionXSALEM}
\frac{|-1 + z_{1,n}|}{|z_{1,n}|}
~  < ~ \, 
\frac{\left|
1 -
\exp\bigl(
\frac{\pi \, X}{a}
\bigr) \,
\exp\Bigl(
 \,
i \,
\frac{\pi \, \sqrt{1-X^2}}{a} 
\Bigr)
\right|
\exp\bigl(
\frac{-\pi \, X}{a}
\bigr)}{\exp\bigl(
\frac{\pi \, X}{a}
\bigr) +\left|
1 -
\exp\bigl(
\frac{\pi \, X}{a}
\bigr) \,
\exp\Bigl(
 \,
i \,
\frac{\pi \, \sqrt{1-X^2}}{a} 
\Bigr)
\right|} = \kappa(X,a).
\end{equation}
The right-hand side function
$\kappa(X,a)$ is 
a function of $(X, a)$, 
on $[-1, +1] \times [1, +\infty)$.
which is strictly decreasing for any
fixed $a$,
and reaches its minimum
at $X=1$; this minimum is always 
strictly positive. 
Consequently 
the inequality of Rouch\'e
\eqref{rouchecercleSALEM} will be satisfied
on $C_{1,n}$ once it is 
satisfied at $X = 1$, as claimed.

Hence, up to
$O(1/n)$-terms, the Rouch\'e condition
\eqref{amaximalfunctionXSALEM}, 
for any fixed $a$,
will be satisfied (i.e. for any
$X \in [-1,+1]$)
by the set of integers
$n = n(a)$ for which $z_{1,n}$
satisfies:
\begin{equation}
\label{amaximalfunctionSALEM}
\frac{|-1 + z_{1,n}|}{|z_{j,n}|} 
< \kappa(1,a) 
=
\frac{\left|
1 -
\exp\bigl(
\frac{\pi}{a}
\bigr)
\right|
\exp\bigl(
\frac{-\pi}{a}
\bigr)}{\exp\bigl(
\frac{\pi}{a}
\bigr) +\left|
1 -
\exp\bigl(
\frac{\pi}{a}
\bigr)
\right|} ,
\end{equation}
equivalently, from Proposition 
\ref{zedeUN},
\begin{equation}
\label{amaximalfunctionnnSALEM}
\frac{\lo n - \lo \lo n}{n} 
< \frac{\kappa(1,a)}
{1 + \kappa(1,a)} .
\end{equation}
In order to 
obtain  
the largest possible range 
of values of $n$,
the value of $a \geq 1$ has to be chosen
such that $a \to \kappa(1,a)$ is maximal
in \eqref{amaximalfunctionnnSALEM}
(Figure \ref{h1a}).
In the proof of Theorem 
\ref{cercleoptiMM} 
we have seen that
the function $a \to \kappa(1,a)$ 
reaches its maximum
$\kappa(1, a_{\max}) := 0.171573\ldots$
at $a_{\max} = 5.8743\ldots$.
We take
$a = a_{\max}$.

The slow decrease
of the functions of the variable $n$
involved in 
the terminants when $n$ tends to infinity,
as a factor of uncertainty on
\eqref{amaximalfunctionnnSALEM},
has to be taken into account
in \eqref{amaximalfunctionnnSALEM}.
It amounts to check 
numerically 
whether \eqref{rouchecercleSALEM} 
is satisfied 
for the small values
$18 \leq n \leq 100$
for $a = a_{\max}$, or not. 
Indeed, for the
large enough values of $n$,
the inequality
\eqref{amaximalfunctionnnSALEM}
is satisfied since
$\lim_{n \to \infty}
\frac{\lo n - \lo \lo n}{n} = 0$.
On the computer, 
the critical threshold
of $n = 32$  
is easily calculated, with
$(\lo 32 - \lo \lo 32)/32 = 
0.0694628\ldots$.
Then
$$\frac{\lo n - \lo \lo n}{n}
<
\frac{\kappa(1,a_{\max})}
{1 + \kappa(1,a_{\max})} = 0.146447\ldots\qquad 
\mbox{for all
$n \geq 32$} .$$
Let us note that
the last inequality
also holds for some values of
$n$ less
than $ 32$.
\end{proof}

\begin{corollary}
\label{smallestSALEM}
Let $\beta > 1$ be any algebraic 
number
of dynamical degree
$\dyg(\beta) \geq 32$. Then
the Parry Upper function
$f_{\beta}(z)$
admits a simple zero 
$\omega_{1,n}$ (of modulus
$< 1$) in the 
open disc $D(z_{1,n}, 
\frac{\pi |z_{1,n}|}{n \, a_{\max}})$.
\end{corollary}

\begin{proof}
The polynomial $G_{n}(z)$ has
simple roots.
Since
\eqref{rouchecercleSALEM} is satisfied, 
the Theorem of Rouch\'e states that
$f_{\beta}(z)$
and
$G_{n}(z) = -1 + z + z^n$
have the same number of roots, 
counted with
multiplicities, in the open disc
$D(z_{1,n}, 
\frac{\pi |z_{1,n}|}{n \, a_{\max}})$, giving the existence of an unique zero
$\omega_{1,n}$. 
\end{proof}

Let us prove that the first zero 
$\omega_{1,n}$ of
$f_{\beta}(z)$ is a zero of the minimal
polynomial of $\beta$, which lies in the cusp
of the fractal of Solomyak
($\S$ \ref{S4.2.2}). In the following theorem
the minimal polynomial need not be
monic.

\begin{theorem}
\label{secondseriesUU}
Let $\beta > 1$ be any algebraic 
number
of dynamical degree
$\dyg(\beta) \geq 32$. Then
the minimal polynomial $P_{\beta}(X)$ 
and
the Parry Upper function
$f_{\beta}(z)$ satisfy the 
canonical identity of the complex variable $z$ 
\begin{equation}
\label{factoUUSALEM}
P_{\beta}(z) ~=~ U_{\beta}(z)
\times f_{\beta}(z)
\end{equation}
where $U_{\beta}(z) = 
\frac{P_{\beta}(z)}{
f_{\beta}(z)}
\in \zb[[z]]$ is 
holomorphic
on the open disc $D_{1,n} =
\{ z \mid |z-z_{1,n}| < 
\frac{\pi |z_{1,n}|}{n \, a_{\max}}\}$
having 
no zero
on this disc. Moreover,
if $\omega_{1,n}$
is the unique zero of
$f_{\beta}(z)$ inside this disc, we have:
\begin{equation}
\label{UbetaNonnulenOmegajn}
U_{\beta}(\omega_{1,n})=
\frac{P'_{\beta}(\omega_{1,n})}
{f'_{\beta}(\omega_{1,n})}.
\end{equation}
The zero $\omega_{1,n}
= \omega_{1,n}(\beta)$ of
$f_{\beta}(z)$ is a nonreal complex zero
of modulus $< 1$ of the minimal polynomial
$P_{\beta}(z)$, and a continous function of
$\beta$.
\end{theorem}

\begin{proof}
The analytic function
$f_{\beta}(z)$ obeys the Carlson-Polya
dichotomy (Bell and Chen \cite{bellchen},
Carlson \cite{carlson}
\cite{carlson2},
Dienes \cite{dienes},
Polya \cite{polya}, 
Szeg\H{o} \cite{szego}),
as already mentioned in
Theorem 
\ref{splitBETAdivisibility+++} 
when $n \geq 260$. 
It is defined as an 
holomorphic function
on the open unit disc. 
Therefore, whatever the 
R\'enyi-Parry dynamics of $\beta$,
i.e. if $\beta$ is a Parry number 
or a nonParrynumber,
$f_{\beta}(z)$ is holomorphic on the open disc
$D_{1,n}$
this disc being included in
$|z| < 1$. 
The function
$f_{\beta}(z)$ admits only one zero
in the disc 
$D_{1,n}$, which is simple,
by Corollary \ref{smallestSALEM}.
This zero $\omega_{1,n}$
satisfies:
$|\omega_{1,n}-z_{1,n}| < 
\frac{\pi |z_{1,n}|}{n \, a_{\max}}$.
The minimal polynomial
$P_{\beta}$ is monic or not; in both cases
the 
function $U_{\beta}(z) = 
\frac{P_{\beta}(z)}{
f_{\beta}(z)}$ belongs to
$\zb[[z]]$ and is analytic 
on the open unit disc; the unit circle  
$|z|=1$ is natural boundary if and only if
$\beta$ is a nonParry number.
Inside the open unit disc, the function
$U_{\beta}(z)$ eventually admits 
poles and zeroes. 
Let us show that $U_{\beta}(z)$ has no pole 
inside $D_{1,n}$.
Indeed, since $P_{\beta}(z)$ 
has only simple roots, 
that 
$P_{\beta}(z) = U_{\beta}(z)
\times f_{\beta}(z)$
and that $f_{\beta}(z)$ has an unique zero
in $D_{1,n}$, the function
$U_{\beta}(z)$ either has a simple pole
at $\omega_{1,n}$ or does not vanish
at $\omega_{1,n}$. For any $z \in D_{1,n}$,
$z \neq \omega_{1,n}$, $U_{\beta}(z)$
is holomorphic at $z$.
 Let us show that the
function $U_{\beta}(z) = 
\frac{P_{\beta}(z)}{
f_{\beta}(z)}$
has no pole at $\omega_{1,n}$.
Deriving the identity
$P_{\beta}(X) = U_{\beta}(X) \times f_{\beta}(X)$
and specializing the formal variable $X$
to the complex variable $z$
gives:
$$P'_{\beta}(z) = U'_{\beta}(z) \times 
f_{\beta}(z)
+
U_{\beta}(z) \times f'_{\beta}(z).$$
Therefore,
$$P'_{\beta}(\omega_{1,n}) = 
U_{\beta}(\omega_{1,n}) \times 
f'_{\beta}(\omega_{1,n}),$$
from which we deduce 
\eqref{UbetaNonnulenOmegajn} 
and the identification of the zero 
$\omega_{1,n}$ of $f_{\beta}(z)$ in
$D_{1,n}$ as a conjugate of $\beta$.
Of course, the minimal polynomial may 
possibly admit other roots
inside $D_{1,n}$; 
in this case these other roots would be zeroes
of the function $U_{\beta}(z)$, not
of $f_{\beta}(z)$.
Since $|\omega_{1,n}| < 1$, that $z_{1,n}$
is a nonreal complex number and the radius 
$\pi |z_{1,n}|/(n \, _{\max})$ 
of the disc
$D_{1,n}$ is small enough, the polynomial
$P_{\beta}(z)$ admits the nonreal complex 
root $\omega_{1,n}$
inside the open unit disc.
The map $\beta \to 
\omega_{1,n}(\beta)$ is continuous
by Corollary 
\ref{zeroesParryUpperfunctionContinuity}
in $\S$ \ref{S4.4}.
In Flatto, Lagarias and Poonen 
\cite{flattolagariaspoonen}
the first root
$\omega_{1,n}$, and its continuity with
$\beta$,
was studied as a zero of 
$f_{\beta}(z)$; here, 
after identifying it as a Galois conjugate
of $\beta$, 
$\omega_{1,n}$ is 
a zero of
the minimal polynomial
$P_{\beta}(z)$. In other terms,
the normal closure
of $\qb(\beta)$
admits the couple of nonreal complex
embeddings $\beta 
\rightarrow (\omega_{1,n}, 
\overline{\omega_{1,n}})$ 
continuous with
$\beta$.
\end{proof}

\subsection{A lower bound for the set of Salem numbers. Proof of Theorem \ref{mainpetitSALEM}}
\label{S7.2}

Negative and (positive)
Salem numbers occur 
in number theory, e.g.
for graphs or integer
symmetric matrices in \cite{mackeesmyth}
\cite{mackeesmyth2}
\cite{mackeesmyth3},
and in other domains (cf $\S$ \ref{S2.2}),
like Alexander polynomials of links
of the variable ``$-x$", e.g.
in Theorem 
\ref{hironakathm2}
\cite{hironaka}.
The closed set $S$ of Pisot numbers
and its successive derivatives
$S^{(i)}$,
were extensively studied
by Dufresnoy and Pisot,
and their students (Amara, \ldots),
by means of compact families of 
meromorphic functions, following ideas of Schur
 \cite{bertinetal}.
On the contrary the set of Salem numbers
is badly known.
Association equations between Pisot numbers
and Salem numbers
were used to 
study these numbers \cite{bertinetal}. 
Salem (1944) (Samet \cite{samet})
proved that every Salem number 
is the quotient of two Pisot numbers.
In the same direction
association equations
between Salem numbers and 
(generalized) Garsia numbers
(Hare and Panju \cite{harepanju})
were recently studied,
using 
interlacing theory
on the unit circle 
\cite{guichardvergergaugry} 
\cite{lakatos2}
\cite{lakatos4}
\cite{lakatos5}
\cite{mackeesmyth3}.

As counterpart, Salem numbers are linked
to units: they are given by closed formulas
from Stark units in Chinburg
\cite{chinburg} \cite{chinburg2},
exceptional units in Silverman
\cite{silverman3}. 
From \cite{chinburg2} they are related 
to relative regulators of number fields
\cite{amoroso5}
\cite{amoroso6}
\cite{christopoulosmackee}
\cite{costafriedman}
\cite{ghatehironaka}.

\vspace{0.2cm}

\noindent
{\it Proof of Theorem \ref{mainpetitSALEM}:}
Assume that $\beta$ is a Salem number
of dynamical degree
$n = \dyg(\beta) \geq 32$. 
Its minimal polynomial $P_{\beta}(X)$
would admit $\beta$, $1/\beta$ as real roots, 
the remaining roots being
on the unit circle, as 
(nonreal) complex-conjugated pairs.
By Theorem \ref{secondseriesUU} 
it would 
admit the pair of nonreal roots
$(\omega_{1,n} , \overline{\omega_{1,n}})$ 
as well, strictly inside the open unit disc.
This fact is impossible. 
We deduce that
$\beta > \theta_{31}^{-1}=
1.08545\ldots$.

\subsection{Totally real algebraic integers, Bogomolov property. Proof of Theorem \ref{main_hminoration_totallyreal}}
\label{S7.3}

Let $\lb$ be a totally real algebraic number field,
or a CM field (a totally complex quadratic extension
of a totally real number field). Then, for any 
nonzero algebraic integer
$\alpha \in \lb$, of
degree $d$, not being a root of unity,
Schinzel \cite{schinzel2} obtained
the minoration
\begin{equation}
\label{minoSCHINZEL}
{\rm M}(\alpha) \geq \theta_{2}^{-d/2}
=
\left(
\frac{1+\sqrt{5}}{2}\right)^{d/2}.
\end{equation}
More precisely, if
$H(X) \in \zb[X]$ is monic
with degree $d$, $H(0) \!= \pm 1$ and
$H(-1) H(1) \!\neq 0$, and if
the zeroes of $H$ are all real,
then
\begin{equation}
\label{schinzel0001}
{\rm M}(H) \geq 
\Bigl(\frac{1 + \sqrt{5}}{2}
\Bigr)^{d/2} 
\end{equation}
with equality if and only if
$H(X)$ is a power of
$X^2 -X-1$. Bertin \cite{bertin}
improved Schinzel's minoration
\eqref{schinzel0001}
for the algebraic integers
$\alpha$, of degree $d$, of norm
$N(\alpha)$, which are totally real,
as
$${\rm M}(\alpha)
\geq 
\max\bigl\{\theta_{2}^{-d/2},
\sqrt{N(\alpha)}\,
\theta_{2}^{-\frac{d}{2|N(\alpha)|^{1/d}}}
\bigr\}.$$
The totally real algebraic numbers
form a subfield, denoted by
$\qb^{tr}$, in
$\overline{\qb} \cap \rb$.
Following \cite{bertin}, 
the natural extension of a Salem number
is a $\nu$-Salem number,
intermediate
between Salem numbers and 
totally real algebraic numbers.
Let us define a $\nu$-Salem
as
an algebraic integer $\alpha$
having $\nu$ conjugates
outside $\{|z| \geq 1\}$ and at least
one conjugate $\alpha^{(q)}$ satisfying
 $|\alpha^{(q)}| = 1$;
denote by $2 \nu+2k$ its degree.
Such an algebraic integer is 
totally real in the sense that 
its conjugates of modulus $> 1$ are all real,
and then
$${\rm M}(\alpha)
\geq \theta_{2}^{-\frac{\nu}{2^{k/\nu}}}.$$
Further, extending Pisot numbers,
lower bounds of ${\rm M}(\alpha)$
were obtained by Zaimi 
\cite{zaimi} \cite{zaimi2}
when $\alpha$ is a $K$-Pisot number.
Rhin \cite{rhin}, following Zaimi 
(cf references in \cite{rhin}), 
obtained minorations of
${\rm M}(\alpha)$ for totally 
positive algebraic integers $\alpha$ 
as functions
of the discriminant ${\rm disc}(\alpha)$.
Let $K$ be an algebraic number field and $\alpha$ an algebraic integer of minimal polynomial
$R$ over $K$; by definition \cite{bergemartinet}
$\alpha$ is $K$-Pisot number if, for any embedding 
$\sigma : K \to \cb$, $\sigma(R_K)$ admits only one root of modulus $> 1$ and no root of modulus 1.
Denote by $\Delta$ the discriminant of $K$. 
Lehmer's problem and small discriminants
were studied by Mahler (1964),
Bertrand \cite{bertrand}, Matveev
\cite{matveev2},
Rhin \cite{rhin}.
For any $K$-Pisot number $\alpha$,
Zaimi \cite{zaimi} \cite{zaimi2}
showed
\begin{equation}
\label{zaimiquadra}
{\rm M}(\alpha) \geq \frac{\sqrt{\Delta}}{2} \qquad \qquad
K~ \mbox{quadratic},
\end{equation}
\begin{equation}
\label{zaimicubic}
{\rm M}(\alpha) \geq \frac{\Delta^{1/4}}{\sqrt{6}} \qquad \qquad
K~ \mbox{cubic and totally real}.
\end{equation}
Other minorations of totally positive  
algebraic integers
were obtained by
Mu and Wu \cite{muwu}.
Denote $\zb^{tr} :=
\qb^{tr} \cap \mathcal{O}_{\overline{\qb}}$.
Because the degree $d$ of the algebraic
number commonly
appears in the exponent of the lower 
bounds of the Mahler measure,
the (absolute logarithmic)
Weil height $h$ is more adapted than
the Mahler measure.
Schinzel's bound, originally
concerned with the algebraic integers in
$\zb^{tr}$, 
reads:
$$\alpha \in \zb^{tr}, \alpha \neq 0, \neq \pm 1 
~~\Rightarrow~~ h(\alpha)
\geq h\bigl(
\theta_{2}^{-1}
\bigr) = \frac{1}{2} \lo (\frac{1+\sqrt{5}}{2})
= 0.2406059\ldots$$
Smyth \cite{smyth3} \cite{smyth4}
proved that the set
$$\{\exp(h(\alpha)) \mid \alpha ~
\mbox{totally real algebraic integer}, 
\alpha \neq 0, \neq \pm 1\}$$
is everywhere dense  in $(1.31427\ldots, \infty)$;
in other terms
$$
\liminf_{\alpha \in \zb^{tr}}
~h(\alpha)
~\leq~ \lo (1.31427\ldots) = 0.27328\ldots
$$
Flammang \cite{flammang} completed
Smyth's results by showing
$$
\liminf_{\alpha \in \zb^{tr}}
~h(\alpha)
~\geq~ \frac{1}{2} \lo (1.720566\ldots) = 0.271327\ldots
$$
with exactly 6 isolated points in
the interval $(0, 0.271327\ldots)$, the smallest
one being Schinzel's bound
$0.2406059\ldots$
In fact, passing
from algebraic integers to algebraic numbers
lead to various smaller minorants of
$h(\alpha)$: for instance
$(\lo 5)/12 = 0.134119\ldots$ 
by Amoroso and Dvornicich 
\cite{amorosodvornicich} for
any nonzero $\alpha \in \lb$ which is not
a root of unity, where
$\lb/\qb$ is an abelian 
extension of number fields,  
or, by
Ishak, Mossinghoff, Pinner and Wiles
\cite{ishakmossinghoffpinnerwiles}, 
for nonzero
$\alpha \in \qb(\xi_m)$, not being a root of unity,

(i)  $h(\alpha) \geq 0.155097\ldots$,
for $3$ not dividing $m$,

(ii) $h(\alpha) \geq 0.166968\ldots$,
for $5$ not dividing $m$, unless $\alpha = 
\alpha_{0}^{\pm 1} \zeta$, 
with $\zeta$ a root of unity, 
whence $h(\alpha) \geq (\lo 7)/12 = 0.162159\ldots$,
$\alpha_{0}$ being a root
of $7 X^{12} - 13 X^6 + 7$,

(iii) $h(\alpha) \geq 0.162368\ldots$,
for $7$ not dividing $m$.

\noindent
(cf also 
\cite{amorosodvornicich}
\cite{amorosozannier2}
\cite{garza} 
\cite{garza2}
\cite{ishakmossinghoffpinnerwiles}
\cite{pottmeyer}
for other results).
For totally real numbers $\alpha$,
Fili and Miner \cite{filiminer}, using 
results of Favre and Rivera-Letelier 
\cite{favreriveraletelier} on
the equidistribution of points of small Weil height,
obtained the limit infimum of the height
$$\liminf_{\alpha \in \qb^{tr}}
~h(\alpha)
~\geq~ \frac{140}{3}
\left(
\frac{1}{8} - \frac{1}{6 \pi}
\right)^2
=
0.120786\ldots
$$

Bombieri and Zannier
\cite{bombierizannier}
have recently introduced the 
concept of 
``Bogomolov property", by analogy with 
the ``Bogomolov Conjecture". 
Let us recall it. 
Assuming a fixed 
choice of embedding
$\overline{\qb} \to \cb$, 
a field 
$\kb \subset \overline{\qb}$
is said to possess the 
Bogomolov property relative to $h$
is and only if 
$h(\alpha)$ is zero or bounded 
from below by a positive constant
for all $\alpha \in \kb$.
The search of small Weil's heights
is important \cite{amorosodvornicich}
\cite{amorosonuccio}
Choi \cite{choi}.
Every number field 
has the Bogomolov property
relative to $h$
by Northcott's theorem
\cite{schmidtw} \cite{schmidtw2}.
Other fields are known to possess
the Bogomolov property:
(i) $\qb^{tr}$ \cite{schinzel2},
(ii) finite extensions
of the maximal abelian extensions
of number fields \cite{amorosozannier}
\cite{amorosozannier2},
(iii) totally $p$-adic fields 
\cite{bombierizannier}, i.e. for algebraic numbers
all of whose conjugates lie in
$\qb_p$,
(iv) $\qb(E_{tors})$ 
for $E/\qb$ an elliptic curve
\cite{habegger}.

Our result is that the 
Weil height of any nonzero totally real 
algebraic number 
$\neq \pm 1$ is bounded
from below as 
stated in
Theorem \ref{main_hminoration_totallyreal}.
The proof below is actually another proof
of Schinzel's theorem which states
(in Bombieri Zannier's notation) that
the field $\qb^{tr}$ has the Bogomolov property relative
to $h$.

\vspace{0.2cm}

\noindent
{\it Proof of Theorem 
\ref{main_hminoration_totallyreal}:}

\noindent
(i) Let $\alpha$ be a totally 
real algebraic integer
$\neq 0, \neq \pm 1$, $\deg(\alpha) \geq 1$. 
Assume that
its minimal polynomial
$P_{\alpha}(x) = \prod_{i=1}^{\deg(\alpha)}
(x-\alpha^{(i)}) $ is totally positive, i.e.
all its roots are real and positive. 
The Mahler measure 
${\rm M}(P_{\alpha})$ of the minimal polynomial
of $\alpha$ is equal to
the Mahler measure 
${\rm M}(P_{\alpha}^{*})$ 
of its reciprocal polynomial.
If $P_{\alpha}$ is reciprocal, then
the number of conjugates
$\alpha^{(i)} > 1$ is equal
to the number of conjugates
$\alpha^{(i)}$ which are in 
$(0,1)$. Denote by $\beta$
the smallest conjugate of $\alpha$
which is $> 1$.
The following inequality holds:
$${\rm M}(\alpha) \geq \beta^{\deg(\alpha)/2}.$$ 
We now apply Theorem \ref{secondseriesUU}
to $\beta$. 
The conjugates of $\beta$ are 
the conjugates of
$\alpha$. They all lie on the real line.
If we assume $n = \dyg(\beta) \geq 32$, 
we arrive at a contradiction 
since $P_{\alpha}$ would admit 
the nonreal complex $\omega_{1,n}$
as zero. Therefore
$\beta > \theta_{31}^{-1}$
and 
$$\frac{\lo {\rm M}(\alpha)}
{\deg(\alpha)} 
= h(\alpha) > \frac{1}{2} \,\lo \theta_{31}^{-1}
=
0.04...$$
If $P_{\alpha}$ is not reciprocal and that 
the number $x$ of conjugates of $\alpha$ 
which are $> 1$
is 
$\geq \deg(\alpha)/2$, we denote
by $\beta$ the smallest conjugate 
of $\alpha$ which is 
$> 1$. Then ${\rm M}(\alpha) \geq \beta^x \geq
\beta^{\deg(\alpha)/2}$ and 
$h(\alpha) > \frac{1}{2} \,
\lo \theta_{31}^{-1} = 0.04...$ as above 
with the same argument.
If $P_{\alpha}$ is not reciprocal and that 
$x < \deg(\alpha)/2$, then we consider
$P_{\alpha}^{*}$. Then the number
of conjugates of 
$\alpha^{-1}$ which are $> 1$
is $\geq \deg(\alpha)/2$.
Let $\beta$ denote the smallest
conjugate of $\alpha^{-1}$ which is $> 1$.
Then ${\rm M}(\alpha) =
{\rm M}(P_{\alpha}^{*}) \geq 
|P_{\alpha}(0)| \, \beta^{\deg(\alpha)/2}$.
All the roots of 
$P_{\alpha}^{*}$ are real.
The same argument (Theorem \ref{secondseriesUU}) 
leads to
$$h(\alpha^{-1})= h(\alpha) \geq 
\frac{\lo |P_{\alpha}(0)|}{\deg(\alpha)}
+
\frac{1}{2} \,\lo \theta_{31}^{-1}
\geq
\frac{1}{2} \,\lo \theta_{31}^{-1}
= 0.04...$$

\noindent
(ii) The case where $\alpha$ is a
totally real algebraic integer 
$\neq 0, \neq \pm 1$ having a minimal
polynomial $P_{\alpha}$ not
totally positive is deduced from (i). 
Indeed,
the polynomial
$(-1)^{\deg(\alpha)} P_{\alpha}(x)
P_{\alpha}(-x)$ is totally positive,
of degree $2 \deg(\alpha)$,
and its Mahler measure is
equal to ${\rm M}(P_{\alpha})^2$.
If $P_{\alpha}(x) P_{\alpha}(-x)$
has a number of roots $> 1$ 
greater than, or equal to, 
$\deg(\alpha)$, then
$\beta^2$ denotes the smallest root of
$P_{\alpha}(x) P_{\alpha}(-x)$ 
which is $> 1$.
If not, $\beta^2$ denotes the 
smallest root of
$x^{2 \deg(\alpha)} 
P_{\alpha}(x^{-1}) P_{\alpha}(-x^{-1})$ 
which is $> 1$.
Then, as above: 
${\rm M}(\alpha)^2 \geq (\beta^2)^{\deg(\alpha)}$
with
$\beta^2 > \theta_{31}^{-1}$.
We deduce the minoration
of $h(\alpha)$ in
\eqref{hALPHATotallyRealmini}.

\noindent
(iii) Let $\alpha$ be a totally real algebraic 
number $\neq 0, \neq \pm 1$
which is a noninteger. 
Let 
$P_{\alpha}(x) = c \prod_{i=1}^{\deg(\alpha)}
(x-\alpha^{(i)})$ denote
the minimal polynomial of $\alpha$,
for some integer $c \geq 2$.
Using (i) and Theorem \ref{secondseriesUU}, 
with $P_{\alpha}(x)$
totally positive and reciprocal, 
the Mahler measure of $\alpha$
satisfies:
${\rm M}(P_{\alpha}) \geq c \,
\beta^{\deg(\alpha)/2}$
where
$\beta$ is
the smallest conjugate of
$\alpha$
which is $> 1$,
and $\beta > \theta_{31}^{-1}$.
Hence,
$$h(\alpha) \geq
\frac{\lo c}{\deg(\alpha)}
+
\frac{1}{2} \lo \theta_{31}^{-1}
\geq \frac{1}{2} \lo \theta_{31}^{-1} .
$$
If  $P_{\alpha}$ is totally positive and nonreciprocal
we conclude as in (i). If
$P_{\alpha}$ is not totally positive, we
invoke the same arguments as in (ii).
Hence \eqref{hALPHATotallyRealmini} holds for all
nonzero totally real 
algebraic numbers $\alpha$ 
$\neq \pm 1$.

\section{Sequences of small algebraic integers converging to $1^+$ in modulus and limit equidistribution of conjugates on the unit circle. Proof of Theorem \ref{main_EquidistributionLimitethm}}
\label{S8}

Given a convergent sequence of
algebraic numbers the
limit equidistribution
of the conjugates
\cite{bilu}
\cite{chambertloir}
\cite{dandreagalligosombra}
\cite{dandreanarvezclausssombra}
\cite{favreriveraletelier}
\cite{hughesnikeghbali}
\cite{petsche}
\cite{petsche2}
\cite{pritsker}
\cite{rumely}
often relies upon
the Erd\H{o}s-Tur\'an-Amoroso-Mignotte theory 
\cite{amorosomignotte}
\cite{belotserkovski}
\cite{erdosturan}
\cite{ganelius}
\cite{mignotte4}
\cite{vergergaugry3}.
We will make use 
of Belotserkovski's Theorem
\cite{belotserkovski},
recalled below as Theorem
\ref{belotserkovskitheorem},
which prefigurates
Bilu's theorem on 
the $n$-dimensional torus 
\cite{bilu};
the discrepancy function
of equidistribution
given by this theorem is well adapted to 
become a function 
of only the dynamical degree.

\begin{theorem}[Belotserkovski]
\label{belotserkovskitheorem}
Let $F(x) = a \prod_{i=1}^{m} (x- \alpha^{(i)})
\in \cb[x], m \geq 1$,
be a polynomial with roots
$\alpha^{(k)} = r_k e^{i \varphi_k}$,
$0 \leq \varphi_k \leq 2 \pi$. For
$0 \leq \varphi \leq \psi \leq 2 \pi$,
denote
$N_{F}(\varphi , \psi)\! =
Card\{k \mid \varphi \leq \varphi_k \leq \psi\}$.
Let $0 \leq \epsilon, \delta \leq 1/2$
and
$$\sigma_{dis} = 
\max\left(
m^{-1/2} \, \lo (m+1),
\sqrt{- \epsilon \, \lo (\epsilon)},
\sqrt{- \delta \,\lo (\delta)}
\right).$$ 
If
$|r_k - 1| \leq \epsilon$
for
$1 \leq k \leq m$,
and
$|\lo a| \leq \delta m$ are satisfied,
then, for some (universal, in the sense 
that it does not depend upon $F$) constant $C > 0$, 
\begin{equation}
\label{discrep_prorotata}
\left|
\frac{1}{m} N_{F}(\varphi, \psi)
-
\frac{\psi - \varphi}{2 \pi}\right|
\leq 
C\, \sigma_{dis} \qquad
\mbox{for all}~~
0 \leq \varphi \leq \psi \leq 2 \pi .
\end{equation}
\end{theorem}

The multiplicative group of nonzero elements
of $\cb$, resp. $\qb$, is denoted by 
$\cb^{\times}$, resp.
$\qb^{\times}$. The unit Dirac measure
supported at $\omega \in \cb$
is denoted
by $\delta_{\omega}$.
We denote by
$\mu_{\tb}$ the (normalized)
Haar measure
(unit Borel measure), invariant by rotation,
that is supported on the unit circle
$\tb
=
\{z \in \cb \mid |z|=1\}$,
compact subgroup of $\cb^{\times}$,
i.e. with
$\mu_{\tb}(\tb) = 1$.
Given $\alpha \in \overline{\qb}^{\times}$,
of degree
$m = \deg(\alpha)$,
we define
the unit Borel measure
(probability)
$$\mu_{\alpha} =
\frac{1}{\deg(\alpha)}
\sum_{j=1}^{\deg(\alpha)}
\delta_{\sigma(\alpha)}$$
on $\cb^{\times}$, 
the sum being taken over
all $m$ embeddings 
$\sigma : \qb(\alpha) \to \cb$.
A sequence $\{\gamma_s\}$
of points of
$\overline{\qb}^{\times}$
is said to be {\em strict} 
if any proper subgroup
of $\overline{\qb}^{\times}$ contains 
$\gamma_s$ for only finitely many values of
$s$.

Theorem 6.2 in \cite{vergergaugry6}
shows that limit equidistribution of conjugates
occurs on the unit circle
for the sequence of Perron numbers
$\{\theta_{n}^{-1} \mid n = 2, 3, 4, \ldots\}$,
as $\mu_{\theta_{n}^{-1}} \to
\mu_{\tb}, n \to \infty$. 
All these Perron numbers have
a Mahler measure $> \Theta$.
We now give a generalization of this limit
result
to convergent sequences of algebraic integers of 
small Mahler measure,
$< \Theta$, where ``{\em convergence 
to $1$}" 
has to be taken in the sense 
of the ``house".
The Theorem is Theorem 
\ref{main_EquidistributionLimitethm}.

{\it Proof of Theorem \ref{main_EquidistributionLimitethm}}:
(i) Denote generically by
$\alpha \in \mathcal{O}_{\overline{\qb}}$
any
element of 
$(\alpha_{q})_{q \geq 1}$.
Let $m = \deg(\alpha)$
and
$\beta = \house{\alpha} 
\in (\theta_{n}^{-1},
\theta_{n-1}^{-1}), n \geq 260$. 
Using the inequality
\eqref{dygdeg}
 between 
$m$
and the dynamical degree
$n = \dyg(\alpha)
=\dyg(\beta)$,
there exists a constant $c_1 > 0$ such that
$$
\frac{\lo (m+1)}{\sqrt{m}}
\leq c_1 \frac{\lo n}{\sqrt{n}}
.$$
On the other hand, 
the minimal polynomial
$P_{\alpha} = P_{\beta}$
is reciprocal and all its roots
$\alpha^{(k)}$, including $\beta$
and $1/\beta$
by Theorem \ref{divisibilityALPHA},
lie in the annulus
$\{z \mid
\frac{1}{\beta} \leq |z| \leq \beta
\}$.
As a consequence, using 
Theorem \ref{betaAsymptoticExpression},
there exists a constant $c_2 > 0$ such that
$$||\alpha^{(k)}| -1| \leq \epsilon ,
\quad 1 \leq k \leq m,
\quad \mbox{with}~~
\epsilon = c_2 \frac{\lo n}{n} .
$$
We take $\delta = 0$
in the definition of $\sigma$
in Theorem \ref{belotserkovskitheorem}
since $P_{\alpha}$ is monic.
We deduce that the discrepancy function,
i.e. the upper bound
in the rhs of 
\eqref{discrep_prorotata},
is equal to
$C \sigma_{dis} = c_3 \frac{\lo n}{\sqrt{n}}$
for some constant $c_3 > 0$.
Hence,
\begin{equation}
\label{discrep_prorotataALPHA}
\left|
\frac{1}{m} N_{P_{\alpha}}(\varphi, \psi)
-
\frac{\psi - \varphi}{2 \pi}\right|
\leq 
c_3 \frac{\lo n}{\sqrt{n}}
 \qquad
\mbox{for all}~~
0 \leq \varphi \leq \psi \leq 2 \pi .
\end{equation}
The discrepancy function of 
\eqref{discrep_prorotataALPHA} tends to
$0$ if $n$ tends to infinity.
By Theorem \ref{betaAsymptoticExpression}
and Theorem \ref{nfonctionBETA}, for
$1 < \beta < \theta_{260}^{-1}$,
$$\beta~ \to ~1^+ \Longleftrightarrow~
n = \dyg(\beta) \to \infty,$$
so that the sequence of Galois orbit measures
in \eqref{haarmeasurelimite} 
converge for the weak topology
as a function of the dynamical degree.

(ii) The sequence $(\alpha_q)$ is strict
since the sequence
$(\house{\alpha_q})$ only admits 1 as
limit point:
$\limsup_{q \to \infty} \house{\alpha_q}
=
\lim_{q \to \infty}\house{\alpha_q}=1$
and the number 
Card$\{
\alpha_q \in 
(\theta_{n}^{-1} , \theta_{n-1}^{-1})\}$
between two successive Perron numbers of 
$(\theta_{n}^{-1})$, for every $n \geq 3$, 
is finite.
In the space of probality measures
equipped with the weak topology,
the reformulation of 
\eqref{discrep_prorotataALPHA}
means \eqref{haarmeasurelimite}, 
equivalently
\eqref{haarlimitefunvtions}.

\section{Some consequences in Geometry}
\label{S9} 

The reader can now transform the following 
Conjectures into Theorems. Some Conjectures
about the problems of Lehmer in higher dimension
($\S$ \ref{S2.2}) will be reconsidered by 
the author later on.

Lehmer's Conjecture for Salem numbers
is called "no small Salem number Conjecture",
or "Salem's Conjecture" for short,
in Breuillard and Deroin \cite{breuillardderoin}.

\begin{theorem}[Sury \cite{sury}]
Salem's Conjecture 
is true if and only if
there exists a neighbourhood $W$ of the identity
in $SL(2, \rb)$ such that, 
for all cocompact
arithmetic Fuchsian 
groups $\Gamma$, the intersection
$\Gamma \cap W$ consists only of 
elements of finite order.
\end{theorem}

This beautiful geometric reformulation
(Ghate and Hironaka \cite{ghatehironaka}
$\S$ 3.5,
Mclachlan and Reid
\cite{maclachlanreid} pp 378-380,
Margulis \cite{margulis})
is rephrased in the equivalent statement 
in \cite{breuillardderoin}:

{\it Salem's Conjecture holds if and only if
there is a uniform positive lower bound
on the length of closed geodesics in arithmetic
hyperbolic 2-orbifolds}.

We refer the reader to \cite{breuillardderoin}
for the definitions of the terms used 
in the following Theorem,
Breuillard and Deroin
obtain the spectral reformulation of Salem's
Conjecture, using the Cheeger-Buser
inequality:

\begin{theorem}[Breuillard - Deroin]
Salem's Conjecture holds 
if and only if
there exists a uniform
constant $c > 0$ such that
$$\lambda_{1}(\widetilde{\Sigma})
\geq \frac{c}{{\rm area}(\widetilde{\Sigma})}$$
for all 2-covers $\widetilde{\Sigma}$ 
of all compact congruence arithmetic hyperbolic
2-orbifolds $\Sigma$.
\end{theorem}

In \cite{ghatehironaka} p. 304, 
Ghate and Hironaka
mention that if
Lehmer's Conjecture is true, then
the following Conjecture is also true:

\begin{conjecture}(Margulis)
\label{CJ13}
Let $G$ be a connected semi-simple group over $\rb$,
having rank$_{\rb}(G) \geq 2$. Then there is a neighbourhood
$U \subset G(\rb)$ of the identity such that for any
irreducible cocompact lattice
$\Gamma \subset G(\rb)$, the intersection
$\Gamma \cap U$ consists only
of elements of finite order.
\end{conjecture}
Let us give a proof of Conjecture \ref{CJ13}
from the two 
arguments of Margulis
(\cite{margulis}, Theorem (B) p. 322): 
(i) first,
the arithmeticity Theorem 1.16 in
\cite{margulis}, p. 299, 
and (ii) the following statement
(Margulis \cite{margulis}, p. 322):

{\it Let $P(x) = x^n + a_{n-1} x^{n-1}
+ \ldots + a_0$ be an irreducible monic 
polynomial with integral coefficients.
Denote by
$\beta_{1}(P), \ldots, \beta_{n}(P)$
the roots of $P$ and by $m(P)$ the number 
of those $i$ with $1 \leq i \leq n$
and $|\beta_{i}(P)| \neq 1$.
Then
\begin{equation}
\label{margulisminoration}
{\rm M}(P) = \prod_{1 \leq i \leq n}
\max\{1, |\beta_{i}(P)|\} > d
\end{equation}
where the constant $d > 1$ depends only
upon $m(P)$ (and does not depend upon $n$). 
}
The minorant $d$ is universal and
is given by
Theorem \ref{mainLEHMERtheorem}
(ex-Lehmer's Conjecture), hence the result.

But
the dependency
of the minorant $d$ of
${\rm M}(P)$
in \eqref{margulisminoration},
expected by Margulis,
with the
number of roots $m(P)/2$
lying outside 
the closed unit disk, 
or equivalently inside the open unit disk
(the polynomial $P$ can be assumed
of small Mahler measure $< \Theta$, hence
reciprocal by Smyth's Theorem), 
and
not with the degree $n$
of $P$, is not clear in view of
Theorem \ref{omegajnexistence}, 
Theorem \ref{divisibilityALPHA}
and
Proposition \ref{argumentlastrootJn}. 
Indeed, the following minorant
$$m(P) \geq 2 (1 + J_{\dyg(P)})$$
can only be deduced from the present study,
this minorant being two times
the cardinal of the lenticulus of roots
associated with the dynamical
degree of the house of the polynomial 
$P$; what can be said is that
the degree $n=\deg(P)$
of $P$ is not involved
in this minorant of $m(P)$, and therefore
that the constant $d$ in \eqref{margulisminoration}
is likely to depend upon the
dynamical degree $\dyg(P)$
(cf $\S$ 1 for its definition).

\vspace{0.05cm}

Ghate and Hironaka, 
in \cite{ghatehironaka}
$\S$ 3.4, proved that an immediate consequence
of Salem's Conjecture is the following Conjecture
(cf the
Corollaries 1 to 4 in 
\cite{emeryratcliffetschantz}):

\begin{conjecture}(Minimization Problem for Geodesics)
\label{main_minimizationCJ}
There is a geodesic of minimal length 
amongst all closed geodesics on 
all arithmetic hyperbolic surfaces.
\end{conjecture}

More numerically,
for short geodesics,
Neumann and Reid 
($\S$ 4.4 and 4.5 in \cite{neumannreid})
proved: if Lehmer's Conjecture is true,
then the following Conjecture is true.

\begin{conjecture}(Neuman-Reid)
There is a universal lower bound
for the lengths of geodesics
in closed arithmetic hyperbolic orbifolds
(the current guess is approximately
0.09174218, or twice this, 0.18348436, if the 
orbifold is derived from a quaternion algebra).
\end{conjecture}

\vspace{0.05cm}

In \cite{emeryratcliffetschantz}
Emery, Ratcliffe and Tschantz show 
the important role played by
the smallness of Salem numbers in 
the existence of lower bounds:

\begin{theorem}[Emery, Ratcliffe, Tschantz]
Let $n \geq 1$ be an integer, and
$$b_n := \min\{\lo \lambda \mid 
\lambda \mbox{is a Salem number with}
~\deg \lambda \leq n+1\}$$
$$c_n :=
\min\{\frac{1}{2} \lo \lambda \mid 
\lambda \mbox{is a Salem number with}
~\deg \lambda \leq n+1 $$
$$\mbox{which is square-rootable over} ~\qb \}.$$
Let $\Gamma$ be a non-cocompact
arithmetic group of isometries
of $\hb^n$
and $C$ is a closed geodesic
in $\hb^n / \Gamma$, then

(i) if $n$ is even, 
length$(C) \geq b_n$, and this 
lower bound is sharp
for each even $n > 0$,

(ii) if $n$ odd and $n > 1$,
length$(C) \geq c_n$,
and this 
lower bound is sharp
for each odd integer $n > 1$.

\end{theorem}

In \cite{silverwilliams3}
Silver and Williams
investigate the problem of Lehmer in terms of
generalized growth rates of Lefshetz 
numbers of iterated 
pseudo-Anosov surface homeomorphisms. 
They prove several
statements equivalent to Lehmer's
Conjecture. Let us mention the following
(the definitions of the terms can be found
in \cite{silverwilliams3} \cite{rolfsen}):

\begin{theorem}[Silver-Williams \cite{silverwilliams3}]
Lehmer's Conjecture is true if and only if
there exists $c > 0$ such that, for all
fibered hyperbolic knots
$K$ in a lens space $L(n,1)$, $n > 0$, 
the Mahler measure of the Alexander polynomial
$\Delta_{K}(t)$ of $K$ satisfies:
$$1 + c < {\rm M}(\Delta_{K}(t)).$$ 
\end{theorem}

\section*{Appendix: Angular asymptotic sectorization of the roots $z_{j,n}, \omega_{j,n}$, of the Parry Upper functions, in lenticular sets of zeroes -- notations for transition regions}
\label{nottata}

The Poincar\'e
asymptotic expansions of the roots $z_{j,n}$ 
of $G_{n}(z)=-1+z+z^n$,
lying in the first quadrant of $\cb$, 
are
divergent formal series of functions of the 
couple of {\it {\bf two variables}} which is:
\begin{itemize}
\item[$\bullet$] $\displaystyle \bigl( n , \frac{j}{n} \bigr),$ 
\quad \quad in the angular sector: \quad 
$\displaystyle \frac{\pi}{2} >~ \arg z ~>~ 2 \pi \frac{\lo n}{n}$,
\item[$\bullet$] $\displaystyle \bigl( n , \frac{j}{\lo n} \bigr),$ \hspace{0.1cm} in the angular sector 
(``bump" sector):
\, $\displaystyle 2 \pi \frac{\lo n}{n} ~>~ \arg z ~\geq~ 0$. 
\end{itemize}
 
\noindent
In the bump sector (cusp sector of Solomyak's fractal
$\mathcal{G}$,
$\S$ \ref{S4.2.2}), the roots $z_{j,n}$
are dispatched into the two
subsectors:

\begin{itemize}
\item[$\bullet$] $ 2 \pi \frac{\sqrt{(\lo n) (\lo \lo n)}}{n} ~>~ \arg z ~>~ 0$,
\item[$\bullet$] $ 2 \pi \frac{\lo n}{n} ~>~ \arg z ~>~ 2 \pi \frac{\sqrt{(\lo n) (\lo \lo n)}}{n}$.
\end{itemize}

The relative angular size of the bump sector, as
$(2 \pi \frac{\lo n}{n})/(\frac{\pi}{2})$,
tends to zero, 
as soon as $n$ is large enough.
By transition region, we mean
a small neighbourhood
of the argument :
$$\arg z  = 2 \pi \frac{\lo n}{n} \quad {\rm or~ of} \quad 
2 \pi \frac{\sqrt{(\lo n) (\lo \lo n)}}{n}.$$
Outside these two transition regions, a
dominant
asymptotic expansion of $z_{j,n}$ exists.
In a transition region an asymptotic expansion
contains more $n$-th order terms 
of the same order of magnitude ($n=2, 3, 4$).
These two neighbourhoods are defined as follows.
Let $\epsilon \in (0, 1)$ small enough. 
Two strictly increasing  sequences
of real numbers $(u_n), (v_n)$
are introduced, which satisfy: 
$$ 
\lfloor n/6 \rfloor ~>~ v_n ~>~ \lo n, \quad
\lo n ~>~ u_n ~>~ \sqrt{(\lo n) (\lo \lo n)},
\quad {\rm for}~ n \geq n_0 = 18,$$
such that
$$\lim_{n \to \infty}
\frac{v_n}{n} ~=~
\lim_{n \to \infty}
\frac{\sqrt{(\lo n) (\lo \lo n)}}{u_n}
~=~
\lim_{n \to \infty} \frac{u_n}{\lo n}
~=~ \lim_{n \to \infty} \frac{\lo n}{v_n} ~=~ 0$$
and
\begin{equation}
\label{unvndifference}
v_n - u_n = O( (\lo n)^{1+\epsilon} )
\end{equation}
with the constant 1 involved in the big O.
The roots $z_{j,n}$ lying in the first transition region about $2 \pi (\lo n)/n$ are such that:
$$2 \pi \frac{v_n}{n} ~>~ \arg z_{j,n} ~>~ 2 \pi \frac{(2 \lo n - v_n)}{n},$$
and the
roots $z_{j,n}$ lying in the second transition region
about $\frac{2 \pi \sqrt{(\lo n) (\lo \lo n)}}{n}$
are such that:
$$2 \pi \, \frac{u_n}{n} ~>~ \arg z_{j,n} ~>~ 2 \pi \, \frac{2 \sqrt{(\lo n) (\lo \lo n)} - u_n}{n}.$$

In Proposition \ref{zjjnnExpression}, for simplicity's sake, these two transition regions are schematically denoted by 
$$\arg z \asymp 2 \pi \, \frac{(\lo n)}{n}
\quad
\mbox{resp.}\quad
\arg z \asymp 2 \pi \, \frac{\sqrt{(\lo n)(\lo \lo n)}}{n}.$$
By complementarity, the other sectors are schematically written:
$$
2 \pi \, \frac{\sqrt{(\lo n) (\lo \lo n)}}{n} ~>~ \arg z ~>~ 0$$
instead of 
$$2 \pi \, \frac{2 \sqrt{(\lo n) (\lo \lo n)} - u_n}{n} ~>~ \arg z ~>~ 0;$$
resp.
$$
2 \pi \, \frac{\lo n}{n} ~>~ \arg z ~>~
2 \pi \, \frac{\sqrt{(\lo n) (\lo \lo n)}}{n}
$$
instead of 
$$2 \pi \, \frac{2 \lo n  - v_n}{n} ~>~ \arg z ~>~ 2 \pi \, \frac{u_n}{n};$$
resp. 
$$
\frac{\pi}{2} ~>~ \arg z ~>~ 2 \pi \, \frac{\lo n}{n}   
\quad \mbox{instead of} \quad 
\frac{\pi}{2} ~>~ \arg z ~>~ 2 \pi \, \frac{v_n}{n}.$$


\frenchspacing

\begin{thebibliography}{99}

\bibitem[AarDd]{aaberdunfeld}
    \textsc{J.W. Aaber {\rm and} N. Dunfeld},
    {\it Closed Surface Bundles of Least Volume},
    Algebr. Geom. Topol. {\bf 10} (2010), 2315--2342.
    
\bibitem[AiB]{adamczewskibugeaud}
    \textsc{B. Adamczewski {\rm and}  Y. Bugeaud},
    {\it Dynamics for $\beta$-Shifts and Diophantine 
    Approximation}
    Ergod. Th \& Dynam. Sys. {\bf 27} (2007), 1--17.
    
\bibitem[AM]{adlermarcus}
    \textsc{R.L. Adler {\rm and} B. Marcus},
    {\it Topological Entropy and Equivalence of 
    Dynamical Systems}, 
    M\'em. Amer. Math. Soc. {\bf 20} (1979), 
    n$^{{\rm o}}$ 219.

\bibitem[Ah]{agoh}
    \textsc{T. Agoh},
    {\it On the Relative Class Number of Special 
    Cyclotomic Fields},
    Math. Appl. {\bf 1} (2012), 1--12.
 
\bibitem[AlLrMt]{agolleiningermargalit} 
    \textsc{I. Agol, C.J. Leininger {\rm an} D. Margalit},
    {\it Pseudo-Anosov Stretch Factors and Homology
    of Mapping Tori},
    J. London Math. Soc. {\bf 93} (2016), 664--682.
    
\bibitem[Aka]{akatsuka} 
    \textsc{H. Akatsuka},
    {\it Zeta Mahler Measures},
    J. Number Theory {\bf 129} (2013), 851--859.
    
\bibitem[Aa]{akiyama}
    \textsc{S. Akiyama},
    {\it A Family of Non-Sofic Beta Expansions},
    Ergod. Th. Dynam. Sys. {\bf 36} (2) (2016),
    343--354.
    
\bibitem[AaPe]{akiyamapethoe}
    \textsc{S. Akiyama {\rm and} A. Peth\H{o}e},
    {\it On Canonical Number Systems},
    Theor. Comp. Sci. {\bf 270} (2002), 921--933.

\bibitem[AkV]{allcockvaaler}
    \textsc{D. Allcock {\rm and}  J. Vaaler}
    {\it A Banach Space Determined by the Weil Height},
    Acta Arith. {\bf 136} (2009), 279--298.
    
\bibitem[Ama]{amara}
    \textsc{M. Amara},
    {\it Ensembles Ferm\'es de Nombres Alg\'ebriques},
    Ann. Sci. \'Ecole Norm. Sup. 
    {\bf 83} (1966), 215--270.

\bibitem[A]{amoroso}
    \textsc{F. Amoroso},
    {\it Sur des Polyn\^omes de Petites Mesures de Mahler},
    C. R. Acad. Sci. Paris S\'er. I Math. 
    {\bf 321} (1995), 11-14.

\bibitem[A2]{amoroso2}
    \textsc{F. Amoroso},
    {\it Algebraic Numbers Close to $1$: 
    Results and Methods},
    in Number Theory (Tiruchirapalli, India 1996),
    Eds. V.K. Murty and M. Waldschmidt, Amer. Math. Soc., 
    Providence,
    Contemp. Math. {\bf 210} (1998), 305--316.  

\bibitem[A3]{amoroso3}
    \textsc{F. Amoroso},
    {\it Algebraic Numbers Close to $1$
    and Variants of Mahler's Measure},
    J. Number Theory {\bf 60} (1996), 80--96.
    
\bibitem[A4]{amoroso4}
    \textsc{F. Amoroso},
    {\it Small Points on Subvarieties of Algebraic Tori: 
    Results and Methods},
    Summer School in Analytic Number Theory 
    and Diophantine Approximation, Riv. Mat. 
    Univ. Parma (7) {\bf 3$^*$} (2004), 1--31.    
    
\bibitem[A5]{amoroso5}
    \textsc{F. Amoroso},
    {\it Une Minoration pour l'Exposant du Groupe des
    Classes d'un Corps Engendr\'e par un Nombre de Salem},
    Int. J. Number Theory {\bf 3} (2007), 217--229.    
 
\bibitem[A6]{amoroso6}
    \textsc{F. Amoroso},
    {\it Small Points on a Multiplicative Group 
    and Class Number Problem},
    J. Th\'eorie Nombres Bordeaux 
    {\bf 19} (2007), 27--39.
    
\bibitem[ADd]{amorosodavid}
    \textsc{F. Amoroso {\rm and} S. David},
    {\it Le Th\'eor\`eme de Dobrowolski 
    en Dimension Sup\'erieure},
    C. R. Acad. Sci. Paris S\'er. I Math. 
    {\bf 326} (1998), 1163--1166. 

\bibitem[ADd2]{amorosodavid2}
    \textsc{F. Amoroso {\rm and} S. David},
    {\it Le Probl\`eme de Lehmer en Dimension Sup\'erieure},
    J. Reine Angew. Math. {\bf 513} (1999), 145--179.
 
\bibitem[ADd3]{amorosodavid3}
    \textsc{F. Amoroso {\rm and} S. David},
    {\it Minoration de la Hauteur Normalis\'ee
    des Hypersurfaces},
    Acta Arith. {\bf 92} (2000), 339--366. 
   
\bibitem[ADd4]{amorosodavid4}  
    \textsc{F. Amoroso {\rm and} S. David},
    {\it Densit\'e des Points \`a Coordonn\'ees 
    Multiplicativement Ind\'ependantes},
    Ramanujan J. {\bf 5} (2001), 237--246. 
    
\bibitem[ADd5]{amorosodavid5}
    \textsc{F. Amoroso {\rm and} S. David},
    {\it Minoration de la Hauteur Normalis\'ee dans un 
    Tore},
    J. Inst. Math. Jussieu {\bf 2} (2003), 335--381.       

\bibitem[ADd6]{amorosodavid6}
    \textsc{F. Amoroso {\rm and} S. David},
    {\it Distribution des Points de Petite Hauteur 
    Normalis\'ee dans les Groupes Multiplicatifs},
    Ann. Scuola Norm. Sup. Pisa Cl. Sci. 
    (5) {\bf 3} (2004), 325--348.    
    
\bibitem[ADd7]{amorosodavid7}
    \textsc{F. Amoroso {\rm and} S. David},
    {\it Points de Petite Hauteur sur une 
    Sous-Vari\'et\'e d'un Tore},   
    Compos. Math. {\bf 142} (2006), 551--562. 

\bibitem[ADdZ]{amorosodavidzannier}
    \textsc{F. Amoroso, S. David {\rm and} U. Zannier},
    {\it On Fields with the Property (B)},
    Proc. Amer. Math. Soc. {\bf 142} (2014), 1893--1910.
        
\bibitem[ADn]{amorosodelsinne}
    \textsc{F. Amoroso {\rm and} S. Delsinne},
    {\it Une Minoration Relative Explicite pour 
    la Hauteur dans une 
    Extension d'une Extension Ab\'elienne},
    Diophantine geometry, CRM Series {\bf 4}, ed. Norm., 
    Pisa (2007), 1--24.
    
    
\bibitem[AD]{amorosodvornicich}
    \textsc{F. Amoroso {\rm and} R. Dvornicich},
    {\it A Lower Bound for the Height in 
    Abelian Extensions},
    J. Number Theory {\bf 80} (2000), 260--272.

\bibitem[AM]{amorosomignotte}
    \textsc{F. Amoroso {\rm and} M. Mignotte},
    {\it On the Distribution on the Roots of Polynomials},
    Ann. Inst. Fourier (Grenoble) {\bf 46} (1996), 1275--1291.

\bibitem[ANo]{amorosonuccio}
    \textsc{F. Amoroso {\rm and} F.A.E. Nuccio},
    {\it Algebraic Numbers of Small Weil's Height 
    in CM-Fields: On a Theorem of Schinzel},
    J. Number Theory {\bf 122} (2007), 247--260.

\bibitem[AVa]{amorosoviada}
    \textsc{F. Amoroso {\rm and} E. Viada},
    {\it Small Points on Subvarieties of a Torus},
    Duke Math. J. {\bf 150} (2009), 407--442.

\bibitem[AVa2]{amorosoviada2}
    \textsc{F. Amoroso {\rm and} E. Viada},
    {\it Small Points on Rational Subvarieties of Tori},
    Commentarii Math. Helv., E.M.S. 
    {\bf 87} (2012), 355--383.
    
\bibitem[AZ]{amorosozannier}
    \textsc{F. Amoroso {\rm and} U. Zannier},
    {\it A Relative Dobrowolski Lower 
    Bound over Abelian Extensions}, 
    Ann. Scuola Norm. Sup. Pisa Cl. Sci. (4) 
    {\bf 29} (2000), 711--727.   
    
\bibitem[AZ2]{amorosozannier2}
    \textsc{F. Amoroso {\rm and} U. Zannier},
    {\it A Uniform Relative Dobrowolski's Lower 
    Bound over Abelian 
    Extensions},
    Bull. London Math. Soc. {\bf 42} (2010), 489--498.

\bibitem[AnMr]{andersonmasser}
    \textsc{M. Anderson {\rm and}  D. Masser},
    {\it Lower Bounds for Heights on Elliptic Curves},
    Math. Z. {\bf 174} (1980), 23--34.

\bibitem[AxYz]{arnouxyoccoz}
    \textsc{P. Arnoux {\rm and} J.-C. Yoccoz},
    {\it Construction de Diff\'eomorphismes Pseudo-Anosov},
    C.R. Acad. Sci. Paris S\'er. I {\bf 292} (1981), 
    75--78.
    
\bibitem[AMr]{artinmazur}
    \textsc{M. Artin {\rm and} B. Mazur},
    {\it On Periodic Points},
    Ann. Math. {\bf 81} (1965), 82--99.

\bibitem[Bk]{baker}
    \textsc{M. Baker},
    {\it Lower Bounds for the Canonical Height 
    on Elliptic Curves 
    over Abelian Extensions},
    Int. Math. Res. Not. (IMRN) {\bf 29} (2003),     
    1571--1589.

\bibitem[BkSn]{bakersilverman}
    \textsc{M. Baker {\rm and} Silverman},
    {\it A Lower Bound for the Canonical Height
    on Abelian Varieties Over Abelian Extensions},
    Math. Res. Lett. {\bf 11} (2004), 377--396.

\bibitem[Ba]{baladi}
    \textsc{V. Baladi},
    {\it Dynamical Zeta Functions},
    Real and Complex Dynamical Systems, NATO ASI Series,
    eds. B. Branner and P. Hjorth, vol. {\bf C-464},
    Kluwer Academic Publishers, Dordrecht (1995), 1--26.   

\bibitem[Ba2]{baladi2}
    \textsc{V. Baladi},
    {\it Dynamical Zeta Functions and 
    Generalized Fredholm Determinants},
    XIth International Congress of 
    Mathematical Physics (Paris, 1994),
    Int. Press, Cambridge, MA, (1995), 249--260.

\bibitem[Ba3]{baladi3}
    \textsc{V. Baladi},
    {\it Periodic Orbits and Dynamical Spectra},
    Ergod. Th. \& Dynam. Sys. {\bf 18} (1998), 255--292.

\bibitem[Ba4]{baladi4}
    \textsc{V. Baladi},
    {\it Positive Transfer Operators and 
    the Decay of Correlations},
    vol. {\bf 16}, World Scientific 
    (Advanced Series in Nonlinear Dynamics), 
    Singapore (2000).
    
\bibitem[BaK]{baladikeller}
    \textsc{V. Baladi {\rm and} G. Keller},
    {\it Zeta Functions and Transfer Operators for 
    Piecewise Monotone Transformations},
    Comm. Math. Phys. {\bf 127} (1990), 459--479.

\bibitem[BaR]{baladiruelle}
    \textsc{V. Baladi {\rm and} D. Ruelle},
    {\it An Extension of the Theorem of Milnor and Thurston
    on the Zeta Functions of Interval Maps},
    Ergod. Th. Dynam. Sys. {\bf 14} (1994), 621--632.

\bibitem[BL]{banli}
    \textsc{J.-C. Ban {\rm and} B. Li},
    {\it The Multifractal Spectra for the 
    Recurrence Rates
    of Beta-Transformations},
    J. Math. Anal. Appl. {\bf 420} (2014), 1662--1679. 

\bibitem[B-T]{baratbertheliardetthuswaldner}
    \textsc{G. Barat, V. Berth\'e, P. Liardet {\rm and}
    J. Thuswaldner},
    {\it Dynamical Directions in Numeration},
    Ann. Inst. Fourier (Grenoble) {\bf 56} (2006),
     1987--2092. 

\bibitem[BHPvV]{barthhulekpetersvandeven}
    \textsc{W.P. Barth, K. Hulek, C. Peters {\rm and} 
    A. van de Ven},
    {\it Compact Complex Surfaces}, vol. 4
    of Ergebnisse des Mathematik und ihrer Grenzgebiete,
    Springer-Verlag, Berlin, second edition (2004).
    
\bibitem[BCSn]{bartholdiceccherinisilberstein}
    \textsc{L. Bartholdi {\rm and} T.G. 
    Ceccherini-Silberstein},
    {\it Salem Numbers and Growth Series of 
    Some Hyperbolic Graphs},
    Geom. Dedicata {\bf 90} (2002), 107--114.
    
\bibitem[Bso]{bassino}
    \textsc{F. Bassino},
    {\it Beta Expansions for Cubic Pisot Numbers},
    LATIN 2002: Theoretical Informatics (Cancun),
    Lecture Notes in Comput. Sci. {\bf 2286}, 
    Springer, Berlin (2002), 141--152.

\bibitem[Bre]{bauchere}
    \textsc{H. Bauch\`ere},
    {\it Minoration de la Hauteur Canonique pour les 
    Modules de Drinfeld \`a Multiplications Complexes},
    J. Number Theory {\bf 157} (2015), 291--328.

\bibitem[Ber]{bauer}
    \textsc{M. Bauer},
    {\it An Upper Bound for the Least Dilatation},
    Trans. Amer. Math. Soc. {\bf 330} (1992), 361--370.
    
\bibitem[Bz]{bazylewicz}    
    \textsc{A. Bazylewicz},
    {\it On the Product of the Conjugates Outside the Unit 
    Circle of an Algebraic Integer},
    Acta Arith. {\bf 30} (1976/77), 43--61.

\bibitem[BBEM]{beauzamybombierienflomontgomery}
    \textsc{B. Beauzamy, E. Bombieri, P. Enflo {\rm and}
    H. Montgomery},
    {\it Products of Polynomials in Many Variables},
    J. Number Theory {\bf 36} (1990), 219--245.
    
\bibitem[BdKm]{bedfordkim}    
    \textsc{E. Bedford {\rm and} K. Kim},
    {\it Periodicities in Linear Fractional
    Recurrences: 
    Degree Growth of Birational Surface Maps},
    Michigan Math. J. {\bf 54} (2006), 647--670.
    
\bibitem[BCn]{bellchen}
    \textsc{J. Bell {\rm and} S. Chen},
    {\it Power Series with Coefficients from a Finite Set},
    preprint (2016).
 
\bibitem[BMW]{bellmilesward}
    \textsc{J. Bell, R. Miles {\rm and} T. Ward},
    {\it Towards a P\'olya-Carlson Dichotomy for 
    Algebraic Dynamics},
    Indag. Math. (N.S.) {\bf 25} (2014), 652--668. 
    
\bibitem[Bebu]{belotserkovski}
    \textsc{Y. Belotserkovski (Bilu)},
    {\it Uniform Distribution of Algebraic Numbers Near 
    the Unit Circle},
    Vests. Akad. Navuk. BSSR Ser. F z.-Mat. Navuk 
    (1988) No 1,  49--52, 124.  
    
\bibitem[BBG]{beresnevichbernikgotze} 
    \textsc{Y. Beresnevich, V. Bernik 
    {\rm and} F. G\"otze},
    {\it The Distribution of Close Conjugate 
    Algebraic Numbers},
    Compos. Math. {\bf 146} (2010), 1165--1179.    
    
\bibitem[BeM]{bergemartinet}
    \textsc{A.-M. Berg\'e {\rm and} J. Martinet},
    {\it Notions Relatives de R\'egulateurs et de 
    Hauteurs},
    Acta Arith. {\bf 54} (1989), 155-170.    
    
\bibitem[BeRo]{bertherigo}
    \textsc{V. Berth\'e {\rm and}  M. Rigo},
    {\it Combinatorics, Automata and Number Theory},
    Cambridge University Press (2010).    
    
\bibitem[Bn]{bertin}
    \textsc{M.J. Bertin},
    {\it Quelques R\'esultats Nouveaux sur les 
    Nombres de Pisot et de Salem},
    Number Theory in Progress, 
    An Intern. Conf. on Number Theory, org. 
    Stefan Banach Int. Math. Research Center 
    in honor of the 60th birthday of 
    Andrzej Schinzel, Zakopane, Poland, 1997,
    eds. K. Gyory, H. Iwaniec, J. Urnanowicz, 
    W. de Gruyter, Berlin (1999), Vol. I, 1--10.   
    
\bibitem[BPD]{bertinpathiauxdelefosse}
    \textsc{M.J. Bertin {\rm et} M. Pathiaux-Delefosse},
    {\it Conjecture de Lehmer et Petits Nombres de Salem},
    Queen's Papers in Pure and Applied Mathematics, 
    N° 81, Ed. A.J. Coleman and
    P. Ribenboim, Kingston, Ontario, Canada (1989).

\bibitem[B-S]{bertinetal}
    \textsc{M.J. Bertin, A. Decomps-Guilloux, 
    M. Grandet-Hugot,
    M. Pathiaux-Delefosse  {\rm and} J.P. Schreiber},
    {\it Pisot and Salem Numbers},
    Birkha\"user  (1992).

\bibitem[B-Ms]{bertinfeaverfuselierlalinmanes}
    \textsc{M.J. Bertin, A. Feaver, J. Fuselier, 
    M. Lal\'in {\rm and} M. Manes},
    {\it Mahler Measures of Some Singular $K3$-surfaces},
    Women in Numbers 2: 
    Research Directions in Number Theory 149--169, 
    Contemp. Math. {\bf 606}, Centre Rech. Math. Proc.,
    Amer. Math. Soc., Providence, RI (2013).

\bibitem[Bed]{bertrand}
    \textsc{D. Bertrand},
    {\it Probl\`eme de Lehmer et Petits Discriminants},
    Arithm\'etix {\bf 6} (1982), 14--15.

\bibitem[BMs]{bertrandmathis}
    \textsc{A. Bertrand-Mathis},
    {\it Nombres de Perron et Questions de Rationalit\'e}, 
    Ast\'erisque n$\mbox{}^0$ 198--200 (1992), 67--76.

\bibitem[BMs2]{bertrandmathis2}
    \textsc{A. Bertrand-Mathis},
    {\it Nombres de Pisot, Matrices Primitives
    et Beta-Conjugu\'es},
    J. Th\'eorie Nombres Bordeaux {\bf 24} (2012), 57--72.

\bibitem[BD]{besserdeninger}
    \textsc{A. Besser {\rm and} C. Deninger},
    {\it $p$-adic Mahler Measures},
    J. Reine Angew. Math., {\bf 517} (1999), 19--50.

\bibitem[Biy]{billingsley}
    \textsc{P. Billingsley},
    {\it Convergence of Probability Measures},
    J. Wiley \& Sons, New York, 2nd Ed. (1999).
    
\bibitem[Bu]{bilu}
    \textsc{Y. Bilu},
    {\it Limit Distribution of Small Points on 
    Algebraic Tori},
    Duke Math. J. {\bf 89} (1997), 465--476.

\bibitem[BuHV]{biluhanrotvoutier}
    \textsc{Y. Bilu, G. Hanrot {\rm and} P. Voutier},
    {\it Existence of Prime Divisors of Lucas and 
    Lehmer Numbers},
    J. Reine Angew. Math. {\bf 539} (2001), 75--122.

\bibitem[Bin]{birman}
    \textsc{J.S. Birman},
    {\it Braids, Links and Mapping Class Groups},
    Ann. Math. Studies {\bf 82}, Princeton Univ. Press,
    Princeton, NJ (1974).
    
\bibitem[Bis]{biswas}
    \textsc{A. Biswas},
    {\it Asymptotic Nature of Higher Mahler Measure},
    Acta Arith. {\bf 166} (2014), 15--21.
    
\bibitem[Bd]{blanchard}
     \textsc{F. Blanchard},
     {\it $\beta$-Expansions and Symbolic Dynamics},
     Theoret. Comput. Sci. {\bf 65} (1989), 131--141.

\bibitem[ByM]{blanskymontgomery}
    \textsc{P.E. Blansky {\rm and} H.L. Montgomery},
    {\it Algebraic Integers Near the Unit Circle},
    Acta Arith. {\bf 18} (1971), 355--369.

\bibitem[ByLu]{boissylanneau}
    \textsc{C. Boissy {\rm and} E. Lanneau}
    {\it Pseudo-Anosov Homeomorphisms on Translation
    Surfaces in Hyperelliptic Components Have Large 
    Entropy},
    Geom. Funct. Anal. {\bf 22} (2012), 74--106.    

\bibitem[Bri]{bombieri}
    \textsc{E. Bombieri},
    {\it Problems and Results on the 
    Distribution of Algebraic Points on 
    Algebraic Varieties},
    J. Th\'eorie Nombres Bordeaux {\bf 21} (2009), 41--57.

\bibitem[BriG]{bombierigubler}
    \textsc{E. Bombieri {\rm and} W. Gubler},
    {\it Heights in Diophantine Geometry},
    Cambridge Univ. Press, Cambridge (2006), xvi+652 pp.
    
\bibitem[BriZ]{bombierizannier}
    \textsc{E. Bombieri {\rm and} U. Zannier},
    {\it A Note on Heights in Certain Infinite Extensions
    of $\qb$},
    Atti. Acad. Naz. Lincei Cl. Sci. Mat. Fis. Natur. 
    Rend. Lincei (9) Mat. Appl. {\bf 12} (2001), 5--14.

\bibitem[Bl]{borel}
    \textsc{E. Borel},
    {\it Le\c{c}ons sur les S\'eries Divergentes},
    Gauthier-Villars, 2e \'edition, Paris, (1928).    

\bibitem[Brn]{bornhorn}
    \textsc{H. Bornhorn}
    {\it Mahler Measures, $K$-Theory and Values of 
    $L$-Functions},
    arXiv:1503.06069v1 (20 Mar 2015); 
    Preprint SFB 478, M\"unster, Heft 35, (1999).    
    
\bibitem[BBSW]{borweinborweinstraubwan}
     \textsc{J.M. Borwein, P. Borwein, 
     A. Straub {\rm and} J. Wan},
     {\it Log-Sine Evaluations of Mahler Measures, II},
     Integers {\bf 12} (2012), 1179--1212.    
    
\bibitem[BDM]{borweindobrowolskimossinghoff}
     \textsc{P. Borwein, E. Dobrowolski {\rm and} M.J. 
     Mossinghoff},
     {\it Lehmer's Problem for Polynomials With 
     Odd Coefficients},
     Ann. of Math. {\bf 166} (2007), 347--366.

\bibitem[BEL]{borweinerdelyilittmann}
     \textsc{P. Borwein, T. Erd\'elyi {\rm and} F. 
     Littmann},
     {\it Polynomials With Coefficients 
     from a Finite Set},
     Trans. Amer. Math. Soc. {\bf 360} (10) (2008), 
     5145--5154.

\bibitem[BHMf]{borweinharemossinghoff}
     \textsc{P. Borwein,  K. Hare {\rm and}
     M. Mossinghoff},
     {\it The Mahler Measure of Polynomials With Odd 
     Coefficients}, 
     Bull. London Math. Soc. {\bf 36} (2004), 332--338.     

\bibitem[BS]{borweinstraub}
     \textsc{J.M. Borwein {\rm and} A. Straub},
     {\it Mahler Measures, Short Walks and Log-Sine 
     Integrals},
      Theoret. Comput. Sci. {\bf 479} (2013), 4--21.

\bibitem[BGS]{bostgilletsoule} 
    \textsc{J.B. Bost, H. Gillet {\rm and}  C. Soul\'e},
    {\it Heights of Projective Varieties 
    and Positive Green Forms},
    J. Amer. Math. Soc. {\bf 7} (1994), 903--1022.    

\bibitem[Bki]{bourbaki}
     \textsc{N. Bourbaki},
     {\it Groupes et Alg\`ebres de Lie},
     Chaps 4--6, Hermann, Paris (1968).
     
\bibitem[Bo]{boyd}
     \textsc{D.W. Boyd},
     {\it Pisot Sequences Which Satisfy No Linear 
     Recurrence},
     Acta Arith. {\bf 32} (1977), 89--98.
	
\bibitem[Bo2]{boyd2}
     \textsc{D.W. Boyd},
     {\it Small Salem Numbers},
     Duke Math. J. {\bf 44} (1977), 315--328.
     
\bibitem[Bo3]{boyd3}
     \textsc{D.W. Boyd},
     {\it Pisot Numbers and the Width of Meromorphic 
     Functions},
     privately circulated manuscript (January 1977).

\bibitem[Bo4]{boyd4}
     \textsc{D.W. Boyd},
     {\it Pisot and Salem Numbers in Intervals of the 
     Real Line},
     Math. Comp. {\bf 32} (1978), 1244--1260.

\bibitem[Bo5]{boyd5}
     \textsc{D.W. Boyd},
     {\it Variations on a Theme of Kronecker},
     Canad. Math. Bull. {\bf 21} (1978), 129--133.

\bibitem[Bo6]{boyd6}
     \textsc{D.W. Boyd},
     {\it Reciprocal Polynomials Having Small Mahler 
     Measure},
     Math. Comp. {\bf 35} (1980), 1361--1377;
     {\it II},
     Math. Comp. {\bf 53} (1989),  355--357, S1--S5.

\bibitem[Bo7]{boyd7}
     \textsc{D.W. Boyd},
     {\it Kronecker's Theorem and 
     Lehmer's Problem for Polynomials
     in Several Variables},
     J. Number Theory {\bf 13} (1981), 116--121.

\bibitem[Bo8]{boyd8}
     \textsc{D.W. Boyd},
     {\it Speculations Concerning the Range of Mahler's 
     Measure},
     Canad. Math. Bull. {\bf 24} (1981), 453--469.

\bibitem[Bo9]{boyd9}
     \textsc{D.W. Boyd},
     {\it The Maximal Modulus of an Algebraic Integer},
     Math. Comp. {\bf 45} (1985), 243--249, S17--S20.

\bibitem[Bo10]{boyd10}
     \textsc{D.W. Boyd},
     {\it Inverse Problems for Mahler's Measure},
     in {\it Diophantine Analysis}, 
     Eds. J. Loxton
      and A. van der Poorten, London Math. Soc. 
      Lecture Notes {\bf 109}, Cambridge 
      University Press (1986), 147--158.
     
\bibitem[Bo11]{boyd11}
     \textsc{D.W. Boyd},
     {\it Perron Units Which Are Not 
     Mahler Measures},
     Ergod. Th. Dynam. Sys. {\bf 6} (1986), 
     485--488.

\bibitem[Bo12]{boyd12}
     \textsc{D.W. Boyd},
     {\it Reciprocal Algebraic Integers Whose 
     Mahler Measures are Non-Reciprocal},
     Canad. Math. Bull. {\bf 30} (1987), 
     3--8.
     
\bibitem[Bo13]{boyd13}
     \textsc{D.W. Boyd},
     {\it Salem Numbers of Degree Four Have Periodic 
     Expansions},
     Number Theory, Walter de Gruyter (1989), 57--64.
     
\bibitem[Bo14]{boyd14}
     \textsc{D.W. Boyd},
     {\it On the Beta Expansion for Salem Numbers of 
     Degree 6},
     Math. Comp. {\bf 65} (1996), 861--875, S29--S31.
     
\bibitem[Bo15]{boyd15}
     \textsc{D.W. Boyd},
     {\it On Beta Expansion for Pisot Numbers},
     Math. Comp. {\bf 65} (1996), 841--860.    

\bibitem[Bo16]{boyd16}
     \textsc{D.W. Boyd},
     {\it Mahler's Measure and Special 
     Values of L-Functions},
     Exp. Math. {\bf 7} (1998), 37--82.
     
\bibitem[Bo17]{boyd17}
     \textsc{D.W. Boyd},
     {\it Mahler's Measure and Invariants 
     of Hyperbolic Manifolds}, 
     Number Theory for the Millenium, I,
     A.K. Peters, Natick (2002), 127--143.
     
\bibitem[Bo18]{boyd18}
     \textsc{D.W. Boyd},
     {\it Mahler's Measure, Hyperbolic Geometry and the
     Dilogarithm},
     2001 CMS Jeffery-Williams-Prize Lecture.    
     
\bibitem[BM]{boydmossinghoff}
    \textsc{D.W. Boyd {\rm and} M.J. Mossinghoff},
    {\it Small Limit Points of Mahler's Measure},
    Exp. Math. {\bf 14} (2005), 403--414.

\bibitem[BPa]{boydparry}
    \textsc{D.W. Boyd {\rm and} W. Parry},
    {\it Limit Points of Salem Numbers},
    Number Theory (Banff, AB, 1988),
    de Gruyter, Berlin (1990), 27--35.

\bibitem[BoRs]{boydrodriguezvillegas}
    \textsc{D.W. Boyd {\rm and} 
    F. Rodriguez-Villegas},
    {\it Mahler's Measure and the Dilogarithm (I)}, 
    Canad. J. Math. {\bf 54} (2002), 468--492;
    {\it (II)}, preprint.
   
\bibitem[Ble]{boyle}
    \textsc{M. Boyle},
    {\it Pisot, Salem and Perron Numbers in 
    Ergodic Theory and Topological Dynamics}, 
    notes, Nov. 1982.
    
\bibitem[Ble2]{boyle2}
    \textsc{M. Boyle},
    {\it Open Problems in Symbolic Dynamics},
    Contemp. Math. {\bf 469} (2008), 69--118. 

\bibitem[BleH]{boylehandelman}
    \textsc{M. Boyle {\rm and} D. Handelman},
    {\it The Spectra of Nonegative Matrices via 
    Symbolic Dynamics},
    Ann. of Math. {\bf 133} (1991), 249--316.

\bibitem[BleH2]{boylehandelman2}
    \textsc{M. Boyle {\rm and} D. Handelman},
    {\it Algebraic Shift Equivalence and Primitive 
    Matrices},
    Trans. Amer. Math. Soc. {\bf 336} (1993), 121--149.

\bibitem[Bst]{brandhorst}
    \textsc{S. Brandhorst},
    {\it Dynamics of Supersingular $K3$ Surfaces
    and Automorphisms of Salem Degree $22$},
    subm. Nagoya Math. J.  (2015), preprint.
    
\bibitem[Bst2]{brandhorst2}
    \textsc{S. Brandhorst},
    {\it Automorphisms of Salem Degree $22$ on
    Supersingular $K3$ Surfaces of Higher Artin 
    Invariant - A Short Note},
    preprint (2016). 
    
\bibitem[BstG]{brandhorstgonzalezalonso}
    \textsc{S. Brandhorst {\rm and} V. Gonz\'alez-Alonso},
    {\it Automorphisms of Minimal Entropy on
    Supersingular $K3$ Surfaces}, 
    preprint (2017).
    
\bibitem[Brr]{brauer}
    \textsc{A. Brauer},
    {\it On Algebraic Equations with All But 
    One Root in the Interior of the Unit Circle},
    Math. Nachr. Bd {\bf 4} (1950/51), 250--257. 

\bibitem[BdDn]{breuillardderoin}
    \textsc{E. Breuillard {\rm and} Deroin},
    {\it Salem Numbers and the Spectrum of 
    Hyperbolic Surfaces},
    preprint (2016).
    
\bibitem[Br]{breusch}
    \textsc{R. Breusch},
    {\it On the Distribution of the Roots of 
    a Polynomial with Integral Coefficients},
    Proc. Amer. Math. Soc. {\bf 2} (1951), 
    939--941.

\bibitem[Bte]{brunotte}
    \textsc{H. Brunotte},
    {\it Algebraic Properties of Weak Perron Numbers},
    Tatra Mt. Math. Publ. {\bf 56} (2013), 1--7.
    
 \bibitem[Bg]{bugeaud}
    \textsc{Y. Bugeaud},
    {\it On the $\beta$-Expansion of an Algebraic Number
    in an Algebraic Base $\beta$},
    Integers {\bf 9}:A 20 (2009), 215--226.    
 
\bibitem[BgMeNn]{bugeaudmignottenormandin}
    \textsc{Y. Bugeaud, M. Mignotte {\rm and} F. Normandin},
    {\it Nombres Alg\'ebriques de Petite Mesure
    et Formes Lin\'eaires en un Logarithme},
    C.R. Acad. Sci. Paris S\'er. I Math. 
    {\bf 321} (1995), 517--522.
    
\bibitem[BgLo]{bugeaudliao}
    \textsc{Y. Bugeaud {\rm and} L. Liao},
    {\it Uniform Diophantine Approximation 
    Related to b-ary and $\beta$-Expansions}, 
    Ergod. Th. Dynam. Sys. {\bf 36} (2016), 1 --22.
    
\bibitem[BgMte]{bugeaudmignotte}
    \textsc{Y. Bugeaud {\rm and} M. Mignotte},
    {\it On the Distance Between Roots of 
    Integer Polynomials},
    Proc. Edinb. Math. Soc. {\bf 47} (2004), 553--556.        
    
\bibitem[Bue]{burde}
    \textsc{G. Burde},
    {\it Alexanderpolynome Neuwirthscher Knoten},
    Topology {\bf 5} (1966), 321--330.    
    
\bibitem[BGPRLS]{burgosgilphilipponriveraleteliersombra}    
    \textsc{J.I. Burgos Gil, P. Philippon, 
    J. Rivera-Letelier {\rm and} M. Sombra},
    {\it The Distribution of Galois Orbits
    of Points of Small Height in Toric Varieties},
    preprint (2016).
    
\bibitem[CiHg]{calegarihuang}    
    \textsc{F. Calegari {\rm and} Z. Huang},
    {\it Counting Perron Numbers by Absolute Values},
    preprint (2017).
    
\bibitem[CNSn]{callahannewmansheingorn}
    \textsc{T. Callahan, M. Newman {\rm and} M. Sheingorn},
    {\it Fields with Large Kronecker Constants},    
    J. Number Theory {\bf 9} (1977), 182--186.
    
\bibitem[Cn]{cannon}
    \textsc{J.W. Cannon},
    {\it The Growth of the Closed Surface Groups 
    and the Compact Hyperbolic Coxeter Groups},
    (1993).
    
\bibitem[CnW]{cannonwagreich}
    \textsc{J.W. Cannon {\rm and} Ph. Wagreich},
    {\it Growth Functions of Surface Groups},
    Math. Ann. {\bf 292} (1992), 239--257.    

\bibitem[Ctt]{cantat}    
    \textsc{S. Cantat},
    {\it Dynamique des Automorphismes des 
    Surfaces Projectives Complexes},
    C.R.A.S. Paris S\'erie I Math.  {\bf 328} 
    (1999), 901--906.
    
\bibitem[Ctt2]{cantat2}    
    \textsc{S. Cantat},
    {\it Dynamics of Automorphisms of Compact 
    Complex Surfaces},
    In Frontiers in Complex Dynamics: In Celebration of
    John Milnor's 80th Birthday, Volume 51 of Princeton
    Math. Series, Princeton University Press
    (2014), 463--514.
    
\bibitem[CttD]{cantatdupont}    
    \textsc{S. Cantat {\rm and} C. Dupont},
    {\it Automorphisms of Surfaces: Kummer Rigidity
    and Measure of Maximal Entropy},
    preprint (2015).  
    
\bibitem[CS]{cantorstrauss}
    \textsc{D.C. Cantor {\rm and} E.G. Strauss},
    {\it On a conjecture of D.H Lehmer},   
    Acta Arith. {\bf 42} (1982/83), 97--100; 
    Correction: {\it ibid}, {\bf 42} (3) (1983), 327.

\bibitem[Co]{cao}
    \textsc{C.-Y. Cao},
    {\it A result on the approximation 
    properties of the orbit of 1 under the $\beta$-
    transformation},
    J. Math. Anal. Appl. {\bf 420} (2014), 242--256.

\bibitem[C]{carlson}
    \textsc{F. Carlson},
    {\it \"Uber Potenzreihen mit endlich vielen 
    verschiedenen Koeffizienten},
    Math. Ann. {\bf 79} (1918), 237--245.
    
\bibitem[C2]{carlson2}
    \textsc{F. Carlson},
    {\it \"Uber Potenzreihen mit ganzzahligen 
    Koeffizienten},
    Math. Z. {\bf 9} (1921), 1--13.
    
\bibitem[Cza]{carrizosa}    
    \textsc{M. Carrizosa},
    {\it Probl\`eme de Lehmer et 
    Vari\'et\'es Ab\'eliennes CM}, 
    C.R. Acad. Sci. Paris, S\'er. I 
    {\bf 346} (2008), 1219--1224.   
    
\bibitem[Cza2]{carrizosa2}    
    \textsc{M. Carrizosa},
    {\it Petits Points et Multiplication Complexe},
    Int. Res. Math. Not. (IRMN) {\bf 16} (2009), 
    3016--3097.
  
\bibitem[Cza3]{carrizosa3}    
    \textsc{M. Carrizosa},
    {\it Survey on Lehmer Problems},
    Sao Paulo J. Math. Sci. {\bf 3} (2009), 317--327.
    
\bibitem[Cs]{cassels}
    \textsc{J.W.S. Cassels},
    {\it On a problem of Schinzel and Zassenhaus},    
    J. Math. Sciences {\bf 1} (1966), 1--8.    
    
\bibitem[CLr]{chambertloir}   
    \textsc{A. Chambert-Loir},
    {\it Th\'eor\`emes d'\'equidistribution 
    pour les Syst\`emes Dynamiques 
    d'Origine Arithm\'etique},
    Quelques Aspects des Syst\`emes Dynamiques 
    Polyn\^omiaux, Panor. Synth\`eses {\bf 30},
    Soc. Math. France, Paris (2010), 203--294.
    
\bibitem[ChrKn]{champanerkarkofman}    
    \textsc{A. Champanerkar {\rm and} I. Kofman},
    {\it On the Mahler Measure of Jones Polynomials 
    Under Twisting},
    Algebr. Geom. Topol. {\bf 5} (2005), 1--22.
            
\bibitem[CgSk]{changshrock}            
    \textsc{S.-C. Chang {\rm and} R. Shrock},
    {\it Zeros of Jones Polynomials for Families of
    Knots and Links},
    Physica A {\bf 301} (2001), 196--218.
            
\bibitem[Crn]{chern}
    \textsc{Sh. Chern},
    {\it Distribution of reducible polynomials with 
    a given coefficient set}, 
    arXiv:1602.07431 (8 Mar 2016).  
    
\bibitem[CV]{chernvaaler}
    \textsc{S.-J. Chern {\rm and} J.D. Vaaler},
    {\it The distribution of values of Mahler's measure},
    J. Reine Angew. Math. {\bf 540} (2001), 1--47. 

\bibitem[Cg]{chinburg}
    \textsc{T. Chinburg},
    {\it On the Arithmetic of Two Constructions of
    Salem Numbers},
    J. Reine Angew. Math. {\bf 348} (1984), 166--179.
 
\bibitem[Cg2]{chinburg2}
    \textsc{T. Chinburg},
    {\it Salem Numbers and $L$-functions},
    J. Number Theory {\bf 18} (1984), 213--214.

\bibitem[CoHm]{choham}
    \textsc{J.-H. Cho {\rm and} J.-Y. Ham},
    {\it The Minimal Dilatation of a Genus-2 Surface},
    Exp. Math. {\bf 17} (2008), 257--267.
    
\bibitem[Chi]{choi}
    \textsc{K.-K. Choi},    
    {\it On the Distribution of Points in 
    Projective Space of Bounded Height},
    Trans. Amer. Math. Soc. {\bf 352} (2000), 1071--1111.   
 
\bibitem[CMcK]{christopoulosmackee} 
    \textsc{C. Christopoulos {\rm and} J. McKee},
    {\it Galois Theory of Salem Polynomials},
    Math. Proc. Cambridge Phil. Soc. {\bf 148} (2010), 
    47--54.
    
\bibitem[Cdn]{condon}
    \textsc{J.D. Condon},
    {\it Mahler Measure Evaluations in Terms of 
    Polylogarithms}, PhD Thesis, Univ. of Texas, 
    Austin (2004).  
    
\bibitem[Cdn2]{condon2}
    \textsc{J.D. Condon},
    {\it Asymptotic Expansion of the Difference of 
    Two Mahler Measures},  
    J. Number Theory {\bf 132} (2012), 1962--1983. 
    
\bibitem[CyMS]{cooleymackeesmyth}
    \textsc{J. Cooley, McKee {\rm and} C. Smyth},
    {\it Non-Bipartite Graphs of Small Mahler Measure},
    J. Combinatorics Number Theory
    {\bf 5} (2014), 53--64.
    
\bibitem[CCGLS]{coopercullergilletlongshalen}
    \textsc{D. Cooper, M. Culler, H. Gillet, 
    D.D. Long {\rm and} P.B. Shalen},
    {\it Plane Curves Associated to Character Varieties
    of 3-Manifolds},
    Invent. Math. {\bf 118} (1994), 47--84.
        
\bibitem[Cn]{copson}
    \textsc{E.T. Copson},
    {\it Asymptotic Expansions},
    Cambridge Tracts in Math. {\bf 55} (1965).

\bibitem[CaFn]{costafriedman}
    \textsc{A. Costa {\rm and} E. Friedman},
    {\it Ratios of Regulators in Totally
    Real Extensions of Number Fields},
    J. Number Theory  {\bf 37} (1991), 288--297.

\bibitem[DGS]{dandreagalligosombra}
    \textsc{C. D'Andrea, A. Galligo {\rm and} M. Sombra},
    {\it Quantitative Equidistribution for the
    Solutions of Systems of Sparse Polynomial Equations},
    Amer. J. Math. {\bf 136} (2014), 1543--1579.

\bibitem[DNS]{dandreanarvezclausssombra}
    \textsc{C. D'Andrea, M. Narv\'aez-Clauss 
    {\rm and} M. Sombra},
    {\it Quantitative Equidistribution of Galois Orbits 
    of Small Points in the N-Dimensional Torus},
    arXiv:1606.04299v1 (2016).

\bibitem[Did]{david}
    \textsc{S. David},
    {\it Minorations de Hauteurs sur 
    les Vari\'et\'es Ab\'eliennes},
    Bull. Soc. Math. France {\bf 121} (1993), 509--544.

\bibitem[Did2]{david2}
    \textsc{S. David},
    {\it Points de Petite Hauteur sur 
    les Courbes Elliptiques},
    J. Number Theory {\bf 64} (1997), 104--129.
        
\bibitem[DH]{davidhindry}
    \textsc{S. David {\rm and} M. Hindry},
    {\it Minoration de la hauteur de N\'eron-Tate sur 
    les vari\'et\'es de type C.M.},
    J. Reine Angew. Math. {\bf 529} (2000), 1--74.    
 
\bibitem[DPao]{davidpacheco}
    \textsc{S. David {\rm and} A. Pacheco},
    {\it Le Probl\`eme de Lehmer Ab\'elien 
    pour un Module de Drinfeld},    
    Int. J. Number Theory {\bf 4} (2008), 1043--1067. 
 
\bibitem[DPn]{davidphilippon}
    \textsc{S. David {\rm and} P. Philippon},
    {\it Minorations des Hauteurs Normalis\'ees
    des Sous-Vari\'et\'es de Vari\'et\'es Ab\'eliennes},
    In Number Theory (Tiruchirapalli, 1996),
    Contemp. Math., Amer. Math. Soc., Providence, RI, 
    (1998), {\bf 210}, 333--364, (1996).
 
\bibitem[DPn2]{davidphilippon2}
    \textsc{S. David {\rm and} P. Philippon},
    {\it Minorations des Hauteurs Normalis\'ees
    des Sous-Vari\'et\'es des Tores},
    Ann. Scuola Norm. Sup. Pisa Cl. Sci. (4) 
    {\bf 28} (1999), 489--543; Errata: {\it ibid}
    (4) {\bf 29} (2000), 729--731.    
 
\bibitem[Dee]{debaenne} 
    \textsc{K. Debaenne},
    {\it The First Factor of the Class Number of 
    the $p$-th Cyclotomic Field},
    Arch. Math. {\bf 102} (2014), 237--244.
    
\bibitem[Dy]{dehornoy}    
    \textsc{P. Dehornoy},
    {\it Les Noeuds de Lorenz},
    Enseign. Math. {\bf 57} (2011), 211--280.
    
\bibitem[Dy2]{dehornoy2}    
    \textsc{P. Dehornoy},
    {\it On the Zeroes of the Alexander Polynomial 
    of a Lorenz Knot}, 
    Ann. Inst. Fourier (Grenoble) {\bf 65} (2015),
     509--548.  
    
\bibitem[Dy3]{dehornoy3}    
    \textsc{P. Dehornoy},
    {\it Small Dilatation Homeomorphisms as 
    Monodromies of Lorenz Knots},
    preprint (2015).    
    
\bibitem[Dne]{deligne}
    \textsc{P. Deligne},
    {\it La Conjecture de Weil I},
    Publ. Math. IHES {\bf 43} (1974), 273--307;
    {\it ibid} II, {\bf 52} (1980), 137--252.    
    
\bibitem[Dle]{delsinne}    
    \textsc{E. Delsinne},
    {\it Le Probl\`eme de Lehmer Relatif en 
    Dimension Sup\'erieure},
    Ann. Sci. Ec. Norm. Sup. {\bf 6} (2009), 981--1028.
    
\bibitem[Dgr]{deninger}
    \textsc{C. Deninger},
    {\it Deligne Periods of Mixed Motives, 
    $K$-Theory and the Entropy of 
    Certain $\zb^n$-Actions},    
    J. Amer. Math. Soc. {\bf 10} (1997), 259--281.
  
\bibitem[Dgr2]{deninger2}
    \textsc{C. Deninger},
    {\it On Extensions of Mixed Motives},
    Journ\'ees Arithm\'etiques (Barcelona, 1995),
    Collect. Math. {\bf 48} (1997), 97--113. 
  
 \bibitem[Dis]{denis}        
    \textsc{L. Denis},
    {\it Hauteurs Canoniques et Modules de Drinfeld},
    Math. Ann. {\bf 294} (1992), 213--223. 
        
\bibitem[Dis2]{denis2}        
    \textsc{L. Denis},
    {\it Probl\`eme de Lehmer en Caract\'eristique Finie},
    Compositio Math. {\bf 98} (1995), 167--175.

\bibitem[DGS]{denkergrillenbergersigmund}
    \textsc{M. Denker, C. Grillenberger {\rm and} 
    K. Sigmund},
    {\it Ergodic Theory on Compact Spaces},
    Lecture Notes in Math. {\bf 527}, Springer, Berlin, 
    (1976).

\bibitem[Da]{desilva}
    \textsc{D. De Silva},
    {\it Lind-Lehmer Constant for Groups of the Form
    $\zb_{p}^{n}$},
    PhD Thesis, Kansas State University (2013).

\bibitem[DaPr]{desilvapinner}
    \textsc{D. De Silva {\rm and} C. Pinner},
    {\it The Lind-Lehmer Constant for $\zb_{p}^{n}$},
    Proc. Amer. Math. Soc.
    {\bf 142} (2014), 1935--1941.
    
\bibitem[Dis]{dienes}
    \textsc{P. Dienes},
    {\it The Taylor Series},
     Clarendon Press, Oxford (1931).        
    
\bibitem[Di]{dingle}
    \textsc{R.B. Dingle},
    {\it Asymptotic Expansions: their Derivation and 
    Interpretation},
    Academic Press (1973).

\bibitem[DhNn]{dinhnguyen}
    \textsc{T.-C. Dinh {\rm and} V.-A. Nguyen},
    {\it Comparion of Dynamical Degrees for 
    Semi-Conjugate Meromorphic Maps},
    Comment. Math. Helv. {\bf 86} (2011), 817--840.

\bibitem[DhNnTg]{dinhnguyentruong}
    \textsc{T.-C. Dinh, V.-A. Nguyen 
    {\rm and} T.T. Truong},
    {\it On the Dynamical Degrees of Meromorphic Maps
    Preserving a Fibration}, 
    subm. Commun. Contemp. Math. (2014), preprint.

\bibitem[DDs]{dixondubickas}
    \textsc{J.D. Dixon {\rm and} A. Dubickas},
    {\it The Values of Mahler Measures},
    Mathematika {\bf 51} (2004), 131--148.
    
\bibitem[Do]{dobrowolski}    
    \textsc{E. Dobrowolski},
    {\it On the Maximal Modulus of Conjugates of an 
    Algebraic Integer},
    Bull. Acad. Polon. Sci. S\'er. Sci. 
    Math. Astronom. Phys. 
    {\bf 26} (1978), 291--292.    
    
\bibitem[Do2]{dobrowolski2}
    \textsc{E. Dobrowolski},
    {\it On a Question of Lehmer and the 
    Number of Irreducible Factors of a Polynomial},
    Acta Arith. {\bf 34} (1979), 391--401.

\bibitem[Do3]{dobrowolski3}
    \textsc{E. Dobrowolski},
    {\it On a Question of Lehmer},
    M\'em. Soc. Math. France No {\bf 2} (1980--81), 35--39.
    
\bibitem[Do4]{dobrowolski4}
    \textsc{E. Dobrowolski},
    {\it Mahler's Measure of a Polynomial in Function
    of the Number of its Coefficients},
    Canad. Math. Bull. {\bf 34} (1991), 186--195.
  
\bibitem[Do5]{dobrowolski5}
    \textsc{E. Dobrowolski},
    {\it Mahler's Measure of a Polynomial in Function
    of the Number of its Monomials},
    Acta Arith. {\bf 123} (2006), 201--231. 
    
\bibitem[DoLS]{dobrowolskilawtonschinzel}
    \textsc{E. Dobrowolski, W. Lawton 
    {\rm and} A. Schinzel},
    {\it On a Problem of Schinzel},
    Studies in Pure Mathematics (Boston, Mass.),
    Birkha\"user (1983), 135--144.        

\bibitem[Dhe]{doche}
    \textsc{C. Doche},
    {\it On the Spectrum of the Zhang-Zagier Height},
    Math. Comp. {\bf 70} (2001), 419--430.

\bibitem[Dhe2]{doche2}
    \textsc{C. Doche},
    {\it Zhang-Zagier Heights of Perturbed Polynomials},
    J. Th\'eorie Nombres Bordeaux {\bf 13} (2001), 
    103--110.

\bibitem[Dgv]{dolgachev}
    \textsc{I. Dolgachev},
    {\it Salem Numbers and Enriques Surfaces},
    preprint (2017).

\bibitem[Dwl]{dowdall}
    \textsc{S. Dowdall},
    {\it Dilatation Versus Self-Intersection Number
    for Point-Pushing Pseudo-Anosov Homeomorphisms},
    J. Topology {\bf 4} (2011), 942--984.
    
\bibitem[Ddn]{dresden}
    \textsc{G.P. Dresden},
    {\it Orbits of Algebraic Numbers with Low Heights},
    Math. Comp. {\bf 67} (1998), 815--820.
    
\bibitem[DrDs]{drungilasdubickas}
    \textsc{P. Drungilas {\rm and} A. Dubickas},
    {\it Every Real Algebraic Number is
    the Difference of Two Mahler Measures},
    Canad. Math. Bull. {\bf 50} (2007), 191--195.

\bibitem[Ds]{dubickas}
    \textsc{A. Dubickas},
    {\it On a Conjecture of A. Schinzel and H. 
    Zassenhaus},
     Acta Arith. {\bf 63} (1993), 15--20.

\bibitem[Ds2]{dubickas2}
    \textsc{A. Dubickas},
    {\it On Algebraic Numbers of Small Measure},
    Lithuanian Math. J. {\bf 35} (1995), 333--342 (1996).
    
\bibitem[Ds3]{dubickas3}
    \textsc{A. Dubickas},
    {\it The Maximal Conjugate of a Non-Reciprocal 
    Algebraic Integer},
    Lithuanian Math. J. {\bf 37} (2) (1997), 
    129--133 (1998).

\bibitem[Ds4]{dubickas4}
    \textsc{A. Dubickas},
    {\it On Algebraic Numbers Close to 1},
    Bull. Austral. Math. Soc. {\bf 58} (1998),423--434.
     
\bibitem[Ds5]{dubickas5}
    \textsc{A. Dubickas},
    {\it Polynomials Irreducible by Eisenstein's 
    Criterion},
    AAECC {\bf 14} (2003), 127--132.  
     
\bibitem[Ds6]{dubickas6}
    \textsc{A. Dubickas},
    {\it Nonreciprocal Algebraic Numbers of Small Measure},
    Comment. Math. Univ. Carolin. {\bf 45} (2004), 693--697.    
 
\bibitem[Ds7]{dubickas7}
    \textsc{A. Dubickas},
    {\it On Numbers which are Mahler Measures},
    Monatsh. Math. {\bf 141} (2004), 119--126.
    
\bibitem[Ds8]{dubickas8}
    \textsc{A. Dubickas},
    {\it Mahler Measures Generate the Largest 
    Possible Groups},
    Math. Res. Lett. {\bf 11} (2004), 279--283.    

\bibitem[Ds9]{dubickas9}
    \textsc{A. Dubickas},
    {\it Algebraic, Arithmetic and Geometric
    Properties of Mahler Measures},
    Proc. Inst. Math. (Nat. Acad. Sc. Belarus)
    {\bf 13} (1) (2005), 70--74.

\bibitem[Ds10]{dubickas10}
    \textsc{A. Dubickas},
    {\it Sumsets of Pisot and Salem Numbers},
    Expo. Math. {\bf 26} (2008), 85--91.
    
\bibitem[Ds11]{dubickas11}
    \textsc{A. Dubickas},
    {\it Divisibility of the Resultant of a 
    Polynomial and a Cyclotomic Polynomial}, 
    Mosc. J. Comb. Number Theory {\bf 1} (2011), 
    11--16.

\bibitem[Ds12]{dubickas12}
    \textsc{A. Dubickas},
    {\it On $\beta$-expansions of Unity for Rational 
    and Transcendental Numbers $\beta$},
    Math. Slovaca {\bf 61} (2011), 705--716.

\bibitem[Ds13]{dubickas13}
    \textsc{A. Dubickas},
    {\it On the Number of Reducible Polynomials 
    of Bounded Na\"ive Height},
    Manuscripta Math. {\bf 144} (2014), 439--456.

\bibitem[DsKn]{dubickaskonyagin}
    \textsc{A. Dubickas {\rm and} S.V. Konyagin},
    {\it On the Number of Polynomials of Bounded Measure},
    Acta Arith. {\bf 86} (1998), 325--342.

\bibitem[DsSha]{dubickassha}
    \textsc{A. Dubickas {\rm and} M. Sha},
    {\it Counting and Testing Dominant Polynomials},
    preprint (2016).

\bibitem[DsSy]{dubickassmyth}
    \textsc{A. Dubickas {\rm and} C. Smyth},
    {\it The Lehmer Constant of an Annulus},
    J. Th\'eorie Nombres Bordeaux 
    {\bf 13} (2001), 413--420.    
 
\bibitem[DsSy2]{dubickassmyth2}
    \textsc{A. Dubickas {\rm and} C. Smyth},
    {\it On the Metric Mahler Measure},
    J. Number Theory {\bf 86} (2001), 368--387. 
    
\bibitem[DsSy3]{dubickassmyth3}
    \textsc{A. Dubickas {\rm and} C. Smyth},
    {\it On Metric Heights},
    Period. Math. Hungar. {\bf 46} (2003), 135--155.
    
\bibitem[Dwk]{dwork}
    \textsc{B. Dwork},
    {\it On the Rationality of the Zeta Function
    of an Algebraic Variety},
    Amer. J. Math. {\bf 82} (1960), 631--648.    
    
\bibitem[En]{eckmann}
    \textsc{J.-P. Eckmann},
    {\it Resonances in Dynamical Systems},
    IXth Int. Congress on Mathematical Physics
    (Swansea 1988), Hilger, Bristol (1989), 192--207.

\bibitem[Er]{einsiedler}
    \textsc{M. Einsiedler},
    {\it A Generalization of Mahler Measure
    and its Application in Algebraic Dynamical Systems},
    Acta Arith. {\bf 88} (1999), 15--29.

\bibitem[ErEW]{einsiedlereverestward}
    \textsc{M. Einsiedler, G. Everest {\rm and} 
    T. Ward},
    {\it Primes in Sequences Associated to 
    Polynomials (after Lehmer)},    
    London Math Soc. J. Comput. Math. {\bf 3} 
    (2000), 125--139.

\bibitem[EoRSe]{elotmanirhinsacepee} 
    \textsc{El Otmani, G. Rhin {\rm and} J.-M. 
    Sac-\'Ep\'ee},
    {\it The EM Algorithm Applied to Determining 
    New Limit Points of Mahler Measures},  
    Control and Cybernetics {\bf 39} (2010), 1185--1192.
      
\bibitem[ERTz]{emeryratcliffetschantz}
    \textsc{V. Emery, J.G. Ratcliffe {\rm and} 
    S.T. Tschantz},
    {\it Salem Numbers and Arithmetic Hyperbolic Geometry},
    preprint (2016).      
        
\bibitem[E]{erdelyi}
    \textsc{A. Erd\'elyi},
    {\it Asymptotic Expansions},
    Dover Publications, New York (1956).

\bibitem[ET]{erdosturan}
    \textsc{P. Erd\"os {\rm and} P. Tur\'an},
    {\it On the Distribution of Roots of 
    Polynomials},
    Ann. Math. {\bf 51} (1950), 105--119.

\bibitem[EOgoY]{esnaultoguisoyu}
    \textsc{H. Esnault, K. Oguiso {\rm and} X. Yu},
    {\it Automorphisms of Elliptic K3 Surfaces and
    Salem Numbers of Maximal Degree},
    Algebr. Geom. {\bf 3} (2016), 496--507. 

\bibitem[Et]{everest}
    \textsc{G. Everest},
    {\it Estimating Mahler Measure},
    Bull. Austral. Math. Soc. {\bf 51} (1995), 145--151.

\bibitem[EtW]{everestward}
    \textsc{G. Everest {\rm and} T. Ward},
    {\it Heights of Polynomials and Entropy in 
    Algebraic Dynamics},
    Springer, London (1999).

\bibitem[EtW2]{everestward2}
    \textsc{G. Everest {\rm and} T. Ward},
    {\it Primes in Divisibility Sequences},
    Cubo Mat. Educ. {\bf 3} (2001), 245--259.

\bibitem[EtSTW]{evereststevenstamsettward}
    \textsc{G. Everest, S. Stevens, D. Tamsett 
    {\rm and} T. Ward},
    {\it Primes Generated by Recurrence Sequences},
    Amer. Math. Monthly {\bf 114} (2007), 417--431.

\bibitem[EPSW]{everestvanderpoortenshparlinskiward}
    \textsc{G. Everest, A. van der Poorten, I. 
    Shparlinski {\rm and} T. Ward},
    {\it Recurrent Sequences},
    Math. Surveys Monographs {\bf 104}, Providence, RI,
    Amer. Math. Soc. (2003).
    
\bibitem[Eve]{evertse}     
    \textsc{J.-H. Evertse},
    {\it Distances Between the Conjugates of 
    an Algebraic Number},
    Publ. Math. Debrecen {\bf 65} (2004), 323--340.       
    
\bibitem[Eiu]{evripidou}    
    \textsc{C.A. Evripidou},
    {\it Coxeter Polynomials of Salem Trees},
    Colloquium Mathematicum {\bf 141} (2015),
    209--226.
    
\bibitem[FPr]{fallerpfister}
    \textsc{B. Faller {\rm and} C.E. Pfister},
    {\it Computation of Topological Entropy via 
    $\varphi$-expansion, an Inverse Problem for the
    Dynamical Systems $\beta x + \alpha$ mod $1$},
    subm. Nonlinearity (2016), preprint. 
     
\bibitem[Fgs]{faltings} 
    \textsc{G. Faltings},
    {\it Endlichkeitss\"atze f\"ur Abelsche Variet\"aten
    \"uber Zahlk\"orpern},  
    Invent. Math. {\bf 73} (1983), 349--366.    
    
\bibitem[FWL]{fangwuli}    
    \textsc{L. Fang, M. Wu {\rm and} B. Li},
    {\it Approximation Orders of Real Numbers by 
    $\beta$-expansions},
    J. Number Theory {\bf 163} (2016), 385--405.

\bibitem[Fb]{farb}
    \textsc{B. Farb},
    {\it Some Problems on Mapping Class Groups and 
    Moduli Space},
    In Problems on Mapping Class Groups and Related 
    Topics (Ed. B. Farb), Proc. Sympos. Pure Math.
    {\bf 74}, Amer. Math. Soc., 
    Providence, RI, (2006), 11--55.

\bibitem[FbLrMt]{farbleiningermargalit}
    \textsc{B. Farb, C. Leininger {\rm and} D. Margalit},
    {\it The Lower Central Series and Pseudo-Anosov 
    Dilatations},
    Amer. J. Math. {\bf 130} (2008), 799--827.     

\bibitem[FbLrMt2]{farbleiningermargalit2}
    \textsc{B. Farb, C. Leininger {\rm and} D. Margalit},
    {\it Small Dilatation Pseudo-Anosov Homeomorphisms 
    and 3-Manifolds},
    Adv. Math. {\bf 228} (2011), 1466--1502.
    
\bibitem[FbMt]{farbmargalit}
    \textsc{B. Farb {\rm and} D. Margalit},
    {\it A primer on Mapping Class Groups},
    Princeton Math. Ser. {\bf 49}, 
    Princeton Univ. Press, New Jersey (2012).   

\bibitem[FiLhPu]{fathilaudenbachpoenaru}
    \textsc{A. Fathi, F. Laudenbach {\rm and} 
    V. Po\'enaru},
    {\it Travaux de Thurston sur les Surfaces},
    Ast\'erisque {\bf 66--67} (1979), 1--284.
    
\bibitem[FLr]{favreriveraletelier}
    \textsc{C. Favre {\rm et} J. Rivera-Letelier},
    {\it Equidistribution Quantitative des Points
    de Petite Hauteur sur la Droite Projective},
    Math. Ann. {\bf 335} (2) (2006), 311-361;
    Corrigendum: ibid, {\bf 339} (4) (2007), 799--801.

\bibitem[FeSo]{feketeszego}
    \textsc{M. Fekete {\rm and} G. Szeg\H{o}},
    {\it On Algebraic Equations with Integral Coefficients
    whose Roots Belong to a Given Point Set},
    Math. Z. {\bf 63} (1955), 158--172.

\bibitem[FMr]{filiminer}
    \textsc{P. Fili {\rm and} Z. Miner},
    {\it Norms Extremal with Respect to the Mahler 
    Measure}, 
    J. Number Theory {\bf 132} (2012), 275--300. 
 
\bibitem[FMr2]{filiminer2}
    \textsc{P. Fili {\rm and} Z. Miner},
    {\it Orthogonal Decomposition of the Space 
    of Algebraic Numbers and Lehmer's Problem},
    J. Number Theory {\bf 133} (2013), 3941--3981. 

\bibitem[FMr3]{filiminer3}
    \textsc{P. Fili {\rm and} Z. Miner},
    {\it Equidistribution and the Heights of
    Totally Real and Totally $p$-adic Numbers},
    Acta Arith. {\bf 170} (2015), 15--25.   

\bibitem[FMr4]{filiminer4}
    \textsc{P. Fili {\rm and} Z. Miner},
    {\it On the Spectrum of the Height for 
    Totally Real Numbers},
    preprint (2017).
    
\bibitem[FSls]{filisamuels}
    \textsc{P. Fili {\rm and} C.L. Samuels},
    {\it On the Non-Archimedean Metric Mahler Measure},
    J. Number Theory {\bf 129} (2009), 1698--1708.    
    
\bibitem[Fg]{flammang}
    \textsc{V. Flammang},
    {\it Two New Points in the Spectrum of the 
    Absolute Mahler Measure of Totally Positive 
    Algebraic Integers},
    Math. Comp. {\bf 65} (213) (1996), 307--311.
        
\bibitem[Fg2]{flammang2}
    \textsc{V. Flammang},
    {\it In\'egalit\'es sur la Mesure 
    de Mahler d'un Polyn\^ome},
    J. Th\'eorie Nombres Bordeaux {\bf 9} (1997), 69--74.             
        
\bibitem[Fg3]{flammang3}
    \textsc{V. Flammang},
    {\it The Mahler Measure of Trinomials of Height 1},
    J. Aust. Math. Soc. {\bf 96} (2014), 231--243.

\bibitem[FGR]{flammanggrandcolasrhin}
    \textsc{V. Flammang, M. Grandcolas {\rm and}  G. Rhin},
    {\it Small Salem numbers},
    Number Theory in Progress, Vol. {\bf 1} 
    (Zakopane-Ko\'scielisko, 1997),
    de Gruyter, Berlin (1999), 165--168.

\bibitem[FRSC]{flammangrhinsacepee}
    \textsc{V. Flammang, G. Rhin {\rm and} J.-M. 
    Sac-\'Ep\'ee},
    {\it Integer Transfinite Diameter and Polynomials with 
    Small Mahler Measure},
    Math. Comp. {\bf 75} (2006), 1527--1540.

\bibitem[FL]{flattolagarias}
    \textsc{L. Flatto {\rm and} J.C. Lagarias},
    {\it The Lap-Counting Function for Linear mod One 
    Transformations I: Explicit Formulas and 
    Renormalizability},
    Ergod. Th. \& Dynam. Sys. {\bf 16} (1996), 451--491;
    {\it II: the Markov Chain for Generalized Lap 
    Numbers},
    {\it ibid} {\bf 17} (1997), 123--146;
    {\it III: the Period of a Markov Chain},
    {\it ibid} {\bf 17} (1997), 369--403.

\bibitem[FLP]{flattolagariaspoonen}
    \textsc{L. Flatto, J.C. Lagarias, {\rm and} B. Poonen},
    {\it The Zeta Function of the Beta-Transformation},
    Ergod. Th. Dynam. Sys. {\bf 14} (1994), 237--266.
 
\bibitem[Fd]{floyd}    
    \textsc{W.J. Floyd},
    {\it Growth of Planar Coxeter Groups, 
    P.V. Numbers, and Salem Numbers},
    Math. Ann. {\bf 293} (1992), 475--483.
    
\bibitem[FdPk]{floydplotnick}    
    \textsc{W.J. Floyd {\rm and} S.P. Plotnick},
    {\it Symmetries of Planar Growth Functions of 
    Coxeter Groups},
    Invent. Math. {\bf 93} (1988), 501--543. 
    
\bibitem[Fl]{fraenkel}
    \textsc{A.S. Fraenkel},
    {\it Systems of Numeration},
    Amer. Math. Monthly {\bf 92} (1985), 105--114.

\bibitem[FsWs]{frankswilliams}
    \textsc{J. Franks {\rm and} R.F Williams},
    {\it Braids and the Jones Polynomial},
    Trans. Amer. Math. Soc. {\bf 303} (1987), 97--108.
    
\bibitem[FYHLMO]{freydetal}    
    \textsc{P. Freyd, D. Ytter, J. Hoste, 
    W.B.R. Lickorish, K. Millett
    {\rm and} A. Oceanu},
    {\it A New Polynomial Invariant of Knots and Links},
    Bull. Amer. Math. Soc. {\bf 12} (1985), 239--246.

\bibitem[FnSa]{friedmanskoruppa}
    \textsc{E. Friedman {\rm and} N.-P. Skoruppa},
    {\it Relative Regulators of Number Fields},
    Invent. Math. {\bf 135} (1999), 115--144.
    
\bibitem[Fy]{frougny}
    \textsc{C. Frougny},
    {\it Representations of Numbers and Finite Automata},
    Math. Systems Thory {\bf 25} (1992), 37--60.

\bibitem[Fy2]{frougny2}
    \textsc{C. Frougny},
    Chap. 7 in Lothaire \cite{lothaire}.
  
\bibitem[FyL]{frougnylai}
    \textsc{C. Frougny {\rm and} A.C. Lai},
    {\it On Negative Bases},  
    in Developments in Language Theory,
    Lecture Notes in Comp. Sci.
    vol. {\bf 5583} (2009), 252--263.    
    
\bibitem[FySh]{frougnysakarovitch}
    \textsc{C. Frougny {\rm and} J. Sakarovitch},
    {\it Number Representation and Finite Automata},
    Chap. 2 in \cite{bertherigo}.    
  
\bibitem[FyS]{frougnysolomyak}
    \textsc{C. Frougny {\rm and} B. Solomyak},
    {\it Finite Beta-Expansions},
    Ergod. Th. Dynam. Sys. {\bf 12} (1992), 713--723.
      
\bibitem[Fs]{fuchs}  
    \textsc{W.H.J. Fuchs},
    {\it On the Zeroes of Power Series with Hadamard Gaps},    
    Nagoya Math. J. {\bf 29} (1967), 167--174.
    
\bibitem[FGleW]{funggranvillewilliams}    
    \textsc{G. Fung, A. Granville {\rm and} H.C. Williams},
    {\it Computation of the First Factor of the 
    Class Number of Cyclotomic Fields},
    J. Number Theory {\bf 42} (1992), 297--312.
    
\bibitem[GM]{galateaumahe}
    \textsc{A. Galateau {\rm and} V. Mah\'e},  
    {\it Some Consequences of Masser's Counting Theorem 
    on Elliptic Curves},
    Math. Z. {\bf 285} (2017), 613--629.
    
\bibitem[G]{ganelius}
    \textsc{T. Ganelius},
    {\it Sequences of Analytic Functions and Their Zeros},
    Arkiv Math. {\bf 3} (1953), 1--50. 

\bibitem[GlZ]{ganglzagier}
    \textsc{H. Gangl {\rm and} D. Zagier},
    {\it Classical and Elliptic Polylogarithms 
    and Special Values of $L$-series},
    The Arithmetic and Geometry of Algebraic Cycles, 
    Proc. (1998), CRM Summer School, Nato Sciences Series
    C, {\bf 548}, Kluwer, Dordrecht, Boston - London (2000),
    561--615.
    
\bibitem[Gza]{garza}
    \textsc{J. Garza},
    {\it On the Height of Algebraic Numbers With 
    Real Conjugates},
    Acta Arith. {\bf 128} (2007), 385--389.
 
\bibitem[Gza2]{garza2}
    \textsc{J. Garza},
    {\it The Mahler Measure of Dihedral Extensions},
    Acta Arith. {\bf 131} (2008), 201--215.

\bibitem[GIMPW]{garzaishakmossinghoffpinnerwiles} 
    \textsc{J. Garza, M.I.M. Ishak, 
    M.J. Mossinghoff, C. Pinner 
    {\rm and} B. Wiles},
    {\it Heights of Roots of Polynomials with Odd 
    Coefficients},
    J. Th\'eorie Nombres Bordeaux 
    {\bf 22} (2010), 369--381. 
    
\bibitem[GH]{ghatehironaka}
    \textsc{E. Ghate {\rm and} E. Hironaka},
    {\it The Arithmetic and Geometry of Salem numbers},
    Bull. Amer. Math. Soc.  {\bf 38} (2001), 293--314.

\bibitem[Gca]{ghioca}
    \textsc{D. Ghioca},
    {\it The Lehmer Inequality and the Mordell-Weil
    Theorem for Drinfeld Modules},
    J. Number Th. {\bf 122} (2007), 37--68.

\bibitem[Gca2]{ghioca2}
    \textsc{D. Ghioca},
    {\it The Local Lehmer Inequality for Drinfeld Modules},
    J. Number Th. {\bf 123} (2007), 426--455.  
    
\bibitem[GOi]{gonoyanagi}
    \textsc{Y. Gon {\rm and} H. Oyanagi},
    {\it Generalized Mahler Measures and 
    Multiple Sine Functions},
    Internat. J. Math. {\bf 15} (2004), 425--442.
         
\bibitem[GAS]{gonzalezacunashort}         
    \textsc{F. Gonz\'alez-Acu$\tilde{{\textsc n}}$a 
    {\rm and} H. Short},
    {\it Cyclic Branched Coverings of Knot Complements
    and Homology Spheres},
    Rev. Mat. Univ. Complut. Madrid {\bf 4} (1991),
    97--120.    
             
\bibitem[Ga]{gora}
    \textsc{P. Gora},
    {\it Invariant Densities for Generalized $\beta$-maps},
    Ergod. Th. Dynam. Sys. {\bf 27} (2007), 1583--1598.

\bibitem[Ga2]{gora2}
    \textsc{P. Gora},
    {\it Invariant Densities for Piecewise Linear 
    Maps of the Unit Interval},
    Ergod. Th. Dynam. Sys. {\bf 29} (2009), 1549--1583.

\bibitem[GtHt]{grandethugot}
    \textsc{M. Grandet-Hugot},
    {\it El\'ements Alg\'ebriques 
    Remarquables dans un Corps 
    de S\'eries Formelles},
    Acta Arith. {\bf 14} (1968), 177--184.
    
\bibitem[Grdn]{gordon}
    \textsc{C.McA. Gordon},
    {\it Knots Whose Branched Cyclic Coverings Have
    Periodic Homology},
    Trans. Amer. Math. Soc. {\bf 168} (1972), 357--370. 
    
\bibitem[Gle]{granville}
    \textsc{A. Granville},
    {\it On the Size of the First Factor of the Class
    Number of a Cyclotomic Field},
    Invent. Math. {\bf 100} (1990), 321--338.

\bibitem[Gv]{gromov}
    \textsc{M. Gromov},
    {\it On the Entropy of Holomorphic Maps},
    Enseign. Math. {\bf 49} (2003), 217--235.
    
\bibitem[GsHaMln]{grosshironakamcmullen}
    \textsc{B.H. Gross, E. Hironaka {\rm and} 
    C.T. McMullen},
    {\it Cyclotomic Factors of Coxeter Polynomials},
    J. Number Theory {\bf 129} (2009), 1034--1043. 

\bibitem[GsMln]{grossmacmullen}
    \textsc{B.H. Gross {\rm and} C.T. McMullen},
    {\it Automorphisms of Even Unimodular 
    Lattices and Unramified Salem Numbers},
    J. Algebra {\bf 257} (2002), 265--290.
    
\bibitem[Gr]{grothendieck}
    \textsc{A. Grothendieck},
    {\it Produits Tensoriels Topologiques et Espaces 
    Nucl\'eaires}, M\'em. Amer. Math. Soc. {\bf 16},
    A.M.S., Providence, RI (1955).

\bibitem[Gr2]{grothendieck2}
    \textsc{A. Grothendieck},
    {\it La Th\'eorie de Fredholm},
    Bull. Soc. Math. France {\bf 84} (1956), 319--384.

\bibitem[GdVG]{guichardvergergaugry}
    \textsc{C. Guichard {\rm and} J.-L. Verger-Gaugry},
    {\it On Salem Numbers, Expansive Polynomials and 
    Stieltjes Continued Fractions},
    J. Th\'eorie Nombres Bordeaux {\bf 27} (2015), 769--804.

\bibitem[Gng]{guting}
    \textsc{R. G\"uting},
    {\it Approximation of Algebraic Numbers by 
    Algebraic Numbers},
    Michigan Math. J. {\bf 8} (1961), 149--159.

\bibitem[Hgr]{habegger}
    \textsc{P. Habegger},
    {\it Small Height and Infinite Non-Abelian Extensions},
    Duke Math. J.  {\bf 162} (2013), 2027--2076.    

\bibitem[HmSg]{hamsong}
    \textsc{J.-Y. Ham {\rm and} W.T. Song},
    {\it the Minimum Dilatation of Pseudo-Anosov 5-Braids},
    Exp. Math. {\bf 16} (2007), 167--179.

\bibitem[H]{hare}
    \textsc{K. Hare},
    {\it Beta-expansions of Pisot and Salem Numbers},
    Computer Algebra 2006, World Sci. Publ., Hackensack,
    NJ, (2007), 67--84.

\bibitem[HMf]{haremossinghoff}
    \textsc{K. Hare {\rm and} M. Mossinghoff},
    {\it Negative Pisot and Salem Numbers as Roots of
    Newman Polynomials},
    Rocky Mountain J. Math. {\bf 44} (2014), 113--138.

\bibitem[HPu]{harepanju}
    \textsc{K. Hare {\rm and} M. Panju},
    {\it Some Comments on Garsia Numbers},
    Math. Comp. {\bf 82} (282) (2013), 1197--1221.

\bibitem[HTe]{haretweedle}
    \textsc{K. Hare {\rm and} D. Tweedle},
    {\it Beta-expansions for Infinite Families
    of Pisot and Salem Numbers},
    J. Number Theory {\bf 128} (2008), 2756--2765.

\bibitem[Hy]{haydn}
    \textsc{N.T.A. Haydn},
    {\it Meromorphic Extension of the Zeta Function 
    for Axiom A Flows},
    Ergod. Th. Dynam. Sys. 
    {\bf 10} (1990), 347--360.
    
\bibitem[Hi]{hichri}
    \textsc{H. Hichri},
    {\it On the beta-expansion of Salem Numbers of Degree 
    8},
    LMS J. Comput. Math. {\bf 17} (2014), 289--301.
   
\bibitem[Hi2]{hichri2}
    \textsc{H. Hichri},
    {\it Beta expansion for Some Particular Sequences of 
    Salem Numbers},   
    Int. J. Number Theory {\bf 10} (2014), 2135--2149.
 
\bibitem[Hi3]{hichri3}
    \textsc{H. Hichri},
    {\it Beta Expansion of Salem Numbers Approaching Pisot 
    Numbers with the Finiteness Property},
    Acta Arith. {\bf 168} (2015), 107--119. 
    
\bibitem[HR]{hilgertrilke}
    \textsc{J. Hilgert {\rm and} F. Rilke},
    {\it Meromorphic Continuation of 
    Dynamical Zeta Functions via Transfer Operators},
    J. Funct. Analysis {\bf 254} (2008), 476--505.

\bibitem[Hy]{hindry}
    \textsc{M. Hindry},
    {\it The Analog of Lehmer Problem on Abelian Varieties 
    with Complex Multiplication},
    preprint.
     
\bibitem[HyS]{hindrysilverman}
    \textsc{M. Hindry {\rm and} J. Silverman},
    {\it On Lehmer's Conjecture for Elliptic Curves},
    in {\it S\'eminaire de Th\'eorie des Nombres, Paris 
    1988--1989},
    Progress in Math. {\bf 91}, Birkh\"auser, 
    Paris (1990), 103--116.

\bibitem[HraMi]{hirasawamurasugi}
    \textsc{M. Hirasawa {\rm and} K. Murasugi}
    {\it Various Stabilities of the Alexander Polynomials
    of Knots and Links},
    preprint (2013).
    
\bibitem[Ha]{hironaka}
    \textsc{E. Hironaka},
    {\it The Lehmer Polynomial and Pretzel Links},
    Canad. Math. Bull. {\bf 44} (2001), 440--451;  
    Erratum, {\it ibid}, {\bf 45} (2002), 231. 

\bibitem[Ha2]{hironaka2}
    \textsc{E. Hironaka},
    {\it Lehmer's Problem, McKay's Correspondence,
    and \,$2, 3, 7$},
    Topics in Algebraic and Noncommutative Geometry,
    Contemp. Math. {\bf 324} (2003), 123--138. 

\bibitem[Ha3]{hironaka3}
    \textsc{E. Hironaka},
    {\it Salem-Boyd sequences and Hopf plumbing},
    Osaka Math. J. {\bf 43} (2006), 497--516.    
    
\bibitem[Ha4]{hironaka4}
    \textsc{E. Hironaka},
    {\it What is ... Lehmer's number?},
    Notices Amer. Math. Soc. {\bf 56} (2009), 274--375.    

\bibitem[Ha5]{hironaka5}
    \textsc{E. Hironaka},
    {\it Small Dilatation Mapping Classes
    Coming From the Simplest Hyperbolic Braid},
    Algebr. Geom. Topol. {\bf 10} (2010), 2041--2060.

\bibitem[Ha6]{hironaka6}
    \textsc{E. Hironaka},
    {\it Dynamics of Mapping Classes on Surfaces},
    Lectures (2012).
 
\bibitem[Ha7]{hironaka7}
    \textsc{E. Hironaka},
    {\it Small Dilatation Pseudo-Anosov Mapping Classes} 
    in Intelligence of Low-Dimensional Topology
    RIMS {\bf 1812} (2012), 25--33.
    
\bibitem[Ha8]{hironaka8}
    \textsc{E. Hironaka},
    {\it Penner Sequences and Asymptotics of Minimum
    Dilatations for Subfamilies  of the Mapping Class 
    Group},
    Topology Proceedings {\bf 44} (2014), 315--324.
    
\bibitem[Ha9]{hironaka9}
    \textsc{E. Hironaka},
    {\it Mapping Classes Associated to 
    Mixed-Sign Coxeter Graphs},
    arXiv:1110.1013, (2012).   
    
\bibitem[Ha10]{hironaka10}
    \textsc{E. Hironaka},
    {\it Quotient Families of Mapping Classes},
    preprint (2016).   

\bibitem[HaKn]{hironakakin}
    \textsc{E. Hironaka {\rm and} E. Kin},
    {\it A Family of Pseudo-Anosov Braids
    With Small Dilatation},
    Algebr. Geom. Topol. {\bf 6} (2006), 699--738. 
   
\bibitem[HaLi]{hironakaliechti}
    \textsc{E. Hironaka {\rm and} L. Liechti},
    {\it On Coxeter Mapping Classes and Fibered 
    Alternating Links},
    Michigan Math. J. {\bf 65} (2016), 799--812.  
    
\bibitem[Hhn]{hoehn}
    \textsc{G. H\"ohn},
    {\it On a Theorem of Garza Regarding Algebraic 
    Numbers With Real Conjugates},     
    Int. J. Number Theory {\bf 7} (2011),  943--945.
    
\bibitem[HhnSk]{hoehnskoruppa}
    \textsc{G. H\"ohn {\rm and} N.-P. Skoruppa},    
    {\it Un r\'esultat de Schinzel},
    J. Th\'eorie Nombres Bordeaux {\bf 5} (1993), 
    185.    
    
\bibitem[Hr]{hofbauer}
    \textsc{F. Hofbauer},
    {\it $\beta$-Shifts Have Unique Maximal Measure},    
    Monatsh. 
    Math. {\bf 85} (1978), 189--198.

\bibitem[Hr2]{hofbauer2}
    \textsc{F. Hofbauer},
    {\it Periodic points for piecewise monotonic 
    transformations},
    Ergod. Th. Dynam. Sys. {\bf 5} (1985), 237--256.

\bibitem[HrK]{hofbauerkeller}
    \textsc{F. Hofbauer {\rm and} G. Keller},
    {\it Ergodic Properties of Invariant Measures 
    for Piecewise Monotonic Transformations}
    Math. Z. {\bf 180} (1982), 119--140.
    
\bibitem[HrK2]{hofbauerkeller2}
    \textsc{F. Hofbauer {\rm and} G. Keller},
    {\it Zeta-functions and Transfer Operators 
    for Piecewise Linear Transformations},    
    J. Reine Angew. Math. {\bf 352} (1984), 100--113.
    
\bibitem[HTY]{hutongyu}    
    \textsc{H. Hu, X. Tong {\rm and} Y. Yu},
    {\it On Consecutive 0 Digits in the 
    $\beta$-expansion of 1},
    J. Number Theory {\bf 166} (2016), 219--234.
    
\bibitem[HNi]{hughesnikeghbali}
    \textsc{C.P. Hughes {\rm and} A. Nikeghbali},
    {\it The zeros of random polynomials 
    cluster uniformly near the unit circle},
    Compos. Math. {\bf 144} (3) (2008), 734--746.   
    
\bibitem[IMPW]{ishakmossinghoffpinnerwiles} 
    \textsc{M.I.M. Ishak, M.J. Mossinghoff, C. Pinner 
    {\rm and} B. Wiles},
    {\it Lower Bounds for Heights in Cyclotomic 
    Extensions},   
    J. Number Theory {\bf 130} (2010), 1408--1424.  
    
\bibitem[IL]{issalalin}    
    \textsc{Z. Issa {\rm and} M. Lal\'in},
    {\it A Generalization of a Theorem of Boyd 
    and Lawton},    
    Canad. Math. Bull. {\bf 56} (2013), 759--768.
    
\bibitem[IS]{itosadahiro}
    \textsc{S. Ito {\rm and} T. Sadahiro},
    {\it Beta-expansions with Negative Bases},
    Integers {\bf 9}:A 22 (2009), 239--259.
    
\bibitem[IT]{itotakahashi}
    \textsc{S. Ito {\rm and} Y. Takahashi},
    {\it Markov Subshifts and Realization of 
    $\beta$-expansions},    
    J. Math. Soc. Japan {\bf 26} (1974), 33--55.

\bibitem[Itm]{ittersum}
    \textsc{J.-W. M. van Ittersum},
    {\it A Group-Invariant Version of Lehmer's 
    Conjecture on Heights},
    J. Number Theory {\bf 171} (2017), 145--154.
    
\bibitem[JSls]{jankauskassamuels}
    \textsc{J. Jankauskas {\rm and} C.L. Samuels},
    {\it The t-Metric Mahler Measures of 
    Surds and Rational Numbers},
    Acta Math. Hungar. {\bf 134} (2012), 481--498.
        
\bibitem[JiQn]{jiqin} 
    \textsc{Q. Ji {\rm and} H. Qin},       
    {\it The Numerical Factors of $\Delta_{n}(f,g)$},
    Indian J. Pure Appl. Math. {\bf 46} (2015), 701--714.
        
\bibitem[JnZg]{jinzhang}
    \textsc{X. Jin {\rm and} F. Zhang},
    {\it Zeros of the Jones Polynomials for        
    Families of Pretzel Links},
    Physica A {\bf 328} (2003), 391--408.    
        
\bibitem[JnZg2]{jinzhang2}
    \textsc{X. Jin {\rm and} F. Zhang},
    {\it Jones Polynomials and Their Zeros for        
    a Family of Links},
    Physica A {\bf 333} (2004), 183--196.       
        
\bibitem[JnZg3]{jinzhang3}
    \textsc{X. Jin {\rm and} F. Zhang},
    {\it Zeros of the Jones Polynomial for 
    Multiple Crossing-Twisted Links},        
    J. Stat. Phys. {\bf 140} (2010), 1054--1064.
            
\bibitem[Jos]{jones}        
    \textsc{V. Jones},
    {\it A Polynomial Invariant for Knots via von Neumann 
    Algebras},
    Bull. Amer. Math. Soc. {\bf 12} (1985), 103--111.
        
\bibitem[Kgr]{kaiblinger}        
    \textsc{N. Kaiblinger},
    {\it On the Lehmer Constant of Finite 
    Cyclic Groups},
    Acta Arith. {\bf 142} (2010), 79--84.
        
\bibitem[Ke]{kalle}
    \textsc{C. Kalle},
    {\it Isomorphisms Between Positive and 
    Negative $\beta$-transformations},
    Ergod. Th. Dynam. Sys. {\bf 34} (2014),
    153--170.         
        
\bibitem[Kbu]{kanenobu} 
    \textsc{T. Kanenobu},
    {\it Module d'Alexander des Noeuds Fibr\'es
    et Polyn\^ome de Hosokawa des Lacements Fibr\'es},
    Math. Sem. Notes Kobe Univ. {\bf 91} (1981), 75--84.       

\bibitem[Kbu2]{kanenobu2} 
    \textsc{T. Kanenobu},
    {\it Infinitely Many Knots With the Same
    Polynomial Invariant},
    Proc. Amer. Math. Soc. {\bf 97} (1986), 158--162.        
        
\bibitem[Kui]{kawauchi} 
    \textsc{A. Kawauchi}       
    {\it A Survey on Knot Theory},
    Birkha\"user Verlag, Basel (1996). 
        
\bibitem[Kya]{kedlaya}
    \textsc{K.S. Kedlaya},
    {\it Fourier Transforms and $p$-adic ``Weil II"},
    Compos. Math. {\bf 142} (2006), 1426--1450.        
        
\bibitem[K]{keller}
    \textsc{G. Keller},
    {\it Generalized Bounded Variation and Applications 
    to Piecewise Monotonic Transformations},
    Z. Wahrsch. Verw. Gebiete {\bf 69} (1985), 461--478.
    
\bibitem[K2]{keller2}
    \textsc{G. Keller},
    {\it Markov Extensions, Zeta Functions, and 
    Fredholm Theory
    for Piecewise Invertible Dynamical Systems},
    Trans. Amer. Math. Soc. {\bf 314} (1989), 433--497.                

\bibitem[Kls]{kellerhals}
    \textsc{R. Kellerhals},
    {\it Cofinite Hyperbolic Coxeter Groups, Minimal 
    Growth Rate and Pisot Numbers},
    Algebr. Geom. Top. {\bf 13} (2013), 1001--1025.
        
\bibitem[KlsKv]{kellerhalskolpakov}
    \textsc{R. Kellerhals {\rm and} A. Kolpakov},
    {\it The Minimal Growth Rate of Cocompact Coxeter 
    Groups in Hyperbolic $3$-Space},
    Canad. J. Math. {\bf 66} (2014), 354--372.
 
\bibitem[KlsN]{kellerhalsnonaka}
    \textsc{R. Kellerhals {\rm and} J. Nonaka},
    {\it The Growth Rates of Ideal Coxeter Polyhedra
    in Hyperbolic 3-Space}, 
    preprint (2017).
    
\bibitem[KlsP]{kellerhalsperren}
    \textsc{R. Kellerhals {\rm and} G. Perren},
    {\it On the Growth of Cocompact Hyperbolic Coxeter 
    Groups},  
    European J. Combin. {\bf 32} (2011), 1299--1316.        
        
\bibitem[Kda]{kerada}        
    \textsc{M. Kerada},
    {\it Une Caract\'erisation de Certaines 
    Classes d'Entiers Alg\'ebriques G\'en\'eralisant les 
    Nombres de Salem},
    Acta Arith. {\bf 72} (1995), 55--65.
          
\bibitem[Kn]{kestelman}       
    \textsc{H. Kestelman},
    {\it Automorphisms of the Field of Complex Numbers},
    Proc. London Math. Soc. {\bf 53} (1951), 1--12.   

\bibitem[KmLe]{kimlee}
    \textsc{S.-H. Kim {\rm and} J.H. Lee},
    {\it On the Zeros of Self-Reciprocal
    Polynomials Satisfying Certain Coefficient Conditions},
    Bull. Korean Math. Soc. {\bf 47} (2010), 1189--1194.

\bibitem[KmPk]{kimpark}
    \textsc{S.-H. Kim {\rm and} C.W. Park},
    {\it On the Zeros of Certain Self-Reciprocal
    Polynomials},
    J. Math. Anal. Appl. {\bf 339} (2008), 240--247.

\bibitem[Kin]{kin} 
    \textsc{E. Kin},
    {\it Notes on Pseudo-Anosovs with Small Dilatations 
    Coming from the Magic 3-Manifold},
    in Representation Spaces, Twisted Topological 
    Invariants and Geometric Structures of 3-Manifolds,
    RIMS Kokyuroku {\bf 1836} (2013), 45--64.

\bibitem[KinKjTa]{kinkojimatakasawa} 
    \textsc{E. Kin, S. Kojima {\rm and} M. Takasawa},
    {\it Entropy Versus Volume for Pseudo-Anosovs},
    Exp. Math. {\bf 18} (2009), 397--407.

\bibitem[KinKjTa2]{kinkojimatakasawa2} 
    \textsc{E. Kin, S. Kojima {\rm and} M. Takasawa},
    {\it Minimal Dilatations of Pseudo-Anosovs 
    Generated by the Magic 3-Manifold and Their Asymptotic 
    Behavior},
    Algebr. Geom. Topol. {\bf 13} (2013), 3537--3602.

\bibitem[KinTa]{kintakasawa} 
    \textsc{E. Kin {\rm and} M. Takasawa},
    {\it An Asymptotic Behavior of the Dilatation
    for a Family of Pseudo-Anosov Braids},
    Kodai Math. J. {\bf 31} (2008), 92--112.
 
\bibitem[KinTa2]{kintakasawa2} 
    \textsc{E. Kin {\rm and} M. Takasawa},
    {\it Pseudo-Anosov Braids with Small Entropy and the
    Magic 3-Manifold},
    Comm. Anal. Geom. {\bf 19} (2011), 705--758.
  
\bibitem[KinTa3]{kintakasawa3} 
    \textsc{E. Kin {\rm and} M. Takasawa},
    {\it Pseudo-Anosovs on Closed Surfaces Having Small 
    Entropy and the Whitehead Sister Link Exterior},
    J. Math. Soc. Japan {\bf 65} (2013), 411--446.  
    
\bibitem[Kpv]{kolpakov} 
    \textsc{A. Kolpakov},
    {\it Deformation of Finite-Volume Hyperbolic Coxeter
    Polyhedra, Limiting Growth rates and Pisot Numbers},
    European J. Combin. {\bf 33} (2012), 1709--1724.
 
\bibitem[KU]{komoriumemoto} 
    \textsc{Y. Komori {\rm and} Y. Umemoto},
    {\it On the Growth of Hyperbolic 3-Dimensional
    Generalized Simplex Reflection Groups},
    Proc. Japan Acad. {\bf 88}, Ser. {\bf A} (2012),
    62--65.

\bibitem[KY]{komoriyukita} 
    \textsc{Y. Komori {\rm and} T. Yukita},
    {\it On the Growth Rate of Ideal Coxeter Groups 
    in Hyperbolic 3-Space},
    Proc. Japan Acad. {\bf 91}, Ser. {\bf A} (2015),
    155--159.   
    
\bibitem[KfPr]{koseleffpecker}    
    \textsc{P.-V. Koseleff {\rm and} D. Pecker},
    {\it On Alexander-Conway Polynomials of 
    Two-Bridges Links},
    J. Symbolic Comput. {\bf 68} (2015), 215--229. 
    
\bibitem[Krr]{kronecker}
    \textsc{L. Kronecker},
    {\it Zwei S\"atze \"uber Gleichungen mit 
    Ganzzahligen Coefficienten},
     J. Reine Angew. Math. {\bf 53} (1857), 173--175.
        
\bibitem[KLO]{kurokawalalinochiai}        
    \textsc{N. Kurokawa, M. Lal\'in {\rm and} H. Ochiai},
    {\it Higher Mahler Measures and Zeta Functions},
     Acta Arith. {\bf 135} (2008), 269--297.         
  
\bibitem[Kwn]{kwon}
    \textsc{D. Kwon},
    {\it Beta-Numbers Whose Conjugates 
    Lie Near the Unit Circle},
    Acta Arith. {\bf 127} (2007), 33--47.    
        
\bibitem[Ls]{lagarias}
    \textsc{J. Lagarias},        
    {\it Number Theory Zeta Functions and Dynamical Zeta     
    Functions},
    Contemp. Math. {\bf 237} (1999), 45--86. 

\bibitem[Los]{lakatos}
    \textsc{P. Lakatos},
    {\it On the Coxeter Polynomials of Wild Stars},
    Linear Algebra Appl. {\bf 293} (1999), 159--170.

\bibitem[Los2]{lakatos2}
    \textsc{P. Lakatos},
    {\it Salem Numbers, PV Numbers and Spectral Radii
    of Coxeter Transformations},
    C.R. Math. Acad. Sci. Soc. R. Can. {\bf 23} (2001), 
    71--77.

\bibitem[Los3]{lakatos3}
    \textsc{P. Lakatos},
    {\it On Zeros of Reciprocal Polynomials},
    C.R. Math. Rep. Acad. Sci. Canada {\bf 24} (2002),
    91--96.   
        
\bibitem[Los4]{lakatos4}
    \textsc{P. Lakatos},
    {\it A New Construction of Salem Polynomials},
    C.R. Math. Acad. Sci. Soc. R. Can. {\bf 25} (2003), 
    47--54.    
    
\bibitem[Los5]{lakatos5}
    \textsc{P. Lakatos},
    {\it Salem Numbers Defined by Coxeter Transformations},
    Linear Algebra Appl. {\bf 432} (2010), 144--154.
   
\bibitem[LosLi]{lakatoslosonczi}
    \textsc{P. Lakatos {\rm and} L. Losonczi},
    {\it Self-Inversive Polynomials Whose Zeros
    are on the Unit Circle},
    Publ. Math. Debrecen {\bf 65} (2004), 409--420.   
   
\bibitem[LosLi2]{lakatoslosonczi2}
    \textsc{P. Lakatos {\rm and} L. Losonczi},
    {\it Polynomials With All Zeros
    on the Unit Circle}, 
    Acta Math. Hung. {\bf 125} (2009), 341--356.  
   
\bibitem[Lin]{lalin}
    \textsc{M.N. Lal\'in},
    {\it An Algebraic Integration for Mahler Measures},
    Duke Math. J. {\bf 138} (2007), 391--422. 
        
\bibitem[Lin2]{lalin2}
    \textsc{M.N. Lal\'in},
    {\it Mahler Measures and Computations with Regulators},
    J. Number Theory {\bf 128} (2008), 1231--1271. 
  
\bibitem[Lin3]{lalin3}
    \textsc{M.N. Lal\'in},
    {\it Higher Mahler Measure as a Massey Product 
    in Deligne Cohomology},  
    Low-Dimensional Topology and Number Theory,
    abstract from the Workshop held 
    August 15--21, 2010., org. P.E. Gunnells,
    W. Neumann, A.S. Sikora and D. Zagier, 
    Oberwolfach Reports (2010).
            
\bibitem[LinS]{lalinsinha}
    \textsc{M.N. Lal\'in {\rm and} K. Sinha},
    {\it Higher Mahler Measure for Cyclotomic Polynomials
    and Lehmer's Question}, 
    Ramanujan J. {\bf 26} (2011), 257--294.            
            
\bibitem[LinSy]{lalinsmyth}
    \textsc{M.N. Lal\'in {\rm and} C.J. Smyth},
    {\it Unimodularity of Zeros of Self-Inversive 
    Polynomials},
    Acta Math. Hung. {\bf 138} (2013), 85--101;
    {\it Addendum}, {\it ibid}, {\bf 147} (2015), 255--257.     
            
\bibitem[Lg]{langevin}
    \textsc{M. Langevin},
    {\it M\'ethode de Fekete-Szeg\H{o} et 
    Probl\`eme de Lehmer},
    C.R. Acad. Sci. Paris S\'erie I Math. 
    {\bf 301} (1) (1985), 463--466.
    
\bibitem[Lg2]{langevin2}
    \textsc{M. Langevin},
    {\it Minorations de la Maison et de la Mesure de 
    Mahler de Certains Entiers Alg\'ebriques},
    C.R. Acad. Sci. Paris S\'erie I Math. 
    {\bf 303} (12) (1986), 523--526.
        
\bibitem[Lg3]{langevin3}
    \textsc{M. Langevin},
    {\it Calculs Explicites de Constantes de Lehmer},
    in {\it Groupe de travail en Th\'eorie Analytique 
    et El\'ementaire des nombres},
    1986--1987, Publ. Math. Orsay, 
    Univ. Paris XI, Orsay {\bf 88} (1988), 52--68.  

\bibitem[LuTt]{lanneauthiffeault}
    \textsc{E. Lanneau {\rm and} J.-L. Thiffeault},
    {\it On the Minimum Dilatation of Pseudo-Anosov
    Homeomorphisms on Surfaces of Small Genus},
    Ann. Inst. Fourier (Grenoble) {\bf 61} (2011), 
    105--144.

\bibitem[LuTt2]{lanneauthiffeault2}
    \textsc{E. Lanneau {\rm and} J.-L. Thiffeault},
    {\it On the Minimum Dilatation of Braids on
    Punctured Discs},
    Geom. Dedicata {\bf 152} (2011), 165--182.

\bibitem[LY]{lasotayorke}
    \textsc{A. Lasota {\rm and} J. Yorke},
    {\it On the Existence of Invariant Measures for 
    Piecewise
    Monotonic Transformations},
    Trans. Amer. Math. Soc. {\bf 186} (1973), 481--488.     
    
\bibitem[LrWn]{lauderwan}    
     \textsc{A.G.B. Lauder {\rm and} D. Wan},
     {\it Counting Rational Points over Finite Fields 
     of Small Characteristic},
    Algorithmic Number Theory: Lattices, Curves and 
    Cryptography, Math. Sci. Res. Inst. Publ. {\bf 44},
    Cambridge Univ. Press, Cambridge    
    (2008), 579--612.

\bibitem[La]{laurent}
    \textsc{M. Laurent},
    {\it Minoration de la Hauteur de N\'eron-Tate},
    in {\it S\'eminaire de Th\'eorie des Nombres Paris 
    1981-1982},
    Progress in Math., Birkh\"auser, Paris {\bf 38} (1983), 
    137--152.

\bibitem[La2]{laurent2}
    \textsc{M. Laurent},
    {\it Sur Quelques R\'esultats R\'ecents 
    de Transcendance},
    Ast\'erisque, Journ\'ees Arithm\'etiques Luminy 1991 
    (1992), 209--230. 

\bibitem[Lw]{lawton}
    \textsc{W.M. Lawton},
    {\it A Problem of Boyd Concerning Geometric Means of 
    Polynomials},
    J. Number Theory {\bf 16} (1983), 356--362.

\bibitem[Let]{le}
    \textsc{T. Le},
    {\it Homology Torsion Growth and Mahler Measure},
    Comment. Math. Helv. {\bf 89} (2014), 719--757.

\bibitem[Let2]{le2}
    \textsc{T. Le},
    {\it Hyperbolic Volume and Homology Torsion},
    preprint (2017).

\bibitem[Le]{lehmer}
    \textsc{D.H. Lehmer},
    {\it Factorization of Certain Cyclotomic Functions},
    Ann. Math. {\bf 34} (1933), 461--479.

\bibitem[Le2]{lehmer2}
    \textsc{D.H. Lehmer},
    {\it Prime Factors of Cyclotomic Class Numbers},
    Math. Comp. {\bf 31} (138) (1977), 599--607.

\bibitem[LeMy]{lehmermasley}
    \textsc{D.H. Lehmer {\rm and} J.M. Masley},
    {\it Table of the Cyclotomic Class Numbers
    $h^{*}(p)$ and their Factors for
    $200 < p < 521$},
    Math. Comp. {\bf 32} (142) (1978), 577--582.

\bibitem[Lgr]{leininger}
    \textsc{C.J. Leininger},
    {\it On Groups Generated by Two Positive
    Multi-Twists: Teichm\"uller Curves and 
    Lehmer's Number},
    Geom. Topol. {\bf 8} (2004), 1301--1359.
    
\bibitem[Lsn]{levinson}
    \textsc{N. Levinson},
    {\it Gap and Density Theorems},
    Amer. Math. Soc. Colloquium Publ. Vol. XXVI (1940).

\bibitem[Lwn]{lewin}
    \textsc{L. Lewin},
    {\it Polylogarithms and Associated Functions},
    North-Holland (1981).

\bibitem[LC]{lichen}
    \textsc{B. Li {\rm and} Y.-C. Chen},
    {\it Chaotic and Topological Properties 
    of beta-transformations},
    J. Math. Anal. Appl. {\bf 383} (2011), 585--596.
    
\bibitem[LPWW]{liperssonwangwu}
    \textsc{B. Li, T. Persson, B. Wang {\rm and} J. Wu},
    {\it Diophantine Approximation of the Orbit of 1 
    in the Dynamical System of beta expansions},    
    Math. Z. {\bf 276} (2014), 799--827.
        
\bibitem[LY]{liyorke}
    \textsc{T.Y. Li {\rm and} J. Yorke},
    {\it Ergodic Transformations from an Interval to 
    Itself},
    Trans. Amer. Math. Soc. {\bf 235} (1978), 183--192.

\bibitem[LS]{liaosteiner}
    \textsc{L. Liao {\rm and} W. Steiner},
    {\it Dynamical Properties of the Negative 
    beta-transformation},
    Ergod. Th. Dynam. Sys. {\bf 32} (2012), 1673--1690.

\bibitem[Ld]{lind}
    \textsc{D.A. Lind},
    {\it Dynamical Properties of Quasihyperbolic 
    Toral Automorphisms},
    Ergod. Th. Dynam. Sys. {\bf 2} (1982), 49--68.

\bibitem[Ld2]{lind2}
    \textsc{D.A. Lind},
    {\it The Entropies of Topological Markov Shifts
    and a Related Class of Algebraic Integers},
    Ergod. Th. Dynam. Sys. {\bf 4} (1984), 283--300.
 
\bibitem[Ld3]{lind3}
    \textsc{D.A. Lind},
    {\it Matrices of Perron Numbers},
    J. Number Theory {\bf 40} (1992), 211--217. 

\bibitem[Ld4]{lind4}
    \textsc{D.A. Lind},
    {\it Lehmer's Problem for Compact Abelian Groups},
    Proc. Amer. Math. Soc., {\bf 133} (2005), 1411--1416. 
   
\bibitem[LdM]{lindmarcus}
    \textsc{D.A. Lind {\rm and} B. Marcus},
    {\it An Introduction to Symbolic Dynamics and Coding},
    Cambridge University Press, Cambridge (1995).

\bibitem[LdSW]{lindschmidtward}
    \textsc{D.A. Lind, K. Schmidt {\rm and} T. Ward},
    {\it Mahler Measures and Entropy for Commuting
    Automorphisms of Compact Groups},
    Invent. Math. {\bf 101} (1990), 593--629.

\bibitem[Ltu]{litcanu}
    \textsc{R. Litcanu},
    {\it Petits Points et Conjecture de Bogomolov},
    Expo. Math. {\bf 25} (2007), 37--51.
    
\bibitem[LSt]{lloydsmith}
    \textsc{C.W. Lloyd-Smith},
    {\it Algebraic Numbers Near the Unit Circle},
    Acta Arith. {\bf 45} (1985), 43--57.
   
\bibitem[Lo]{lothaire}
    \textsc{M. Lothaire},
    {\it Algebraic Combinatorics on Words},
    in Encylopedia of Mathematics and its Applications,
    Vol. {\bf 90}, Cambridge University Press, 
    Cambridge (2002).
    
\bibitem[Lt]{louboutin}
    \textsc{R. Louboutin},
    {\it Sur la Mesure de Mahler d'un Nombre Alg\'ebrique},
    C. R. Acad. Sci. Paris S\'erie I, t. {\bf 296} (1983), 
    707--708.

\bibitem[LW]{luwu}
    \textsc{F. L\"u {\rm and} J. Wu},
    {\it Diophantine Analysis in beta-Dynamical Systems
    and Hausdorff Dimensions},
    Adv. Math. {\bf 290} (2016), 919--937.

\bibitem[LyhMi]{lyubichmurasugi}
    \textsc{L. Lyubich {\rm and} K. Murasugi},
    {\it On Zeros of the Alexander Polynomial of an 
    Alternating Knot},
    Topology Appl. {\bf 159} (2012), 290--303.    

\bibitem[Lkn]{lubkin}
    \textsc{S. Lubkin},
    {\it A $p$-adic proof of Weil's Conjectures},
    Ann. Math. {\bf 87} (1968), 105--194; {\it ibid}
    {\bf 87} (1968), 195--255.
    
\bibitem[MS]{mackeesmyth}
    \textsc{J.F. McKee {\rm and} C. Smyth},
    {\it Salem Numbers, Pisot Numbers, Mahler Measures 
    and Graphs},
    Exp. Math. {\bf 14} (2005), 211--229. 

\bibitem[MS2]{mackeesmyth2}
    \textsc{J.F. McKee {\rm and} C. Smyth},
    {\it Integer Symmetric Matrices of Small Spectral 
    Radius and Small Mahler Measure}, 
    Int. Math. Res. Not. (IMRN) {\bf 1} (2012), 102--136.

\bibitem[MS3]{mackeesmyth3}
    \textsc{J.F. McKee {\rm and} C. Smyth},
    {\it Salem Numbers and Pisot Numbers via Interlacing},
    Canad. J. Math. {\bf 64} (2012), 345--367.

\bibitem[Ml]{maclachlan} 
    \textsc{G.J. Mclachlan},
    {\it Commensurability Classes of Discrete 
    Arithmetic Hyperbolic Groups},
    Groups Geom. Dyn. {\bf 5} (2011), 767--785. 
 
\bibitem[MlK]{maclachlankrishnan} 
    \textsc{G.J. Mclachlan {\rm and}  T. Krishnan},
    {\it The EM Algorithm and Extensions},
    Wiley Series in Probability and Statistics (2nd Ed.),
    Wiley Interscience, Hoboken, NJ (2008).
  
\bibitem[MlK]{maclachlanreid} 
    \textsc{C. Mclachlan {\rm and} A.W. Reid},
    {\it The Arithmetic of Hyperbolic $3$-manifolds},
    Graduate Texts in Mathematics, Springer, 
    New York {\bf 219} (2003).  
    
\bibitem[Mln]{macmullen}
    \textsc{C.T. McMullen},
    {\it From Dynamics on Surfaces to Rational Points on 
    Curves},
    Bull. Amer. Math. Soc. {\bf 37} (2000), 119--140. 
    
\bibitem[Mln2]{macmullen2}
    \textsc{C.T. McMullen},
    {\it Polynomial Invariants for Fibered 3-Manifolds 
    and Teichmuller Geodesics for Foliations},
    Ann. Sci. Ecole Norm. Sup. {\bf 33} (2000),     
    519--560.
    
\bibitem[Mln3]{macmullen3}
    \textsc{C.T. McMullen},
    {\it Dynamics on $K3$ Surfaces: Salem Numbers 
    and Siegel Disks},
    J. Reine Angew. Math. {\bf 545} (2002), 201--233.
    
\bibitem[Mln4]{macmullen4}
    \textsc{C.T. McMullen},
    {\it Coxeter Groups, Salem Numbers and the Hilbert 
    Metric}, 
    Publ. Math. I.H.\'E.S. {\bf 95} (2002), 151--183. 

\bibitem[Mln5]{macmullen5}
    \textsc{C.T. McMullen},
    {\it Dynamics on Blowups of the Projective Plane}, 
    Publ. Math. I.H.\'E.S. {\bf 105} (2007), 49--89.
    
\bibitem[Mln6]{macmullen6}
    \textsc{C.T. McMullen},
    {\it K3 Surfaces, Entropy and Glue}, 
    J. Reine Angew. Math. {\bf 658} (2011), 1--25.   

\bibitem[Mln7]{macmullen7}
    \textsc{C.T. McMullen},
    {\it Automorphisms of Projective K3 Surfaces with
    Minimum Entropy},
    Invent. Math. {\bf 203} (2016), 179--215.

\bibitem[Mhr]{mahler}
    \textsc{K. Mahler},
    {\it On Two Extremal Properties of Polynomials},
    Illinois J. Math. {\bf 7} (1963), 681--701.

\bibitem[Mhr2]{mahler2}
    \textsc{K. Mahler},
    {\it An Inequality for the Discriminant 
    of a Polynomial},
    Michigan Math. J. {\bf 11} (1964), 257--262.

\bibitem[Mis]{margulis}
    \textsc{G.A. Margulis},
    {\it Discrete Subgroups of Semisimple Lie Groups},
    Springer Verlag (1991).
    
\bibitem[dMaF]{delamasafriedman} 
    \textsc{A.C. de la Masa {\rm and} E. Friedman},
    {\it Heights of Algebraic Numbers Modulo 
    Multiplicative Group Actions},
    J. Number Theory {\bf 128} (2008), 2199--2213.       
    
\bibitem[MP]{masakovapelantova}
    \textsc{Z. Mas\'akov\'a {\rm and} E. Pelantov\'a},    
    {\it Ito-Sadahiro Numbers vs. Parry Numbers},
    Acta Polytechnica {\bf 51} (2011), 59--64.
  
\bibitem[My]{masley}  
    \textsc{J.M. Masley},
    {\it On the First Factor of the Class 
    Number of Prime Cyclotomic Fields},
    J. Number Theory {\bf 10} (1978), 273--290.    

\bibitem[Mr]{masser}
    \textsc{D.W. Masser},
    {\it Small Values of the Quadratic Part of 
    the N\'eron-Tate Height on an Abelian Variety},
    Compos. Math. {\bf 53} (1984), 153--170.

\bibitem[Mr2]{masser2}
    \textsc{D.W. Masser},
    {\it Letter to D. Bertrand}, Nov. 17th 1986.

\bibitem[Mr3]{masser3}
    \textsc{D.W. Masser},
    {\it Small Values of Heights on Families of 
    Abelian Varieties},
    Diophantine Approximation and Transcendence Theory
    (Bonn, 1985), Lect. Notes in Math., 
    Springer, Berlin,  {\bf 1290} (1987), 109--148.
    
\bibitem[Mr4]{masser4}
    \textsc{D.W. Masser},
    {\it Counting Points of Small Height on Elliptic 
    Curves},
    Bull. Soc. Math. France {\bf 117} (1989), 247--265.

\bibitem[Mv]{matveev}
    \textsc{E.M. Matveev},
    {\it On the Cardinality of Algebraic Integers},
    Math. Notes {\bf 49} (1991), 437--438.

\bibitem[Mv2]{matveev2}
    \textsc{E.M. Matveev},
    {\it On a Connection Between the Mahler Measure and the 
    Discriminant of Algebraic Numbers},
    Math. Notes {\bf 59} (1996), 293--297.

\bibitem[Me]{meyer}
    \textsc{M. Meyer},
    {\it Le Probl\`eme de Lehmer, M\'ethode de Dobrowolski 
    et Lemme de Siegel ``\`a la Bombieri-Vaaler"}, 
    Publ. Math. Univ. P. et M. Curie (Paris VI), {\bf 90},
    Probl\`emes Diophantiens, 1988/89, No 5.
    
\bibitem[Mt]{mignotte}
    \textsc{M. Mignotte},
    {\it Entiers Alg\'ebriques dont les Conjugu\'es 
    sont Proches du Cercle Unit\'e},
    S\'eminaire Delange-Pisot-Poitou, 19e ann\'ee: 1977/78, 
    Th\'eorie des Nombres, Fasc. 2, Exp. No. 39, 6 pp, 
    Paris (1978).
    
\bibitem[Mt2]{mignotte2}
    \textsc{M. Mignotte},
    {\it Sur les Nombres Alg\'ebriques de Petite Mesure},
    Mathematics, pp 65--80, CTHS: Bull. Sec. Sci., 
    III, Bib. Nat., Paris (1981).    
    
\bibitem[Mt3]{mignotte3}
    \textsc{M. Mignotte},
    {\it Sur un Th\'eor\`eme de M. Langevin},
    Acta Arith. {\bf 54} (1989), 81--86.

\bibitem[Mt4]{mignotte4}
    \textsc{M. Mignotte},
    {\it Remarque sur une Question Relative \`a 
    des Fonctions Conjugu\'ees},
    C. R. Acad. Sci. Paris, S\'erie I, t. {\bf 315} 
    (1992), 907--911.

\bibitem[MtW]{mignottewaldschmidt}
    \textsc{M. Mignotte {\rm and} M. Waldschmidt},
    {\it On Algebraic Numbers of Small Height:
    Linear Forms in One Logarithm},
    J. Number Theory {\bf 4} (1994), 43--62.

\bibitem[Mor]{milnor}
    \textsc{J. Milnor},
    {\it Problems and Solutions: Advanced Problems
    5600--5609},
    Amer. Math. Mon. {\bf 75} (1968), 685--686.

\bibitem[MorT]{milnorthurston}
    \textsc{J. Milnor {\rm and} Thurston},
    {\it On Iterated Maps of the Interval}, In:
    {\it Dynamical systems},
    (College Park, MD, 1986--87),  
    Lecture Notes in Math. {\bf 1342}, 
    Springer, Berlin, (1988), 465--563.
    
\bibitem[Mka]{minakawa}    
    \textsc{H. Minakawa},
    {\it Examples of Pseudo-Anosov Homeomorphisms With 
    Small Dilatations},
    J. Math. Sci. Univ. Tokyo {\bf 13} (2006), 95--111.
    
\bibitem[Mo]{mori}
    \textsc{M. Mori},
    {\it Fredholm Determinant for Piecewise Linear 
    Transformations},
    Osaka J. Math. {\bf 27} (1990), 81--116.

\bibitem[Mo2]{mori2}
    \textsc{M. Mori},
    {\it Fredholm Determinant for Piecewise Monotonic 
    Transformations},
    Osaka J. Math. {\bf 29} (1992), 497--529.

\bibitem[Mon]{morton}
    \textsc{H.R. Morton},
    {\it Infinitely Many Fibred Knots Having the Same
    Alexander Polynomial},
    Topology {\bf 17} (1978), 101--104. 

\bibitem[Mf]{mossinghoff}
    \textsc{M.J. Mossinghoff},
    {\it Algorithms for the Determination of 
    Polynomials with Small Mahler Measure},
    Ph.D. Thesis, Univ. Texas, Austin (1995).

\bibitem[Mf2]{mossinghoff2}
    \textsc{M.J. Mossinghoff},
    {\it Polynomials with Small Mahler Measure},
    Math. Comp. {\bf 67} (1998), 1697--1705, S11-S14.

\bibitem[Mlist]{mossinghofflist}
    \textsc{M.J. Mossinghoff},
    {\it Known Polynomials with Small Mahler Measure 
    Through Degree 180},
    \mbox{{\tt http://www.cecm.sfu.caz/\~.mjm/Lehmer}}, 
    (1996) ; implemented (2001): P. Lisonek; and (2003):
    G. Rhin and J.-M. Sac-\'Ep\'ee; 
    complete through degree 40. 

\bibitem[MPV]{mossinghoffpinnervaaler}
    \textsc{M.J. Mossinghoff, C.G. Pinner {\rm 
    and} J.D. Vaaler},
    {\it Perturbing Polynomials with All Their 
    Roots on the Unit Circle},
    Math. Comp. {\bf 67} (1998), 1707--1726.

\bibitem[MRW]{mossinghoffrhinwu}
    \textsc{M.J. Mossinghoff, G. Rhin {\rm and} Q. Wu},
    {\it Minimal Mahler Measures},
    Exp. Math. {\bf 17} (2008), 451--458.

\bibitem[MuWu]{muwu}
    \textsc{Q. Mu {\rm and} Q. Wu},
    {\it The Measure of Totally Positive 
    Algebraic Integers},
    J. Number Theory {\bf 133} (2013), 12--19.

\bibitem[Mui]{murasugi}
    \textsc{K. Murasugi},
    {\it Classical Knot Invariants and 
    Elementary Number Theory},
    Contemp. Math. {\bf 416} (2006), 167--196.
    
\bibitem[NmRd]{neumannreid}
    \textsc{W.D. Neumann {\rm and} A.W. Reid},
    {\it Arithmetic of Hyperbolic Manifolds},
    Topology '90 (Colombus, OH, 1990), 
    Ohio State Univ. Math. Res. Inst. Publ., de Gruyter, 
    Berlin, {\bf 1} (1992), 273--310.
    
\bibitem[N]{nguemandong}
    \textsc{F. Ngu\'ema Ndong},
    {\it The $(-\beta)$-Shift
    and Associated Zeta Function},
    arXiv:1701.00774v1 (3 Janv 2017).

\bibitem[Nhi]{noguchi}
    \textsc{A. Noguchi},
    {\it Zeros of the Alexander Polynomial of Knot},
    Osaka Math. J. {\bf 44} (2007), 567--577.
    
\bibitem[Nka]{nonaka}
    \textsc{J. Nonaka},
    {\it The Growth Rates of Ideal Coxeter 
    Polyhedra in Hyperbolic 3-Space},
    arXiv:1504.06718

\bibitem[Ni]{notari}
    \textsc{C. Notari},
    {\it Sur le Produit des Conjugu\'es \`a 
    l'Ext\'erieur du Cercle Unit\'e d'un 
    Nombre Alg\'ebrique},
    C.R. Acad. Sci. Paris S\'er. A-B
    {\bf 286} (1978), A313--A315.    
    
\bibitem[OP]{odlyzkopoonen}
    \textsc{A.M. Odlyzko {\rm and} B. Poonen},
    {\it Zeros of Polynomials with $0, 1$ Coefficients},
    Enseign. Math. {\bf 39} (1993), 317--348.     

\bibitem[Ogo]{oguiso}
    \textsc{K. Oguiso},
    {\it Salem Polynomials and the Bimeromorphic 
    Automorphism Group of a Hyper-K\"ahler Manifold},
    Selected Papers on Analysis and Differential 
    Equations, Amer. Math. Soc. Transl. Ser. 2, 
    {\bf 230} (2010), 201--227.
    
\bibitem[Ogo2]{oguiso2}
    \textsc{K. Oguiso},
    {\it The Third Smallest Salem Number in Automorphisms
    of $K3$ Surfaces},
    Algebraic Geometry in East Asia-Seoul 2008, 
    Adv. Stud. Pure Math. {\bf 60} (2010), 331--360.  
 
\bibitem[Ogo3]{oguiso3}
    \textsc{K. Oguiso},
    {\it A Remark on Dynamical Degrees of Automorphisms
    of HyperK\"ahler Manifolds}, 
    Manuscripta Math. {\bf 130} (2009), 
    101--111.
 
\bibitem[Ogo4]{oguiso4}
    \textsc{K. Oguiso},
    {\it Pisot Units, Salem Numbers and Higher 
    Dimensional Projective Manifolds with 
    Primitive Automorphisms of Positive Entropy},
    preprint (2017). 
    
\bibitem[OgoTg]{oguisotruong}
    \textsc{K. Oguiso {\rm and} T.T. Truong},
    {\it Salem Numbers in Dynamics on K\"ahler
    Threefolds and Complex Tori},
    Math. Z. {\bf 278} (2014), 93--117.
   
 \bibitem[OgoTg2]{oguisotruong2}
    \textsc{K. Oguiso {\rm and} T.T. Truong},
    {\it Explicit Examples of Rational and
    Calabi-Yau Threefolds with Primitive Automorphisms
    of Positive Entropy},
    Kodaira Centennial Issue of the Journal of 
    Mathematical Sciences,
    the University of Tokyo {\bf 22} (2015), 361--385.  
    
\bibitem[Pao]{pacheco}
    \textsc{A. Pacheco},
    {\it Analogues of Lehmer's Problem in 
    Positive Characteristic},
    preprint (2003).  
    
\bibitem[Pol]{panaitopol}    
    \textsc{L. Panaitopol},
    {\it Minorations pour les Mesures de Mahler 
    de Certains Polyn\^omes Particuliers},
    J. Th\'eorie Nombres Bordeaux {\bf 12} (2000), 127--132.
    
\bibitem[Pu]{panju}
    \textsc{M. Panju},
    {\it Beta Expansion for Regular Pisot Numbers},
    J. Integer Sequences {\bf 14} (2011), 
    article 11.6.4., 22 pp.
 
\bibitem[PRS]{papanikolasrogerssamart} 
    \textsc{M.A. Papanikolas, M.D. Rogers 
    {\rm and} D. Samart},
    {\it The Mahler Measure of a Calabi-Yau Threefold
    and Special $L$-Values},
    Math. Z. {\bf 276} (2014), 1151--1163.

\bibitem[Paw]{parryw}
    \textsc{W(alter) Parry},
    {\it Growth Series of Coxeter Groups and Salem Numbers},
    J. of Alg. {\bf 154} (1993), 406--415. 
    
\bibitem[Pa]{parry}
    \textsc{W(illiam) Parry},
    {\it On the $\beta$-expansions of Real Numbers},
    Acta Math. Acad. Sci. Hungar.
    {\bf 11} (1960), 401--416.

\bibitem[Pa2]{parry2}
    \textsc{W. Parry},
    {\it Representations for Real Numbers},
    Acta Math. Acad. Sci. Hungar.
    {\bf 15} (1964), 95--105.

\bibitem[Pa3]{parry3}
    \textsc{W. Parry},
    {\it Symbolic Dynamics and Transformations 
    of the Unit Interval},
    Trans. Amer. Math. Soc. {\bf 122} (1964), 368--378.

\bibitem[PaPt]{parrypollicott}
     \textsc{W. Parry {\rm and} M. Pollicott},
     {\it Zeta functions and Periodic Orbit Structure of 
     Hyperbolic Dynamics},
     Ast\'erisque {\bf 187--188}, Societe Math\'ematique 
     de France, 
     Paris (1990).
     
\bibitem[Pxe]{pathiauxdelefosse}
    \textsc{M. Pathiaux-Delefosse},
    {\it R\'esultat de Cantor et Strauss 
    sur la Conjecture de Lehmer},
    Groupe de travail en th\'eorie 
    analytique et \'el\'ementaire des nombres,
    1986--1987, Publ.
    Math. Orsay, 88--01, Univ. Paris XI, 
    Orsay, (1988), 77--83.  

\bibitem[Per]{penner}
    \textsc{R.C. Penner},
    {\it Bounds on Least Dilatations},
    Proc. Amer. Math. Soc. {\bf 113} (1991),
    443--450.

\bibitem[PSg]{perssonschmeling}
    \textsc{T. Persson {\rm and}  J. Schmeling},
    {\it Dyadic Diophantine Approximation 
    and Katok's Horseshoe
    Approximation},
    Acta Arith. {\bf 132} (2008), 205--230. 
    
\bibitem[Pe]{petsche}
    \textsc{C. Petsche},
    {\it The Distribution of Galois Orbits of 
    Low Height},
    PhD Thesis, Univ. Texas, Austin (2003).   
    
\bibitem[Pe2]{petsche2}
    \textsc{C. Petsche},
    {\it A Quantitative Version of Bilu's 
    Equidistribution Theorem},
    Int. J. Number Theory {\bf 1} (2005), 281--291.

\bibitem[PS]{pfistersullivan}
    \textsc{C.-E. Pfister {\rm  and}  W.G. Sullivan},
    {\it Large Deviations Estimates for Dynamical 
    Systems Without the Specification Property. 
    Application to the beta-shifts},
    Nonlinearity {\bf 18} (2005), 237--261.
    
\bibitem[Phn]{philippon}
    \textsc{P. Philippon},
    {\it Sur des Hauteurs Alternatives I.},
    Math. Ann. {\bf 289} (1991), 255-283;
    {\it ibid, II.}, Ann. Inst. Fourier (Grenoble)
    {\bf 44} (1994), 1043--1065;
    {\it ibid, III.}, J. Math. Pures Appl. 
    {\bf 74} (1995), 345--365.  
    
\bibitem[Pi]{pierce}
    \textsc{T.A. Pierce},
    {\it The Numerical Factors of the Arithmetic 
    Forms $\prod_{i=1}^{n} (1 \pm \alpha_{i}^{m})$}, 
    Ann. of Math. {\bf 18} (1916-17), 
    53--64.
        
\bibitem[PoPr]{pignopinner}        
    \textsc{V. Pigno {\rm and} C.G. Pinner},        
    {\it The Lind-Lehmer Constant for Cyclic Groups 
    of Order Less Than $892,371,480$},
    Ramanujan J. {\bf 33} (2014), 295--300.
        
\bibitem[PoPrVl]{pignopinnervipismakul}        
    \textsc{V. Pigno, C.G. Pinner 
    {\rm and} W. Vipismakul},   
    {\it The Lind-Lehmer Constant for $\zb_m \times
    \zb_{p}^{n}$},
    (2016), preprint.   
        
\bibitem[PrVr]{pinnervaaler}        
    \textsc{C.G. Pinner {\rm and} J.D. Vaaler},
    {\it The Number of Irreducible Factors
    of a Polynomial, I},
    Trans. Amer. Math. Soc. {\bf 339} (1993), 
    809--834; {\it II},
    Acta Arith. {\bf 78} (1996), 125--142;
    {\it III},
    in Number Theory in Progress, {\bf 1}, 
    Zakopane-Ko\'scielsko, 1997,
    de Gruyter, Berlin (1999), 395--405. 

\bibitem[PrVr2]{pinnervaaler2}        
    \textsc{C.G. Pinner {\rm and} J.D. Vaaler},
    {\it Polynomials with Lind Mahler Measure One},
    preprint (2015). 
        
\bibitem[P]{poincare}
    \textsc{H. Poincar\'e},
    {\it Le\c{c}ons de M\'ecanique C\'eleste},
    Paris, Gauthier-Villars, t. {\bf I} (1905), 
    t. {\bf II-1} (1907),
    t. {\bf II-2} (1909), 
    t. {\bf III} (1910).

\bibitem[Pt]{pollicott}
    \textsc{M. Pollicott},
    {\it Meromorphic Extension of Generalized Zeta 
    Function},
    Invent. Math. {\bf 85} (1986), 147--164.

\bibitem[Pt2]{pollicott2}
    \textsc{M. Pollicott},
    {\it Periodic Orbits and Zeta Functions},
    Handbook of Dynamical Systems, Vol. {\bf 1A},
    North-Holland, Amsterdam (2002), 409--452.
    
\bibitem[Pl]{polya}
    \textsc{G. P\'olya},
    {\it \"Uber  Gewisse Notwendige 
    Determinantenkriterien    
    f\"ur die Fortsetzbarkeit einer Potenzreihe},
    Math. Ann. {\bf 99} (1928), 687--706.           

\bibitem[Pru]{pontreau}
    \textsc{C. Pontreau},
    {\it Minoration Effective de la Hauteur des
    Points d'une Courbe de $\gb_{m}^{2}$},
    Acta Arith. {\bf 120} (2005), 1--26.
    
\bibitem[Pru2]{pontreau2}
    \textsc{C. Pontreau},
    {\it Petits Points d'une Surface},
    Canad. J. Math. {\bf 61} (2009), 1118--1150.   
    
\bibitem[Pyr]{pottmeyer}
    \textsc{L. Pottmeyer},
    {\it Heights and Totally Real Numbers},
    Atti Accad. Naz. Lincei Rend. Lincei Mat. Appl. 
    {\bf 24} (2013), 471--483. 
    
\bibitem[Pn]{preston}
    \textsc{C. Preston},
    {\it What you Need to Know to Knead},
    Adv. in Math. {\bf 78} (1989), 192--252.

\bibitem[Pr]{pritsker}
    \textsc{I.E. Pritsker},
    {\it Small Polynomials with Integer Coefficients},
    J. Anal. Math. {\bf 96} (2005), 151--190.

\bibitem[Pr2]{pritsker2}
    \textsc{I.E. Pritsker},
    {\it An Areal Analog of Mahler's Measure},
    Illinois J. Math. {\bf 52} (2008), 347--363.

\bibitem[Pr3]{pritsker3}
    \textsc{I.E. Pritsker},
    {\it Distribution of Algebraic Numbers},
    J. Reine Angew. Math.
    {\bf 657} (2011), 5--80.

\bibitem[PF]{pytheasfogg}
    \textsc{N. Pytheas Fogg},
    {\it Substitutions in Dynamics, Arithmetics 
    and Combinatorics}, Eds. V. Berth\'e, S. Ferenczi,
    C. Mauduit and A. Siegel, Lect. Notes Math.
    {\bf 1794}, Springer-Verlag, Berlin (2002).

\bibitem[Rz]{ratazzi}
    \textsc{N. Ratazzi},
    {\it Th\'eor\`eme de Dobrowolski-Laurent pour les 
    Extensions Ab\'eliennes sur une Courbe Elliptique \`a 
    Multiplications Complexes},
    Int. Math. Res. Not. {\bf 58} (2004), 3121--3152.    
  
\bibitem[Rz2]{ratazzi2}
    \textsc{N. Ratazzi},
    {\it Densit\'e de Points et Minoration de Hauteur},
    J. Number Theory {\bf 106} (2004), 112--127.
 
 \bibitem[Rz3]{ratazzi3}
    \textsc{N. Ratazzi},
    {\it Probl\`eme 
    de Lehmer pour les Hypersurfaces de 
    Vari\'et\'es Ab\'eliennes de Type C.M.},
    Acta Arith. {\bf 113} (2004), 273--290.
    
\bibitem[Rz4]{ratazzi4}
    \textsc{N. Ratazzi},
    {\it Probl\`eme 
    de Lehmer sur $\gb_m$ et 
    M\'ethode des Pentes},
    J. Th\'eorie Nombres Bordeaux {\bf 19} (2007), 231--248.
    
\bibitem[Rz5]{ratazzi5}
    \textsc{N. Ratazzi},
    {\it Minoration de la Hauteur sur les 
    Vari\'et\'es Ab\'eliennes de type C.M. et 
    Applications},
    preprint (2007).
    
\bibitem[RzU]{ratazziullmo}
    \textsc{N. Ratazzi {\rm and} E. Ullmo},
    {\it Galois $+$ Equidistribution $=$ Manin-Mumford},    
    Arithmetic Geometry, Clay Math. Proc., 
    Amer. Math. Soc., Providence, RI, {\bf 8} (2009),         
    419--430.
    
\bibitem[Ra]{rausch}
    \textsc{U. Rausch},
    {\it On a Theorem of Dobrowolski about the Product 
    of Conjugate Numbers},
    Colloq. Math. {\bf 50} (1985), 137--142. 

\bibitem[Ry]{ray}
    \textsc{G.A. Ray},
    {\it Relations between Mahler's Measure
    and Values of $L$-Series},
    Canad. J. Math. {\bf 39} (1987), 694--732.

\bibitem[Ry2]{ray2}
    \textsc{G.A. Ray},
    {\it A Locally Parametrized Version of 
    Lehmer's Problem},
    Proc. Symp. in Applied Math., Amer. Math. 
    Soc. {\bf 48}
    (1994), 573--576.
    
\bibitem[Rd]{remond}
    \textsc{G. R\'emond},
    {\it Intersection de Sous-Groupes et de 
    Sous-Vari\'et\'es I.},
    Math. Ann. {\bf 333} (2005), 525--548.

\bibitem[Re]{renyi}
    \textsc{A. R\'enyi}, 
    {\it Representations
    for Real Numbers and their Ergodic Properties},
    Acta Math. Acad. Sci. Hungar.
    {\bf 8} (1957), 477--493.

\bibitem[Rke]{reschke}
    \textsc{P. Reschke},
    {\it Salem Numbers and Automorphisms 
    of Complex Surfaces},
    Math. Res. Lett. {\bf 19} (2012), 475--482.
 
\bibitem[Rke2]{reschke2}
    \textsc{P. Reschke},
    {\it Salem Numbers and Automorphisms 
    of Abelian Surfaces},
    Osaka J. Math. {\bf 54} (2017), 1--15.  
    
\bibitem[Rn]{rhin}
    \textsc{G. Rhin},
    {\it A Generalization of a Theorem of Schinzel},
    Colloquium Mathematicum {\bf 101}(2) (2004), 155--159. 

\bibitem[RnSE]{rhinsacepee}
    \textsc{G. Rhin {\rm and} J.-M. Sac-\'Ep\'ee},
    {\it New Methods Providing High Degree Polynomials
    with Small Mahler Measure},
    Exp. Math. {\bf 12} (2003), 457--461.

\bibitem[RnS]{rhinsmyth}
    \textsc{G. Rhin {\rm and} C.J. Smyth},
    {\it On the Absolute Mahler Measure of Polynomials 
    Having
    all Zeros in a Sector},
    Math. Comp. {\bf 64} (1995), 295--304.

\bibitem[RnW]{rhinwu}
    \textsc{G. Rhin {\rm and} Q. Wu},
    {\it On the Absolute Mahler Measure of 
    Polynomials Having
    all Zeros in a Sector II},
    Math. Comp. {\bf 74} (2005), 383--388.

\bibitem[RSN]{riesznagy}
    \textsc{F. Riesz {\rm and} B. Sz-Nagy},
    {\it Le\c{c}ons d'Analyse Fonctionnelle},
    Acad\'emie des Sciences de Hongrie,
    Gauthier-Villars, Paris (1955), 3e \'edition.

\bibitem[Ro]{rigo}
    \textsc{M. Rigo},
    {\it Formal Languages, Automata and Numeration Systems.
    1. Introduction to Combinatorics on Words.
    With a foreword by V. Berth\'e}; 
    {\it 
    2. Applications to Recognizability and Decidability.
    With a foreword by V. Berth\'e},
    ISTE, London, John Wiley \& Sons, Inc., 
    Hoboken, NJ, (2014).

\bibitem[Rey]{riley}
    \textsc{R. Riley},
    {\it Growth of Order of Homology of Cyclic
    Branched Covers of Knots},
    Bull. London Math. Soc. {\bf 22} (1990), 287--297.
    
\bibitem[Rba]{robba}
    \textsc{P. Robba},
    {\it Une Introduction Na\"ive aux Cohomologies 
    de Dwork},  
    M\'emoires de la S.M.F., 2e s\'erie, tome {\bf 23} 
    (1986), 61--105.

\bibitem[Ron]{robinson}
    \textsc{R.M. Robinson},
    {\it An Extension of P\'olya's Theorem on 
    Power Series with Integer Coefficients},
    Trans. Amer. Math. Soc. {\bf 130} (1968), 532--543.

\bibitem[RVs]{rodriguezvillegas}
    \textsc{F. Rodriguez-Villegas},
    {\it Modular Mahler Measures I},
    Topics in Number Theory, Kluwer Aacd. Publ., 
    Math. Appl. {\bf 467}, Dordrecht, 
    (1999), 17--48. 
    
\bibitem[Rgs]{rogers}
    \textsc{M.D. Rogers},
    {\it Hypergeometric Formulas for Lattice Sums and 
    Mahler Measures}, 
    Int. Math. Res. Not. (IMRN) {\bf 17} 
    (2011), 4027--4058.
    
\bibitem[Ron]{rolfsen} 
    \textsc{D. Rolfsen},
    {\it Knots and Links},
    Publish or Perish, Inc., Berkeley (1976).   
    
\bibitem[Ru]{ruelle}
    \textsc{D. Ruelle},
    {\it Statistical Mechanics of a One-Dimensional 
    Lattice Gas},
    Comm. Math. Phys.
    {\bf 9} (1968), 267--278.

\bibitem[Ru2]{ruelle2}
    \textsc{D. Ruelle},
    {\it Zeta-Functions for Expanding Maps and 
    Anosov Flows},
    Invent. Math. {\bf 34} (1976), 231--242.

\bibitem[Ru3]{ruelle3}
    \textsc{D. Ruelle},
    {\it A Measure Associated with Axiom A Attractors},
    Amer. J. Math. {\bf 98} (1976), 619--654.

\bibitem[Ru4]{ruelle4}
    \textsc{D. Ruelle},
    {\it Thermodynamic Formalism},
    Addison Wesley, Reading, MA, (1978).

\bibitem[Ru5]{ruelle5}
    \textsc{D. Ruelle},
    {\it An Extension of the Theory of Fredholm
    Determinants}, 
    Publ. Math. I.H.E.S. {\bf 72} (1990), 175--193.

\bibitem[Ru6]{ruelle6}
    \textsc{D. Ruelle},
    {\it Analytic Completion for Dynamical Zeta Functions},
    Helv. Phys. Acta {\bf 66} (1993), 181--191.

\bibitem[Ru7]{ruelle7}
    \textsc{D. Ruelle},
    {\it Dynamical Zeta Functions for Piecewise Monotone
    Maps of the Interval},
    CRM Monograph Series {\bf 4}, Amer. Math. Soc. (1994).

\bibitem[Ru8]{ruelle8}
    \textsc{D. Ruelle},
    {\it Dynamical Zeta Functions for 
    Maps of the Interval},
    Bull. Amer. Math. Soc. {\bf 30} (1994), 212--214.

\bibitem[Ru9]{ruelle9}
    \textsc{D. Ruelle},
    {\it Dynamical Zeta Functions and Tranfer Operators},
    Notices Amer. Math. Soc. {\bf 49} (2002), 887--895. 

\bibitem[Rly]{rumely}
    \textsc{R. Rumely},
    {\it On Bilu's Equidistribution Theorem},
    In {\it Spectral Problems in Geometry and Arithmetics
    (Iowa City, IA, 1997)}, Contemp. Math., 
    Amer. Math. Soc., Providence, RI,
    Vol. {\bf 237} (1999),159--166. 

\bibitem[SsZ]{sakszygmund}
    \textsc{S. Saks {\rm and} A. Zygmund},
    {\it Fonctions Analytiques},
    Masson et C$^{{\rm ie}}$, Paris (1970).
    
\bibitem[Sa]{salem}
    \textsc{R. Salem},
    {\it A Remarkable Class of Algebraic Integers. 
    Proof of a Conjecture of Vijayaraghavan},
    Duke Math. J. {\bf 11} (1944), 103--108.

\bibitem[Sa2]{salem2}
    \textsc{R. Salem},
    {\it Power Series with Integral Coefficients},
    Duke Math. J. {\bf 12} (1945), 153--172. 

\bibitem[Srt]{samart}
    \textsc{D. Samart},
    {\it Three-Variable Mahler Measures 
    and Special Values of Modular and 
    Dirichlet $L$-Series},
    Ramanujan J. {\bf 32} (2013), 245--268. 
    
\bibitem[Set]{samet}    
    \textsc{P.A. Samet},
    {\it Algebraic Integers with two Conjugates 
    Outside the Unit Circle, I},
    Proc. Cambridge Philos. Soc. {\bf 49} (1953), 
    421--436; {\it II}, {\it ibid} {\bf 50} (1954), 346.
    
\bibitem[Sls]{samuels}
    \textsc{C.L. Samuels},
    {\it Lower Bounds on the Projective Heights 
    of Algebraic Points},
    Acta Arith. {\bf 125} (2006), 41--50.
 
\bibitem[Sls2]{samuels2}
    \textsc{C.L. Samuels},
    {\it A Collection of Metric Mahler Measures},
    J. Ramanujan Math. Soc. {\bf 25} (2010), 433--456.
 
\bibitem[Sls3]{samuels3}
    \textsc{C.L. Samuels},
    {\it The Infimum in the Metric Mahler Measure},
    Canad. Math. Bull. {\bf 54} (2011), 739--747.
    
\bibitem[Sls4]{samuels4}
    \textsc{C.L. Samuels},
    {\it The Parametrized Family of Metric Mahler Measures},
    J. Number Theory {\bf 131} (2011), 1070--1088.    

\bibitem[Ski]{sasaki}
    \textsc{Y. Sasaki},
    {\it On Multiple Higher Mahler Measures 
    and Multiple $L$ values},
    Acta Arith. {\bf 144} (2010), 159--165.
        
\bibitem[Scl]{schanuel}
    \textsc{S.H. Schanuel},    
    {\it Heights in Number Fields},
    Bull. Soc. Math. France {\bf 107} (1979), 433--449.
    
\bibitem[SrTr]{scheicherthuswaldner}
    \textsc{K. Scheicher {\rm and} J. Thuswaldner},
    {\it Canonical Number Systems, Counting Automata and 
    Fractals},
    Math. Proc. Cambridge Philos. Soc. {\bf 133} (2002),
    163--182.
    
\bibitem[ShVo]{shindervlasenko}
    \textsc{E. Shinder {\rm and} M. Vlasenko},
    {\it Linear Mahler Measures and Double 
    $L$-Values of Modular Forms},
    J. Number Theory {\bf 142} (2014), 149--182.
        
\bibitem[Sc]{schinzel}
    \textsc{A. Schinzel},
    {\it Reducibility of Lacunary Polynomials},
    Acta Arith. {\bf 16} (1969), 123--159.

\bibitem[Sc2]{schinzel2}
    \textsc{A. Schinzel},
    {\it On the Product of the Conjugates Outside 
    the Unit Circle 
    of an Algebraic Number}, 
    Acta Arith. {\bf 24} (1973), 385--399. 
    Addendum: ibid. {\bf 26} (1974/75), 329--331.

\bibitem[Sc3]{schinzel3}
    \textsc{A. Schinzel},
    {\it On the Mahler Measure of Polynomials 
    in Many Variables},
    Acta Arith. {\bf 79} (1997), 77--81.

\bibitem[Sc4]{schinzel4}
    \textsc{A. Schinzel},
    {\it On Values of the Mahler Measure in a 
    Quadratic Field (Solution of a Problem
    of Dixon and Dubickas)},
    Acta Arith. {\bf 113} (2004), 401--408.
    
\bibitem[Sc5]{schinzel5}
    \textsc{A. Schinzel},
    {\it Self-Inversive Polynomials With All Zeroes on the
    Unit Circle},
    Ramanujan J. {\bf 9} (2005), 19--23.    
    
\bibitem[SZ]{schinzelzassenhaus}
    \textsc{A. Schinzel {\rm and} H. Zassenhaus},
    {\it A Refinement of Two Theorems of Kronecker},
    Michigan Math. J. {\bf 12} (1965), 81--85.

\bibitem[Sg]{schmeling}
    \textsc{J. Schmeling},
    {\it Symbolic Dynamics for $\beta$-Shifts 
    and Self-Normal Numbers},
    Ergod. Th. Dynam. Sys. {\bf 17} (1997), 675--694.    
    
\bibitem[Sdt]{schmidt}
    \textsc{K. Schmidt},
    {\it On Periodic Expansions of Pisot Numbers
    and Salem Numbers},
    Bull. London Math. Soc. {\bf 12} (1980), 269--278.

\bibitem[Sdt2]{schmidt2}
    \textsc{K. Schmidt},
    {\it Dynamical Systems of Algebraic Origin},
    Progress in Math. {\bf 128}, Birkha\"user (1995).

\bibitem[Swt]{schmidtw}
    \textsc{W.M. Schmidt},
    {\it Asymptotic Formulae for Point Lattices of
    Bounded Determinant and Subspaces of Bounded Height},
    Duke Math. J. {\bf 35} (1968), 327--339.
    
\bibitem[Swt2]{schmidtw2}
    \textsc{W.M. Schmidt},
    {\it Northcott's Theorem on Heights II. The Quadratic 
    Case},
    Acta Arith. {\bf 70} (1995), 343--375.    

\bibitem[Swt3]{schmidtw3}
    \textsc{W.M. Schmidt},
    {\it Heights of Points on Subvarieties of 
    $\gb_{m}^{n}$},
    Number Theory (Paris, 1993-1994), London 
    Math. Soc. lecture Notes Ser.
    {\bf 235}, Cambridge Univ. Press, Cambridge (1996),     
    157--187.
    
\bibitem[Srt]{seifert}  
    \textsc{H. Seifert},
    {\it \"Uber das Geschlecht von Knoten},
    Math. Ann. {\bf 110} (1934), 571--592.  
    
\bibitem[Sr]{selmer}
    \textsc{E.S. Selmer},
    {\it On the Irreducibility of Certain Trinomials},
    Math. Scand. {\bf 4} (1956), 287--302.

\bibitem[Sda]{shimada}
    \textsc{I. Shimada},
    {\it Automorphisms of Supersingular $K3$ Surfaces
    and Salem Polynomials},
    Exp. Math. {\bf 25} (2016), 389--398.

\bibitem[Shin]{shin}
    \textsc{H. Shin},
    {\it Algebraic degrees of Stretch Factors in Mapping
    Class Groups},
    Algebr. Geom. Topol. {\bf 16} (2016), 1567--1584.

\bibitem[ShinSr]{shinstrenner}
    \textsc{H. Shin {\rm and} B. Strenner},
    {\it Pseudo-Anosov Mapping Classes Not Arising
    from Penner's Construction},
    Geom. Topol.{\bf 19} (2015), 3645--3656.
    
\bibitem[Si]{siegel}
     \textsc{C.L. Siegel},
     {\it Algebraic Integers whose Conjugates Lie in 
     the Unit Circle},
     Duke Math. J. {\bf 11} (1944), 597--602.

\bibitem[SWs]{silverwilliams}
    \textsc{D.S. Silver {\rm and} S.G. Williams},
    {\it Mahler Measure, Links and Homology Growth},
    Topology {\bf 41} (2002), 979--991.
    
\bibitem[SWs2]{silverwilliams2}
    \textsc{D.S. Silver {\rm and} S.G. Williams},
    {\it Mahler Measure of Alexander Polynomials},
    J. London Math. Soc. {\bf 69} (2004), 767--782.  

\bibitem[SWs3]{silverwilliams3}
    \textsc{D.S. Silver {\rm and} S.G. Williams},
    {\it Lehmer's Question, Knots and Surface Dynamics},
    Math. Proc. Cambridge Philos. Soc. 
    {\bf 143} (2007), 649--661.

\bibitem[Sn]{silverman}
    \textsc{J.H. Silverman},
    {\it Lower Bound for the Canonical Height 
    on Elliptic Curves},
    Duke Math. J. {\bf 48} (1981), 633--648.

\bibitem[Sn2]{silverman2}
    \textsc{J.H. Silverman},
    {\it Lower Bounds for Height Functions},
    Duke Math. J. {\bf 51} (1984), 395--403. 
 
\bibitem[Sn3]{silverman3}
    \textsc{J.H. Silverman},
    {\it Small Salem Numbers, Exceptional Units, 
    and Lehmer's Conjecture},
    Rocky Mountain J. Math. {\bf 26} (1996), 1099--1114.
    
\bibitem[Sn4]{silverman4}
    \textsc{J.H. Silverman},
    {\it A Lower Bound for the Canonical Height 
    on Elliptic Curves over Abelian Extensions},    
     J. Number Theory {\bf 104} (2004), 353--372.
    
\bibitem[Sn5]{silverman5}
    \textsc{J.H. Silverman},
    {\it Lehmer's Conjecture for Polynomials 
    Satisfying a Congruence Divisibility Condition 
    and an Analogue for Elliptic Curves},
    J. Th\'eorie Nombres Bordeaux {\bf 24} (2012), 
    751--772.

\bibitem[Si]{sinclair}
    \textsc{C.D. Sinclair},
    {\it The Distribution of Mahler's Measures of 
    Reciprocal Polynomials},
    Int. J. Math. Math. Sci. {\bf 49--52} (2004), 
    2773--2786.

\bibitem[SiYv]{sinclairyattselev}
    \textsc{C.D. Sinclair {\rm and} M.L. Yattselev},
    {\it Root Statistics of Random Polynomials with 
    Bounded Mahler Measure},
    Adv. Math. {\bf 272} (2015), 124--199.

\bibitem[Sy]{smyth}
    \textsc{C. Smyth},
    {\it On the Product of the 
    Conjugates Outside the Unit Circle 
    of an Algebraic Integer},
    Bull. Lond. Math. Soc. {\bf 3} 
    (1971), 169--175.

\bibitem[Sy2]{smyth2}
    \textsc{C. Smyth},
    {\it Topics in the Theory of 
    Numbers},
    PhD Thesis, Univ. of Cambridge (1972). 

\bibitem[Sy3]{smyth3}
    \textsc{C. Smyth},
    {\it On the Measure of Totally Real 
    Algebraic Integers},
    J. Australian Math. Soc. A {\bf 30} (1980/81), 137--149;
    {\it II},
    Math. Comp. {\bf 37} (1981), 205--208.

\bibitem[Sy4]{smyth4}
    \textsc{C. Smyth},
    {\it On Measures of Polynomials in 
    Several Variables},
    Bull. Austral. Math. Soc. {\bf 23} 
    (1981), 49--63; Corrigendum by C. Smyth and G. Myerson,
    Bull. Austral. Math. Soc. {\bf 26} 
    (1982), 317--319.

\bibitem[Sy5]{smyth5}
    \textsc{C. Smyth},
    {\it The Mahler Measure of Algebraic 
    Numbers: A Survey},
    in {\it Number Theory and 
    Polynomials}, London Math. Soc. 
    Lecture Note Ser. 
    {\bf 352}, Cambridge Univ. Press, 
    Cambridge (2008), 322--349.

\bibitem[Sy6]{smyth6}
    \textsc{C. Smyth},
    {\it Seventy Years of Salem Numbers},
    Bull. London Math. Soc. {\bf 47} (2015),
    379--395. 

\bibitem[Sk]{solomyak}
    \textsc{B. Solomyak},
    {\it Conjugates of beta-Numbers and the Zero-Free 
    Domain for a Class of Analytic Functions},
    Proc. London Math. Soc. {\bf 68} (1994), 477--498.

\bibitem[Sle]{soule}
    \textsc{C. Soul\'e},
    {\it G\'eom\'etrie d'Arakelov et 
    Th\'eorie des Nombres Transcendants},
    Journ\'ees Arithm\'etiques, 1989 (Luminy),
    Ast\'erisque, {\bf 198--200} (1992), 355--371.
    
\bibitem[Sts]{staines}
    \textsc{M. Staines},
    {\it On the Inverse Problem for 
    Mahler Measure},
    PhD Thesis, Univ. East Anglia 
    (2012).

\bibitem[Sst]{standfest}
    \textsc{G. Standfest},
    {\it Mahler Masse Linearer Polynome},
    Ph.D. Thesis, Westf\"ahliche 
    Wilhelms-Universit\"at M\"unster (2001).
    
\bibitem[Stv]{stankov}    
    \textsc{D. Stankov},
    {\it The Necessary and Sufficient Condition for an
    Algebraic Integer to Be a Salem Number},
    preprint (2017).
    
\bibitem[Sff]{steffensen}
    \textsc{J.F. Steffensen},
    {\it Interpolation}, (1927); 
    reprint, 2nd ed.
    Chelsea Publ. Co, New York (1950).  

\bibitem[Stg]{steinberg}
    \textsc{R. Steinberg},
    {\it Endomorphisms of Linear Algebraic Groups},
    M\'em. Amer. Math. Soc. {\bf80}, A.M.S.,
    Providence, RI (1968).
    
\bibitem[St]{stewart}
     \textsc{C.L. Stewart},
     {\it On a Theorem of Kronecker and a Related 
     Question of Lehmer},
     S\'em. Delange-Pisot-Poitou, Bordeaux (1977/78),
     No 7, 11 p.
     
\bibitem[St2]{stewart2}
     \textsc{C.L. Stewart},
     {\it Algebraic Integers Whose 
     Conjugates Lie Near the Unit 
     Circle},
     Bull. Soc. Math. France {\bf 106} 
     (1978), 169--176.

\bibitem[SB]{stoerbulirsch}
     \textsc{J. Stoer {\rm and} R. 
     Bulirsch},
     {\ Introduction to Numerical 
     Analysis},     
     Texts in Appl. Math. {\bf 12}, 2nd 
     ed., Springer-Verlag, New York 
     (1993).
  
\bibitem[Stw]{stoimenow} 
    \textsc{A. Stoimenow},
    {\it Log-Concavity and Zeros of the 
    Alexander Polynomial},
    Bull. Korean Math. Soc. {\bf 51} (2014), 539--545.     
  
\bibitem[Stw2]{stoimenow2} 
    \textsc{A. Stoimenow},
    {\it Hoste's Conjecture and Roots of Link 
    Polynomials}, 
    subm. Ann. Combinatorics (2016), preprint.
  
\bibitem[Sry]{sury}  
     \textsc{B. Sury},
     {\it Arithmetic Groups and Salem Numbers},
     Manuscr. Math. {\bf 75} (1992), 97--102.
     
\bibitem[Szi]{suzuki} 
    \textsc{M. Suzuki},
    {\it On Zeros of Self-Reciprocal Polynomials},
    preprint (2012).
     
\bibitem[Szo]{szego}
    \textsc{G. Szeg\H{o}},
    {\it \"Uber Potenzreihen mit Endlich Vielen 
    Verschiedenen Koeffizienten},
    Sitzungs-berichte der preussischen Akademie der 
    Wissenschaften Berlin (1922), 88--91.
    
\bibitem[SUZ]{szpiroullmozhang}
    \textsc{L. Szpiro, E. Ullmo {\rm and} S. Zhang},
    {\it Equir\'epartition des Petits Points},
    Invent. Math. {\bf 127} (1997), 337--347.    
    
\bibitem[T]{takahashi}
    \textsc{Y. Takahashi},
    {\it Isomorphisms of $\beta$-Automorphisms to Markov 
    Automorphisms},
    Osaka J. Math. {\bf 10} (1973), 175--184.
    
\bibitem[T2]{takahashi2}
    \textsc{Y. Takahashi},
    {\it A Formula for Topological Entropy of 
    One-Dimensional Dynamics},    
    Sci. Papers College Gen. Educ. Univ. Tokyo 
    {\bf 30} (1980), 11--22.    
    
\bibitem[T3]{takahashi3}
    \textsc{Y. Takahashi},
    {\it Fredholm Determinant of Unimodal Linear Maps},    
    Sci. Papers College Gen. Educ. Univ. Tokyo 
    {\bf 31} (1981), 61--87.
    
\bibitem[T4]{takahashi4}
    \textsc{Y. Takahashi},
    {\it An Ergodic-Theoretical Approach to the Chaotic 
    Behaviour of Dynamical Systems},
    Publ. RIMS, Kyoto Univ. {\bf 19} (1983), 1265--1282.    
    
\bibitem[T5]{takahashi5}
    \textsc{Y. Takahashi},
    {\it Shift with Orbit Basis and Realization 
    of One Dimensional Maps},
    Osaka J. Math. {\bf 20} (1983), 599--629.
    Corrigendum: ibid {\bf 22} (1985), 637.
    
\bibitem[Thi]{takeuchi}    
    \textsc{K. Takeuchi},
    {\it Commensurability Classes of 
    Arithmetic Triangle Groups},
    Fac. Sci. Univ. Tokyo {\bf 24} (1977), 201--212.
        
\bibitem[TW]{tanwang}    
    \textsc{B. Tan {\rm and} B.-W. Wang},
    {\it Quantitative Recurrence Properties for 
    beta-Dynamical Systems},
    Adv. Math. {\bf 228} (2011), 2071--2097.
        
\bibitem[Tao]{tao}        
    \textsc{T. Tao},
    {\it Dwork's Proof of Rationality of the Zeta 
    Function over Finite Fields},
    blog: https://terrytao.wordpress.com/2014/05/13/
        
\bibitem[Th]{thompson}
    \textsc{D.J. Thompson},
    {\it Irregular Sets, the $\beta$-Transformation 
    and the Almost Specification Property},
    Trans. Amer. Math. Soc. {\bf 364} (2012), 5395--5414.
    
\bibitem[Th2]{thompson2}
    \textsc{D.J. Thompson},
    {\it Generalized $\beta$-Transformation and 
    the Entropy of Unimodal Maps},    
    arXiv:1602.03518v2, (17 Feb. 2016). 

\bibitem[Tn]{thurston}    
    \textsc{W.P. Thurston},
    {\it On the Geometry and Dynamics of
    Diffeomorphisms of Surfaces},
    Bull. Amer. Math. Soc. (N.S.) {\bf 19} (1988),
    417--431.
    
\bibitem[Tn2]{thurston2}    
    \textsc{W.P. Thurston},
    {\it Entropy in Dimension One},
    in {\it Frontiers in Complex Dynamics}, volume {\bf 51}
    of Princeton Math. Series, Princeton Univ. Press, 
    Princeton,
    NJ, (2014), pp. 339--384.

\bibitem[To]{toledano}
    \textsc{R. Toledano},
    {\it Polynomials with All Their Roots on S$^1$},
    Lithuanian Math. J. {\bf 49} (2009), 331--336.
 
\bibitem[Tsi]{tsai} 
    \textsc{C.-Y. Tsai},
    {\it The Asymptotic Behavior of Least Pseudo-Anosov
     Dilatations},
    Geom. Topol. {\bf 13} (2009), 2253--2278.    

\bibitem[Tuv]{turaev}
    \textsc{V.G. Turaev},
    {\it Introduction to Combinatorial Torsions; 
    Notes Taken by F. Schlenk}, 
    Lect. in Math. ETH Z\"urich, Springer Basel AG (2001).
   
\bibitem[U]{umemoto}    
    \textsc{Y. Umemoto},
    {\it The Growth Function of Coxeter Dominoes
    and 2-Salem Numbers},     
    Algebr. Geom. Topology {\bf 14} (2014), 
    2721--2746.   
    
\bibitem[U2]{umemoto2}    
    \textsc{Y. Umemoto},
    {\it Growth Rates of Cocompact Hyperbolic Coxeter 
    Groups and $2$ - Salem Numbers},
    preprint (2015).
        
\bibitem[Var]{vaaler}
    \textsc{J.D. Vaaler},
    {\it Heights on Groups and Small 
    Multiplicative Dependencies},
    Trans. Amer. Math. Soc. {\bf 366} (2014), 
    3295--3323.            
        
\bibitem[VG]{vergergaugry}
    \textsc{J.-L. Verger-Gaugry},
    {\it On Gaps in R\'enyi $\beta$-Expansions of Unity 
    for $\beta > 1$ an Algebraic Number},
    Ann. Institut Fourier (Grenoble)
    {\bf 56} (2006), 2565--2579.
    
\bibitem[VG2]{vergergaugry2}
    \textsc{J.-L. Verger-Gaugry},
    {\it On the Dichotomy of Perron Numbers and 
    Beta-Conjugates},
    Monatsh. Math. {\bf 155} (2008), 277--299.   
    
\bibitem[VG3]{vergergaugry3}
    \textsc{J.-L. Verger-Gaugry},
    {\it Uniform Distribution of the 
    Galois Conjugates
    and Beta-Conjugates of a Parry 
    Number Near the Unit Circle and
    Dichotomy of Perron Numbers},  
    Uniform Distribution Theory J. {\bf 
    3} (2008), 157--190.
    
\bibitem[VG4]{vergergaugry4}
     \textsc{J.-L. Verger-Gaugry},
     {\it Beta-Conjugates of Real Algebraic 
     Numbers as Puiseux Expansions}, 
     Integers {\bf 11B} (2011). 
     
\bibitem[VG5]{vergergaugry5}
     \textsc{J.-L. Verger-Gaugry},
     {\it R\'enyi-Parry Germs of Curves and Dynamical 
     Zeta Functions Associated with Real 
     Algebraic Numbers},
     RIMS K\^oky\^uroku Bessatsu, Kyoto {\bf B46}
     (2014), 241--247.    
     
\bibitem[VG6]{vergergaugry6}
     \textsc{J.-L. Verger-Gaugry},
     {\it On the Conjecture of Lehmer,  
     Limit Mahler Measure of Trinomials
     and Asymptotic Expansions},
     Uniform Distribution Theory J.      
     {\bf 11} (2016), 79--139. 

\bibitem[VG7]{vergergaugry7}
    \textsc{J.-L. Verger-Gaugry},
    {\it A proof of Lehmer's Conjecture by the 
    dynamical zeta function of the $\beta$-shift 
    and a new Dobrowolski type minoration },
    C.R. Acad. Sci. Paris, Ser. I (2017), subm.
    
\bibitem[Vra]{vieira}
    \textsc{R.S. Vieira},
    {\it On the Number of Roots of Self-Inversive 
    Polynomials on the Complex Unit Circle},
    Ramanujan J. {\bf 42} (2017), 363--369.
    
\bibitem[VS]{vinbergshvartsman}
    \textsc{E.B. Vinberg {\rm and} O.V. Shvartsman},
    {\it Discrete Groups of Motions of Spaces
    of Constant Curvature},
    in Geometry II, Encyclopaedia Math. Sci.,
    Springer, Berlin, {\bf 29} (1993), 139--248.
    
\bibitem[Vl]{vipismakul}
    \textsc{W. Vipismakul},
    {\it The Stabilizer of the Group Determinant 
    and Bounds for Lehmer's Conjecture on 
    Finite Abelian Groups},
    PhD Thesis, University of Texas, Austin (May 2013).
    
\bibitem[V]{voutier}
     \textsc{P.M. Voutier},
     {\it An Effective Lower Bound for 
     the Height of Algebraic Numbers},
     Acta Arith. {\bf 74} (1996), 
     81--95.

\bibitem[W]{waldschmidt}
     \textsc{M. Waldschmidt},
     {\it Sur le Produit des Conjugu\'es 
     Ext\'erieurs au Cercle},
     Enseign. Math. {\bf 26} 
     (1980), 201--209.
     
\bibitem[W2]{waldschmidt2}
     \textsc{M. Waldschmidt},
     {\it Auxiliary Functions in 
     Transcendental Number Theory},
     SASTRA Ramanujan lectures, 
     The Ramanujan Journal {\bf 20} (3), 
     (2009), 341--373.

\bibitem[W3]{waldschmidt3}
     \textsc{M. Waldschmidt},
     {\it Diophantine Approximation on 
     Linear Algebraic Group:
     Transcendence Properties of the 
     Exponential Function 
     in Several Variables},
     Grund. Math. Wiss. {\bf 326}, 
     Springer-Verlag Berlin (2000).
     
\bibitem[We]{weil}
     \textsc{A. Weil},
     {\it Number of Solutions of Equations over Finite 
     Fields},
     Bull. Amer. Math. Soc. {\bf 55} (1949), 497--508.

\bibitem[WuWg]{wuwang}
    \textsc{F.Y. Wu {\rm and} J. Wang},
    {\it Zeroes of the Jones Polynomial},
    Physica A {\bf 296} (2001), 483--494.
           
\bibitem[Wu]{wu}
     \textsc{Q. Wu},
      {\it The Smallest Perron Numbers},
      Math. Comp. {\bf 79} (2010), 
      2387--2394.

\bibitem[Y]{yale}
    \textsc{P.B. Yale},
    {\it Automorphisms of the Complex Numbers},
    Math. Mag. {\bf 39} (1966), 135--141.

\bibitem[Yn]{yomdin}
    \textsc{Y. Yomdin},
    {\it Volume Growth and Entropy},
    Israel J. Math. {\bf 57} (1987), 285-300.
  
\bibitem[Yu]{yu}
    \textsc{X. Yu},
    {\it Elliptic Fibrations on $K3$ Surfaces and
    Salem Numbers of Maximal Degree},
    preprint (2016).  

\bibitem[Za]{zagier}
     \textsc{D. Zagier},
    {\it Algebraic Numbers Close to $0$        
    and $1$}, Math. Comp. {\bf 61} 
    (1993), 485--491.
    
\bibitem[Za2]{zagier2}
    \textsc{D. Zagier},
    {\it Two Results in Number Theory with 
    Elementary Aspects},
    Algebraic Number Theory - recent developments and 
    their backgrounds (Kyoto, 1992),
    S\'urikaisekikenky\'usho Koky\'uroku {\bf 844} 
    (1993), 84--86.
        
\bibitem[Za3]{zagier3}
    \textsc{D. Zagier},
    {\it The Dilogarithm Function},
    Frontiers in Number Theory, 
    Physics and Geometry. II, Springer, 
    Berlin (2007), 3--65. 
    
\bibitem[Zi]{zaimi}
    \textsc{T. Za\"imi},
    {\it Sur les Nombres de Pisot 
    Relatifs},
    Doctoral Thesis, University Paris 6, 
    (1994).
    
\bibitem[Zi2]{zaimi2}
    \textsc{T. Za\"imi},
    {\it Sur les $K$-Nombres de Pisot de 
    Petite Mesure},
    Acta Arith. {\bf 77} (1996), 
    103--131.

\bibitem[ZZ]{zehrtzehrtliebendorfer}
    \textsc{T. Zehrt {\rm and} C. Zehrt-Liebend\"orfer},
    {\it The Growth Function of Coxeter Garlands in
    $\hb^4$},
    Beitr. Alge. Geom. {\bf 53} (2012), 451--460.

\bibitem[Zhg]{zhang}
    \textsc{S. Zhang},
    {\it Positive Line Bundles on Arithmetic Surfaces},
    Ann. of Math. {\bf 136} (1992), 569--587.

\bibitem[Zhg2]{zhang2}
    \textsc{S. Zhang},
    {\it Small Points and Adelic Metrics},
    J. Algebraic Geom. {\bf 4} (1995), 281--300.
    
\bibitem[Zhg3]{zhang3}
    \textsc{S. Zhang},
    {\it Equidistribution of Small Points 
    on Abelian Varieties},
    Ann. of Math. {\bf 147} (1998), 159--165.    

\bibitem[Zhao]{zhao}
    \textsc{S. Zhao},
    {\it Nombres de Salem et Automorphismes \`a
    Entropie Positive de Surfaces Ab\'eliennes et
    de Surfaces $K3$},
    preprint (2016).

\bibitem[Zhv]{zhirov}
    \textsc{A.Y. Zhirov},
    {\it On the Minimum Dilatation of Pseudo-Anosov 
    Diffeomorphisms of a Double Torus},
    Uspekhi Mat. Nauk {\bf 50} (1995), no 1(301), 197--198.
    
\bibitem[Zou]{zhou}    
    \textsc{J. Zhou},
    {\it Integrality Properties of
    Variations of Mahler Measures},
    preprint (2011).
    
\bibitem[Zun]{zudilin}
    \textsc{T. Zudilin},
    {\it Regulator of Modular Units and Mahler Measures},
    Math. Proc. Cambridge Phil. Soc.
    {\bf 156}  (2014), 313--326.


\end{thebibliography}
\end{document}